\documentclass[11pt, oneside]{report}   	
\usepackage{geometry}                		
\usepackage{graphicx}				

\usepackage[all]{xy}
\CompileMatrices

\usepackage{amssymb}
\usepackage{amsmath}
\usepackage{stmaryrd}
\usepackage{amsthm}
\usepackage{tikz-cd}
\usepackage{extarrows}
\usepackage[mathscr]{euscript}
\usepackage{mathrsfs} 
\usepackage{verbatim}
\usepackage[OT2,T1]{fontenc}
\DeclareSymbolFont{cyrletters}{OT2}{wncyr}{m}{n}
\DeclareMathSymbol{\Sha}{\mathalpha}{cyrletters}{"58}
\DeclareMathSymbol{\Che}{\mathalpha}{cyrletters}{"51}

\usepackage{calligra}
\newcommand{\calHom}{\mathscr{H}\mathit{om}}
\newcommand{\calExt}{\mathscr{E}\mathit{xt}}

\newcommand{\Ga}{{\mathbf{G}}_{\rm{a}}}
\newcommand{\Gm}{{\mathbf{G}}_{\rm{m}}}
\DeclareMathOperator{\red}{red}
\DeclareMathOperator{\prim}{prim}
\DeclareMathOperator{\Gal}{Gal}

\DeclareMathOperator{\Res}{Res}

\DeclareMathOperator{\ev}{ev}
\DeclareMathOperator{\inv}{inv}
\DeclareMathOperator{\Br}{Br}

\DeclareMathOperator{\Pic}{Pic}
\DeclareMathOperator{\ab}{\rm{ab}}
\DeclareMathOperator{\Ext}{Ext}
\DeclareMathOperator{\Hom}{Hom}

\DeclareMathOperator{\R}{R}
\DeclareMathOperator{\sm}{sm}
\DeclareMathOperator{\Div}{div}
\DeclareMathOperator{\Spec}{Spec}
\DeclareMathOperator{\coker}{coker}
\DeclareMathOperator{\im}{im}
\DeclareMathOperator{\pro}{pro}
\DeclareMathOperator{\tors}{tors}
\DeclareMathOperator{\ord}{ord}
\DeclareMathOperator{\et}{\acute{e}t}
\DeclareMathOperator{\fppf}{fppf}
\newcommand*{\RR}{\ensuremath{\mathbf{R}}}                        
\newcommand*{\Z}{\ensuremath{\mathbf{Z}}}                        
\newcommand*{\Q}{\ensuremath{\mathbf{Q}}}                     
\newcommand*{\F}{\ensuremath{\mathbf{F}}}                        
\newcommand*{\C}{\ensuremath{\mathbf{C}}}                        
\newcommand*{\A}{\ensuremath{\mathbf{A}}}                        
\newcommand*{\N}{\ensuremath{\mathbf{N}}}                        
\newcommand*{\calO}{\mathcal{O}}                                  

\usepackage{rotating}
\newcommand*{\isoarrow}[1]{\arrow[#1,"\rotatebox{90}{\(\sim\)}"
]}
\usepackage{bm}

\numberwithin{equation}{section}
\newcounter{defcounter}
\newenvironment{appendixequation}{\addtocounter{equation}{-1}
\refstepcounter{defcounter}

\begin{equation}}
{\end{equation}}

\numberwithin{defcounter}{chapter}

\newtheorem{theorem}{Theorem}[section]
\newtheorem{lemma}[theorem]{Lemma}
\newtheorem{proposition}[theorem]{Proposition}
\newtheorem{corollary}[theorem]{Corollary}

\theoremstyle{appendixtheorem}
\newtheorem{appendixtheorem}{Theorem}[chapter]

\newtheorem{appendixproposition}[appendixtheorem]{Proposition}

\newtheorem{appendixdefinition}[appendixtheorem]{Definition}
\theoremstyle{remark}
\newtheorem{appendixremark}[appendixtheorem]{Remark}

\theoremstyle{definition}
  \newtheorem{definition}[theorem]{Definition}

\theoremstyle{remark}
  \newtheorem{remark}[theorem]{Remark}
  
 \theoremstyle{remark}
  \newtheorem{example}[theorem]{Example}

\usepackage[OT2,T1]{fontenc}
\newcommand\textcyr[1]{{\fontencoding{OT2}\fontfamily{wncyr}\selectfont #1}}

\tikzset{commutative diagrams/.cd,
mysymbol/.style={start anchor=center,end anchor=center,draw=none}
}

\usepackage{stmaryrd}
\usepackage{hyperref}

\title{\textbf{TATE DUALITY IN POSITIVE DIMENSION OVER FUNCTION FIELDS}} 
\author{Zev Rosengarten}

\begin{document}

\date{}
\maketitle

\begin{abstract}
We extend the classical duality results of Poitou and Tate for finite discrete Galois modules over local and global fields (local duality, nine-term exact sequence, etc.)\,to all affine commutative group schemes of finite type, building on the recent work of \v{C}esnavi\v{c}ius \cite{ces2} extending these results to all finite commutative group schemes. We concentrate mainly on the more difficult function field setting, giving some remarks about the number field case along the way.
\end{abstract}

\tableofcontents{}

\chapter{Introduction and Main Results}

\section{Introduction}\label{app}

Tate's theorems on the Galois cohomology of finite discrete Galois modules over local and global fields
are the cornerstone for the arithmetic theory of abelian varieties, Galois deformation theory, and beyond.
In the function field case, these results traditionally required the hypothesis that the order of the Galois module
is not divisible by the characteristic.  

The restriction on the order was recently removed in work of \v{C}esnavi\v{c}ius  (\cite{ces1}, \cite{ces2}), 
also permitting general (i.e., not necessarily \'etale) finite commutative group schemes
over local and global function fields. This requires using fppf cohomology rather than \'etale cohomology
(the two coincide for coefficients in an \'etale finite commutative group scheme,
or equivalently a finite discrete Galois module; more generally, they agree for coefficients in a smooth commutative group scheme, by \cite[Thm.\,11.7]{briii}). It also requires using
cohomology over the spectrum of an adele ring in the global case in place of various restricted
direct products and direct sums as in the traditional formulation. 

Our aim in the present work is to build on those results in the finite case to establish analogous duality theorems
for the fppf cohomology of commutative affine group schemes $G$ of finite type with {\em nonnegative dimension}
over local and global fields $k$. 
This is much easier in characteristic 0 (as we discuss briefly in Appendix \ref{char0appendix}),
so our main focus in this work is on the function field case, where  
the imperfection of such fields leads to substantial difficulties. Even the case $G=\Ga$ presents  serious challenges 
for duality theorems when ${\rm{char}}(k)>0$, in large part because the $\Gm$-dual functor $\widehat{G} : = \mathscr{H}om(G, \Gm)$ 
over the category of $k$-schemes is non-representable (see Proposition \ref{hatrepresentable}). 

\section{Main results}\label{intro}

Let us now record our main results over local and global fields, giving references for the finite case already treated by 
Poitou, Tate, and \v{C}esnavi\v{c}ius. Before stating these results, however, let us remark that while we mainly restrict our attention to the more difficult function field setting in this work, all of our main results remain true over number fields (and non-archimedean fields of characteristic $0$ for the local results) unless we explicitly state otherwise, and are usually much easier in that setting.

To state the local results, we need some notation.
Given any locally compact Hausdorff topological abelian group $A$, we let $A^D : = \Hom_{\rm{cts}}(A, \mathbf{R}/\Z)$ denote the Pontryagin dual, and we denote by $A_{{\rm{pro}}}$ the profinite completion of $A$. (For the definition of the profinite completion, see the beginning of \S \ref{profiniterightexactsection}.) Consider the class $\mathscr{C}$ of locally compact Hausdorff abelian groups $A$ that are either (i) profinite, (ii) discrete torsion, or (iii) of finite exponent. Many of the groups that we will be concerned with in the local setting 
lie in this class (but there are important exceptions, such as $k^{\times}$ for a non-archimedean
local field $k$). 
The class $\mathscr{C}$ has several nice properties. For example, such $A$ satisfy 
\[
A^D = \Hom_{\rm{cts}}(A, (\Q/\Z)_{\rm{disc}})
\] 
where $(\Q/\Z)_{\rm{disc}}$ denotes $\Q/\Z$ with its discrete topology. Another nice property of 
$\mathscr{C}$ is that it is stable under $A \rightsquigarrow A^D$: if $A$ is profinite then $A^D$ is discrete torsion, and vice-versa, and the dual of a group of finite exponent is still of finite exponent.

\begin{remark}
\label{perfectpairing}
Given two locally compact Hausdorff abelian groups $A, B$ and a continuous bilinear pairing $A \times B \rightarrow \RR/\Z$, the induced map $A \rightarrow B^D$ is a topological isomorphism if and only if the map $B \rightarrow A^D$ is, by Pontryagin double duality \cite[Thm.\,1.7.2]{rudin}. If both of these equivalent conditions hold, then we say that the pairing is {\em perfect}.
\end{remark}

Let $k$ be a local field. Given an fppf abelian sheaf $\mathscr{F}$ on $\Spec(k)$, 
we define $\widehat{\mathscr{F}} : = \mathscr{H}om(\mathscr{F}, \Gm)$
and we give ${{\rm{H}}}^2(k, \mathscr{F})$ the discrete topology. Given a commutative affine $k$-group scheme $G$ of finite type, we also give $\widehat{G}(k)$ the discrete topology. On the other hand, we give $G(k)$ its natural topology arising from that on $k$. That is, we choose a closed immersion $G \hookrightarrow \mathbf{A}^n_k$ into an affine space over $k$, give $\mathbf{A}^n_k(k) = k^n$ its usual topology, and then endow $G(k)$ with the subspace topology. This is independent of the choice of the embedding and makes $G(k)$ into a locally compact Hausdorff group, functorially in $G$. If $G$ is finite, this is simply the discrete topology on $G(k)$. Our first main local result is the following.

\begin{theorem}
\label{localdualityH^2(G)}
$($Theorem $\ref{H^2(G)G^(k)dualityprop}$$)$
Let $k$ be a local field of positive characteristic, and let $G$ be an affine commutative $k$-group scheme of finite type. The cohomology group  ${{\rm{H}}}^2(k, G)$ is torsion, $\widehat{G}(k)$ is finitely generated, and the cup product ${{\rm{H}}}^2(k, G) \times \widehat{G}(k) \rightarrow {{\rm{H}}}^2(k, \mathbf{G}_m) \xrightarrow[\sim]{\inv} \Q/\Z$ induces a functorial continuous perfect pairing of locally compact Hausdorff abelian groups 
\[
{{\rm{H}}}^2(k, G) \times \widehat{G}(k)_{{\rm{pro}}} \rightarrow \Q/\Z,
\]
where ${\rm{H}}^2(k, G)$ is discrete.
\end{theorem}

If $G$ is finite then $\widehat{G}(k) = \widehat{G}(k)_{\pro}$, so this case is \cite[Ch.\,III, Thm.\,6.10]{milne}.
By ``functoriality'' in Theorem \ref{localdualityH^2(G)}, 
 we mean that if we have a homomorphism $f:  G \rightarrow H$ of $k$-group schemes as in 
Theorem \ref{localdualityH^2(G)}, then the diagram
\[
\begin{tikzcd}
{{\rm{H}}}^2(k, G) \arrow{d}[swap]{{{\rm{H}}}^2(f)} \arrow[r, phantom, "\times"] & \widehat{G}(k)_{{\rm{pro}}} \arrow{r} & \Q/\Z \arrow[d, equals] \\
{{\rm{H}}}^2(k, H) \arrow[r, phantom, "\times"] & \widehat{H}(k)_{{\rm{pro}}} \arrow{u}{\widehat{f}(k)} \arrow{r} & \Q/\Z
\end{tikzcd}
\]
commutes. This functoriality follows from the functoriality of cup product.

Our second local result is the analogue of the preceding one with the roles of $G$ and $\widehat{G}$ reversed. Unlike in Tate local duality for finite commutative group schemes, there is no analogous symmetry between $G$ and $\widehat{G}$ in the generality of arbitrary affine commutative group schemes of finite type. In fact, the sheaf $\widehat{G}$ may not be representable (as we have already noted
for $G = \Ga$ due to Proposition \ref{hatrepresentable}); even if representable, it may not be finite type (as for $G = \Gm$, whose fppf dual sheaf is $\Z$). Theorem \ref{localdualityH^2(G)} above
and Theorem \ref{localdualityH^2(G^)} below are therefore distinct results.

\begin{theorem}
\label{localdualityH^2(G^)}
$($Theorem $\ref{H^2(G^)altori}$$)$
Let $k$ be a local field of positive characteristic, and let $G$ be an affine commutative $k$-group scheme of finite type. The cohomology group ${{\rm{H}}}^2(k, \widehat{G})$ is torsion, and the cup product pairing ${{\rm{H}}}^2(k, \widehat{G}) \times {\rm{H}}^0(k, G) \rightarrow {{\rm{H}}}^2(k, \mathbf{G}_m) \xrightarrow[\sim]{\inv} \Q/\Z$ induces a functorial continuous perfect pairing of locally compact Hausdorff groups 
\[
{{\rm{H}}}^2(k, \widehat{G}) \times {\rm{H}}^0(k, G)_{{\rm{pro}}} \rightarrow \Q/\Z,
\]
where ${\rm{H}}^2(k, \widehat{G})$ is discrete.
\end{theorem}

When $G$ is finite, this is the same as Theorem \ref{localdualityH^2(G)}, and therefore is once again \cite[Ch.\,III, Thm.\,6.10]{milne}.

Finally, we come to the duality between the cohomology groups ${{\rm{H}}}^1(k, G)$ and ${{\rm{H}}}^1(k, \widehat{G})$. This is subtle in positive characteristic because it involves endowing ${{\rm{H}}}^1(k, G)$ and ${{\rm{H}}}^1(k, \widehat{G})$ with suitable topologies that are not quite as obvious as the ones for ${{\rm{H}}}^0$ and ${{\rm{H}}}^2$. (In characteristic $0$, both groups are finite, and we simply take them to be discrete.) In \cite{ces1}, \v{C}esnavi\v{c}ius adapts an idea of Moret-Bailly to define a locally compact topology on ${{\rm{H}}}^1(k, G)$ for $G$ a $k$-group scheme locally of finite type. The topology is defined as follows: a subset $U \subset {{\rm{H}}}^1(k, G)$ is {\em open} if for every locally finite type $k$-scheme $X$ and every $G$-torsor $\mathscr{X} \rightarrow X$, the subset
$$\{ x \in X(k) \mid \mathscr{X}_x \in U \} \subset X(k)$$ is open. (This amounts to declaring
 $X(k) \rightarrow (BG)(k)$ to be continuous
for all $k$-morphisms $X \rightarrow BG$.)
In \cite{ces1},  various desirable properties are proved (e.g., given a map $G \rightarrow H$ between two such group schemes, the induced map ${{\rm{H}}}^1(k, G) \rightarrow {{\rm{H}}}^1(k, H)$ is continuous, and likewise
for {\em connecting maps} in the commutative case from degree $i$ to degree $i+1$ for $i=0,1$ in a long exact sequence associated to a short exact sequence of such group schemes); precise references for proofs of these results in \cite{ces1} are given near the start of
\S\ref{sectiontopologyoncohomology}.

This procedure defines a topology on ${{\rm{H}}}^1(k, \widehat{G})$ when $G$ is an almost-torus
(in the sense of Definition \ref{almosttorusdef}), by 
the representability of $\widehat{G}$ in Proposition \ref{hatrepresentable} for such $G$. 
A generalization to define a locally compact Hausdorff 
topology on ${{\rm{H}}}^1(k, \widehat{G})$ for general (affine commutative finite type) 
 $G$ will be given in \S\ref{h1dualgen}.
 
\begin{remark}\label{discontinuous}
Let us warn the reader
that the topology on ${\rm{H}}^1(k, \widehat{(\cdot)})$ is not in general $\delta$-functorial,
in the sense that for a short exact sequence of affine commutative $k$-group schemes of finite type
\begin{equation}\label{gggseq}
1 \longrightarrow G' \longrightarrow G \longrightarrow G'' \longrightarrow 1, 
\end{equation}
we will see later (Proposition \ref{hatisexact}) that the fppf $\Gm$--dual sequence
\[
1 \longrightarrow \widehat{G''} \longrightarrow \widehat{G} \longrightarrow \widehat{G'} \longrightarrow 1
\]
is also exact, hence we obtain a connecting map ${\rm{H}}^1(k, \widehat{G'}) \rightarrow {\rm{H}}^2(k, \widehat{G''})$. When we say that $\delta$-functoriality fails in general for the cohomology of the dual sheaves, we mean that this connecting map is {\em not} necessarily continuous, where (as usual)
${\rm{H}}^2(k, \widehat{G''})$ has the discrete topology. To give an example, consider the exact sequence
\begin{equation}
\label{alpha_pH^1seq1}
1 \longrightarrow \alpha_p \longrightarrow \Ga \longrightarrow \Ga \longrightarrow 1
\end{equation}
where the last map on the right is the Frobenius $k$-isogeny. Since ${\rm{H}}^1(k, \widehat{\Ga}) = 0$ (Proposition \ref{cohomologyofG_adualgeneralk}), the map ${\rm{H}}^1(k, \widehat{\alpha_p}) \rightarrow {\rm{H}}^2(k, \widehat{\Ga})$ is injective. To say that it is continuous, therefore, would be equivalent to saying that ${\rm{H}}^1(k, \widehat{\alpha_p})$ is discrete. But $\widehat{\alpha_p} \simeq \alpha_p$, and we will see (Proposition \ref{topcohombasics}(v)) that the connecting map $k/k^p \rightarrow {\rm{H}}^1(k, \alpha_p)$ coming from (\ref{alpha_pH^1seq1}) is continuous. (It is even a topological isomorphism, but we will not need this.) Therefore, if ${\rm{H}}^1(k, \alpha_p)$ were discrete, then so would be $k/k^p$. But this quotient is not discrete since $k^p \subset k$ is not open, so the connecting map ${\rm{H}}^1(k, \widehat{\alpha_p}) \rightarrow {\rm{H}}^2(k, \widehat{\Ga})$ is {\em not continuous}. This bad behavior cannot happen if all of the 
 groups in (\ref{gggseq}) are almost-tori, since then their fppf dual sheaves are represented by locally finite type $k$-group schemes (Lemma \ref{hatrepresentable}), and so the connecting map is continuous by \cite[Prop. 4.2]{ces1}. We will never use this failure of $\delta$--functoriality.
\end{remark}
 
 With these preliminaries out of the way, the final local duality result is:

\begin{theorem}
\label{H^1localduality}
$($Theorem $\ref{H^1dualityprop}$$)$
Let $k$ be a local field of positive characteristic, and let $G$ be an affine commutative $k$-group scheme of finite type. The 
fppf cohomology groups ${{\rm{H}}}^1(k, G)$ and ${{\rm{H}}}^1(k, \widehat{G})$ are of finite exponent, and cup product 
\[
{{\rm{H}}}^1(k, G) \times {{\rm{H}}}^1(k, \widehat{G}) \rightarrow {{\rm{H}}}^2(k, \mathbf{G}_m) \xrightarrow[\sim]{\inv} \Q/\Z
\]
 is a functorial 
continuous perfect pairing of locally compact Hausdorff abelian groups.
\end{theorem}

When $G$ is finite, this is once again \cite[Ch.\,III, Thm.\,6.10]{milne}. We remark that the topology we use on the cohomology groups agrees with Milne's \v{C}ech topology, as will be discussed in \S \ref{sectiontopologyoncohomology}.

All three of the preceding theorems are proved in Chapter \ref{chapterlocalfields}. They will be deduced from the already-known finite cases.

Part of classical local duality describes integral (or in Galois-theoretic contexts, unramified) cohomology classes on each side as the exact annihilators of those on the other side. More precisely, let $k$ be a local function field with ring of integers $\mathcal{O}$, and let $\mathscr{G}$ be a finite flat
commutative $\mathcal{O}$-group scheme with generic fiber $G$. Then the maps ${\rm{H}}^i(\mathcal{O}, \mathscr{G}) \rightarrow {\rm{H}}^i(k, G)$ are injective, and similarly for $\widehat{\mathscr{G}}$ by Cartier duality, and via the perfect cup product pairing ${\rm{H}}^i(k, G) \times {\rm{H}}^{2-i}(k, \widehat{G}) \rightarrow \Q/\Z$, the exact annihilator of ${\rm{H}}^{2-i}(\mathcal{O}, \widehat{\mathscr{G}})$
is ${\rm{H}}^i(\mathcal{O}, \mathscr{G})$ (these groups clearly annihilate each other because ${\rm{H}}^2(\mathcal{O}, \Gm) = 0$):
for $i = 1$ this is \cite[Ch.\,III, Cor.\,7.2]{milne}, and for $i = 2$ it follows from \cite[Cor.\,2.9]{ces2} and the fact that $\widehat{\mathscr{G}}(\calO) = \widehat{G}(k)$ due to $\widehat{\mathscr{G}}$ being finite. The case $i = 0$ follows from the case $i=2$ by Cartier duality.
 
Unfortunately, due to the lack of a good structure theory for flat affine commutative $\mathcal{O}$-group schemes of finite type, we do not obtain such a satisfactory result in general when $\dim G > 0$. However, we do obtain useful replacements 
that are not only interesting in their own right but also play an essential role in our work in \S\ref{sectionadeliclocalcohom}
that relates cohomology over adele rings to cohomology over local fields. Part of our interest in these theorems lies in the injectivity statements below for the pullback maps ${\rm{H}}^i(\mathcal{O}_v, \cdot) \rightarrow {\rm{H}}^i(k_v, \cdot)$
for {\em all but finitely many} $v$, as may be proved by a direct d{\'e}vissage for number fields but appears to be much more subtle in the function field setting.

To state these local duality results when $\dim G>0$, let us first note that any affine commutative group scheme $G$ of finite type over a global field $k$ spreads out to an {\em $\mathcal{O}_S$-model} $\mathscr{G}$ 
 (i.e., an affine commutative flat $\mathcal{O}_S$-group scheme of finite type) for some non-empty finite set $S$ of places of $k$ containing the archimedean places. 

\begin{theorem}
\label{exactannihilatorH^2(G)}
$($Theorem $\ref{H^2(G)G^(k)integraldualityprop}$$)$
Let $k$ be a global function field, $G$ an affine commutative $k$-group scheme of finite type, and $\mathscr{G}$ an $\mathcal{O}_S$-model of $G$. Then for all but finitely many places $v$ of $k$, we have ${\rm{H}}^2(\mathcal{O}_v, \mathscr{G}) = 0$ and the map $\widehat{\mathscr{G}}(\mathcal{O}_v) \rightarrow \widehat{G}(k_v)$ is an isomorphism. 

In particular, 
for such $v$ the maps ${\rm{H}}^2(\mathcal{O}_v, \mathscr{G}) \rightarrow {\rm{H}}^2(k_v, G)$ and ${\rm{H}}^0(\mathcal{O}_v, \widehat{\mathscr{G}}) \rightarrow {\rm{H}}^0(k_v, \widehat{G})$ are injective, and ${\rm{H}}^2(\mathcal{O}_v, \mathscr{G})$ is the exact annihilator of ${\rm{H}}^0(\mathcal{O}_v, \widehat{\mathscr{G}})$.
\end{theorem}

\begin{theorem}
\label{exactannihilatorH^2(G^)}
$($Theorem $\ref{H^2(O,G^)dualityprop}$$)$
Let $k$ be a global function field, $G$ an affine commutative $k$-group scheme of finite type, and $\mathscr{G}$ an $\mathcal{O}_S$-model of $G$. Then for all but finitely many places $v$ of $k$, the maps ${\rm{H}}^0(\mathcal{O}_v, \mathscr{G}) \rightarrow {\rm{H}}^0(k_v, G)$ and ${\rm{H}}^2(\mathcal{O}_v, \widehat{\mathscr{G}}) \rightarrow {\rm{H}}^2(k_v, \widehat{G})$ are injective
and ${\rm{H}}^2(\mathcal{O}_v, \widehat{\mathscr{G}})$ is the exact annihilator of ${\rm{H}}^0(\mathcal{O}_v, \mathscr{G})$. 
\end{theorem}

\begin{theorem}
\label{exactannihilatorH^1(G)}
$($Theorem $\ref{H^1(k,G^)/H^1(O,G^)}$$)$
Let $k$ be a global function field, $G$ an affine commutative $k$-group scheme of finite type, and $\mathscr{G}$ an $\mathcal{O}_S$-model of $G$. Then for all but finitely many places $v$ of $k$, the maps ${\rm{H}}^1(\mathcal{O}_v, \mathscr{G}) \rightarrow {\rm{H}}^1(k_v, G)$ and ${\rm{H}}^1(\mathcal{O}_v, \widehat{\mathscr{G}}) \rightarrow {\rm{H}}^1(k_v, \widehat{G})$ are injective
and ${\rm{H}}^1(\mathcal{O}_v, \mathscr{G})$ and ${\rm{H}}^1(\mathcal{O}_v, \widehat{\mathscr{G}})$ are orthogonal complements under the local duality $($i.e., cup product$)$ pairing.
\end{theorem}

Note that these results are independent of the chosen $\mathcal{O}_S$-model $\mathscr{G}$, since any two such models become isomorphic over $\mathcal{O}_{S'}$ for some $S' \supset S$.  Theorems \ref{exactannihilatorH^2(G)}--\ref{exactannihilatorH^1(G)} are all proved in Chapter \ref{chapterlocalintcohom}. 

\medskip

Now we turn to our main results in the global setting.   First, we introduce some notation.  
 Let $k$ be a global function field. We denote by $\A_k$ (or $\A$ when there is no confusion as to which field we are 
 working over) the ring of adeles for $k$.
For an abelian group $A$, we denote by $A^*$ the group $\Hom(A, \Q/\Z)$. The functor $A \rightsquigarrow A^*$
is exact precisely because $\Q/\Z$ is an injective abelian group.  We aim to prove the following result, which extends the classical Poitou--Tate exact sequence for finite commutative group schemes to all affine commutative group schemes of finite type:

\begin{theorem}
\label{poitoutatesequence}
$($Theorem $\ref{poitoutatesequenceprop}$$)$
Let $k$ be a global function field, and $G$ an affine commutative $k$-group scheme of finite type. The following sequence 
$($with maps to be defined below$)$ is exact and functorial in $G$:
\[
\begin{tikzcd}
0 \arrow{r} & {\rm{H}}^0(k, G)_{\pro} \arrow{r} & {\rm{H}}^0(\A, G)_{\pro} \arrow{r} \arrow[d, phantom, ""{coordinate, name=Z_1}] & {\rm{H}}^2(k, \widehat{G})^* \arrow[dll, rounded corners,
to path={ -- ([xshift=2ex]\tikztostart.east)
|- (Z_1) [near end]\tikztonodes
-| ([xshift=-2ex]\tikztotarget.west) -- (\tikztotarget)}] & \\
& {\rm{H}}^1(k, G) \arrow{r} & {\rm{H}}^1(\A, G) \arrow{r} \arrow[d, phantom, ""{coordinate, name=Z_2}] & {\rm{H}}^1(k, \widehat{G})^* \arrow[dll, rounded corners,
to path={ -- ([xshift=2ex]\tikztostart.east)
|- (Z_2) [near end]\tikztonodes
-| ([xshift=-2ex]\tikztotarget.west) -- (\tikztotarget)}] & \\
& {\rm{H}}^2(k, G) \arrow{r} & {\rm{H}}^2(\A, G) \arrow{r} & {\rm{H}}^0(k, \widehat{G})^* \arrow{r} & 0
\end{tikzcd}
\]
\end{theorem}

Let us now describe the maps in this sequence. We endow ${\rm{H}}^0(k, G)$ with the discrete topology and ${\rm{H}}^0(\A, G)$ with the natural topology
arising from that on $\A$ (via a closed embedding from $G$ into some affine space). Then the maps ${\rm{H}}^i(k, G) \rightarrow {\rm{H}}^i(\A, G)$ -- as well as the induced map on profinite completions when $i = 0$ -- are induced by the diagonal inclusion $k \hookrightarrow \A$ sending $\lambda \in k$ to the adele with each coordinate equal to $\lambda$.

The maps ${\rm{H}}^i(\A, G) \rightarrow {\rm{H}}^{2-i}(k, \widehat{G})^*$ are obtained by cupping everywhere locally and then adding the invariants. That is, for each place $v$ of $k$ we have the natural projection map ${\rm{H}}^i(\A, G) \rightarrow {\rm{H}}^i(k_v, G)$, and we also have the map ${\rm{H}}^{2-i}(k, \widehat{G}) \rightarrow {\rm{H}}^{2-i}(k_v, \widehat{G})$. Given $\alpha \in {\rm{H}}^i(\A, G)$ and $\beta \in {\rm{H}}^{2-i}(k, \widehat{G})$, therefore, we obtain for each $v$ the cup product $\alpha_v \cup \beta_v \in {\rm{H}}^2(k_v, \Gm)$, hence by taking the invariant an element of $\Q/\Z$. The pairing between $\alpha$ and $\beta$ is obtained by summing these invariants over all places $v$ of $k$. Of course, one must show that this sum has only finitely many nonzero terms (and that it induces a map ${\rm{H}}^0(\A, G)_{\rm{pro}} \rightarrow {\rm{H}}^2(k, \widehat{G})^*$). All of this will be shown in \S \ref{globalsectionpreliminaries}.

The maps ${\rm{H}}^2(k, \widehat{G})^* \rightarrow {\rm{H}}^1(k, G)$ and ${\rm{H}}^1(k, \widehat{G})^* \rightarrow {\rm{H}}^2(k, G)$ in Theorem \ref{poitoutatesequence} are the hardest ones to describe. In order to explain how they arise, we need to introduce the Tate-Shafarevich groups of a sheaf.

For any
abelian fppf sheaf $\mathscr{F}$ on the category of all $k$-schemes, we define the $i$th {\em Tate-Shafarevich group} of $\mathscr{F}$ -- denoted $\Sha^i(k, \mathscr{F})$ -- by the formula
\[
\Sha^i(k, \mathscr{F}) : = \ker \left({\rm{H}}^i(k, \mathscr{F}) \rightarrow {\rm{H}}^i(\A, \mathscr{F})\right)
\]
(cohomology for the small fppf sites on $k$ and $\A$ respectively). 
When there is no confusion, we will often simply write $\Sha^i(\mathscr{F})$. This agrees with the more classical definition
\[
\Sha^i(k, \mathscr{F}) : = \ker \left({\rm{H}}^i(k, \mathscr{F}) \rightarrow \prod_v {\rm{H}}^i(k_v, \mathscr{F})\right)
\]
(where the product is over all places $v$ of $k$) if $i \leq 2$ and $\mathscr{F} = G$ or $\widehat{G}$ for some affine commutative $k$-group 
scheme $G$ of finite type, by \cite[Th.\,2.13]{ces2} and Proposition \ref{H^i(A,G^)--->prodH^i(k_v,G^)}.

The maps ${\rm{H}}^2(k, \widehat{G})^* \rightarrow {\rm{H}}^1(k, G)$ and ${\rm{H}}^1(k, \widehat{G})^* \rightarrow {\rm{H}}^2(k, G)$ in Theorem \ref{poitoutatesequence} rest on the following result that for finite $G$ is an analogue of 
a classical result of Tate for finite Galois modules over global fields:

\begin{theorem}
\label{shapairing}
$($Theorem $\ref{Shapairingprop}$$)$
Let $k$ be a global function field, $G$ an affine commutative $k$-group scheme of finite type. Then the groups $\Sha^i(k, G)$ and $\Sha^i(k, \widehat{G})$ are finite for $i = 1,2$. Further, for $i = 1,2$, we have perfect pairings, functorial in $G$, 
\[
\Sha^i(k, G) \times \Sha^{3-i}(k, \widehat{G}) \rightarrow \Q/\Z
\]
\end{theorem} 

Theorem \ref{shapairing} furnishes an isomorphism $$\Sha^2(k, \widehat{G})^* \xrightarrow{\sim} \Sha^1(k, G)$$ and the map ${\rm{H}}^2(k, \widehat{G})^* \rightarrow {\rm{H}}^1(k, G)$ in Theorem \ref{poitoutatesequence} is
{\em defined} to be the composition $${\rm{H}}^2(k, \widehat{G})^* \twoheadrightarrow \Sha^2(k, \widehat{G})^* \xrightarrow{\sim} \Sha^1(k, G) \hookrightarrow {\rm{H}}^1(k, G).$$ Similarly, the map ${\rm{H}}^1(k, \widehat{G})^* \rightarrow {\rm{H}}^2(k, G)$ is 
{\em defined} to be the composition defined by the following chain of maps: ${\rm{H}}^1(k, \widehat{G})^* \twoheadrightarrow \Sha^1(k, \widehat{G})^* \xrightarrow{\sim} \Sha^2(k, G) \hookrightarrow {\rm{H}}^2(k, G)$. 

Proving Theorems \ref{poitoutatesequence} and \ref{shapairing} will be the main work of Chapter \ref{globalfieldschap}, building on results in the finite case from \cite[\S1.2]{ces2}. 

In the finite case, there is a natural symmetry between $G$ and $\widehat{G}$ due to Cartier duality and double duality. For positive-dimensional groups, there is no symmetry in the hypotheses between $G$ and $\widehat{G}$, since the latter sheaf need not be finite type or even representable for general affine commutative $G$ of finite type over $k$ (cf.\,Proposition \ref{hatrepresentable}). Nevertheless, we still have double duality for such $G$ (Proposition \ref{doubleduality}), and the statements of the local results (taken together) are unchanged if we switch the roles of $G$ and $\widehat{G}$; the same goes for Theorem \ref{shapairing}. There should therefore be an analogue of Theorem \ref{poitoutatesequence} with the roles of $G$ and $\widehat{G}$ reversed. This is indeed the case:

\begin{theorem}
\label{poitoutatesequencedual}
$($Theorem $\ref{poitoutatedualsequenceprop}$$)$
Let $G$ be an affine commutative group scheme of finite type over a global function field $k$. Then the sequence
\[
\begin{tikzcd}
0 \arrow{r} & {\rm{H}}^0(k, \widehat{G})_{\pro} \arrow{r} & {\rm{H}}^0(\A, \widehat{G})_{\pro} \arrow{r} \arrow[d, phantom, ""{coordinate, name=Z_1}] & {\rm{H}}^2(k, G)^* \arrow[dll, rounded corners,
to path={ -- ([xshift=2ex]\tikztostart.east)
|- (Z_1) [near end]\tikztonodes
-| ([xshift=-2ex]\tikztotarget.west) -- (\tikztotarget)}] & \\
& {\rm{H}}^1(k, \widehat{G}) \arrow{r} & {\rm{H}}^1(\A, \widehat{G}) \arrow{r} \arrow[d, phantom, ""{coordinate, name=Z_2}] & {\rm{H}}^1(k, G)^* \arrow[dll, rounded corners,
to path={ -- ([xshift=2ex]\tikztostart.east)
|- (Z_2) [near end]\tikztonodes
-| ([xshift=-2ex]\tikztotarget.west) -- (\tikztotarget)}] & \\
& {\rm{H}}^2(k, \widehat{G}) \arrow{r} & {\rm{H}}^2(\A, \widehat{G}) \arrow{r} & ({\rm{H}}^0(k, G)_{\rm{pro}})^D \arrow{r} & 0
\end{tikzcd}
\]
is exact and functorial in $G$.
\end{theorem}

The maps are entirely analogous to those in Theorem \ref{poitoutatesequence}, and as usual, all cohomology groups 
${\rm{H}}^i(k, \cdot)$ are equipped with the discrete topology. 
The topology on ${\rm{H}}^0(\A, \widehat{G})$ is defined to be the restricted product topology coming from some $\calO_S$-model $\mathscr{G}$ of $G$ (where $S$ is a non-empty finite set of places as usual) and the discrete topology on 
each group $\widehat{G}(k_v)$ of local characters. (The choice of the discrete topology on such character groups
is appropriate, since they are finitely generated.) 
By Theorem \ref{exactannihilatorH^2(G)}, this identifies $\widehat{G}(\A)$ with $\prod_v \widehat{G}(k_v)$ endowed with the product topology. We will deduce Theorem \ref{poitoutatesequencedual} as the Pontryagin dual of Theorem \ref{poitoutatesequence} by applying local duality.

\begin{remark}
In classical Tate duality for finite discrete Galois modules over a global field $k$ (i.e., finite \'etale $k$-group schemes) with order prime to ${\rm{char}}(k)$, the role of the group ${\rm{H}}^0(\A, G)_{\pro}$ is usually taken by the product $\prod_v G(k_v)$ of the local groups, that of ${\rm{H}}^1(\A, G)$ is taken by the restricted product of the local cohomology groups ${\rm{H}}^1(k_v, G)$ with respect to unramified cohomology groups, and the role of the group ${\rm{H}}^2(\A, G)$ is taken by the direct sum of the local cohomology groups. In fact, these statements are identical to the ones given above when $G$ is finite \'etale. Indeed, $G(\A) = \prod_v G(k_v)$ thanks to the valuative criterion of properness applied to a finite flat $\calO_S$-model $\mathscr{G}$ of $G$. Since $\prod_v G(k_v)$ is profinite (being compact and totally disconnected), it follows that $G(\A)_{\pro} = \prod_v G(k_v)$. For the group ${\rm{H}}^1(\A, G)$, \cite[Th.\,2.18]{ces2} implies that it is the restricted product of the ${\rm{H}}^1(k_v, G)$ with respect to the integral cohomology groups ${\rm{H}}^1(\calO_v, \mathscr{G})$, and these integral groups are the same as the unramified cohomology groups by standard results on the \'etale cohomology of discrete valuation rings. Finally, ${\rm{H}}^2(\A, G)$ agrees with the direct sum of the local cohomology groups, again by \cite[Th.\,2.18]{ces2}.

For the relation between the adelic cohomology groups ${\rm{H}}^i(\A, \widehat{G})$ (and ${\rm{H}}^0(\A, \widehat{G})_{\pro}$) appearing in Theorem \ref{poitoutatesequencedual} and the corresponding local cohomology groups for general (not necessarily finite) $G$, see Propositions \ref{H^i(A,G^)--->prodH^i(k_v,G^)} and \ref{G(A)_pro=prodG(k_v)_pro}.
\end{remark}

\begin{remark}\label{prior}
Let us be very precise about which aspects of Tate local and global duality for finite commutative group schemes we are taking as input from prior
work of others. First, the local results in Theorems \ref{localdualityH^2(G)}--\ref{H^1localduality} and \ref{exactannihilatorH^1(G)} for finite $G$ are already known:
references for these results in the finite case were given along with their statements above. As for the global results, consider the three sequences obtained by deleting the maps from ${\rm{H}}^i(k, \widehat{G})^*$, $i=1,2$:
\[
0 \longrightarrow G(k) \longrightarrow G(\A) \longrightarrow {\rm{H}}^2(k, \widehat{G})^*
\]
\[
{\rm{H}}^1(k, G) \longrightarrow {\rm{H}}^1(\A, G) \longrightarrow {\rm{H}}^1(k, \widehat{G})^*
\]
\[
{\rm{H}}^2(k, G) \longrightarrow {\rm{H}}^2(\A, G) \longrightarrow G(k)^* \longrightarrow 0
\]
The exactness of these sequences for finite commutative $G$ holds by \cite[\S1.2]{ces2}. 

We do {\em not} assume Theorem \ref{shapairing} for finite group schemes, even though this case has in fact been fully proved in prior work of
others; we will prove this case directly. The reason that we do not treat the exactness of the full 9-term exact sequence 
as known at the outset for finite $G$ is that \v{C}esnavi\v{c}ius defines the maps from ${\rm{H}}^i(k, \widehat{G})^*$ in a very different manner from the way we do, and then deduces Theorem \ref{shapairing} from the exactness of the resulting sequence. Our approach proceeds in the opposite order. 
To avoid verifying the compatibility between our sequence and \v{C}esnavi\v{c}ius', and between our pairings in Theorem \ref{shapairing} and his, we have chosen to simply prove Theorem \ref{shapairing} directly, even in the finite case. 

Throughout this work, whenever we refer to ``local duality'' or ``Poitou--Tate'' for {\em finite} commutative group schemes, we will mean the results that have been summarized in this remark.
\end{remark}

\section{Overview of methods}

To orient the reader, we now provide a brief overview of the manuscript, and indicate how the various parts fit together to establish the main results.

In Chapter \ref{chaptergeneralfields}, we establish some crucial results about affine commutative group schemes $G$ of finite type over general fields, including their structure theory (\S \ref{sectionstructure}), their Ext and dual sheaves (\S\S \ref{charactersheavessection}--\ref{doubledualitysection}), the cohomology of $\widehat{\Ga}$ (\S\S \ref{cohomGahat}--\ref{sectionH^3(Ga^)=0}), and the relation between \v{C}ech and derived functor cohomology of $G$ and $\widehat{G}$ (\S \ref{cech=derivedsection}), which will play an essential role in proving the continuity of the local duality pairings, as well as in defining the $\Sha$-pairings of Theorem \ref{shapairing}.

In Chapter \ref{chapterlocalfields} we prove the main local duality results (Theorems \ref{localdualityH^2(G)}, \ref{localdualityH^2(G^)}, and \ref{H^1localduality}). In Chapter \ref{chapterlocalintcohom}, we establish the main results on local integral cohomology (Theorems \ref{exactannihilatorH^2(G)}, \ref{exactannihilatorH^2(G^)}, and \ref{exactannihilatorH^1(G)}). Finally, in Chapter \ref{globalfieldschap}, we establish the global Tate duality results (Theorems \ref{poitoutatesequence}--\ref{poitoutatesequencedual}), essentially analyzing the sequence bit by bit (as well as proving perfection of the $\Sha$-pairings). An interesting feature of the proof of the $9$-term exact sequence for positive-dimensional groups is that if we grant that the entire sequence is exact for finite commutative group schemes (cf.\,Remark \ref{prior}) then the different parts of the sequence may be shown to be exact essentially independently, in stark contrast to Tate's original proof of his duality results.

It is difficult to briefly explain the main ideas of the proofs of the preceding theorems 
except to say that the arguments typically proceed by a long d\'evissage from the cases of finite commutative group schemes, $\Gm$, and $\Ga$. This d\'evissage typically (though not always) proceeds as follows. One uses Lemma \ref{affinegroupstructurethm} to reduce to the case when $G$ is split unipotent or an almost-torus (see Definition \ref{almosttorusdef}). The split unipotent case reduces to that of $\Ga$, which is often delicate due to the non-representability of $\widehat{\Ga}$. The case of almost-tori is related to the cases of finite commutative group schemes and 
the case of $\Gm$ by means of Lemma \ref{almosttorus}(iv).

The case of finite $G$ in the preceding theorems will usually amount to known results of others (though see Remark \ref{prior}), while the main results for $\Gm$ 
typically amount to the main results of local and global class field theory. Finally, the main results for $\Ga$ rest on our analysis of the cohomology of $\widehat{\Ga}$ in Chapter \ref{chaptergeneralfields}.

Let us also briefly explain how we analyze the group ${\rm{H}}^2(k, \widehat{\Ga})$
when $p = {\rm{char}}(k)>0$, since working with the non-representable sheaf $\widehat{\Ga}$ is one of the major features of Tate duality in positive dimension that does not appear for finite commutative group schemes, and (unlike $\Gm$) results for $\Ga$ do not simply emerge from class field theory. Proposition \ref{H^2=Ext^2=Br} relates the $p$-torsion group ${\rm{H}}^2(k, \widehat{\Ga})$ to 
$p$-torsion in the Brauer group of $\mathbf{G}_{a,\,k}$. In \S \ref{sectionbrauergpsdiff}, we adapt an idea of Kato to relate such 
$p$-torsion Brauer classes to differential forms on $\Ga$ (Proposition \ref{brauerdifferentialforms}). Finally, we utilize this relationship 
in \S\ref{sectiongeneralH^2(Ga^)} to explicitly compute the group ${\rm{H}}^2(k, \widehat{\Ga})$ for all fields $k$ of characteristic $p > 0$
 such that $[k:  k^p] = p$ (Proposition \ref{omega=H^2(Ga^)}), a class of fields that includes local and global function fields.

\section{Acknowledgements}

I would like to thank Ofer Gabber and Hendrik Lenstra for pointing the author to the connection between ultraproducts and local rings of product rings discussed in appendix \ref{ultraproductsappendix}. I am also very happy to thank the anonymous referee for his tremendous effort and many helpful suggestions while going over this manuscript, which significantly improved both its presentation and its mathematical quality. Finally, it is a pleasure to thank my advisor, Brian Conrad, for providing helpful guidance throughout my work on this book. Those familiar with his mathematical style will discern his influence throughout these pages.

\newpage

\section{Notation and terminology}

Throughout this work, $k$ denotes a field and $k_s, \overline{k}$ denote separable and algebraic closures of $k$, respectively. Also, $p$ denotes a prime, equal to ${\rm{char}}(k)$ when ${\rm{char}}(k) > 0$. All cohomology is fppf unless stated otherwise. 

If $k$ is a non-archimedean local field, then $\calO_k$ denotes the integer ring of $k$. 

If $k$ is a global field (i.e., a number field or the function field of a smooth proper geometrically connected curve over a finite field), and $v$ is a place of $k$, then $k_v$ denotes the local field associated to $v$ -- that is, the completion of $k$ with respect to $v$ -- and $\calO_v$ denotes the integer ring of $k_v$. The symbol $\A_k$, also denoted $\A$ when the field $k$ is clear from the context, denotes the topological ring of adeles of $k$; that is, $\A : = \prod'_v k_v$, the restricted product over all places $v$ of $k$ of the fields $k_v$ with respect to the subrings $\calO_v \subset k_v$. If $S$ is a set of places of $k$, then $\A^S : = \prod'_{v \notin S} k_v$ denotes the ring of $S$-adeles, and $\calO_S \subset k$ denotes the ring of $S$-integers of $k$; that is, $\calO_S : = \{ \lambda \in k \mid \lambda \in \calO_v \mbox{ for all non-archimedean } v \notin S\}$.

We will frequently use without comment the fact that, for smooth group schemes $G$ over a scheme $X$, ${{\rm{H}}}^i_{\et}(X, G) = {{\rm{H}}}^i_{\fppf}(X, G)$ when $i = 0, 1$, and for all $i$ if we also assume that $G$ is commutative. For $i = 0$, this is obvious. For $i = 1$, it follows from the fact that any fppf $G$-torsor over $X$ has a section {\'e}tale-locally. The agreement for all $i$ in the commutative case follows from \cite[Thm.\,11.7]{briii}.

For an abelian sheaf $\mathscr{F}$ on the big fppf site on the category of schemes
over a base scheme, the functor $\calHom(\mathscr{F}, \mathbf{G}_m)$ is denoted by $\widehat{\mathscr{F}}$.

For an abelian group $A$, we let $A^* : = {\rm{Hom}}(A, \Q/\Z)$, and if $A$ is a topological abelian group, then $A^D : = {\rm{Hom}}_{\rm{cts}}(A, \mathbf{R}/\Z)$ denotes the Pontryagin dual group to $A$, which is also a topological abelian group. We denote by $A_{\Div} : = \bigcap_{n \in \Z_+} nA$ the subgroup of divisible elements, and by $A_{\tors} \subset A$ the subgroup of torsion elements.

An {\em isogeny} $f:  G \rightarrow H$ between finite type group schemes over a field is a finite flat surjective homomorphism.
A {\em split unipotent} group over a field $k$ is a $k$-group admitting a finite composition
series of (smooth connected) $k$-subgroups such that the successive quotients are each $k$-isomorphic to
$\Ga$.

When we are working with the set of places of a global field, we say ``almost every'' to mean ``all but finitely many''.

If $G$ is a group scheme of finite type over a global field $k$ and $S$ is a non-empty finite set of places of $k$ containing
the archimedean places, then an {\em $\mathcal{O}_S$-model} of $G$ is a flat separated finite type $\mathcal{O}_S$-group scheme
$\mathscr{G}$ equipped with a $k$-group scheme isomorphism $\mathscr{G}_k \simeq G$.  It is well known
that for sufficiently large $S$ (depending on $G$) such an $\mathcal{O}_S$-model exists, that any two become
isomorphic over $\mathcal{O}_{S'}$ for some $S' \supset S$, and that if $G$ is affine then $\mathscr{G}_{\mathcal{O}_{S'}}$
is affine for sufficiently large $S' \supset S$.  By flatness and separatedness, if $G$ is commutative then $\mathscr{G}$ is commutative. 
We will nearly always work with affine $G$, though occasionally we will permit more general $G$ when a proof
does not require affineness.  Consequently, we adopt the convention that {\em if $G$ is affine, it is always understood
that we only use $\mathscr{G}$ that are affine}.

\chapter{General fields}
\label{chaptergeneralfields}

In this chapter we prove various results about affine commutative group schemes of finite type over general fields, before turning to the special cases of local and global function fields in later chapters. We begin by discussing the structure of such groups (\S \ref{sectionstructure}). We then turn in \S \ref{charactersheavessection} to a discussion of the (fppf or \'etale) sheaf $\calExt^1_k(G, \Gm)$ for affine commutative $k$-group schemes $G$ of finite type, the main result being the vanishing of this sheaf in the fppf topology in characteristic $p$ (Proposition \ref{ext=0}), and the same for the \'etale topology in characteristic $0$, or more generally for perfect fields (Proposition \ref{ext0}). In \S \ref{applicationsofext=0section}, we give several applications of this vanishing, the most significant of which is the exactness of the functor $G \mapsto \widehat{G}$ for affine commutative $G$ of finite type over $k$ (in the \'etale topology in characteristic $0$ and in the fppf topology in characteristic $p$; Proposition \ref{hatisexact}). This exactness plays a crucial role throughout this work, as it allows us to carry out various d\'evissage arguments involving exact sequences of affine groups. In the spirit of treating the sheaves $G$ and $\widehat{G}$ symmetrically, in \S \ref{doubledualitysection} we prove that the natural map $G \rightarrow G^{\wedge\wedge}$ is an isomorphism of fppf sheaves (Proposition \ref{doubleduality}). 

We next turn to a study of the cohomology of the sheaf $\widehat{\Ga}$, which plays a crucial role in proving our duality theorems for the additive group $\Ga$, a fundamental building block for general affine groups. We begin to undertake this study in \S \ref{cohomGahat}, and we relate the most interesting cohomology group, ${\rm{H}}^2(k, \widehat{\Ga})$, to the (primitive) Brauer group of $\Ga$ (we actually undertake this study for somewhat more general rings than fields; see Propositions \ref{H^2=Ext^2=Brdvr} and \ref{H^2=Ext^2=Br}). We then study these Brauer groups by relating them to differential forms in \S \ref{sectionbrauergpsdiff} (see Proposition \ref{brauerdifferentialforms}), and we use this relationship to explicitly compute ${\rm{H}}^2(k, \widehat{\Ga})$ in \S \ref{sectiongeneralH^2(Ga^)} (see Proposition \ref{omega=H^2(Ga^)} and Proposition \ref{omega=H^2(Ga^)}). We then complete our study of the cohomology of $\widehat{\Ga}$ by showing in \S \ref{sectionH^3(Ga^)=0} that the third cohomology group ${\rm{H}}^3(k, \widehat{\Ga})$ vanishes (Proposition \ref{H^3(G_a^)=0}). The chapter concludes with \S \ref{cech=derivedsection} by showing that the \v{C}ech and derived cohomology groups agree for the sheaves of interest in this work. This agreement will play a crucial role in Chapter \ref{chapterlocalfields} in proving the continuity of the local duality pairings, as well as in defining the pairings of Theorem \ref{shapairing} in \S \ref{sectiondefiningshapairings}.

\section{Structure of affine commutative group schemes of finite type}
\label{sectionstructure}

In this section we prove some results that we shall require on the structure of affine commutative group schemes of finite type over a general field. 

\begin{lemma}
\label{finitequotient=smoothandconnected}
Let $k$ be a field, $G$ a commutative $k$-group scheme of finite type. Then there is a finite 
$k$-subgroup scheme $E \subset G$ such that $G/E$ is smooth and connected.
\end{lemma}

\begin{proof}
By \cite[VII$_{\rm{A}}$, Prop.\,8.3]{sga3}, there is an infinitesimal subgroup scheme $I \subset G$ such that $G/I$ is smooth. Replacing $G$ with $G/I$, therefore, we may assume that $G$ is smooth. 
We next reduce to the case when $k$ is perfect (so $\overline{k}/k$ is Galois).  Without loss of generality we may
assume ${\rm{char}}(k)=p>0$.  Let $F: G \rightarrow G^{(p)}$ be the relative Frobenius homomorphism; this is an isogeny
since $G$ is smooth.  If $E' \subset G^{(p)}$ is a finite $k$-subgroup scheme
such that the smooth $k$-group $G^{(p)}/E'$ is connected then $E : = F^{-1}(E')$ is finite
and clearly the smooth $k$-group $G/E$ is connected (as $G/E \simeq G^{(p)}/E'$).  Hence, we may replace
$G$ with $G^{(p)}$, or more generally with $G^{(p^n)}$ for any desired $n \ge 0$.  

If the problem can be solved over the perfect closure of $k$ then by standard limit considerations
it is solved over some purely inseparable finite extension of $k$. Any such extension is contained in
$k^{p^{-n}}$ for some large $n$, and the field extension $k \hookrightarrow k^{p^{-n}}$ 
is identified with the $p^n$-power map $k \rightarrow k$.  Hence, passing to $G^{(p^n)}$ would then do the job.
Thus, we may and do assume that $k$ is perfect (of any characteristic).

We claim that each component of $G$ contains a torsion point of $G(\overline{k})$. Assuming this, we can choose one such point in each component of $G$, and the subgroup of $G(\overline{k})$ generated by their Galois-orbits
is ${\rm{Gal}}(\overline{k}/k)$-stable and hence descends to a finite $k$-subgroup
$E \subset G$ with the desired property.

To prove the claim, we may assume that $k = \overline{k}$. Let $N$ be the exponent of $G/G^0$. Consider the smooth $k$-subgroups $[N^r]G \subset G$. These form a descending chain, and for $r > 0$ they are contained in $G^0$. This chain must stabilize; that is, $[N^r]G = [N^{r+1}]G$ for some $r > 0$. Let $H : = [N^r]G$ be this stabilized group. Then $H \subset G^0$ is $N$-divisible.

Now let $X$ be a component of $G$. We want to find a torsion point in $X(k)$. Choose an $x \in X(k)$. Then $N^rx = N^rh$ for some $h \in H(k) \subset G^0$, since $H$ is $N$-divisible and $k = \overline{k}$. Since $h \in G^0$, we have $x - h \in X(k)$ and by construction, $N^r(x - h) = 0$. Thus, $x - h \in X(k)$ is a torsion point.
\end{proof}

Next we will introduce the notion of an almost-torus.

\begin{lemma}
\label{almosttorus}
Let $k$ be a field, $G$ a commutative affine $k$-group scheme of finite type. The following are equivalent: \begin{itemize}
\item[(i)]  $(G_{\overline{k}})^0_{\red}$ is a torus.
\item[(ii)] There is a $k$-torus $T \subset G$ such that $G/T$ is finite.
\item[(iii)] There exist a $k$-torus $T$, a finite $k$-group scheme $A$, and an isogeny $A \times T \twoheadrightarrow G$.
\item[(iv)] There exist a positive integer $n$, a finite commutative $k$-group scheme $A$, finite separable extensions $k_1, k_2/k$, split $k_i$-tori $T_i$, and an isogeny $A \times \R_{k_1/k}(T_1) \twoheadrightarrow G^n \times \R_{k_2/k}(T_2)$.
\end{itemize}
\end{lemma}

\begin{proof}
(i) $\Longrightarrow$ (ii):  Let $l \neq {\rm{char}}(k)$ be a fixed prime, and let $G[l^{\infty}] : = \cup_{n=1}^{\infty} G[l^n](\overline{k})$. 
Let $T' \subset G_{\overline{k}}$ be the identity component of the Zariski closure of $G[l^\infty]$. Since $(G_{\overline{k}})^0_{\red}$ is a torus, and the $l$-power torsion is dense in any $\overline{k}$-torus, we see that $T' = (G_{\overline{k}})^0_{\red}$. We claim that $T'$ descends to a $k$-torus in $G$. Since $G_{\overline{k}}/T'$ is finite, this will show what we want.

First, $G[l^n](k_s) = G[l^n](\overline{k})$, since $G[l^n]$ is {\'e}tale, so, since the Zariski closure of a set of rational points commutes with field extension, $T'$ descends to a $k_s$-torus $T'' \subset G_{k_s}$. Now $G[l^\infty]$ is clearly preserved
by ${\rm{Aut}}(\overline{k}/k) = {\rm{Gal}}(k_s/k)$, so $T'$ is also preserved and hence descends to a $k$-torus $T$, as claimed.
\\
(ii) $\Longrightarrow$ (iii):  Let $A \subset G$ be a finite $k$-subgroup scheme such that $G/A$ is smooth and connected (Lemma \ref{finitequotient=smoothandconnected}). Then $T \times A \rightarrow G$ is the desired isogeny, since the cokernel of this map is smooth, connected, and finite, hence trivial. 
\\
(iii) $\Longrightarrow$ (iv):  We may assume that $G = T$ is a torus. Then (iv) is essentially \cite[Thm.\,1.5.1]{ono}; there the theorem is stated over number fields, but it works over any field. The idea of the proof is to use the equivalence between tori and Galois lattices, combined with Artin's theorem on induced representations.
\\
(iv) $\Longrightarrow$ (i):  Since $k_i/k$ are separable, each $\R_{k_i/k}(T_i)$ is a torus. So $((G_{\overline{k}})^0_{\red})^n \times T_2'$ is the isogenous quotient of a torus, for some torus $T_2'$. It follows that $(G_{\overline{k}})^0_{\red}$ is itself a torus.
\end{proof}

\begin{definition}\label{almosttorusdef}
A $k$-group scheme $G$ satisfying the equivalent conditions of Lemma \ref{almosttorus} is called an {\em almost-torus}.
\end{definition}

Condition (iv) in Lemma \ref{almosttorus} will be the most useful for us, since it reduces many questions about tori to the case of separable Weil restrictions of split tori. Note in particular that finite group schemes are almost-tori. Let us also note the following easy fact.

\begin{lemma}
\label{almosttorusextension}
Suppose that we have a short exact sequence of affine commutative $k$-group schemes of finite type: 
\[
1 \longrightarrow G' \longrightarrow G \longrightarrow G'' \longrightarrow 1
\]
Then $G$ is an almost-torus if and only if both $G'$ and $G''$ are almost-tori.
\end{lemma}

\begin{proof}
This follows easily from the well-known analogous statement for tori.
\end{proof}

Now we study unipotent groups. The key result is the following lemma.

\begin{lemma}
\label{unipotentmodetale=split}
Let $k$ be a field, $U$ a smooth connected commutative unipotent $k$-group. Then there is an infinitesimal $k$-subgroup scheme $A \subset U$ such that $U/A$ is split unipotent. The same result also holds with $A$ {\'e}tale.
\end{lemma}

\begin{proof}
When ${\rm{char}}(k) = 0$, this is trivial, since every smooth connected unipotent $k$-group is split. So assume that 
${\rm{char}}(k) = p > 0$. First we treat the infinitesimal case. The unipotent group $U$ splits over the perfect closure of $k$, hence over $k^{1/p^n}$ for some $n > 0$. Hence $U^{(p^n)}$ is split. Since $U$ is smooth, the $n$-fold relative Frobenius map $U \rightarrow U^{(p^n)}$ is an isogeny. Therefore, letting $I$ denote its infinitesimal kernel, we have $U/I \simeq U^{(p^n)}$, hence $U/I$ is split.

Next we treat the {\'e}tale case. When $U$ is $p$-torsion, this is part of \cite[Lemma B.1.10]{cgp}. Otherwise, we proceed by induction:  suppose that we have a short exact sequence
\[
1 \longrightarrow U' \longrightarrow U \longrightarrow U'' \longrightarrow 1
\]
such that $U', U''$ are nontrivial smooth connected commutative unipotent $k$-groups for which the lemma holds. (We may take $U' = [p]U, U'' = U/U'$.) That is, there exist finite {\'e}tale $k$-subgroup schemes $E' \subset U'$, $E'' \subset U''$ such that $U'/E'$ and $U''/E''$ are split. Replacing $U'$ with $U'/E'$ and $U$ with $U/E'$, we may assume that $U'$ is split. So we only need to find a finite {\'e}tale $k$-subgroup scheme $E \subset U$ such that $E \twoheadrightarrow E''$, as $U/E$ is then split. To do this, we may assume that $k = k_s$, as we may then replace $E$ with the subgroup generated by its (finitely many) Galois translates in order to ensure that it is defined over $k$. 

So suppose that $k = k_s$. Then $E''$ is constant, hence we merely need to show that for every $e'' \in E''(k)$, there is $e \in U(k)$ such that $e \mapsto e''$ since $U$ is torsion (if we choose one such $e$ for each $e'' \in E''(k)$, then the finite subgroup generated by the various $e$ surjects onto $E''$). But the existence of such $e$ is clear:  the map $U \rightarrow U''$ is smooth (having smooth kernel $U'$), so the fiber above each $e'' \in E''(k)$ contains a $k$-point, since $k=k_s$.
\end{proof}

The following lemma is the main result that we will use on the structure of affine commutative group schemes.

\begin{lemma}
\label{affinegroupstructurethm}
Let $k$ be a field, $G$ an affine commutative $k$-group scheme of finite type. Then there is a short exact sequence
of $k$-group schemes 
\[
1 \longrightarrow H \longrightarrow G \longrightarrow U \longrightarrow 1
\]
with $H$ an almost-torus and $U$ a split unipotent $k$-group.
\end{lemma} 

\begin{proof}
By Lemma \ref{finitequotient=smoothandconnected}, there is a finite $k$-subgroup scheme $A \subset G$ such that $G/A$ is smooth and connected. Since an extension of an almost-torus by an almost-torus is an almost-torus (Lemma \ref{almosttorusextension}), and $A$ is an almost-torus, we may therefore assume that $G$ is smooth and connected. Letting $T \subset G$ be the maximal torus, $G/T$ is smooth connected unipotent. We may therefore assume that $G = U$ is smooth connected unipotent. But then by Lemma \ref{unipotentmodetale=split}, there is a finite $k$-subgroup scheme $A \subset U$ such that $U/A$ is split unipotent, so we are done.
\end{proof}

By analogy with the situation of almost-tori, we will on occasion have use for the notion of almost-unipotence (though this will come up far less often than almost-tori): 

\begin{lemma}
\label{almostunipotent}
Let $G$ be an affine commutative $k$-group scheme of finite type. The following are equivalent: 
\begin{itemize}
\item[(i)] $(G_{\overline{k}})_{\red}^0$ is unipotent, 
\item[(ii)] $G$ contains no nontrivial $k$-torus, 
\item[(iii)] there is an exact sequence
\[
1 \longrightarrow A \longrightarrow G \longrightarrow U \longrightarrow 1
\]
with $A$ a finite commutative $k$-group scheme and $U$ split unipotent over $k$.
\end{itemize}
\end{lemma}

\begin{proof}
(i) $\Longrightarrow$ (ii):  Obvious.
\\
(ii) $\Longrightarrow$ (iii):  By Lemma \ref{affinegroupstructurethm}, there is such a sequence with $A$ an almost-torus. Since $G$, hence $A$, contains no nontrivial tori, it follows from Lemma \ref{almosttorus}(ii) that $A$ is finite.
\\
(iii) $\Longrightarrow$ (i):  We may assume that $k = \overline{k}$. If $G_{\red}^0$ is not unipotent, then it contains a nontrivial torus $T$. The map $T \rightarrow U$ is trivial, so $T \subset A$, an absurdity.
\end{proof}

\begin{definition}
\label{almostunipdef}
A $k$-group scheme $G$ satisfying the equivalent conditions of Lemma \ref{almostunipotent} is said to be {\em almost-unipotent}.
\end{definition}

Almost-unipotent groups have the expected permanence properties:

\begin{lemma}
\label{almostunipotentpermanence}
A $k$-subgroup scheme of an almost-unipotent $k$-group scheme is almost-unipotent, as is the quotient of an almost-unipotent $k$-group scheme by a $k$-subgroup scheme. Any commutative extension of almost-unipotent $k$-group schemes is almost-unipotent. 
\end{lemma}

\begin{proof}
For subgroups, this is probably most easily seen by appealing to Lemma \ref{almostunipotent}(ii). For quotients and extensions, the assertion follows from the corresponding fact for smooth connected unipotent groups, by applying Lemma \ref{almostunipotent}(i).
\end{proof}

A very useful property of almost-unipotence is that it is preserved by Weil restriction: 

\begin{lemma}
\label{weilrestrictionalmostunipotent}
Let $k'/k$ be a finite extension of fields, and $G'$ an almost-unipotent $k'$-group scheme. Then $\R_{k'/k}(G')$ is an almost-unipotent $k$-group scheme.
\end{lemma}

\begin{proof}
Since Weil restriction is transitive, we may assume that $k'/k$ is either separable or purely inseparable. Further, almost-unipotence may be checked after replacing $k$ by an extension. If $k'/k$ is separable (so $k'_s$ is identified with a separable closure
$k_s$ of $k$) then $\R_{k'/k}(G')_{k_s}$ becomes a product of Galois-twisted copies of $G'_{k'_s}$, hence almost-unipotent.

Now suppose that $k'/k$ is purely inseparable. By \cite[App.\,3, A.3.6]{oesterle}, there is an exact sequence of $k'$-group schemes (note that the discussion in \cite[App.\,3, A.3.6]{oesterle} does not require smoothness of $G'$ for the exactness of the sequence below, only to extend it to a short exact sequence)
\[
1 \longrightarrow U \longrightarrow \R_{k'/k}(G')_{k'} \longrightarrow G'
\]
with $U$ split unipotent. By Lemma \ref{almostunipotentpermanence}, $\R_{k'/k}(G')_{k'}$ is therefore almost-unipotent.
\end{proof}

\section{Vanishing of $\calExt^1(\cdot, \Gm)$}
\label{charactersheavessection}

The purpose of this section is to prove that if $G$ is an affine commutative group scheme of finite type over a field $k$ of positive characteristic, then the fppf sheaf $\calExt^1_k(G, \Gm)$ vanishes (Proposition \ref{ext=0}). The key ingredient for the proof is to show that for any $\F_p$-scheme $S$, the sheaf $\calExt^1_S(\Ga, \Gm)$ on the big fppf site over $S$ vanishes (Proposition \ref{ext^1(Ga,Gm)=0}). This vanishing is specific to characteristic $p$, as the analogous result fails for $\Q$-schemes; see Remark \ref{gabberexample}. The characteristic-$p$ vanishing will be an easy consequence of the fact that for any $\F_p$-algebra $R$ with surjective Frobenius map (but not necessarily perfect; the main difficulty -- and really the whole point in some sense -- is to deal with non-reduced $R$), the fppf Ext group $\Ext^1_R(\Ga, \Gm)$ vanishes (Proposition \ref{calExt^1=0surjectivefrob}).

In order to prove these vanishing theorems, we will need to analyze certain formal power series. An important notion in this context will be that of a formal character.

\begin{definition}
\label{formalchardef}
Let $A$ be a ring. A {\em formal character} over $A$ is a formal power series $\chi \in A\llbracket X\rrbracket$ such that $\chi(X + Y) = \chi(X)\chi(Y)$ and $\chi(0) = 1$.
\end{definition}

Note that a product of two formal characters is again a formal character. The relevance of formal characters is that if $\chi$ is a formal character over $A$ which lies in $A[X] \subset A\llbracket X\rrbracket$, then $\chi$ yields a homomorphism of $A$-groups $\Ga \rightarrow \Gm$ via the map defined functorially on $R$-valued points ($R$ an $A$-algebra) by the formula $x \mapsto \chi(x)$.

We will require a description of all formal characters over $\F_p$-algebras. In order to do this, we need to introduce the mod $p$ exponential.

\begin{definition}
\label{modpexponential}
The {\em mod $p$ exponential} ${\rm{exp}}_p(X) \in \F_p[X]$ is defined by the formula
\begin{equation}
\label{exp_pdefinition}
{\rm{exp}}_p(X) : = \sum_{n=0}^{p-1} \frac{T^n}{n!} \in \F_p[X].
\end{equation}
\end{definition}

Note that for any $\F_p$-algebra $A$ and any $a \in A$ such that $a^p = 0$, ${\rm{exp}}_p(aX) \in A[X]$ is a formal character over $A$. It follows that the same holds when we compose this character with the additive polynomial $X^{p^n}$. That is, if $A$ is an $\F_p$-algebra, then ${\rm{exp}}_p(aX^{p^n}) \in A[X]$ is a formal character for every nonnegative integer $n$ and every $a \in A$ such that $a^p = 0$. We will soon show that all formal characters over $\F_p$-algebras are (potentially infinite) products of these.

\begin{lemma}
\label{charactersofalpha_p^n}
Let $A$ be an $\F_p$-algebra. Then every element of $\alpha_{p^n}(A)$ is given by the reduction of the formal character $\prod_{i=0}^{n-1}{\rm{exp}}_p(a_iX^{p^i})$ in $A[X]/X^{p^n}$ for some $a_i \in A$ such that $a_i^p = 0$.
\end{lemma}

\begin{proof}
Forgetting the group structure on the characters for a moment, the functor which sends an $\F_p$-algebra $A$ to the set of polynomials of the form $\prod_{i=0}^{n-1}{\rm{exp}}_p(a_iX^{p^i}) \in A[X]/(X^{p^n})$ with $a_i \in A$ such that $a_i^p = 0$ is represented by the $\F_p$-algebra $\F_p[T_1, \dots, T_n]/(T_1^p, \dots, T_n^p)$. The associated scheme $X_n$ is a finite $\F_p$-scheme of order $p^n$, as is the Cartier dual $\widehat{\alpha_{p^n}}$ of $\alpha_{p^n}$. We have a natural inclusion $X_n \hookrightarrow \widehat{\alpha_{p^n}}$ via the map sending the polynomial $\prod_{i=0}^{n-1}{\rm{exp}}_p(a_iX^{p^i})$ to the associated character of $\alpha_{p^n}$, i.e., the character defined functorially by $x \mapsto \prod_{i=0}^{n-1}{\rm{exp}}_p(a_ix^{p^i})$. By cardinality considerations, this inclusion must be an isomorphism. The lemma follows.
\end{proof}

\begin{lemma}
\label{G_a^--->>alpha_p^}
Let $A$ be an $\F_p$-algebra. The restriction map $\widehat{\mathbf{G}_a}(A) \rightarrow \widehat{\alpha_{p^n}}(A)$ is surjective
for all $n \ge 1$. In particular, $\widehat{\alpha_{p^{n+1}}}(A) \rightarrow \widehat{\alpha_{p^n}}(A)$ is surjective for all $n \ge 1$.
\end{lemma}

\begin{proof}
The second assertion follows immediately from the first, so we concentrate on the first statement. By Lemma \ref{charactersofalpha_p^n}, any characters of $\alpha_{p^n}$ over $A$ is given by reducing the formal character $\prod_{i = 0}^{n-1} {\rm{exp}}_p(a_iX^{p^i})$ mod $X^{p^n}$, where the $a_i \in A$ satisfy $a_i^p = 0$. But the polynomial $\prod_{i = 0}^{n-1} {\rm{exp}}_p(a_iX^{p^i}) \in A[X]$ itself yields a character of $\Ga$, precisely because it is a formal character, and this provides our lift of the $A$-character of $\alpha_{p^n}$.
\end{proof}

\begin{lemma}
\label{formalchars}
\begin{itemize}
\item[(i)] If $A$ is an $\F_p$-algebra, then the formal characters over $A$ are precisely the power series $\prod_{n=0}^{\infty} \exp_p(a_nX^{p^n}) \in A\llbracket X \rrbracket$, where the $a_n \in A$ satisfy $a_n^p = 0$.
\item[(ii)] If $A$ is an $\F_p$-algebra, then the characters $\mathbf{G}_{a,\,A} \rightarrow \mathbf{G}_{m,\,A}$ are precisely the polynomials of the form $\prod_{n=0}^{N} \exp_p(a_nX^{p^n}) \in A[X]$ for some nonnegative integer $N$, where the $a_n \in A$ satisfy $a_n^p = 0$.
\item[(iii)] If $A$ is a $\Q$-algebra, then the formal characters over $A$ are precisely the power series ${\rm{exp}}(aX) = \sum_{n=0}^{\infty} a^nX^n/n! \in A\llbracket X\rrbracket$ with $a \in A$.
\item[(iv)] If $A$ is a $\Q$-algebra, then the characters $\mathbf{G}_{a,\,A} \rightarrow \mathbf{G}_{m,\,A}$ are precisely the polynomials ${\rm{exp}}(aX) \in A[X]$ with $a \in A$ nilpotent.
\end{itemize}
\end{lemma}

\begin{proof}
Assertions (ii) and (iv) follow from (i) and (iii), respectively. In order to prove (i), we note that an element $g \in A \llbracket X \rrbracket$ is a formal character if and only if it is a bona fide character mod $X^{p^m}$ for each $m$, i.e., if and only if $g$ mod $X^{p^m}$ is a bona fide character of $\alpha_{p^m}$. Assertion (i) therefore follows from Lemma \ref{charactersofalpha_p^n}.

For (iii), we first note that all power series of the given form are formal characters, so we only need to check that 
there are no others. Let $f \in A\llbracket X\rrbracket^{\times}$ be a formal character, which is to say that $f(0) = 1$ and
\begin{equation}
\label{charscharacteristic0pfeqn1}
f(X + Y) = f(X)f(Y).
\end{equation}
We have $f(X) = 1 + \sum_{n \geq 1} c_nX^n$ for some $c_n \in A$. Comparing coefficients of $XY^m$ on both sides of (\ref{charscharacteristic0pfeqn1}) implies $(m+1)c_{m+1} = c_1c_m$; i.e., $c_{m+1} = c_1c_m/(m+1)$. An easy induction now shows that $c_m = c_1^m/m!$ for all $m \geq 1$. 
\end{proof}

\begin{lemma}
\label{Ext^1(alpha,Gm)=0,perfect}
Let $R$ be an $\F_p$-algebra
 such that the Frobenius map $F:  R \rightarrow R$ is surjective. Then $\Ext^1_R(\alpha_{p^n}, \mathbf{G}_m) = 0$.
\end{lemma}

\begin{proof}
We first claim that ${{\rm{H}}}^1(R, \widehat{\alpha_{p^n}}) = 0$. The group scheme $\alpha_{p^n}$, and hence $\widehat{\alpha_{p^n}}$, admits a filtration by $\alpha_p$'s, so it is  enough to show that ${{\rm{H}}}^1(R, \alpha_p) = 0$. This in turn follows from the 
vanishing of ${\rm{H}}^1_{\rm{fppf}}(R, \Ga) = 
{\rm{H}}^1_{\rm{Zar}}(R, \Ga)$ and the exact sequence
\[
1 \longrightarrow \alpha_p \longrightarrow \mathbf{G}_a \xlongrightarrow{F} \mathbf{G}_a \longrightarrow 1
\]
due to our assumption that the Frobenius map on $R$ is surjective.

We have a Leray spectral sequence
\[
E_2^{i,j} = {{\rm{H}}}^i(R, \calExt^j_R(\alpha_{p^n}, \mathbf{G}_m)) \Longrightarrow \Ext^{i+j}_R(\alpha_{p^n}, \mathbf{G}_m)
\]
We have just shown that ${{\rm{H}}}^1(R, \widehat{\alpha_{p^n}}) = E_2^{1,0}$ vanishes, so it remains to prove the same for $E_2^{0,1} = {{\rm{H}}}^0(R, \calExt^1(\alpha_{p^n}, \mathbf{G}_m))$. This follows from \cite[VIII, Prop.\,3.3.1]{sga7}, which implies that $\calExt^1(\alpha_{p^n}, \mathbf{G}_m) = 0$.
\end{proof}

\begin{lemma}
\label{binomialcoeffcongruence}
Let $p$ be a prime, and let $a_i, b_i$ ($i = 0, \dots, m$) be integers such that for each $i$, $0 \leq a_i, b_i < p$. Then
\begin{equation}
\label{binomcingruenceeqn15}
\binom{a_0 + a_1p + a_2p^2 + \dots + a_mp^m}{b_0 + b_1p + b_2p^2 + \dots + b_mp^m} \equiv \binom{a_0}{b_0} \binom{a_1}{b_1} \binom{a_2}{b_2} \dots \binom{a_m}{b_m} \pmod{p},
\end{equation}
where if $a_i < b_i$, then $\binom{a_i}{b_i} = 0$.
\end{lemma}

\begin{proof}
This is well-known, but we give a proof for the reader's convenience. Consider the polynomial
\[
G(X) : = (X+1)^{\sum_{i=0}^m a_ip^i} \in \F_p[X].
\]
On the one hand, the coefficient of $X^{\sum_{i=0}^m b_ip^i}$ in $G(X)$ is the left side of (\ref{binomcingruenceeqn15}). On the other hand, we have
\begin{equation}
\label{G(X)newexpreqn}
G(X) = \prod_{i=0}^m (X+1)^{a_ip^i} = \prod_{i=0}^m (X^{p^i} + 1)^{a_i}.
\end{equation}
Since $0 \leq a_i, b_i < p$, the uniqueness of the base-$p$ expansion for any nonnegative integer and the expression (\ref{G(X)newexpreqn}) for $G(X)$ shows that the coefficient of $X^{\sum_{i=0}^m b_ip^i}$ in $G(X)$ is also equal to the right side of (\ref{binomcingruenceeqn15}).
\end{proof}

The following lemma is the key to proving the vanishing of the fppf sheaf $\calExt^1_S(\Ga, \Gm)$ for any $\F_p$-scheme $S$.

\begin{lemma}
\label{schemesplitimpliesgpsplit}
Let $R$ be an $\F_p$-algebra such that the Frobenius map $F:  R \rightarrow R$, $r \mapsto r^p$, is surjective. Suppose given a commutative extension of $R$-group schemes
\begin{equation}
\label{extensionwithsectionlemmaeqn}
1 \rightarrow \Gm \xrightarrow{j} E \xrightarrow{\pi} \Ga \rightarrow 1
\end{equation}
such that there is a scheme-theoretic section (that is not necessarily an $R$-group homomorphism) $\Ga \rightarrow E$. Then the sequence $(\ref{extensionwithsectionlemmaeqn})$ splits. 
\end{lemma}

\begin{proof}
Let $s:  \mathbf{G}_a \longrightarrow E$ be a scheme-theoretic section. We seek an $R$-group section when $F:  R \rightarrow R$ is surjective. 
By translating $s$ by $-s(0) \in E(R)$, we may assume that $s(0) = 0_E$. Then by means of $s$, we may identify $E$ with $\mathbf{G}_m \times \mathbf{G}_a$ as $\Gm$-torsors over $\Ga$, with $j$ the inclusion $t \mapsto (t,0)$, $s(x) = (1,x)$, and $(1,0)$ the identity of $E$. The group law on $E$ is $(t_1, x_1) \cdot (t_2, x_2) = (t_1t_2h(x_1,x_2), x_1 + x_2)$ for some $h:  \mathbf{G}_a \times \mathbf{G}_a \rightarrow \mathbf{G}_m$. That is, $h(X,Y) \in R[X,Y]^{\times}$. Note that $h(X, 0) = h(0,Y) = 1$, since $(1,0)$ is the identity. 

A group-theoretic section $\mathbf{G}_a \rightarrow E$ is the same as an $f(X) \in R[X]^{\times}$ such that 
\begin{equation}
\label{sectionequation}
f(X+Y) = f(X)f(Y)h(X,Y)
\end{equation}
Indeed, there is a bijection between the set of such $f$ and the set of group-theoretic sections by assigning to $f$ the map $x \mapsto (f(x), x)$. We need to show that there exists $f \in R[X]^{\times}$ satisfying (\ref{sectionequation}).

Let us first note that $h \equiv 1 \pmod I$, where $I \subset R$ is a finitely generated nilpotent ideal; i.e., $I^e = 0$ for some $e > 0$. Indeed, if we show that the non-constant coefficients of $h$ are nilpotent, then this follows from the fact that $h(X, 0) = h(0, Y) = 1$. In order to prove
such nilpotence, it is  enough to show that $h = 1$ in every residue field of $R$, and this in turn follows from the fact that $h \in R[X]^{\times}$, and that $L[X]^{\times} = L^{\times}$ for a field $L$.

By Lemma \ref{Ext^1(alpha,Gm)=0,perfect}, the pullback of the extension $E$ to an extension of $\alpha_{p^n} \subset \mathbf{G}_a$ by $\Gm$ splits. Further, if we let $s_n$ denote a chosen (group-theoretic) section of $E|_{\alpha_{p^n}}$, then $s_{n+1}|_{\alpha_{p^n}} = s_n + \overline{\chi_n}$ for some $\overline{\chi_n} \in \widehat{\alpha_{p^n}}(R)$. Choose a lift $\chi_n \in \widehat{\alpha_{p^{n+1}}}(R)$ of $\chi_n$, which exists by Lemma \ref{G_a^--->>alpha_p^}. Then replace $s_{n+1}$ with $s_{n+1} - \chi_n$. Carrying out this process inductively, we may assume that the sections $s_n$ are compatible; i.e., $s_{n+1}|_{\alpha_{p^n}} = s_n$. In terms of $f$ and $h$, this says that we have for each $n$ an element $f_n \in (R[X]/(X^{p^n}))^{\times}$ such that
\[
f_n(X+Y) \equiv f_n(X)f_n(Y)h(X,Y) \mod (X^{p^n}, Y^{p^n})
\]
and such that $f_{n+1}(X) \equiv f_n(X) \mod X^{p^n}$. The $f_n$'s therefore ``glue'' to give an element $f \in R[\![X]\!]^{\times}$ such that (\ref{sectionequation}) holds; i.e., 
\begin{equation}
\label{sectioneqn2}
f(X+Y) = f(X)f(Y)h(X,Y).
\end{equation}

The (non-empty!) set of $f \in R[\![X]\!]^{\times}$ satisfying (\ref{sectioneqn2}) is a torsor for the group of formal characters over $R$. More precisely, the solutions $g \in R[\![X]\!]^{\times}$ to (\ref{sectioneqn2}) are precisely those power series of the form $g = f\chi$, with $\chi$ a formal character over $R$. We will therefore try to modify $f$ by formal characters to make it a polynomial (rather than just a power series). We will accomplish this in several steps. Write $$f(X) = \sum_{n \geq 0} b_nX^n,\,\,\,h(X, Y) = 1 + \sum_{0 < i,j < N}a_{ij}X^iY^j$$ for some $N$ and where $b_0 = 1$; note that $h$ is in this form since $h(X, 0) = h(0,Y) = 1$ and that $f$ has constant term $1$ 
by inspection upon setting $X = Y = 0$ on both sides of (\ref{sectioneqn2}), since $f \in R \llbracket X \rrbracket^{\times}$. Further, $a_{mn} \in I$ due to the fact that $h \equiv 1 \pmod I$. Comparing the coefficients of $X^mY^s$ in the equation $f(X+Y) = f(X)f(Y)h(X, Y)$, we obtain
\begin{equation}
\label{coeffsextvanish}
\binom{m+s}{m} b_{m+s} = b_mb_s + \sum_{0< i,j < N}a_{ij}b_{m-i}b_{s-j},
\end{equation}
where $b_r = 0$ if $r < 0$.

Let $I = (c_1, \dots, c_g)$ (recall that $I$ is finitely generated), and choose $d_1, \dots, d_g \in R$ such that $d_i^p = c_i$. Set $J : = (d_1, \dots, d_g) \subset R$. For the ease of the reader, we break the proof into several steps. 

\medskip\noindent
\underline{Step 1}: 
By multiplying $f$ by a formal character, we may assume that $b_{p^n} \in J$ for all $n \geq 0$.

\medskip
By looking at (\ref{sectioneqn2}) mod $I$, and using the fact that $h \equiv 1 \pmod I$, we see that $f$ mod $I$ is a formal character, so $f^p \bmod I = (f \bmod I)^p$ is equal to 1 (since $\Ga$ is $p$-torsion).  Hence, $b_n^p \in I$ for all $n >0$, so $b_{p^m}^p \in I$
for all $m \ge 0$. Therefore, for each $m \ge 0$ we have $(\alpha_m - b_{p^m})^p = 0$ for some $\alpha_m \in J$, and this remains true after multiplying $f$ by a formal character.

Let $\beta_0 : = \alpha_0 - b_1$. Multiplying $f$ by the formal character exp$_p(\beta_0X)$, we may assume that $b_1 \in J$. Now let $\beta_1 = \alpha_1 - b_p$ for this new $f$, and multiply the new $f$ by exp$_p(\beta_1X^p)$ to ensure that $b_p \in J$. Continuing in this fashion, (more precisely, multiplying $f$ by the formal character $\prod \exp_p(\beta_iX^{p^i})$ where the $\beta_i$ are defined by this inductive procedure), we see that we may ensure that $b_{p^m} \in J$ for all $m \ge 0$.

\medskip
Using Step 1, we now assume that each $b_{p^n} \in J$.

\medskip\noindent
\underline{Step 2}: 
Fix a positive integer $M$. Then for all $n$ sufficiently large (depending on $M$), $b_{rp^n - i} = 0$ for all positive integers $r$ and all integers $0 < i < M$.

\medskip
The key observation is the following: Let $S(n)$ denote the sum of the terms in the base $p$ expansion of $n$, so if $n = c_0 + c_1p + \dots + c_lp^l$, $0 \leq c_j < p$, then $S(n) : = \sum_{i=0}^l c_i$. For any integer $l \ge 0$, we claim that $b_n \in J^l$ provided that $S(n)$ is sufficiently large (a priori depending on $l$, though in fact not, since $I$, hence $J$, is nilpotent). Then, since $J^h = 0$ for some $h$, taking $l \geq h$ shows that $b_n = 0$ provided that $S(n)$ is sufficiently large. Since for each positive integer $i$, we have $S(rp^n - i) \rightarrow \infty$ uniformly in $r \in \Z_+$ as $n \rightarrow \infty$, this will prove Step 2.

We prove the above observation by induction on $l$, the case $l = 0$ being trivial. So suppose that the assertion is true for $l$, and we will prove it for $l+1$. Suppose that $S(n)$ is large, and let $n = c_0 + c_1p + \dots + c_tp^t$ be the base $p$ expansion of $n$, with $c_t \neq 0$. Then $S(n) = S(n-p^t) + 1$, and taking $m = n-p^t$ and $s = p^t$ in (\ref{coeffsextvanish}) yields 
\[
\binom{n}{p^t}b_n = b_{n-p^t}b_{p^t} + \sum_{0 < i, j < N} a_{ij} b_{n-p^t-i}b_{p^t-j}
\]
Now $S(n-p^t) = S(n)-1$ is large, so, by induction, $b_{n-p^t} \in J^l$. Further, $b_{p^t} \in J$ (as mentioned above, we may arrange this by Step 1), so the first term on the right lies in $J^{l+1}$ if $S(n)$ is sufficiently large. Further, if $S(n)$ (and hence $S(n-p^t)$) is large, then so are $S(n-p^t-i)$ for $0 < i < N$, so $b_{n-p^t-i} \in J^l$ for each such $i$ by induction. Finally, $a_{ij} \in I \subset J$, so each term in the sum on the right lies in $J^{l+1}$ if $S(n)$ is large. 

By Lemma \ref{binomialcoeffcongruence}, $\binom{n}{p^t} \equiv c_t \pmod p$, hence is nonzero in $\F_p$, so $b_n \in J^{l+1}$. This completes the proof of Step 2.

\medskip\noindent
\underline{Step 3}:  $b_{p^n}^p = 0$ for $n \gg 0$.

\medskip
Apply (\ref{coeffsextvanish}) with $m = p^n, s = (p-1)p^n$ to obtain
\[
\binom{p^{n+1}}{p^n}b_{p^{n+1}} = b_{p^n}b_{(p-1)p^n} + \sum_{0 < i,j < N} a_{ij}b_{p^n-i}b_{(p-1)p^n-j}.
\]
Once again using Lemma \ref{binomialcoeffcongruence}, we now see that $\binom{p^{n+1}}{p^n} = 0$ in $\F_p$. Further, if we take $M = N$ in Step 2, then for $n$ sufficiently large the sum on the right side is $0$. Therefore,
\[
0 = b_{p^n}b_{(p-1)p^n}
\]
for all sufficiently large $n$.
Now we will show by induction on $0 < r \le p$ that, provided $n$ is sufficiently large, $b_{p^n}^rb_{(p-r)p^n} = 0$ for $0 < r \leq p$; the base case $r = 1$ has just been handled. Taking $r = p$ will complete the proof of Step 3. So suppose that 
the case of some positive $r < p$ is settled. Taking $m = p^n, s = (p-r-1)p^n$ in (\ref{coeffsextvanish}), we see that
\begin{equation}\label{lastb}
\binom{(p-r)p^n}{p^n} b_{(p-r)p^n} = b_{p^n}b_{(p-r-1)p^n} + \sum_{0 < i,j < N} a_{ij}b_{p^n-i}b_{(p-r-1)p^n-j}.
\end{equation}
Again using Lemma \ref{binomialcoeffcongruence}, we see that $\binom{(p-r)p^n}{p^n} = p-r \neq 0$ in $\F_p$. Further, taking $M = N$ once again in Step 2, we see that for $n \gg 0$, the sum on the right side of (\ref{lastb}) vanishes. Therefore, since $b_{p^n}^rb_{(p-r)p^n} = 0$ by  our inductive hypothesis on $r$, multiplying both sides of (\ref{lastb}) by $b_{p^n}^r$ shows that  $b_{p^n}^{r+1}b_{(p-r-1)p^n} = 0$, which proves the desired result for $r+1$, hence completes the induction and the proof of Step 3.

\medskip\noindent
\underline{Step 4}:  By replacing $f$ with $f\chi$ for some formal character $\chi$, we may arrange that $b_{p^n} = 0$ for all sufficiently large $n$.

\medskip
Indeed, using Step 3, suppose that $b_{p^n}^p = 0$ for $n \geq N$. This remains true upon replacing $f$ with $f\chi$ for some formal character $\chi$, since $\chi^p = 1$. Let $\chi_N = \exp_p(-b_{p^N}X^{p^N})$. This is a formal character, and the coefficient of $X^{p^n}$ in $f\chi_N$ is $0$. Denote the new coefficient of $X^{p^{N+1}}$ in $f\chi_N$ by $b'_{p^{N+1}}$. Then multiplying by $\chi_{N+1} = \exp_p(-b'_{p^{N+1}}X^{p^{N+1}})$, we get that the coefficients of $X^{p^N}$ and $X^{p^{N+1}}$ in $f\chi_N\chi_{N+1}$ are both $0$. Continuing in this way, we see that we may take the $\chi$ in the statement of Step 4 to be $\prod_{n=N}^{\infty}\chi_n$.

\medskip
Since we are free to multiply $f$ by a formal character, we now may and do assume that $b_{p^n} = 0$ for all sufficiently large $n$.

\medskip\noindent
\underline{Step 5}:  $f \in R[X]$; that is, $f$ is a polynomial.

\medskip
We will prove by induction on $r$ that $f$ is a polynomial mod $I^r$ for each positive integer $r$. That is, $b_n \in I^r$ for $n \gg_r 0$. Taking $r = e$ (where $I^e = 0$) will prove that $f$ is a polynomial. 

The base case $r = 0$ is trivial, so assume that $r \geq 0$ and that $f$ is a polynomial mod $I^r$, and we will show that the same holds mod $I^{r+1}$; that is, $b_n \in I^{r+1}$ for $n \gg_r 0$. Write $n$ in base $p$:  $n = c_0 + c_1p + \dots + c_tp^t$ with $0 \leq c_i < p$ and $c_t \neq 0$. Apply (\ref{coeffsextvanish}) with $m = p^t$, $s = n-p^t$ to conclude that
\[
\binom{n}{p^t} b_n = b_{p^t}b_{n-p^t} + \sum_{0 < i,j < N} a_{ij}b_{p^t-i}b_{n-p^t-j}.
\]
Since $b_l \in I^r$ for $l \gg 0$, the sum on the right lies in $I^{r+1}$ provided $n$ is sufficiently large (since $a_{ij} \in I$). We also have $b_{p^t} = 0$ if $n \gg 0$ (since then $t \gg 0$, and we have arranged that $b_{p^m}=0$ for all sufficiently large $m$). 
Finally, by Lemma \ref{binomialcoeffcongruence}, $\binom{n}{p^t} = c_t \neq 0$
in $\F_p$. We therefore deduce that $b_n \in I^{r+1}$, as desired. This completes the induction and the proof of Step 5.

\medskip

The proof of Proposition \ref{ext^1(Ga,Gm)=0} is now almost complete. We have found a polynomial $f \in R[X] \cap R[\![X]\!]^{\times}$ satisfying (\ref{sectioneqn2}). We need to show that $f \in R[X]^{\times}$. But $h^{p^l} = 1$ when $p^l \geq e$, since $h \equiv 1 \pmod I$, and $I^e = 0$, so $f^{p^l}$ is a formal character. Therefore, $f^{p^{l+1}} = 1$, so $f \in R[X]^{\times}$. The proof of Lemma \ref{schemesplitimpliesgpsplit} is finally complete.
\end{proof}

In order to remove the hypothesis that there is a scheme-theoretic section in Lemma \ref{schemesplitimpliesgpsplit}, we need a couple of lemmas.

\begin{lemma}
\label{picX = pic(Xred)}
Let $X$ be an affine scheme, and $G$ a smooth commutative $X$-group scheme of finite type. The natural map ${\rm{H}}^i(X, G) \rightarrow {\rm{H}}^i(X_{\red}, G)$ is an isomorphism for all $i > 0$.
\end{lemma}

\begin{proof}
Since $G$ is smooth, we may take the cohomology to be \'etale. If we write $X = \Spec(A)$, then $A$ is the filtered direct limit of its finitely generated $\Z$-subalgebras, and ${\rm{H}}^i(\cdot, G)$ and $(\cdot)_{\red}$ commute with filtered direct limits of rings, so we may assume that $X$ is Noetherian. We therefore have a filtration $0 = \mathscr{I}_0 \subset \dots \subset \mathscr{I}_n = \mathscr{N}$ of the sheaf $\mathscr{N} \subset \mathcal{O}_X$ of nilpotents by quasi-coherent sheaves of ideals such that $\mathscr{I}_{m+1}^2 \subset \mathscr{I}_m$, so it suffices to prove the following assertion: given an affine scheme $X$ and a quasi-coherent ideal sheaf $\mathscr{I} \subset \calO_X$ such that $\mathscr{I}^2 = 0$, the map ${\rm{H}}^i(X, G) \rightarrow {\rm{H}}^i(V(\mathscr{I}), G)$ is an isomorphism, where $V(\mathscr{I})$ is the closed subscheme defined by $\mathscr{I}$. Since $G$ is smooth over $X$ and $\mathscr{I}$ is a square-zero ideal sheaf, we have an exact sequence of \'etale sheaves on $X$
\[
1 \longrightarrow \calHom_{\calO_X}(e^*\Omega_{G/X}, \mathscr{I})  \longrightarrow G \longrightarrow i_*G \longrightarrow 1,
\]
where $i:  V(\mathscr{I}) \rightarrow X$ is the canonical map and $e:  X \rightarrow G$ is the identity section. Since $i$ is a closed immersion, $i_*$ is an exact functor, hence ${\rm{H}}^i(X, i_*G) = {\rm{H}}^i(V(\mathscr{I}), G)$.  In order to prove the lemma it suffices to show that ${\rm{H}}^i_{\et}(X, \calHom_{\calO_X}(e^*\Omega_{G/X}, \mathscr{I})) = 0$ for $i > 0$. This in turn follows from the fact that $\calHom_{\calO_X}(e^*\Omega_{G/X}, \mathscr{I})$ is a quasi-coherent sheaf and $X$ is affine.
\end{proof}

\begin{lemma}
\label{F^*Pic[p]=0}
Let $R$ be an $\F_p$-algebra such that the Frobenius map $F_R:  R \rightarrow R$, $r \mapsto r^p$, is surjective, and let $X$ be an affine $R$-scheme. Let $F_{X/R}:  X \rightarrow X^{(p)}$ denote the relative Frobenius $R$-morphism of $X$. Then for any $\mathscr{L} \in \Pic(X^{(p)})[p]$, we have $F_{X/R}^*(\mathscr{L}) = 0$ in $\Pic(X)$.
\end{lemma}

\begin{proof}
By Lemma \ref{picX = pic(Xred)}, we may base change $X$ to $X \otimes_R R_{{\rm{red}}}$, and thereby assume that $F_R$ is an isomorphism. Consider the following commutative diagram.
\[
\begin{tikzcd}
X^{(p)} \arrow[rr, bend left, "F_{X^{(p)}}"] \arrow{r}{\phi}[swap]{\sim} \arrow{d} \arrow[dr, phantom, "\square"] & X \arrow{d} \arrow{r}[swap]{F_{X/R}} & X^{(p)} \\
\Spec(R) \arrow{r}{\sim}[swap]{F_R} & \Spec(R) &
\end{tikzcd}
\]
in which we abuse notation and also refer to the map on $\Spec(R)$ induced by $F_R$ as $F_R$, where $F_{X^{(p)}}$ denotes the absolute Frobenius map, and $\phi$ is defined by the above diagram. The map $F_{X^{(p)}}^*:  \Pic(X^{(p)}) \rightarrow \Pic(X^{(p)})$ is multiplication by $p$, hence kills $\mathscr{L}$. Since $\phi$, hence $\phi^*$, is an isomorphism, it follows that $F_{X/R}^*(\mathscr{L}) = 0$, as desired.
\end{proof}

\begin{proposition}
\label{calExt^1=0surjectivefrob}
Let $R$ be a ring of characteristic $p$ such that the Frobenius map $F:  R \rightarrow R$ is surjective. Then $\Ext^1_R(\Ga, \Gm) = 0$.
\end{proposition}

\begin{proof}
Suppose given an extension of fppf abelian sheaves over $R$
\[
1 \longrightarrow \Gm \longrightarrow E \longrightarrow \Ga \longrightarrow 1.
\]
We need to show that this extension splits. The sheaf $E$ (which is actually an $R$-group scheme due to the effectivity of fpqc descent for relatively affine schemes) has, via multiplication by $\Gm$, the structure of a $\Gm$-torsor over $\Ga$. That is, we obtain from $E$ an element $\mathscr{L} \in \Pic(\mathbf{G}_{a,\,R})$. This yields a homomorphism $\Ext^1_R(\Ga, \mathbf{G}_m) \rightarrow {\rm{Pic}}(\mathbf{G}_{a,\,R})$. Note that this map has $p$-torsion image since $\Ext^1_R(\mathbf{G}_a, \mathbf{G}_m)$ is $p$-torsion (because $\mathbf{G}_a$ is).

Consider the relative Frobenius $R$-isogeny $F_{\mathbf{G}_a/R}:  \mathbf{G}_a \rightarrow \mathbf{G}_a$. By Lemma \ref{F^*Pic[p]=0}, $F_{\mathbf{G}_a/R}^*(\mathscr{L}) \in \Pic(\mathbf{G}_{a,\,R})$ is trivial. By Lemma \ref{schemesplitimpliesgpsplit}, therefore, the extension $F_{\mathbf{G}_a/R}^*(E)$ defined by the pullback diagram
\[
\begin{tikzcd}
1 \arrow{r} & \Gm \arrow{r} \arrow[d, equals] & F_{\mathbf{G}_a/R}^*(E) \arrow[dr, phantom, "\square"] \arrow{d} \arrow{r} & \Ga \arrow{r} \arrow{d}{F_{\mathbf{G}_a/R}} & 1 \\
1 \arrow{r} & \Gm \arrow{r} & E \arrow{r} & \Ga \arrow{r} & 1
\end{tikzcd}
\]
splits. That is, the element $F_{\mathbf{G}_a/R}^*(E) \in \Ext^1_R(\Ga, \Gm)$ vanishes.
The exact sequence of $R$-group schemes
\[
1 \longrightarrow \alpha_p \longrightarrow \mathbf{G}_a \xlongrightarrow{F_{\mathbf{G}_a/R}} \mathbf{G}_a \longrightarrow 1
\]
yields an exact sequence
\[
\widehat{\mathbf{G}_a}(R) \longrightarrow \widehat{\alpha_p}(R) \longrightarrow \Ext^1_R(\mathbf{G}_a, \mathbf{G}_m) \xlongrightarrow{F_{\mathbf{G}_a/R}^*} \Ext^1_R(\mathbf{G}_a, \mathbf{G}_m), 
\]
so Lemma \ref{G_a^--->>alpha_p^} implies that $E$ is trivial as an element of $\Ext^1_R(\Ga, \Gm)$; that is, $E$ splits. 
\end{proof}

Let us note the following corollary of Proposition \ref{calExt^1=0surjectivefrob}, which we will have occasion to use later.

\begin{corollary}
\label{H^1(U^)=0}
Let $R$ be an $\F_p$-algebra such that the Frobenius map $F:  R \rightarrow R$ is surjective. Let $U$ be a
commutative $R$-group scheme admitting a finite filtration by $\mathbf{G}_a$'s. Then ${{\rm{H}}}^1(R, \widehat{U}) = 0$.
\end{corollary}

\begin{proof}
The Leray spectral sequence
\[
E_2^{i,j} = {{\rm{H}}}^i(R, \calExt^j(U, \mathbf{G}_m)) \Longrightarrow \Ext^{i+j}(U, \mathbf{G}_m)
\]
shows that it suffices to show that $\Ext^1_R(U, \mathbf{G}_m) = 0$. Since $U$ is filtered by $\mathbf{G}_a$'s it suffices
to treat the case $U = \mathbf{G}_a$, and this is Proposition \ref{calExt^1=0surjectivefrob}.
\end{proof}

In order to prove the vanishing of the fppf Ext sheaf, we need the following well-known result, which identifies the Ext sheaf with the sheafification of the the Ext groups.

\begin{lemma}
\label{extsheaf=sheafification}
Let $S$ be a site, and let ${\rm{Sh}}$ denote the category of sheaves of abelian groups on $S$. Let $A \in {\rm{D(Sh)}}$, $B \in {\rm{D}}^+({\rm{Sh}})$. Then for all $i$, the sheaf $\calExt^i(A, B)$ on $S$ is the sheafification of the presheaf $U \mapsto \Ext^i_U(A, B)$.
\end{lemma}

\begin{proof}
The bifunctor $\calHom:  {\rm{D(Sh)}} \times {\rm{D}}^+({\rm{Sh}}) \rightarrow {\rm{D(Sh)}}$ is the sheafification of the bifunctor $\Hom:  {\rm{D(Sh)}} \times {\rm{D}}^+({\rm{Sh}}) \rightarrow {\rm{D(PSh)}}$, where PSh is the category of presheaves of abelian groups on $S$. (This makes sense because sheafification is an exact functor.) It therefore suffices to show that on D(PSh), taking cohomology commutes with sheafification, and this follows from the fact that sheafification is exact.
\end{proof}

\begin{proposition}
\label{ext^1(Ga,Gm)=0}
For any $\mathbf{F}_p$-scheme $S$, the fppf sheaf $\calExt^1_S(\mathbf{G}_a, \mathbf{G}_m)$ vanishes.
\end{proposition}

\begin{proof}
By Lemma \ref{extsheaf=sheafification}, the fppf sheaf $\calExt^1_S(\Ga, \Gm)$ is the sheafification of the presheaf which assigns to an fppf scheme $U/S$ the group $\Ext^1_U(\Ga, \Gm)$. Thus, the desired vanishing is equivalent to the assertion that for any $\F_p$-algebra $R$, any extension
\begin{equation}
\label{extensionsproofofext^1=0eqn3}
1 \longrightarrow \mathbf{G}_m \longrightarrow E \longrightarrow \mathbf{G}_a \longrightarrow 1
\end{equation}
of abelian sheaves over $R$ splits fppf locally on $\Spec(R)$. But any such extension with $E$ merely a sheaf is actually representable, due to the effectivity of fpqc descent for relatively affine schemes. By repeatedly adjoining $p$th roots and taking a direct limit, we may assume that the Frobenius map $F:  R \rightarrow R$ is surjective, since then a section $\Ga \rightarrow E$ over $R$ of (\ref{extensionsproofofext^1=0eqn3}) necessarily descends to one of the rings over which one is taking a direct limit. The proposition therefore follows from Proposition  \ref{calExt^1=0surjectivefrob}.
\end{proof}

Before giving Gabber's example of an extension of $\Ga$ by $\Gm$ in characteristic $0$ that does not split fppf locally, we first require the following lemma.

\begin{lemma}
\label{nonconstantunitsred}
If $X$ is a reduced scheme, then every global unit on $\A^n_X$ (affine $n$-space over $X$) is the pullback of a global unit on $X$.
\end{lemma} 

\begin{proof}
Let $u$ be a global unit on $\A^n_X$. Replacing $u$ with $u/u(0)$, we may assume that $u(0) = 1$, and we need to show $u = 1$. This holds over every residue field of a point on $X$, because over a field $\A^n$ has no nonconstant global units. Since $X$ is reduced, therefore, the equality holds on $\A^n_X$.
\end{proof}

\begin{remark}
\label{gabberexample}
Proposition \ref{ext^1(Ga,Gm)=0} is false in characteristic $0$. We will never use this, so the uninterested reader may skip this remark. Before constructing a counterexample due to Ofer Gabber, we will prove the following claim:  given a $\Q$-scheme $X$ and a commutative extension $E$ of $\Ga$ by $\Gm$ over $X$, if $E$ splits fppf-locally over $X$ then it splits \'etale-locally over $X$. (Actually, our argument will show that it splits Nisnevich-locally.) Indeed, $\Ext^1(\Ga, \Gm)$ commutes with filtered direct limits, so we may assume that $X = \Spec(R)$ with $R$ Henselian local. Any fppf cover of $X$ may be refined by one that is finite locally free with constant rank $n>0$ because by \cite[IV$_4$, Cor. 17.16.2]{ega} any fppf cover of $\Spec(R)$ may be refined by an affine quasi-finite fppf cover, and (since $R$ is Henselian) any such cover may be refined by one of the form $\Spec(S) \rightarrow \Spec(R)$ with $S$ a finite local $R$-algebra that is free as an $R$-module due to \cite[IV$_4$, Thm.\,18.5.11]{ega}. We may therefore assume that $\pi:  X' \rightarrow X$ is finite locally free of degree $n$, and that $E$ splits over $X'$. Let $X' = \Spec(R')$.

Now for some generalities. Given fppf abelian sheaves $\mathscr{F}$ on $X$ and $\mathscr{F}'$ on $X'$, we have a natural map
\begin{equation}
\label{abstractextmap}
\Ext^1_X(\mathscr{F}, \pi_*\mathscr{F}') \xrightarrow{\phi} \Ext^1_{X'}(\pi^*\mathscr{F}, \mathscr{F}')
\end{equation}
defined as follows:  given an extension $\mathscr{E} \in \Ext^1_X(\mathscr{F}, \pi_*\mathscr{F}')$, we may pull it back to $X'$ to obtain the extension $\pi^*\mathscr{E} \in \Ext^1_{X'}(\pi^*\mathscr{F}, \pi^*\pi_*\mathscr{F}')$. Then we push this out along the adjunction map $\pi^*\pi_*\mathscr{F}' \rightarrow \mathscr{F}'$ to obtain an element of $\Ext^1_{X'}(\pi^*\mathscr{F}, \mathscr{F}')$. We claim that the map $\phi$ in (\ref{abstractextmap}) is an inclusion. Indeed, a section from our new extension $\phi(\mathscr{E})$ to $\mathscr{F}'$ is the same thing as a commutative diagram
\[
\begin{tikzcd}
\pi^*\pi_*\mathscr{F}' \arrow{r} \arrow{d} & \pi^*\mathscr{E} \arrow{dl} \\
\mathscr{F}' &
\end{tikzcd}
\]
which, by adjointness, is the same thing as a section $\mathscr{E} \rightarrow \pi_*\mathscr{F}'$. That is, $\mathscr{E}$ splits if and only if $\phi(\mathscr{E})$ does.

Suppose that $\mathscr{F}' = \pi^*\mathscr{G}$ for some sheaf $\mathscr{G}$ on $X$. Then the adjunction map $\mathscr{G} \rightarrow \pi_*\pi^*\mathscr{G} = \pi_*\mathscr{F}'$ induces a map
\begin{equation}
\label{secondextmap}
\Ext^1_X(\mathscr{F}, \mathscr{G}) \xrightarrow{\psi} \Ext^1_X(\mathscr{F}, \pi_*\mathscr{F}')
\end{equation}
One can check that the composition $\phi \circ \psi:  \Ext^1_X(\mathscr{F}, \mathscr{G}) \rightarrow \Ext^1_{X'}(\pi^*\mathscr{F}, \pi^*\mathscr{G})$ is simply the pullback map sending the extension $\mathscr{E}$ to $\pi^*\mathscr{E}$. This uses the well-known fact that the composition $\pi^*\mathscr{G} \rightarrow \pi^*(\pi_*\pi^*\mathscr{G}) = (\pi^*\pi_*)\pi^*\mathscr{G} \rightarrow \pi^*\mathscr{G}$ is the identity, where the first map is induced by the adjunction map ${\rm{id}} \rightarrow \pi_*\pi^*$, and the second by the adjunction map $\pi^*\pi_* \rightarrow {\rm{id}}$.

We apply this with $\mathscr{F} = \Ga$, $\mathscr{G} = \Gm$. Note that $\pi^*\Ga = \Ga$, $\pi^*\Gm = \Gm$. Then we see, using the injectivity of $\phi$ mentioned above, that if an extension $E$ of $\Ga$ by $\Gm$ over $X$ splits over $X'$, then it is killed by the map
\[
\psi:  \Ext^1_X(\Ga, \Gm) \rightarrow \Ext^1_X(\Ga, \pi_*\Gm)
\]
But we also have a norm map ${\rm{Nm}}:  \pi_*\Gm \rightarrow \Gm$ defined functorially on $R$-algebras $S$ by the norm map $(S \otimes_R R')^{\times} \rightarrow S^{\times}$, which is defined by sending $s' \in S \otimes_R R'$ to the determinant of multiplication by $s'$ on the finite free $R$-algebra $S \otimes_R R'$. This induces a norm map
\[
{\rm{Nm}}:  \Ext^1_X(\Ga, \pi_*\Gm) \rightarrow \Ext^1_X(\Ga, \Gm)
\]
such that the composition
\[
\Ext^1_X(\Ga, \Gm) \xrightarrow{\psi} \Ext^1_X(\Ga, \pi_*\Gm) \xrightarrow{{\rm{Nm}}} \Ext^1_X(\Ga, \Gm)
\]
is induced by the composition $\Gm \rightarrow \pi_*\pi^*\Gm = \pi_*\Gm \xrightarrow{{\rm{Nm}}} \Gm$, which is 
the $n$th-power endomorphism. Thus, we have shown that if an element of $\Ext^1_X(\Ga, \Gm)$ dies when pulled back to $X'$, then it is $n$-torsion. But multiplication by $n$ on $\Ext^1_X(\Ga, \Gm)$ is induced by multiplication by $n$ on $(\Ga)_X$, which is an automorphism because $X$ is a $\Q$-scheme. We deduce that for Henselian local $X$, any element of $\Ext^1_X(\Ga, \Gm)$ which is killed fppf locally is already trivial.

It therefore suffices to construct a $\Q$-scheme $X$ and an element of $\Ext^1_X(\Ga, \Gm)$ that does not die \'etale locally on $X$. Here is Gabber's example. Let $X$ be two copies of the affine line glued along the non-reduced subscheme $Z : = \Spec(k[\varepsilon]/(\varepsilon^2))$ (a thickened copy of the origin). Note that this gluing process makes sense by \cite[Thm.\,3.4]{schwede}. Now we take the trivial extension $\Ga \times \Gm$ of $\Ga$ by $\Gm$ on each copy of the affine line, and then glue these along the closed subscheme $Z$ by a nontrivial automorphism of the trivial extension, again using \cite[Thm.\,3.4]{schwede}. In order to construct a suitable nontrivial automorphism of the trivial extension (as an extension), we note that such an automorphism is given by a homomorphism $\Ga \rightarrow \Gm$. Over a $\Q$-algebra $R$, examples of these are given by $T \mapsto {\rm{exp}}(rT)$ where $r \in R$ is nilpotent. (Actually, although we do not need this here, these are all such homomorphisms over a $\Q$-algebra; see Lemma \ref{formalchars}(iv).) Take the automorphism of the trivial extension on $Z$ given the homomorphism $T \mapsto {\rm{exp}}(\varepsilon T)$.

We claim that this extension $E$ of $\Ga$ by $\Gm$ over $X$ does not split over any \'etale neighborhood of the common point of the two copies of the affine line. Any \'etale $X$-scheme is reduced since $X$ is reduced. A section $s:  \Ga \rightarrow E$ over an \'etale 
$X$-scheme $U$ must be unique since any two differ by a homomorphism from $\Ga$ to $\Gm$, and there are no nontrivial such over a reduced scheme. Indeed, the only global units on $\Ga$ over a reduced scheme $Y$ are the global units of $Y$ by Lemma \ref{nonconstantunitsred}.

Now the restriction of $s$ to the preimage of the first copy $Y_1$ of the affine line must therefore be the obvious section of the trivial extension $E_{Y_1}$, hence the same holds for its restriction to $U|_Z$. The same holds for its restriction to the preimage of the second affine line, hence also to $U|_Z$. But the difference between these two restrictions to $U|_Z$ must be the 
pullback of the automorphism $T \mapsto {\rm{exp}}(\varepsilon T)$ described earlier, so
$\varepsilon$ pulls back to 0 on $U|_Z$. If the \'etale map $U \to X$ has
non-empty fiber over the common origin point then $U|_Z \to Z$ is an \'etale cover
and so the nonzero $\varepsilon$ on $Z$ cannot 
pull back to 0 on $U|_Z$. Hence, $E$ cannot split over any \'etale cover of $X$, so it cannot split fppf-locally over $X$ either.
\end{remark}

We may now finally prove the main result of this section.

\begin{proposition}
\label{ext=0}
Let $k$ be a field of characteristic $p > 0$, and let $G$ be an affine commutative $k$-group scheme of finite type. Then the fppf sheaf $\calExt^1_k(G, \mathbf{G}_m)$ vanishes.
\end{proposition}

\begin{proof}
We may replace $k$ with a finite extension, and after making such an extension $G$ is filtered by finite group schemes, $\mathbf{G}_m$'s, and $\mathbf{G}_a$'s. By the 
$\delta$-functoriality of $\calExt$ in the first variable, therefore we are reduced to these three cases. The case in which $G$ is finite or $\mathbf{G}_m$ is handled by \cite[VIII, Prop.\,3.3.1]{sga7}, so we are reduced to the case 
 $G = \mathbf{G}_a$, which follows from Proposition \ref{ext^1(Ga,Gm)=0}.
 \end{proof}
 
As we saw in Remark \ref{gabberexample}, Proposition \ref{ext=0} fails in characteristic $0$. But we may nevertheless prove a (much easier) \'etale analogue. We first prove the following proposition.

\begin{proposition}
\label{Ext^1(Ga,Gm)=0ifPic(Ga)=0}
Let $X$ be a reduced scheme such that $\Pic(\mathbf{G}_{a,X}) = 0$. Then the group $\Ext^1_X(\Ga, \Gm)$ vanishes.
\end{proposition}

\begin{proof}
Given a commutative extension $E$ of $\Ga$ by $\Gm$ over $X$, $E$ in particular is a $\Gm$-torsor over $\mathbf{G}_{a,X}$, so since $\Pic(\mathbf{G}_{a,X}) = 0$ by assumption, the map $E \rightarrow \mathbf{G}_{a,\,X}$ has a scheme-theoretic section. We need to show that there is a group-theoretic section.

Let $s:  \mathbf{G}_{a,\,X} \rightarrow E$ be a scheme-theoretic section. Replacing $s$ with $s - s(0)$, we may assume that $s(0) = 0_E$. Identifying $E$ with the trivial $\Gm$-torsor $\Gm \times \Ga$ over $\mathbf{G}_{a,\,X}$ by means of $s$, the group law on $E = \Gm \times \Ga$ is then given by
\[
(t, y) + (t', y') = (tt'h(y, y'), y + y') 
\]
for some map $h:  \mathbf{G}_{a,\,X} \times \mathbf{G}_{a,\,X} \rightarrow \mathbf{G}_{m,\,X}$. By Lemma \ref{nonconstantunitsred}, the map $h$ factors through an $X$-valued point $\lambda:  X \rightarrow \Gm$. Then the map $y \mapsto (\lambda^{-1}, y)$ yields a group-theoretic section $\mathbf{G}_{a,\,X} \rightarrow E$, so $E$ splits.
\end{proof}

\begin{proposition}
\label{ext0}
On the small \'etale site of ${\rm{Spec}}(k)$ for a perfect field $k$, the sheaf $\calExt^1(G, \mathbf{G}_m)$ for the \'etale topology vanishes
for every affine commutative $k$-group scheme $G$ of finite type. 
\end{proposition}

\begin{proof} 
We once again use Lemma \ref{extsheaf=sheafification} which tells us that the \'etale sheaf $\calExt^1_k(G, \Gm)$ is the sheafification of the presheaf $U \mapsto \Ext^i_U(G, \Gm)$. Combining Proposition \ref{affinegroupstructurethm} and Lemma \ref{almosttorus}(ii), we see that $G$ admits a filtration by finite group schemes, tori, and $\Ga$'s. Using the $\delta$-functoriality of $\calExt$ in the first variable, we are therefore reduced to these cases. When $G$ is finite or a torus, \cite[VIII, Prop.\,3.3.1]{sga7} shows that even the fppf Ext sheaf vanishes, so any extension of $G$ by $\Gm$ over $\overline{k}$ splits, hence (since $k$ is perfect), any extension of $G$ by $\Gm$ over an \'etale $k$-scheme splits \'etale-locally. In order to show that the \'etale sheaf $\calExt^1_k(\Ga, \Gm)$ is trivial, it suffices to show that any extension
\[
1 \longrightarrow \Gm \longrightarrow E \longrightarrow \Ga \longrightarrow 1
\]
over a field $L$ splits. This follows from Proposition \ref{Ext^1(Ga,Gm)=0ifPic(Ga)=0}.
\end{proof}

\section{Applications of the vanishing of $\calExt^1(\cdot, \Gm)$}
\label{applicationsofext=0section}

The first application, and the one which will play the most significant role in the present work, is the following.

\begin{proposition}
\label{hatisexact}
Let $k$ be a field, and suppose that we have a short exact sequence of affine commutative $k$-group schemes of finite type\: 
\[
1 \longrightarrow G' \longrightarrow G \longrightarrow G'' \longrightarrow 1.
\]
Then the resulting sequence of fppf dual sheaves
\[
1 \longrightarrow \widehat{G''} \longrightarrow \widehat{G} \longrightarrow \widehat{G'} \longrightarrow 1
\]
is also short exact when ${\rm{char}}(k)=p>0$, and likewise for the small \'etale site of $k$ when ${\rm{char}}(k)=0$.
\end{proposition}

\begin{proof}
Pullback yields (in the fppf or \'etale topology) an exact sequence
\[
1 \longrightarrow \widehat{G''} \longrightarrow \widehat{G} \longrightarrow \widehat{G'} \longrightarrow \calExt^1(G'', \Gm).
\]
The proposition therefore follows from Propositions \ref{ext=0} and \ref{ext0}.
\end{proof}

\begin{corollary}
\label{characterfinitecokernel}
Let $k$ be a field, and suppose that we have an inclusion $H \hookrightarrow G$ of affine commutative $k$-group schemes of finite type. Then the restriction map $\widehat{G}(k) \rightarrow \widehat{H}(k)$ has finite cokernel. If $k$ is algebraically closed, then this map is surjective.
\end{corollary}

\begin{proof}
The surjectivity assertion for algebraically closed fields is an immediate consequence of Proposition \ref{hatisexact} because the quotient group $G/H$ is affine \cite{conradaffinequotients}. To treat general fields, we first note that $\widehat{H}(k)$ is finitely generated. Indeed, in order to prove this claim we may assume that $k = \overline{k}$, hence that $G$ has a filtration by finite group schemes, $\Gm$'s, and $\Ga$'s, for all of which the claim is clear.

Returning once again to the case of general fields, the map $\widehat{G}(k) \rightarrow \widehat{H}(k)$ therefore has finitely generated cokernel. In order to show that this cokernel is finite, therefore, it suffices to show that it is torsion. That is, given $\chi \in \widehat{H}(k)$, we want to show that $\chi^n$ extends to a character of $\widehat{G}(k)$ for some positive integer $n$. Since the restriction map is surjective over $\overline{k}$, there exists $\psi \in \widehat{G}(\overline{k})$ such that $\psi|_H = \chi$.

We first claim that for suitable $r > 0$ the character $\psi^r$ descends to a character in $\widehat{G}(k_s)$. Indeed, when ${\rm{char}}(k) = 0$ we may take $r = 1$, so suppose that ${\rm{char}}(k) = p > 0$. Then $\psi$ is in particular a global unit $\psi \in \overline{k}[G]^{\times}$, where $\overline{k}[G] : = \overline{k} \otimes_k k[G]$ is the affine coordinate ring of $G_{\overline{k}}$ (and $k[G]$, of course, is the affine coordinate ring over $k$). In particular, $\psi \in k_s^{1/p^n}[G]^{\times}$ for some integer $n > 0$. It follows that $\psi^{p^n} \in k_s[G]^{\times}$, hence $\psi^{p^n}$ defines a character of $G_{k_s}$. This proves the claim.

Therefore, by replacing $\chi$ with $\chi^r$ we may assume that $\psi \in \widehat{G}(k_s)$. Thus, $\psi \in \widehat{G}(k')$ for some finite Galois extension $k'/k$ -- of degree $d$, say.
Hence, $\prod_{\sigma \in \Gal(k'/k)} \psi^{\sigma} \in \widehat{G}(k)$. But $\prod_{\sigma \in \Gal(k'/k)} \psi^{\sigma}|_H = \chi^d$, so we are done.
\end{proof}

\begin{remark}
\label{characterfinitecokerremark}
Oesterl\'e proved that if we have an inclusion $H \hookrightarrow G$ with $G$ smooth, connected, and affine (and $H$ a closed $k$-subgroup scheme), then the restriction map $\widehat{G}(k) \rightarrow \widehat{H}(k)$ has finite cokernel \cite[A.1.4]{oesterle}. That is, he avoids commutativity hypotheses, but adds in smoothness and connectedness. The smoothness may be dispensed with:  \cite[VII$_{\rm{A}}$, Prop.\,8.3]{sga3} furnishes an infinitesimal subgroup scheme $I \subset G$ such that $G/I$ is smooth, and we then have an inclusion $H/H\cap I \hookrightarrow G/I$. If $\chi \in \widehat{H}(k)$, then $\chi^n|_{H \cap I} = 1$ for some positive integer $n$, hence some power of $\chi^n$ extends to a character of $G/I$ (by the smooth connected case), hence to one of $G$, so again the cokernel is torsion and finitely generated, hence finite.

But in the absence of commutativity assumptions, the connectedness of $G$ is absolutely crucial (even when $k = \overline{k}$), as the following example due to Brian Conrad illustrates. Let $G = \Gm \rtimes \Z/2\Z$ with the nontrivial element of $\Z/2\Z$ acting on $\Gm$ by inversion, and let $H$ be the normal $\Gm$ inside of $G$. Then we claim that for any $\chi \in \widehat{G}(k)$, we have $\chi|_H = 1$. Indeed, $\chi$ is invariant under conjugation, hence due to the $\Z/2\Z$-action, $\chi|_{\Gm}$ must be invariant under inversion, hence trivial.
\end{remark}

\begin{corollary}
\label{H^1=Ext^1}
Let $k$ be a field, $G$ an affine commutative $k$-group scheme of finite type. We have a canonical isomorphism ${{\rm{H}}}^1(k, \widehat{G}) \xrightarrow{\sim} \Ext^1_k(G, \mathbf{G}_m)$.
\end{corollary}

\begin{proof} 
By Proposition \ref{ext0}, any element $E \in \Ext^1_k(G, \Gm)$ is an fppf form of the trivial extension (since that proposition implies that this holds when extending scalars to the perfection $k_{\rm{perf}}$ of $k$, hence over some finite purely inseparable extension of $k$, hence over $k$). We therefore have ${{\rm{H}}}^1(k, A) \simeq \Ext^1(G, \mathbf{G}_m)$, where $A$ is the automorphism sheaf of the trivial extension; that is, $A$ is the sheaf whose sections over a $k$-algebra $R$ are the automorphisms of $\mathbf{G}_m \times G$ as an extension:  the group of isomorphisms $\phi:  \mathbf{G}_m \times G \rightarrow \mathbf{G}_m \times G$ of $R$-group schemes
 such that the following diagram commutes: 
\[
\begin{tikzcd}
1 \arrow{r} & \mathbf{G}_m \arrow[d, equals] \arrow{r} & \mathbf{G}_m \times G \arrow{d}{\phi} \arrow{r} & G \arrow[d, equals] \arrow{r} & 1 \\
1 \arrow{r} & \mathbf{G}_m \arrow{r} & \mathbf{G}_m \times G \arrow{r} & G \arrow{r} & 1
\end{tikzcd}
\]
We have a canonical identification $A \simeq \widehat{G}$, hence we get an isomorphism 
 ${{\rm{H}}}^1(k, \widehat{G}) \simeq \Ext^1(G, \mathbf{G}_m)$, canonical up to a universal choice of sign.
\end{proof}

Let us also use Proposition \ref{hatisexact} to prove the following result, which we will use in chapter \ref{chapterlocalfields}.

\begin{proposition}
\label{hatrepresentable}
Let $k$ be a field, $G$ an affine commutative $k$-group scheme of finite type. Then the fppf sheaf $\widehat{G}$ is representable if and only if $G$ is an almost-torus, in which case it is locally finite and its \'etale component group has finitely-generated group of $k_s$-points. 
\end{proposition}

\begin{proof}
First suppose that $G$ is an almost-torus. Then by Lemma \ref{almosttorus}(iii), there is an isogeny $T \times A \twoheadrightarrow G$ with $A$ a finite $k$-group scheme and $T$ a $k$-torus. Let $B$ denote its (finite) kernel. Then we have an exact sequence
\[
1 \longrightarrow B \longrightarrow T \times A \longrightarrow G \longrightarrow 1
\]
hence an exact sequence
\[
1 \longrightarrow \widehat{G} \longrightarrow \widehat{T} \times \widehat{A} \longrightarrow \widehat{B}
\]
As is well-known, the dual sheaves of finite commutative group schemes and of tori are representable by commutative finite type group schemes that are locally finite and whose component groups have finitely-generated group of $k_s$-points, hence $\widehat{G}$ is the kernel of a homomorphism of such $k$-group schemes, hence is also of this form.

Now suppose that $G$ is not an almost-torus, and we will show that $\widehat{G}$ is not representable. By Lemma \ref{affinegroupstructurethm}, there is an almost-torus $H \subset G$ such that the quotient $U : = G/H$ is split unipotent
(and nontrivial!). We then have an exact sequence
\[
1 \longrightarrow \widehat{U} \longrightarrow \widehat{G} \longrightarrow \widehat{H}
\]
so since $\widehat{H}$ is representable, if $\widehat{G}$ is representable, then so is $\widehat{U}$, as it is then the kernel of a morphism of $k$-group schemes. We may therefore assume that $G = U$ is nontrivial split unipotent.

So suppose that $\widehat{U}$ is representable. We have $\widehat{U}(\varinjlim \R_i) = \varinjlim \widehat{U}(\R_i)$ for any filtered directed system of $k$-algebras $\R_i$, so by \cite[IV$_3$, Prop.\,8.14.2]{ega}, $\widehat{U}$ is locally of finite type. But $\widehat{U}$ has no nontrivial field-valued points, so it must be infinitesimal.

If ${\rm{char}}(k) = 0$, then this means that $\widehat{U} = 0$. But this is false:  there is a surjection $U \twoheadrightarrow \mathbf{G}_a$, hence an inclusion $\widehat{\mathbf{G}_a} \hookrightarrow \widehat{U}$, and $\widehat{\mathbf{G}_a} \neq 0$. Indeed, over any ring $R$, the group $\widehat{\mathbf{G}_a}(R)$ includes all of the global units exp$(aX) = \sum_{n=0}^\infty (aX)^n/n! \in R[X] = {{\rm{H}}}^0(\mathbf{G}_{a,R}, \mathbf{G}_m)$ with $a \in R$ nilpotent. (Actually, these are all of the elements of $\widehat{\mathbf{G}_a}(R)$ by Lemma \ref{formalchars}(iv).) So $\widehat{\mathbf{G}_a}(R) \neq 0$ for any non-reduced $k$-algebra $R$.

Next suppose that ${\rm{char}}(k) > 0$. Then the infinitesimal $\widehat{U}$ is finite. But there is an inclusion $\mathbf{G}_a \hookrightarrow U$, hence an inclusion $\alpha_{p^n} \hookrightarrow U$ for each $n > 0$. By Proposition \ref{hatisexact}, therefore, there is a surjection $\widehat{U} \twoheadrightarrow \widehat{\alpha_{p^n}}$ for each positive integer $n$. This is impossible 
since $\widehat{U}$ is finite.
\end{proof}

\section{Double duality}
\label{doubledualitysection}

In accordance with the philosophy that our results should be symmetric in $G$ and $\widehat{G}$, the purpose of this section is to prove that the canonical map 
\begin{equation}
\label{doubledualityeqn3}
G \rightarrow G^{\wedge\wedge}
\end{equation}
is an isomorphism of fppf sheaves for any affine commutative group scheme $G$ of finite type over a field (Proposition \ref{doubleduality}). We will never use this, so the reader may skip this section.

It is well-known (part of Cartier duality) that the map (\ref{doubledualityeqn3}) is an isomorphism when $G$ is a finite commutative group scheme, and for $G = \Gm$, one has $\widehat{G} = \Z$, and the assertion is clear. The most difficult case is when $G = \Ga$.

\begin{lemma}
\label{doubledualityGa}
Let $X$ be either a $\Q$-scheme or an $\F_p$-scheme. Then the canonical map $\Ga \rightarrow \Ga^{\wedge\wedge}$ is an isomorphism of fppf sheaves on $X$.
\end{lemma}

\begin{proof}
We may assume that $X= \Spec(R)$. First we treat the case when $R$ is a $\Q$-algebra. Choose an $R$-morphism $\phi:  \widehat{\mathbf{G}_{a,\,R}} \rightarrow \Gm$
(i.e., an element of $\Ga^{\wedge\wedge}(R)$) for a $\Q$-algebra $R$; in particular, $\phi(1) = 1$. By Lemma \ref{formalchars}(iv), we need to show that there is a unique $r \in R$ such that $\phi(\exp(aX)) = \sum_{n \geq 0} a^nr^n/n!$ for any $R$-algebra $A$ and any nilpotent $a \in A$. The uniqueness of $r$ -- if it exists -- follows from taking $A : = R[\epsilon]/(\epsilon^2)$ and $a : = \epsilon$.

Let $A_n : = R[T]/(T^n)$, and let $f_n(T) : = \phi(\exp(TX)) \in (R[T]/(T^n))^{\times}$. (Then $\exp(TX)$ is the universal character obtained from nilpotents of order $\le n$.) 
For any $a \in A$ satisfying $a^n = 0$ there is a unique $R$-algebra map $R[T]/(T^n) \rightarrow A$ satisfying $T \mapsto a$, so 
naturality of $\phi$ as a map of functors implies $\phi(\exp(aX)) = f_n(a)$ for any such $a \in A$. Further use of naturality of $\phi$ implies
the ``coherence condition'' 
$f_{n+1} \bmod T^n = f_n$, so the $f_n$'s arise as the reductions of  a single $f \in 1 + TR\llbracket T \rrbracket$.
(We have $f(0) = 1$ because $\phi(1) = 1$.) 
In particular, $\phi(\exp(aX)) = f(a)$ for any $R$-algebra $A$ and nilpotent $a \in A$. 

For any $R$-algebra $A$ and nilpotent elements $a, b \in A$, we have 
$$f(a + b) = \phi(\exp((a+b)X)) = \phi(\exp(aX)\exp(bX)) = \phi(\exp(aX))\phi(\exp(bX)) = f(a)f(b).$$ 
It follows that $f(S + T) = f(S)f(T)$ in $A[\![S, T]\!]$. This says that $f$ is a formal character. It follows from Lemma \ref{formalchars}(iv) that $f(T) = \exp(rT)$ for some $r \in R$, so $\phi(\exp(aX)) = f(a) = \exp(ra)$ for any nilpotent $a$
belonging to any $R$-algebra $A$, which is what we wanted to show.

Now consider the case when $R$ is an $\F_p$-algebra. This is similar to the case of $\Q$-algebras treated above, but more complicated.  Once again, choose $\phi \in \Ga^{\wedge\wedge}(R)$. By Lemma \ref{formalchars}(ii), we need to find a unique $r \in R$ such that $\phi(\prod_{n=0}^N \exp_p(a_nX^{p^n})) = \prod_{n=0}^N \exp_p(a_nr^{p^n})$ for all 
$R$-algebras $A$ and all $a_n \in \alpha_p(A)$. Uniqueness of $r$ -- if it exists -- is seen by taking $A : = R[\epsilon]/(\epsilon^2)$ and considering $\phi(\exp_p(\epsilon X)) = 1 + r\epsilon \in A^{\times}$.

In order to prove existence, define for $N \geq 0$, $A_N : = R[T_0, T_1, \dots, T_N]/(T_0^p, \dots, T_N^p)$ and $$f_N(T_0, \dots, T_N) : = \phi(\prod_{n=0}^N \exp_p(T_nX^{p^n})) \in A_N^{\times}.$$ Note that $\prod_{n=0}^N \exp_p(T_nX^{p^n})$ is the ``universal character of degree $< p^{N+1}$''. For any $R$-algebra $A$ and $a_0, \dots, a_N \in \alpha_p(A)$ we have $\phi(\prod_{n=0}^N \exp_p(a_nX^{p^n})) = f_N(a_0, \dots, a_N)$. To proceed, it is convenient to introduce the completion $B$ of 
$R[T_0, T_1, \dots]/(T_0^p, T_1^p, \dots)$ for the topology defined by the decreasing sequence of ideals
$J_N = (T_{N+1}, T_{N+2}, \dots)$.  Elements of $B$ have a unique expansion (that can be manipulated $R$-linearly) as formal series
$$f(T_0, T_1, \dots) = \sum_I r_I T^I$$ where each multi-index $I = (i_0, i_1, \dots)$ satisfies $0 \le i_j \le p-1$ for all $j$, 
$i_j=0$ for all but finitely many $j$ (so only finitely many such $I$ exist with $i_j$'s vanishing for
$j$ outside a fixed finite set), and $T^I : = \prod_{j \ge 0} T_j^{i_j}$. Since $f_{N+1} \bmod T_{N+1} = f_N$ by naturality of $\phi$, 
the $f_N$'s are reductions of a single $f \in B$.  
(We will only ever actually need to consider $f$ after specializing all $T_m$ to be 0 for sufficiently large $m$, but 
to streamline notation it is convenient to express
the subsequent considerations in terms of the single $f$.)

Since $\phi(1) = 1$, we have $f(0, 0, \dots) = 1$. Thus, $f = 1 + \sum_{I \neq 0} r_I T^I$. We have $\phi(\prod_{n=0}^N \exp_p(a_nX^{p^n})) = f(a_0, a_1, \dots, a_N, 0, 0, \dots)$ for any $R$-algebra $A$ and for any elements $a_0, a_1, \dots, a_N  \in \alpha_p(A)$. 
By multiplicativity, $\phi$ is determined by where it sends characters of the form $\exp_p(aX^{p^n})$ for $R$-algebras $A$ and 
$a \in \alpha_p(A)$. For $A = R[T_n]/(T_n^p)$ and $a = T_n \bmod T_n^p \in A$, we
have $\phi(\exp_p(T_n X^{p^n})) = f(\dots, 0, T_n, 0, \dots) \in (R[T_n]/(T_n^p))^{\times}$, where one evaluates $f$ at the infinity-tuple with all entries $0$ except $T_n$. So 
$\phi$ is determined by the $r_I$'s for $I = (i_0, i_2, \dots)$ such that $i_j = 0$ for all but one $j$ (and the remaining
$i_n$ belonging to $\{1, \dots, p-1\}$). 

Let $C(n, m) : = r_{I(n,m)}$, where $I(n,m)$ has its $n$th component  
equal to $m \in \{0, \dots, p-1\}$ and all of the other components equal to 0; i.e., $C(n, m)$ is the $T_n^m$-coefficient of $f$.  We will now prove:  
\begin{itemize}
\item[(i)] $C(n, m) = C(n, m-1)C(n, 1)/m$ if $1 \leq m \le p-1$, $n \ge 0$, 
\item[(ii)] $C(n+1, 1) = (p-1)!C(n, 1)C(n, p-1)$ for all $n \geq 0$.
\end{itemize}
This will imply what we want. Indeed, these facts imply that the $C(n, m)$'s are determined by $C(0,1)$, so $\phi$ is determined once we specify the coefficient $r$ of $T_0$ in $f_0(T_0) \in R[T_0]/(T_0^p)$. But the image $\phi_r$ of $r$ under the natural map $\Ga \rightarrow \Ga^{\wedge\wedge}$ sends $\exp(T_0X)$ to $\exp(T_0r)$; i.e., its ``$f_0$'' has linear coefficient $r$. Since any $\phi$ is determined by this coefficient, it follows that $\phi = \phi_r$, which is what we wanted. It therefore only remains to prove (i) and (ii) above.

Let $A : = R[S,T]/(S^p, T^p)$.  Viewing $S$ and $T$ as elements of $A$, we have 
$$\exp_p(SX^{p^n})\exp_p(TX^{p^n}) = \exp_p((S + T)X^{p^n})\exp_p\left(\left(\sum_{i=1}^{p-1} \frac{S^iT^{p-i}}{i!(p-i)!}\right)X^{p^{n+1}}\right)$$
in $A[X]^{\times}$; this equality may be checked directly, but the simplest way to see it is to note that both sides are formal characters over $A$ 
such that the coefficients of $X^{p^m}$ agree
 for all $m$, vanishing except possibly for $m = n, n+1$, and then to apply Lemma \ref{formalchars}. To exploit this identity, it is convenient to
 introduce some notation as follows.  For $a, b \in A = R[S,T]/(S^p, T^p)$ we define 
 $f(a_n)$ to be $f(0, 0, \dots, a, 0, \dots) \in A^{\times}$ (all entries vanishing away from the $n$th, which is $a$) and define
 $f(a_n, b_{n+1})$ to be the evaluation of $f$ on the vector whose $n$th entry is $a$, whose $(n+1)$th entry is $b$, and 
 whose other entries vanish.
 In terms of this notation, we have 
\begin{eqnarray*}
f(S_n)f(T_n) &=& \phi(\exp_p(SX^{p^n}))\phi(\exp_p(TX^{p^n})) \\
&=& \phi \left(\exp_p(SX^{p^n})\exp_p(TX^{p^n})\right) \\
&=& \phi \left(\exp_p((S + T)X^{p^n})\exp_p\left(\left(\sum_{i=1}^{p-1} \frac{S^iT^{p-i}}{i!(p-i)!}\right)X^{p^{n+1}}\right)\right)\\
&=& f\left((S+T)_n, \left(\sum_{i=1}^{p-1} \frac{S^iT^{p-i}}{i!(p-i)!}\right)_{n+1}\right)
\end{eqnarray*}
(the final equality by the design of $f$). For $1 \le m \le p-1$, comparing coefficients $S^{m-1}T$ 
in the first and last expressions for this string of equalities yields (i) and comparing coefficients of $ST^{p-1}$ yields (ii).
\end{proof}

\begin{lemma}
\label{doubledualitydevissage}
Suppose that we have a short exact sequence
\[
1 \longrightarrow G' \longrightarrow G \longrightarrow G'' \longrightarrow 1
\]
of commutative group schemes over a field $k$ such that the fppf sheaf $\calExt^1_k(G'', \Gm)$ vanishes. If the canonical maps $G' \rightarrow G'^{\wedge\wedge}$ and $G'' \rightarrow G''^{\wedge\wedge}$ are isomorphisms of fppf sheaves, then so is the map $G \rightarrow G^{\wedge\wedge}$. 
\end{lemma}

\begin{proof}
Our assumption on $\calExt^1_k(G'', \Gm)$ implies that the dual sequence
\[
1 \longrightarrow \widehat{G''} \longrightarrow \widehat{G} \longrightarrow \widehat{G'} \longrightarrow 1
\]
is exact. Dualizing once more yields a left-exact sequence
\[
1 \longrightarrow G'^{\wedge\wedge} \longrightarrow G^{\wedge\wedge} \longrightarrow G''^{\wedge\wedge}, 
\]
so we obtain a commutative diagram (of sheaves) with exact rows
\[
\begin{tikzcd}
1 \arrow{r} & G' \arrow{r} \isoarrow{d} & G \arrow{r} \arrow{d} & G'' \arrow{r} \isoarrow{d} & 1 \\
1 \arrow{r} & G'^{\wedge\wedge} \arrow{r} & G^{\wedge\wedge} \arrow{r} & G''^{\wedge\wedge} &
\end{tikzcd}
\]
where the left and right vertical arrows are isomorphisms by assumption. A simple diagram chase now shows that the middle vertical arrow is an isomorphism. 
\end{proof}

\begin{proposition}
\label{doubleduality}
Let $k$ be a field, $G$ an affine commutative $k$-group scheme of finite type. Then the canonical map $G \rightarrow G^{\wedge\wedge}$ is an isomorphism of fppf sheaves.
\end{proposition}

\begin{proof}
By Lemma \ref{doubledualitydevissage} and the fact that $\calExt^1_k(E, \Gm) = 0$ when $E$ is a finite $k$-group scheme \cite[VIII, Prop.\,3.3.1]{sga7}, we may replace $G$ with $G^0$ and thereby assume that $G$ is connected. The assertion is fppf local, so we may also replace $k$ by a finite extension and so assume that $G_{\red} \subset G$ is a smooth $k$-subgroup scheme. Then $G/G_{\red}$ is finite, so applying Lemma \ref{doubledualitydevissage} again, we may assume that $G$ is smooth and connected. Replacing $k$ by a further finite extension, we may assume that $G$ is the product of a split torus and a split unipotent group. The map
\[
\Gm \rightarrow \Gm^{\wedge\wedge} = \Z^{\wedge}
\]
is an isomorphism, so we are left with the case when $G = U$ is split unipotent. When ${\rm{char}}(k) = 0$, we have $U \simeq \Ga^n$ for some $n$ \cite[Prop.\,14.32]{milnealggroups}, so we are done by Lemma \ref{doubledualityGa}. When ${\rm{char}}(k) > 0$, Lemma \ref{doubledualitydevissage} and Proposition \ref{ext=0} reduce us to the case $U = \Ga$, which is once again handled by Lemma \ref{doubledualityGa}.
\end{proof}

\section{Cohomology of $\widehat{\mathbf{G}_a}$}
\label{cohomGahat}

The crucial cases for the proofs of our results in a certain sense boil down to the groups $\mathbf{G}_a$ and $\mathbf{G}_m$, which are the fundamental building blocks for arbitrary affine commutative group schemes of finite type over fields. The main cohomological results in the case $G = \mathbf{G}_m$ essentially come down to the major statements of class field theory. There is no analogous theory, however, in the case of $\mathbf{G}_a$, and it is therefore necessary for us to undertake a separate study of the cohomology of its (not even representable) dual sheaf $\widehat{\mathbf{G}_a}$. That  is the object of this section. 

We first note that when $k$ is perfect, the cohomology of $\widehat{\mathbf{G}_a}$ is very simple: 

\begin{proposition}
\label{cohomologyofG_adualwhenkisperfect}
If $k$ is a perfect field, then ${{\rm{H}}}^i(k, \widehat{\mathbf{G}_a}) = 0$ for all $i$.
\end{proposition}

\begin{proof}
By Proposition \ref{etrem}, ${{\rm{H}}}^i(k, \widehat{\mathbf{G}_a}) = {{\rm{H}}}^i_{\et}(k, \widehat{\mathbf{G}_a})$. But the sheaf 
$\widehat{\mathbf{G}_a}$ on the small \'etale site of a field vanishes
since $\mathbf{G}_a$ has no nontrivial characters over a field. 
\end{proof}

\begin{proposition}
\label{cohomologyofG_adualgeneralk}
Let $X$ be a reduced scheme.
\begin{itemize}
\item[(i)] ${\rm{H}}^0(X, \widehat{\Ga}) = 0$.
\item[(ii)] If in addition $\Pic(\mathbf{G}_{a,X}) = 0$, then ${\rm{H}}^1(X, \widehat{\Ga}) = 0$.
\end{itemize}
\end{proposition}

\begin{proof}
(i) An $X$-homomorphism $\mathbf{G}_{a,\,X} \rightarrow \mathbf{G}_{m,\,X}$ yields in particular a global unit $u$ on $\mathbf{G}_{a,\,X}$ whose restriction to the identity section is $1$. By Lemma \ref{nonconstantunitsred}, it follows that this unit must be $1$. That is, the homomorphism is trivial.
\newline
\newline
(ii) The Leray spectral sequence $E_2^{i,j} = {\rm{H}}^i(X, \calExt^j(\Ga, \Gm)) \Longrightarrow \Ext^{i+j}_X(\Ga, \Gm)$ yields an injection $E_2^{1,0} = {\rm{H}}^1(X, \widehat{\Ga}) \hookrightarrow \Ext^1_X(\Ga, \Gm)$, so the assertion follows from Proposition \ref{Ext^1(Ga,Gm)=0ifPic(Ga)=0}.
\end{proof}

\begin{proposition}
\label{unipotentcohomology}
Let $k$ be a field, $U$ a smooth connected commutative unipotent $k$-group. 
\begin{itemize}
\item[(i)] If $i > 1$, then ${{\rm{H}}}^i(k, U) = 0$.
\item[(ii)] If $U$ is split, then ${{\rm{H}}}^i(k, U) = 0$ for all $i > 0$.
\item[(iii)] If $U$ is split, then ${{\rm{H}}}^1(k, \widehat{U}) = 0$.
\end{itemize}
\end{proposition} 

\begin{proof}
Assertion (ii) follows from the well-known case $U = \mathbf{G}_a$ via filtering $U$ by $\mathbf{G}_a$'s over $k$.
To prove (i), we first claim that there is a $k$-group inclusion $U \hookrightarrow U'$ for some split unipotent $k$-group $U'$. Indeed, let $k'/k$ be a finite extension over which $U$ splits. Then the canonical inclusion $U \hookrightarrow \R_{k'/k}(U_{k'})$ does the job. (The Weil restriction is split since $\R_{k'/k}(\Ga) \simeq \Ga^{[k': k]}$.) 
The quotient $U'' : = U'/U$ is then also necessarily split, so the exact sequence
\[
1 \longrightarrow U \longrightarrow U' \longrightarrow U'' \longrightarrow 1
\]
reduces the vanishing of ${{\rm{H}}}^i(k,U)$ for $i > 1$ to the vanishing in positive degrees in the split case as in the settled 
assertion (ii).

Finally, (iii) follows immediately from Propositions \ref{cohomologyofG_adualgeneralk} and \ref{hatisexact}, as well as Proposition \ref{etrem} when ${\rm{char}}(k) = 0$, by filtering $U$ by $\Ga$. 
\end{proof}

As we shall see later (Proposition \ref{H^2=Ext^2=Br} and Corollary \ref{omegaH^2injective}), ${{\rm{H}}}^2(k, \widehat{\mathbf{G}_a})$ is nontrivial for imperfect fields $k$. We will spend the rest of this section studying this cohomology group, and in particular relating it to other groups that may be accessed more directly. But first let us note a property of these cohomology groups that we shall require later.

\begin{lemma}
\label{H^2(k,Ga^)injectiveseparable}
Let $L/k$ be a (not necessarily algebraic) separable extension of fields. Then the pullback map ${\rm{H}}^2(k, \widehat{\Ga}) \rightarrow {\rm{H}}^2(L, \widehat{\Ga})$ is injective.
\end{lemma}

\begin{proof}
Since $L/k$ is separable, $L$ is the direct limit of its smooth $k$-subalgebras. Specializing to a separable point of a suitable such algebra, we see that it suffices to treat the case in which $k = k_s$. Let $\mathfrak{g} : = \Gal(k_s/k)$. Then we have a Hochschild--Serre spectral sequence
\[
E_2^{i,j} = {\rm{H}}^i(\mathfrak{g}, {\rm{H}}^j(k_s, \widehat{\Ga})) \Longrightarrow {\rm{H}}^{i+j}(k, \widehat{\Ga}),
\]
and the edge map ${\rm{H}}^2(k, \widehat{\Ga}) = E_2 \rightarrow E_2^{0,2} = {\rm{H}}^2(k_s, \widehat{\Ga})^{\mathfrak{g}}$ is the pullback map of the lemma. In order to show that this map is injective, it suffices to show that $E_2^{2,0} = E_2^{1,1} = 0$. The vanishing of both of these groups follows from Proposition \ref{cohomologyofG_adualgeneralk} applied with $X = \Spec(k)$.
\end{proof}

\begin{definition}
\label{primitivedefinition}
If $G$ is a commutative group scheme, then an element $\alpha \in {\rm{H}}^2(G, \Gm)$ is called 
{\em primitive} if $m^*\alpha = \pi_1^*\alpha + \pi_2^*\alpha$, where $m, \pi_i:  G \times G \rightarrow G$ are the multiplication and projection maps, respectively. We denote the subgroup of primitive elements by ${\rm{H}}^2(G, \Gm)_{\prim}$.
\end{definition}

\begin{remark}
\label{linearactiononBrauerprim}
Note that over any ring $A$, $\mathbf{G}_{a,\,A}$ has a natural $A$-linear action, given on $R$-valued points (for an $A$-algebra $R$) by $a \cdot r = ar$ for $a \in A$, $r \in \Ga(R) = R$. This yields a ``multiplicative action'' of $A$ on any functor evaluated at $\Ga$. For certain such functors valued in abelian groups, this action is also additive. That is, we obtain an $A$-module structure on the associated functor evaluated at $\Ga$. This is the case, for example, for the groups ${\rm{H}}^i(A, \widehat{\Ga})$ and $\Ext^i_A(\Ga, \Gm)$, by general nonsense, since $m_{a+b} = m_a + m_b$, where $m_a:  \Ga \rightarrow \Ga$ is multiplication by $a \in A$. We do not obtain an $A$-linear action on ${\rm{H}}^2(\mathbf{G}_{a,\,A}, \Gm)$ in general. We do, however, obtain one on ${\rm{H}}^2(\mathbf{G}_{a,\,A}, \Gm)_{\prim}$. Indeed, the $A$-linearity of the action follows from the fact that for $X \in {\rm{H}}^2(\Ga, \Gm)_{\prim}$, pulling back the equality $m^*X = p_1^*X + p_2^*X$ along the map $\Ga \rightarrow \Ga^2$ given by $y \mapsto (a_1y, a_2y)$ yields $(a_1 + a_2)^*X = a_1^*X + a_2^*X$. More generally, the same reasoning yields, for any scheme $S$, a functorial (in $S$) $\Gamma(S, \calO_S)$-module structure on ${\rm{H}}^2(\mathbf{G}_{a,\,S}, \Gm)_{\prim}$.
\end{remark}

\begin{remark}
\label{nolinearaction}
In fact, if $k$ is an imperfect field, then there is no $k$-linear action on ${\rm{H}}^2(\mathbf{G}_{a,\,k}, \Gm)$. (If $k$ is perfect, then ${\rm{H}}^2(\mathbf{G}_{a,\,k}, \Gm) = {\rm{H}}^2(k, \Gm)$, via pullback along the structure map. This can be deduced from the case $k = \overline{k}$ \cite[Cor.\,1.2]{briii} by using a Hochschild--Serre spectral sequence \cite[Cor.\,6.7.8]{poonen}. Note that while \cite{poonen} states the result for proper varieties, it only uses properness in order to deduce -- in the notation of that corollary -- that ${\rm{H}}^0(X^s, \Gm) = k_s^{\times}$, a statement which still holds for $X = \Ga$.) Indeed, by Corollary \ref{omegaH^2injective}, ${\rm{H}}^2(\mathbf{G}_{a,\,k}, \Gm)[p^\infty] \neq 0$, while \cite[\S 3, Prop.\,]{treger} shows that ${\rm{H}}^2(\mathbf{G}_{a,\,k}, \Gm)[p^\infty]$ is $p$-divisible. (Note that the definition of the Brauer group in \cite{treger} is the Azumaya Brauer group rather than the cohomological Brauer group, which is why we need the cited theorem to ensure that for $p$-torsion classes they agree.) It follows that ${\rm{H}}^2(\mathbf{G}_{a,\,k}, \Gm)$ is not $p$-torsion, hence admits no $k$-vector space structure. We will never use this.
\end{remark}

Let $X$ be a scheme. In order to understand the groups ${\rm{H}}^2(X, \widehat{\Ga})$, we will relate them to other groups that are easier to try to analyze directly. We will accomplish this by constructing maps
\begin{equation}
\label{mapsH^2ExtBr}
{\rm{H}}^2(X, \widehat{\Ga}) \rightarrow \Ext^2_X(\Ga, \Gm) \rightarrow {\rm{H}}^2(\mathbf{G}_{a,\,X}, \Gm)_{\prim}.
\end{equation}
functorial in $X$ and in $X$-morphisms of $\Ga$. First, the Leray spectral sequence
\[
D_2^{i, j} = {\rm{H}}^i(X, \calExt^j_X(\Ga, \Gm)) \Longrightarrow \Ext^{i+j}_X(\Ga, \Gm)
\]
yields an edge map
\[
D_2^{2, 0} = {\rm{H}}^2(X, \widehat{\Ga}) \rightarrow \Ext^2_X(\Ga, \Gm),
\]
which yields the first map in (\ref{mapsH^2ExtBr}).

\begin{proposition}
\label{H^2(X,G_a^)-->Ext^2(Ga,Gm)injection}
If $X$ is an $\F_p$-scheme, then the map ${\rm{H}}^2(X, \widehat{\Ga}) \rightarrow \Ext^2_X(\Ga, \Gm)$ in $(\ref{mapsH^2ExtBr})$ is injective.
\end{proposition}

\begin{proof}
The map in question is the edge map $D_2^{2,0} \rightarrow D_2$ in the Leray spectral sequence
\[
D_2^{i,j} = {\rm{H}}^i(X, \calExt^j_X(\Ga, \Gm)) \Longrightarrow \Ext^{i+j}_X(\Ga, \Gm).
\]
In order to show that this edge map is injective, it suffices to show that the group $D_2^{0,1} = {\rm{H}}^0(X, \calExt^1_X(\Ga, \Gm))$ vanishes, and this follows from Proposition \ref{ext^1(Ga,Gm)=0}.
\end{proof}

In order to define the second map in (\ref{mapsH^2ExtBr}), we note that for any commutative group scheme $G$ over $X$, Yoneda's Lemma yields a natural transformation of functors from the category of fppf abelian sheaves on $G$ to the category of abelian groups
\[
\Hom_X(G, \cdot) \rightarrow {\rm{H}}^0(G, \cdot).
\]
Further, this map factors through ${\rm{H}}^0(G, \cdot)_{\prim}$, hence so does the induced map of derived functors: 
\begin{equation}
\label{Yonedamapdef}
\Ext^i_X(G, \cdot) \rightarrow {\rm{H}}^i(G, \cdot)_{\prim}.
\end{equation}
Specializing to the case $G = \Ga$ and evaluating at $\Gm$ yields the second map in (\ref{mapsH^2ExtBr}). Our goal will be to show that in certain favorable situations (especially when $X$ is the spectrum of a field) these maps are isomorphisms; see Proposition \ref{H^2=Ext^2=Br}. In order to do this we will make essential use of some spectral sequences constructed by Breen, so we discuss these in Remark \ref{breensequences} below.

\begin{remark}
\label{breensequences}
We now discuss a few spectral sequences constructed by Breen that play an important role in our analysis of the cohomology of $\widehat{\Ga}$. For more details, see \cite[\S1]{breen} and especially \cite[page 1250]{breen2}. Let $S$ be a scheme. Associated to any commutative $S$-group scheme $G$, there is a complex $A(G) = A(G)_{\bullet}$ of fppf abelian sheaves concentrated in nonnegative degrees such that each term of $A(G)$ is a product of sheaves of the form $\Z[G^n]$ (the sheaf freely generated by $G^n$). Further, $A(G)_0 = \Z[G]$, the canonical map $G \rightarrow A(G)_0$ induces an isomorphism 
\begin{equation}
\label{G=H_0(A)}
G \simeq {\rm{H}}_0(A(G)),
\end{equation}
and we have 
\begin{equation}
\label{H_1(A)=0}
{\rm{H}}_1(A(G)) = 0 \hspace{.3 in}
{\rm{H}}_2(A(G)) = G/2G.
\end{equation}
We also have
\[
A(G)_1 = \Z[G^2],
\]
and the differential $A(G)_1 = \Z[G^2] \rightarrow \Z[G] = A(G)_0$ is the map induced by $m - p_1 - p_2:  G \times G \rightarrow G$, where, as before, $m, p_i:  G \times G \rightarrow G$ are the multiplication and projection maps, respectively.

Breen obtains first a spectral sequence
\[
F_1^{i, j} = \Ext^j(A(G)_i, H) \Longrightarrow \Ext^{i+j}(A(G), H).
\]
We have a canonical isomorphism $\Ext^j(\Z[G], \cdot) \simeq {{\rm{H}}}^j(G, \cdot)$ of functors
on fppf abelian sheaves, since both are the derived functors of $\Gamma(G, \cdot)$, thanks to Yoneda's Lemma. Thus, the sequence above becomes
\begin{equation}
\label{F_1specseq22}
F_1^{i,j} = {{\rm{H}}}^j(X_i, H) \Longrightarrow \Ext^{i+j}(A(G), H),
\end{equation}
where $X_i$ is some explicit disjoint union of products of copies of $G$, and in particular, 
\begin{equation}
\label{X_idescription}
X_0 = G \hspace{.3 in}  X_1 = G^2 \hspace{.3 in} X_2 = G^2 \coprod G^3
\end{equation}
with the differentials $F_1^{0, j} \rightarrow F_1^{1, j}$ being 
\begin{equation}
\label{breenX_0to1differential}
m^* - \pi_1^* - \pi_2^*,
\end{equation}
where $m, \pi_i:  G \times G \rightarrow G$ are the multiplication and projection maps, and the differential $F_1^{1, j} \rightarrow F_1^{2,j}$ being
\begin{equation}
\label{breenX_1to2differential}
(1^* - \sigma^*) \times (\pi^3_1, m\circ \pi_{23})^* + \pi_{23}^* - (m\circ \pi_{12}, \pi^3_3)^* - \pi_{12}^*,
\end{equation}
where $\sigma:  G \times G \rightarrow G$ is the switching map $(x, y) \mapsto (y, x)$, $1$ is the identity of $G \times G$, $\pi^3_i:  G^3 \rightarrow G$ is projection onto the $i$th factor, $\pi_{ij}:  G^3 \rightarrow G^2$ is projection onto the $i$th and $j$th factors, and $m:  G^2 \rightarrow G$ is once again multiplication.

In fact, Breen shows that we may replace the above sequence with another one that is somewhat more convenient, involving ``reduced'' cohomology groups $\widetilde{\rm{H}}^j(X_i, H)$ defined as follows. Let $Y_i$ be the analogue of $X_i$ for the $0$ group; that is, $Y_i$ is a corepresenting object for $\Hom(A(0)_i, \cdot)$ (so $Y_i$ is a disjoint union of copies of $S$). Then via the identity section $S \rightarrow G$, we obtain maps $Y_i \rightarrow X_i$, and we define $\widetilde{\rm{H}}^j(X_i, H) : = \ker({{\rm{H}}}^j(X_i, H) \rightarrow {{\rm{H}}}^j(Y_i, H))$ to be the kernel of the induced map on cohomology. Breen proved that these reduced cohomology groups provide a spectral sequence analogous to (\ref{F_1specseq22}) and with the same abutment. That is, we have a spectral sequence
\begin{equation}
\label{E_1specseqbreen20}
E_1^{i,j} = \widetilde{{{\rm{H}}}}^j(X_i, H) \Longrightarrow \Ext^{i+j}(A(G), H).
\end{equation}
Let us note in particular that the differential $\widetilde{{{\rm{H}}}}^j(G, H) = E_1^{0,j} \rightarrow E_1^{1, j}  = \widetilde{{{\rm{H}}}}^j(G \times G, H)$ is given by $m^* - \pi_1^* - \pi_2^*$, so we have
\begin{equation}
\label{E_2^0j=primcohom}
E_2^{0, j} = {\rm{H}}^j(G, H)_{\prim},
\end{equation}
where the $\widetilde{{\rm{H}}}^j$ may be replaced with just ${\rm{H}}^j$ because ${\rm{H}}^j(0_S, H)_{\prim} = 0$ where $0_S$ is the trivial $S$-group scheme.

The second spectral sequence constructed by Breen takes the following form: 
\begin{equation}
\label{E_2spectralseqbreen21}
{}^{\prime}E_2^{i,j} = \Ext^i_S({\rm{H}}_j(A(G)), H) \Longrightarrow \Ext^{i+j}_S(A(G), H).
\end{equation}
Finally, the composition of the edge maps ${}^{\prime}E_2^{i,0} = \Ext_S^i(G, H) \rightarrow \Ext_S^i(A(G), H) \rightarrow E_2^{0, i} = {\rm{H}}^i(G, H)_{\prim}$ (the last equality by (\ref{E_2^0j=primcohom})) coming from the sequences (\ref{E_1specseqbreen20}) and (\ref{E_2spectralseqbreen21}) is just the Yoneda map (\ref{Yonedamapdef}).
\end{remark}

We shall apply the Breen spectral sequences with $G = \mathbf{G}_a$ and $H = \mathbf{G}_m$ in order to study the second map in (\ref{mapsH^2ExtBr}).

\begin{proposition}
\label{Ext^2=Brprimbasescheme}
Let $S$ be a reduced scheme such that pullback induces isomorphisms $\Pic(S) \rightarrow \Pic(\A^n_S)$ for $n = 2, 3$. (This holds, for instance, when $S$ is seminormal.) Then the functorial $\Gamma(S, \calO_S)$-module homomorphism $\Ext^2_S(\Ga, \Gm) \rightarrow {\rm{H}}^2(\mathbf{G}_{a,\,S}, \Gm)_{\prim}$ in $(\ref{mapsH^2ExtBr})$ is an isomorphism. (For the $\Gamma(S, \calO_S)$-structure on ${\rm{H}}^2(\mathbf{G}_{a,\,S}, \Gm)_{\prim}$, see Remark \ref{linearactiononBrauerprim}.)
\end{proposition}

\begin{proof}
We will use the Breen spectral sequences discussed above. We first use the spectral sequence (\ref{E_2spectralseqbreen21}) with $G = \Ga$ and $H = \Gm$: 
\[
{}^{\prime}E_2^{i,j} = \Ext^i_S({\rm{H}}_j(A(\Ga)), \Gm) \Longrightarrow \Ext^{i+j}_S(A(\Ga), \Gm).
\]
This yields a map $\Ext^2_S(\Ga, \Gm) = {}^{\prime}E_2^{2,0} \rightarrow \Ext^2_S(A(\Ga), \Gm)$, and we will now show that this map is an isomorphism. In order to do this, it suffices to show that the groups ${}^{\prime}E_2^{i,1}$ and ${}^{\prime}E_2^{0,2}$ vanish. The groups ${}^{\prime}E_2^{i,1} = \Ext_S^i({\rm{H}}_1(A(\Ga)), \Gm)$ vanish because the sheaf ${\rm{H}}_1(A(\Ga))$ does by (\ref{H_1(A)=0}), and the group ${}^{\prime}E_2^{0,2} = \Hom({\rm{H}}_2(A(\Ga)), \Gm)$ vanishes because ${\rm{H}}_2(A(\Ga)) = \Ga/2\Ga$, again by (\ref{H_1(A)=0}), and there are no nontrivial homomorphisms from $\Ga$ to $\Gm$ over the reduced scheme $S$, by Lemma \ref{nonconstantunitsred}.

Now we use the spectral sequence (\ref{E_1specseqbreen20}), again with $G = \Ga$ and $H = \Gm$:  
\[
E_1^{i,j} = \widetilde{{{\rm{H}}}}^j(X_i, \Gm) \Longrightarrow \Ext^{i+j}(A(\Ga), \Gm).
\]
This yields a map $\Ext^2(A(\Ga), \Gm) \rightarrow E_2^{0,2} = {\rm{H}}^2(\mathbf{G}_{a,\,S}, \Gm)_{\prim}$, this last equality by (\ref{E_2^0j=primcohom}), and we want to show that this map is an isomorphism, since the composition $\Ext^2_S(\Ga, \Gm) = {}^{\prime}E_2^{2,0} \rightarrow \Ext^2_S(A(\Ga), \Gm) \rightarrow E_2^{0,2} = {\rm{H}}^2(\mathbf{G}_{a,\,S}, \Gm)_{\prim}$ is the Yoneda map. In order to do this, it suffices to show that the groups $E_1^{i,0}$ all vanish, as do the groups $E_1^{i,1}$ for $i = 1, 2$.

The groups $E_1^{i,0} = \widetilde{{\rm{H}}}^0(X_i, \Gm)$ vanish because the only global units on $\mathbf{G}_{a,S}^n$ are the elements of $\Gamma(S, \calO_S)^{\times}$ by Lemma \ref{nonconstantunitsred} (since $S$ is reduced). Next, $E_1^{i,1} = \widetilde{{\rm{H}}}^1(X_i, \Gm)$. But $X_1 = \mathbf{G}_{a,\,S}^2$ and $X_2 = \mathbf{G}_{a,\,S}^2 \coprod \mathbf{G}_{a,\,S}^3$ by (\ref{X_idescription}), so the groups $\widetilde{{\rm{H}}}^1(X_i, \Gm)$ vanish for $i = 1, 2$ by assumption.
\end{proof}

Proposition \ref{Ext^2=Brprimbasescheme} shows that the map $\Ext^2_X(\Ga, \Gm) \rightarrow {\rm{H}}^2(\mathbf{G}_{a,\, X}, \Gm)_{\prim}$ is an isomorphism if $X$ is seminormal, while Proposition \ref{H^2(X,G_a^)-->Ext^2(Ga,Gm)injection} shows that the map ${\rm{H}}^2(X, \widehat{\Ga}) \rightarrow \Ext^2_X(\Ga, \Gm)$ is always injective. We will show later (Proposition \ref{H^2=Ext^2=Brdvr}) that this map is an isomorphism in various desirable circumstances (for example, if $X$ is regular).

It will be important later for us to know that the maps (\ref{mapsH^2ExtBr}) are compatible with evaluation at a point $x \in (\mathbf{G}_a)(X)$.
More precisely, given such an $x$, we get a morphism of fppf sheaves $\widehat{\mathbf{G}_a} \rightarrow \mathbf{G}_m$ on $X$ given by evaluation at $x$. This induces a map ev$_x:  {{\rm{H}}}^2(X, \widehat{\mathbf{G}_a}) \rightarrow {{\rm{H}}}^2(X, \mathbf{G}_m)$. On the other hand, via evaluation (i.e., restriction) at $x$,
we have a map ${\rm{H}}^2(\mathbf{G}_{a,\,X}, \Gm) \rightarrow {\rm{H}}^2(X, \Gm)$ that we will also denote by ev$_x$. Then we have the following lemma.

\begin{lemma} 
\label{evaluationcompatibility}
The following diagram commutes: 
\[
\begin{tikzcd}
{{\rm{H}}}^2(X, \widehat{\mathbf{G}_a}) \arrow{r} \arrow{d}{{\ev}_x} & {\rm{H}}^2(\mathbf{G}_{a,\,X}, \Gm)_{\prim} \arrow{d}{{\ev}_x} \\
{\rm{H}}^2(X, \Gm) \arrow[r, equals] & {\rm{H}}^2(X, \Gm)
\end{tikzcd}
\]
where the horizontal map ${{\rm{H}}}^2(X, \widehat{\mathbf{G}_a}) \xrightarrow{\sim} {\rm{H}}^2(\mathbf{G}_a, \Gm)_{\prim}$ is the composition in $(\ref{mapsH^2ExtBr})$.
\end{lemma}

\begin{proof}
We will define a map ev$_x:  \Ext^2_X(\mathbf{G}_a, \mathbf{G}_m) \rightarrow {{\rm{H}}}^2(X, \mathbf{G}_m)$ and show that the following diagram commutes: 
\begin{equation}
\label{evaluationdiagram}
\begin{tikzcd}
{{\rm{H}}}^2(X, \widehat{\mathbf{G}_a}) \arrow{r} \arrow{d}{{\ev}_x} & \Ext^2_X(\mathbf{G}_a, \mathbf{G}_m) \arrow{r} \arrow{d}{{\ev}_x} & {{\rm{H}}}^2(\mathbf{G}_{a,\,X}, \mathbf{G}_m) \arrow{d}{{\ev}_x} \\
{{\rm{H}}}^2(X, \mathbf{G}_m) \arrow[r, equals] & {{\rm{H}}}^2(X, \mathbf{G}_m) \arrow[r, equals] & {{\rm{H}}}^2(X, \mathbf{G}_m)
\end{tikzcd}
\end{equation}
where all maps are those in (\ref{mapsH^2ExtBr}). The map ${\ev}_x:  \Ext^2_X(\mathbf{G}_a, \mathbf{G}_m) \rightarrow {{\rm{H}}}^2(X, \mathbf{G}_m)$ is defined as follows. We have a natural transformation of functors $\Hom_X(\mathbf{G}_a, \cdot) \rightarrow \Gamma(X, \cdot)$ from the category of abelian fppf sheaves on $X$ to the category of abelian groups, defined by evaluation at $x \in \mathbf{G}_a(X)$. That is, given an element of $\Hom_X(\mathbf{G}_a, \mathscr{F})$, we get a map $\mathbf{G}_a(X) \rightarrow \mathscr{F}(X)$ and we take the image of $x$ under this map. This yields a corresponding map of derived functors ${\ev}_x:  \Ext^{\bullet}_X(\mathbf{G}_a, \mathscr{F}) \rightarrow {{\rm{H}}}^{\bullet}(X, \mathscr{F})$, and we specialize this to the case $\mathscr{F} = \mathbf{G}_m$ in degree 2.

Now let us check the commutativity of the first square in (\ref{evaluationdiagram}). Recall that the map ${{\rm{H}}}^2(X, \widehat{\mathbf{G}_a}) \rightarrow \Ext^2_X(\mathbf{G}_a, \mathbf{G}_m)$ was defined as the edge map $D_2^{2,0} \rightarrow D_2$ in the Leray spectral sequence
\[
D_2^{i,j} = {{\rm{H}}}^i(X, \calExt^j_X(\mathbf{G}_a, \mathscr{F})) \Longrightarrow \Ext^{i+j}_X(\mathbf{G}_a, \mathscr{F}),
\]
where $\mathscr{F}$ is an fppf abelian sheaf on $X$. (We apply this with $\mathscr{F} = \mathbf{G}_m$.) But the composite functor spectral sequence is natural in the associated functors. That is, given two pairs of functors $F_1, G_1$ and $F_2, G_2$, together with natural transformations $F_1 \rightarrow F_2$, $G_1 \rightarrow G_2$, the induced maps of derived functors yield a natural transformation from the spectral sequence $R^iF_1(R^jG_1) \Longrightarrow R^{i+j}(F_1 \circ G_1)$ to the spectral sequence $R^iF_2(R^jG_2) \Longrightarrow R^{i+j}(F_2 \circ G_2)$. 

We now apply this with $F_1 = F_2 = \Gamma(X, \cdot)$ with the identity transformation $F_1 \rightarrow F_2$, and $G_1 = \calHom_X(\mathbf{G}_a, \cdot)$, $G_2 = $ Identity, with the transformation $G_1 \rightarrow G_2$ being evaluation at $x$. (That is, given an fppf sheaf $\mathscr{F}$ on $X$, for each $X$-scheme $U$ we have a map $\calHom(\mathbf{G}_a, \mathscr{F})(U) = \Hom((\mathbf{G}_a)_U, \mathscr{F}|_U) \rightarrow \mathscr{F}(U)$ given by evaluating at the pullback of $x \in \mathbf{G}_a(X)$ to a section $x_U \in \mathbf{G}_a(U)$.) This induces on the derived functor level the first two vertical maps ${\ev}_x$ appearing in diagram (\ref{evaluationdiagram}). The maps ${{\rm{H}}}^i(X, \mathscr{F}) \rightarrow {{\rm{H}}}^i(X, \mathscr{F})$ associated to the spectral sequence for $F_2, G_2$ are the identity map for all $i$; specializing the functoriality of the spectral sequence to this situation, for $\mathscr{F} = \mathbf{G}_m$ and $i=2$ we obtain the commutativity of the first square in (\ref{evaluationdiagram}).

It remains to prove the commutativity of the second square. Recall that the map $\Ext^2_X(\mathbf{G}_a, \mathbf{G}_m) \rightarrow {{\rm{H}}}^2(\mathbf{G}_{a,\,X}, \mathbf{G}_m)$ was defined as follows. We have a natural transformation of functors $\Hom_X(\mathbf{G}_a, \cdot) \xrightarrow{\rm{Yon}} \Gamma(\mathbf{G}_{a,\,X}, \cdot)$ defined on an fppf 
abelian sheaf $\mathscr{F}$ on $X$ by using Yoneda's Lemma to assign
to any $\Ga \rightarrow \mathscr{F}$ the corresponding element of $\Gamma(\mathbf{G}_a, \mathscr{F})$. Then we obtain an induced map on derived functors $\Ext^{\bullet}_X(\mathbf{G}_a, \mathscr{F}) \rightarrow {{\rm{H}}}^{\bullet}(\mathbf{G}_{a,\,X}, \mathscr{F})$, and the map in (\ref{evaluationdiagram}) is simply this map specialized to the case $\mathscr{F} = \mathbf{G}_m$
and degree 2. So to check that the second square in (\ref{evaluationdiagram}) commutes, we merely need to check that the associated square for the $0$th derived functors commutes for any abelian fppf sheaf $\mathscr{F}$. That is, we need commutativity of
\[
\begin{tikzcd}
\Hom_X(\mathbf{G}_a, \mathscr{F}) \arrow{r}{\rm{Yon}} \arrow{d}{{\ev}_x} & \Gamma(\mathbf{G}_{a,\,X}, \mathscr{F}) \arrow{d}{{\ev}_x} \\
\Gamma(X, \mathscr{F}) \arrow[r, equals] & \Gamma(X, \mathscr{F})
\end{tikzcd}
\]
and this is clear. 
\end{proof}

We will use the isomorphism ${{\rm{H}}}^2(k, \widehat{\mathbf{G}_a}) \simeq {\rm{H}}^2(\mathbf{G}_a, \Gm)_{\prim}$ to study ${{\rm{H}}}^2(k, \widehat{\mathbf{G}_a})$ if ${\rm{char}}(k)=p>0$. (There is nothing to do when ${\rm{char}}(k)=0$,
since in that case these isomorphic groups vanish by Proposition \ref{H^2=Ext^2=Br}.) To this end, in \S \ref{sectionbrauergpsdiff} we turn to a method for studying $p$-torsion Brauer classes by relating them to differential forms.

\section{Brauer groups and differential forms}
\label{sectionbrauergpsdiff}

Throughout this section, $k$ denotes a field of characteristic $p > 0$. Due to Proposition \ref{H^2=Ext^2=Br}, in order to understand ${{\rm{H}}}^2(k, \widehat{\mathbf{G}_a})$ we need to understand ${\rm{H}}^2(\mathbf{G}_{a,\,k}, \Gm)_{\prim}$. The first observation is to recall that this lies inside ${\rm{H}}^2(\mathbf{G}_{a,\,k}, \Gm)[p]$. Indeed, this follows from the existence of a $k$-linear action on ${\rm{H}}^2(\mathbf{G}_{a,\,k}, \Gm)_{\prim}$; see Remark \ref{linearactiononBrauerprim}. In order to do this, we recall an observation of Kato \cite{kato} that over fields $p$-torsion Brauer elements can be related to differential forms by utilizing the (inverse) Cartier operator.

Let $X$ be an $\F_p$-scheme, and let $\Omega^1_X = \Omega^1_{X/\mathbf{F}_p}$ be the sheaf of K{\"a}hler differential forms on $X$. Let $B^1_X \subset \Omega^1_X$ denote the subsheaf of coboundaries; that is, $B^1_X : = \mbox{Im}(d:  \calO_X \rightarrow \Omega^1_X)$. There is a morphism $C^{-1}:  \Omega^1_X \rightarrow \Omega^1_X/B^1_X$ defined by
\[
C^{-1}(fdg) = f^pg^{p-1}dg
\]
(The reason for the ``inverse'' notation is that the Cartier operator is usually defined in a relative setting, and is essentially built as the inverse of the above operator. The ``inverse'' notation should not be taken to mean that the operator defined above is the inverse of some operator $C$.) The above map is well-defined. The only nontrivial point is to check that $C^{-1}(d(f + g)) = C^{-1}(df) + C^{-1}(dg)$. This is a consequence of the identity
\[
(f+g)^{p-1}d(f + g) - f^{p-1}df - g^{p-1}dg = d\left(\frac{(f+g)^p - f^p - g^p}{p}\right),
\]
where the expression in parentheses is defined to be $Q(f, g)$, where $Q \in \Z[X, Y]$ is defined by the formula $Q(X, Y) : = ((X + Y)^p - X^p - Y^p)/p$. 

\begin{lemma}
\label{pthpoweriffdx=0}
Let $X$ be a normal $\F_p$-scheme. Then the complex
\[
0 \longrightarrow \calO_X \xlongrightarrow{F} \calO_X \xlongrightarrow{d} \Omega^1_X
\]
is exact, where $F$ is the Frobenius map $s \mapsto s^p$.
\end{lemma}

\begin{proof}
Exactness on the left follows from the reducedness of $X$. To see that the complex is exact at the second $\calO_X$, first note that by a standard limit argument, we may assume that $X$ is locally Noetherian. Now suppose that we have an element $s \in \Gamma(X, \calO_X)$ such that $ds = 0$. Then the same holds for the differential of $s$ in the residue field of each generic point of $X$. Since the lemma holds for spectra of fields by \cite[Thm.\,26.5]{crt}, it follows that $s = f^p$ for some rational function $f$ on $X$. Since $s$ has nonnegative order at each point of codimension one on the normal scheme $X$, the same holds for $f$, hence -- since $X$ is normal and locally Noetherian -- $f$ extends uniquely to a global section $f' \in \Gamma(X, \calO_X)$, and this section satisfies $f'^p = s$.
\end{proof}

We will require the following lemma.

\begin{lemma}
\label{logdiffformsvalringsdegimp1}
Let $A$ be a valuation ring of characteristic $p$ such that the fraction field $K$ of $A$ has degree of imperfection $\leq 1$. Then, for an element $a \in K^{\times}$, the differential form $da/a \in \Omega^1_K$ arises from a differential form in $\Omega^1_A$ if and only if $a = x^pu$ for some $x \in K^{\times}$ and some $u \in A^{\times}$ -- in other words, if and only if the valuation of $a$ is divisible by $p$.
\end{lemma}

\begin{proof}
If $a = x^pu$, then $da/a = du/u \in \Omega^1_A$. Conversely, let $a \in K^{\times}$, and assume that the valuation $v(a)$ is not $p$-divisible in the value group $\Gamma$ of $A$: $v(a) \notin p\Gamma$. We must show that $da/a \notin \Omega^1_A \subset \Omega^1_K$. (The map $\Omega^1_A \rightarrow \Omega^1_K$ is an inclusion because $K$ is a localization of $A$.) We claim that, for $\alpha \in K$, if $\alpha da \in \Omega^1_A$, then $a\alpha \in \mathfrak{m}$, the maximal ideal of $A$. This will in particular imply that $da/a \notin \Omega^1_A$ (since $da \neq 0$). Since the condition on $\alpha$ that $a\alpha \in \mathfrak{m}$ is preserved under addition (i.e., if true for $\alpha$ and $\beta$, then it is also true for $\alpha + \beta$) and scalar multiplication by $A$ (if true for $\alpha$, and if $z \in A$, then it is true for $z\alpha$), it suffices in order to prove the claim to show that, for every $y \in A$, one has $dy \in \mathfrak{m}da/a$.

We first note that, since $a$ is not a $p$th power in $K$ and $K$ has degree of imperfection $\leq 1$, $a$ forms a $p$-basis for $K$. We may therefore write
\begin{equation}
\label{logdiffformsvalringsdegimp1pfeqn1}
y = \sum_{i = 0}^{p-1} e_i^pa^i
\end{equation}
for some $e_i \in K$. Then
\begin{equation}
\label{logdiffformsvalringsdegimp1pfeqn2}
dy = \left(\sum_{i=1}^{p-1} ie_i^pa^i\right)\frac{da}{a}.
\end{equation}
We need to show that the expression in the parentheses lies in $\mathfrak{m}$. In fact, the terms in the sum in (\ref{logdiffformsvalringsdegimp1pfeqn1}) have distinct valuations, since their valuations are even distinct in $\Gamma/p\Gamma$, hence the valuation of the sum is the minimum of the valuations of the individual terms. Since the sum lies in $A$, it follows that each term has valuation $\geq 0$. But each term with $i \neq 0$ does not lie in $p\Gamma$, hence must have valuation $> 0$. That is, each such term lies in $\mathfrak{m}$, hence the sum in (\ref{logdiffformsvalringsdegimp1pfeqn2}) lies in $\mathfrak{m}$, as desired.
\end{proof}

Define the map ${\rm{dlog}}:  \mathbf{G}_m/(\mathbf{G}_m)^p \rightarrow \Omega^1_X$ by $f \mapsto df/f$, and let $i:  \Omega^1_X \rightarrow \Omega^1_X/B^1_X$ denote the projection. Then we have the following lemma, which is the key to relating Brauer elements to differential forms.

\begin{lemma}
\label{cartiersequence}
Let $X$ be an $\F_p$-scheme. Assume either that $X$ is regular, or that $X$ is smooth over a scheme $S$ such that every local ring of $S$ is a valuation ring whose fraction field has degree of imperfection $\leq 1$. Then the following sequence of {\'e}tale sheaves on $X$ is exact: 
\[
0 \longrightarrow \mathbf{G}_m/(\mathbf{G}_m)^p \xlongrightarrow{{\rm{dlog}}} \Omega^1_X\xlongrightarrow{C^{-1} - i} \Omega^1_X/B^1_X \longrightarrow 0.
\]
\end{lemma}

\begin{proof}
Note that in both cases $X$ is normal. Exactness on the left therefore follows from Lemma \ref{pthpoweriffdx=0}. For exactness on the right, in order to hit a class in $\Omega^1_X/B^1_X$ (over an \'etale $X$-scheme $U$) represented by $fdg$, over  an {\'e}tale cover of $U$ we can find an $H$ such that ${{\rm{H}}}^pg^{p-1} - H = f$. Then $(C^{-1} - i)(Hdg) = fdg$. It is also easy to see that the sequence is a complex:  we have $(C^{-1} - i)({\rm{dlog}}(f)) = (C^{-1} - i)(df/f) = (1/f^p)f^{p-1}df - df/f = 0$.

It is harder to show that the sequence is exact at $\Omega^1_X$. For this, we break the proof up into two cases, based upon whether we assume $X$ to be regular or smooth over a base each local ring of which is a valuation ring whose fraction field has degree of imperfection $\leq 1$.

First we treat the regular case. In the case that $X = \Spec(K)$ for a field $K$, the desired exactness is \cite[Thm.\,9.2.2]{gille}. For the general case, the claimed exactness is a local assertion, so we may assume that $X = {\rm{Spec}}(R)$ for $R$ a regular local ring. Given $\omega \in \Omega^1_R$ such that $C^{-1}(\omega) = i(\omega)$, we need to check that $\omega = du/u$ for some $u \in R^{\times}$. By the already-known case in which $R$ is a field, we know that $\omega = df/f$ for some $f \in K : = \mbox{Frac}(R)$. Since $R$ is regular local, it is a UFD, so we may write $f = u\prod \pi_i^{e_i}$ for some 
pairwise non-associate prime elements $\pi_i \in R$, some $e_i \in \Z$, and some $u \in R^{\times}$. We then have $df/f = du/u + \sum e_id\pi_i/\pi_i$. We may assume that $p\nmid e_i$ for each $i$, since the terms with $p\mid e_i$ disappear. We need to show that if $df/f$ extends to a differential form in $\Omega^1_R$, then the sum is empty, hence $\omega = du/u$. We may localize at one of the primes $(\pi_i)$ to reduce ourselves to the following assertion:  if $R$ is an equicharacteristic discrete valuation ring with uniformizer $t$ and fraction
field $K$ then the element $dt/t \in K \otimes_R \Omega^1_R = \Omega^1_K$ does not arise from an element of $\Omega^1_R$.

To prove this claim, we may replace $R$ by its completion and thereby assume that $R \simeq \kappa \llbracket t\rrbracket$ for some field $\kappa$. Extending scalars, we may also assume that $\kappa$ is algebraically closed. Let $\kappa\{t\}$ denote the (strict) Henselization of $\kappa[t]$ at the origin. Then we claim that the map ${\rm{Spec}}(\kappa \llbracket t\rrbracket) \rightarrow {\rm{Spec}}(\kappa \{t\})$ is regular. First, we note that the map is flat, because $\kappa\{t\}$ is a filtered direct limit of Dedekind schemes, so flatness is equivalent to torsion-freeness. Next, we note that ${\rm{Spec}}(\kappa\llbracket t\rrbracket)$ has only two points, the origin and the generic point. Regularity at the origin is trivial (because the fiber over the origin of $\kappa\{t\}$ is just the identity). To see regularity over the generic point, we note that regularity of a scheme is preserved by \'etale extensions, and that it therefore suffices to check that $\kappa\llbracket t \rrbracket \otimes_{\kappa[t]} \kappa(t^{1/p^n}) = \kappa (\!( t^{1/p^n} )\!)$ is regular for every integer $n > 0$, and this is clear. The map is therefore regular, as claimed.

It now follows from Popescu's theorem \cite[Tag 07GC]{stacks} that $\kappa\llbracket t \rrbracket$ is a filtered colimit $\varinjlim A_i$ of smooth $\kappa \{t\}$-algebras. Therefore, if $dt/t$ is a regular differential form over $\kappa \llbracket t\rrbracket$, then it descends to a regular differential form over $A_i$ for some $i$. Since $\kappa \{t\}$ is strictly Henselian, the smooth algebra $A_i$ admits a section over it. Therefore, we find that $dt/t \in \Omega^1_{\kappa\{t\}}$. Since $\kappa\{t\}$ is a filtered direct limit of \'etale $\kappa[t]$-algebras, we have $\Omega^1_{\kappa\{t\}} = \Omega^1_{\kappa[t]}$, hence $dt/t \in \Omega^1_{\kappa[t]}$, which is false. This contradiction completes the proof in the case when $X$ is regular.

Now we treat the case in which $X$ is smooth over a scheme $S$ each local ring of which is a valuation ring whose fraction field has degree of imperfection $\leq 1$. The assertion is local on $X$, so we may assume that $X$ is finitely presented over $S$. Then the map $g: X \rightarrow S$ is open, so replacing $S$ by $g(X)$, we may assume that $g$ is faithfully flat, hence $S$ is an $\F_p$-scheme (since $X$ is). Since the assertion is local, a standard limit argument reduces us to the case in which $S = {\rm{Spec}}(A)$ for a characteristic $p$ valuation ring $A$ whose fraction field has degree of imperfection $\leq 1$ and $X = {\rm{Spec}}(B)$, where $A \rightarrow B$ is a local homomorphism such that $B$ is (as an $A$-algebra) a local ring of a smooth $A$-scheme. 

Let $K := {\rm{Frac}}(A)$, and let $\omega \in \Omega^1_B$ be such that $(C^{-1} - i)(\omega) = 0$. We need to check that $\omega = dv/v$ for some $v \in B^{\times}$. Since $\Omega^1_B \hookrightarrow \Omega^1_{K \otimes_A B}$ is an inclusion, it suffices to check such an equality over the latter module of differentials. The generic fiber $K \otimes_A B$ of $B$ is a localization of a smooth scheme over the field $K$. Hence, by the already-treated regular case, we know that $\omega = d\beta/\beta \in \Omega^1_{K \otimes_A B}$ for some $\beta \in (K \otimes_A B)^{\times}$. By Proposition \ref{genericunitalmost=unit}, $\beta = b\alpha$ for some $b \in B^{\times}$ and some $\alpha \in K^{\times}$. Modifying $\beta$ by $b$ and $\omega$ by $db/b$, we may therefore assume that $\beta = \alpha \in K^{\times}$.

Since $B$ is a local ring of a smooth $A$-scheme, there is a local homomorphism $h: A \rightarrow A'$, where $A'$ is a local ring of an \'etale $A$-algebra, such that the morphism ${\rm{Spec}}(B) \rightarrow {\rm{Spec}}(A)$ admits a section $s$ when pulled back to $A'$: ${\rm{Spec}}(A') \xrightarrow{s} {\rm{Spec}}(B) \rightarrow {\rm{Spec}}(A)$, where the composition is ${\rm{Spec}}(h)$. Pulling back along $A'$, we then see that $d\alpha/\alpha \in \Omega^1_{A'}$. By \cite[Tag 0ASJ]{stacks}, $A'$ is a valuation ring such that the map $A \rightarrow A'$ induces an isomorphism of value groups. Further, ${\rm{Frac}}(A')$ is a finite separable extension of $K$, hence also has degree of imperfection $\leq 1$. Lemma \ref{logdiffformsvalringsdegimp1} therefore implies that the valuation of $\alpha$ is a multiple of $p$ in the value group of $A'$. It follows that the same holds in the value group of $A$. That is, $\alpha = x^pu$ for some $x \in K^{\times}$ and some $u \in A^{\times}$. Then $d\alpha/\alpha = du/u$, and the proof is complete.
\end{proof}

\begin{proposition}
\label{brauerdifferentialforms}
Let $A$ be a ring of characteristic $p > 0$. Assume either that $A$ is regular, or that $A$ is smooth over a scheme $S$ such that each local ring of $S$ is a valuation ring whose fraction field has degree of imperfection $\leq 1$. Then we have a functorial short exact sequence
\[
0 \longrightarrow \frac{\Pic(A)}{p \cdot \Pic(A)} \longrightarrow \frac{\Omega^1_A}{dA + (C^{-1}-i)\Omega^1_A} \xlongrightarrow{\psi} {\rm{H}}^2(A, \Gm)[p] \longrightarrow 0,
\]
where $d:  A \rightarrow \Omega^1_A$ is the canonical derivation, $C^{-1}$ is the inverse Cartier map, and $\psi$ is the map induced by the composition of the connecting map ${\rm{H}}^0(\Omega^1_A/B^1_A) \rightarrow {\rm{H}}^1(A, \Gm/(\Gm)^p)$ associated to the sequence of Lemma $\ref{cartiersequence}$, and the connecting map ${\rm{H}}^1(A, \Gm/(\Gm)^p) \rightarrow {\rm{H}}^2(A, \Gm)[p]$ associated to the exact sequence
\[
0 \longrightarrow \mathbf{G}_m \xlongrightarrow{p} \mathbf{G}_m \longrightarrow \mathbf{G}_m/(\mathbf{G}_m)^p \longrightarrow 0.
\]
\end{proposition}

\begin{proof}
All of the cohomology in this proof is {\'e}tale. Let $X := \Spec(A)$. 
Lemma \ref{cartiersequence} yields an exact sequence
\begin{equation}
\label{brauerdifferentialformsproofeqn2}
0 \longrightarrow \frac{{{\rm{H}}}^0(X, \Omega^1_X/B^1_X)}{(C^{-1} - i)({{\rm{H}}}^0(X, \Omega^1_X))} \longrightarrow {{\rm{H}}}^1(X, \mathbf{G}_m/(\mathbf{G}_m)^p) \longrightarrow {{\rm{H}}}^1(X, \Omega^1_X).
\end{equation}
Since $\Omega^1_X$ is a quasi-coherent sheaf, its {\'e}tale and Zariski cohomology agree. Therefore, since $X$ is affine, ${{\rm{H}}}^1(X, \Omega^1_X) = 0$. We claim that the natural maps ${{\rm{H}}}^0(X, \Omega^1_X) \rightarrow {{\rm{H}}}^0(X, \Omega^1_X/B^1_X)$ and $d:  A = {\rm{H}}^0(X, \calO_X) \rightarrow {\rm{H}}^0(X, B^1_X)$ are surjective. For the surjectivity of the first map, it is  enough show that ${{\rm{H}}}^1(X, B^1_X) = 0$. By Lemma \ref{pthpoweriffdx=0}, we have an exact sequence
\begin{equation}
\label{diffeqn1}
0 \longrightarrow \calO_X \xlongrightarrow{F} \calO_X \xlongrightarrow{d} B^1_X \longrightarrow 0. 
\end{equation}
That ${{\rm{H}}}^1(X, B^1_X) = 0$ therefore follows from the fact that ${{\rm{H}}}^i(X, \calO_X) = 0$ for all $i>0$, since $X$ is affine. The surjectivity of the map $d:  {\rm{H}}^0(X, \calO_X) \rightarrow {\rm{H}}^0(X, B^1_X)$ follows from the exact sequence (\ref{diffeqn1}) and the fact that ${\rm{H}}^1(X, \calO_X) = 0$. We therefore obtain from (\ref{brauerdifferentialformsproofeqn2}) an isomorphism
\begin{equation}
\label{brauerdifferentialformsproofeqn3}
\frac{\Omega^1_A}{dA + (C^{-1}-i)(\Omega^1_A)} \xrightarrow{\sim} {\rm{H}}^1(X, \Gm/(\Gm)^p).
\end{equation}

The exact sequence of {\'e}tale sheaves on $X$
\[
0 \longrightarrow \mathbf{G}_m \xlongrightarrow{p} \mathbf{G}_m \longrightarrow \mathbf{G}_m/(\mathbf{G}_m)^p \longrightarrow 0
\]
(exactness on the left because $X$ is reduced, hence so is every \'etale $X$-scheme) yields an exact sequence
\[
0 \longrightarrow \frac{\Pic(X)}{p\cdot \Pic(X)} \longrightarrow {{\rm{H}}}^1(X, \mathbf{G}_m/(\mathbf{G}_m)^p) \longrightarrow {\rm{H}}^2(X, \Gm)[p] \longrightarrow 0.
\]
Combining this with (\ref{brauerdifferentialformsproofeqn3}) yields the proposition.
\end{proof}

\begin{lemma}
\label{valringlimoffppfalg}
Let $V$ be a valuation ring, and let $C$ be a (faithfully) flat $V$-algebra. Then $C$ is a filtered direct limit of (faithfully) flat $V$-algebras of finite presentation.
\end{lemma}

\begin{proof}
It suffices to prove the result without the faithful adjective, since the surjectivity of the map into ${\rm{Spec}}(V)$ is inherited by each $V$-algebra in a filtered system mapping to $C$. Trivially, $C$ is a filtered direct limit $\varinjlim_i B_i$ of finite type $V$-algebras $B_i$. Since $C$ is $V$-torsion free, the map $B_i \rightarrow C$ factors through the $V$-torsion free quotient $\overline{B}_i$ of $B_i$ (obtained by quotienting out by the ideal of $V$-torsion elements of $B_i$). Thus, we may write $C = \varinjlim_i \overline{B}_i$ as a filtered direct limit of $V$-torsion free $V$-algebras of finite type. Since a $V$-module is flat if and only if it is $V$-torsion free \cite[Tag 0539]{stacks}, it follows that each $\overline{B}_i$ is a flat $V$-algebra of finite type. By \cite[Tag 053E]{stacks}(i), it follows that $\overline{B}_i$ is finitely presented over $V$.
\end{proof}

Next we will use Proposition \ref{brauerdifferentialforms} show that the maps of (\ref{mapsH^2ExtBr}) are isomorphisms under suitable hypotheses on $X$.

\begin{proposition}
\label{H^2=Ext^2=Brdvr}
Let $X$ be an $\F_p$-scheme, and assume either that $X$ is regular, or that it is smooth over a scheme $S$, each local ring of which is a valuation ring whose fraction field has degree of imperfection $\leq 1$. Then the maps ${\rm{H}}^2(X, \widehat{\Ga}) \xrightarrow{f_X} \Ext^2_X(\Ga, \Gm) \xrightarrow{g_X} {\rm{H}}^2(\mathbf{G}_{a,\,X}, \Gm)_{\prim}$ in $(\ref{mapsH^2ExtBr})$ are $\Gamma(X, \calO_X)$-module isomorphisms.
\end{proposition}

\begin{proof}
The assumptions of the proposition imply that $X$ is seminormal, so the second map is an isomorphism by Proposition \ref{Ext^2=Brprimbasescheme}. The first map is an inclusion by Proposition \ref{H^2(X,G_a^)-->Ext^2(Ga,Gm)injection}. It remains to show that it is surjective. We have the Leray spectral sequence
\[
E_2^{i, j} = {\rm{H}}^i(X, \calExt^j(\Ga, \Gm)) \Longrightarrow \Ext^{i+j}_X(\Ga, \Gm).
\]
The terms $E_2^{i, 1}$ vanish by Proposition \ref{ext^1(Ga,Gm)=0}. We therefore have an exact sequence
\[
{\rm{H}}^2(X, \widehat{\Ga}) \xlongrightarrow{f_X} \Ext^2_X(\Ga, \Gm) \xlongrightarrow{h} {\rm{H}}^0(X, \calExt^2(\Ga, \Gm)),
\]
where we recall that the first map is $f_X$ by definition. Let $\alpha \in \Ext^2_X(\Ga, \Gm)$. We need to show that $h(\alpha) = 0$. 

This vanishing assertion is fppf local on $X$, so we may assume that $X = {\rm{Spec}}(A)$ is affine, and since the relevant cohomology group ${\rm{H}}^0(X, \calExt^2(\Ga, \Gm))$ behaves well with respect to filtered direct limits of rings, in the case in which $X$ is smooth over a base whose local rings are valuation rings, we may further assume that $X$ is smooth over a valuation ring. By Proposition \ref{brauerdifferentialforms}, we have $g_X(\alpha) = \psi(\overline{\omega})$ for some $\omega \in \Omega^1_A$, where $\overline{\omega}$ denotes the image of $\omega$ in $\Omega^1_{A[X]}/(d(A[X]) + (C^{-1}-i)\Omega^1_{A[X]})$. We claim that $(F_A^n)^*(\overline{\omega}) = 0$ for some $n > 0$, where $F_A$ is the absolute Frobenius morphism on $A$.

We may write $\omega = G(X)dX + \sum_{i=1}^nH_i(X)da_i$ for some $G, H_i \in A[X]$ and $a_i \in A$. Since $d(a_i^p) = 0$ for each $i$, it then follows that $F_A^*\omega = R(X)dX$ for some $R \in A[X]$. Working term by term, we may assume that $R(X) = bX^m$ for some $b \in A$ and some $m \geq 0$. We then prove the claim by induction on $m$. If $m \not \equiv -1 \pmod{p}$, then
\[
F_A^*(bX^mdX) = b^pX^mdX = d\left(b^p\frac{X^{m+1}}{m+1}\right),
\]
so $\overline{\omega} = 0$. If, on the other hand, $m \equiv -1 \pmod{p}$, then
\[
F_A^*(bX^mdX) = b^pX^mdX = bX^{\frac{m+1-p}{p}}dX + (C^{-1}-i)(bX^{\frac{m+1-p}{p}}dX),
\]
so
\[
F_A^*(\overline{\omega}) = \overline{bX^{\frac{m+1-p}{p}}dX}
\]
in this case. By induction, the element on the right is killed upon pullback by a suitable power of $F_A$. This completes the proof of the claim. 

Since the cohomology in question behaves well with respect to filtered direct limits, it therefore suffices to show that the map $F_A$ is a filtered direct limit of fppf morphisms. That is, if $A'$ denotes the ring $A$ with the $A$-algebra structure given by $F_A$, then we want to show that $A'$ is a filtered direct limit of fppf $A$-algebras. It suffices to write $A'$ as a filtered direct limit of finitely presented flat $A$-algebras, since the surjectivity of the underlying map of sets onto ${\rm{Spec}}(A)$ is then inherited by each ring in the filtered system.

We first check that $F_A$ is a filtered direct limit of finitely presented flat morphisms when $A$ is a regular ring (thanks to Ofer Gabber for this argument). By Popescu's Theorem \cite[Tag 07GC]{stacks}, $A$ is a filtered direct limit of smooth $\F_p$-algebras $B_i$. The map $F_{B_i}$ is finite flat for each $B_i$. Since $F_A$ is the filtered direct limit of the maps $A \rightarrow A \otimes_{B_i, {\rm{Frob}}} B_i$ obtained via base change from Frobenius on $B_i$, it follows that $F_A$ is a filtered direct limit of finite flat maps.

Next we consider the case in which $X = {\rm{Spec}}(A)$ is smooth over a valuation ring $V$. The map $F_{X/V}\colon X \rightarrow X^{(p)}$ is finitely presented (because $X$ is of finite type over $V$), and we claim that it is flat. Indeed, for each $v \in V$, the relative Frobenius map $X_v \rightarrow X_v^{(p)}$ is flat (since $X_v$ is smooth over a field), so the flatness of $F_{X/V}$ follows from the fibral flatness criterion \cite[${\rm{IV_3}}$, Cor.\,11.3.11]{ega}. We therefore see that $F_{X/V}$ is fppf. We claim that the base change map $g: X^{(p)} \rightarrow X$ is a filtered direct limit of finitely presented flat maps. It suffices to show the same for $F_V\colon V \rightarrow V$, and, by Lemma \ref{valringlimoffppfalg}, it suffices in turn to show that $F_V$ is flat. This follows from the fact that a $V$-module is flat if and only if it is $V$-torsion free \cite[Tag 0539]{stacks}. It now follows that the map $F_X = F_{X/V} \circ g$ is a filtered direct limit (on the level of rings rather than schemes) of finitely presented flat maps. This completes the proof of the proposition.
\end{proof}

\begin{proposition}
\label{H^2=Ext^2=Br}
Let $k$ be a field. The maps $(\ref{mapsH^2ExtBr})$ are functorial isomorphisms of $k$-vector spaces ${{\rm{H}}}^2(k, \widehat{\mathbf{G}_a}) \simeq \Ext^2_k(\mathbf{G}_a, \mathbf{G}_m) \simeq {\rm{H}}^2(\mathbf{G}_{a,k}, \Gm)_{\prim}$.
If $k$ is perfect, then all three groups vanish. 
\end{proposition}

For imperfect fields $k$ we will see later that these three common groups can be rather interesting (and are always nontrivial; see Corollary \ref{omegaH^2injective}).

\begin{proof}  
The isomorphism assertion when ${\rm{char}}(k) > 0$ follows from Proposition \ref{H^2=Ext^2=Brdvr}, so it only remains to show the vanishing assertion when $k$ is perfect, which we now assume.

The group ${{\rm{H}}}^2(k, \widehat{\mathbf{G}_a})$ vanishes by Proposition \ref{cohomologyofG_adualwhenkisperfect}. To show that ${\rm{H}}^2(\mathbf{G}_{a,k}, \Gm)_{\prim} = 0$, we note that the map ${\rm{H}}^2(k, \Gm) \rightarrow {\rm{H}}^2(\mathbf{G}_{a,k}, \Gm)$ is an isomorphism by Remark \ref{nolinearaction}. Restricting the equality $m^*\alpha = p_1^*\alpha + p_2^*\alpha$ to the point $(0, 0)$, we find that ${\rm{H}}^2(\mathbf{G}_{a,k}, \Gm)_{\prim}$ contains no nontrivial constant Brauer classes, hence ${\rm{H}}^2(\mathbf{G}_{a,k}, \Gm)_{\prim} = 0$. The vanishing of $\Ext^2_k(\Ga, \Gm)$ now follows from Proposition \ref{Ext^2=Brprimbasescheme}.
\end{proof}

\section{Computation of ${\rm{H}}^2(R, \widehat{\Ga})$ and ${\rm{H}}^2(\mathbf{G}_{a,\,R}, \Gm)_{\rm{prim}}$}
\label{sectiongeneralH^2(Ga^)}

In this section we will compute ${\rm{H}}^2(R, \widehat{\Ga}) \simeq {\rm{H}}^2(\mathbf{G}_{a,\,R}, \Gm)_{\prim}$ for a ring $R$ of characteristic $p > 0$ that is either a field or a DVR, and such that the fraction field $k$ of $R$ satisfies $[k:  k^p] = p$. By Proposition \ref{H^2=Ext^2=Br} these are isomorphic as $R$-modules, and we will show that they are free of rank $1$ for such rings (Proposition \ref{omega=H^2(Ga^)}). We will accomplish this by exploiting the relationship between $p$-torsion Brauer elements (among which are the primitive Brauer elements on $\Ga$) and differential forms obtained in \S \ref{sectionbrauergpsdiff}, particularly Proposition \ref{brauerdifferentialforms}.

\begin{lemma}
\label{primetop}
Let $R$ be an $\F_p$-algebra, and suppose that there exists a $t \in R$ such that $R = \sum_{i=0}^{p-1}t^iR^p$. In the quotient
\begin{equation}
\label{Omega^1quotient}
\frac{\Omega^1_{R[X_1, \dots, X_n]}}{d(R[X_1, \dots, X_n]) + (C^{-1}-i)(R[X_1, \dots, X_n])},
\end{equation}
every element represented by a differential form $Fdt$ with $F \in R[X_1, \dots, X_n]$ is represented by a differential form $Gdt$, where $G \in R[X_1, \dots, X_n]$ contains no monomials of the form $X^{pI}$ with $I \in \N^n - \{0\}$ and such that $G$ has the following property:  if $I \in \N^n - (p\N)^n$ is such that the monomial $X^I : = \prod_{i=1}^n X_i^{I_i}$ appears in $F$ with nonzero coefficient, and for every positive integer $r$ the monomial $X^{p^rI}$ has coefficient $0$ in $F$, then the monomial $X^I$ appears in $G$ with nonzero coefficient.
\end{lemma}

\begin{proof}
It suffices to show that every differential form $\lambda X^{pI}dt$
with $\lambda \in R$ and $I \neq 0$ is equivalent in (\ref{Omega^1quotient}) to a 1-form $\lambda' X^Idt$ with $\lambda' \in R$. By assumption, we may write $\lambda = \sum_{i = 0}^{p-1} a_i^pt^i$ for some $a_i \in R$, so we may assume that the 
representative 1-form is $a^pt^iX^{pI}dt$
for some $a \in R$ and $0 \leq i < p$. If $i \neq p-1$ then $a^pt^iX^{pI}dt = d(a^pX^{pI}\frac{t^{i+1}}{i+1})$, 
so this element is trivial in (\ref{Omega^1quotient}). If $i = p-1$, then $a^pt^{p-1}X^{pI}dt = aX^Idt + (C^{-1} - i)(aX^Idt)$, 
so as elements of (\ref{Omega^1quotient}), $a^pt^{p-1}X^{pI}dt = aX^Idt$. 
\end{proof}

\begin{lemma}
\label{brauerclassdeterminespolynomial}
Let $R$ be a regular $\F_p$-algebra, and let $F \in R[X_1, \dots, X_n]$. Suppose that there exists $I \in \N^n - (p\N)^n$ such that $X^I : = \prod_{i=1}^n X_i^{I_i}$ appears in $F$ with nonzero coefficient, and such that for every positive integer $r$, the coefficient of $X^{p^rI}$ in $F$ vanishes. Let $0 \neq \omega \in \Omega^1_R$. Then for the map $\psi$ in Proposition $\ref{brauerdifferentialforms}$, the element $\psi(F\omega) \in {\rm{H}}^2(\A^n_R, \Gm)[p]$ is nonzero.
\end{lemma}

\begin{proof}
We may break $Y : = {\rm{Spec}}(R)$ up into its connected components. The assumption on $X^I$ then still holds on at least one of the components, so we may restrict to such a component and assume that $Y$ is connected, i.e., since $R$ is regular, that $R$ is an integral domain. Let $k : = {\rm{Frac}}(R)$. Since $\A^n_R$ is regular, the map ${\rm{H}}^2(\A^n_R, \Gm) \rightarrow {\rm{H}}^2(\A^n_k, \Gm)$ is an inclusion. Further, since $\Spec(k)$ is an inverse limit of \'etale $R$-schemes, $\omega|\Spec(k) \in \Omega^1_k$ is nonzero. Therefore, in order to prove the lemma, we may replace $R$ by $k$ and thereby assume that $R = k$ is a field. We have $\omega = \sum_{i=1}^m \lambda_idt_i$ for some $\lambda_i \in k^{\times}$ and some $t_i \in k$ such that the elements $\{t_1, \dots, t_m\}$ form part of a $p$-basis for $k$, due to \cite[Thm.\,26.5]{crt}. In order to prove the desired nonvanishing we may extend scalars to $k(t_2^{1/p}, t_3^{1/p}, \dots, t_m^{1/p})$ and thereby assume that $\omega = \lambda dt$ for some $\lambda \in k^{\times}$ and some $t \in k - k^p$. Replacing $F$ by $\lambda F$, we may assume that $\omega = dt$ for some $t \in k- k^p$. We will next reduce to the case in which $k$ is a local function field.

In order to do this, we may apply a standard spreading out and specializing argument, as follows. If $Fdt$ represents the trivial Brauer class over $\A^n_k$, then Proposition \ref{brauerdifferentialforms} implies that there exists a differential form $\omega' = \sum_{i=1}^r g_idh_i$ such that $Fdt = (C^{-1}-i)(\omega')$ as elements of $\Omega^1_{\A^n_k}/B^1_{\A^n_k}$. Since $(C^{-1}-i)(\omega') = \sum_{i=1}^r (g_i^ph_i^{p-1}dh_i - g_idh_i)$, we obtain
\begin{equation}
\label{Fdt=d+C}
Fdt = df + \sum_{i=1}^r (g_i^ph_i^{p-1}dh_i - g_idh_i). 
\end{equation}
for some $f \in \A^n_k$. Since $t \in k$ is not a $p$th power, the extension $k/\F_p(t)$ is separable, so we may write $k$ as the filtered direct limit $k = \varinjlim_i R_i$ of smooth $\F_p(t)$-algebras $R_i$. The equation (\ref{Fdt=d+C}) then descends to some such smooth $\F_p(t)$-algebra $R_i$, where we descend $F$ to a polynomial in which the same monomials appear with nonzero coefficients. Replacing $R_i$ with some $R_j$ if necessary, we may arrange that the nonzero coefficients $r_I$ of $F$ are generically nonzero in $R_i$. Since $R_i$ is smooth and the $r_I \in R$ are generically nonzero, we may specialize (\ref{Fdt=d+C}) to an $L$-valued point $x:  \Spec(L) \rightarrow R$ for some finite separable extension $L/\F_p(t)$ such that $0 \neq r_I(x) \in L$ for all $I$. Then the class $F_x(X_1, \dots, X_n)dt \in {\rm{H}}^2(\A^n_L, \Gm)$ vanishes and, because $L/\F_p(t)$ is separable, we still have $t \notin L^p$, so all of the hypotheses of the lemma are satisfied for $F_xdt$ and $R = L$. Replacing $k$ with $L$, therefore, we have reduced to the case in which $k$ is a global function field. Replacing $k$ with the separable extension $k_v/k$ for some place $v$ of $k$, we may in fact assume that $k$ is a local function field.

So we now assume that $R = k$ is a local function field, and we prove the assertion of the lemma. We first treat the case $n = 1$. Then $F \in k[X]$ is nonconstant, and by Lemma \ref{primetop} (whose hypotheses are satisfied for the pair $k$, $t \in k$ because $[k:  k^p] = p$), we may assume that $F$ contains no monomials of the form $X^{pr}$ with $r > 0$. In particular, the degree of $F$ is prime to $p$. We need to show that the element of ${\rm{H}}^2(\A^1_k, \Gm)[p]$ represented via Proposition \ref{brauerdifferentialforms} by the differential form $Fdt$ (with $t \in k - k^p$) is nonzero. Since $\Omega^1_k = kdt$ (as follows from \cite[Thm.\,26.5]{crt} and the fact that $[k:  k^p] = p$), using the isomorphism ${\rm{H}}^2(k, \Gm) \simeq \frac{1}{p}\Z/\Z$ provided by taking invariants, choose an element $\lambda \in k$ such that the Brauer class represented by $\lambda dt$ has invariant $1/p$. Then the polynomial $F(X) - \lambda \in k[X]$ has degree prime to $p$, and therefore has an irreducible factor of some degree, say $m$, prime to $p$. Let $\beta$ be a root of this irreducible factor in some finite extension of $k$. It suffices to show that setting $X = \beta$ yields a nonzero element of $\Br(k(\beta))$. In fact, setting $X = \beta$, we get that $F(\beta)dt = \lambda dt$ as Brauer classes in ${\rm{H}}^2(k(\beta), \Gm)$. The invariant of an element of ${\rm{H}}^2(k, \Gm)$, after extending scalars to some finite extension of $k$, gets multiplied by the degree of the extension. Since $p \nmid [k(\beta):  k] = m$, therefore, we deduce that ${\rm{inv}}_{k(\beta)}(\lambda dt) = m*{{\rm{inv}}}_k(\lambda dt) = m/p \in \Q/\Z$ is nonzero.

Now we treat the general case. Let $F \in k[X_1, \dots, X_n]$ be a nonconstant polynomial as in the lemma. By Lemma \ref{primetop}, we may assume that $F$ contains no monomials of the form $X^{pI}$ with $I \in \N^n - \{0\}$. Now consider the lexicographic ordering on $\N^n$, which is a total ordering defined by declaring that, for $I, J \in \N^n$, $I$ is greater than $J$ precisely when there exists an $r \in \{1, \dots, n\}$ such that $I_s = J_s$ for $s < r$ and such that $I_r > J_r$.

Write
\[
F = \sum_{I \in \N^n} C_IX^I
\]
with $C_I \in k$, and let 
\begin{equation}
\label{M=maxindex}
M : = \max \{I \in \N^n \mid C_I \neq 0\}.
\end{equation}
Since $F$ contains no monomials of the form $X^{pI}$ with $I \in \N^n - \{0\}$, we have 
\begin{equation}
\label{pnmidMj0}
p \nmid M_{i_0}
\end{equation}
for some $i_0 \in \{1, \dots, n\}$.

Now we inductively choose a sequence $r_n, r_{n-1}, \dots, r_1$ of positive integers as follows. Having chosen $r_j$ for all $j > i$, we choose $r_i$ (including at the beginning of the process, with $i = n$) to be a positive integer satisfying the conditions
\begin{equation}
\label{r_jlarge}
r_i > \max_{C_I \neq 0} \left\{\sum_{j = i+1}^n r_jI_j \right\},
\end{equation}
\begin{equation}
\label{r_jmodp}
r_i
\begin{cases}
\equiv 0 \pmod{p}, & i \neq i_0, \\
\not \equiv 0 \pmod{p}, & i = i_0.
\end{cases}
\end{equation}
Then we pull back the element of ${\rm{H}}^2(\A^n_k, \Gm)$ represented by $Fdt$ to an element of the group ${\rm{H}}^2(\A^1_k, \Gm)$ via the map $\A^1_k \rightarrow \A^n_k$ given on coordinate rings by the $k$-algebra map $k[X_1, \dots, X_n] \rightarrow k[Y]$, $X_i \mapsto Y^{r_i}$. This amounts to replacing $F(X_1, \dots, X_n)dt$ by $G(Y)dt$, where $G(Y) : = F(Y^{r_1}, \dots, Y^{r_n})$. We claim that $G$ has degree prime to $p$. Assuming this, it will follow from the already-treated $n = 1$ case of the lemma that $G(Y)dt$ represents a nontrivial Brauer class, hence the same must hold for $Fdt$, which will prove the lemma.

In order to show that $G$ has degree prime to $p$, we claim that the degree of $G$ is $\sum_{j = 1}^n r_jM_j$. That is, the degree is determined by making the replacements $X_i \mapsto Y^{r_i}$ in the monomial $X^M$. Assuming this, it then follows from (\ref{pnmidMj0}) and (\ref{r_jmodp}) that the degree of $G$ is prime to $p$. In order to show that $G$ does in fact have degree $\sum_{i = 1}^n r_jM_j$, it suffices to show that for any $I \neq M \in \N^n$ such that $C_I \neq 0$, we have
\begin{equation}
\label{sumrj>sumrj}
\sum_{j = 1}^n r_jM_j > \sum_{j=1}^n r_jI_j.
\end{equation}
Since $M > I$ for any such $I$ by (\ref{M=maxindex}), there exists $i \in \{1, \dots, n\}$ such that $M_j = I_j$ for all $j < i$ and $M_i > I_i$. Then, using (\ref{r_jlarge}), we have
\[
\sum_{j = 1}^n r_jM_j = \sum_{j < i} r_jI_j + r_iM_i + \sum_{j > i} r_jM_j \geq \sum_{j < i} r_jI_j + r_iM_i 
\]
\[
\geq \sum_{j < i} r_jI_j + r_i(I_i + 1) > \sum_{j \leq i} r_jI_j + \sum_{j > i} r_jI_j = \sum_{j=1}^n r_jI_j,
\]
which proves (\ref{sumrj>sumrj}).
\end{proof}

Now let $R$ be a regular $\F_p$-algebra. The map $\Omega^1_R \rightarrow \Omega^1_{\mathbf{G}_{a,\,R}}$ defined by $\omega \mapsto X\omega$ (where $X$ is the variable of $\Ga$) is $R$-linear for the $R$-action induced by the $R$-action on $\Ga$, since for $\lambda \in R$, the action of $\lambda$ on $\Ga$ is given by $X \mapsto \lambda X$, and $X(\lambda \omega) = (\lambda X)\omega$. Via the map
\begin{equation}
\label{psionGa}
\psi:  \frac{\Omega^1_{\mathbf{G}_{a,\,R}}}{dR + (C^{-1}-i)(\Omega^1_R)} \rightarrow {\rm{H}}^2(\mathbf{G}_{a,\,R}, \Gm)[p]
\end{equation}
in Proposition \ref{brauerdifferentialforms}, we thereby obtain a map $\Omega^1_R \rightarrow {\rm{H}}^2(\mathbf{G}_{a,\,R}, \Gm)[p]$ with image contained in the primitive Brauer group, since $m^*(X\omega) = (X_1 + X_2)\omega = X_1\omega + X_2\omega = \pi_1^*(X\omega) + \pi_2^*(X\omega)$. We thereby obtain for any regular $\F_p$-algebra $R$ a functorial $R$-linear map
\begin{equation}
\label{phidefbrauer}
\phi_R:  \Omega^1_R \rightarrow {\rm{H}}^2(\mathbf{G}_{a,\,R}, \Gm)_{\prim}
\end{equation}
defined by $\phi(\omega) := \psi(X\omega)$, where $X$ is the variable on $\mathbf{G}_{a,\, R}$ and $\psi$ is the map in (\ref{psionGa}). We will first show that this map $\phi_R$ is injective.

\begin{corollary}
\label{omegaH^2injective}
Let $R$ be a regular $\F_p$-algebra. The $R$-linear 
map $$\phi_R:  \Omega^1_R \rightarrow {\rm{H}}^2(\mathbf{G}_{a,\,R}, \Gm)_{\prim}$$ defined by $\omega \mapsto \psi(X\omega) \in {\rm{H}}^2(\mathbf{G}_{a,\,R}, \Gm)_{\prim}$ is injective, where $\psi$ is the map in Proposition $\ref{brauerdifferentialforms}$.
\end{corollary}

\begin{proof}
This is an immediate consequence of Lemma \ref{brauerclassdeterminespolynomial}.
\end{proof}

\begin{lemma}
\label{frobsurjectiveH^2(Ga^)}
Let $R$ be an $\F_p$-algebra that is either regular, or smooth over a base each local ring of which is a valuation ring whose fraction field has degree of imperfection $\leq 1$, and let $F:  \Ga \rightarrow \Ga$ be the relative Frobenius $R$-isogeny, given on coordinate rings by $X \mapsto X^p$. Then the pullback map $F^*:  {\rm{H}}^2(\mathbf{G}_{a,\,R}, \Gm)_{\prim} \rightarrow {\rm{H}}^2(\mathbf{G}_{a,\,R}, \Gm)_{\prim}$ is surjective.
\end{lemma}
 
\begin{proof}
By Proposition \ref{H^2=Ext^2=Brdvr}, it suffices to prove this with ${\rm{H}}^2(\mathbf{G}_{a,\,R}, \Gm)_{\prim}$ replaced by ${\rm{H}}^2(R, \widehat{\Ga})$. We have an exact sequence of $R$-group schemes
\begin{equation}
\label{frobsurjproofsequence}
1 \longrightarrow \alpha_p \longrightarrow \Ga \xlongrightarrow{F} \Ga \longrightarrow 1
\end{equation}
hence by Proposition \ref{ext^1(Ga,Gm)=0} an exact sequence
\[
1 \longrightarrow \widehat{\Ga} \xlongrightarrow{\widehat{F}} \widehat{\Ga} \longrightarrow \widehat{\alpha_p} \longrightarrow 1, 
\]
so it suffices to show that ${\rm{H}}^2(R, \widehat{\alpha_p}) = 0$. This follows from the sequence (\ref{frobsurjproofsequence}), since $\widehat{\alpha_p} \simeq \alpha_p$ and the groups ${\rm{H}}^i(X, \Ga)$ agree with the Zariski cohomology groups for any scheme $X$, hence vanish in positive degree when $X$ is affine.
\end{proof}

\begin{lemma}
\label{Omega^1=Rdpi}
Let $R$ be a DVR in which $p = 0$ and such that the fraction field $k$ of $R$ satisfies $[k:  k^p] = p$. Then for any uniformizer $\pi$ of $R$, we have $R = \oplus_{i=0}^{p-1} \pi^iR^p$, and $\Omega^1_R$ is a free $R$-module of rank one with generator $d\pi$.
\end{lemma}

\begin{proof}
Since $\pi \in k - k^p$ and $[k:  k^p] = p$, we have $k = \oplus_{i=0}^{p-1} \pi^ik^p$. It follows that for any $r \in R$, we may write
\[
r = \sum_{i=0}^{p-1} \pi^i\lambda_i^p
\]
for some unique $\lambda_i \in k$. Since the valuation of $\pi^i\lambda_i^p$ is congruent to $i$ modulo $p$, the summands all have distinct valuations. The sum, $r$, therefore must have valuation equal to the minimum of the valuations of the summands. In particular, each summand has nonnegative valuation, i.e., each $\lambda_i$ has nonnegative valuation. That is, each $\lambda_i \in R$. This proves the first assertion. This first assertion also implies that $\Omega^1_R$ is generated by $d\pi$. To see that it is freely generated by $d\pi$, it suffices to show that as elements of the $k$-vector space $\Omega^1_k$, we have $rd\pi \neq 0$ for any $0 \neq r \in R$. For this, in turn, it suffices to note that $0 \neq d\pi \in \Omega^1_k$ because $\pi \notin k^p$ \cite[Thm.\,26.5]{crt}.
\end{proof}

\begin{proposition}
\label{omega=H^2(Ga^)}
Let $R$ be an $\F_p$-algebra that is either a DVR or a field, and suppose that the fraction field $k$ of $R$ satisfies $[k:  k^p] = p$. Then the $R$-linear map
$\phi_R$ in $(\ref{phidefbrauer})$ is an $R$-module isomorphism: 
\[
\Omega^1_R \underset{\sim}{\xrightarrow{\phi}} {\rm{H}}^2(\mathbf{G}_{a,\,R}, \Gm)_{\prim} \simeq {\rm{H}}^2(R, \widehat{\Ga}) \simeq \Ext^2_R(\Ga, \Gm).
\] 
In particular, ${\rm{H}}^2(R, \widehat{\Ga})$ is a free $R$-module of rank $1$.
\end{proposition}

The isomorphisms between the last three groups are given by Proposition \ref{H^2=Ext^2=Brdvr}. What is new is that $\phi_R$ is an isomorphism when $[k:  k^p] = p$.

\begin{proof}
The last claim follows from Lemma \ref{Omega^1=Rdpi} in the DVR case and by \cite[Thm.\,26.5]{crt} in the field case, so we just concentrate on the first assertion. The map $\phi_R$ is injective by Corollary \ref{omegaH^2injective}, so it only remains to check surjectivity. Let $\alpha \in {\rm{H}}^2(\mathbf{G}_{a,\,R}, \Gm)_{\prim}$. By Lemma \ref{frobsurjectiveH^2(Ga^)}, $\alpha = F^*(\beta)$ for some $\beta \in {\rm{H}}^2(\mathbf{G}_{a,\,R}, \Gm)_{\prim}$, where $F:  \Ga \rightarrow \Ga$ is the Frobenius $R$-isogeny. Since $[k:  k^p] = p$, Lemma \ref{Omega^1=Rdpi} and \cite[Thm.\,26.5]{crt} imply that there exists $t \in R$ such that $R = \sum_{i=0}^{p-1} t^iR^p$ and $\Omega^1_R = Rdt$. By Proposition \ref{brauerdifferentialforms}, therefore, $\beta$ is represented by a differential form $H_1(X)dX + H_2(X)dt$ for some $H_i \in R[X]$. Then $\alpha = F^*(\beta)$ is represented by $H_1(X^p)d(X^p) + H_2(X^p)dt = H_2(X^p)dt$. That is, $\alpha$ is represented by $Gdt$ for some $G \in R[X]$. By Lemma \ref{primetop}, we may assume that $G$ contains no monomials of the form $X^{pr}$ with $r > 0$. The primitivity of $\alpha$ implies that $0 = m^*\alpha - \pi_1^*\alpha - \pi_2^*\alpha$, and this Brauer class on $\A^2_k = \Spec(k[X_1, X_2])$ is represented by the form $(G(X_1 + X_2) - G(X_1) - G(X_2))dt$. Lemma \ref{brauerclassdeterminespolynomial} therefore implies that $G(X) = \lambda + H(X)$ for some $\lambda \in R$ and $H(X) \in R[X]$ having no monomials of degree divisible by $p$ and such that $H(X_1 + X_2) = H(X_1) + H(X_2)$. This implies that $H(X) = \gamma X$ for some $\gamma \in R$. Thus, subtracting the primitive Brauer class represented by $\gamma Xdt = \phi_R(\gamma dt)$ from $\alpha$, we may assume that $\alpha \in {\rm{H}}^2(R, \Gm)$. But ${\rm{H}}^2(\mathbf{G}_{a,\,R}, \Gm)_{\prim}$ contains no nontrivial elements of ${\rm{H}}^2(R, \Gm)$, so this implies that $\alpha = 0$.
\end{proof}

\begin{remark}
\label{Br(G_a)=/=phi(omega)}
Let $k$ be a field of characteristic $p > 0$. The map $\phi : = \phi_k$ of Corollary \ref{omegaH^2injective}, though always injective, is an isomorphism if and only if $[k:  k^p] \leq p$. We will never use this, so the reader may skip this remark. If $k$ is perfect, then ${\rm{H}}^2(\Ga, \Gm)_{\prim} = 0$ by Proposition \ref{H^2=Ext^2=Br}, so $\phi$ is clearly surjective in this case. It is also surjective if $[k:  k^p] = p$, by Proposition \ref{omega=H^2(Ga^)}.

Now suppose that $[k:  k^p] > p$; we will show that $\phi$ is not surjective. In fact, the cokernel of 
the $k$-linear $\phi$ is infinite-dimensional over $k$ whenever $[k:  k^p] > p$. 
To see this, it suffices to show that if $t, w \in k$ are $p$-independent (that is, $\{t^iw^j\}_{0 \leq i, j < p}$ are linearly independent over $k^p$), then no nontrivial $k$-linear combination of the primitive Brauer elements represented (via Proposition \ref{brauerdifferentialforms}) by $tw^{p^n}X^{p^n}dw/w$ ($n \ge 1$) is of the form $X\omega$ for some $\omega \in \Omega^1_k$. Suppose that some
$k$-linear combination of these differential forms, say $\sum_{n=1}^N \lambda_n^{p^n}tw^{p^n}X^{p^n}dw/w$, does lie in the image of $\phi$. (Note for the $k$-action on $\Ga$, an element $\lambda \in k$ acts on $X^{p^n}$ via $X^{p^n} \mapsto (\lambda X)^{p^n} = \lambda^{p^n}X^{p^n}$.) We want to check that each $\lambda_n$ vanishes. It suffices to check such
vanishing in $k(t^{1/p^N})$. 

For any extension field $F/k$, $\mu \in F$, and $n \ge 1$, we have
$$(C^{-1} - i)(\mu w^{p^{n-1}}X^{p^{n-1}}dw/w) = \mu^p w^{p^n}X^{p^n}dw/w - \mu w^{p^{n-1}}X^{p^{n-1}}dw/w$$
in $\Omega^1_{F[X]}/d(F[X])$.  
Hence, as Brauer classes for $\mathbf{G}_{a,F}$  we have 
\[
\mu^p w^{p^n}X^{p^n}dw/w = \mu w^{p^{n-1}}X^{p^{n-1}}dw/w
\]
for $n \ge 1$. 
Applying this repeatedly with $F = k(t^{1/p^N})$, we see for $1 \le n \le N$
that $\lambda_n^{p^n}tw^{p^n}X^{p^n}dw/w = \lambda_nt^{1/p^n}Xdw/w$ as Brauer classes
over $k(t^{1/p^N})$. By Corollary \ref{omegaH^2injective}, therefore, we have
\[
\left(\sum_{n=1}^N \lambda_nt^{1/p^n}\right)\frac{dw}{w} = f^*\omega
\]
in $\Omega^1_{k(t^{1/p^N})}$ 
for some $\omega \in \Omega^1_k$, where $f:  \Spec(k(t^{1/p^N})) \rightarrow \Spec(k)$ is the indicated map. Since $\{t, w\}$ are $p$-independent in $k$, so part of a $p$-basis, the relationship between $p$-bases and differential bases \cite[Thm.\,26.5]{crt} gives
that $dw \neq 0$ in $\Omega^1_{k(t^{1/p^N})}$ and $\omega = adw + b dt +  \sum_{i =1}^M \alpha_i dz_i$
for some $a,  b, \alpha_i \in k$ and $\{w, t, z_1, \dots, z_M\}$ 
part of a $p$-basis for $k$. Since $\{w, t^{1/p^N}, z_1, \dots, z_M\}$ is part of a $p$-basis for $k(t^{1/p^N})$, the $\alpha_i$'s vanish, and 
\[
\sum_{n=1}^N \lambda_n t^{1/p^n} = aw \in k.
\]
Raising both sides to the $p^N$th power then shows that all of the $\lambda_n$ must vanish, since $\{1, t, t^p, \dots, t^{p^{N-1}}\}$ are linearly independent over $k^{p^N}$. 
\end{remark}

\section{Vanishing of ${\rm{H}}^3(R, \widehat{\Ga})$}
\label{sectionH^3(Ga^)=0}

The goal of this section is to show that the group ${\rm{H}}^3(R, \widehat{\Ga})$ vanishes when $R$ is either a field or a Henselian DVR of characteristic $p$ whose fraction field $k : = {\rm{Frac}}(R)$ satisfies $[k:  k^p] = p$ (Proposition \ref{H^3(G_a^)=0}), or, more generally, a product of such rings (Proposition \ref{H^3(Ga^)=0prodofdvrs}).

\begin{lemma}
\label{H^3(R,G_a^)injective}
Let $R$ be an $\F_p$-algebra that is either a field or a DVR. Suppose that the fraction field $k$ of $R$ satisfies $[k:  k^p] = p$. Then the pullback map ${\rm{H}}^3(R, \widehat{\Ga}) \rightarrow {\rm{H}}^3(R^{{\rm{sh}}}, \widehat{\Ga})$ is injective, where $R \rightarrow R^{{\rm{sh}}}$ is a strict Henselization of $R$.
\end{lemma}

\begin{proof}
We have the Hochschild-Serre spectral sequence
\[
E_2^{i, j} = {\rm{H}}^i(\pi, {\rm{H}}^j(R^{\rm{sh}}, \widehat{\Ga})) \Longrightarrow {\rm{H}}^{i+j}(R, \widehat{\Ga}),
\]
where $\overline{\kappa}$ is the residue field of $R^{\rm{sh}}$ and $\pi = \pi_1(R, \overline{\kappa})$, the \'etale fundamental group of $\Spec(R)$. The map ${\rm{H}}^3(R, \widehat{\Ga}) = E^3 \rightarrow E_2^{0,3} = {\rm{H}}^3(R^{\rm{sh}}, \widehat{\Ga})^{\pi}$ is the pullback map that we want to show is injective. In order to do this it suffices to show that the groups $E_2^{3, 0}$, $E_2^{2,1}$, and $E_2^{1,2}$ all vanish. The groups $E_2^{3,0}$ and $E_2^{2,1}$ vanish by Proposition \ref{cohomologyofG_adualgeneralk}.

It remains to show that $E_2^{1,2} = {\rm{H}}^1(\pi, {\rm{H}}^2(R^{\rm{sh}}, \widehat{\Ga}))$ vanishes. Note that the hypotheses on $R$ are preserved upon replacing $R$ with $R^{\rm{sh}}$. Proposition \ref{omega=H^2(Ga^)} and Lemma \ref{Omega^1=Rdpi} therefore yield an isomorphism of $\pi$-modules ${\rm{H}}^2(R^{\rm{sh}}, \widehat{\Ga}) \simeq R^{\rm{sh}}$. We therefore have an isomorphism $E_2^{1,2} \simeq {\rm{H}}^1(\pi, R^{\rm{sh}})$. This latter group is isomorphic to ${\rm{H}}^1_{\et}(R, \Ga)$, which vanishes.
\end{proof}

\begin{lemma}
\label{identities}
Let $R$ be an $\F_p$-algebra, and assume that $F \in R[X, Y]$ satisfies the identities: 
\[
F(X,Y) = F(Y, X)
\]
\begin{equation}
\label{identitieseqn1}
F(X, Y+Z) + F(Y, Z) - F(X+Y, Z) - F(X, Y) = 0.
\end{equation}
Then there exist $G \in R[X]$, an integer $N \geq 0$, and $r_0, \dots, r_N \in R$ such that
\[
F(X, Y) = G(X+Y) - G(X) - G(Y) + \sum_{i=0}^N r_iS(X^{p^i}, Y^{p^i}),
\]
where $S(X, Y)$ is the mod $p$ reduction of the Witt polynomial $((X+Y)^p - X^p - Y^p)/p \in \Z[X,Y]$. Further, any polynomial of the above form satisfies the identities $(\ref{identitieseqn1})$.
\end{lemma}

\begin{proof}
It is straightforward to check that any polynomial of the form $G(X+Y) - G(X) - G(Y)$ (over any ring) satisfies (\ref{identitieseqn1}). It follows that the Witt polynomial, which is of this form with $G = X^p/p$, satisfies these identities over $\Z$, hence so does its mod $p$ reduction. Therefore, over any $\F_p$-algebra, the composition of the Witt polynomial with the additive map $X \mapsto X^{p^i}$ also satisfies (\ref{identitieseqn1}). It remains to show that any polynomial satisfying (\ref{identitieseqn1}) is of the form given in the lemma.

The identities (\ref{identitieseqn1}) break up degree by degree, so we may assume that $F$ is homogeneous of degree $n\ge 0$. So let $F(X, Y) = \sum_{i=0}^n \beta_{i,\,n-i}X^iY^{n-i} \in R[X, Y]$ be a homogenous polynomial of degree $n$ satisfying the identities (\ref{identitieseqn1}). If $n = 0$, then $F(X) = r$ for some $r \in R$ and we may take $G(X) = -r$, so assume that $n > 0$. The first identity in (\ref{identitieseqn1}) says that $\beta_{i,\,j} = \beta_{j,\,i}$ for any $i, j$.

Comparing coefficients of $X^n$ on both sides of the identity 
\begin{equation}
\label{identitieseqn2}
F(X, Y+Z) + F(Y, Z) - F(X+Y, Z) - F(X, Y) = 0
\end{equation}
gives $\beta_{n,\,0} = 0$, so $\beta_{0,\,n} = 0$ by symmetry. We may therefore assume that $n > 1$. We will show that $F$ is of the form given in the lemma by showing that for all $0 < i, j < n$, either $\beta_{i,\,n-i} = 0$ for all polynomials $F$, over all $\F_p$-algebras $R$, satisfying (\ref{identitieseqn1}), or else $\beta_{j,\,n-j} = \lambda \beta_{i,\,n-i}$ for some {\em universal} $\lambda \in \F_p$; that is, $\lambda \in \F_p$ depends only on $n$ and $i, j$, and is independent of both $F$ and $R$. Once we show this, it follows that if $0 \neq H(X, Y) \in \F_p[X, Y]$ homogeneous of degree $n$ satisfies (\ref{identitieseqn1}), then every polynomial $F \in R[X, Y]$ satisfying (\ref{identitieseqn1}) must be some $R$-multiple of $H$. If $n$ is not a power of $p$, then we may take $H(X, Y) = G(X + Y) - G(X) - G(Y)$, where $G(X) : = X^n$. If, on the other hand, $n = p^{j+1}$ for some $j \geq 0$, then we may take $H(X, Y) : = S(X^{p^j}, Y^{p^j})$. This claim about the $\beta_{i,\,n-i}$ all being universal multiples of one another will therefore prove the lemma. So we concentrate on proving this claim.

Comparing coefficients of $X^aY^bZ^c$ on both sides of (\ref{identitieseqn2}) gives
\begin{equation}
\label{identitiesproofeqn1}
\binom{b+c}{b} \beta_{a,\,b+c} = \binom{a+b}{a} \beta_{a+b,\,c} \mbox{\hspace{.2 in} if } a, c > 0.
\end{equation}
Let $0 < i < n$, and let let $j$ be such that the coefficient of $p^j$ in the base $p$ expansion of $i$ is nonzero. We may apply (\ref{identitiesproofeqn1}) with $a = i - p^j, b = p^j, c = n - i$ provided that $i \neq p^j$. Hence, when $i \ne p^j$ we have 
\[
\binom{n-i+p^j}{p^j} \beta_{i-p^j,\,n-i+p^j} = \binom{i}{p^j} \beta_{i,\,n-i}.
\]
By Lemma \ref{binomialcoeffcongruence}, $\binom{i}{p^j} \neq 0$ in $\F_p$ because the $j$th digit in the base-$p$ expansion of $i$ is nonzero. We therefore get that $\beta_{i,\,n-i}$ is some universal $\F_p$-multiple of $\beta_{i-p^j,\,n-i + p^j}$. That is, the $X^iY^{n-i}$-coefficient of $F$ is some universal $\F_p$-multiple of the $X^{i-p^j}Y^{n-i+p^j}$-coefficient. Applying this inductively, we see that the $X^iY^{n-i}$-coefficient is some universal $\F_p$-multiple of $\beta_{p^l,\,n-p^l}$, where $l$ is any index such that the coefficient of $p^l$ in the base $p$ expansion of $i$ is nonzero. It therefore only remains to show that for any powers $p^j, p^{j'} < n$ of $p$, $\beta_{p^j,\,n-p^j}$ is a universal $\F_p$-multiple of $\beta_{p^{j'},\,n-p^{j'}}$ or vice versa (if $\beta_{p^{j'},\,n-p^{j'}} = 0$).

We first show that if $p^j, p^{j'} < n$ both have nonzero coefficients in the base $p$ expansion of $n$, then $\beta_{p^j,\,n-p^j}$ is some universal $\F_p$-multiple of $\beta_{p^{j'},\,n-p^{j'}}$. We may of course suppose that $j \neq j'$. We have $\beta_{p^j,\,n-p^j} = \beta_{n-p^j,\,p^j}$, and the coefficient of $p^{j'}$ in the base-$p$ expansion of $n-p^j$ is nonzero. By what we have just shown, therefore, $\beta_{n-p^j,\,p^j}$ is some universal $\F_p$-multiple of $\beta_{p^{j'},\,n-p^{j'}}$. 

Next suppose that the coefficient of $p^j$ in the base $p$ expansion of $n$ is $0$, so in particular 
\[
\label{ngeqp^j+1}
n \geq p^{j+1}.
\]
Then we will show that $\beta_{p^j,\,n-p^j} = 0$ unless $n = p^{j+1}$. This will complete the proof of the lemma because we will have shown that if $n > 1$ is not a power of $p$, then each of the $\beta_{p^j,\,n-p^j}$ either universally vanishes or is a universal multiple of one particular $\beta_{p^{j_0},\,n-p^{j_0}}$, while if $n = p^{r+1}$, then they all vanish except perhaps for $\beta_{p^r,\,n-p^r}$. So it remains to prove this claim.

So assume that $p^j < n$ and that $n \neq p^{j+1}$, and apply (\ref{identitiesproofeqn1}) with $a = p^j$, $b = (p-1)p^j$, $c = n - p^{j+1} \geq 0$ by (\ref{ngeqp^j+1}). Since $\binom{p^{j+1}}{p^j} = 0$ in $\F_p$ by Lemma \ref{binomialcoeffcongruence}, we get
\[
\binom{n-p^j}{(p-1)p^j}\beta_{p^j,\,n-p^j} = 0.
\]
On the other hand, because $p^j < n$ and the base $p$ expansion of $n$ has vanishing $p^j$ coefficient, the coefficient of $p^j$ in the base $p$ expansion of $n - p^j$ is $p-1$. By Lemma \ref{binomialcoeffcongruence} again, therefore, $\binom{n-p^j}{(p-1)p^j} \neq 0$, hence $\beta_{p^j,\,n-p^j} = 0$, as claimed.
\end{proof}

\begin{definition}
\label{frobeniussemilinear}
Let $R$ be an $\F_p$-algebra. An additive map $T:  M \rightarrow N$ of $R$-modules is said to be {\em Frobenius semilinear} if $T(rm) = r^pT(m)$ for all $r \in R$ and $m \in M$. It is said to be {\em Frobenius inverse semilinear} if $T(r^pm) = rT(M)$ for all $r \in R$ and $m \in M$.
\end{definition}

\begin{lemma}
\label{F^*isomH^i(Ga^)}
For any $\F_p$-algebra $R$, the relative Frobenius $R$-isogeny $F:  \mathbf{G}_{a,\,R} \rightarrow \mathbf{G}_{a,\,R}$ induces a Frobenius inverse semilinear isomorphism of $R$-modules ${\rm{H}}^i(\widehat{F})\colon {\rm{H}}^i(R, \widehat{\Ga}) \rightarrow {\rm{H}}^i(R, \widehat{\Ga})$ for all $i > 2$.
\end{lemma}

\begin{proof}
The Frobenius inverse semilinearity follows from the contravariant functoriality of $\widehat{\hspace{.1 in}}^*$, together with the equality $F\circ m_r = m_{r^p} \circ F$ for $r \in R$, where $m_r:  \mathbf{G}_{a,\,R} \rightarrow \mathbf{G}_{a,\,R}$ is the multiplication by $r$ map. To show that ${\rm{H}}^i(\widehat{F})$ is an isomorphism, we begin with the short exact sequence of $R$-group schemes
\begin{equation}
\label{F^*isomH^i(Ga^)pfeqn1}
1 \longrightarrow \alpha_p \longrightarrow \Ga \xlongrightarrow{F} \Ga \longrightarrow 1.
\end{equation}
By Proposition \ref{ext^1(Ga,Gm)=0}, the induced sequence
\[
1 \longrightarrow \widehat{\Ga} \xlongrightarrow{\widehat{F}} \widehat{\Ga} \longrightarrow \widehat{\alpha_p} \longrightarrow 1
\]
is also short exact. To prove the lemma, it therefore suffices to show that ${\rm{H}}^i(R, \widehat{\alpha_p}) = 0$ for all $i > 1$. But this follows from the isomorphism $\widehat{\alpha_p} \simeq \alpha_p$, together with the sequence (\ref{F^*isomH^i(Ga^)pfeqn1}) and the fact that the groups ${\rm{H}}^i(R, \Ga)$ agree with the Zariski cohomology groups (because $\Ga$ is quasi-coherent), hence -- due to the affineness of $\Spec(R)$ -- vanish for $i > 0$.
\end{proof}

\begin{lemma}
\label{F-Ikerp^n}
Let $R$ be a strictly Henselian local ring in which $p = 0$, and let $M$ be a free $R$-module of rank $n$. Let $T:  M \rightarrow M$ be a Frobenius semilinear isomorphism, and let $I$ denote the identity map of $M$. Then $T - I:  M \rightarrow M$ is surjective with kernel of cardinality $p^n$.
\end{lemma}

Pairs $(M, T)$ as in Lemma \ref{F-Ikerp^n} are related to modules over rings of Witt vectors; see \cite[Prop.\,4.1.1]{katz}.

\begin{proof}
The map $T - I$ defines a morphism $\phi$ of the vector group scheme $V(M)$ associated to $M$. We claim that $\phi$ is \'etale. In fact, this may be checked on the fibers over a point of $\Spec(R)$, due to the fibral flatness criterion, hence we reduce this \'etaleness claim to the case when $R$ is a field, for which it suffices to note that the morphism $T$ has trivial differential due to Frobenius semilinearity, hence $\phi$ induces an isomorphism on Lie algebras, and is therefore \'etale.

Next we claim that $\phi$ is a finite morphism of constant degree $p^n$. In fact, if we choose a basis $\{e_i\}$ for $M$, then we have for each $i$ that
\[
T(e_i) = \sum_j r_{ij} e_j
\]
for some $r_{ij} \in R$. Then on coordinate rings the map $\phi$ is given by the map
\[
R[E_1, \dots, E_n] \rightarrow R[X_1, \dots, X_n]
\]
given by $E_j \mapsto -X_j + \sum_i X_i^pr_{ij}$. Thus, $\ker(\phi)$ is given by the vanishing locus of these $n$ expressions in the $X_i$. Since $T$ is an isomorphism, the matrix $A : = \{r_{ij}\}$ is invertible, hence this kernel is defined by the equation
\[
[E_1^p \dots E_n^p] = [E_1 \dots E_n]A^{-1},
\]
and is therefore a finite $R$-module. The morphism $\phi:  V(M) \rightarrow V(M)$ is therefore finite \'etale surjective of constant degree $p^n$.

Now an element of $M$ is the same as an $R$-valued point of $V(M)$, so in order to prove the lemma we need to show that for any $x \in V(M)(R)$, the fiber $\phi^{-1}(x)$ has exactly $p^n$ $R$-valued points. But because $\phi$ is finite \'etale of constant degree $p^n$, this fiber is a finite \'etale $R$-scheme of order $p^n$. Since $R$ is strictly Henselian local, it follows that it is a disjoint union of $p^n$ copies of $\Spec(R)$, hence has exactly $p^n$ $R$-points.
\end{proof}

\begin{lemma}
\label{F-Ikerp^t+f}
Let $R$ be a strictly Henselian DVR of characteristic $p > 0$ with residue field $\kappa$, and let $M$ be a finitely generated $R$-module. Let $n : = {\rm{dim}}_{\kappa}(\kappa \otimes_R M)$. Suppose that $T:  M \rightarrow M$ is a group isomorphism that is either Frobenius semilinear or Frobenius inverse semilinear, and let $I$ denote the identity map of $M$. Then $T-I:  M \rightarrow M$ is surjective with kernel of order $p^n$.
\end{lemma}

\begin{proof}
We first treat the case in which $T$ is Frobenius semilinear. We claim that $T$ induces isomorphisms $M_{\rm{tors}} \rightarrow M_{\rm{tors}}$ and $M/M_{\rm{tors}} \rightarrow M/M_{\rm{tors}}$. Indeed, it maps torsion elements into torsion elements, and the first map is clearly injective while the second is clearly surjective. To show that both maps are isomorphisms, it suffices to check that the map $M_{\rm{tors}} \rightarrow M_{\rm{tors}}$ induced by $T$ is surjective. But given $y \in M_{\rm{tors}}$, we have $y = T(x)$ for some $x \in M$. Since $y$ is torsion, we have $\pi^{p\ell}y = 0$ for some $\ell \geq 0$, hence $\pi^{p\ell}T(x) = T(\pi^{\ell}x) = 0$, so $\pi^{\ell}x = 0$, hence $x$ is torsion.

Applying the snake lemma to the diagram
\[
\begin{tikzcd}
0 \arrow{r} & M_{\rm{tors}} \arrow{d}{T-I} \arrow{r} & M \arrow{r} \arrow{d}{T-I} & M/M_{\rm{tors}} \arrow{r} \arrow{d}{T-I} & 0 \\
0 \arrow{r} & M_{\rm{tors}} \arrow{r} & M \arrow{r} & M/M_{\rm{tors}} \arrow{r} & 0
\end{tikzcd}
\]
shows that it suffices to treat the $R$-modules $M_{\rm{tors}}$ and $M/M_{\rm{tors}}$. That is, we may assume that $M$ is either free or torsion. The case when $M$ is free follows from Lemma \ref{F-Ikerp^n}.

Next we treat the case when $M$ is torsion. If $\pi M \neq 0$, then $0 \neq T^{\ell}(\pi M) \subset \pi^{p\ell}M$ for all positive integers $\ell$, which is a contradiction because $\pi^{p\ell}M = 0$ for large $\ell$. Therefore, $\pi M = 0$, so the result follows from Lemma \ref{F-Ikerp^n} because $\kappa$ is separably closed. This completes the proof of the lemma when $T$ is Frobenius semilinear.

Now suppose that $T$ is Frobenius inverse semilinear. Then $T^{-1}$ is a Frobenius semilinear automorphism of $M$, hence $T^{-1} - I:  M \rightarrow M$ is surjective with kernel of order $p^n$. The same therefore holds for $T \circ (T^{-1} - I) = -(T - I)$.
\end{proof}

\begin{lemma}
\label{H^0(O,Ext^2(Ga,Gm))=0}
Let $S$ be a scheme, each local ring of which is a valuation ring. Then ${\rm{H}}^0(S, \calExt^2(\Ga, \Gm)) = 0$, where, as usual, $\calExt^2$ denotes the fppf Ext sheaf.
\end{lemma}

\begin{proof}
We may assume that $S = {\rm{Spec}}(V)$ with $V$ a valuation ring. By \cite[Tag 0ETC]{stacks}, any fppf cover of $S$ is refined by a (not necessarily finitely presented) cover of the form ${\rm{Spec}}(A)$ for $V \rightarrow A$ an extension of valuation rings. Choose an extension $K/{\rm{Frac}}(A)$ with degree of imperfection $\leq 1$ (e.g., take $K$ perfect). By \cite[Tag 01J8]{stacks}(ii), there is an extension $A \rightarrow B$ of valuation rings such that ${\rm{Frac}}(B) = K$. We may assume that our global section $\alpha$ of $\calExt^2(\Ga, \Gm)$ arises from an element of $\Ext^2_B(\Ga, \Gm)$. Any such element dies under pullback by a suitable power of the absolute Frobenius $F_B$ of $B$; see the argument in the third paragraph of the proof of Proposition \ref{H^2=Ext^2=Brdvr}. (In fact, if we take $K$ to be perfect, then one may show by the same method as there that $\Ext^2_B(\Ga, \Gm) = 0$.) We have thus constructed a faithfully flat cover ${\rm{Spec}}(C) \rightarrow {\rm{Spec}}(B)$ over which $\alpha$ dies. By Lemma \ref{valringlimoffppfalg}, $C$ is a filtered direct limit of fppf $V$-algebras, hence $\alpha$ dies over some fppf $V$-algebra, and therefore $\alpha = 0$.
\end{proof}

We now give an inclusion of ${\rm{H}}^3(R, \widehat{\Ga})$ into a group that is more tangible.

\begin{lemma}
\label{H^3(G_a^)inclusionBrauerseq}
Let $R$ be a ring. Consider the complex
\[
{\rm{H}}^2(\mathbf{G}_{a,\,R}, \Gm) \xrightarrow{f_0} {\rm{H}}^2(\mathbf{G}_{a,\,R}^2, \Gm) \xrightarrow{f_1} {\rm{H}}^2(\mathbf{G}_{a,\,R}^2, \Gm) \times {\rm{H}}^2(\mathbf{G}_{a,\,R}^3, \Gm),
\]
where $f_0 := m^* - \pi_1^* - \pi_2^*$, with $m, \pi_i\colon \Ga^2 \rightarrow \Ga$ the addition and projection maps, respectively, and where $f_1$ on the first coordinate is ${\rm{id}} \times \sigma^*$, where $\sigma(x, y) = (y, x)$, and on the second coordinate $f_1 = (\pi_1 \times m_{23})^* + \pi_{23}^* - (m_{12} \times \pi_3)^* - \pi_{12}^*$, where $\pi_{ij}\colon \Ga^3 \rightarrow \Ga^2$ is projection onto the $i, j$ coordinates, $m_{ij}\colon \Ga^3 \rightarrow \Ga$ is addition of the $i$ and $j$ coordinates, and $\pi_i\colon \Ga^3 \rightarrow \Ga$ is projection onto the $i$ coordinate. 

\begin{itemize}
\item[(i)] If $R$ is an $\F_p$-algebra, then the map $$\ker(f_1)[p] \rightarrow \left(\frac{\ker(f_1)}{\im(f_0)}\right)[p]$$ is surjective.
\item[(ii)] If $R$ is an $\F_p$-algebra each local ring of which is a valuation ring, then there is a functorial inclusion
\[
{\rm{H}}^3(R, \widehat{\Ga}) \hookrightarrow \left(\frac{\ker(f_1)}{\im(f_0)}\right)[p].
\]
\end{itemize}
\end{lemma}

\begin{proof}
Assertion (i) follows from the fact that the group ${\rm{H}}^2(\mathbf{G}_{a,\,R}, \Gm)$ is $p$-divisible \cite[Prop.\,]{treger}, hence so is $\im(f_0)$. It remains to prove (ii). The proof will use Breen's spectral sequences discussed in Remark \ref{breensequences}. We first consider the Leray spectral sequence
\[
D_2^{i,j} = {\rm{H}}^i(R, \calExt^j_R(\Ga, \Gm)) \Longrightarrow \Ext^{i+j}_R(\Ga, \Gm).
\]
The group $D_2^{1,1}$ vanishes by Proposition \ref{ext^1(Ga,Gm)=0}, while $D_2^{0,2}$ does by Lemma \ref{H^0(O,Ext^2(Ga,Gm))=0}. It follows that the $R$-linear (for the actions coming from the $R$-action on $\Ga$ given by $X \mapsto rX$ for $r \in R$) map 
\begin{equation}
\label{H^3(G_a^)=0pfeqn1}
{\rm{H}}^3(R, \widehat{\Ga}) = D_2^{3,0} \rightarrow D^3 = \Ext^3_R(\Ga, \Gm)
\end{equation}
is injective.

Next we consider the Breen spectral sequence (\ref{E_2spectralseqbreen21}): 
\[
{}^{\prime}E_2^{i,j} = \Ext^i_R({\rm{H}}_j(A(\Ga)), \Gm) \Longrightarrow \Ext^{i+j}_R(A(\Ga), \Gm).
\]
This provides an $R$-linear map 
\begin{equation}
\label{H^3(G_a^)=0pfeqn2}
\Ext^3_R(\Ga, \Gm) = {}^{\prime}E_2^{3,0} \rightarrow {}^{\prime}E^3 = \Ext^3_R(A(\Ga), \Gm),
\end{equation}
(where the first equality is by (\ref{G=H_0(A)})) which we claim is also injective. To prove this injectivity, it suffices to show that ${}^{\prime}E_2^{1,1} = {}^{\prime}E_2^{0,2} = 0$. The first group vanishes because ${\rm{H}}_1(A(\Ga)) = 0$ by (\ref{H_1(A)=0}), and the second group -- using (\ref{H_1(A)=0}) again -- equals $\Hom_R(\Ga/2\Ga, \Gm)$, and $\Ga$ admits no nontrivial homomorphisms to $\Gm$ over a reduced ring by Proposition \ref{cohomologyofG_adualgeneralk}(i).

In order to study the group $\Ext^3_R(A(\Ga), \Gm)$, we use the other Breen spectral sequence (\ref{E_1specseqbreen20}): 
\[
E_1^{i, j} = \widetilde{{\rm{H}}}^j(X_i, \Gm) \Longrightarrow \Ext^i_R(A(\Ga), \Gm),
\]
where recall that the $X_i$ are some explicit finite disjoint unions of finite products of copies of $\mathbf{G}_{a,\,R}$. Because the only global units on $\mathbf{G}_{a,\,R}^n$ are those coming from the base (Lemma \ref{nonconstantunitsred}), and because $\Pic(\mathbf{G}_{a,\,R}^n) = \Pic(R)$ (because $R$ is seminormal), we see that 
\begin{equation}
\label{H^3(G_a^)=0pfeqn3}
E_1^{i,0} = E_1^{i,1} = 0 \mbox{ for all } i.
\end{equation}
We claim that the edge map $E^3 = \Ext^3_R(A(\Ga), \Gm) \rightarrow E_1^{0,3} \subset {\rm{H}}^3(X_0, \Gm)$ is trivial. Indeed, we claim that the source is $p$-torsion, hence it suffices to show that ${\rm{H}}^3(X_0, \Gm)[p] = 0$, and this follows from \cite[\S 1, Thm.\,]{treger}. To see that the source is $p$-torsion, we note that $A(\cdot)$ is a functor from the category of representable group schemes to the category of chain complexes of abelian fppf sheaves, hence -- since $\Ga$ is $p$-torsion -- the same must hold for $A(\Ga)$, hence for $\Ext^3_R(A(\Ga), \Gm)$. 

It follows from the above and (\ref{H^3(G_a^)=0pfeqn3}) that we obtain an inclusion -- functorial in $R$-morphisms of $\Ga$ --
\begin{equation}
\label{H^3(G_a^)=0pfeqn4}
\Ext^3_R(A(\Ga), \Gm) = E^3 \hookrightarrow E_2^{1,2}[p].
\end{equation}

So we now compute $E_2^{1,2}$. By (\ref{X_idescription}), $X_0 = \mathbf{G}_{a,\,R}$, $X_1 = \mathbf{G}_{a,\,R}^2$, and $X_2 = \mathbf{G}_{a,\,R}^2 \coprod \mathbf{G}_{a,\,R}^3$. The differentials $E_1^{0, i} \rightarrow E_1^{1,i}$ and $E_1^{1,i} \rightarrow E_1^{2,i}$ are given by the maps in (\ref{breenX_0to1differential}) and (\ref{breenX_1to2differential}), so that we have the complex of the lemma:
\[
{\rm{H}}^2(\mathbf{G}_{a,\,R}, \Gm) \xrightarrow{f_0} {\rm{H}}^2(\mathbf{G}_{a,\,R}^2, \Gm) \xrightarrow{f_1} {\rm{H}}^2(\mathbf{G}_{a,\,R}^2, \Gm) \times {\rm{H}}^2(\mathbf{G}_{a,\,R}^3, \Gm)
\]
(where we may omit the tildes above the various ${\rm{H}}^2$'s because we are interested only in the cohomology of the complex, which is unaffected by this change), then
\[
E_2^{1,2} = \frac{\ker(f_1)}{{\rm{im}}(f_0)}.
\]
Taking the composition of the maps (\ref{H^3(G_a^)=0pfeqn1}), (\ref{H^3(G_a^)=0pfeqn2}), and (\ref{H^3(G_a^)=0pfeqn4}), therefore, we obtain a functorial inclusion
\begin{equation}
\label{H^3(Ga^)=0pfeqn5}
\zeta:  {\rm{H}}^3(R, \widehat{\Ga}) \hookrightarrow \left(\frac{\ker(f_1)}{{\rm{im}}(f_0)}\right)[p].
\end{equation}
\end{proof}

We are now ready to prove one of the main results of this section.

\begin{proposition}
\label{H^3(G_a^)=0}
Let $R$ be an $\F_p$-algebra that is either a field or a DVR and such that the fraction field $k$ of $R$ satisfies $[k:  k^p] = p$. Then ${\rm{H}}^3(R, \widehat{\Ga}) = 0$.
\end{proposition}

\begin{proof}
By Lemma \ref{H^3(R,G_a^)injective}, we may assume that $R$ is strictly Henselian. We use the injection of Lemma \ref{H^3(G_a^)inclusionBrauerseq}(ii), which (by functoriality) respects the (not a priori additive in $R$ for the group associated to the complex; see Remark \ref{nolinearaction}) $R$-action on both groups. Applying Lemma \ref{H^3(G_a^)inclusionBrauerseq}(i) and Proposition \ref{brauerdifferentialforms}, any element of $(\ker(f_1)/{\rm{im}}(f_0))[p]$ is represented by a Brauer element of the form $\psi(\overline{\omega})$ for some differential form $\omega \in \Omega^1_{R[X, Y]}$, where $\overline{\omega}$ denotes the image in $\Omega^1_{R[X, Y}/(d(R[X, Y]) + (C^{-1}-i)(\Omega^1_{R[X, Y]}))$. By Lemma \ref{F^*isomH^i(Ga^)}, any such element coming from ${\rm{H}}^3(R, \widehat{\Ga})$ via the map $\zeta$ of (\ref{H^3(Ga^)=0pfeqn5}) is represented by $\psi(\overline{F^*\omega})$ for some $\omega \in \Omega^1_{R[X,Y]}$. Then $F^*\omega = H(X, Y)\omega'$ for some $H \in R[X, Y]$ and $\omega' \in \Omega^1_R$, since $d(X^p) = d(Y^p) = 0$. Using \cite[Thm.\,26.5]{crt} when $R$ is a field and Lemma \ref{Omega^1=Rdpi} when $R$ is a DVR, there is an element $t \in R$ such that $R = \sum_{i=0}^{p-1} t^iR^p$ and such that $\Omega^1_R$ is free of rank one with generator $dt$. By Lemma \ref{primetop}, therefore, for any $\alpha \in {\rm{H}}^3(R, \widehat{\Ga})$, we have (where $\psi$ is the map in Proposition \ref{brauerdifferentialforms}) $\zeta(\alpha) = \psi(\overline{H(X, Y)dt})$ for some $H \in R[X, Y]$ such that $H$ contains no nonconstant monomials of the form $(X^aY^b)^p$. Since $\psi(Hdt) \in \ker(f_1)$, it follows from the very definition of $f_1$ that
\[
\psi(\overline{(H(X, Y) - H(Y, X))dt}) = 0,
\]
\[
\psi(\overline{(H(X, Y+Z) + H(Y, Z) - H(X+Y, Z) - H(X, Y))dt}) = 0.
\]

Since $H$ contains no non-constant monomials of the form $X^{pa}Y^{pb}$, Lemma \ref{brauerclassdeterminespolynomial} implies that we must have
\[
H(X, Y) - H(Y, X) = 0,
\]
\[
H(X, Y+Z) + H(Y, Z) - H(X+Y, Z) - H(X, Y) = 0.
\]
Then -- again using the fact that $H$ contains no non-constant monomials with both exponents divisible by $p$ -- Lemma \ref{identities} implies that $H(X, Y) = G(X+Y) - G(X) - G(Y) + rS(X, Y)$ for some $G \in R[X, Y]$ and some $r \in R$, where recall that $S$ is the mod $p$ reduction of the Witt polynomial $((X+Y)^p-X^p-Y^p)/p \in \Z[X, Y]$. 

This means that, by modifying $\psi(\overline{H(X, Y)dt})$ by the element $$f_0(\psi(\overline{G(X)dt})) = \psi(\overline{(G(X+Y)-G(X)-G(Y))dt}),$$ we may assume that $H(X, Y) = rS(X, Y)$ for some $r \in R$. Because $S(X, Y)$ is homogenous of degree $p$, the $R$-action on $\Ga$ induces an $R$-action on $\psi(\overline{R\cdot S(X, Y)dt})$ by having $r \in R$ act via multiplication by $r^p$:  $r \cdot \psi(\overline{r'S(X, Y)dt}) = \psi(\overline{r'S(rX, rY)dt}) = \psi(\overline{r^pr'S(X, Y)dt})$. We have just shown that the map of $R$-sets (with actions that are multiplicative in $R$ but not a priori necessarily additive in $R$; see Remark \ref{nolinearaction})
\begin{equation}
\label{RtokernelH^3vanisheqn}
R \longrightarrow \left(\frac{\ker(f_1)}{{\rm{im}}(f_0)}\right)[p]
\end{equation}
sending $r$ to the class of $\psi(\overline{rS(X, Y)dt})$ -- and with $R$-action on the left given by $r\cdot s : = r^ps$ -- has image containing the $R$-submodule (and not just $R$-set) ${\rm{H}}^3(R, \widehat{\Ga}) \simeq \zeta({\rm{H}}^3(R, \widehat{\Ga}))$. Further, since the $R$-action on the source is $R$-additive (since $r \mapsto r^p$ is an additive endorphism of $R$), we deduce that the $R$-action on the target is also additive. That is, (\ref{RtokernelH^3vanisheqn}) is a homomorphism of $R$-modules when the codomain is restricted to the image of (\ref{RtokernelH^3vanisheqn}). We deduce that the $R$-submodule ${\rm{H}}^3(R, \widehat{\Ga})$ of this image is a {\em finitely generated} $R$-module, because $R$ with $R$-action given for $r \in R$ by multiplication by $r^p$ is a finitely generated $R$-module.

If $F:  \mathbf{G}_{a,\,R} \rightarrow \mathbf{G}_{a,\,R}$ is the relative Frobenius $R$-isogeny, then by Lemma \ref{F^*isomH^i(Ga^)}, the induced map 
\[
{\rm{H}}^3(\widehat{\Ga}):  {\rm{H}}^3(R, \widehat{\Ga}) \rightarrow {\rm{H}}^3(R, \widehat{\Ga})
\]
is a Frobenius inverse semilinear isomorphism. We claim that the map 
\[
{\rm{H}}^3(\widehat{F}) - I:  {\rm{H}}^3(R, \widehat{\Ga}) \rightarrow {\rm{H}}^3(R, \widehat{\Ga})
\]
has trivial kernel. Since ${\rm{H}}^3(R, \widehat{\Ga})$ is finitely generated and $R$ is strictly Henselian, Lemma \ref{F-Ikerp^t+f}, together with Nakayama's Lemma, will then imply that ${\rm{H}}^3(R, \widehat{\Ga}) = 0$.

We have ${\rm{H}}^3(\widehat{F}) - I = {\rm{H}}^3(\widehat{F-I})$. The short exact sequence of $R$-group schemes
\[
1 \longrightarrow \Z/p\Z \longrightarrow \Ga \xlongrightarrow{F-I} \Ga \longrightarrow 1,
\]
yields by Proposition \ref{ext^1(Ga,Gm)=0} a short exact sequence
\[
1 \longrightarrow \widehat{\Ga} \xlongrightarrow{\widehat{F-I}} \widehat{\Ga} \longrightarrow \mu_p \longrightarrow 1.
\]
In order to prove that ${\rm{H}}^3(\widehat{F}) - I:  {\rm{H}}^3(R, \widehat{\Ga}) \rightarrow {\rm{H}}^3(R, \widehat{\Ga})$ has trivial kernel, therefore, it suffices to show that ${\rm{H}}^2(R, \mu_p) = 0$. But this follows from the exact sequence of $R$-group schemes
\[
1 \longrightarrow \mu_p \longrightarrow \Gm \xlongrightarrow{p} \Gm \longrightarrow 1
\]
together with the fact that ${\rm{H}}^i(R, \Gm) = 0$ for $i > 0$, since the cohomology may be taken to be \'etale due to the smoothness of $\Gm$, and because the higher \'etale cohomology of strictly Henselian local rings vanishes.
\end{proof}

We next turn to proving that the analogue of Proposition \ref{H^3(G_a^)=0} holds for a (possibly infinite) product of rings as in that proposition, a result that will be useful in our study of the cohomology of the adeles. To this end, we first prove the following result.

\begin{proposition}
\label{prodprimbrauer}
Let $\{A_i\}_{i \in I}$ be a set of $\F_p$-algebras, and let $k_i := {\rm{Frac}}(A_i)$. Assume that $[k_i: k_i^p] = p$ for each $i$, and that each $A_i$ is either a field or a DVR. Let $A := \prod_{i \in I} A_i$. Then the canonical map
\[
{\rm{H}}^2(\mathbf{G}_{a,\,A}, \Gm)_{\prim} \longrightarrow \prod_{i \in I} {\rm{H}}^2(\mathbf{G}_{a,\, A_i}, \Gm)_{\prim}
\]
is an isomorphism.
\end{proposition}

\begin{proof}
By Lemmas \ref{prodvalrings} and \ref{proddegimpleq1}, each local ring of ${\rm{Spec}}(A)$ is a valuation ring whose fraction field has degree of imperfection $\leq 1$. Since $${\rm{H}}^2(\mathbf{G}_{a,\, A}, \Gm)_{\prim} \subset {\rm{H}}^2(\mathbf{G}_{a,\,A}, \Gm)[p],$$ Proposition \ref{brauerdifferentialforms} therefore allows us to express any element of the former group as the image of an element in 
\begin{equation}
\label{prodprimbrauerpfeqn1}
\frac{\Omega^1_{A[X]}}{d(A[X]) + (C^{-1}-i)(\Omega^1_{A[X]})}
\end{equation}
under the functorial map $\psi$ of that proposition. Lemma \ref{frobsurjectiveH^2(Ga^)} therefore allows us to write any element $\alpha \in {\rm{H}}^2(\mathbf{G}_{a,\, A}, \Gm)_{\prim}$ in the form $\psi(\overline{F^*\omega}_\alpha)$ for some $\omega_\alpha \in \Omega^1_{A[X]}$, where $\overline{\omega}_\alpha$ denotes the image of $\omega_\alpha$ in the quotient (\ref{prodprimbrauerpfeqn1}). Since $F^*\omega_\alpha = G(X)\omega$ for some $\omega \in \Omega^1_A$ (because $d(X^p) = 0$), it follows that $\alpha = \psi(\overline{G(X)\omega})$. Assume that $\alpha \mapsto 0 \in {\rm{H}}^2(\mathbf{G}_{a,\,A_i}, \Gm)_{\prim}$ for each $i$, and we will show that $\alpha = 0$. Indeed, for each $A_i$ define an element $\pi_i$ as follows: if $A_i$ is a field, let $\pi_i$ be a p-basis of $A_i$, while if $A_i$ is a DVR, let $\pi_i$ be a uniformizer. Then $A_i = \oplus_{j =0}^{p-1} A_i^p\pi_i^j$ (this uses Lemma \ref{Omega^1=Rdpi} when $A_i$ is a DVR). Let $\pi := (\pi_i)_i \in A$. Then $A = \oplus_{j =0}^{p-1} A^p\pi^j$. Lemma \ref{primetop} then implies that $\alpha = \psi(\overline{H(X)d\pi})$ for some $H \in A[X]$ containing no non-constant monomials of degree divisible by $p$. Since $\psi_i(\overline{H_id\pi_i}) \in {\rm{H}}^2(\mathbf{G}_{a,\, A_i}, \Gm)$ vanishes, Lemma \ref{brauerclassdeterminespolynomial} then implies that $H_i(X) = c_i$ is constant for all $i$. Therefore, $H(X)$ is a constant polynomial, say $H(X) = a \in A$. Restricting a primitive Brauer class to the identity point $X = 0$ kills it, so we find that $cd\pi$ represents the trivial element of ${\rm{H}}^2(A, \Gm)$. It follows that $\alpha = 0$. This proves injectivity.

It remains to prove that the map in the proposition is surjective. Suppose given $\alpha_i \in {\rm{H}}^2(\mathbf{G}_{a,\, A_i}, \Gm)_{\prim}$ for each $i$. Proposition \ref{omega=H^2(Ga^)} implies that, for each $i$, there is $a_i \in A_i$ such that $\alpha_i = \psi_i(\overline{a_iXd\pi_i})$. Let $a := (a_i)_i \in A$. Then the element $\alpha := \psi(\overline{aXd\pi}) \in {\rm{H}}^2(\mathbf{G}_{a,\,A}, \Gm)_{\prim}$ maps to $\alpha_i$ for each $i$.
\end{proof}

\begin{proposition}
\label{H^3(Ga^)=0prodofdvrs}
Let $\{A_i\}_{i \in I}$ be a set of $\F_p$-algebras, and let $k_i := {\rm{Frac}}(A_i)$. Assume that $[k_i: k_i^p] = p$ for each $i$, and that each $A_i$ is either a field or a DVR. Let $A := \prod_{i \in I} A_i$. Then ${\rm{H}}^3(A, \widehat{\Ga}) = 0$.
\end{proposition}

\begin{proof}
By Lemmas \ref{prodvalrings} and \ref{proddegimpleq1}, each local ring of $A$ is a valuation ring whose fraction field has degree of imperfection $\leq 1$. Lemma \ref{H^3(G_a^)inclusionBrauerseq} furnishes a commutative diagram in which the horizontal arrows are inclusions:
\[
\begin{tikzcd}
{\rm{H}}^3(A, \widehat{\Ga}) \arrow[r, hookrightarrow, "j"] \arrow{d} & \frac{\ker(f_1)}{\im(f_0)}[p] \arrow{d} \\
{\rm{H}}^3(A_i, \widehat{\Ga}) \arrow[r, hookrightarrow, "\prod j_i"] & \frac{\ker(f_{1,\,i})}{\im(f_{0,\,i})}[p]
\end{tikzcd}
\]
Let $\alpha \in {\rm{H}}^3(A, \widehat{\Ga})$. By Lemma \ref{H^3(G_a^)inclusionBrauerseq}(i), $j(\alpha)$ is represented by an element $\beta \in \ker(f_1)[p]$. By Proposition \ref{H^3(G_a^)=0}, there exists for each $i \in I$ an element $\gamma_i \in {\rm{H}}^2(\mathbf{G}_{a,\,A_i}, \Gm)$ such that $f_{0,\,i}(\gamma_i) = \beta_i$. Then $p\gamma_i \in \ker(f_{0,\,i}) = {\rm{H}}^2(\mathbf{G}_{a,\,A_i}, \Gm)_{\prim}$. By Proposition \ref{prodprimbrauer}, there exists $\delta \in {\rm{H}}^2(\mathbf{G}_{a,\, A}, \Gm)_{\prim}$ such that $\delta_i = p\gamma_i$ for each $i$. Since ${\rm{H}}^2(\mathbf{G}_{a,\, A}, \Gm)$ is $p$-divisible \cite[\S 3, Prop.\,]{treger}, there exists $\epsilon \in {\rm{H}}^2(\mathbf{G}_{a,\, A}, \Gm)$ such that $p\epsilon = \delta$. Then
\[
\beta_i - f_{0,\,i}(\epsilon_i) = f_{0,\,i}(\gamma_i - \epsilon_i),
\]
and 
\[
p(\gamma_i - \epsilon_i) = p\gamma_i - \delta_i = 0,
\]
so replacing $\beta$ by $\beta - f_0(\epsilon)$, we see that $j(\alpha)$ is represented by an element $\zeta \in \ker(f_1)[p]$ with the property that, for each $i$, there is $\eta_i \in {\rm{H}}^2(\mathbf{G}_{a,\,A_i}, \Gm)[p]$ such that $\zeta_i = f_{0,\,i}(\eta_i)$.

The element $j({\rm{H}}^3(\widehat{F})(\alpha))$ is represented by ${\rm{H}}^3(\widehat{F})(\zeta)$, and ${\rm{H}}^3(\widehat{F})(\zeta_i) = f_{0,\,i}({\rm{H}}^3(\widehat{F})(\eta_i))$ for all $i$. Hence, renaming ${\rm{H}}^3(\widehat{F})(\alpha)$ as $\alpha$ using Lemma \ref{F^*isomH^i(Ga^)}, then by Proposition \ref{brauerdifferentialforms}, for every element $\alpha \in {\rm{H}}^3(A, \widehat{\Ga})$, there exist $G \in A[X, Y], \omega \in \Omega^1_A, H_i \in A_i[X]$, and $\theta_i \in \Omega^1_{A_i}$ such that $j(\alpha)$ is represented by $\psi(\overline{G(X)\omega})$ and $\psi_i(\overline{G(X)\omega_i}) = f_{0,\,i}(\psi(\overline{H_i(X)\theta_i}))$, where $\psi$ is the map of Proposition \ref{brauerdifferentialforms}, and $\overline{\cdot}$ denotes the image in $\Omega^1_{A[X, Y]}/(d(A[X, Y]) + (C^{-1}-i)(\Omega^1_{A[X, Y]}))$ (and similarly for $A_i[X]$).

Choose elements $\pi_i \in A_i$ as follows: if $A_i$ is a field, then $\pi_i$ is a $p$-basis of $A_i$, while if $A_i$ is a DVR, then $\pi_i \in A_i$ is a uniformizer. Then $A_i = \oplus_{j=0}^{p-1} A_i^p\pi_i^j$ (using Lemma \ref{Omega^1=Rdpi} when $A_i$ is a DVR). Let $\pi := (\pi_i)_i \in A$. Then $A = \oplus_{j=0}^{p-1} A^p\pi^j$. It follows from Lemma \ref{primetop} that we may choose $\omega = d\pi$ and $\theta_i = d\pi_i$, and $G$ and $H_i$ to have no nonconstant monomials whose degrees in both $X$ and $Y$ are divisible by $p$. Then, since $\psi_i(\overline{G(X)d\pi_i}) = f_{0,\,i}(\psi(\overline{H_i(X)d\pi_i}))$, we have for each $i$
\[
\psi(\overline{G_i(X, Y)d\pi_i}) = \psi(\overline{(H_i(X + Y) - H_i(X) - H_i(Y))d\pi_i}).
\]
It follows from Proposition \ref{brauerclassdeterminespolynomial} that $G_i(X, Y) - (H_i(X+Y) - H_i(X) - H_i(Y))$ is a constant polynomial. Modifying $H_i$ by a constant in $A_i$, therefore, we may assume that
\[
G_i(X, Y) = H_i(X+Y) - H_i(X) - H_i(Y)
\]
for all $i$. The above identity holds with $H_i$ replaced by the part of $H_i$ with degree $\leq {\rm{deg}}(G)$. The $H_i$ may therefore be chosen to be of uniformly bounded degree. Therefore, there exists $\tilde{H} \in A[X]$ such that $\tilde{H}_i = H_i$ for all $i$. Then $j(\alpha)$ is represented by
\[
\psi(\overline{G(X, Y)d\pi}) = \psi\left(\overline{(\tilde{H}(X+Y) - \tilde{H}(X) - \tilde{H}(Y))d\pi}\right) = f_0\left(\psi\left(\overline{\tilde{H}(X)d\pi}\right)\right),
\]
hence $j(\alpha) = 0$, so $\alpha = 0$. Since $\alpha \in {\rm{H}}^3(A, \widehat{\Ga})$ was arbitrary, this proves the proposition.
\end{proof}

\section{Comparison between \v{C}ech and sheaf cohomology}
\label{cech=derivedsection}

We shall in several places require an explicit description of certain cohomology classes in terms of \v{C}ech cocycles. In particular, if $G$ is an affine commutative group scheme of finite type over a field $k$, then we will need to know that the canonical map $\check{{\rm{H}}}^i(k, \mathscr{F}) \rightarrow {\rm{H}}^i(k, \mathscr{F})$ from fppf \v{C}ech to derived functor cohomology is an isomorphism in low degrees for $\mathscr{F} = G$ or $\widehat{G}$. In this section we prove this for $\mathscr{F} = G$ in all degrees and without affineness assumptions on $G$ (Proposition \ref{cech=derivedsmoothinf}), and we prove this for $\mathscr{F} = \widehat{G}$ up to degree $2$ (Proposition \ref{cech=derivedG^}). Milne proved the agreement of \v{C}ech and derived functor cohomology with coefficients in a finite type group scheme \cite[Chap.\,III, Prop.\,6.1]{milne}, but the proof has some gaps, especially the lack of a proof of the compatibility between \v{C}ech and derived functor constructions; see Appendix \ref{appendixcechvsderived} and particularly Proposition \ref{cechderivedconnectingmap}.

\begin{remark}
\label{cech=deriveddeg01remark}
The canonical map from \v{C}ech to derived functor cohomology is an isomorphism in degrees 0 and 1, and injective in degree $2$, for any sheaf on any site \cite[Chap.\,III, Cor.\,2.10]{milneetalecohomology}. The real content in the statement above for $\mathscr{F} = \widehat{G}$ is therefore the surjectivity in degree 2.
\end{remark}

\begin{proposition}
\label{cech=derivedperfect}
Let $k$ be a perfect field, and let $\mathscr{F}$ be an fppf abelian sheaf on the category of all $k$-schemes that is locally of finite presentation (Definition \ref{finitepresentationdef}). Then the canonical map $\check{{\rm{H}}}^i(k, \mathscr{F}) \rightarrow {\rm{H}}^i(k, \mathscr{F})$ is an isomorphism for all $i$.
\end{proposition}

\begin{proof}
By the Nullstellensatz and because $k$ is perfect, any fppf cover of $\Spec(k)$ may be refined by one of the form $\Spec(L)$ for some finite Galois extension $L/k$ contained in $\overline{k}$. We may therefore compute the \v{C}ech cohomology using only these covers. But these \v{C}ech cohomology groups are precisely the Galois cohomology groups of $\varinjlim_L \mathscr{F}(L)$, and these agree with the \'etale cohomology groups. By Proposition \ref{etrem}, they therefore also agree with the fppf cohomology groups.
\end{proof}

\begin{lemma}
\label{mayusekbar/k}
Let $k$ be a field, and let $\mathscr{F}$ be a presheaf on the category of all $k$-schemes such that $\mathscr{F}$ is locally of finite presentation. Then the \v{C}ech cohomology groups $\check{{\rm{H}}}^i(k, \mathscr{F})$ may be computed using the $($not generally fppf$)$ cover $\Spec(\overline{k}) \rightarrow \Spec(k)$.
\end{lemma}

\begin{proof}
By the Nullstellensatz, any fppf cover of $\Spec(k)$ may be refined by one of the form $\Spec(L)$, where $L$ is a finite extension of $k$ contained in $\overline{k}$, hence the \v{C}ech cohomology groups may be computed using only covers of the form $\Spec(L)$ for such $L$. Since $\mathscr{F}(\overline{k}^{\otimes_k n}) = \varinjlim_L \mathscr{F}(L^{\otimes_k n})$ for all $n > 0$, one may therefore replace a direct limit over these covers with the \v{C}ech groups obtained using the single cover $\Spec(\overline{k})$.
\end{proof}

\begin{lemma}
\label{filtrationinfinitesimal}
Let $k$ be a field of characteristic $p > 0$, and let $I$ be a commutative infinitesimal $k$-group scheme.
\begin{itemize}
\item[(i)] If $I$ is local-local, then it admits a finite filtration with successive quotients all $k$-isomorphic to $\alpha_p$.
\item[(ii)] If $k$ is separably closed, then $I$ admits a finite filtration with successive quotients all $k$-isomorphic to $\alpha_p$ or $\mu_p$.
\end{itemize}
\end{lemma}

\begin{proof}
(i) By repeatedly filtering by the kernel and cokernel of the Frobenius and Verschiebung maps, we may assume that $I$ is killed by both maps. Then the classification of finite group schemes over perfect fields coming from Dieudonn\'e theory implies that $I_{\overline{k}} \simeq \alpha_p^n$ for some $n \geq 0$. It therefore suffices to show that $\alpha_p^n$ admits no nontrivial $\overline{k}/k$-forms. Since the automorphism functor of $\alpha_p^n$ is ${\rm{GL}}_n$, such forms are classified by ${\rm{H}}^1(k, {\rm{GL}}_n)$, and this set is trivial because any vector bundle over ${\rm{Spec}}(k)$ is trivial.
\\

\noindent (ii) By dualizing the connected-\'etale sequence of the Cartier dual $\widehat{I}$, we may assume that $\widehat{I}$ is either \'etale or infinitesimal. The latter case is handled by (i). In the former case, the fact that $k = k_s$ implies that $\widehat{I}$ is a constant group scheme, hence $I$ is a product of groups of the form $\mu_{p^n}$.
\end{proof}

\begin{lemma}
\label{H^j(kbar^n,I)=0}
Let $G$ be a commutative group scheme locally of finite type over a field $k$. Then ${\rm{H}}^j(\overline{k}^{\otimes_k n}, G) = 0$ for all $j, n > 0$.
\end{lemma}

\begin{proof}
We first treat the case in which $G$ is smooth. Note that $(\overline{k}^{\otimes_k n})_{{\rm{red}}} = \overline{k}^{\otimes_{k_{\rm{perf}}} n}$. By Lemma \ref{picX = pic(Xred)}, therefore, we may replace $k$ with $k_{{\rm{perf}}}$ and thereby assume (for the purposes of proving the lemma for smooth $G$) that $k$ is perfect. By Proposition \ref{directlimitscohom}, in order to prove the lemma it is enough to show that
\begin{equation}
\label{cechderivedsmoothpfeqn2}
\varinjlim_L {\rm{H}}^j(L^{\otimes_k n}, G) = 0,
\end{equation}
where the limit is over all finite Galois extensions $L/k$ contained in $\overline{k}$. Since $G$ is smooth, the cohomology may be taken to be \'etale. But because higher \'etale cohomology over separably closed fields is trivial, it suffices to note that since $L^{\otimes_k n}$ is just a product of copies of $L$ for $L/k$ finite Galois, it follows that for any $\alpha \in {\rm{H}}^j(L^{\otimes_k n}, G)$ there is a finite extension $L'/L$ -- still Galois over $k$ -- such that the image of $\alpha$ dies in ${\rm{H}}^j(L'^{\otimes_k n}, G)$, because the induced map from a product of copies of $L$ to a product of copies of $L'$ is given on each factor of $L$ by a diagonal embedding into some subset of the factors in the product of copies of $L'$.

Next we treat the case in which $G = I$ is infinitesimal. We may assume that ${\rm{char}}(k) = p > 0$, since in the characteristic $0$ case $I$ must be trivial. By Lemma \ref{filtrationinfinitesimal}(iI), over $\overline{k}$, hence over $\overline{k}^{\otimes_k n}$, $I$ has a filtration with successive quotients isomorphic to either $\alpha_p$ or $\mu_p$. We may therefore assume that $I$ is one of these groups. The short exact sequences
\[
1 \longrightarrow \alpha_p \longrightarrow \Ga \xlongrightarrow{F} \Ga \longrightarrow 1 
\]
\[
1 \longrightarrow \mu_p \longrightarrow \Gm \xlongrightarrow{p} \Gm \longrightarrow 1
\]
(where $F$ is the relative Frobenius isogeny of $\Ga$ over $k$), together with the already-treated smooth case, reduce us to showing that the $p$th power map on $\overline{k}^{\otimes_k n}$ is surjective. Since this map is additive, the desired surjectivity follows from the surjectivity of the $p$th power map on $\overline{k}$.

In the general case, by \cite[VII$_{\rm{A}}$, Prop.\,8.3]{sga3}, there is an infinitesimal subgroup scheme $I \subset G$ such that $G/I$ is smooth. The lemma therefore follows from the already-treated smooth and infinitesimal cases.
\end{proof}

For a scheme $S$ and an fppf abelian sheaf $\mathscr{F}$ on $S$, let $\mathscr{H}^j(\mathscr{F})$ denote the presheaf $U \mapsto {\rm{H}}^j(U, \mathscr{F})$.

\begin{proposition}
\label{cech=derivedsmoothinf}
Let $G$ be a commutative group scheme locally of finite type over a field $k$. Then the canonical map $\check{{\rm{H}}}^i(k, G) \rightarrow {\rm{H}}^i(k, G)$ is an isomorphism for all $i$.
\end{proposition}

\begin{proof}
Consider the \v{C}ech-to-derived functor spectral sequence
\[
E_2^{i,j} = \check{{\rm{H}}}^i(k, \mathscr{H}^j(G)) \Longrightarrow {\rm{H}}^{i+j}(k, G).
\]
In order to prove the proposition, it suffices to show that $E_2^{i,j} = 0$ for all $j > 0$. By Proposition \ref{directlimitscohom}, the presheaves $\mathscr{H}^j(G)$ are locally of finite presentation, so by Lemma \ref{mayusekbar/k}, their \v{C}ech cohomology groups may be computed using the cover $\Spec(\overline{k}) \rightarrow \Spec(k)$. In order to show that $E_2^{i, j} = 0$ for all $j > 0$, therefore, it suffices to show that
\begin{equation}
\label{cechderivedsmoothpfeqn1}
{\rm{H}}^j(\overline{k}^{\otimes_k n}, G) = 0
\end{equation}
for all $j, n > 0$, and this is the content of Lemma \ref{H^j(kbar^n,I)=0}.
\end{proof}

We now turn to the comparison between \v{C}ech and derived functor cohomology for $\widehat{G}$.

\begin{lemma}
\label{cech=derivedsplitunipdual}
If $U$ is a smooth connected unipotent group over a field $k$ of characteristic $p > 0$, then the canonical map $\check{{\rm{H}}}^i(k, \widehat{U}) \rightarrow {\rm{H}}^i(k, \widehat{U})$ is an isomorphism for $i = 2$ and injective for $i = 3$.
\end{lemma}

Recall that for an fppf abelian sheaf on a scheme $S$, $\mathscr{H}^j(\mathscr{F})$ denotes the presheaf $V \mapsto {\rm{H}}^j(V, \mathscr{F})$.

\begin{proof}
Recall from Remark \ref{cech=deriveddeg01remark} that the map is automatically injective for $i = 2$. Consider the \v{C}ech-to-derived functor spectral sequence
\[
E_2^{i,j} = \check{{\rm{H}}}^i(k, \mathscr{H}^j(\widehat{U})) \Longrightarrow {{\rm{H}}}^{i+j}(k, \widehat{U}).
\]
To prove the lemma, it suffices to show that $E_2^{0,2} = E_2^{1,1} = 0$. Note that the presheaf $\mathscr{H}^j(\widehat{U})$ is locally of finite presentation by Proposition \ref{etrem}, so by Lemma \ref{mayusekbar/k} we may compute its \v{C}ech cohomology groups using the cover $\Spec(\overline{k}) \rightarrow \Spec(k)$.

In order to show that $E_2^{0,2}$ vanishes, it suffices to show that ${\rm{H}}^2(\overline{k}, \widehat{U}) = 0$, and this holds because the higher cohomology of $\overline{k}$ with coefficients in any fppf abelian sheaf vanishes due to the Nullstellensatz. In order to show that $E_2^{1,1} = \check{{\rm{H}}}^1(k, \mathscr{H}^1(\widehat{U}))$
vanishes, it suffices to show that ${{\rm{H}}}^1(\overline{k} \otimes_k \overline{k}, \widehat{U}) = 0$, and this in turn follows from Corollary \ref{H^1(U^)=0} because $U$ admits a filtration with successive $\Ga$ quotients over $\overline{k}$, hence over $\overline{k} \otimes_k \overline{k}$.
\end{proof}

\begin{lemma}
\label{leshat,step4}
Let $k$ be a field of positive characteristic, and suppose that we have a short exact sequence
\begin{equation}
\label{leshat,step4eqn1}
1 \longrightarrow H \longrightarrow G \longrightarrow U \longrightarrow 1
\end{equation}
of affine commutative $k$-group schemes with $U$ smooth, connected, and unipotent. Then this induces a long exact \v{C}ech cohomology sequence
\[
\dots \longrightarrow \check{{\rm{H}}}^i(k, \widehat{U}) \longrightarrow \check{{\rm{H}}}^i(k, \widehat{G}) \longrightarrow \check{{\rm{H}}}^i(k, \widehat{H}) \longrightarrow \check{{\rm{H}}}^{i+1}(k, \widehat{U}) \longrightarrow \dots
\]
that is compatible with the derived functor long exact sequence $($which exists by Proposition $\ref{hatisexact}$$)$.
\end{lemma}

\begin{proof}
We first note that by Lemma \ref{mayusekbar/k}, the \v{C}ech cohomology groups may be computed using the cover $\Spec(\overline{k}) \rightarrow \Spec(k)$. We want to show that the resulting sequence of \v{C}ech complexes obtained from applying $\calHom(\cdot, \Gm)$ to (\ref{leshat,step4eqn1}) is short exact. The compatibility of the \v{C}ech long exact sequence with the derived functor long exact sequence will then follow from Proposition \ref{cechderivedconnectingmap}. The left-exactness is immediate, so the only issue is right-exactness, which amounts to showing that for every $n > 0$, the map $\widehat{G}(\overline{k}^{\otimes_k n}) \rightarrow \widehat{H}(\overline{k}^{\otimes_k n})$ is surjective. The group ${{\rm{H}}}^1(\overline{k}^{\otimes_k n}, \widehat{U})$ vanishes by Corollary \ref{H^1(U^)=0} because $U$ admits a filtration with successive $\Ga$ quotients over $\overline{k}$, hence over $\overline{k}^{\otimes_k n}$. The map $\widehat{G}(\overline{k}^{\otimes_k n}) \rightarrow \widehat{H}(\overline{k}^{\otimes_k n})$ is therefore surjective by Proposition \ref{hatisexact}.
\end{proof}

\begin{proposition}
\label{cech=derivedG^}
Let $G$ be an affine commutative group scheme of finite type over a field $k$. Then the canonical map $\check{{\rm{H}}}^i(k, \widehat{G}) \rightarrow {\rm{H}}^i(k, \widehat{G})$ is an isomorphism for $i \leq 2$.
\end{proposition}

\begin{proof}
By Proposition \ref{cech=derivedperfect}, we may assume that $k$ is imperfect, so that ${\rm{char}}(k) > 0$. By Remark \ref{cech=deriveddeg01remark}, the only nontrivial point is the surjectivity when $i = 2$. By Lemma \ref{affinegroupstructurethm}, there is an exact sequence
\[
1 \longrightarrow H \longrightarrow G \longrightarrow U \longrightarrow 1
\]
with $H$ an almost-torus and $U$ smooth connected unipotent over $k$. Proposition \ref{hatisexact} and Lemma \ref{leshat,step4} then provide a commutative diagram with exact rows
\[
\begin{tikzcd}
\check{{\rm{H}}}^2(k, \widehat{U}) \arrow{r} \isoarrow{d} & \check{{\rm{H}}}^2(k, \widehat{G}) \arrow{r} \arrow[d, hookrightarrow] & \check{{\rm{H}}}^2(k, \widehat{H}) \arrow{r} \isoarrow{d} & \check{{\rm{H}}}^3(k, \widehat{U}) \arrow[d, hookrightarrow] \\
{\rm{H}}^2(k, \widehat{U}) \arrow{r} & {\rm{H}}^2(k, \widehat{G}) \arrow{r} & {\rm{H}}^2(k, \widehat{H}) \arrow{r} & {\rm{H}}^3(k, \widehat{U})
\end{tikzcd}
\]
in which the first vertical arrow is an isomorphism and the last one is an inclusion by Lemma \ref{cech=derivedsplitunipdual}, and the third vertical arrow is an isomorphism by Propositions \ref{hatrepresentable} and \ref{cech=derivedsmoothinf}. A diagram chase now shows that the second vertical arrow, which is an inclusion by Remark \ref{cech=deriveddeg01remark}, is an isomorphism.
\end{proof}

\chapter{Local Fields}
\label{chapterlocalfields}

In this chapter we establish the main local duality results (involving the perfection of the local duality pairings) stated in \S\ref{intro}, particularly Theorems \ref{localdualityH^2(G)}, \ref{localdualityH^2(G^)}, and \ref{H^1localduality} (see Propositions \ref{H^2(G)G^(k)dualityprop}, \ref{H^2(G^)altori}, and \ref{H^1dualityprop}). An important step in this process is defining and studying the topology on the groups ${\rm{H}}^1(k, G)$ and ${\rm{H}}^1(k, \widehat{G})$, where $k$ is a local field, and $G$ is an affine commutative $k$-group scheme of finite type. This is done in \S \ref{sectiontopologyoncohomology}. The proofs of Theorems \ref{localdualityH^2(G^)} and \ref{H^1localduality} are intertwined, so we begin by only proving Theorem \ref{H^1localduality} for almost-tori (\S \ref{localdualaltori}), and then turn to proving Theorem \ref{localdualityH^2(G^)} (\S\ref{sectionproofofH^2(G^)localduality}) before finally completing the proof of Theorem \ref{H^1localduality} (\S \ref{h1dualgen}). Before turning to these local duality results, however, we prove vanishing theorems for the cohomology of $G$ and $\widehat{G}$ over local and global fields that will prove useful in the sequel (Propositions \ref{cohomologicalvanishing} and \ref{H^3(G^)=0}).

\section{Cohomological vanishing over local and global fields}
\label{sectioncohomvanish}

The purpose of this section is to prove that the groups ${\rm{H}}^i(k, G)$ vanish for $i > 2$ whenever $k$ is either a non-archimedean local field or a global field with no real places and $G$ is an affine commutative $k$-group scheme of finite type (Proposition \ref{cohomologicalvanishing}), and to prove the same for ${\rm{H}}^3(k, \widehat{G})$ (Proposition \ref{H^3(G^)=0}). Although this chapter is primarily about local fields, we also prove these vanishing results in the global setting here because it is more efficient to handle both cases simultaneously.

\begin{lemma}
\label{R^if_*=0forineq1}
Let $S$ be a scheme, $I$ a finite locally free commutative $S$-group scheme. Let $f:  S_{\fppf} \rightarrow S_{\et}$ denote the natural morphism. Then $\R^if_*I = 0$ for $i > 1$. If we further assume that $I$ has connected fibers over $S$ and that $S$ is reduced, then the same holds for $i = 0$.
\end{lemma}

\begin{proof}
By \cite[II, 3.2.5]{messing}, there is an exact sequence
\[
1 \longrightarrow I \longrightarrow G \longrightarrow H \longrightarrow 1
\]
with $G, H$ smooth fiberwise connected affine commutative $S$-groups. Explicitly,
$G = {\rm{R}}_{\widehat{I}/S}({\rm{GL}}_1)$, with $\widehat{I}$ the (representable) Cartier dual of $I$, and $H : = G/A$, where the quotient exists as an affine $S$-group of
finite type over which $G$ is faithfully flat by \cite[V, 4.1]{sga3}.  The $S$-group $H$ inherits smoothness and fiberwise connectedness from its fppf cover $G$.

In order to prove the first assertion, therefore, it suffices to show that $\R^if_*G = 0$ for any smooth $S$-group scheme $G$ and all $i > 0$. But this \'etale sheaf is the (\'etale) sheafification of the presheaf $U \mapsto {\rm{H}}^i_{\rm{fppf}}(U, G)$. Since fppf and {\'e}tale cohomology agree for smooth commutative group schemes, this is the same as the \'etale sheafification of the presheaf $U \mapsto {\rm{H}}^i_{\et}(U, G)$, and this vanishes for $i > 0$, as it agrees with $\R^i{\rm{Id}}_*G$, where ${\rm{Id}}:  S_{\et} \rightarrow S_{\et}$ is the identity.

Finally, assume that $S$ is reduced and that $I$ has connected fibers. Then any \'etale $S$-scheme $T$ is also reduced. Since $I_T$ is finite and fiberwise connected, any $T$-section of $I_T$ is trivial in each fiber, hence, since $T$ is reduced, it is trivial. Therefore, $f_*I = 0$.
\end{proof}

\begin{proposition}
\label{cohomologicalvanishing}
Let $k$ be a non-archimedean local field or a global field with no real places, and let $G$ be a commutative $k$-group scheme of finite type. Then ${{\rm{H}}}^i(k, G) = 0$ for $i > 2$.
\end{proposition}

\begin{proof}
\underline{Step 1}:  $G$ is finite. We may write $G$ as the product of its $l$-primary parts for different primes $l$, hence we may assume that it is of prime power order. In the case that ${\rm{char}}(k) \nmid |G|$, the group $G$ is {\'e}tale, and the lemma follows from results of Poitou--Tate, using the fact that {\'e}tale and fppf cohomology agree for commutative smooth group schemes. (See \cite[Ch.\,I, Thm.\,4.10(c)]{milne} for the case in which $k$ is global, and \cite[Ch.\,II, \S5.3, Prop.\,15]{serre} for the case in which $k$ is local.) Thus, we may assume that ${\rm{char}}(k) = p > 0$ and that $G$ is of $p$-power order. 

Using the connected-{\'e}tale sequence, it suffices to treat the cases in which $G$ is {\'e}tale or infinitesimal. If $G$ is {\'e}tale, then since any field of characteristic $p$ has $p$-cohomological dimension $\leq 1$ \cite[Ch.\,II, \S2.2, Prop.\,3]{serre}, we are done. If $G$ is infinitesimal, then by Lemma \ref{R^if_*=0forineq1}, $R^if_*G \neq 0$ for $i \neq 1$, where $f:  \Spec(k)_{\fppf} \rightarrow \Spec(k)_{\et}$ is the natural morphism of sites. Therefore ${{\rm{H}}}^i_{\fppf}(k, G) = {{\rm{H}}}^{i-1}_{\et}(k, R^1f_*G)$, and this latter group is $0$ for $i > 2$, again because fields of characteristic $p$ have $p$-cohomological dimension $\leq 1$.
\\
\underline{Step 2}:  $G$ is a torus. By Lemma \ref{almosttorus} (iv), we may harmlessly modify $G$ in order to assume that there is an isogeny $\R_{k'/k}(T') \times A \twoheadrightarrow G$, for some finite $k$-group scheme $A$, some finite separable extension $k'/k$, and some split $k'$-torus $T'$.
Thus, by Step 1, we may assume that $G = \R_{k'/k}(\mathbf{G}_m)$. Since $\R_{k'/k}(\Gm)$ is smooth, we may take our cohomology to be {\'e}tale. By Lemma \ref{sepblepushforward}, we have an isomorphism ${\rm{H}}^i(k, \R_{k'/k}(\Gm)) \simeq {\rm{H}}^i(k', \Gm)$, so (renaming $k'$ as $k$) we are reduced to the case $G = \Gm$.

By Step 1, $k$ has cohomological dimension at most $2$, hence has strict cohomological dimension at most $3$. It follows that ${\rm{H}}^i(k, \Gm) = 0$ for $i > 3$, hence it only remains to show that ${\rm{H}}^3(k, \Gm) = 0$, which is proved in \cite[Ch.\,VII, \S11.4]{cf}. Alternatively, we may treat this by reduction to the finite case treated in Step 1. Since higher Galois cohomology is torsion, it suffices to show that ${\rm{H}}^i(k, \Gm)[n] = 0$ for $i > 2$ and any positive integer $n$. We have the Kummer sequence (of fppf sheaves)
\[
1 \longrightarrow \mu_n \longrightarrow \Gm \xlongrightarrow{n} \Gm \longrightarrow 1
\]
so the desired vanishing of ${\rm{H}}^i(k, \Gm)[n]$ for $i > 2$ follows from Step 1.
\\
\underline{Step 3}:  General $G$. By Step 1 and Lemma \ref{finitequotient=smoothandconnected}, we may assume that $G$ is smooth and connected. Using the exact sequence
\[
1 \longrightarrow \ker(F^{(p^n)}) \longrightarrow G \xlongrightarrow{F^{(p^n)}} G^{(p^n)} \longrightarrow 1,
\]
where $F^{(p^n)}$ is the $n$-fold relative Frobenius isogeny, it suffices to prove the proposition for some $n$-fold Frobenius twist of $G$. Over the perfect closure $k_{\rm{perf}}$ of $k$, Chevalley's Theorem \cite[Thm.\,1.1]{conradchevalley} provides an exact sequence
\[
1 \longrightarrow L \longrightarrow G_{k_{{\rm{perf}}}} \longrightarrow A \longrightarrow 1
\] 
with $L$ smooth, connected, and affine and $A$ an abelian variety. There is therefore such an exact sequence over $k$ for some $G^{(p^n)}$. We may therefore assume that $G$ is either smooth, connected, and affine, or an abelian variety.

First consider the affine case. If $T \subset G$ is the maximal torus, then $G/T$ is unipotent. By Step 2, we may therefore assume that $G$ is unipotent, and then we are done by Proposition \ref{unipotentcohomology}(i). Now suppose that $G = A$ is an abelian variety. Since higher Galois cohomology is torsion, it suffices to show that ${\rm{H}}^i(k, A)[n] = 0$ for all $i > 2$ and all positive integers $n$. This follows from Step 1 and the short exact sequence
\[
1 \longrightarrow A[n] \longrightarrow A \xlongrightarrow{n} A \longrightarrow 1. \qedhere
\]
\end{proof}

\begin{proposition}
\label{H^3(G^)=0}
Let $k$ be a non-archimedean local field or a global field with no real places, and let $G$ be an affine commutative $k$-group scheme of finite type. Then ${\rm{H}}^3(k, \widehat{G}) = 0$.
\end{proposition}

\begin{proof}
In the characteristic $0$ setting, we are free to replace the fppf with the {\'e}tale 
site everywhere, due to Proposition \ref{etrem}. First, if we have a short exact sequence 
\[
1 \longrightarrow G' \longrightarrow G \longrightarrow G'' \longrightarrow 1
\]
of group schemes as in the proposition, then, by Proposition \ref{hatisexact}, if the 
result holds for $G'$ and $G''$, then it also holds for $G$. Thanks to Lemma \ref{affinegroupstructurethm}, therefore, we may assume that $G$ is either $\Ga$ or an almost-torus. The $\Ga$ case follows from Proposition \ref{H^3(G_a^)=0} in positive characteristic and the much easier Proposition \ref{cohomologyofG_adualwhenkisperfect} in characteristic $0$. Thanks to Lemma \ref{almosttorus}(ii), we may therefore assume that $G$ is either finite or a torus. When $G$ is finite, then so is $\widehat{G}$, so we are done by Proposition \ref{cohomologicalvanishing}. We may therefore assume that $G = T$ is a torus. 

Let $k'/k$ be a finite separable extension such that $T$ splits over $k'$. Then we have the natural inclusion $T \hookrightarrow \R_{k'/k}(T')$, where $T'$ is the split $k'$-torus $T_{k'}$, so we have an exact sequence
\begin{equation}
\label{H^3(G^)=0pfeqn1}
1 \longrightarrow T \longrightarrow \R_{k'/k}(T') \longrightarrow S \longrightarrow 1
\end{equation}
for some $k$-torus $S$. This yields by Proposition \ref{hatisexact} an exact sequence
\[
{\rm{H}}^3(k, \widehat{\R_{k'/k}(T')}) \longrightarrow {\rm{H}}^3(k, \widehat{T}) \longrightarrow {\rm{H}}^4(k, \widehat{S})
\]
The groups in (\ref{H^3(G^)=0pfeqn1}) are all tori, so their fppf $\Gm$ dual sheaves are represented by \'etale group schemes, hence we may take their cohomology to be \'etale. In particular, ${\rm{H}}^4(k, \widehat{S})$ vanishes because local and global function fields have strict cohomological dimension $\leq 3$ (by Proposition \ref{cohomologicalvanishing}). We may therefore assume that $T = \R_{k'/k}(\Gm)$ for some finite separable extension $k'/k$.

By Proposition \ref{charactersseparableweilrestriction}, we have $\widehat{{\rm{R}}_{k'/k}(\Gm)} \simeq \R_{k'/k}(\widehat{\Gm}) = \R_{k'/k}(\Z)$. Since finite pushforward is exact between categories of {\'e}tale sheaves, we have ${\rm{H}}^3(k, \widehat{\R_{k'/k}(\Gm)}) \simeq {\rm{H}}^3(k', \Z)$. 
Replacing $k$ with $k'$, it therefore suffices to show that ${\rm{H}}^3(k, \Z) = 0$. Using the exact sequence
\[
0 \longrightarrow \Z \longrightarrow \Q \longrightarrow \Q/\Z \longrightarrow 0
\]
and the fact that the constant Galois module $\Q$ is uniquely divisible (hence its higher Galois cohomology vanishes), this is equivalent to the assertion that ${\rm{H}}^2(k, \Q/\Z) = 0$. This is proved when ${\rm{char}}(k) = 0$ in \cite[\S6.5]{serremodularforms}. The proof goes through verbatim in characteristic $p$ except that one must show that the map ${\rm{H}}^1(k, \Q_p/\Z_p) \rightarrow {\rm{H}}^2(k, \Z/p\Z)$ coming from the exact sequence
\[
0 \longrightarrow \Z/p\Z \longrightarrow \Q_p/\Z_p \xlongrightarrow{p} \Q_p/\Z_p \longrightarrow 0
\]
is surjective. This holds because the latter group vanishes (e.g., because any field of characteristic $p$
has $p$-cohomological dimension at most 1, by \cite[Ch.\,II, \S2.2, Prop.\,3]{serre}). 
\end{proof}

\section{Duality between ${\rm{H}}^2(k, G)$ and ${\rm{H}}^0(k, \widehat{G})_{\rm{pro}}$}
\label{sectiondualityG(k)pro}

The purpose of this section is to prove Theorem \ref{localdualityH^2(G)}.

\begin{lemma}
\label{H^2(G)istorsion}
Let $k$ be a field, and let $G$ be a commutative $k$-group scheme of finite type. Then ${{\rm{H}}}^2(k, G)$ is a torsion group.
\end{lemma}

\begin{proof}
If $G$ is smooth then this is immediate from the fact that ${{\rm{H}}}^2(k, G) = {{\rm{H}}}^2_{\et}(k, G)$ because higher Galois cohomology is torsion. For general $G$, there exists
by \cite[VII$_{\rm{A}}$, Prop.\, 8.3]{sga3} an infinitesimal $k$-subgroup scheme $I \subset G$ such that $H : = G/I$ is smooth. Since ${{\rm{H}}}^2(k, I)$ is clearly torsion, and ${{\rm{H}}}^2(k, H)$ is torsion because $H$ is smooth, ${{\rm{H}}}^2(k, G)$ is torsion as well.
\end{proof}

The following lemma will be of repeated use to us throughout this manuscript.

\begin{lemma}
\label{sepblepushforward}
For a finite separable extension $k'/k$, with $f: {\rm{Spec}}(k') \rightarrow {\rm{Spec}}(k)$ the corresponding map, $f_*$ is an exact functor between categories of abelian sheaves on the corresponding big fppf sites.
\end{lemma}

\begin{proof}
Such exactness may be checked after pulling to a Galois closure $K$ of $k'/k$. For an fppf sheaf $\mathscr{F}$ on ${\rm{Spec}}(k')$, we have $(f_*\mathscr{F})_K \simeq \prod_{\sigma} \mathscr{F}_{\sigma}$, where the product is over the $k$-embeddings of $k'$ into $K$.
\end{proof}

\begin{theorem}
\label{H^2(G)G^(k)dualityprop} $($Theorem $\ref{localdualityH^2(G)}$$)$ If $k$ is a local field of positive characteristic, and  $G$ is an affine commutative $k$-group scheme of finite type, then the cohomology group  ${{\rm{H}}}^2(k, G)$ is torsion, $\widehat{G}(k)$ is finitely generated, and cup product $${{\rm{H}}}^2(k, G) \times \widehat{G}(k) \rightarrow {{\rm{H}}}^2(k, \mathbf{G}_m) \xrightarrow[\sim]{\inv} \Q/\Z$$ induces a functorial continuous perfect pairing of locally compact Hausdorff abelian groups 
\[
{{\rm{H}}}^2(k, G) \times \widehat{G}(k)_{{\rm{pro}}} \rightarrow \Q/\Z,
\]
where ${\rm{H}}^2(k, G)$ and $\widehat{G}(k)$ are discrete. Equivalently, the map ${\rm{H}}^2(k, G) \rightarrow {\rm{H}}^0(k, \widehat{G})^*$ is an (algebraic) isomorphism.
\end{theorem}

\begin{proof}
The group ${\rm{H}}^2(k, G)$ is torsion by Lemma \ref{H^2(G)istorsion}. The continuity of the cup product pairing 
\[
{\rm{H}}^2(k, G) \times \widehat{G}(k) \xrightarrow{\cup} {\rm{H}}^2(k, \Gm) \underset{{\rm{inv}}}{\xrightarrow{\sim}} \Q/\Z
\]
is trivial, since, by definition, both groups ${{\rm{H}}}^2(k, G)$ and $\widehat{G}(k)$ are discrete. The functor $H \mapsto \widehat{H}(k)$ is left-exact. Thus, $\widehat{G}(k)$ is finitely generated, by reduction to the case of finite group schemes, unipotent groups, and tori, using Lemmas \ref{affinegroupstructurethm} and \ref{almosttorus}(ii). Therefore, $\Hom_{\rm{cts}}(\widehat{G}(k)_{{\rm{pro}}}, \Q/\Z) = \Hom(\widehat{G}(k), \Q/\Z)$. We claim that the perfectness of the pairing in the theorem follows from the assertion that the map ${{\rm{H}}}^2(k, G) \rightarrow \widehat{G}(k)^*$ is an isomorphism. Indeed, it then follows from the equality just shown that the map ${\rm{H}}^2(k, G) \rightarrow \widehat{G}(k)_{{\rm{pro}}}^D$ is an (algebraic) isomorphism. Since the source and target are discrete (for the target, this holds because it is the Pontryagin dual of a compact group), it then follows that this map is even a topological isomorphism. 

We already know the theorem if $G$ is finite, thanks to Tate local duality for finite group schemes (as supplemented in \cite{ces2} for group schemes of $p$-power order in characteristic $p>0$; see Remark \ref{prior}). Next we prove the theorem in the special case when $G = \R_{k'/k}(\mathbf{G}_m)$ with $k'/k$ a finite separable extension. We have a natural isomorphism ${{\rm{H}}}^2(k', \Gm) \simeq {{\rm{H}}}^2(k, \R_{k'/k}(\Gm))$, thanks to Lemma \ref{sepblepushforward}.  In addition, by Proposition \ref{charactersseparableweilrestriction}, the norm map $N_{k'/k}:  \widehat{\Gm}(k') \rightarrow \widehat{\R_{k'/k}(\Gm)}(k)$ is an isomorphism. By Proposition \ref{diagramcommuteslocal}, we have a commutative diagram
\[
\begin{tikzcd}
{{\rm{H}}}^2(k', \mathbf{G}_m) \arrow[r, phantom, "\times"] \arrow{d}{\mbox{ \rotatebox{90}{ $\sim$ }}} & \widehat{\mathbf{G}_m}(k') \arrow{d}{N_{k'/k}}[swap]{\mbox{ \rotatebox{90}{ $\sim$ }}} \arrow{r} & \Q/\Z \arrow[d, equals] \\
{{\rm{H}}}^2(k, \R_{k'/k}(\mathbf{G}_m)) \arrow[r, phantom, "\times"] & \widehat{\R_{k'/k}(\mathbf{G}_m)}(k) \arrow{r} & \Q/\Z
\end{tikzcd}
\]
Therefore, in order to prove the theorem for $G = \R_{k'/k}(\mathbf{G}_m)$ over $k$, it suffices to prove it for $G = \mathbf{G}_m$ over $k'$, for which it is trivial, since the pairing ${{\rm{H}}}^2(k', \mathbf{G}_m) \times \widehat{\mathbf{G}_m}(k') \rightarrow {{\rm{H}}}^2(k', \mathbf{G}_m)$ is the one which sends a generator for $\widehat{\mathbf{G}_m}(k')$, namely the identity character of $\mathbf{G}_m$, to the identity map ${{\rm{H}}}^2(k', \mathbf{G}_m) \rightarrow {{\rm{H}}}^2(k', \mathbf{G}_m)$.

Now suppose that $G$ is an almost-torus. By Lemma \ref{almosttorus}(iv), after harmlessly modifying $G$ we may assume that there is an isogeny $B \times \R_{k'/k}(T') \twoheadrightarrow G$ for some finite separable extension $k'/k$, some split $k'$-torus $T'$, and some finite commutative $k$-group scheme $B$. For notational convenience, let us denote $B \times \R_{k'/k}(T')$ by $X$, and let $A : = \ker(X \rightarrow G)$. 

 The exact sequence
\[
1 \longrightarrow A \longrightarrow X \longrightarrow G \longrightarrow 1
\]
yields a commutative diagram of exact sequences
\[
\begin{tikzcd}
{{\rm{H}}}^2(k, A) \isoarrow{d} \arrow{r} & {{\rm{H}}}^2(k, X) \isoarrow{d} \arrow{r} & {{\rm{H}}}^2(k, G) \arrow{d} \arrow{r} & 0 \\
\widehat{A}(k)^* \arrow{r} & \widehat{X}(k)^* \arrow{r} & \widehat{G}(k)^* \arrow{r} & 0
\end{tikzcd}
\]
The first two vertical arrows are isomorphisms because we already know the theorem for $X$
and the $k$-finite $A$. The bottom row is exact because $H \mapsto \widehat{H}(k)$ is a left-exact functor on the category of fppf abelian sheaves and because $\Q/\Z$ is an injective abelian group, and the $0$ in the top row is because ${{\rm{H}}}^3(k, A) = 0$ by Proposition \ref{cohomologicalvanishing}. The diagram shows that ${{\rm{H}}}^2(k, G) \rightarrow \widehat{G}(k)^*$ is an isomorphism, hence proves the theorem for $G$.

Finally, suppose $G$ is an arbitrary affine commutative $k$-group scheme of finite type. By Lemma \ref{affinegroupstructurethm}, there is an exact sequence
\[
1 \longrightarrow H \longrightarrow G \longrightarrow U \longrightarrow 1
\]
with $H$ an almost-torus and $U$ split unipotent. This yields a commutative diagram of exact sequences
\[
\begin{tikzcd}
0 \arrow{r} & {{\rm{H}}}^2(k, H) \isoarrow{d} \arrow{r} & {{\rm{H}}}^2(k, G) \arrow{d} \arrow{r} & 0 \\
0 \arrow{r} & \widehat{H}(k)^* \arrow{r} & \widehat{G}(k)^* \arrow{r} & 0
\end{tikzcd}
\]
The $0$'s in the top row are because ${{\rm{H}}}^i(k, U) = 0$ for $i=1,2$, by Proposition \ref{unipotentcohomology}. The $0$'s in the bottom row are due to Proposition \ref{hatisexact} and the vanishing of $\widehat{U}(k)$ (Proposition \ref{cohomologyofG_adualgeneralk} and left--exactness of $H \mapsto \widehat{H}(k)$) and ${{\rm{H}}}^1(k, \widehat{U})$ (Proposition \ref{unipotentcohomology}). 
The first vertical map is an isomorphism by the already-treated case of almost-tori, so the other vertical
map is also an isomorphism.
\end{proof}

\section{Topology on local cohomology}
\label{sectiontopologyoncohomology}

In this section we describe a topology on the groups ${\rm{H}}^1(R, G)$ for $R$ either a local field or its ring of integers, and $G$ a locally finite type commutative $R$-group scheme, and we prove various desirable properties for this topology.
In \cite{ces1}, \v{C}esnavi\v{c}ius defines a topology on ${{\rm{H}}}^1(R, H)$ as follows. A subset $U \subset {{\rm{H}}}^1(R, H)$ is {\em open} if for every locally finite type $R$-scheme $X$, and every $H$-torsor sheaf $\mathscr{X} \rightarrow X$ for the fppf topology, the set $\{ x \in X(R) \mid \mathscr{X}_x \in U \} \subset X(R)$ is open, where $X(R)$ is endowed with its usual topology inherited from that on $R$.

When $R = k$ is a local field, this topology on ${{\rm{H}}}^1(k, G)$ for commutative $k$-group schemes $G$ 
locally of finite type applies in particular to the fppf $\Gm$-dual of an almost-torus by Proposition \ref{hatrepresentable}. We shall later require a topology on ${{\rm{H}}}^1(k, \widehat{G})$ for arbitrary affine commutative $k$-group schemes $G$ of finite type (not just for $G$ an almost-torus), a topic we address in \S\ref{h1dualgen}.

One may endow ${\rm{H}}^1(k, G)$ with a topology in a different manner, by means of the \v{C}ech topology, defined as follows. For any finite extension $L/k$, both $L$ and $L \otimes_k L$ come equipped with natural 
topologies as $k$-algebras and therefore so do $G(L)$ and $G(L \otimes_k L)$ as topological groups. We endow the \v{C}ech cohomology group $\check{{\rm{H}}}^1(L/k, G)$ with the subquotient topology, and the \v{C}ech cohomology group 
$\check{\rm{H}}^1(k, G) = \varinjlim_L \check{{\rm{H}}}^1(L/k, G)$ with the direct limit topology.

Finally, for any locally finite type $k$-group scheme $G$, we endow $G(k)$ with its usual topology (inherited from that on $k$), and ${{\rm{H}}}^2(k, G)$ with the discrete topology. We summarize some of the basic facts that we require about the topology on ${\rm{H}}^1(k, G)$ below.

\begin{proposition}
\label{topcohombasics}
Let $k$ be a local field, and let $G, H$ be commutative $k$-group schemes locally of finite type.
\begin{itemize}
\item[(i)] The group ${\rm{H}}^1(k, G)$ is a locally compact topological group. If the \'etale component group of $G$ has finitely-generated group of rational points, then ${\rm{H}}^1(k, G)$ is also Hausdorff.
\item[(ii)] Given a $k$-group homomorphism $G \rightarrow H$, the induced map ${\rm{H}}^1(k, G) \rightarrow {\rm{H}}^1(k, H)$ is continuous.
\item[(iii)] If $G$ is smooth, then ${\rm{H}}^1(k, G)$ is discrete. In particular, if ${\rm{char}}(k) = 0$ then ${\rm{H}}^1(k, G)$ is discrete.
\item[(iv)] The topology on ${\rm{H}}^1(k, G)$ agrees with the \v{C}ech topology via the natural isomorphism $\check{{\rm{H}}}^1(k, G) \simeq {\rm{H}}^1(k, G)$.
\end{itemize}
Now suppose that one has a short exact sequence of locally finite type commutative $k$-group schemes
\[
1 \longrightarrow G' \longrightarrow G \longrightarrow G'' \longrightarrow 1.
\]
\begin{itemize}
\item[(v)] The maps in the associated long exact cohomology sequence up to the ${{\rm{H}}}^2$ level are continuous. That is, all of the maps in the long exact sequence beginning with ${\rm{H}}^0(k, G')$ and ending with ${\rm{H}}^2(k, G'')$ are continuous.
\item[(vi)] The map ${\rm{H}}^1(k, G) \rightarrow {\rm{H}}^1(k, G'')$ is open.
\item[(vii)] If $G$ is smooth, then the map $G''(k) \rightarrow {\rm{H}}^1(k, G')$ is open.
\end{itemize}
\item[(viii)] If $G''$ is smooth, then the map ${\rm{H}}^1(k, G') \rightarrow {\rm{H}}^1(k, G)$ is open.
\end{proposition}

\begin{proof}
(i) The group is a locally compact topological group by \cite[Props.\,3.6(c), 3.7(c)]{ces1}. To see that it is Hausdorff when the component group is finitely-generated, consider the exact sequence
\[
1 \longrightarrow G^0 \longrightarrow G \longrightarrow E \longrightarrow 1
\]
with $G^0$ the identity component and $E$ \'etale. All of the induced maps on ${\rm{H}}^1$ are continuous by \cite[Prop.\,4.2]{ces1}. The group ${\rm{H}}^1(k, E)$ is Hausdorff (discrete even) by \cite[Prop.\,3.5(a)]{ces1}, hence it suffices to show that the kernel of the map ${\rm{H}}^1(k, G) \rightarrow {\rm{H}}^1(k, E)$ is Hausdorff. The continuous injection 
\[
{\rm{H}}^1(k, G^0)/{\rm{im}}(E(k)) \hookrightarrow {\rm{H}}^1(k, G)
\]
is open by \cite[4.3(c)]{ces1}, hence a homeomorphism onto its image. In order to complete the proof, therefore, we only need to show that ${\rm{im}}(E(k)) \subset {\rm{H}}^1(k, G^0)$ is closed, for which it suffices to show that it is finite, since ${\rm{H}}^1(k, G^0)$ is Hausdorff \cite[3.9]{ces1}. Since $E(k)$ is finitely-generated by assumption, it suffices to verify that ${\rm{H}}^1(k, G^0)$ is torsion. When $G^0$ is smooth, this follows from the fact that higher Galois cohomology is torsion. In general, it follows from Lemma \ref{finitequotient=smoothandconnected}. \\
(ii) \cite[Cor.\,2.7(i)]{ces1}. \\
(iii) \cite[Prop.\,3.5(a)]{ces1}. \\
(iv) \cite[Thm. 5.11]{ces1}. \\
(v) \cite[Prop.\,4.2]{ces1}. \\
(vi) \cite[4.3(d)]{ces1}. \\
(vii) \cite[4.3(b)]{ces1}. \\
(viii) \cite[4.3(c)]{ces1}.
\end{proof}

Our next goal is to show that the groups ${\rm{H}}^1(k, G)$ and ${\rm{H}}^1(k, \widehat{G})$ are second countable and locally profinite (see Proposition \ref{secondcountable}). In order to do this, we will require some preparatory lemmas.

\begin{lemma}
\label{H^1=0Weilrestrictionsplittori}
Let $k'/k$ be a finite separable extension of $($arbitrary$)$ fields, and let $T'$ be a split $k'$-torus. Then ${\rm{H}}^1(k, \R_{k'/k}(T')) = {\rm{H}}^1(k, \widehat{\R_{k'/k}(T')}) = 0$.
\end{lemma}

\begin{proof}
We may assume that $T' = \Gm$. By Lemma \ref{sepblepushforward}, ${\rm{H}}^1(k, \R_{k'/k}(\Gm)) \simeq {\rm{H}}^1(k', \Gm) = 0$. For the other cohomology group, Proposition \ref{charactersseparableweilrestriction} implies that $\widehat{\R_{k'/k}(\Gm)} \simeq \R_{k'/k}(\widehat{\Gm}) \simeq \R_{k'/k}(\Z)$. We may therefore take the cohomology to be \'etale, hence we have ${\rm{H}}^1(k, \R_{k'/k}(\Z)) \simeq {\rm{H}}^1(k', \Z) = 0$.
\end{proof}

\begin{lemma}
\label{H^1(finitefield)=finite}
Let $\kappa$ be a finite field and $G$ a $\kappa$-group scheme of finite type. Then the group ${\rm{H}}^1(\kappa, G)$ is finite.
\end{lemma}

\begin{proof}
Because $\kappa$ is perfect, we may take our cohomology to be \'etale by Proposition \ref{etrem}. First we prove the finiteness of ${\rm{H}}^1(\kappa, G)$. We begin with the exact sequence
\[
1 \longrightarrow G^0 \longrightarrow G \longrightarrow E \longrightarrow 1,
\]
in which $E$ is finite \'etale. By using the method of twisting by Galois cocycles (see \cite[Ch.\,I, \S5]{serre}), we see that the finiteness of ${\rm{H}}^1(\kappa, G)$ would follow from that of all $\kappa$-forms of $G^0$ and of $E$. All forms $H$ of $G^0$ are connected, and replacing $H$ with its reduced subscheme has no effect on the associated \'etale sheaf, so we may assume that $H$ is smooth and connected. Then the desired vanishing follows from Lang's Theorem. We may therefore assume that $G$ is \'etale.

If $G$ is \'etale, then using twisting once again, we may filter $G$ and thereby assume that it is commutative. If $G$ is a form of $\Z/N\Z$, then it becomes the trivial form over the unique (inside a fixed algebraic closure of $\kappa$) extension $\kappa'/\kappa$ of order $\#{\rm{Aut}}(\Z/N\Z)$. After passing to $\kappa'$, every cohomology class becomes trivial upon passage to the further extension of $\kappa''/\kappa'$ of degree $N$. Thus, very cohomology class in ${\rm{H}}^1(\kappa, G)$ dies after passing to $\kappa''$. It follows that ${\rm{H}}^1(\kappa, G)$ is finite.
\end{proof}

\begin{proposition}
\label{secondcountable}
Let $k$ be a local field, and let $G$ be an affine commutative $k$-group scheme of finite type. Then the locally compact
Hausdorff group ${\rm{H}}^1(k, G)$ is second-countable
and locally profinite, hence in particular totally disconnected. The same holds for ${\rm{H}}^1(k, \widehat{G})$ if $G$ is an almost-torus. 
\end{proposition}

Here, by ``locally profinite'' we mean that there is a profinite open subgroup (and hence a base of such around the identity). 
In the final assertion of Proposition \ref{secondcountable} we must limit ourselves (for now) to the case of almost-tori since beyond that case
$\widehat{G}$ is not representable (Proposition \ref{hatrepresentable}),
and we therefore have not yet defined a topology on ${\rm{H}}^1(k, \widehat{G})$.

\begin{proof}
We first show that ${\rm{H}}^1(k, G)$ contains a second-countable profinite open subgroup. This is trivial if ${\rm{char}}(k) = 0$ and in particular when $k = \mathbf{R}$, due to Proposition \ref{topcohombasics}(iii), so we may assume for the purposes of this assertion that $k$ is non-archimedean.

Choose a separated, flat, finite type $\mathcal{O}_k$-group scheme $\mathscr{G}$ 
with $k$-fiber $G$ (e.g., schematic closure in ${\rm{GL}}_{n,\,\mathcal{O}_k}$ relative
to some $k$-subgroup inclusion $G \hookrightarrow {\rm{GL}}_{n,\,k}$). 
By \cite[Cor.\,B.7]{ces1} and the first assertion in Lemma \ref{H^1(finitefield)=finite}, there is an affine $\mathcal{O}_k$-scheme $X$ and a smooth morphism $X \rightarrow B\mathscr{G}$ such that the induced map $X(\mathcal{O}_k) \rightarrow 
(B\mathscr{G})(\mathcal{O}_k)/{\rm{isom}}$ is surjective. The map $X(\calO_k) \rightarrow 
(B\mathscr{G})(\mathcal{O}_k)/{\rm{isom}}$ is continuous by definition, and it is open by \cite[Prop.\,2.9(a)]{ces1}. 

Since $X(\mathcal{O}_k)$ is second-countable and  has a base of quasi-compact open neighborhoods of every point, 
it follows that $(B\mathscr{G})(\mathcal{O}_k)/{\rm{isom}}$ has the same properties. By 
\cite[Prop.\,2.9(e)]{ces1}, the map $(B\mathscr{G})(\mathcal{O}_k) \rightarrow (BG)(k)/{\rm{isom}} = {\rm{H}}^1(k,G)$ is open, hence its image is a second-countable open subgroup of the locally compact Hausdorff abelian
group ${\rm{H}}^1(k,G)$ in which every point has a base of compact open neighborhoods. 
By the Hausdorff condition, such compact open neighborhoods are also closed, and hence ${\rm{H}}^1(k, G)$ is totally
disconnected.  Thus, any compact open subgroup must be profinite (as profinite groups are exactly the topological groups that are compact and totally disconnected \cite[Ch.\,I, \S1.1, Prop 0]{serre}). This proves that ${\rm{H}}^1(k, G)$ contains a second-countable profinite open subgroup. 

Next we show that ${\rm{H}}^1(k, G)$ is itself second-countable. We claim that there is an inclusion $G \hookrightarrow G'$ with $G'$ a smooth connected commutative affine $k$-group scheme such that ${\rm{H}}^1(k, G') = 0$. By Lemma \ref{finitequotient=smoothandconnected}, there is an exact sequence
\begin{equation}
\label{propsecondcountablepfeqn1}
1 \longrightarrow F \longrightarrow G \longrightarrow H \longrightarrow 1
\end{equation}
with $F$ a finite $k$-group and $H$ smooth and connected. By \cite[II, 3.2.5]{messing}, there is an inclusion $F \hookrightarrow H'$ with $H'$ a smooth connected commutative affine $k$-group scheme. Pushing out the sequence (\ref{propsecondcountablepfeqn1}) along the inclusion $F \hookrightarrow H'$, we obtain an inclusion $G \hookrightarrow H''$ with $H''$ smooth, connected, commutative, and affine. In order to prove the claim that $G$ admits an inclusion into a smooth, connected, commutative, affine $k$-group with vanishing ${\rm{H}}^1$, therefore, we may assume that $G$ is smooth and connected. But then there is a finite extension $k'/k$ such that $G_{k'} \simeq \Gm^n \times V$ for some $n \geq 0$ and some split unipotent $k'$-group $V$. The natural inclusion $G \hookrightarrow \R_{k'/k}(G_{k'})$ then does the job.

By the claim just proved, we have an exact sequence
\[
1 \longrightarrow G \longrightarrow H \xlongrightarrow{\pi} H' \longrightarrow 1
\]
with $H$ and $H'$ smooth, connected, commutative, and affine $k$-groups of finite type such that ${\rm{H}}^1(k, H) = 0$. Then we have a topological isomorphism ${\rm{H}}^1(k, G) \simeq H'(k)/\pi(H(k))$, since the map $H'(k) \rightarrow {\rm{H}}^1(k, G)$ is open by Proposition \ref{topcohombasics}(vii). The latter group is second-countable, hence so is the former. This completes the proof of the proposition for ${\rm{H}}^1(k, G)$. 

Now we prove the assertion for ${\rm{H}}^1(k, \widehat{G})$. The map ${\rm{H}}^1(k, \widehat{G}^0) \rightarrow {\rm{H}}^1(k, \widehat{G})$ is open by Lemma \ref{topcohombasics}(viii). Since $\widehat{G}^0$ is finite by Proposition \ref{hatrepresentable}, the conclusion of the proposition is true for ${\rm{H}}^1(k, \widehat{G}^0)$. It follows that ${\rm{H}}^1(k, \widehat{G})$ also contains a second-countable profinite open subgroup. In order to show that it is itself second-countable, it suffices to show that ${\rm{H}}^1(k, E)$ is countable, where $E$ is the \'etale component group of $\widehat{G}$. 

By Proposition \ref{hatrepresentable}, $E(k_s)$ is a finitely-generated abelian group, so it suffices too show that for any continuous Galois module $M$ over $k$ which is finitely-generated as an abelian group, ${\rm{H}}^1(k, M)$ is countable. We may assume that $M$ is either torsion or torsion-free. In the torsion case, we are done by the already-treated case of ${\rm{H}}^1$ of finite \'etale group schemes. So assume that $M$ is torsion-free. Passing to a finite Galois extension $k'/k$, the action on $M$ becomes trivial, hence ${\rm{H}}^1(k', M) = 0$. It follows that ${\rm{H}}^1(k, M) = {\rm{H}}^1(k'/k, M)$, and this latter group is countable, since even the set of all maps ${\rm{Gal}}(k'/k) \rightarrow M$ is countable. This completes the proof.
\end{proof}

\section{Duality between ${\rm{H}}^1(k,G)$ and ${\rm{H}}^1(k, \widehat{G})$ for almost-tori}\label{localdualaltori}

In this section we prove Theorem \ref{H^1localduality} when $G$ is an almost-torus (Definition \ref{almosttorusdef}). For such $G$, by Proposition \ref{hatrepresentable}, the sheaf $\widehat{G}$ is represented by a commutative, locally finite group scheme, hence the procedure in \S \ref{sectiontopologyoncohomology} defines a topology not only on the group ${\rm{H}}^1(k, G)$, but also on ${\rm{H}}^1(k, \widehat{G})$. The non-representability of $\widehat{G}$ beyond the setting of almost-tori (Proposition \ref{hatrepresentable}) is the reason for our restriction to almost-tori in this section.

Next we prove the finite exponent aspect of Theorem \ref{H^1localduality}.

\begin{lemma}
\label{H^1finiteexponent}
Let $k$ be a field, and let $G$ be an affine commutative $k$-group scheme of finite type. Then ${\rm{H}}^1(k, G)$ and ${\rm{H}}^1(k, \widehat{G})$ have finite exponent.
\end{lemma}

\begin{proof}
First suppose that $G$ is an almost-torus. By Lemma \ref{almosttorus}(iv), we may assume that there is an exact sequence
\[
1 \longrightarrow B \longrightarrow A \times \R_{k'/k}(T') \longrightarrow G \longrightarrow 1
\]
for some finite commutative $k$-group schemes $A, B$, some finite separable extension $k'/k$, and some split $k'$-torus $T'$. Since the cohomology groups of $A, B$, and their duals are clearly of finite exponent, and both ${\rm{H}}^1(k, \R_{k'/k}(T'))$
and ${\rm{H}}^1(k, \widehat{\R_{k'/k}(T')})$ both vanish by Lemma \ref{H^1=0Weilrestrictionsplittori}, we see (using Proposition \ref{hatisexact}) that ${\rm{H}}^1(k, G)$ and ${\rm{H}}^1(k, \widehat{G})$ are of finite exponent. 

For general $G$, by Lemma \ref{affinegroupstructurethm} there is an exact sequence
\[
1 \longrightarrow H \longrightarrow G \longrightarrow U \longrightarrow 1
\]
with $H$ an almost-torus and $U$ split unipotent. Now ${\rm{H}}^1(k, U)$ and ${\rm{H}}^1(k, \widehat{U})$
vanish, the latter due to Proposition \ref{cohomologyofG_adualgeneralk}. Since we have already shown that ${\rm{H}}^1(k, H)$ and ${\rm{H}}^1(k, \widehat{H})$ are of finite exponent, the same holds for ${\rm{H}}^1(k, G)$ and ${\rm{H}}^1(k, \widehat{G})$.
\end{proof}

\begin{lemma}
\label{cupproductctsH^1almosttori}
If $G$ is an almost-torus over a non-archimedean local field $k$, then the cup product pairing
\[
{\rm{H}}^1(k, G) \times {\rm{H}}^1(k, \widehat{G}) \xrightarrow{\cup} {\rm{H}}^2(k, \Gm) \xrightarrow[\sim]{\inv} \Q/\Z
\]
is continuous.
\end{lemma}

\begin{proof}
The desired continuity is equivalent to the following pair of assertions: 
\begin{itemize}
\item[(i)] There are neighborhoods $U_G \subset {\rm{H}}^1(k, G)$, $U_{\widehat{G}} \subset {\rm{H}}^1(k, \widehat{G})$ of $0$ such that $\langle U_G, U_{\widehat{G}} \rangle$ $= \{0\}$. 
\item[(ii)] For any $\alpha \in {\rm{H}}^1(k, G)$, there exists a neighborhood $U \subset {\rm{H}}^1(k,\widehat{G})$ of $0$ such that $\langle \alpha, U \rangle = \{0\}$, and the same holds if we switch the roles of $G$ and $\widehat{G}$.
\end{itemize}

In order to prove (i) and (ii), we use the fact that the topologies on ${\rm{H}}^1(k, G)$ and ${\rm{H}}^1(k, \widehat{G})$ agree with the \v{C}ech topologies, by Proposition \ref{topcohombasics}(iv). Both assertions therefore follow easily from the continuity of the map $G(L \otimes_k L) \times \widehat{G}(L \otimes_k L) \xrightarrow{\cup} (L \otimes_k L \otimes_k L)^{\times}$
(using the explicit formula defining this cup product) combined with the openness of the subset $(\calO_L \otimes_{\calO_k} \calO_L \otimes_{\calO_k} \calO_L)^{\times} \subset (L \otimes_k L \otimes_k L)^{\times}$ and the fact that $\Br(\calO_k) = 0$.
\end{proof}

\begin{lemma}
\label{isom=homeo}
Let $H_1, H_2$ be locally compact topological groups, with $H_1$ second-countable and $H_2$ Hausdorff. 
Any continuous surjective homomorphism $f:  H_1 \rightarrow 
H_2$ is open. In particular, if $f$ is also bijective then it is a homeomorphism. 
\end{lemma}

\begin{proof}
This is \cite[Ch.\,IX, \S5, Cor.\, to Prop.\,6]{bourbakitopology}.
\end{proof}

\begin{lemma}
\label{H^1localdualityalmosttori} $($Theorem $\ref{H^1localduality}$ for almost-tori$)$ Let $k$ be a local field of positive characteristic, and let $G$ be an almost-torus over $k$. The 
fppf cohomology groups ${{\rm{H}}}^1(k, G)$ and ${{\rm{H}}}^1(k, \widehat{G})$ are of finite exponent, and cup product 
\[
{{\rm{H}}}^1(k, G) \times {{\rm{H}}}^1(k, \widehat{G}) \rightarrow {{\rm{H}}}^2(k, \mathbf{G}_m) \xrightarrow[\sim]{\inv} \Q/\Z
\]
 is a functorial continuous perfect pairing of locally compact Hausdorff abelian groups.
\end{lemma}

\begin{proof}
Let $G$ be an almost-torus over a non-archimedean local field $k$, so that Proposition \ref{hatrepresentable} implies that the fppf sheaf $\widehat{G}$ on $\Spec(k)$ is represented by a commutative locally finite type $k$-group scheme. 
Proposition \ref{topcohombasics}(i) and Lemma \ref{H^1finiteexponent} therefore imply that ${\rm{H}}^1(k, G)$ and ${\rm{H}}^1(k, \widehat{G})$ are locally compact Hausdorff groups of finite exponent.

Lemma \ref{cupproductctsH^1almosttori} implies that the cup product pairing yields a continuous map ${\rm{H}}^1(k, G) \rightarrow {\rm{H}}^1(k, \widehat{G})^D$, and we want to show that this map is a topological isomorphism; see Remark \ref{perfectpairing}. By Proposition \ref{secondcountable} and Lemma \ref{isom=homeo}, it suffices to prove bijectivity (i.e., this map is an isomorphism
of groups, disregarding their topologies).

As usual, we may assume by Lemma \ref{almosttorus}(iv) that there is an exact sequence
\[
1 \longrightarrow A \stackrel{j}{\longrightarrow} B \times \R_{k'/k}(T') \xlongrightarrow{\pi} G \longrightarrow 1
\]
for some finite commutative $k$-group schemes $A, B$, some finite separable extension $k'/k$, and some split $k'$-torus $T'$. For notational simplicity, let us denote $B \times \R_{k'/k}(T')$ by $X$. Using Proposition \ref{hatisexact}
and the $\delta$-functoriality of cup products (for fppf abelian sheaves), we claim that we obtain a commutative diagram 
\begin{equation}
\label{H^1localdualityalmosttoridiagram2}
\begin{tikzcd}
{\rm{H}}^1(k, A) \isoarrow{d} \arrow{r} & {\rm{H}}^1(k, X) \isoarrow{d} \arrow{r} & {\rm{H}}^1(k, G) \arrow{d} \arrow{r} & {\rm{H}}^2(k, A) \isoarrow{d} \arrow{r} & {\rm{H}}^2(k, X) \arrow[d, hookrightarrow] \\
{\rm{H}}^1(k, \widehat{A})^D \arrow{r} & {\rm{H}}^1(k, \widehat{X})^D \arrow{r} & {\rm{H}}^1(k, \widehat{G})^D \arrow{r} & {\rm{H}}^0(k, \widehat{A})^D \arrow{r} & \Hom(\widehat{X}(k), \Q/\Z)
\end{tikzcd}
\end{equation}
where the maps in the bottom row are obtained by dualizing the long exact cohomology sequence associated to the exact sequence
\[
1 \longrightarrow \widehat{G} \longrightarrow \widehat{X} \longrightarrow \widehat{A} \longrightarrow 1
\]
using Proposition \ref{topcohombasics}(v). 
Note that ${\rm{H}}^0(k, \widehat{A})^* = {\rm{H}}^0(k, \widehat{A})^D$ because ${\rm{H}}^0(k, \widehat{A})$ is discrete, where recall that for an abelian group $H$, $H^* : = {\rm{Hom}}(H, \Q/\Z)$, the group of algebraic (not necessarily continuous) group homomorphisms $H \rightarrow \Q/\Z$. The commutativity of the diagram is clear, except for the square with top left corner ${\rm{H}}^1(k, G)$, where the desired commutativity comes from the standard compatibility between cup products and connecting maps; see \cite[Tag 07MC]{stacks}. (The statement given there is only for \v{C}ech cohomology, but see the comment at the beginning of Tag 01FP, where it is explained that one may use hypercoverings and thereby obtain the result for derived functor cohomology.)
We will now show that the bottom row of (\ref{H^1localdualityalmosttoridiagram2}) is exact (as the top row certainly is) and that the indicated vertical maps are isomorphisms or inclusions.  Granting these properties of the diagram, a
simple diagram chase (or the 5-lemma) then shows that the map ${\rm{H}}^1(k, G) \rightarrow {\rm{H}}^1(k, \widehat{G})^D$ is an isomorphism of groups, which is all we need.

The first and fourth vertical arrows in (\ref{H^1localdualityalmosttoridiagram2}) are isomorphisms by local duality for finite commutative group schemes. The second vertical arrow is an isomorphism by local duality for finite commutative group schemes and Lemma \ref{H^1=0Weilrestrictionsplittori}. The fifth vertical arrow is an inclusion by Theorem \ref{H^2(G)G^(k)dualityprop}.

To see that the bottom row of (\ref{H^1localdualityalmosttoridiagram2}) is exact, we first note that exactness at $\widehat{A}(k)^D$
is unaffected by replacing $\Hom(\widehat{X}(k), \Q/\Z)$ with $\widehat{X}(k)^D$ (since $\Hom(\widehat{X}(k), \Q/\Z)$ is a subgroup of $\widehat{X}(k)^D$, and $\widehat{A}(k)^D$ lands inside this subgroup because $\widehat{A}(k)$ is of finite exponent). Further, the first map in the bottom row is the Pontryagin dual of the first map in the exact sequence of {\em continuous} maps
\[
{\rm{H}}^1(k, \widehat{X}) \longrightarrow {\rm{H}}^1(k, \widehat{A}) \longrightarrow {\rm{H}}^2(k, \widehat{G})
\]
by Proposition \ref{topcohombasics}(v). Exactness of the bottom row of (\ref{H^1localdualityalmosttoridiagram2}) is reduced by Proposition \ref{secondcountable} to showing that if
$$M' \stackrel{f}{\rightarrow} M \stackrel{g}{\rightarrow} M'' \stackrel{h}{\rightarrow} M'''$$
is an algebraically exact sequence of continuous maps between Hausdorff
topological abelian groups, with all except for perhaps $M'''$ being locally compact and second-countable, then the diagram of Pontryagin duals 
$${M''}^D \rightarrow M^D \rightarrow {M'}^D$$
is algebraically exact. For this, it suffices to check that the map $M/f(M') \hookrightarrow M''$ is a homeomorphism onto a closed subgroup, since the Pontryagin dual of an embedding onto a closed subgroup is surjective \cite[Thm.\,2.1.2]{rudin}. By Lemma \ref{isom=homeo}, it suffices to check that the image $g(M) \subset M''$ is closed, and this holds because ${\rm{im}}(g) = \ker(h)$, which is closed because $M'''$ is Hausdorff. This completes the proof of the lemma.
\end{proof}

Before we can finish the proof of Theorem \ref{H^1localduality} in general, to push through some exact sequence arguments we need to prove Theorem \ref{localdualityH^2(G^)}, so we now turn to that task.

\section{Duality between ${\rm{H}}^2(k, \widehat{G})$ and $G(k)_{\rm{pro}}$}\label{sectionproofofH^2(G^)localduality}

In this section we prove Theorem \ref{localdualityH^2(G^)}. We begin with its first assertion, which amounts to the following lemma.

\begin{lemma}
\label{H^2(G^)istorsion}
Let $k$ be a field, $G$ an affine commutative $k$-group scheme of finite type. Then ${{\rm{H}}}^2(k, \widehat{G})$ is a torsion group.
\end{lemma}

\begin{proof}
Suppose first that $G$ is an almost-torus. By Lemma \ref{finitequotient=smoothandconnected}, there is a finite $k$-subgroup scheme $A \subset G$ such that $T : = G/A$ is a torus. Proposition \ref{hatisexact} yields an exact sequence
\[
{{\rm{H}}}^2(k, \widehat{T}) \longrightarrow {{\rm{H}}}^2(k, \widehat{G}) \longrightarrow {{\rm{H}}}^2(k, \widehat{A})
\]
Since $\widehat{T}$ is representable by a smooth group scheme, ${{\rm{H}}}^2(k, \widehat{T}) = {{\rm{H}}}^2_{\et}(k, \widehat{T})$ is a torsion group since higher Galois cohomology is torsion. Further, ${{\rm{H}}}^2(k, \widehat{A})$ is clearly torsion, so ${{\rm{H}}}^2(k, \widehat{G})$ is torsion as well.

Next suppose that $G = U$ is split connected unipotent. If ${\rm{char}}(k) = 0$ then ${{\rm{H}}}^2(k, \widehat{U}) = {{\rm{H}}}^2_{\et}(k, \widehat{U})$ by Proposition \ref{etrem}, and this vanishes since $\widehat{U}$ is $0$ as an {\'e}tale sheaf. 
If ${\rm{char}}(k) = p > 0$ then $U$ has finite exponent, so $\widehat{U}$ does as well, and hence the lemma is clear. 

Finally, in the general case, by Lemma \ref{affinegroupstructurethm} there is an exact sequence
\[
1 \longrightarrow H \longrightarrow G \longrightarrow U \longrightarrow 1
\]
with $H$ an almost-torus and $U$ split unipotent. We therefore have an exact sequence
\[
{{\rm{H}}}^2(k, \widehat{U}) \longrightarrow {{\rm{H}}}^2(k, \widehat{G}) \longrightarrow {{\rm{H}}}^2(k, \widehat{H})
\]
Since ${{\rm{H}}}^2(k, \widehat{U})$ is torsion and ${{\rm{H}}}^2(k, \widehat{H})$ is torsion, so is ${{\rm{H}}}^2(k, \widehat{G})$.
\end{proof}

Next we prove that the pairing described in Theorem \ref{localdualityH^2(G^)} is continuous. This is the content of the following lemma.

\begin{lemma}
\label{G(k)H^2(G^)continuous}
Let $k$ be a non-archimedean local field, $G$ an affine commutative $k$-group scheme of finite type. The pairing
\[
{{\rm{H}}}^2(k, \widehat{G}) \times {\rm{H}}^0(k, G) \xrightarrow{\cup} {\rm{H}}^2(k, \Gm) \underset{{\rm{inv}}}{\xrightarrow{\sim}} \Q/\Z
\]
is continuous, where the isomorphism on the right is given by taking invariants, ${\rm{H}}^2(k, \widehat{G})$ is endowed with the discrete topology, and ${\rm{H}}^0(k, G)$ is given its usual topology arising from that on $k$. This pairing extends uniquely to a continuous bilinear pairing
\begin{equation}
\label{pairingextendsG(k)_pro}
{\rm{H}}^2(k, \widehat{G}) \times {\rm{H}}^0(k, G)_{\pro} \rightarrow \Q/\Z.
\end{equation}
\end{lemma}

\begin{proof}
We first prove that the first pairing is continuous. It suffices to show that if $\alpha \in {{\rm{H}}}^2(k, \widehat{G})$, then there is an open set $U \subset G(k)$ containing $0 \in G(k)$ such that for all $g \in U$, we have $\alpha \cup g = g(\alpha) = 0$. By Proposition \ref{cech=derivedG^} and Lemma \ref{mayusekbar/k}, $\alpha$ is represented by a cocycle $\check{\alpha} \in \widehat{G}(L \otimes_k L \otimes_k L)$ for some finite extension $L/k$. The natural pairing $G(k) \times \widehat{G}(L \otimes_k L \otimes_k L) \rightarrow (L \otimes_k L \otimes_k L)^{\times}$ is continuous. Since $(\calO_L \otimes_{\calO_k} \calO_L \otimes_{\calO_k} \calO_L)^{\times} \subset (L \otimes_k L \otimes_k L)^{\times}$ is open, it follows that there is a neighborhood $U$ of $0 \in G(k)$ such that $g(\check{\alpha}) \in \check{Z}^2(\calO_L/\calO_k, \mathbf{G}_m)$ (the set of 2-cocycles for the cover $\Spec(\calO_L) \rightarrow \Spec(\calO_k)$ valued in $\Gm$) for all $g \in U$. Therefore, $g(\alpha) \in \Br(\calO_k) = 0$ for all $g \in U$,  so this $U$ does the job.

Next we check that the pairing extends to a continuous one on ${\rm{H}}^0(k, G)_{\pro}$. Such an extension, if it exists, is clearly unique. To see that it exists, we note that the first continuous pairing yields a continuous map ${\rm{H}}^0(k, G) \rightarrow {\rm{H}}^2(k, \widehat{G})^D$, so it suffices to show that this latter group is profinite. Since it is the Pontryagin dual of a discrete group, it is compact. To see that it is totally disconected (and therefore profinite \cite[Ch.\,I, \S1.1, Prop 0]{serre}), we note that ${\rm{H}}^2(k, \widehat{G})^D = \Hom_{\rm{cts}}({\rm{H}}^2(k, \widehat{G}), \Q/\Z)$ by Lemma \ref{H^2(G^)istorsion}, and this latter group is totally disconnected.
\end{proof}

\begin{lemma}
\label{H^2(G^)dualityspecialtori}
Let $k'/k$ be a finite separable extension of non-archimedean local fields. Then Theorem $\ref{localdualityH^2(G^)}$ holds for $G = \R_{k'/k}(\mathbf{G}_m)$.
\end{lemma}

\begin{proof}
We have by Lemma \ref{G(k)H^2(G^)continuous} and Proposition \ref{diagramcommuteslocal}
a commutative diagram of continuous pairings 
\[
\begin{tikzcd}
{{\rm{H}}}^2(k', \widehat{\mathbf{G}_m}) \isoarrow{d} \arrow[r, phantom, "\times"] & \mathbf{G}_m(k')\arrow{r} \isoarrow{d} & \Q/\Z \arrow[d, equals] \\
{{\rm{H}}}^2(k, \widehat{\R_{k'/k}(\mathbf{G}_m)}) \arrow[r, phantom, "\times"] & \R_{k'/k}(\mathbf{G}_m)(k) \arrow{r} & \Q/\Z \\
\end{tikzcd}
\]
To prove Lemma \ref{H^2(G^)dualityspecialtori}, it thus suffices to treat the case $G = \mathbf{G}_m$. We have ${{\rm{H}}}^2(k, \widehat{\mathbf{G}_m}) = {{\rm{H}}}^2(k, \Z)$. Since $\Z$ is smooth, ${{\rm{H}}}^2(k, \Z) = {{\rm{H}}}^2_{\et}(k, \Z)$. Therefore, thanks to the exact sequence of {\'e}tale sheaves
\[
0 \longrightarrow \Z \longrightarrow \Q \longrightarrow \Q/\Z \longrightarrow 0
\] 
and the fact that $\Q$ is uniquely divisible, we have a canonical isomorphism
\[
{{\rm{H}}}^2(k, \Z) \simeq {{\rm{H}}}^1_{\et}(k, \Q/\Z) = (\mathfrak{g}_k^{\rm{ab}})^D
\]
onto the Pontryagin dual of the profinite topological abelianization of the absolute Galois group $\mathfrak{g}_k$ of $k$. (This is a topological isomorphism, because the compactness of $\mathfrak{g}_k^{\rm{ab}}$ implies that its dual is discrete.)

We obtain a continuous map $\phi:  \mathbf{G}_m(k) = k^{\times} \rightarrow {{\rm{H}}}^2(k, \widehat{\mathbf{G}_m})^D = (\mathfrak{g}_k^{\rm{ab}})^{DD} \simeq \mathfrak{g}_k^{\rm{ab}}$ due to Pontryagin double duality. But there is another natural map $k^{\times} \rightarrow \mathfrak{g}_k^{\rm{ab}}$, namely the local reciprocity map, and it is natural to ask whether these two maps agree
(at least up to a sign). If they do, then by local class field theory $\phi$ induces a
topological isomorphism $(k^{\times})_{{\rm{pro}}} \simeq \mathfrak{g}_k^{\rm{ab}}$, which is what we want to show. Lemma \ref{localreciprocitymap} below therefore completes the proof of Lemma \ref{H^2(G^)dualityspecialtori}.
\end{proof}

\begin{lemma}
\label{localreciprocitymap}
The map $\phi:  k^{\times} \rightarrow \mathfrak{g}_k^{\rm{ab}}$ induced by the local duality pairing as above is the local reciprocity map.
\end{lemma}

\begin{proof}
Let $L/k$ be a finite Galois extension with Galois group $\Gamma$. The local reciprocity map induces an isomorphism ${\rm{rec}}_{L/k}:  k^{\times}/N_{L/k}(L^{\times}) \xrightarrow{\sim} \Gamma^{\rm{ab}}$; the inverse ${\rm{rec}}^{-1}_{L/k}$ of this map is the more natural one to define. In order to show that $\phi$ is the local reciprocity map, therefore, we will show that $\phi|_L \circ {\rm{rec}}^{-1}_{L/k}:  
\Gamma^{\rm{ab}} \rightarrow \Gamma^{\rm{ab}}$ is the identity map. Of course, a priori this does not make sense unless we show that $\phi|_L(N_{L/k}(L^{\times})) = 0$ in $\Gamma^{\rm{ab}}$, but our argument will show that the composition is the identity regardless of which lift to $k^{\times}$ of ${\rm{rec}}^{-1}_{L/k}(\sigma)$ we take for any $\sigma \in \Gamma$, and it will then follow that $\phi|_L$ 
factors through a map $k^{\times}/N_{L/k}(L^{\times}) \rightarrow \Gamma^{\rm{ab}}$.

Let us recall how ${\rm{rec}}_{L/k}^{-1}$ is defined; cf. \cite[Ch.\,VI, \S2.2]{cf}. Let $n = [L :  k]$. We have an isomorphism ${\rm{H}}^2(\Gamma, L^{\times}) \simeq (1/n)\Z/\Z$ by taking invariants. Let $u \in {\rm{H}}^2(\Gamma, L^{\times})$ be the element with invariant $1/n$. Letting $\widehat{\rm{H}}^{\bullet}$ denote Tate cohomology groups, 
$\widehat{\rm{H}}^{-2}(\Gamma, \Z) : = {\rm{H}}_1(\Gamma, \Z) \simeq \Gamma^{\rm{ab}}$,
and by definition ${\rm{rec}}_{L/k}^{-1}$ is the composition
\[
\Gamma^{\rm{ab}} \simeq \widehat{\rm{H}}^{-2}(\Gamma, \Z) \xlongrightarrow{\cup u} \widehat{\rm{H}}^0(\Gamma, L^{\times}) = 
k^{\times}/N_{L/k}(L^{\times}).
\]

Choose a $\sigma \in \Gamma^{\rm{ab}}$. We want to check that $\phi|_L \circ {\rm{rec}}_{L/k}^{-1} (\sigma) = \sigma$. Let $\delta \colon  \Gamma^D = \widehat{\rm{H}}^1(\Gamma, \Q/\Z) \xrightarrow{\sim} \widehat{\rm{H}}^2(\Gamma, \Z)$ be the connecting map in the long exact sequence obtained from the exact sequence
\[
0 \longrightarrow \Z \longrightarrow \Q \longrightarrow \Q/\Z \longrightarrow 0.
\]
Then for any $c \in k^{\times}$ with image $[c] \in \widehat{\rm{H}}^0(\Gamma, L^{\times})$, our $\phi|_L(c)$ is, by definition, the element $\tau \in \Gamma^{\rm{ab}}$ such that ${\inv}([c] \cup \delta \beta) = \beta(\tau)$ for all $\beta \in \widehat{\rm{H}}^1(\Gamma, \Q/\Z)$;
note that $[c]$ has even degree, so $[c] \cup \delta \beta = \delta \beta \cup [c]$. 
So what we need to check is that ${\inv}(\sigma \cup u \cup \delta \beta) = \beta(\sigma)$. 

If we once again think of $\sigma$ as an element of $\widehat{\rm{H}}^{-2}(\Gamma, \Z)$, then, by the definition of the cup product, $\beta(\sigma) = \sigma \cup \beta \in \widehat{\rm{H}}^{-1}(\Gamma, \Q/\Z) \simeq \ker(N_{\Gamma}:  \Q/\Z \rightarrow \Q/\Z) = (1/n)\Z/\Z$, where for any $\Gamma$-module $A$ the map $N_{\Gamma}:  A \rightarrow A$ is $a \mapsto \sum_{\tau \in \Gamma} \tau a$. We need to show, therefore, that
\begin{equation}
\label{localreciprocitymaplemmaeqn1}
\inv(\sigma \cup u \cup \delta \beta) = \sigma \cup \beta.
\end{equation}

Now $\sigma \cup u \cup \delta \beta = (\sigma \cup \delta \beta) \cup u$ since $u$ has even degree,
and this in turn is equal to $\delta(\sigma \cup \beta) \cup u$ by the $\delta$-functoriality of cup products
(see \cite[Ch.\,IV, \S7, Thm.\,4(iv)]{cf} with $p=-2$).
Since $\delta(\sigma \cup \beta) \in \widehat{\rm{H}}^0(\Gamma, \Z) = \Z/n\Z$, 
the left side of (\ref{localreciprocitymaplemmaeqn1}) equals $\delta(\sigma \cup \beta) {\inv}(u) = 
\delta(\sigma \cup \beta)\cdot (1/n) \in (1/n)\Z/\Z$. 

What we need to show, finally, is that the following diagram commutes: 
\[
\begin{tikzcd}
\widehat{\rm{H}}^{-1}(\Gamma, \Q/\Z) \arrow{d}{\delta} \arrow{r}{\sim} & (1/n)\Z/\Z \arrow{d}{x \mapsto nx} \\
\widehat{\rm{H}}^0(\Gamma, \Z) \arrow{r}{\sim} & \Z/n\Z
\end{tikzcd}
\]
This is simple, going back to how $\delta$ is defined in low degrees, as made explicit in \cite[IV, (6.2)]{cf}. 
Explicitly, an element $\xi \in (1/n)\Z/\Z$ lifts to an element $a/n \in \Q$ with $a \in \Z$, and $N_{\Gamma}: \Q \rightarrow \Q$
(which underlies the connecting map in Tate cohomology from degree $-1$ to degree 0) 
carries $a/n$ to $a$.  The class of $\delta(\xi) \in \widehat{\rm{H}}^0(\Gamma, \Z) = \Z/n\Z$
is therefore represented by $a \bmod n\Z$, which comes from $a/n \bmod \Z$ under $n: (1/n)\Z/\Z \simeq \Z/n\Z$.
\end{proof}

Next we show that Theorem $\ref{localdualityH^2(G^)}$ holds for $G = \Ga$.

\begin{proposition}
\label{k_pro=H^2(k,Ga^)*}
For a non-archimedean local field $k$, the map $k_{\pro} \rightarrow {\rm{H}}^2(k, \widehat{\Ga})^*$ induced by the local duality pairing $(\ref{pairingextendsG(k)_pro})$ for $G = \Ga$ is a topological isomorphism.
\end{proposition}

\begin{proof}
When ${\rm{char}}(k) = 0$, both groups vanish: $k_{\pro}$ because $k$ is divisible, and ${\rm{H}}^2(k, \widehat{\Ga})^*$ by Proposition \ref{cohomologyofG_adualwhenkisperfect}, so assume that ${\rm{char}}(k) > 0$. Since the map in question is a continuous map (by Lemma \ref{G(k)H^2(G^)continuous}) from a compact space to a Hausdorff space, it suffices to show that it is a bijection, and for this it suffices, by Pontryagin duality, to show that the dual map is an isomorphism. But both dual groups are one-dimensional $k$-vector spaces. Indeed, ${\rm{H}}^2(k, \widehat{\Ga})$ is one-dimensional by Proposition \ref{omega=H^2(Ga^)}. To see that $(k_{\pro})^D$ is one-dimensional, we first note that the map $(k_{\pro})^D \rightarrow k^D$ is an (algebraic) isomorphism. Indeed, it is clearly injective, and surjectivity follows from the fact that any continuous homomorphism $k \rightarrow \mathbf{R}/\Z$ has kernel that is closed of finite index, since $k$ has finite exponent. It therefore suffices to show that $k^D$ is one-dimensional, and this is \cite[Ch.\,XV, Lemma 2.2.1]{cf}.

Since the dual map is $k$-linear by the functoriality of cup product, we only have to show that it is nonzero. So we only need to show that there exist $\lambda \in k$ and $\alpha \in {\rm{H}}^2(k, \widehat{\Ga})$ such that $\alpha(\lambda) \in {\rm{H}}^2(k, \Gm) = {\rm{H}}^2(k, \Gm)$ is nonzero. Let $\pi \in \calO_k$ be a uniformizer. We take $\alpha = \phi(d\pi) = \psi(Xd\pi)$, where $\phi$ is the isomorphism in Proposition \ref{omega=H^2(Ga^)}, and $\psi$ is the map in Proposition \ref{brauerdifferentialforms}. Recall that $X$ is the variable on $\mathbf{G}_{a,\, k}$ and that $psi$ is the map sending a differential form on $\Ga$ to the associated Brauer class, via Proposition \ref{brauerdifferentialforms}. Then $\alpha(\lambda) = \psi(\lambda d\pi)$, by Lemma \ref{evaluationcompatibility}. Since every $p$-torsion element of ${\rm{H}}^2(k, \Gm)$ is given by an element of this form for some $\lambda \in k$ (by Proposition \ref{brauerdifferentialforms} and Lemma \ref{Omega^1=Rdpi}), and since ${\rm{H}}^2(k, \Gm)[p] \neq 0$, some such Brauer class is nonzero.
\end{proof}

Next we prove some basic topological properties of $G(k)$ which we shall require in the proof of Theorem \ref{localdualityH^2(G^)}.

\begin{lemma}
\label{fundamentalsystem}
Let $k$ be a non-archimedean local field, and $G$ an affine $k$-group scheme of finite type. Then $G(k)$ if locally profinite,
and in particular has a fundamental system of open neighborhoods of the identity consisting of profinite subgroups. 
\end{lemma}

\begin{proof}
There is a closed $k$-subgroup inclusion $G \hookrightarrow {\rm{GL}}_n$ \cite[Ch.\,I, Prop.\,1.10]{borel}, which reduces us to the case $G = {\rm{GL}}_n$, for which the subgroup ${\rm{GL}}_n(\calO_k) \subset {\rm{GL}}_n(k)$ is a profinite open subgroup.
\end{proof}

\begin{lemma}
\label{G(k)modopenfg}
Let $k$ be a non-archimedean local field, $G$ an almost-torus
over $k$. For any open subgroup $U \subset G(k)$, the quotient $G(k)/U$ is finitely generated.
\end{lemma}

\begin{proof}
First, by replacing $G$ with its maximal smooth $k$-subgroup scheme (see \cite[Lemma C.4.1]{cgp}), we may assume that $G$ is smooth. Suppose that we have a short exact sequence
\[
1 \longrightarrow G' \xlongrightarrow{j} G \xlongrightarrow{\pi} G'' \longrightarrow 1
\]
of smooth almost-tori over $k$, and that the lemma holds for $G'$ and $G''$.
The map $G \rightarrow G''$ is smooth, so $G(k) \rightarrow G''(k)$ is open. 
Hence, $\pi(U)$ is an open subgroup of $G''(k)$, so the exact sequence
\[
\frac{G'(k)}{j^{-1}(U)} \longrightarrow \frac{G(k)}{U} \longrightarrow \frac{G''(k)}{\pi(U)},
\]
in which both ends are finitely generated by hypothesis, shows that $G(k)/U$ is finitely generated as well.

Since $G/G^0$ is finite {\'e}tale, and the lemma is clear for finite group schemes, we may therefore assume that $G$ is smooth and connected; i.e., $G$ is a torus. If we have a closed $k$-subgroup inclusion $G \hookrightarrow H$ and if the lemma holds for $H$, then it also holds for $G$ because $G(k) \hookrightarrow H(k)$ is a homeomorphism onto its closed image. Since every torus embeds into the Weil restriction of a split torus through a finite separable extension, we may assume that $G = \R_{k'/k}(\Gm^n)$
for some finite separable extension $k'/k$ and some $n \ge 1$. Then we need to show that for any open subgroup $U \subset (k')^{\times}$, $(k')^{\times}/U$ is finitely generated. 

Since $(k')^{\times}/\calO_{k'}^{\times} \simeq \Z$ is finitely generated, it suffices to show that for any non-archimedean local field $k$, and any open subgroup $U \subset \calO_k^{\times}$, the group $\calO_k^{\times}/U$ is finitely generated. In fact, this group is finite because $\calO_k^{\times}$ is compact.
\end{proof}

\begin{proposition}
\label{G(k)_pro*---->G(k)*isom}
Let $k$ be a local field of characteristic $p > 0$. For any affine commutative $k$-group scheme  $G$ of finite type, the
continuous restriction map $(G(k)_{{\rm{pro}}})^D \rightarrow \Hom_{{\rm{cts}}}(G(k), \Q/\Z)$ is an $($algebraic, not necessarily topological$)$ isomorphism. 
\end{proposition}

\begin{proof}
The map $G(k)_{{{\rm{pro}}}}^D \rightarrow \Hom_{{\rm{cts}}}(G(k), \Q/\Z)$ is clearly injective. To see that it is surjective, it is equivalent to show that any continuous homomorphism $\phi:  G(k) \rightarrow \Q/\Z$
(where $\Q/\Z$ is given the subspace topology from $\mathbf{R}/\Z$) factors through $(1/n)\Z/\Z$ for some positive integer $n$. Indeed, then $\ker \phi$ is a closed subgroup of finite index and hence $\phi$ factors
(continuously) through $G(k)_{\rm{pro}}$. 

Since $\Q/\Z$ has no nontrivial subgroups contained in a small neighborhood of $0$, Lemmas \ref{fundamentalsystem} and \ref{G(k)modopenfg} imply that for any almost-torus $G$, any continuous homomorphism $\phi:  G(k) \rightarrow \Q/\Z$ 
(where $\Q/\Z$ is viewed with its subspace topology from $\mathbf{R}/\Z$)
factors through some $(1/n)\Z/\Z$, as needed.

Now we treat general $G$. Given an arbitrary affine commutative $k$-group scheme $G$ of finite type, we have, by Lemma \ref{affinegroupstructurethm}, an exact sequence
\[
1 \longrightarrow H \longrightarrow G \longrightarrow U \longrightarrow 1
\]
with $H$ an almost-torus and $U$ split unipotent. Let $\phi:  G(k) \rightarrow \Q/\Z$ be a continuous homomorphism. 
By the case of almost-tori, $\phi|_{H(k)}$ factors through some $(1/m)\Z/\Z$. Since $G(k)/H(k) \hookrightarrow U(k)$ has finite exponent
(as $U$ does, since ${\rm{char}}(k) > 0$; this is the place in the proof where
we use the avoidance of characteristic 0), it follows that $\phi$ factors through $(1/n)\Z/\Z$ for some positive multiple $n$ of $m$.
\end{proof}

\begin{remark}
\label{remarkchar0localdualityH^2(G^)}
Proposition \ref{G(k)_pro*---->G(k)*isom} is false in characteristic $0$. Indeed, if $G = \Ga$ then $k_{{{\rm{pro}}}} = 0$ (because $k$ is divisible),
yet $\Hom_{{\rm{cts}}}(k, \Q/\Z) \neq 0$, as may already be seen in the case $k = \Q_p$, for which we have the canonical continuous isomorphism between $\Q_p/\Z_p$ and the $p$-primary part of $\Q/\Z$. Proposition \ref{G(k)_pro*---->G(k)*isom} does remain true in characteristic $0$, however, if we assume that $G$ is an almost-torus, as the proof given above goes through in that case.
\end{remark}

\begin{lemma}
\label{dualstillexact}
Given an exact sequence of continuous maps
\[
A \longrightarrow B \longrightarrow C \longrightarrow E
\]
of Hausdorff topological abelian groups all of which are either (i) compact, or (ii) locally compact and second-countable, then the Pontryagin dual diagram 
\[
C^D \longrightarrow B^D \longrightarrow A^D
\] 
is also exact.
\end{lemma}

\begin{proof}
Consider $\phi \in B^D$ whose restriction to $A$ vanishes, so $\phi \in (B/\im(A))^D$.  We want to show that $\phi$ arises from $C^D$.
 If the continuous inclusion $B/\im(A) \hookrightarrow C$ is a homeomorphism onto a closed subgroup, then $\phi$ extends to an element of $C^D$ and so we would be done. If $B$ is compact, then so is $B/{\rm{im}}(A)$, hence this map is a homeomorphism onto a closed subgroup, since the source is compact and the target Hausdorff. Otherwise, $B$ is second-countable by hypothesis, and the inclusion $B/\im(A) \hookrightarrow C$ has closed image, since this image is the kernel of a continuous map to a Hausdorff group. That this map is a homeomorphism onto its image then follows from Lemma \ref{isom=homeo}. This proves the claim and the proposition.
\end{proof}

\begin{theorem}\label{H^2(G^)altori} $($Theorem $\ref{localdualityH^2(G^)}$$)$ For $k$ a local field of positive characteristic, and $G$ an affine commutative $k$-group scheme of finite type, the cohomology group ${{\rm{H}}}^2(k, \widehat{G})$ is torsion, and the cup product pairing ${{\rm{H}}}^2(k, \widehat{G}) \times {\rm{H}}^0(k, G) \rightarrow {{\rm{H}}}^2(k, \mathbf{G}_m) \xrightarrow[\sim]{\inv} \Q/\Z$ induces a functorial continuous perfect pairing of locally compact Hausdorff groups 
\[
{{\rm{H}}}^2(k, \widehat{G}) \times {\rm{H}}^0(k, G)_{{\rm{pro}}} \rightarrow \Q/\Z,
\]
where ${\rm{H}}^2(k, \widehat{G})$ is discrete.
\end{theorem}

\begin{proof}
Lemma \ref{H^2(G^)istorsion} ensures that the group ${\rm{H}}^2(k, \widehat{G})$ is torsion. The continuity of the cup product pairing, and its extension to a pairing ${\rm{H}}^2(k, \widehat{G}) \times {\rm{H}}^0(k, G)_{\pro} \rightarrow \Q/\Z$, is Lemma \ref{G(k)H^2(G^)continuous}. It remains to show that this last pairing is perfect. For this, we first make some preliminary remarks. For any affine commutative $k$-group scheme $G$ of finite type, the continuous map 
of profinite groups $G(k)_{{\rm{pro}}} \rightarrow {\rm{H}}^2(k, \widehat{G})^D$ is an isomorphism (algebraically, or equivalently
topologically) if and only if the dual map ${\rm{H}}^2(k, \widehat{G}) \rightarrow (G(k)_{{\rm{pro}}})^D$ (between discrete groups) is
bijective. Since $(G(k)_{{\rm{pro}}})^D \xrightarrow{\sim} \Hom_{{\rm{cts}}}(G(k), \Q/\Z)$ as groups (not necessarily as topological groups)
by Proposition \ref{G(k)_pro*---->G(k)*isom}, 
we see that Theorem \ref{localdualityH^2(G^)} 
is equivalent to bijectivity of the natural map ${\rm{H}}^2(k, \widehat{G}) \rightarrow \Hom_{{\rm{cts}}}(G(k), \Q/\Z)$. This latter
reformulation will be used below without comment.

We first prove the proposition when $G$ is an almost-torus. After harmlessly modifying $G$, we may assume by Lemma \ref{almosttorus}(iv) that there is an isogeny $\R_{k'/k}(T') \times B \twoheadrightarrow G$ for some finite $k$-group scheme $B$, some finite separable extension $k'/k$, and some split $k'$-torus $T'$. For notational convenience, denote $\R_{k'/k}(T') \times B$ by $X$. Let $A : = {{\rm{ker}}}(X \rightarrow G)$. Then we have an exact sequence
\[
1 \longrightarrow A \longrightarrow X \longrightarrow G \longrightarrow 1.
\]
In the commutative diagram (\ref{proofH^2(G^)dualityalmosttorisequence1}) below, the top row comes from Proposition \ref{hatisexact} and the maps in the bottom row come from the fact that all of the maps in the ``undualized'' exact sequence are continuous by Proposition \ref{topcohombasics}(v). The map from ${\rm{H}}^1(k, A)^D$ is well-defined because ${\rm{H}}^1(k, A)^D = \Hom_{\rm{cts}}(A, \Q/\Z)$ by Lemma \ref{H^1finiteexponent}. We claim that the indicated maps are isomorphisms and that the bottom row is exact at ${\rm{H}}^1(k, A)^D$ and $\Hom_{{\rm{cts}}}(G(k), \Q/\Z)$. A diagram chase would then give the (algebraic) isomorphism property for the middle vertical map, establishing Theorem \ref{H^2(G^)altori} for $G$.
\begin{equation}
\label{proofH^2(G^)dualityalmosttorisequence1}
\begin{tikzcd}[column sep = tiny]
{\rm{H}}^1(k, \widehat{X}) \arrow{r} \isoarrow{d} & {\rm{H}}^1(k, \widehat{A}) \isoarrow{d} \arrow{r} & {\rm{H}}^2(k, \widehat{G}) \arrow{d} \arrow{r} &  {\rm{H}}^2(k, \widehat{X}) \isoarrow{d} \arrow{r} & {\rm{H}}^2(k, \widehat{A}) \isoarrow{d} \\
{\rm{H}}^1(k, X)^D \arrow{r} & {\rm{H}}^1(k, A)^D \arrow{r} & \Hom_{{\rm{cts}}}(G(k), \Q/\Z) \arrow{r} & \Hom_{{\rm{cts}}}(X(k), \Q/\Z) \arrow{r} & \Hom_{{\rm{cts}}}(A(k), \Q/\Z)
\end{tikzcd}
\end{equation}

The second and the fifth vertical arrows in (\ref{proofH^2(G^)dualityalmosttorisequence1}) 
are (topological) isomorphisms by local duality for finite 
commutative $k$-group schemes. The fourth is an algebraic (perhaps not topological)
isomorphism by local duality for finite 
commutative $k$-group schemes and Lemma \ref{H^2(G^)dualityspecialtori}.
The first vertical arrow is an isomorphism by Lemma \ref{H^1localdualityalmosttori}. Finally, the bottom row is exact at ${\rm{H}}^1(k, A)^D$ and $\Hom_{{\rm{cts}}}(G(k), \Q/\Z)$ by Lemma \ref{dualstillexact}. This completes the proof of the proposition for almost-tori. 

In order to treat the general case, we need the following lemma.

\begin{lemma}
\label{localdualityH^2(G^)devissage}
Consider a short exact sequence
\[
1 \longrightarrow H \longrightarrow G \longrightarrow U \longrightarrow 1
\]
of affine commutative $k$-group schemes of finite type such that $U$ is split unipotent. If 
Theorem $\ref{H^2(G^)altori}$ holds for $H$ and $U$, and $H$ is either split unipotent or an almost-torus, then Theorem $\ref{H^2(G^)altori}$ holds for $G$.
\end{lemma}

\begin{proof}
Using Proposition \ref{hatisexact}, we have the commutative diagram (\ref{anothercommdiagram17}) below, with exact top row by Proposition \ref{H^3(G^)=0}, the second and fourth vertical arrows isomorphisms by hypothesis, and the first vertical arrow an isomorphism when $H$ is an almost-torus by Lemma \ref{H^1localdualityalmosttori}, and when $H$ is split unipotent by Proposition \ref{unipotentcohomology}. The first map in the bottom row is well-defined because the map $U(k) \rightarrow {\rm{H}}^1(k, H)$ is continuous by Proposition \ref{topcohombasics}(v), and because ${\rm{H}}^1(k, H)^D = \Hom_{\rm{cts}}({\rm{H}}^1(k, H), \Q/\Z)$ by Lemma \ref{H^1finiteexponent}. Further, the bottom row of (\ref{anothercommdiagram17}) is exact by Lemma \ref{dualstillexact}. A simple diagram chase now proves the lemma.
\begin{equation}
\label{anothercommdiagram17}
\begin{tikzcd}[column sep = tiny]
{\rm{H}}^1(k, \widehat{H}) \isoarrow{d} \arrow{r} & {\rm{H}}^2(k, \widehat{U}) \isoarrow{d} \arrow{r} & {\rm{H}}^2(k, \widehat{G}) \arrow{d} \arrow{r} & {\rm{H}}^2(k, \widehat{H}) \isoarrow{d} \arrow{r} & 0 \\
{\rm{H}}^1(k, H)^D \arrow{r} & \Hom_{{\rm{cts}}}(U(k), \Q/\Z) \arrow{r} & \Hom_{{\rm{cts}}}(G(k), \Q/\Z) \arrow{r} & \Hom_{{\rm{cts}}}(H(k), \Q/\Z)
\end{tikzcd}
\end{equation}
\end{proof}

Returning to the proof of Theorem \ref{H^2(G^)altori}, induction together with Proposition \ref{k_pro=H^2(k,Ga^)*} and Lemma \ref{localdualityH^2(G^)devissage} complete the proof when $G$ is split unipotent. For the general case, suppose that $G$ is an affine commutative $k$-group scheme of finite type. By Lemma \ref{affinegroupstructurethm} we have an exact sequence
\[
1 \longrightarrow H \longrightarrow G \longrightarrow U \longrightarrow 1
\]
with $H$ an almost-torus and $U$ split unipotent. By Lemma \ref{localdualityH^2(G^)devissage} and the already-treated cases of almost-tori and split unipotent groups, Theorem \ref{H^2(G^)altori} holds for $G$.
\end{proof}

\section{Duality between ${\rm{H}}^1(k, G)$ and ${\rm{H}}^1(k, \widehat{G})$}\label{h1dualgen}

Recall that Theorem \ref{H^1localduality} is only well-posed so far for almost-tori, since if $G$ is not an almost-torus, then $\widehat{G}$ is not representable (Proposition \ref{hatrepresentable}) and so we have not yet defined a topology on
${\rm{H}}^1(k, \widehat{G})$.  In the case of almost-tori, Theorem \ref{H^1localduality} has been proved 
in Lemma \ref{H^1localdualityalmosttori}.  To go beyond that, the first order of business is to define
a reasonable topology on ${\rm{H}}^1(k, \widehat{G})$ for arbitrary affine commutative group schemes of finite type over the local function field $k$. In this section we define such a topology and then prove Theorem \ref{H^1localduality}.

Let $G$ be an affine commutative $k$-group scheme of finite type, so Lemma \ref{affinegroupstructurethm} furnishes an exact sequence
\begin{equation}
\label{H^1dualityproofsequence}
1 \longrightarrow H \longrightarrow G \longrightarrow U \longrightarrow 1
\end{equation}
with $H$ an almost-torus and $U$ split unipotent. 
By Proposition \ref{hatisexact}, we have an fppf-exact sequence of dual sheaves
$$1 \rightarrow \widehat{U} \rightarrow \widehat{G} \rightarrow \widehat{H} \rightarrow 1$$ 
Since ${\rm{H}}^1(k, \widehat{U}) = 0$ (Proposition \ref{unipotentcohomology}(iii)), the map
${\rm{H}}^1(k, \widehat{G}) \rightarrow {\rm{H}}^1(k, \widehat{H})$ is injective.

Recall that $\widehat{H}$ is represented by a locally finite type commutative $k$-group scheme (Proposition \ref{hatrepresentable}), so we obtain a topology on ${\rm{H}}^1(k, \widehat{H})$, by the methods in \cite{ces1} as discussed in \S \ref{sectiontopologyoncohomology}, that is even second-countable and locally profinite (Proposition \ref{secondcountable}).
We wish to give ${\rm{H}}^1(k, \widehat{G})$ the subspace topology from its inclusion into
${\rm{H}}^1(k, \widehat{H})$.  

In the commutative diagram of pairings
\[
\begin{tikzcd}
{\rm{H}}^1(k, G) \arrow[phantom, r, "\times"] & {\rm{H}}^1(k, \widehat{G}) \arrow{r} \arrow[d, hookrightarrow] & \Q/\Z \arrow[d, equals] \\
{\rm{H}}^1(k, H) \arrow[u, twoheadrightarrow] \arrow[phantom, r, "\times"] & {\rm{H}}^1(k, \widehat{H}) \arrow{r} & \Q/\Z
\end{tikzcd}
\]
the continuous surjection along the left side is a topological quotient map
(due to Proposition \ref{secondcountable} and Lemma \ref{isom=homeo}).  Thus, if we define the topological
group structure on ${\rm{H}}^1(k, \widehat{G})$ as a subgroup of ${\rm{H}}^1(k, \widehat{H})$ for a fixed choice
of (\ref{H^1dualityproofsequence}), then 
continuity of the pairing ${\rm{H}}^1(k, H) \times {\rm{H}}^1(k, \widehat{H}) \rightarrow \Q/\Z$ (Lemma \ref{cupproductctsH^1almosttori}) implies that of ${\rm{H}}^1(k, G) \times {\rm{H}}^1(k, \widehat{G}) \rightarrow \Q/\Z$.  (Recall that it does not matter if we view $\Q/\Z$ discretely or with its subspace
topology from $\mathbf{R}/\Z$ for this continuity because the cohomologies involved have
finite exponent, due to Lemma \ref{H^1finiteexponent}.)  

There are two immediate problems:   
\begin{itemize}
\item[(i)] is ${\rm{H}}^1(k, \widehat{G}) \hookrightarrow {\rm{H}}^1(k, \widehat{H})$ a closed subgroup
(and therefore locally compact)? 
\item[(ii)] is this topology independent of the choice of sequence (\ref{H^1dualityproofsequence})? 
\end{itemize}
To obtain an affirmative answer to (ii), we will use the topology defined by a fixed choice
of such sequence to make sense of Theorem \ref{H^1localduality} for $G$ and to actually {\em prove}
the Theorem for $G$. Since the topology on ${\rm{H}}^1(k, G)$ is intrinsic, it will then follow that the topology 
just defined on ${\rm{H}}^1(k, \widehat{G})$ using a choice of (\ref{H^1dualityproofsequence}) is the Pontryagin dual topology, and hence this topology 
is independent of the choice of (\ref{H^1dualityproofsequence})! 
Note in particular that if $G$ is an almost-torus, then this method of defining
an intrinsic topology on ${\rm{H}}^1(k, \widehat{G})$ would have to recover the topology as already defined earlier 
in such cases, since we can use the choice $H = G$ and $U = 0$. 
We now settle problem (i): 

\begin{lemma} 
\label{H^1(G^)closed}
The subgroup ${\rm{H}}^1(k, \widehat{G}) \subset {\rm{H}}^1(k, \widehat{H})$ arising from a sequence ($\ref{H^1dualityproofsequence}$) is closed
$($and hence is locally profinite and second-countable, by Proposition $\ref{secondcountable}$ applied to $H$$)$.
\end{lemma}

\begin{proof} It suffices to show that for any $\alpha \in {\rm{H}}^1(k, \widehat{H})$ such that $\phi(\alpha) = 0$ for every $\phi \in {\rm{H}}^1(k, \widehat{H})^D$ satisfying $\phi|_{{\rm{H}}^1(k, \widehat{G})}=0$, 
necessarily $\alpha \in {\rm{H}}^1(k, \widehat{G})$.

We want to show that $\alpha$ maps to $0$ in ${\rm{H}}^2(k, \widehat{U})$.
By Theorem \ref{H^2(G^)altori}, the map ${\rm{H}}^2(k, \widehat{U}) \rightarrow U(k)^D$ arising from the local duality (i.e., cup product) pairing is an inclusion, so it suffices to show that $\alpha$ (more precisely, its image inside ${\rm{H}}^2(k, \widehat{U})$) pairs trivially with $U(k)$. But $U(k)$ pairs trivially with ${\rm{H}}^1(k, \widehat{G})$, 
which is to say that the image of the natural map
$U(k) \rightarrow {\rm{H}}^2(k, \widehat{U})^* \rightarrow {\rm{H}}^1(k, \widehat{H})^*$
consists of elements $\phi$ vanishing on ${\rm{H}}^1(k, \widehat{G})$. Hence, $U(k)$ pairs trivially with $\alpha$ by the hypothesis on $\alpha$. 
\end{proof}

\begin{theorem}
\label{H^1dualityprop} $($Theorem $\ref{H^1localduality}$$)$
Let $G$ be an affine commutative group scheme of finite type over a local field $k$ of positive characteristic. Then for the topology on ${\rm{H}}^1(k, \widehat{G})$ arising from any choice of short exact sequence $($$\ref{H^1dualityproofsequence}$$)$ with $H$ an almost-torus and $U$ split unipotent, Theorem $\ref{H^1localduality}$ holds for $G$. That is, cup product yields a perfect pairing between locally compact Hausdorff groups of finite exponent
\[
{\rm{H}}^1(k, G) \times {\rm{H}}^1(k, \widehat{G}) \rightarrow {\rm{H}}^2(k, \Gm) \xrightarrow[\sim]{{\rm{inv}}} \Q/\Z.
\]
In particular, the topology on ${\rm{H}}^1(k, \widehat{G})$ is independent of the choice of $($$\ref{H^1dualityproofsequence}$$)$.
\end{theorem}

\begin{proof}
That the groups ${\rm{H}}^1(k, G)$ and ${\rm{H}}^1(k, \widehat{G})$ have finite exponent follows from Lemma \ref{H^1finiteexponent}. As discussed above, the continuity of the cup product pairing for $G$ follows from the corresponding continuity for the almost-torus $H$, which is Lemma \ref{cupproductctsH^1almosttori}. As was also discussed above, the independence of the topology on ${\rm{H}}^1(k, \widehat{G})$ from the choice of sequence (\ref{H^1dualityproofsequence}) follows once we prove that the cup product pairing is perfect, since the topology on ${\rm{H}}^1(k, G)$ is completely intrinsic, defined in \S \ref{sectiontopologyoncohomology}. It therefore only remains to prove this perfectness.

Since the group ${\rm{H}}^1(k, \widehat{G})$ will be locally compact, Hausdorff, and second-countable, it suffices, by Lemma \ref{isom=homeo}, to show that the continuous map ${\rm{H}}^1(k, \widehat{G}) \rightarrow {\rm{H}}^1(k, G)^D$ is an algebraic isomorphism. We claim that we have a commutative diagram
\[
\begin{tikzcd}
0 \arrow{r} & {\rm{H}}^1(k, \widehat{G}) \arrow{r} \arrow{d} & {\rm{H}}^1(k, \widehat{H}) \isoarrow{d} \arrow{r} & {\rm{H}}^2(k, \widehat{U}) \arrow[d, hookrightarrow] \\
0 \arrow{r} & {\rm{H}}^1(k, G)^D \arrow{r} & {\rm{H}}^1(k, H)^D \arrow{r} & U(k)^D
\end{tikzcd}
\]
in which the top row is exact, the bottom row is exact at ${\rm{H}}^1(k, G)^D$, and the indicated maps are isomorphisms or inclusions. Assuming this, a simple diagram chase shows that the first vertical arrow is an isomorphism. As we have already seen, the top row is exact due to Propositions \ref{hatisexact} and \ref{unipotentcohomology}(iii). The second vertical arrow is an isomorphism by Lemma \ref{H^1localdualityalmosttori}. The third vertical arrow is an inclusion by Theorem \ref{H^2(G^)altori}.  

The maps in the bottom row are well-defined because the ``undualized'' maps are continuous by Proposition \ref{topcohombasics}(v). Finally, the bottom row is exact at ${\rm{H}}^1(k, G)^D$ because the map ${\rm{H}}^1(k, H) \rightarrow {\rm{H}}^1(k, G)$ is surjective due to the vanishing of ${\rm{H}}^1(k, U)$ (Proposition \ref{unipotentcohomology}(ii)).
\end{proof}

Before ending this section, let us a make a couple more observations about the topology on ${\rm{H}}^1(k, \widehat{G})$. First, lest the reader become complacent, we note that it is not in general $\delta$-functorial. See Remark \ref{discontinuous}. On the other hand, we observe that ${\rm{H}}^1(k, \widehat{(\cdot)})$ {\em is} functorial: 

\begin{proposition}
\label{H^1(G^)functorial}
For a homomorphism $G \rightarrow G'$ of affine commutative group schemes of finite type over a local field $k$, the induced map ${\rm{H}}^1(k, \widehat{G'}) \rightarrow {\rm{H}}^1(k, \widehat{G})$ is continuous.
\end{proposition}

\begin{proof}
The key is to show that we have a commutative diagram of short exact sequences
\begin{equation}
\label{H^1(G^)functorialpfeqn}
\begin{tikzcd}
1 \arrow{r} & H \arrow{r} \arrow{d} & G \arrow{r} \arrow{d} & U \arrow{r} \arrow{d} & 1 \\
1 \arrow{r} & H' \arrow{r} & G' \arrow{r} & U' \arrow{r} & 1
\end{tikzcd}
\end{equation}
with $H, H'$ almost-tori and $U, U'$ split unipotent. Then the continuity follows (just by the definition of the topologies) from the continuity of the map ${\rm{H}}^1(k, \widehat{H'}) \rightarrow {\rm{H}}^1(k, \widehat{H})$ between cohomology groups of locally finite type $k$-group schemes (Proposition \ref{hatrepresentable}).

To see that we have such a commutative diagram, let $f\colon  G \rightarrow G'$ be the given map. First choose an almost-torus $H \subset G$ such that $U : = G/H$ is split unipotent (Lemma \ref{affinegroupstructurethm}). Then $f(H) \subset G'$ is an almost-torus, by Lemma \ref{almosttorusextension}. Let $\overline{G'} : = G/f(H)$, and let $\pi:  G' \rightarrow \overline{G'}$ denote the quotient map. By Lemma \ref{affinegroupstructurethm}, there is an almost-torus $\overline{H'} \subset \overline{G'}$ such that $U' : = \overline{G'}/\overline{H'}$ is split unipotent. Then $H' : = \pi^{-1}(\overline{H'})$ is an extension of $\overline{H'}$ by $f(H)$, hence is an almost-torus by Lemma \ref{almosttorusextension}, and we have the diagram (\ref{H^1(G^)functorialpfeqn}) because $f(H) \subset H'$.
\end{proof}

\chapter{Local Integral Cohomology}
\label{chapterlocalintcohom}

In this chapter we prove the integral annihilator aspects of local Tate duality mentioned in \S\ref{intro}, namely, Theorems \ref{exactannihilatorH^2(G)}, \ref{exactannihilatorH^2(G^)}, and \ref{exactannihilatorH^1(G)} (see Propositions \ref{H^2(G)G^(k)integraldualityprop}, \ref{H^2(O,G^)dualityprop}, and \ref{H^1(k,G^)/H^1(O,G^)}), generalizing the analogous classical results for finite discrete Galois modules (often stated in terms of unramified cohomology classes). Since there is no good structure theory for arbitrary affine commutative flat group schemes of finite type over discrete valuation rings, we are not able to prove results of such precision, only obtaining the classical results at all but finitely many places (by choosing a model for our group scheme over some dense open subscheme of the scheme of integers of our global field $k$, by which we mean the scheme $\Spec(\calO_k)$ when $k$ is a number field, and the smooth proper curve $X$ of which $k$ is the function field in the function field setting). This is not just an artifact of the proofs. The results really are not true at all places of $k$, but only almost all places, even if one restricts to flat affine models of the generic fiber; for an illustration of this phenomenon, see Example \ref{almostallnotall}.

The results of this chapter -- and especially the injectivity statements for the maps from cohomology of $\calO_v$ to cohomology of $k_v$ for almost all places $v$ of $k$ -- while interesting in their own right, will also play an important role in relating the cohomology of $\A$ to that of the local fields $k_v$ in \S \ref{sectionadeliclocalcohom}. Perhaps somewhat surprisingly, it seems that the injectivity of the maps ${\rm{H}}^2(\calO_v, \widehat{\mathscr{G}}) \rightarrow {\rm{H}}^2(k_v, \widehat{G})$ for almost all $v$ cannot be proved directly, but actually requires one to prove Theorem \ref{exactannihilatorH^2(G^)} in its entirety, and the proof of this result in turn depends upon first proving Theorem \ref{exactannihilatorH^1(G)}.

\section{Vanishing of ${\rm{H}}^3(\calO_v, \widehat{\mathscr{G}})$}

In this section, we prove the integral analogue of Proposition \ref{H^3(G^)=0} (Proposition \ref{H^3(O, G^)=0}). The following important lemma will play a crucial role in our proofs of the integral annihilator aspects of local duality.

\begin{lemma}
\label{spreadingoutdualsheaves}
Suppose that we have a short exact sequence
\[
1 \longrightarrow G' \longrightarrow G \longrightarrow G'' \longrightarrow 1
\]
of affine commutative group schemes of finite type over a global function field $k$. Then there is a finite set $S$ of places of $k$ such that this sequence spreads out to an exact sequence 
\[
1 \longrightarrow \mathscr{G}' \longrightarrow \mathscr{G} \longrightarrow \mathscr{G}'' \longrightarrow 1
\]
of $\mathcal{O}_S$-group schemes such that the corresponding sequence
\[
1 \longrightarrow \widehat{\mathscr{G}''} \longrightarrow \widehat{\mathscr{G}} \longrightarrow \widehat{\mathscr{G}'} \longrightarrow 1
\]
of fppf dual sheaves is exact. The same holds with this $\mathcal{O}_S$ replaced by $\prod_{v \notin S} \mathcal{O}_v$.
\end{lemma}

\begin{proof}
We will prove the lemma for $\mathcal{O}_S$; the proof for $\prod \mathcal{O}_v$ is exactly the same. Since faithful flatness of a map spreads out, the only thing that is not clear is that we may obtain surjectivity of the map $\widehat{\mathscr{G}} \rightarrow \widehat{\mathscr{G}'}$. For this, it is enough to show that $\calExt^1_{\mathcal{O}_S}(\mathscr{G}'', \Gm) = 0$ for sufficiently large $S$. We know that $G''$ admits a filtration by finite group schemes, tori, and $\Ga$. Indeed, this follows from Lemmas \ref{affinegroupstructurethm}  and \ref{almosttorus}(ii). This filtration spreads out to one by finite flat group schemes, tori, and $\Ga$ over some $\mathcal{O}_S$. It therefore suffices to show that each of these group schemes has vanishing $\calExt^1(\cdot, \Gm)$. For tori and finite flat group schemes, this follows from \cite[VIII, Prop.\,3.3.1]{sga7}, so we are reduced to the case of $\Ga$, which follows from Proposition \ref{ext^1(Ga,Gm)=0}.
\end{proof}

\begin{proposition}
\label{H^i(O,G)=0}
For a finite type group scheme $G$ over a global field $k$, and a finite type $\calO_S$-model $\mathscr{G}$ of $G$ for some finite set $S$ of places of $k$, we have for all but finitely many places $v$ of $k$ that ${\rm{H}}^i(\calO_v, \mathscr{G}) = 0$ for all $i > 1$.
\end{proposition}

\begin{proof}
This is essentially shown in the proof of \cite[Thm.\,2.18]{ces2}, but we repeat the argument here for the reader's convenience. By \cite[Lemma 6.4]{ces1}, $G$ is a closed subgroup scheme of a smooth commutative finite type group scheme $H$, and the quotient $H' := H/G$ is also smooth of finite type. We may then spread out the resulting exact sequence over $k_v$ to an exact sequence
\[
1 \longrightarrow \mathscr{G} \longrightarrow \mathscr{H} \longrightarrow \mathscr{H}' \longrightarrow 1
\]
over some $\calO_S$, with $\mathscr{H}$ and $\mathscr{H}'$ still smooth of finite type. Then, letting $\F_v$ denote the residue field of $\calO_v$, the natural map ${\rm{H}}^i(\calO_v, \mathscr{H}) \rightarrow {\rm{H}}^i(\F_v, \mathscr{H})$ is an isomorphism for $i \geq 1$ by \cite[Thm.\,11.7, 2]{briii}, and similarly for $\mathscr{H}'$. By the five lemma, the same holds for $\mathscr{G}'$, but now only for $i \geq 2$. It therefore suffices to show that ${\rm{H}}^2(\F_v, \mathscr{G}') = 0$, and this follows from \cite[Lemma 2.7(e)]{ces2}.
\end{proof}

\begin{proposition}
\label{H^3(O, G^)=0}
For a global function field $k$, an affine commutative $k$-group scheme of finite type, and an $\calO_S$-model $\mathscr{G}$ of $G$ $($for some finite set $S$ of places of $k$$)$, we have ${\rm{H}}^3(\calO_v, \widehat{\mathscr{G}}) = 0$ for all but finitely many places $v$ of $k$.
\end{proposition}

\begin{proof}
Thanks to Lemmas \ref{spreadingoutdualsheaves}, \ref{affinegroupstructurethm}, and \ref{almosttorus}(ii), we are reduced to the cases in which $G$ is either $\Ga$, finite, or a torus. The case $G = \Ga$ follows from Proposition \ref{H^3(G_a^)=0}, and the case of finite group schemes follows from Proposition \ref{H^i(O,G)=0} and Cartier duality. We are therefore reduced to the case when $G = T$ is a torus.

In this case, we may by enlarging $S$ suppose that $\mathscr{G} = \mathscr{T}$ is a torus over $\calO_S$. Then $\widehat{\mathscr{T}}$ is an \'etale $\calO_S$-group scheme, so we may take our cohomology to be \'etale. By \cite[Thm.\,11.7, 2]{briii}, we have ${\rm{H}}^3(\calO_v, \widehat{\mathscr{T}}) \simeq {\rm{H}}^3(\F_v, \widehat{\mathscr{T}})$, where $\F_v$ denotes the residue field of $\calO_v$. The latter group vanishes because finite fields have strict cohomological dimension $2$.
\end{proof}

\section{Integral annihilator aspects of duality between ${\rm{H}}^2(k_v,G)$
and ${\rm{H}}^0(k_v, \widehat{G})$}

In this section we prove Theorem \ref{exactannihilatorH^2(G)}. 

\begin{theorem}
\label{H^2(G)G^(k)integraldualityprop} $($Theorem $\ref{exactannihilatorH^2(G)}$$)$ If $k$ if a global function field, and $G$ is an affine commutative $k$-group scheme of finite type, and $\mathscr{G}$ an $\mathcal{O}_S$-model of $G$, then for all but finitely many places $v$ of $k$, we have ${\rm{H}}^2(\mathcal{O}_v, \mathscr{G}) = 0$ and the map $\widehat{\mathscr{G}}(\mathcal{O}_v) \rightarrow \widehat{G}(k_v)$ is an isomorphism. 

In particular, for such $v$ the maps ${\rm{H}}^2(\mathcal{O}_v, \mathscr{G}) \rightarrow {\rm{H}}^2(k_v, G)$ and ${\rm{H}}^0(\mathcal{O}_v, \widehat{\mathscr{G}}) \rightarrow {\rm{H}}^0(k_v, \widehat{G})$ are injective, and ${\rm{H}}^2(\mathcal{O}_v, \mathscr{G})$ is the exact annihilator of ${\rm{H}}^0(\mathcal{O}_v, \widehat{\mathscr{G}})$.
\end{theorem}

\begin{proof}
Let us first note that it suffices to prove the proposition for some $S$ and some $\calO_S$-model of $G$, since any two such become isomorphic once we enlarge $S$. The vanishing of ${\rm{H}}^2(\calO_v, \mathscr{G})$ for all but finitely many $v$ follows from \cite[Lem.\,2.12]{ces2}. It therefore only remains to check that $\widehat{\mathscr{G}}(\calO_v) \xrightarrow{\sim} \widehat{G}(k_v)$ for all but finitely many $v$. When $G$ is finite we may, perhaps after shrinking $S$, assume that $\mathscr{G}$ is finite flat over $\calO_S$. The same then holds for the Cartier dual $\widehat{\mathscr{G}}$, so the valuative criterion for properness yields that $\widehat{\mathscr{G}}(\mathcal{O}_v) = \widehat{G}(k_v)$.

Next suppose that $G = \R_{k'/k}(T')$ for some finite separable extension $k'/k$ and some split $k'$-torus $T'$. We may assume that $T' = \Gm$. By Proposition \ref{charactersseparableweilrestriction}, $\widehat{\R_{k'/k}(\Gm)}(k_v) = \R_{k'/k}(\Z)(k_v) = \prod_{v'\mid v} \Z$, where the product is over the places $v'$ of $k'$ lying above $v$. On the other hand, we may take our $\calO_S$-model for $G$ to be $\mathscr{G} : = \R_{\mathcal{O}_{S'}/\mathcal{O}_S}(\Gm)$, where $S'$ is the set of places of $k'$ lying above $S$. If we choose $S$ to consist of all places ramified in $k'$, then for each $v \notin S$, Proposition \ref{charactersseparableweilrestriction} again yields $\widehat{\R_{\mathcal{O}_{S'}/\mathcal{O}_S}(\Gm)}(\mathcal{O}_v) = \R_{\mathcal{O}_{S'}/\mathcal{O}_S}(\Z)(\mathcal{O}_v) = \prod_{v' \mid v} \Z$, so the map $\widehat{\mathscr{G}}(\mathcal{O}_v) \rightarrow \widehat{G}(k_v)$ is an isomorphism.

Next suppose that $G$ is an almost-torus. By Lemma \ref{almosttorus}(iv), we may harmlessly modify $G$ and thereby assume that there is an exact sequence
\[
1 \longrightarrow B \longrightarrow C \times \R_{k'/k}(T') \longrightarrow G \longrightarrow 1,
\]
where $B, C$ are finite commutative $k$-group schemes, $k'/k$ is a finite separable extension, and $T'$ is a split $k'$-torus. Let $X : = C \times \R_{k'/k}(T')$. We may spread this out and apply Lemma \ref{spreadingoutdualsheaves} to obtain an exact sequence
\[
1 \longrightarrow \widehat{\mathscr{G}} \longrightarrow \widehat{\mathscr{X}} \longrightarrow \widehat{\mathscr{B}} \longrightarrow 1.
\]
Consider the following commutative diagram: 
\[
\begin{tikzcd}
0 \arrow{r} & \widehat{\mathscr{G}}(\mathcal{O}_v) \arrow{d} \arrow{r} & \widehat{\mathscr{X}}(\mathcal{O}_v) \isoarrow{d} \arrow{r} & \widehat{\mathscr{B}}(\mathcal{O}_v) \isoarrow{d} \\
0 \arrow{r} & \widehat{G}(k_v) \arrow{r} & \widehat{X}(k_v) \arrow{r} & \widehat{B}(k_v) 
\end{tikzcd}
\]
The rows are clearly exact, and the second and third vertical arrows are isomorphisms by the already-treated cases of finite group schemes and separable Weil restrictions of split tori. It follows that the first vertical arrow is an isomorphism.

Finally, suppose that $G$ is an arbitrary affine commutative $k$-group scheme of finite type. By Lemma \ref{affinegroupstructurethm}, there is an exact sequence
\[
1 \longrightarrow H \longrightarrow G \longrightarrow U \longrightarrow 1
\]
with $H$ an almost-torus and $U$ split unipotent. By Lemma \ref{spreadingoutdualsheaves}, this spreads out to yield an exact sequence
\[
1 \longrightarrow \widehat{\mathscr{U}} \longrightarrow \widehat{\mathscr{G}} \longrightarrow \widehat{\mathscr{H}} \longrightarrow 1.
\]

In the commutative diagram (\ref{diagramagain13}) below, the top horizontal arrow is injective for almost all $v$ because $\widehat{\mathscr{U}}(\mathcal{O}_v) = 0$ for almost all $v$, since $\mathscr{U}$ admits a filtration by $\Ga$'s (after enlarging $S$). This same arrow is surjective because ${\rm{H}}^1(\mathcal{O}_v, \widehat{\mathscr{U}}) = 0$ for almost all $v$ (by filtering $\mathscr{U}$, this reduces to showing that ${\rm{H}}^1(\mathcal{O}_v, \widehat{\Ga}) = 0$; that in turn follows from Proposition \ref{cohomologyofG_adualgeneralk}). The right vertical arrow is an isomorphism by the already-treated case of almost-tori. Finally, the bottom arrow is an inclusion because $\widehat{U}(k_v) = 0$. A simple diagram chase now shows that the left vertical arrow is an isomorphism.
\begin{equation}
\label{diagramagain13}
\begin{tikzcd}
\widehat{\mathscr{G}}(\mathcal{O}_v) \arrow{r}{\sim} \arrow{d} & \widehat{\mathscr{H}}(\mathcal{O}_v) \isoarrow{d} \\
\widehat{G}(k_v) \arrow[r, hookrightarrow] & \widehat{H}(k_v)
\end{tikzcd}
\end{equation}
\end{proof}

\section{Injectivity of ${\rm{H}}^1(\mathcal{O}_v, \widehat{\mathscr{G}}) \rightarrow {\rm{H}}^1(k_v, \widehat{G})$ and ${\rm{H}}^1(\calO_v, \mathscr{G}) \rightarrow {\rm{H}}^1(k_v, G)$}

In this section we show that for an affine commutative group scheme $G$ of finite type over a global field $k$, a finite non-empty set $S$ of places of $k$, and an $\calO_S$-model $\mathscr{G}$ of $G$, for all but finitely many $v$ the maps ${\rm{H}}^1(\calO_v, \widehat{\mathscr{G}}) \rightarrow {\rm{H}}^1(k_v, \widehat{G})$ and ${\rm{H}}^1(\calO_v, \mathscr{G}) \rightarrow {\rm{H}}^1(k_v, G)$ are injective (Propositions \ref{H^1(O,G^)-->H^1(k,G^)injective} and \ref{H^1(G)injective}).

\begin{proposition}
\label{H^1(O,G^)-->H^1(k,G^)injective}
Let $G$ be an affine commutative group scheme of finite type over the global field $k$, 
and $\mathscr{G}$ an $\mathcal{O}_S$-model of $G$ for a non-empty finite set $S$ of places
of $k$. Then for all but finitely many places $v$ of $k$, the map ${\rm{H}}^1(\mathcal{O}_v, \widehat{\mathscr{G}}) \rightarrow {\rm{H}}^1(k_v, \widehat{G})$ is injective.
\end{proposition}

\begin{proof}
When $G$ is finite, then one may (after enlarging $S$) take $\mathscr{G}$ to be finite and locally free over $\calO_S$, hence the same holds for $\widehat{\mathscr{G}}$. Therefore, any torsor $\mathscr{X}$ for this group over $\mathcal{O}_v$ 
is a finite flat $\mathcal{O}_v$-scheme. If such a scheme has a $k_v$-point then it has
an $\mathcal{O}_v$-point by the valuative criterion for properness. This settles injectivity when $G$ is finite. 

Next suppose that $G = \R_{k'/k}(T')$ for some finite separable extension $k'/k$ and some split $k'$-torus $T'$. We may assume that $T' = \Gm$. Choose any $S$
that contains all places of $k$ that are ramified in $k'$, and let $S'$ be the set of places of $k'$ lying above $S$. For the $\mathcal{O}_S$-model $\mathscr{G} = \R_{\mathcal{O}_{S'}/\mathcal{O}_S}(\Gm)$ of $G$,
we claim that ${\rm{H}}^1(\mathcal{O}_v, \widehat{\mathscr{G}}) = 0$ for all $v \notin S$. 

By Proposition \ref{charactersseparableweilrestriction} we have $\widehat{\mathscr{G}} = \R_{\mathcal{O}_{S'}/\mathcal{O}_S}(\Z)$, so we just need to show that $${\rm{H}}^1(\mathcal{O}_v, \prod_{v' \mid v} \R_{\mathcal{O}_{v'}/\mathcal{O}_v}(\Z)) = 0$$
for all $v \not\in S$.  
The group scheme $\R_{\mathcal{O}_{v'}/\mathcal{O}_v}(\Z)$ is smooth, so we may take our cohomology to be {\'e}tale. Since finite pushforward is an exact functor between categories of {\'e}tale sheaves, it suffices to show that ${\rm{H}}^1(\mathcal{O}_v, \Z) = 0$. But $\mathcal{O}_v$ is a normal noetherian domain, so ${\rm{H}}^1(\mathcal{O}_v, \Z) = \Hom_{{\rm{cts}}}(\pi_1(\mathcal{O}_v), \Z) = 0$, since $\pi_1(\mathcal{O}_v)$ is profinite. 

Now suppose that $G = U$ is split unipotent. Then ${\rm{H}}^1(\mathcal{O}_v, \widehat{\mathscr{U}}) = 0$ for almost all $v$. Indeed, Lemma \ref{spreadingoutdualsheaves} reduces us to the case $U = \Ga$, so we need to show 
 ${\rm{H}}^1(\mathcal{O}_v, \widehat{\Ga}) = 0$. That vanishing in turn follows from Proposition \ref{cohomologyofG_adualgeneralk}.

Next assume that $G$ is an almost-torus. By Lemma \ref{almosttorus}(iv), after harmlessly modifying $G$ we may assume that there is an exact sequence
\[
1 \longrightarrow B \longrightarrow X \longrightarrow G \longrightarrow 1,
\]
where $X = C \times \R_{k'/k}(T')$ for a finite separable extension field $k'/k$,
$B$ and $C$ are commutative finite $k$-group schemes, and $T'$ is a split $k'$-torus. Spreading out, by Lemma \ref{spreadingoutdualsheaves} we obtain (perhaps after enlarging $S$) an exact sequence
\[
1 \longrightarrow \widehat{\mathscr{G}} \longrightarrow \widehat{\mathscr{X}} \longrightarrow \widehat{\mathscr{B}} \longrightarrow 1.
\]
We therefore obtain for almost all $v$ a commutative diagram with exact rows: 
\[
\begin{tikzcd}
\widehat{\mathscr{X}}(\mathcal{O}_v) \isoarrow{d} \arrow{r} & \widehat{\mathscr{B}}(\mathcal{O}_v) \isoarrow{d} \arrow{r} & {\rm{H}}^1(\mathcal{O}_v, \widehat{\mathscr{G}}) \arrow{r} \arrow{d} & {\rm{H}}^1(\mathcal{O}_v, \widehat{\mathscr{X}}) \arrow[d, hookrightarrow] \\
\widehat{X}(k_v) \arrow{r} & \widehat{B}(k_v) \arrow{r} & {\rm{H}}^1(k_v, \widehat{G}) \arrow{r} & {\rm{H}}^1(k_v, \widehat{X})
\end{tikzcd}
\]
in which the first two vertical arrows are isomorphisms for all but finitely many $v$ by Theorem \ref{H^2(G)G^(k)integraldualityprop}, and the last vertical arrow is an inclusion for all but finitely many $v$ by the already-treated cases of finite group schemes and separable Weil restrictions of split tori. An application of the five lemma now shows that the third vertical arrow is an inclusion.

Now consider the general case; that is, let $G$ be an affine commutative $k$-group scheme of finite type. By Lemma \ref{affinegroupstructurethm}, there is an exact sequence
\[
1 \longrightarrow H \longrightarrow G \longrightarrow U \longrightarrow 1
\]
with $H$ an almost-torus and $U$ split unipotent. By Lemma \ref{spreadingoutdualsheaves}, after enlarging $S$ this spreads out to yield an exact sequence
\[
1 \longrightarrow \widehat{\mathscr{U}} \longrightarrow \widehat{\mathscr{G}} \longrightarrow \widehat{\mathscr{H}} \longrightarrow
 1.
\]
We therefore obtain a commutative diagram for almost every $v$
\[
\begin{tikzcd}
{\rm{H}}^1(\mathcal{O}_v, \widehat{\mathscr{G}}) \arrow[r, hookrightarrow] \arrow{d} & {\rm{H}}^1(\mathcal{O}_v, \widehat{\mathscr{H}}) \arrow[d, hookrightarrow] \\
{\rm{H}}^1(k_v, \widehat{G}) \arrow{r} & {\rm{H}}^1(k_v, \widehat{H})
\end{tikzcd}
\]
in which the top arrow is an inclusion because ${\rm{H}}^1(\mathcal{O}_v, \widehat{\mathscr{U}}) = 0$ for almost all $v$
and the right arrow is an inclusion because of the already-treated case of almost-tori. It follows that the left vertical arrow is an inclusion.
\end{proof}

Now we turn to proving the injectivity of ${\rm{H}}^1(\calO_v, \mathscr{G}) \rightarrow {\rm{H}}^1(k_v, G)$ for all but finitely many $v$. We will actually prove this beyond the affine setting, and we remark that our result (Proposition \ref{H^1(G)injective}) strengthens \cite[Lem.\,2.19]{ces2}. What we want to show is that every commutative finite type group scheme over a global field $k$ has the following property: 
\begin{equation}
\label{H^1(G)injectsproperty}
\text{\parbox{.85\textwidth}{For any (equivalently, for some) finite non-empty set $S$ of places of $k$ and $\calO_S$-model $\mathscr{G}$ of $G$, the map ${\rm{H}}^1(\calO_v, \mathscr{G}) \rightarrow {\rm{H}}^1(k_v, G)$ is injective for all but finitely many places $v$ of $k$.}}
\end{equation}

The truth of the statement for one $S$ and one $\calO_S$-model of $G$ implies it for all $S$ and all $\calO_S$-models because any two such models over $\calO_S$ and $\calO_{S'}$ become isomorphic over some $\calO_{S''}$.

\begin{lemma}
\label{H^1injectivefinite}
$(\ref{H^1(G)injectsproperty})$ holds if $G$ is a finite commutative $k$-group scheme.
\end{lemma}

\begin{proof}
This follows from the valuative criterion of properness. For more details, see the first paragraph of the proof of Proposition \ref{H^1(O,G^)-->H^1(k,G^)injective}.
\end{proof}

\begin{lemma}
\label{connectedreductionH^1injective}
Let $G$ be a finite type commutative group scheme over a global field $k$, and suppose that we have an exact sequence
\[
1 \longrightarrow H \longrightarrow G \longrightarrow B \longrightarrow 1
\]
with $B$ a finite commutative $k$-group scheme. If $(\ref{H^1(G)injectsproperty})$ holds for $H$, then it also holds for $G$.
\end{lemma}

\begin{proof}
We have an exact sequence
\[
1 \longrightarrow H \longrightarrow G \longrightarrow B \longrightarrow 1
\]
and this spreads out to an exact sequence
\[
1 \longrightarrow \mathscr{H} \longrightarrow \mathscr{G} \longrightarrow \mathscr{B} \longrightarrow 1
\]
over some $\mathcal{O}_S$.  At the cost of enlarging $S$, for $v \not\in S$
the resulting commutative diagram of exact sequences
\[
\begin{tikzcd}
\mathscr{B}(\mathcal{O}_v) \arrow{r} \isoarrow{d} & {\rm{H}}^1(\mathcal{O}_v, \mathscr{H}) \arrow{r} \arrow[d, hookrightarrow] & {\rm{H}}^1(\mathcal{O}_v, \mathscr{G}) \arrow{r} \arrow{d} & {\rm{H}}^1(\mathcal{O}_v, \mathscr{B}) \arrow[d, hookrightarrow] \\
B(k_v) \arrow{r} & {\rm{H}}^1(k_v, H) \arrow{r} & {\rm{H}}^1(k_v, G) \arrow{r} & {\rm{H}}^1(k_v, B)
\end{tikzcd}
\]
has the first vertical arrow an isomorphism due to the valuative criterion of properness because $\mathscr{B}$ is finite, the second vertical arrow an inclusion by hypothesis
on $H$, and the last vertical arrow an inclusion by Lemma \ref{H^1injectivefinite}. A simple diagram chase now shows that the third vertical arrow is injective.
\end{proof}

\begin{lemma}
\label{H^1(G)injectivesmooth}
$(\ref{H^1(G)injectsproperty})$ holds for smooth commutative $G$.
\end{lemma}

\begin{proof}
By Lemma \ref{connectedreductionH^1injective}, we may replace $G$ with $G^0$ to arrange
 that $G$ is connected. Then in fact ${\rm{H}}^1(\calO_v, \mathscr{G}) = 0$ for all but finitely many $v$ by \cite[Lem.\,2.12]{ces2}.
\end{proof}

Note that in the number field setting, Lemma \ref{H^1(G)injectivesmooth} completes the proof of (\ref{H^1(G)injectsproperty}) for all commutative groups schemes $G$ of finite type. The argument for function fields is more difficult, beginning with some more lemmas.

\begin{lemma}
\label{H^1(G)injectweilrestrict}
Let $k'$ be a finite extension of the global field $k$, $B'$ a finite commutative $k'$-group scheme, and $G \subset \R_{k'/k}(B')$ a $k$-subgroup scheme. Then $(\ref{H^1(G)injectsproperty})$ holds for $G$.
\end{lemma}

\begin{proof}
We first treat the case $G = \R_{k'/k}(B')$. Let $\mathscr{B}'$ be a
finite flat commutative $\mathcal{O}_{S'}$-model for $B'$, where $S'$ is the set of places of $k'$ above a 
non-empty finite set $S$ of places of $k$, so $\mathscr{G} : = \R_{\mathcal{O}_{S'}/\mathcal{O}_S}(\mathscr{B}')$
is an $\mathcal{O}_S$-model of $G$. The Leray spectral sequence associated to the morphism $\Spec(\prod_{v' \mid v} \mathcal{O}_{v'}) \rightarrow \Spec(\mathcal{O}_v)$ yields an inclusion ${\rm{H}}^1(\mathcal{O}_v, \mathscr{G}) \hookrightarrow \prod_{v'\mid v} {\rm{H}}^1(\mathcal{O}_{v'}, \mathscr{B}')$, and similarly with $\mathcal{O}_v$ replaced by $k_v$. By functoriality of the spectral sequence, therefore, we obtain a commutative diagram
\[
\begin{tikzcd}
{\rm{H}}^1(\mathcal{O}_v, \mathscr{G}) \arrow[r, hookrightarrow] \arrow{d} & \prod_{v'\mid v} {\rm{H}}^1(\mathcal{O}_{v'}, \mathscr{B}') \arrow[d, hookrightarrow] \\
{\rm{H}}^1(k_v, G) \arrow{r} & \prod_{v\mid v} {\rm{H}}^1(k_{v'}, B')
\end{tikzcd}
\]
where the second vertical arrow is an inclusion by Lemma \ref{H^1injectivefinite}. It follows that the first vertical arrow is also an inclusion.

Now we turn to the general case. Let $H : = \R_{k'/k}(B')/G$; note that $H$ is affine \cite[Ch.\,III, \S3, no.\,5, Th.\,5.6]{demazuregabriel}. For sufficiently big $S$
(and $S'$ the set of places over it in $k'$) the exact sequence
\[
1 \longrightarrow G \longrightarrow \R_{k'/k}(B') \longrightarrow H \longrightarrow 1
\]
spreads out to an exact sequence
\[
1 \longrightarrow \mathscr{G} \longrightarrow \R_{\mathcal{O}_{S'}/\mathcal{O}_S}(\mathscr{B}') \longrightarrow \mathscr{H} \longrightarrow 1
\]
with $\mathscr{B}'$ a finite flat commutative $\mathcal{O}_{S'}$-group scheme, and $\mathscr{H}$ 
a commutative flat affine $\mathcal{O}_S$-group of finite type. 
For each $v \notin S$, in the commutative diagram with exact rows
\[
\begin{tikzcd}
\R_{\mathcal{O}_{S'}/\mathcal{O}_S}(\mathscr{B}')(\mathcal{O}_v) \isoarrow{d} \arrow{r} & \mathscr{H}(\mathcal{O}_v) \arrow[d, hookrightarrow] \arrow{r} & {\rm{H}}^1(\mathcal{O}_v, \mathscr{G}) \arrow{d} \arrow{r} & {\rm{H}}^1(\mathcal{O}_v, \R_{\mathcal{O}_{S'}/\mathcal{O}_S}(\mathscr{B}')) \arrow[d, hookrightarrow] \\
\R_{k'/k}(B')(k_v) \arrow{r} & H(k_v) \arrow{r} & {\rm{H}}^1(k_v, G) \arrow{r} & {\rm{H}}^1(k_v, \R_{k'/k}(B'))
\end{tikzcd}
\]
the first vertical arrow is an isomorphism because $\mathscr{B}'(\mathcal{O}_{v'}) = B(k_{v'})$ for each $v'\mid v$
(as $\mathscr{B}'$ is finite), the second vertical arrow is an inclusion by the affineness of $\mathscr{H}$,
and the last vertical arrow is an inclusion by the settled case ``$G = \R_{k'/k}(B')$.'' 
The five lemma now shows that the third vertical arrow is an inclusion.
\end{proof}

\begin{lemma}
\label{G(O)=G(k)finiteweilrestriction}
Let $k'$ be a finite extension of the global field $k$, $B'$ a finite commutative $k'$-group scheme, and $G \subset \R_{k'/k}(B')$ a $k$-subgroup scheme. Let $\mathscr{G}$ be an $\mathcal{O}_S$-model for $G$ $($where, as usual, $S$ denotes a
non-empty finite set of places of $k$$)$. Then for all but finitely many $v$, the map $\mathscr{G}(\mathcal{O}_v) \rightarrow G(k_v)$ is an isomorphism.
\end{lemma}

\begin{proof}
We may spread $B'$ out to an $\calO_{S'}$-scheme $\mathscr{B}'$ such that $\mathscr{G}$ is the Zariski closure of $G$ inside $\R_{\calO_{S'}/\calO_S}(\mathscr{B}')$. Now $\R_{k'/k}(B')(k_v) = \prod_{v'\mid v} B'(k'_{v'}) = \prod_{v' \mid v} \mathscr{B}'(\calO_{v'}) = (\R_{\calO_{S'}/\calO_S}(\mathscr{B}'))(\calO_v)$, the penultimate equality by the valuative criterion of properness. Since $\mathscr{G}$ is the Zariski closure of $G$ inside $\R_{\calO_{S'}/\calO_S}(\mathscr{B}')$, any $\calO_v$-point of $\R_{\calO_{S'}/\calO_S}(\mathscr{B}')$ whose generic giber lies in $G$ lies entirely within $\mathscr{G}$. Since $\R_{k'/k}(B')(k_v) = (\R_{\calO_{S'}/\calO_S}(\mathscr{B}'))(\calO_v)$, the lemma follows.
\end{proof}

\begin{lemma}
\label{H^1(G)injectivedevissage}
For a global field $k$, consider an exact sequence of commutative $k$-group schemes of finite type
\[
1 \longrightarrow H \longrightarrow G \longrightarrow \R_{k'/k}(B')
\]
with $k'/k$ a finite extension and $B'$ a finite commutative 
$k'$-group scheme. If $(\ref{H^1(G)injectsproperty})$ holds for $H$, then it also holds for $G$.
\end{lemma}

\begin{proof}
Let $C : = \im(G \rightarrow \R_{k'/k}(B'))$. Then we have a short exact sequence
\[
1 \longrightarrow H \longrightarrow G \longrightarrow C \longrightarrow 1
\]
which we may spread out to a short exact sequence
\[
1 \longrightarrow \mathscr{H} \longrightarrow \mathscr{G} \longrightarrow \mathscr{C} \longrightarrow 1
\]
After enlarging $S$, the commutative diagram with exact rows (for $v \not\in S$) 
\[
\begin{tikzcd}
\mathscr{C}(\mathcal{O}_v) \isoarrow{d} \arrow{r} & {\rm{H}}^1(\mathcal{O}_v, \mathscr{H}) \arrow[d, hookrightarrow] \arrow{r} & {\rm{H}}^1(\mathcal{O}_v, \mathscr{G}) \arrow{d} \arrow{r} & {\rm{H}}^1(\mathcal{O}_v, \mathscr{C}) \arrow[d, hookrightarrow] \\
C(k_v) \arrow{r} & {\rm{H}}^1(k_v, H) \arrow{r} & {\rm{H}}^1(k_v, G) \arrow{r} & {\rm{H}}^1(k_v, C)
\end{tikzcd}
\]
has the first vertical arrow an isomorphism by Lemma \ref{G(O)=G(k)finiteweilrestriction}, the second an inclusion by the hypothesis, and the last an inclusion by Lemma \ref{H^1(G)injectweilrestrict}. The five lemma now shows that the third vertical arrow is an inclusion.
\end{proof}

Here is the crucial lemma that allows us to go beyond the smooth case.

\begin{lemma}
\label{maptoweilrestrictfinite}
Let $k$ be a field, $G$ a connected commutative $k$-group scheme of finite type. Suppose that 
the underlying reduced scheme 
$G_{\red} \subset G$ is not a smooth $k$-subgroup scheme. Then there is an exact sequence
\[
1 \longrightarrow H \longrightarrow G \longrightarrow \R_{k'/k}(I')
\]
for some finite purely inseparable extension $k'/k$ and some infinitesimal 
commutative $k'$-group scheme $I'$, such that $\dim(H) < \dim(G)$.
\end{lemma}

\begin{proof}
By descent from the perfect closure, there exists 
a finite purely inseparable extension $k'/k$ such that $(G_{k'})_{\red} \subset G_{k'}$ is a smooth $k'$-subgroup scheme.
Let $I' : = G_{k'}/(G_{k'})_{\red}$, so that $I'$ is infinitesimal.
Let $H$ be the kernel of the composition $G \hookrightarrow \R_{k'/k}(G_{k'}) \rightarrow \R_{k'/k}(I')$. We need to show that $\dim(H) < \dim(G)$. 

Suppose to the contrary that $\dim(H) = \dim(G)$. The composition $H_{k'} \hookrightarrow G_{k'} \rightarrow I'$
vanishes by the definition of $H$.
Thus, by the definition of $I'$, we have $H_{k'} \subset (G_{k'})_{\red}$. But $(G_{k'})_{\red}$ is smooth and connected, so its only closed $k'$-subgroup scheme of the same dimension is $(G_{k'})_{\red}$ itself. It follows that $H_{k'}$ is smooth, hence so is $H$. Therefore we must have $H = G_{\red}$ (since $G$ is connected and $\dim(H) = \dim(G)$), so $G_{\red}$ is a smooth $k$-subgroup scheme, violating our assumption. This contradiction shows that $\dim(H) < \dim(G)$.
\end{proof}

\begin{proposition}
\label{H^1(G)injective}
Let $G$ be a commutative group scheme of finite type over a global field $k$, $S$ a finite non-empty set of places of $k$, and let $\mathscr{G}$ be an $\calO_S$-model of $G$. Then for all but finitely many places $v$ of $k$, the map ${\rm{H}}^1(\calO_v, \mathscr{G}) \rightarrow {\rm{H}}^1(k_v, G)$ is injective.
\end{proposition}

\begin{proof}
We proceed by induction on $\dim(G)$. 
Since the component group $G/G^0$ is finite, by Lemma \ref{connectedreductionH^1injective} it suffices to prove Proposition \ref{H^1(G)injective} for $G^0$; i.e., we may assume that $G$ is connected. If $G_{\red} \subset G$ is a smooth $k$-subgroup scheme then, by finiteness of $G/G_{\red}$, we are done by Lemmas \ref{connectedreductionH^1injective} and \ref{H^1(G)injectivesmooth}. If $G_{\red}$ is not a smooth $k$-subgroup scheme, then we apply Lemmas \ref{maptoweilrestrictfinite} and \ref{H^1(G)injectivedevissage}, together with the induction hypothesis, to conclude.
\end{proof}

\begin{example}
\label{almostallnotall}
As we mentioned in the introduction to the present chapter, the results of this chapter really do apply only to almost all places of $k$, rather than to all places, even if one restricts to flat affine finite-type models of the generic fiber. Here we give an example to show that Proposition \ref{H^1(G)injective} fails if one replaces ``almost all'' by ``all.'' 

Let $v$ be a nonarchimedean place of the global field $k$, and let $\pi \in k^{\times}$ be a uniformizer at $v$. Let $S$ be a finite set of places of $k$ not including $v$. If $\pi$ is regular on $\calO_S$, then consider the smooth $\calO_S$-group scheme 
\[
\mathscr{G} := \{Y^p = Y + \pi X\} \subset \mathbf{G}_{a,\,\calO_S}^2.
\]
Then $\mathscr{G}$ is an $\calO_S$-model of $\Ga$. Indeed, the map $(X, Y) \mapsto Y$ defines a homomorphism $\mathscr{G} \rightarrow \Ga$ which over $k$ admits the inverse $T \mapsto (\pi^{-1}(T^p - T), T)$. In particular, ${\rm{H}}^1(k_v, \mathscr{G}) = {\rm{H}}^1(k_v, \Ga) = 0$, but we claim that ${\rm{H}}^1(\calO_v, \mathscr{G}) \neq 0$.

In order to see this, choose $\alpha \in \calO_v$ such that the image $\overline{\alpha}$ of $\alpha$ in the residue field $\kappa_v$ of $\calO_v$ is not in the image of the Artin-Schreier map. That is, there does not exist $t \in \kappa_v$ such that $t^p - t = \overline{\alpha}$. Consider the $\mathscr{G}$-torsor $\mathscr{X}$ over $\calO_v$ defined by the following equation:
\[
\mathscr{X} := \{T^p - T = \alpha + \pi S\} \subset \mathbf{G}_{a,\,\calO_v}^2.
\]
The action of $\mathscr{G}$ on $\mathscr{X}$ is given by $(X, Y)\cdot (S, T) := (S + X, T + Y)$. One may check that this does indeed make $\mathscr{X}$ into a $\mathscr{G}$-torsor over $\calO_v$. We claim that $\mathscr{X}(\calO_v) = \emptyset$. Indeed, if we have an integral point $(S, T) \in \mathscr{X}(\calO_v)$, then reducing modulo $\pi$ yields $\overline{T}^p - \overline{T} = \overline{\alpha}$, in violation of our choice of $\alpha$. Thus $\mathscr{X}$ is a nontrival $\mathscr{G}$-torsor over $\calO_v$.
\end{example}

\section{Integral annihilator aspects of duality between
${\rm{H}}^1(k_v, G)$ and ${\rm{H}}^1(k_v, \widehat{G})$}\label{sec126}

In this section we prove Theorem \ref{exactannihilatorH^1(G)}. In order to do this, we first prove Theorem \ref{exactannihilatorH^2(G^)} for $\Ga$: 

\begin{proposition}
\label{H^2(G^)integralGa}
For a local function field $k$, the map ${\rm{H}}^2(\mathcal{O}_v, \widehat{\Ga}) \rightarrow (\Ga(k_v)/\Ga(\calO_v))^D$ $= (k_v/\calO_v)^D$ is an isomorphism.

Equivalently, the map ${\rm{H}}^2(\mathcal{O}_v, \widehat{\Ga}) \rightarrow {\rm{H}}^2(k_v, \widehat{\Ga})$ is injective
with image that is the exact annihilator of $\Ga(\mathcal{O}_v)$ under the cup product pairing.
\end{proposition}

The equivalence between the two formulations follows from the already-established local duality in the case of $\Ga$ (Theorem \ref{H^2(G^)altori}) and from the equivalence of (algebraic, not topological) groups $((k_v)_{\pro})^D \xrightarrow{\sim} k_v^D$ coming from Proposition \ref{G(k)_pro*---->G(k)*isom}, which implies that $(k_v/\calO_v)^D \simeq ((k_v)_{\rm{pro}}/\calO_v)^D$. (Note that $k_v^D = {\rm{Hom}}_{\rm{cts}}(k_v, \Q/\Z)$ because $k_v$ is $p$-torsion.)

\begin{proof}
First we check that the map ${\rm{H}}^2(\calO_v, \widehat{\Ga}) \rightarrow {\rm{H}}^2(k_v, \widehat{\Ga})$ is injective. By Proposition \ref{H^2=Ext^2=Brdvr}, it suffices to check that the map $\Br(\mathbf{G}_{a,\,\calO_v}) \rightarrow \Br(\mathbf{G}_{a,\,k_v})$ is injective, and this in turn holds because $\mathbf{G}_{a,\,\calO_v}$ is a regular scheme \cite[Cor.\,1.8]{brii}.

Next we check that ${\rm{H}}^2(\calO_v, \widehat{\Ga})$ is the exact annihilator of $\calO_v$. By Propositions \ref{H^2=Ext^2=Brdvr} and \ref{omega=H^2(Ga^)}, and Lemmas  \ref{evaluationcompatibility} and \ref{Omega^1=Rdpi}, it suffices to prove that for a uniformizer $\pi$ of $\calO_v$, if $\lambda \in k_v$ satisfies $\psi(\lambda \alpha d\pi) = 0 \in \Br(k_v)$ for every $\alpha \in \calO_v$, where $\psi$ is the map in Proposition \ref{brauerdifferentialforms} for $X = \Spec(k_v)$, then $\lambda \in \calO_v$. That is, we must show that if $\lambda \notin \calO_v$, then $\psi(\lambda \alpha d\pi) \neq 0$ for some $\alpha \in \calO_v$.

Let $\F_v$ denote the finite residue field of $\calO_v$. Choose a $c \in \F_v$ such that ${\rm{Tr}}_{\F_v/\F_p}(c) \neq 0$. Write $\lambda = \sum_{n \geq -N} b_n\pi^n$ with $b_n \in \F_v$, $N > 0$, and $b_{-N} \neq 0$. Choose $\alpha : = c\pi^{N-1}/b_{-N}$, so that $\lambda \alpha d\pi$ has residue $\Res(\lambda \alpha d\pi) = c$. Therefore, identifying $\F_p$ with $\frac{1}{p}\Z/\Z$, this Brauer element has invariant $\inv(\lambda \alpha d\pi) = {\rm{Tr}}_{\F_v/\F_p}(\Res(\lambda \alpha d\pi)) \neq 0$ by construction. We therefore have $\psi(\lambda \alpha d\pi) \neq 0 \in \Br(k_v)$, as desired.
\end{proof}

We will prove the following reformulation of Theorem \ref{exactannihilatorH^1(G)}: 

\begin{theorem}
\label{H^1(k,G^)/H^1(O,G^)} $($Theorem $\ref{exactannihilatorH^1(G)}$$)$ For $k$ a global function field, $G$ an affine commutative $k$-group scheme of finite type, and $\mathscr{G}$ an $\mathcal{O}_S$-model of $G$, for all but finitely many $v$ the maps ${\rm{H}}^1(\calO_v, \mathscr{G}) \rightarrow {\rm{H}}^1(k_v, G)$ and ${\rm{H}}^1(\calO_v, \widehat{\mathscr{G}}) \rightarrow {\rm{H}}^1(k_v, \widehat{G})$ are inclusions and the map ${\rm{H}}^1(\mathcal{O}_v, \widehat{\mathscr{G}}) \rightarrow ({\rm{H}}^1(k_v, G)/{\rm{H}}^1(\mathcal{O}_v, \mathscr{G}))^D$ induced by the composition ${\rm{H}}^1(\mathcal{O}_v, \widehat{\mathscr{G}}) \rightarrow {\rm{H}}^1(k_v, \widehat{G}) \rightarrow {\rm{H}}^1(k_v, G)^D$ 
$($the last map being induced by cup product$)$ is a topological isomorphism for almost all $v$, where we endow ${\rm{H}}^1(\calO_v, \widehat{\mathscr{G}})$ with the subspace topology from ${\rm{H}}^1(k_v, \widehat{G})$.
\end{theorem}

\begin{remark}
\label{H^1(O_v,G)compact}
In order to show that Theorem \ref{H^1(k,G^)/H^1(O,G^)} is equivalent to Theorem \ref{exactannihilatorH^1(G)}, we note that it is equivalent to saying that ${\rm{H}}^1(\calO_v, \widehat{\mathscr{G}}) \subset {\rm{H}}^1(k_v, \widehat{G})$ is closed, and that its annihilator under the local duality pairing is the closure of ${\rm{H}}^1(\calO_v, \mathscr{G}) \subset {\rm{H}}^1(k_v, G)$. Thus, we only need to show that ${\rm{H}}^1(\calO_v, \mathscr{G}) \subset {\rm{H}}^1(k_v, G)$ is closed. In order to do this, we recall that we also defined a topology on ${\rm{H}}^1(\calO_v, \mathscr{G})$ in \S \ref{sectiontopologyoncohomology}, defined as in \cite{ces1}.

If we choose $\mathscr{G}$ to be affine (as we may by choosing an inclusion $G \subset {\rm{GL}}_{n,\,k}$ and then taking $\mathscr{G}$ to be the Zariski closure of $G$ inside ${\rm{GL}}_{n,\, \calO_S}$), then this makes ${\rm{H}}^1(\calO_v, \mathscr{G})$ into a topological group \cite[Prop.\,3.6(c)]{ces1}, and this group is compact by \cite[Prop.\,2.9(d)]{ces1} and Lemma \ref{H^1(finitefield)=finite}. With this topology, the map ${\rm{H}}^1(\calO_v, \mathscr{G}) \rightarrow {\rm{H}}^1(k_v, G)$ is continuous, due to the continuity of the map $X(\calO_v) \rightarrow X(k_v)$ for any locally finite type $\calO_v$-scheme $X$. It follows that the image is compact, hence closed, which is what we wanted to show. It is perhaps worth noting that this map is actually a homeomorphism onto its image, as follows from the compactness of the source and the Hausdorffness of the target.
\end{remark}

\begin{proof}[Proof of Theorem $\ref{H^1(k,G^)/H^1(O,G^)}$]
The injectivity assertion is contained in Propositions \ref{H^1(O,G^)-->H^1(k,G^)injective} and \ref{H^1(G)injective}. Further, if the map of the proposition is an algebraic isomorphism, then it is automatically a topological isomorphism since Pontryagin duality then topologically identifies ${\rm{H}}^1(\calO_v, \widehat{\mathscr{G}})$ with $({\rm{H}}^1(k_v, G)/{\rm{H}}^1(\calO_v, \mathscr{G}))^D$. If $G$ is finite, then when ${\rm{char}}(k) \nmid \# G$, this is part of classical Tate local duality \cite[Th.\,2.4]{tateduality}, and the general case is due to Milne \cite[Ch.\,3, Cor.\,7.2]{milne}. To go beyond finite $G$, we next note that 
Theorem \ref{H^1(k,G^)/H^1(O,G^)} holds for $G = \R_{k'/k}(\Gm)$ with $k'/k$ a finite separable extension because ${\rm{H}}^1(k_v, G)$ and ${\rm{H}}^1(k_v, \widehat{G})$ vanish for all $v$ by Lemma \ref{H^1=0Weilrestrictionsplittori}.

Next, we turn to the case when $G$ is an almost-torus over $k$. By Lemma \ref{almosttorus}(iv), after harmlessly modifying $G$ we may assume that there is an exact sequence
\[
1 \longrightarrow B \longrightarrow X \longrightarrow G \longrightarrow 1
\]
where $X = C \times \R_{k'/k}(T')$ for a finite separable extension field $k'/k$
and split $k'$-torus $T'$, and $B$ and $C$ are finite commutative $k$-group schemes. 
Using Lemma \ref{spreadingoutdualsheaves}, we may spread this out to obtain an exact sequence 
of affine flat commutative group schemes of finite type over some $\mathcal{O}_S$: 
\[
1 \longrightarrow \mathscr{B} \longrightarrow \mathscr{X} \longrightarrow \mathscr{G} \longrightarrow 1
\]
such that the dual sequence
\[
1 \longrightarrow \widehat{\mathscr{G}} \longrightarrow \widehat{\mathscr{X}} \longrightarrow \widehat{\mathscr{B}} \longrightarrow 
1
\]
is also exact. Define the covariant functor
$Q^i : = {\rm{H}}^i(k_v,(\cdot)_{k_v})/{\rm{H}}^i(\mathcal{O}_v, \cdot)$ on affine flat commutative $\mathcal{O}_v$-group
schemes of finite type. In the commutative diagram 
\begin{equation}\label{Qexact}
\begin{tikzcd}
\widehat{\mathscr{X}}(\mathcal{O}_v) \isoarrow{d} \arrow{r} & \widehat{\mathscr{B}}(\mathcal{O}_v) \isoarrow{d} \arrow{r} & {\rm{H}}^1(\mathcal{O}_v, \widehat{\mathscr{G}}) \arrow{d} \arrow{r} & {\rm{H}}^1(\mathcal{O}_v, \widehat{\mathscr{X}}) \isoarrow{d} \arrow{r} & {\rm{H}}^1(\mathcal{O}_v, \widehat{\mathscr{B}}) \isoarrow{d} \\
Q^2(\mathscr{X})^D \arrow{r} &  Q^2(\mathscr{B})^D \arrow{r} &  Q^1(\mathscr{G})^D \arrow{r} &
Q^1(\mathscr{X})^D \arrow{r} & Q^1(\mathscr{B})^D
\end{tikzcd}
\end{equation}
(where all maps in the bottom row are well-defined due to Proposition \ref{topcohombasics}(v)), the top row is exact, and perhaps after expanding $S$ the first two vertical arrows are isomorphisms by Theorem \ref{H^2(G)G^(k)integraldualityprop}, and 
the last two are isomorphisms by the already-treated cases of finite group schemes and separable Weil restrictions of split tori. We claim that the bottom row is exact at the second and third entries for all but finitely many $v$. (It is actually exact everywhere
for almost all $v$, but we will not need this.) Assuming this, a simple diagram chase shows that the middle vertical arrow is an isomorphism, as desired.

Let us first check exactness at the second entry along the bottom of (\ref{Qexact}) for almost all $v$.
Consider the 3-term complex of continuous maps 
\[
\frac{{\rm{H}}^1(k_v, G)}{{\rm{H}}^1(\mathcal{O}_v, \mathscr{G})} \longrightarrow \frac{{\rm{H}}^2(k_v, B)}{{\rm{H}}^2(\mathcal{O}_v, \mathscr{B})} \longrightarrow \frac{{\rm{H}}^2(k_v, X)}{{\rm{H}}^2(\mathcal{O}_v, \mathscr{X})},
\]
in which the last 2 terms are discrete (the $k_v$-cohomologies are discrete in degree 2).  This is exact for almost all $v$
because ${\rm{H}}^2(\mathcal{O}_v, \mathscr{B})$ and ${\rm{H}}^2(\calO_v, \mathscr{X})$ vanish
for all but finitely many $v$ by Theorem \ref{H^2(G)G^(k)integraldualityprop}. Thus, the map
$$Q^2(\mathscr{B})/\im Q^1(\mathscr{G}) \rightarrow Q^2(\mathscr{X})$$ between discrete groups is an inclusion, so we have an exact sequence of $\mathbf{R}/\mathbf{Z}$-duals of discrete groups
$$Q^2(\mathscr{X})^D \rightarrow Q^2(\mathscr{B})^D \rightarrow ({\rm{im}} \,Q^1(\mathscr{G}))^D.$$
This yields exactness at the second term along the bottom of (\ref{Qexact}).

To complete the proof of Theorem \ref{H^1(k,G^)/H^1(O,G^)} for almost-tori, it remains to check exactness
at the third entry along the bottom of (\ref{Qexact}).  It suffices to show that the complex of continuous maps
$$Q^1(\mathscr{X}) \rightarrow Q^1(\mathscr{G}) \rightarrow Q^2(\mathscr{B})$$
between locally compact Hausdorff groups with discrete third term is algebraically exact for almost all $v$. The algebraic exactness
is immediate because ${\rm{H}}^2(\calO_v, \mathscr{B}) = 0$ for almost
all $v$ by Theorem \ref{H^2(G)G^(k)integraldualityprop} again. This proves the proposition for almost-tori.

Now we treat the general case. Let $G$ be an affine commutative $k$-group scheme of finite type. By induction on the dimension of the unipotent radical of $(G_{\overline{k}})^0_{\red}$ (the 
$0$-dimensional case being the already-treated case of almost-tori), we may assume (by Lemma \ref{affinegroupstructurethm}) that we have an exact sequence
\[
1 \longrightarrow H \longrightarrow G \longrightarrow \Ga \longrightarrow 1
\]
such that Theorem \ref{H^1(k,G^)/H^1(O,G^)} holds for $H$. Using Lemma \ref{spreadingoutdualsheaves}, we may then spread this out to obtain exact sequences over some $\mathcal{O}_S$: 
\[
1 \longrightarrow \mathscr{H} \longrightarrow \mathscr{G} \longrightarrow \Ga \longrightarrow 1,
\]
\[
1 \longrightarrow \widehat{\Ga} \longrightarrow \widehat{\mathscr{G}} \longrightarrow \widehat{\mathscr{H}} \longrightarrow 1.
\]

In the commutative diagram of homomorphisms 
\[
\begin{tikzcd}
0 \arrow{r} & {\rm{H}}^1(\mathcal{O}_v, \widehat{\mathscr{G}}) \arrow{d} \arrow{r} & {\rm{H}}^1(\mathcal{O}_v, \widehat{\mathscr{H}}) \isoarrow{d} \arrow{r} & {\rm{H}}^2(\mathcal{O}_v, \widehat{\Ga}) \isoarrow{d} \\
0 \arrow{r} & Q^1(\mathscr{G})^D \arrow{r} & Q^1(\mathscr{H})^D \arrow{r} & Q^0(\Ga)^D
\end{tikzcd}
\]
(with well-defined maps in the bottom row due to Proposition \ref{topcohombasics}(v)), the top row is exact because ${\rm{H}}^1(\mathcal{O}_v, \widehat{\Ga}) = 0$
by Proposition \ref{cohomologyofG_adualgeneralk}, the middle vertical arrow is an isomorphism by hypothesis, and the last vertical arrow is an (algebraic) isomorphism by Proposition \ref{H^2(G^)integralGa}. The bottom row is exact at $Q^1(\mathscr{G})^D$ because the map ${\rm{H}}^1(k_v, H) \rightarrow {\rm{H}}^1(k_v, G)$ is surjective (since ${\rm{H}}^1(k_v, \Ga) = 0$). A simple diagram chase now shows that the first vertical arrow is an isomorphism. This completes the proof of Theorem \ref{H^1(k,G^)/H^1(O,G^)}.
\end{proof}

\section{Integral annihilator aspects of duality between ${\rm{H}}^2(k_v, \widehat{G})$ and ${\rm{H}}^0(k_v, G)$}

In this section we prove Theorem \ref{exactannihilatorH^2(G^)}. Let $G$ be an affine commutative group scheme of finite type over a global function field $k$. Consider the following statement: 
\begin{equation}
\label{H^2(O,G^)perfectlyannihilatesG(O)property}
\text{\parbox{.85\textwidth}{For any (equivalently, for some) finite non-empty set $S$ of places of $k$ and $\calO_S$-model $\mathscr{G}$ of $G$, the continuous map ${\rm{H}}^2(\mathcal{O}_v, \widehat{\mathscr{G}}) \rightarrow \Hom_{\rm{cts}}(G(k_v)/\mathscr{G}(\mathcal{O}_v), \Q/\Z)$ is an isomorphism for all but finitely many places $v$ of $k$.}}
\end{equation}

\begin{remark}
\label{H^2(O)dualityequivalence}
We claim that (\ref{H^2(O,G^)perfectlyannihilatesG(O)property}) is equivalent to Theorem \ref{exactannihilatorH^2(G^)} for $G$. Indeed, we first note that the injectivity of the map $\mathscr{G}(\mathcal{O}_v) \rightarrow G(k_v)$ for all but finitely many $v$ is trivial, because we may spread $G$ out to some affine $\mathcal{O}_S$-model $\mathscr{G}$. The injectivity of the map ${\rm{H}}^2(\calO_v, \widehat{\mathscr{G}}) \rightarrow {\rm{H}}^2(k_v, \widehat{G})$ for all but finitely many $v$ follows from (\ref{H^2(O,G^)perfectlyannihilatesG(O)property}) and the fact that the map ${\rm{H}}^2(\mathcal{O}_v, \widehat{\mathscr{G}}) \rightarrow \Hom_{\rm{cts}}(G(k_v)/\mathscr{G}(\mathcal{O}_v), \Q/\Z)$ factors as a composition ${\rm{H}}^2(\mathcal{O}_v, \widehat{\mathscr{G}}) \rightarrow {\rm{H}}^2(k_v, \widehat{G}) \rightarrow \Hom_{\rm{cts}}(G(k_v)/\mathscr{G}(\mathcal{O}_v), \Q/\Z)$. Finally, the equivalence of Theorem \ref{exactannihilatorH^2(G^)} with (\ref{H^2(O,G^)perfectlyannihilatesG(O)property}) follows from the perfectness of the local duality pairing (Theorem \ref{H^2(G^)altori}) and from the fact that the map $(G(k_v)_{\pro})^D \rightarrow \Hom_{\rm{cts}}(G(k_v), \Q/\Z)$ is an (algebraic) isomorphism (Proposition \ref{G(k)_pro*---->G(k)*isom}).
\end{remark}

\begin{lemma}
\label{H^2(O,G^)Gm}
For a global function field $k$ a finite separable extension $k'/k$, and a split $k'$-torus $T'$, 
$(\ref{H^2(O,G^)perfectlyannihilatesG(O)property})$ holds for $G = \R_{k'/k}(T')$.
\end{lemma}

\begin{proof}
We may assume that $T' = \Gm$ and that $S$ contains the set of places of $k$ unramified in $k'$, 
and that $\mathscr{G} : = \R_{\mathcal{O}_{S'}/\mathcal{O}_S}(\Gm)$, where $S'$ is the set of places of $k'$ lying over $S$. By Proposition \ref{charactersseparableweilrestriction}, $\widehat{\R_{\mathcal{O}_{S'}/\mathcal{O}_S}(\Gm)} = \R_{\mathcal{O}_{S'}/\mathcal{O}_S}(\Z)$. Since 
the constant group $\Z$ is smooth, we may consider cohomology 
relative to the {\'e}tale topology. For each $v \notin S$, we have, by Proposition \ref{diagramcommuteslocal}, a commutative diagram with vertical isomorphisms
\[
\begin{tikzcd}
{\rm{H}}^2(\mathcal{O}_v, \R_{\mathcal{O}_{S'}/\mathcal{O}_S}(\Z)) \arrow{r} \isoarrow{d} & \Hom_{{\rm{cts}}}(\R_{\mathcal{O}_{S'}/\mathcal{O}_S}(\Gm)(k_v)/\R_{\mathcal{O}_{S'}/\mathcal{O}_S}(\Gm)(\mathcal{O}_v), \Q/\Z) \isoarrow{d} \\
\prod_{v' \mid v} {\rm{H}}^2(\mathcal{O}_{v'}, \Z) \arrow{r} & \prod_{v'\mid v} \Hom_{\rm{cts}}(k^{\times}_{v'}/\mathcal{O}^{\times}_{v'}, \Q/\Z)
\end{tikzcd}
\]
in which the products are taken over all places $v'$ of $k'$ lying above $v$. We are therefore reduced to the case $k' = k$. 

Our task now is to show that the map ${\rm{H}}^2(\mathcal{O}_v, \Z) \rightarrow (k^{\times}_v/\mathcal{O}^{\times}_v)^*$ is an isomorphism. Via the exact sequence of {\'e}tale sheaves
\[
0 \longrightarrow \Z \longrightarrow \Q \longrightarrow \Q/\Z \longrightarrow 0
\]
and the fact that ${\rm{H}}^i(\mathcal{O}_v, \Q) = 0$ for all $i > 0$ (this actually holds for any noetherian normal scheme), we see that ${\rm{H}}^2(\mathcal{O}_v, \Z) = \Hom_{{\rm{cts}}}(\pi_1(\mathcal{O}_v), \Q/\Z) = (\Gal(k_v^{\rm{nr}}/k_v))^D$, where $k_v^{\rm{nr}}$ is the maximal unramified extension of $k_v$. Further, the map $(\Gal(k_v^{\rm{nr}}/k_v))^D \rightarrow \Hom_{\rm{cts}}(k^{\times}_v/\mathcal{O}^{\times}_v, \Q/\Z)$ is compatible with the map $\Gal(k_v^{{\rm{ab}}}/k_v)^D \rightarrow \Hom_{\rm{cts}}(k_v^{\times}, \Q/\Z)$
that is the Pontryagin dual of the local reciprocity map (see Lemma \ref{localreciprocitymap}). What we want to show, therefore, is that the profinite group $\mathcal{O}^{\times}_v$ (which injects into $(k_v)_{\pro} \simeq \calO_v^{\times} \times \Z_{\pro}$) is the kernel of the composition $k^{\times}_v \rightarrow \Gal(k_v^{{\rm{ab}}}/k_v) \rightarrow \Gal(k_v^{\rm{nr}}/k_v)$, and this follows from local class field theory.
\end{proof}

\begin{theorem}
\label{H^2(O,G^)dualityprop} $($Theorem $\ref{exactannihilatorH^2(G^)}$$)$ Assertion $(\ref{H^2(O,G^)perfectlyannihilatesG(O)property})$ holds for any affine commutative group scheme $G$ of finite type over a global function field $k$.
\end{theorem}

\begin{proof}
The equivalence between (\ref{H^2(O,G^)perfectlyannihilatesG(O)property}) and Theorem \ref{exactannihilatorH^2(G^)} was discussed in Remark \ref{H^2(O)dualityequivalence}. Note also that Theorem \ref{exactannihilatorH^2(G^)} holds for finite $G$ by Theorem \ref{H^2(G)G^(k)integraldualityprop} and Cartier duality.

First, suppose that $G$ is an almost-torus over $k$. To prove (\ref{H^2(O,G^)perfectlyannihilatesG(O)property}) for $G$, by Lemma \ref{almosttorus}(iv) we may harmlessly modify $G$ so that there is an exact sequence
\[
1 \longrightarrow B \longrightarrow X \longrightarrow G \longrightarrow 1,
\]
where $B$ is a finite commutative $k$-group scheme and $X = C \times \R_{k'/k}(T')$ for a finite 
commutative $k$-group scheme $C$, a finite separable extension $k'/k$, and a split $k'$-torus $T'$. By Lemma \ref{spreadingoutdualsheaves}, this spreads out to yield exact sequences over some $\mathcal{O}_S$: 
\[
1 \longrightarrow \mathscr{B} \longrightarrow \mathscr{X} \longrightarrow \mathscr{G} \longrightarrow 1,
\]
\[
1 \longrightarrow \widehat{\mathscr{G}} \longrightarrow \widehat{\mathscr{X}} \longrightarrow \widehat{\mathscr{B}} \longrightarrow 
1.
\]
Recall that $Q^i(\mathscr{H}) : = {\rm{H}}^i(k_v, H)/{\rm{H}}^i(\calO_v, \mathscr{H})$ for a commutative $\calO_v$-group scheme $\mathscr{H}$ with generic fiber $H$. After possibly enlarging $S$, we can arrange that in the resulting commutative diagram (for $v \not\in S$)
\begin{equation}\label{lastQ}
\begin{tikzcd}[column sep = tiny]
{\rm{H}}^1(\mathcal{O}_v, \widehat{\mathscr{X}}) \isoarrow{d} \arrow{r} & {\rm{H}}^1(\mathcal{O}_v, \widehat{\mathscr{B}}) \isoarrow{d} \arrow{r} & {\rm{H}}^2(\mathcal{O}_v, \widehat{\mathscr{G}}) \arrow{d} \arrow{r} & {\rm{H}}^2(\mathcal{O}_v, \widehat{\mathscr{X}}) \isoarrow{d} \arrow{r} & 0 \\
Q^1(\mathscr{X})^D \arrow{r} & Q^1(\mathscr{B})^D \arrow{r} & \Hom_{\rm{cts}}(Q^0(\mathscr{G}), \Q/\Z) \arrow{r} & \Hom_{\rm{cts}}(Q^0(\mathscr{X}), \Q/\Z)
\end{tikzcd}
\end{equation} 
the top row is exact (because ${\rm{H}}^2(\mathcal{O}_v, \widehat{\mathscr{B}}) = 0$ for almost all $v$ by Theorem \ref{H^2(G)G^(k)integraldualityprop} and Cartier duality), the first two vertical arrows are isomorphisms 
(by Theorem \ref{H^1(k,G^)/H^1(O,G^)}), the last vertical arrow is an isomorphism by the already-treated cases of finite group schemes and separable Weil restrictions of split tori (Lemma \ref{H^2(O,G^)Gm}), and the maps in the bottom row are well-defined due to Proposition \ref{topcohombasics}(v) and Lemma \ref{H^1finiteexponent}. We claim that the bottom row is exact for all but finitely many $v \not\in S$. Assuming this, the five lemma shows that the third vertical arrow is an isomorphism, as required.

So we now prove exactness of the bottom row of (\ref{lastQ}). We first claim that (for all but finitely many $v$) the complex
\begin{equation}
\label{H^2(G^)integralsequence1}
\frac{X(k_v)}{\mathscr{X}(\mathcal{O}_v)} \longrightarrow \frac{G(k_v)}{\mathscr{G}(\mathcal{O}_v)} \longrightarrow \frac{{\rm{H}}^1(k_v, B)}{{\rm{H}}^1(\mathcal{O}_v, \mathscr{B})} \longrightarrow \frac{{\rm{H}}^1(k_v, X)}{{\rm{H}}^1(\mathcal{O}_v, \mathscr{X})} \longrightarrow \frac{{\rm{H}}^1(k_v, G)}{{\rm{H}}^1(\mathcal{O}_v, \mathscr{G})}
\end{equation}
is exact. Indeed, in the commutative diagram with exact rows
\[
\begin{tikzcd}
\mathscr{X}(\calO_v) \arrow{r} \arrow{d} & \mathscr{G}(\mathcal{O}_v) \arrow{d} \arrow{r} & {\rm{H}}^1(\mathcal{O}_v, \mathscr{B}) \arrow{d} \arrow{r} & {\rm{H}}^1(\mathcal{O}_v, \mathscr{X}) \arrow[d, hookrightarrow] \arrow{r} & {\rm{H}}^1(\mathcal{O}_v, \mathscr{G}) \arrow[d, hookrightarrow] \arrow{r} & 0 \\
X(k_v) \arrow{r} & G(k_v) \arrow{r} & {\rm{H}}^1(k_v, B) \arrow{r} & {\rm{H}}^1(k_v, X) \arrow{r} & {\rm{H}}^1(k_v, G) &
\end{tikzcd}
\]
the last two vertical arrows are inclusions for almost all $v$ by Proposition \ref{H^1(G)injective}, and the top row is exact because ${\rm{H}}^2(\calO_v, \mathscr{B}) = 0$ for all but finitely many $v$ by Theorem \ref{H^2(G)G^(k)integraldualityprop}. A simple diagram chase now shows that (\ref{H^2(G^)integralsequence1}) is exact.

We may now prove that the bottom row of (\ref{lastQ}) is exact. The groups $Q^1(\mathscr{F})$ for $\mathscr{F} = \mathscr{B}$, $\mathscr{X}$, or $\mathscr{G}$ are Hausdorff since ${\rm{H}}^1(\calO_v, \mathscr{F}) \subset {\rm{H}}^1(k_v, F)$ are closed by Theorem \ref{H^1(k,G^)/H^1(O,G^)} (the exact annihilator of any set is closed). Thus, the exactness of (\ref{H^2(G^)integralsequence1}) implies that the maps $Q^1(\mathscr{B})/Q^0(\mathscr{G}) \rightarrow Q^1(\mathscr{X})$ and $Q^0(\mathscr{G})/Q^0(\mathscr{X}) \rightarrow Q^1(\mathscr{B})$ are inclusions with closed images, hence these inclusions are homeomorphisms onto closed subgroups by Propositions \ref{topcohombasics}(i) and \ref{secondcountable}, and Lemma \ref{isom=homeo}. Pontryagin duality then yields the exactness of the version of the bottom row of (\ref{lastQ}) in which one takes Pontryagin duals everywhere rather than $\Hom_{\rm{cts}}(\cdot, \Q/\Z)$. To see why we may replace $(\cdot)^D$ with $\Hom_{\rm{cts}}(\cdot, \Q/\Z)$ in the indicated places, it suffices to note that the map $Q^1(\mathscr{B})^D \rightarrow Q^0(\mathscr{G})^D$ lands inside the subgroup $\Hom_{\rm{cts}}(Q^0(\mathscr{G}), \Q/\Z)$. This completes the proof of the theorem for almost-tori.

Now we turn to the proof of the theorem in general. Let $G$ be an affine commutative $k$-group scheme of finite type. By induction on the dimension of the unipotent radical of $(G_{\overline{k}})^0_{\red}$ (the 0-dimensional case corresponding to 
the settled case of almost-tori), we may assume, by Lemma \ref{affinegroupstructurethm}, that there is an exact sequence
\[
1 \longrightarrow H \longrightarrow G \longrightarrow \Ga \longrightarrow 1
\]
such that the proposition holds for $H$. By Lemma \ref{spreadingoutdualsheaves}, we may then spread this out to obtain exact sequences over some $\mathcal{O}_S$: 
\[
1 \longrightarrow \mathscr{H} \longrightarrow \mathscr{G} \longrightarrow \Ga \longrightarrow 1,
\]
\[
1 \longrightarrow \widehat{\Ga} \longrightarrow \widehat{\mathscr{G}} \longrightarrow \widehat{\mathscr{H}} \longrightarrow 1.
\]
In the resulting commutative diagram
\begin{equation}
\label{lastQ2}
\begin{tikzcd}[column sep = tiny]
{\rm{H}}^1(\mathcal{O}_v, \widehat{\mathscr{H}}) \isoarrow{d} \arrow{r} & {\rm{H}}^2(\mathcal{O}_v, \widehat{\Ga}) \isoarrow{d} \arrow{r} & {\rm{H}}^2(\mathcal{O}_v, \widehat{\mathscr{G}}) \arrow{d} \arrow{r} & {\rm{H}}^2(\mathcal{O}_v, \widehat{\mathscr{H}}) \arrow{r} \isoarrow{d} & 0 \\
Q^1(\mathscr{H})^D \arrow{r} & Q^0(\Ga)^D \arrow{r} & \Hom_{\rm{cts}}(Q^0(\mathscr{G}), \Q/\Z) \arrow{r} & \Hom_{\rm{cts}}(Q^0(\mathscr{H}), \Q/\Z) &
\end{tikzcd}
\end{equation}
the top row is exact because ${\rm{H}}^3(\mathcal{O}_v, \widehat{\Ga}) = 0$ (Proposition \ref{H^3(G_a^)=0}),
the first vertical arrow is an isomorphism for almost all $v$ by Theorem \ref{H^1(k,G^)/H^1(O,G^)},
the second vertical arrow is an isomorphism for all $v$ by Proposition \ref{H^2(G^)integralGa},
and the last vertical arrow is an isomorphism for almost all $v$ by hypothesis. The first and second maps in the bottom row are well-defined by Proposition \ref{topcohombasics}(v) and because $Q^0(\Ga)$ is $p$-torsion. We claim that the bottom row is exact
for all but finitely many $v$. Assuming this, the five lemma shows that the third vertical arrow is an isomorphism for such $v$, which would prove the proposition.

We first claim that for almost all $v$ the following complex is exact: 
\begin{equation}
\label{H^2(G^)integraldualitysequence2}
\frac{H(k_v)}{\mathscr{H}(\mathcal{O}_v)} \longrightarrow \frac{G(k_v)}{\mathscr{G}(\mathcal{O}_v)} \longrightarrow \frac{\Ga(k_v)}{\Ga(\mathcal{O}_v)} \longrightarrow \frac{{\rm{H}}^1(k_v, H)}{{\rm{H}}^1(\mathcal{O}_v, \mathscr{H})} \longrightarrow \frac{{\rm{H}}^1(k_v, G)}{{\rm{H}}^1(\mathcal{O}_v, \mathscr{G})}.
\end{equation}
To see this, consider the following commutative diagram: 
\[
\begin{tikzcd}
\mathscr{H}(\calO_v) \arrow{r} \arrow{d} & \mathscr{G}(\mathcal{O}_v) \arrow{d} \arrow{r} & \Ga(\mathcal{O}_v) \arrow{r} \arrow{d} & {\rm{H}}^1(\mathcal{O}_v, \mathscr{H}) \arrow[d, hookrightarrow] \arrow{r} & {\rm{H}}^1(\mathcal{O}_v, \mathscr{G}) \arrow[d, hookrightarrow] \arrow{r} & 0 \\
H(k_v) \arrow{r} & G(k_v) \arrow{r} & \Ga(k_v) \arrow{r} & {\rm{H}}^1(k_v, H) \arrow{r} & {\rm{H}}^1(k_v, G) &
\end{tikzcd}
\]
The top row is exact because ${\rm{H}}^1(\calO_v, \Ga) = 0$, and the last two vertical arrows are inclusions by Proposition \ref{H^1(G)injective}. A simple diagram chase thereby yields that (\ref{H^2(G^)integraldualitysequence2}) is exact.

To show that the bottom row of (\ref{lastQ2}) is exact, it suffices to show that the corresponding sequence with all $\Hom_{\rm{cts}}(\cdot, \Q/\Z)$ terms replaced by $(\cdot)^D$ is exact. In order to show this, in turn, it suffices by Pontryagin duality to show that the maps $Q^0(\Ga)/Q^0(\mathscr{G}) \rightarrow Q^1(\mathscr{H})$ and $Q^0(\mathscr{G})/Q^0(\mathscr{H}) \rightarrow Q^0(\Ga)$ are homeomorphisms onto closed subgroups. Exactness of the sequence (\ref{H^2(G^)integraldualitysequence2}) of {\em continuous} (by Proposition \ref{topcohombasics}(v)) maps and Hausdorffness of $Q^1(\mathscr{H})$ and $Q^1(\mathscr{G})$ (which follows from the fact that the integral cohomology groups ${\rm{H}}^1(\calO_v, \mathscr{H}) \subset {\rm{H}}^1(k_v, H)$ and ${\rm{H}}^1(\calO_v, \mathscr{G}) \subset {\rm{H}}^1(k_v, G)$ are closed, since they are the exact annihilators of subsets of the Pontryagin dual group by Theorem \ref{H^1(k,G^)/H^1(O,G^)}) imply that these maps are isomorphisms onto closed subgroups. They are therefore homeomorphisms onto their closed images by Lemma \ref{isom=homeo} and Proposition \ref{topcohombasics}(i). This completes the proof that the bottom row of (\ref{lastQ2}) is exact, and of the proposition.
\end{proof}

\chapter{Global Fields}\label{globalfieldschap}

In this chapter, we establish the main global theorems stated in \S\ref{intro}. That is, we prove Theorems \ref{poitoutatesequence}, \ref{shapairing}, and \ref{poitoutatesequencedual}. We begin by verifying the easy parts of Theorem \ref{poitoutatesequence} in \S\ref{globalsectionpreliminaries}, and we then describe the relationship between the cohomology of $\A_k$ and that of the fields $k_v$, both algebraically and topologically (\S\S \ref{sectionadeliclocalcohom}--\ref{sectionadeliccohomtopology}). Among other things, this explains the relationship between our results stated in terms of adelic cohomology and the classical results of Poitou and Tate stated in terms of various products and restricted products of local cohomology groups. We will also use this discussion to prove duality theorems for the adelic cohomology groups (Proposition \ref{H^1(A,G)=H^1(A,G^)^D}), which are consequences of the main local duality results. The first half of the chapter (as well as \S \ref{sectionexactnessatvariousH^1's}) proves the exactness of the nine-term exact sequence in Theorem \ref{poitoutatesequence} at various places. The second half is primarily concerned with the pairings between Tate-Shafarevich groups given by Theorem \ref{shapairing}. Unlike the rest of the results in this manuscript, these results are proven from scratch, without using the case of finite commutative group schemes as a black box. This is done for two reasons:  (i) it is convenient for the reader; and (ii) it allows us to avoid checking compatibility between our pairings and the pairings given in \cite{ces2} (defined in a totally different manner). Finally, the chapter concludes with a proof of the ``dual nine-term exact sequence'' (\S \ref{sectiondual9termsequence}), obtained by dualizing the sequence of Theorem \ref{poitoutatesequence} and then applying the duality results for adelic cohomology groups mentioned above.

\section{Preliminaries}
\label{globalsectionpreliminaries}

In this section we describe all of the maps in Theorem \ref{poitoutatesequence}, show that they are well defined, and prove the easy parts of that theorem. Let us first recall the sequence of Theorem \ref{poitoutatesequence}: 
\[
\begin{tikzcd}
0 \arrow{r} & {\rm{H}}^0(k, G)_{\pro} \arrow{r} & {\rm{H}}^0(\A, G)_{\pro} \arrow{r} \arrow[d, phantom, ""{coordinate, name=Z_1}] & {\rm{H}}^2(k, \widehat{G})^* \arrow[dll, rounded corners,
to path={ -- ([xshift=2ex]\tikztostart.east)
|- (Z_1) [near end]\tikztonodes
-| ([xshift=-2ex]\tikztotarget.west) -- (\tikztotarget)}] & \\
& {\rm{H}}^1(k, G) \arrow{r} & {\rm{H}}^1(\A, G) \arrow{r} \arrow[d, phantom, ""{coordinate, name=Z_2}] & {\rm{H}}^1(k, \widehat{G})^* \arrow[dll, rounded corners,
to path={ -- ([xshift=2ex]\tikztostart.east)
|- (Z_2) [near end]\tikztonodes
-| ([xshift=-2ex]\tikztotarget.west) -- (\tikztotarget)}] & \\
& {\rm{H}}^2(k, G) \arrow{r} & {\rm{H}}^2(\A, G) \arrow{r} & {\rm{H}}^0(k, \widehat{G})^* \arrow{r} & 0
\end{tikzcd}
\]

Recall that we endow $G(k)$ with the discrete topology and $G(\A)$ with the natural topology
arising from that on $\A$ (via a closed embedding from $G$ into some affine space).  This latter topology coincides with the restricted product topology defined as follows.  We spread out $G$ to an
 $\calO_S$-model $\mathscr{G}$ (where $S$ is a non-empty finite set of places of $k$), and declare a fundamental system of neighborhoods of the identity to be the sets of the form $\prod_{v \in S'} U_v \times \prod_{v \notin S'} \mathscr{G}(\calO_v)$, where $S' \supset S$ is a finite set of places and $U_v \subset G(k_v)$ is a neighborhood of $0 \in G(k_v)$. This topology is independent of the chosen model, since any two become isomorphic over some $\calO_{S''}$. The natural map $G(k) \rightarrow G(\A)$ induced by the diagonal inclusion $k \hookrightarrow \A$ is trivially continuous, and therefore uniquely extends to a continuous map $G(k)_{\pro} \rightarrow G(\A)_{\pro}$. The maps ${\rm{H}}^1(k, G) \rightarrow {\rm{H}}^1(\A, G)$ and ${\rm{H}}^2(k, G) \rightarrow {\rm{H}}^2(\A, G)$ are also induced by the diagonal inclusion $k \hookrightarrow \A$. 

The maps ${\rm{H}}^i(\A, G) \rightarrow {\rm{H}}^{2-i}(k, \widehat{G})^*$ are defined by the pairing $${\rm{H}}^i(\A, G) \times {\rm{H}}^{2-i}(k, \widehat{G}) \rightarrow \Q/\Z$$ which for every place $v$ of $k$ cups the corresponding elements of ${\rm{H}}^i(k_v, G)$ and ${\rm{H}}^{2-i}(k_v, \widehat{G})$ to obtain an element of ${\rm{H}}^2(k_v, \Gm)$, takes the invariant, and then sums the result over all places $v$. We need to check that this sum contains only finitely many nonzero terms. In order to do this, we spread out $G$ to an $\calO_S$-model $\mathscr{G}$ for some non-empty finite set $S$ of places of $k$. Since $\A = \varinjlim_{S'} \left(\prod_{v \in S'} k_v \times \prod_{v \notin S'} \calO_v \right)$, where the limit is over all finite sets $S'$ of places of $k$ that contain $S$, Proposition \ref{directlimitscohom} shows that $${\rm{H}}^i(\A, G) = \varinjlim_{S'} \left(\prod_{v \in S'} {\rm{H}}^i(k_v, G) \times {\rm{H}}^i(\prod_{v \notin S'} \calO_v, \mathscr{G})\right).$$ We similarly obtain the equality ${\rm{H}}^{2-i}(k, \widehat{G}) = \varinjlim_{S'} \left({\rm{H}}^{2-i}(\calO_{S'}, \widehat{\mathscr{G}})\right)$.

We conclude that any element of ${\rm{H}}^i(\A, G)$ lands inside the image of  ${\rm{H}}^i(\calO_v, \mathscr{G}) \rightarrow {\rm{H}}^i(k_v, G)$ for all but finitely many places $v$ of $k$, and any element of ${\rm{H}}^{2-i}(k, \widehat{G})$ comes from ${\rm{H}}^{2-i}(\calO_{S'}, \widehat{\mathscr{G}})$ for sufficiently large $S' \supset S$ and so
restricts into the image of ${\rm{H}}^{2-i}(\calO_v, \widehat{\mathscr{G}}) \rightarrow {\rm{H}}^{2-i}(k_v, \widehat{G})$ for 
all but finitely many $v$. For any $\alpha \in {\rm{H}}^i(\A, G)$ and $\beta \in {\rm{H}}^{2-i}(k, \widehat{G})$, therefore, $\langle \alpha_v, \beta_v \rangle \in {\rm{im}}({\rm{H}}^2(\calO_v, \Gm) \rightarrow {\rm{H}}^2(k_v, \Gm))$
for all but finitely many $v$. But ${\rm{H}}^2(\calO_v, \Gm) = {\rm{Br}}(\calO_v) = 0$, so the local pairing is $0$ for all but finitely many places, as desired.

To extend the map $G(\A) \rightarrow {\rm{H}}^2(k, \widehat{G})^*$ to a map from $G(\A)_{\pro}$, endow ${\rm{H}}^2(k, \widehat{G})$ with the discrete topology. Since ${\rm{H}}^2(k, \widehat{G})$ is torsion (Lemma \ref{H^2(G^)istorsion}), ${\rm{H}}^2(k, \widehat{G})^*$ is the Pontryagin dual of the discrete group ${\rm{H}}^2(k, \widehat{G})$, and 
so it is naturally a profinite group. We next show that the map $G(\A) \rightarrow {\rm{H}}^2(k, \widehat{G})^*$ is continuous, hence 
it uniquely extends to a continuous map $G(\A)_{\pro} \rightarrow {\rm{H}}^2(k, \widehat{G})^*$.

\begin{lemma}
\label{G(A)toH^2(k,G^)*cts}
For $G$ a commutative group scheme of finite type over a global field $k$, the global duality map $G(\A) \rightarrow {\rm{H}}^2(k, \widehat{G})^*$, obtained from the pairing that cups everywhere locally and then adds the invariants, is continuous. Here, the topology on ${\rm{H}}^2(k, \widehat{G})^*$ is the compact-open topology.
\end{lemma}

\begin{proof}
Recall that the topology on the Pontryagin dual is the compact-open topology. Therefore, we need to show that for any finite subset $T \subset {\rm{H}}^2(k, \widehat{G})$ and any $\epsilon > 0$, there exists a neighborhood $U \subset G(\A)$ of the identity such that $|\langle u, t \rangle| < \epsilon$ for all $u \in U, t \in T$. (Here, $|\cdot|$ denotes the distance function on $\Q/\Z$, defined to be the distance of any lift to the nearest integer.) It suffices to treat the case in which $\# T = 1$; say, $T = \{t\}$. Choose some $\calO_S$-model $\mathscr{G}$ of $G$, and enlarge $S$ if necessary so that $t$ extends to a class in ${\rm{H}}^2(\calO_S, \widehat{\mathscr{G}})$. By continuity of the local duality pairing (Lemma \ref{G(k)H^2(G^)continuous}), we may choose for each $v \in S$ a neighborhood $U_v \subset G(k_v)$ of the identity such that $\langle U_v, t\rangle \subset (-\epsilon/\#S + \Z, \epsilon/\#S + \Z)$. Now we may take $U : = \prod_{v \in S} U_v \times \prod_{v \notin S} \mathscr{G}(\calO_v)$. Indeed, given $u \in U$, we have $|\langle u_v, t_v \rangle| < \epsilon /\#S$ for $v \in S$ by our choice of $U$ (where $|\cdot|$ denotes the distance function on $\Q/\Z$), while for $v \notin S$, the cup product $\langle u_v, t_v \rangle$ factors through ${\rm{H}}^2(\calO_v, \Gm) = 0$, hence vanishes. Thus, $|\langle u, t\rangle| = |\sum_v \langle u_v, t_v \rangle| < \epsilon$, as desired.
\end{proof}
 
 Finally, recall that the maps ${\rm{H}}^2(k, \widehat{G})^* \rightarrow {\rm{H}}^1(k, G)$ and ${\rm{H}}^1(k, \widehat{G})^* \rightarrow {\rm{H}}^2(k, G)$ in Theorem \ref{poitoutatesequence} are defined by using the (still unproven) Theorem \ref{shapairing}. Indeed, that theorem provides isomorphisms $\Sha^i(k, \widehat{G})^* \xrightarrow{\sim} \Sha^{3-i}(k, G)$, for $i = 1, 2$, and the above maps are then defined to be the compositions
 \[
 {\rm{H}}^2(k, \widehat{G})^* \twoheadrightarrow \Sha^2(k, \widehat{G})^* \xrightarrow{\sim} \Sha^1(k, G) \hookrightarrow {\rm{H}}^1(k, G),
 \]
 \[
 {\rm{H}}^1(k, \widehat{G})^* \twoheadrightarrow \Sha^1(k, \widehat{G})^* \xrightarrow{\sim} \Sha^2(k, G) \hookrightarrow {\rm{H}}^2(k, G).
 \]
 
Some easy parts of the proof of Theorem \ref{poitoutatesequence} can be settled immediately. First we check that the sequence is a complex. That it is a complex at ${\rm{H}}^1(k, G)$ and ${\rm{H}}^2(k, G)$ is immediate from the definitions. That the sequence is a complex at $G(\A)$, ${\rm{H}}^1(\A, G)$, and ${\rm{H}}^2(\A, G)$ follows from the fact that the sum of the local invariants of a class in ${\rm{H}}^2(k, \Gm)$ is $0$. It is a complex at ${\rm{H}}^1(k, \widehat{G})^*$ 
due to the fact that the pairing between ${\rm{H}}^1(\A, G)$ and $\Sha^1(k, \widehat{G})$ is trivial (since by definition any element of $\Sha^1(k, \widehat{G})$ has trivial image in ${\rm{H}}^1(k_v, \widehat{G})$ for each $v$). Analogously, the sequence is a complex at ${\rm{H}}^2(k, \widehat{G})^*$. Finally, the sequence is exact at ${\rm{H}}^i(k, G)$ ($i = 1, 2$) by definition (though recall that the definition of the map into ${\rm{H}}^i(k, G)$ depends on Theorem \ref{shapairing}). Proving that the rest of the sequence is exact will be the main work of this chapter.

\section{Relation between adelic and local cohomology}
\label{sectionadeliclocalcohom}

Let $k$ be a global function field with ring of adeles $\A$. Let $\mathscr{F}$ be an fppf abelian sheaf on $\Spec(\A)$. Then the projection maps $\A \rightarrow k_v$ induce maps ${\rm{H}}^i(\A, \mathscr{F}) \rightarrow \prod_v {\rm{H}}^i(k_v, \mathscr{F})$. The purpose of this section is to study these maps. \v{C}esnavicius has shown that when $\mathscr{F} = G$ for a commutative $k$-group scheme $G$ of finite type, these maps identify ${\rm{H}}^i(\A, G)$ with $\oplus_v {\rm{H}}^i(k_v, G)$ when $i > 1$ (and even when $i > 0$ if $G$ is smooth and connected), and with the restricted product $\prod'_v {\rm{H}}^1(k_v, G)$ with respect to integral cohomology groups when $i = 0, 1$ \cite[Th.\,2.18]{ces2}. (Note that the groups ${\rm{H}}^i(\calO_v, \mathscr{G})$ are subgroups of ${\rm{H}}^i(k_v, G)$ for almost all $v$ if $\mathscr{G}$ is an $\calO_S$-model of $G$; this is \cite[Cor.\,2.9, Lem.\,2.19(1)]{ces2} for $i \neq 1$, and Proposition \ref{H^1(G)injective} for $i = 1$.) In this section we will determine the relation between these adelic cohomology groups and the local cohomology groups when $\mathscr{F} = \widehat{G}$ for affine commutative $G$ of finite type over $k$. The main result (Proposition \ref{H^i(A,G^)--->prodH^i(k_v,G^)}) says that the projections induce an isomorphism onto the restricted direct product of the groups ${\rm{H}}^i(k_v, \widehat{G})$ with respect to the integral cohomology groups when $i = 1$ and an isomorphism onto the torsion subgroup of this restricted product when $i = 2$. Note in particular that the two definitions
\[
\Sha^i(k, \mathscr{F}) : = \ker \left({\rm{H}}^i(k, \mathscr{F}) \rightarrow {\rm{H}}^i(\A, \mathscr{F})\right),
\]
\[
\Sha^i(k, \mathscr{F}) : = \ker \left( {\rm{H}}^i(k, \mathscr{F}) \rightarrow \prod_v {\rm{H}}^i(k_v, \mathscr{F}) \right)
\]
therefore agree for all $i$ when $\mathscr{F} = G$, and for $i \leq 2$ when $\mathscr{F} = \widehat{G}$. We are therefore free to use either definition in subsequent sections, since the only $\Sha^i$ that we will deal with will be when $i = 1$ or $2$. (Our main interest in these groups lies in Theorem \ref{shapairing}.)

Let us introduce a bit of notation. For a set $S$ of places of a global field $k$ containing all of the archimedean places of $k$, let
\[
\widehat{\calO_S} : = \prod_{v \notin S} \calO_v.
\]

\begin{lemma}
\label{opencoverpartition}
Let $\{R_i\}_{i \in I}$ be a set of local rings, and let $R : = \prod_{i \in I} R_i$. For a subset $J \subset I$, let $R_J : = \prod_{i \in J} R_i$. Then any open cover of $\Spec(R)$ may be refined by one of the form $\{ \Spec(R_{I_1}), \dots, \Spec(R_{I_n})\}$, where $I = I_1 \amalg \dots \amalg I_n$ is a finite partition of $I$.
\end{lemma}

\begin{proof}
Any open cover of $\Spec(R)$ may be refined by $\{\Spec(R_{f_1}), \dots, \Spec(R_{f_n})\}$ for some $f _j \in R$ that generate the unit ideal. Since each $R_i$ is local, it follows that for each $i \in I$, we have $(f_j)_i \in R_i^{\times}$ for some $1 \leq j \leq n$. Choose one such $j : = j(i)$ for each $i \in I$. Then, for $1 \leq m \leq n$, we may take $I_m : = \{i \in I \mid j(i) = m\}$.
\end{proof}

\begin{lemma}
\label{H^1(prodO_v,G^)=prodH^1(O_v,G^)weilrestrictiontori}
If $k'/k$ is a finite separable extension of global fields, then ${\rm{H}}^1(\A_k, \widehat{\R_{k'/k}(\Gm^n)}) = 0$, and ${\rm{H}}^1(k_v, \widehat{\R_{k'/k}(\Gm^n)}) = 0$.
\end{lemma}

\begin{proof}
It suffices to treat the case $n = 1$. Using Proposition \ref{charactersseparableweilrestriction}, we have $\widehat{\R_{k'/k}(\Gm)} = \R_{k'/k}(\Z)$. Since this is a smooth group scheme, we may take our cohomology to be \'etale. By Lemma \ref{sepblepushforward}, we have ${\rm{H}}^1(\A_k, \R_{k'/k}(\Z)) \simeq {\rm{H}}^1(\A_{k'}, \Z)$, and similarly for the $k_v$ cohomology. Renaming $k'$ as $k$, we may therefore assume that $k' = k$. We have ${\rm{H}}^1(k_v, \Z) = \Hom_{\rm{cts}}(\Gal((k_v)_s/k), \Z) = 0$, so it only remains to prove the vanishing of ${\rm{H}}^1(\A, \Z)$. By Proposition \ref{directlimitscohom}, we have
\[
{\rm{H}}^1(\A, \Z) = \varinjlim_S \left( {\rm{H}}^1(\widehat{\calO}_S, \Z) \times \prod_{v \in S} {\rm{H}}^1(k_v, \Z) \right),
\]
so it suffices to show that ${\rm{H}}^1(\widehat{\calO}_S, \Z) = 0$.

We first note that for any noetherian normal domain $A$, we have ${\rm{H}}^1(A, \Z) = 0$. Indeed, we have ${\rm{H}}^1(A, \Z) = \Hom_{{\rm{cts}}}(\pi_1(\Spec(A)), \Z)$ where $\Z$ is discrete, and this group vanishes because $\pi_1(\Spec(A))$ is profinite and $\Z$ contains no nontrivial finite subgroup. Next, we show that ${\rm{H}}^1((\widehat{\mathcal{O}_S})_{\mathfrak{p}}, \Z) = 0$ for each prime $\mathfrak{p} \in \Spec(\widehat{\mathcal{O}_S})$.

By Lemma \ref{productofnormal=normal}, $(\widehat{\mathcal{O}_S})_{\mathfrak{p}}$ is a normal domain, so in order to prove the claim it suffices to show that ${\rm{H}}^1(A, \Z) = 0$ for any normal domain $A$. But $A$ is the direct limit of its finitely generated $\Z$-algebras, and because $A$ is a normal domain, the normalization of any such subalgebra (which is a module-finite extension, due to excellence \cite[${\rm{IV_2}}$, Prop.\,7.8.6(ii)]{ega}) is also contained in $A$. It follows that $A$ is the direct limit of its noetherian normal subrings, so we may assume that $A$ is noetherian by Proposition \ref{directlimitscohom}.

We conclude that ${\rm{H}}^1((\widehat{\mathcal{O}_S})_{\mathfrak{p}}, \Z) = 0$ for every prime $\mathfrak{p}$ of $\widehat{\mathcal{O}_S}$. By spreading out, it then follows from Proposition \ref{directlimitscohom} that for any $\alpha \in {\rm{H}}^1(\widehat{\mathcal{O}_S}, \Z)$ there is a Zariski-open cover $\{U_i\}$ of $\Spec(\widehat{\mathcal{O}_S})$ such that $\alpha|_{U_i} = 0$. By Lemma \ref{opencoverpartition}, this cover may be refined to
the open cover arising from a partition of the index set $I$ (the set of places outside $S$) into finitely many pairwise disjoint subsets, so $\alpha = 0$.
\end{proof}

\begin{lemma}
\label{productofpoints}
Let $\{R_i\}_{I \in I}$ be a set of rings, let $A:  = \prod_{i \in I} R_i$, and suppose that $X$ is an affine $A$-scheme. Then the natural map $X(A) \rightarrow \prod_{i \in I} X(R_i)$ is bijective.
\end{lemma}

\begin{proof}
The statement becomes clear once translated into the language of ring homomorphisms.
\end{proof}

\begin{lemma}
\label{Picvanishingislocal}
Let $X$ be a reduced scheme such that $\Pic(X) = 0$. If $\Pic(\A^n_{\calO_{X, x}}) = 0$ for every $x \in X$ then $\Pic(\A^n_X) = 0$; here, $\A^n_Y$ denotes affine $n$-space over a scheme $Y$ $($no relationship with adele rings$)$.
\end{lemma}

\begin{proof}
Let $\mathscr{L}$ be a line bundle on $\A^n_X$, and let $0:  X \rightarrow \A^n_X$ denote the zero section. By our assumptions, there is a trivialization of $0^*\mathscr{L}$. Fix one such trivialization $\phi$. Given an open subset $U \subset X$ such that $\mathscr{L}|_{\A^n_U}$ is trivial, there is a unique trivialization of $\mathscr{L}|_{\A^n_U}$ such that its pullback along the $0$ section is compatible with $\phi$. Indeed, existence is clear and uniqueness follows from the equality $\Gamma(\A^n_U, \Gm) = \Gamma(U, \Gm)$, which holds because $X$ is reduced (Lemma \ref{nonconstantunitsred}). 

Thus, if there exists an open cover $\{U_i\}_{i \in I}$ of $X$ such that each $\mathscr{L}|_{\A^n_{U_i}}$ is trivial then once we modify these trivializations to be compatible with $\phi$ we see that they must glue, hence yield a trivialization of $\mathscr{L}$.
But there exists such an open cover because for each $x \in X$, the pullback $\mathscr{L}|_{\A^n_{\calO_{X,x}}}$ is trivial by hypothesis and such triviality spreads out over some open neighborhood of $x$ in $X$. 
\end{proof}

\begin{lemma}
\label{Pic(prod)injects}
Let $\{R_i\}_{i \in I}$ be a collection of rings $($commutative with identity$)$ indexed by a non-empty set $I$, and let $R := \prod_{i \in I} R_i$. Then the natural map $\Pic(R) \rightarrow \prod_{i \in I} \Pic(R_i)$ is injective.
\end{lemma}

\begin{proof}
Let $\mathscr{Y}$ be a $\Gm$-torsor over $R$ that is trivial over each $R_i$. Then $\mathscr{Y}(R_i) \neq \emptyset$ for each $i$, hence $\mathscr{Y}(R) \neq \emptyset$ by Lemma \ref{productofpoints}. That is, $\mathscr{Y}$ is a trivial $\Gm$-torsor over $R$.
\end{proof}

\begin{proposition}
\label{Pic(prod R_i)[X]=0}
Let $\{R_i\}_{i \in I}$ be a set of seminormal local rings, and let $R : = \prod_{i \in I} R_i$. Then $\Pic(R[X_1, \dots, X_n]) = 0$.
\end{proposition}

\begin{proof}
Since each $R_i$ is local, Lemma \ref{Pic(prod)injects} implies that $\Pic(R) = 0$. By Lemma \ref{Picvanishingislocal}, therefore, it suffices to show that $\Pic(R_{\mathfrak{p}}[X_1, \dots, X_n]) = 0$ for each prime ideal $\mathfrak{p}$ of $R$. But each such localization $R_{\mathfrak{p}}$ of $R$ is a seminormal ring by Lemma \ref{seminormalring}, so the desired vanishing follows from \cite[Th.\,1]{swan}.
\end{proof}

\begin{lemma}
\label{H^1(prodO_v,G_a^)=0}
For a global function field $k$, we have ${\rm{H}}^1(\A_k, \widehat{\Ga}) = 0$ and ${\rm{H}}^1(k_v, \widehat{\Ga}) = 0$ for all places $v$ of $k$.
\end{lemma}

\begin{proof}
The assertion for $k_v$ follows from Proposition \ref{cohomologyofG_adualgeneralk}. For $\A_k$, Proposition \ref{directlimitscohom} implies that
\[
{\rm{H}}^1(\A_k, \widehat{\Ga}) = \varinjlim_S \left( {\rm{H}}^1(\widehat{\calO}_S, \widehat{\Ga}) \times \prod_{v \in S} {\rm{H}}^1(k_v, \widehat{\Ga}) \right).
\]
It therefore suffices to show that ${\rm{H}}^1(\widehat{\calO}_S, \widehat{\Ga}) = 0$. This follows from Proposition \ref{cohomologyofG_adualgeneralk} and Lemma \ref{Pic(prod R_i)[X]=0}.
\end{proof}

\begin{lemma}
\label{H^2(prodO_v,G^)-->prodH^2(O_v,G^)injectiveG_a^}
Let $k$ be a global function field. Then the canonical map
\[
{\rm{H}}^2(\A, \widehat{\Ga}) \longrightarrow \prod {\rm{H}}^2(k_v, \widehat{\Ga})
\]
induces an isomorphism
\[
{\rm{H}}^2(\A, \widehat{\Ga}) \xlongrightarrow{\sim} {\prod}' {\rm{H}}^2(k_v, \widehat{\Ga}),
\]
where the restricted product is with respect to the subgroups ${\rm{H}}^2(\calO_v, \widehat{\Ga}) \subset {\rm{H}}^2(k_v, \widehat{\Ga})$.
\end{lemma}

Note that ${\rm{H}}^2(\calO_v, \widehat{\Ga})$ really is a subgroup of ${\rm{H}}^2(k_v, \widehat{\Ga})$ for all $v$, thanks to Proposition \ref{H^2=Ext^2=Brdvr} and \cite[Cor.\,1.8]{brii}.

\begin{proof}
Let $S$ be a finite set of places of $k$. By Lemmas \ref{prodvalrings} and \ref{proddegimpleq1}, each local ring of $\widehat{\calO}_S$ is a valuation ring whose fraction field has degree of imperfection $\leq 1$. Propositions \ref{H^2=Ext^2=Brdvr} and \ref{prodprimbrauer} therefore imply that the map
\[
{\rm{H}}^2(\widehat{\calO}_S, \widehat{\Ga}) \longrightarrow \prod_{v \notin S} {\rm{H}}^2(\calO_v, \widehat{\Ga})
\]
is an isomorphism. The lemma therefore follows from the following equality, which is a consequence of Proposition \ref{directlimitscohom}:
\[
{\rm{H}}^2(\A, \widehat{\Ga}) = \varinjlim_S \left({\rm{H}}^2(\widehat{\calO}_S, \widehat{\Ga}) \times \prod_{v \in S} {\rm{H}}^2(k_v, \widehat{\Ga}) \right).
\]
\end{proof}

\begin{lemma}
\label{H^i(Q)=0}
If $X$ is a noetherian normal scheme, then ${\rm{H}}^i(X, \Q) = 0$ for $i > 0$.
\end{lemma}

\begin{proof}
Since the constant commutative $X$-group $\Q$ is smooth, we may take our cohomology to be {\'e}tale. We may assume that $X$ is connected, hence (because of normality) irreducible. Let $\eta$ be its generic point, $f:  \eta \rightarrow X$ the canonical inclusion. We have a Leray spectral sequence
\[
E_2^{i,j} = {\rm{H}}^i(X, \R^jf_*\Q) \Longrightarrow {\rm{H}}^{i+j}(\eta, \Q).
\]
We claim that $\R^jf_*\Q = 0$ for $j > 0$. Indeed, it is the sheafification of the \'etale presheaf $U \mapsto {\rm{H}}^j(U_{\eta}, \Q)$. But if $U$ is an \'etale $X$-scheme, then $U_{\eta}$ is a disjoint union of spectra of fields, so because of the vanishing of the higher Galois cohomology of $\Q$ (due to its unique divisibility), we see that ${\rm{H}}^j(U_{\eta}, \Q) = 0$ for $j > 0$. For the same reason, we have ${\rm{H}}^i(\eta, \Q) = 0$ for $i > 0$. The lemma will follow, therefore, if we show that the natural map $\Q \rightarrow f_*\Q$ is an isomorphism.

For any {\'e}tale $X$-scheme $U$, we have $(f_*\Q)(U) = \Q(U_{\eta})$ by definition. Since
$U$ is a disjoint union of connected components, it suffices to show that $U_{\eta}$ is connected
when $U$ is connected. Because $U$ is \'etale over $X$, the scheme $U_{\eta}$ is the disjoint union of the generic points of $U$, so 
it suffices to show that $U$ is irreducible. But $U$ is connected, noetherian, and normal (because it is {\'e}tale over the noetherian normal scheme $X$), so it is indeed irreducible.
\end{proof}

\begin{lemma}
\label{H^i(O_S^,Q)=0}
For a global function field $k$, we have ${\rm{H}}^i(\A_k, \Q) = 0$ for $i > 0$.
\end{lemma}

\begin{proof}
Proposition \ref{directlimitscohom} implies that 
\[
{\rm{H}}^i(\A_k, \Q) = \varinjlim_S \left( {\rm{H}}^i(\widehat{\calO}_S, \Q) \times \prod_{v \in S} {\rm{H}}^i(k_v, \Q) \right),
\]
where the limit is over all finite sets of places of $k$ containing the archimedean ones. By Lemma \ref{H^i(Q)=0}, ${\rm{H}}^i(k_v, \Q) = 0$ for all $v$, hence it suffices to show that ${\rm{H}}^i(\widehat{\calO}_S, \Q) = 0$. We will first show that ${\rm{H}}^i((\widehat{\calO_S})_{\mathfrak{p}}, \Q) = 0$ for $i > 0$ and for every prime ideal $\mathfrak{p}$ of $\widehat{\calO_S}$. By Lemma \ref{productofnormal=normal}, $(\widehat{\calO_S})_{\mathfrak{p}}$ is a normal domain
and hence is the direct limit of its finite-type $\Z$-subalgebras $A_i$. Replacing $A_i$ with its normalization (which is finite over $A_i$ due to excellence), we see that $(\widehat{\calO_S})_{\mathfrak{p}}$ is the filtered direct limit of noetherian normal subrings $A_i$. By Lemma \ref{directlimitscohom}, therefore, ${\rm{H}}^i((\widehat{\calO_S})_{\mathfrak{p}}, \Q) = \varinjlim_i {\rm{H}}^i(A_i, \Q) = 0$ for $i > 0$ by Lemma \ref{H^i(Q)=0}. 

Let $\alpha \in {\rm{H}}^i(\widehat{\calO_S}, \Q)$ for $i > 0$. We need to show that $\alpha = 0$. Since the pullback of $\alpha$ to each local ring of $\widehat{\calO_S}$ vanishes, there is an open cover $\{U_j\}$ of $\Spec(\widehat{\calO_S})$ such that $\alpha|_{U_j} = 0$ for each $j$. By Lemma \ref{opencoverpartition}, any open cover of $\Spec(\widehat{\calO_S})$ may be refined by one obtained by a finite partition of the index set $I$ for the product $\widehat{\calO_S} = \prod_{v \notin S} \calO_v$.
That is, there is a partition $I = I_1 \amalg \dots \amalg I_n$ (depending on $i$) such that $\alpha|_{\prod_{v \in I_m} \calO_v} = 0$ for each $1 \leq m \leq n$. It follows that $\alpha = 0$.
\end{proof}

\begin{lemma}
\label{H^2(prod)-->prodH^2weilrestriction}
Let $k'/k$ be a finite separable extension of global fields, let $S$ be a set of places of $k$ containing the archimedean places as well as all of the places of $k$ that are ramified in $k'$, and let $S'$ be the set of places of $k'$ lying above $S$. Then the natural map 
\[
{\rm{H}}^2(\A, \widehat{\R_{k'/k}(\Gm^n)}) \longrightarrow {\prod_{v \notin S}} {\rm{H}}^2(k_v, \widehat{\R_{k'/k}(\Gm^n)})
\]
induces an isomorphism
\[
{\rm{H}}^2(\A, \widehat{\R_{k'/k}(\Gm^n)}) \xlongrightarrow{\sim} \left({\prod_{v \notin S}}' {\rm{H}}^2(k_v, \widehat{\R_{k'/k}(\Gm^n)})\right)_{{\rm{tors}}},
\]
where the restricted product is with respect to the subgroups ${\rm{H}}^2(\calO_v, \widehat{\R_{\calO_{S'}/\calO_S}(\Gm^n)})$.
\end{lemma}

Note that the groups ${\rm{H}}^2(\calO_v, \widehat{\R_{\calO_{S'}/\calO_S}(\Gm^n)})$ are indeed subgroups of ${\rm{H}}^2(k_v, \widehat{\R_{k'/k}(\Gm^n)})$ for almost all $v$, by Theorem \ref{H^2(O,G^)dualityprop}.

\begin{proof}
We may assume that $n = 1$. By \cite[Thm.\,2.18]{ces2}, the map
\[
\textstyle {\rm{H}}^1(\A, \frac{1}{n}\Z/\Z) \longrightarrow {\prod}' {\rm{H}}^1(k_v, \frac{1}{n}\Z/\Z)
\]
is an isomorphism for all $n > 0$. Since $\Q/\Z = \varinjlim_n \frac{1}{n}\Z/\Z$, taking the direct limit over $n$ yields by Proposition \ref{directlimitscohom} an isomorphism
\[
\textstyle {\rm{H}}^1(\A, \Q/\Z) \xlongrightarrow{\sim} \varinjlim_n {\prod}' {\rm{H}}^1(k_v, \frac{1}{n}\Z/\Z).
\]
We claim that
\[
\textstyle {\rm{H}}^1(k_v, \frac{1}{n}\Z/\Z) \longrightarrow {\rm{H}}^1(k_v, \Q/\Z)[n]
\]
is an isomorphism for all $n > 0$. This follows from the exact sequence
\[
\textstyle 0 \longrightarrow \frac {1}{n}\Z/\Z \longrightarrow \Q/\Z \xlongrightarrow{[n]} \Q/\Z \longrightarrow 0.
\]
We therefore obtain an isomorphism
\[
{\rm{H}}^1(\A, \Q/\Z) \xlongrightarrow{\sim} \varinjlim_n \left({\prod}' {\rm{H}}^1(k_v, \Q/\Z)\right)[n] = \left({\prod}' {\rm{H}}^1(k_v, \Q/\Z)\right)_{\rm{tors}}.
\]
Now the exact sequence
\[
0 \longrightarrow \Z \longrightarrow \Q \longrightarrow \Q/\Z \longrightarrow 0,
\]
in conjunction with Lemmas \ref{H^i(O_S^,Q)=0} and \ref{H^i(Q)=0}, then implies that we have the isomorphism
\[
{\rm{H}}^2(\A, \Z) \xlongrightarrow{\sim} \left({\prod}' {\rm{H}}^2(k_v, \Z)\right)_{\rm{tors}}.
\]
Since $\widehat{\Gm} = \Z$, this proves the lemma when $k' = k$.

We now turn to the general case. By Proposition \ref{charactersseparableweilrestriction}, 
$$\widehat{\R_{\calO_{S'}/\calO_S}(\Gm)} = \R_{\calO_{S'}/\calO_S}(\widehat{\Gm}) = \R_{\calO_{S'}/\calO_S}(\Z),$$ and similarly for the sheaves over $k$.
The constant $\calO_{S'}$-group scheme $\Z$ is smooth, so its Weil restriction to $\calO_S$ is as well, and hence we may take our cohomology to be {\'e}tale. By Lemma \ref{sepblepushforward}, ${\rm{H}}^2(\A_k, \R_{k'/k}(\Z)) = {\rm{H}}^2(A_{k'}, \Z)$, and similarly for the $k_v$ and $\calO_v$ cohomology groups. Renaming $k'$ as $k$, we are thereby reduced to the already treated case when $k' = k$.
\end{proof}

\begin{proposition}
\label{H^3(A,Ga^)=0}
For a global function field $k$, we have ${\rm{H}}^3(\A_k, \widehat{\Ga}) = 0$.
\end{proposition}

\begin{proof}
By Proposition \ref{directlimitscohom}, we have
\[
{\rm{H}}^3(\A, \widehat{\Ga}) = \varinjlim_S \left( {\rm{H}}^3(\widehat{\calO}_S, \widehat{\Ga}) \times \prod_{v \in S} {\rm{H}}^3(k_v, \widehat{\Ga}) \right),
\]
where the limit is over all finite sets $S$ of places of $k$. The result therefore follows from Proposition \ref{H^3(Ga^)=0prodofdvrs}.
\end{proof}

\begin{lemma}
\label{H^1H^2(A)torsion}
Let $k$ be a global function field, and let $G$ be an affine commutative $k$-group scheme of finite type. Then ${\rm{H}}^1(\A, \widehat{G})$ is of finite exponent, and ${\rm{H}}^2(\A, \widehat{G})$ is torsion.
\end{lemma}

\begin{proof}
By Proposition \ref{hatisexact} and Lemma \ref{affinegroupstructurethm}, it suffices to treat the case when $G$ is either $\Ga$ or an almost torus. The case $G = \Ga$ is trivial (since $\Ga$ is $p$-torsion), so we may assume that $G$ is an almost-torus. By Lemma \ref{almosttorus}(iv), we may harmlessly modify $G$ and thereby assume that there is an exact sequence
\[
1 \longrightarrow B \longrightarrow X \longrightarrow G \longrightarrow 1,
\]
where $X = C \times \R_{k'/k}(\Gm^n)$, $B$ and $C$ are finite commutative $k$-group schemes, and $k'/k$ is a finite separable extension. Since $B$ and $C$ are of finite exponent, we are reduced to proving the lemma when $G = \R_{k'/k}(\Gm)$. In this case, the desired assertions follow from Lemmas \ref{H^1(prodO_v,G^)=prodH^1(O_v,G^)weilrestrictiontori} and \ref{H^2(prod)-->prodH^2weilrestriction}.
\end{proof}

We are now prepared to prove the main result of this section, which relates the adelic cohomology of $\widehat{G}$ for affine commutative $G$ of finite type to the local cohomology groups.

\begin{proposition}
\label{H^i(A,G^)--->prodH^i(k_v,G^)}
Let $k$ be a global function field, $G$ an affine commutative $k$-group scheme of finite type, and $\mathscr{G}$ a finite type $\calO_S$-model of $G$ for some finite non-empty set $S$ of places of $k$. Then the maps ${\rm{H}}^i(\A_k, \widehat{G}) \rightarrow \prod_v {\rm{H}}^i(k_v, \widehat{G})$ induce isomorphisms:
\begin{itemize}
\item[(i)] $\displaystyle {\rm{H}}^1(\A_k, \widehat{G}) \xlongrightarrow{\sim} {\prod_v}' {\rm{H}}^1(k_v, \widehat{G})$,
\item[(ii)] $\textstyle {\rm{H}}^2(\A_k, \widehat{G}) \xlongrightarrow{\sim} \left({\prod'_v} {\rm{H}}^2(k_v, \widehat{G})\right)_{\rm{tors}}$,
\end{itemize}
where the restricted products are with respect to the subgroups ${\rm{H}}^i(\calO_v, \widehat{\mathscr{G}}) \subset {\rm{H}}^i(k_v, \widehat{G})$. $($Note that these are indeed subgroups for all but finitely many $v$ by Theorems $\ref{H^1(k,G^)/H^1(O,G^)}$ and $\ref{H^2(O,G^)dualityprop}$.$)$
\end{proposition}

\begin{proof}
First consider the case in which $G$ is an almost torus. By Lemma \ref{almosttorus}(iv), after harmlessly modifying $G$, we obtain an exact sequence
\[
1 \longrightarrow B \longrightarrow X \longrightarrow G \longrightarrow 1
\]
where $X = C \times \R_{k'/k}(T')$, $k'/k$ is a finite separable extension, $T'$ is a split $k'$-torus, and $B$ and $C$ are finite
commutative $k$-group schemes. By Lemma \ref{spreadingoutdualsheaves}, we may spread this out to obtain the dual exact sequence (for the fppf topology on the category of all $\calO_S$-schemes) 
\[
1 \longrightarrow \widehat{\mathscr{G}} \longrightarrow \widehat{\mathscr{X}} \longrightarrow \widehat{\mathscr{B}} \longrightarrow 1,
\]
with $\mathscr{B}$ a finite flat commutative $\calO_S$-group scheme. Then, using Lemmas \ref{H^1(prodO_v,G^)=prodH^1(O_v,G^)weilrestrictiontori} and \ref{H^2(prod)-->prodH^2weilrestriction}, and \cite[Thm.\,2.18]{ces2} (and the fact that $B$ is of finite exponent for the last vertical map), the indicated maps in the commutative diagram below (in which all restricted products are with respect to the integral cohomology groups) are isomorphisms. Note that the ${\rm{H}}^2$ groups map into the torsion subgroups of the restricted product by Lemma \ref{H^1H^2(A)torsion}.
\[
\begin{tikzcd}
\widehat{X}(\A) \arrow{r} \isoarrow{d} & \widehat{B}(\A) \isoarrow{d} \arrow{r} & {\rm{H}}^1(\A, \widehat{G}) \arrow{d} \arrow{r} & {\rm{H}}^1(\A, \widehat{X}) \isoarrow{d} \\
{\prod}' \widehat{X}(k_v) \arrow{r} & {\prod}' \widehat{B}(k_v) \arrow{r} & {\prod}' {\rm{H}}^1(k_v, \widehat{G}) \arrow{r} & {\prod}' {\rm{H}}^1(k_v, \widehat{X})
\end{tikzcd}
\]
\[
\begin{tikzcd}[column sep = tiny]
\arrow{r} & {\rm{H}}^1(\A, \widehat{B}) \arrow{r} \isoarrow{d} & {\rm{H}}^2(\A, \widehat{G}) \arrow{d} \arrow{r} & {\rm{H}}^2(\A, \widehat{X}) \arrow{r} \isoarrow{d} & {\rm{H}}^2(\A, \widehat{B}) \isoarrow{d} \\
\arrow{r} & {\prod}' {\rm{H}}^1(k_v, \widehat{B}) \arrow{r} & ({\prod}' {\rm{H}}^2(k_v, \widehat{G}))_{\rm{tors}} \arrow{r} & ({\prod}' {\rm{H}}^2(k_v, \widehat{X}))_{\rm{tors}} \arrow{r} & ({\prod}' {\rm{H}}^2(k_v, \widehat{B}))_{\rm{tors}}
\end{tikzcd}
\]
The fact that the group ${\prod}' {\rm{H}}^1(k_v, \widehat{B})$ is torsion implies that the bottom row is exact. The five lemma now implies that the remaining two vertical maps are also isomorphisms. This completes the proof of the proposition for almost-tori.

Now we prove the proposition in general. We proceed by induction on the dimension of the unipotent radical of $(G_{\overline{k}})^0_{\rm{red}}$, the $0$-dimensional case being the already-handled situation of almost-tori. Applying Lemma \ref{affinegroupstructurethm} and induction, we may therefore suppose that there is an exact sequence
\[
1 \longrightarrow H \longrightarrow G \longrightarrow \Ga \longrightarrow 1
\]
such that the proposition is true for $H$. By Lemma \ref{spreadingoutdualsheaves}, we may for some finite set $S$ of places of $k$ spread this out to an exact sequence of fppf sheaves over $\calO_S$
\[
1 \longrightarrow \widehat{\mathscr{\Ga}} \longrightarrow \widehat{\mathscr{G}} \longrightarrow \widehat{\mathscr{H}} \longrightarrow 1.
\]
In the exact diagram below, the second and fifth vertical maps are isomorphisms by hypothesis, and the third is by Lemma \ref{H^2(prodO_v,G^)-->prodH^2(O_v,G^)injectiveG_a^}. The first horizontal arrows in each row are inclusions by Lemma \ref{H^1(prodO_v,G_a^)=0}, and the last map in the top row is a surjection by Proposition \ref{H^3(A,Ga^)=0}. Further, the penultimate map in the bottom row maps into the torsion subgroup because $\Ga$ is $p$-torsion. A diagram chase now implies that the remaining two vertical maps are isomorphisms, which proves the proposition.
\[
\begin{tikzcd}[column sep = tiny]
{\rm{H}}^1(\A, \widehat{G}) \arrow{d} \arrow[r, hookrightarrow] & {\rm{H}}^1(\A, \widehat{H}) \arrow{r} \isoarrow{d} & {\rm{H}}^2(\A, \widehat{\Ga}) \arrow{r} \isoarrow{d} & {\rm{H}}^2(\A, \widehat{G}) \arrow{d} \arrow[r, twoheadrightarrow] & {\rm{H}}^2(\A, \widehat{H}) \isoarrow{d} \\
{\prod}' {\rm{H}}^1(k_v, \widehat{G}) \arrow[r, hookrightarrow] & {\prod}' {\rm{H}}^1(k_v, \widehat{H}) \arrow{r} & {\prod}' {\rm{H}}^2(k_v, \widehat{\Ga}) \arrow{r} & ({\prod}' {\rm{H}}^2(k_v, \widehat{G}))_{\rm{tors}} \arrow{r} & ({\prod}' {\rm{H}}^2(k_v, \widehat{H}))_{\rm{tors}}
\end{tikzcd}
\]
\end{proof}

\section{Topology on cohomology of the adeles and adelic duality}
\label{sectionadeliccohomtopology}

Let $k$ be a global function field, $G$ an affine commutative $k$-group scheme of finite type. We have made the groups ${\rm{H}}^i(k_v, G)$ and ${\rm{H}}^i(k_v, \widehat{G})$ into topological groups for $i \leq 2$, the most subtle case being $i = 1$, which was treated in \S\S \ref{sectiontopologyoncohomology} and \ref{h1dualgen}. The purpose of this section is to similarly topologize the corresponding adelic cohomology groups. This topology will play an important role later in this work. We will also prove analogues of some of the local duality theorems for these adelic cohomology groups.

In order to define these topologies, we consider the following general situation. Let $\mathscr{F}$ be a sheaf on the big fppf site of ${\rm{Spec}}(k)$, and suppose that we have, for each place $v$ of $k$, a topology on ${\rm{H}}^i(k_v, \mathscr{F})$. Then there are two natural ways to obtain from these a topology on ${\rm{H}}^i(\A, \mathscr{F})$, where $\A$ is the adele ring of $k$. For any set $S$ of places of $k$, let $\A^S$ denote the ring of $S$-adeles (the usual restricted product over all $v \notin S$), and consider the natural map
\[
\pi_S: {\rm{H}}^i(\A^S, \mathscr{F}) \longrightarrow \prod_{v \notin S} {\rm{H}}^i(k_v, \mathscr{F}).
\]
Then the first way to topologize ${\rm{H}}^i(\A, \mathscr{F})$ is to take as a basis those subsets of the form
\[
U \times \pi_S^{-1}\left(\prod_{v \notin S} {\rm{im}}({\rm{H}}^i(\calO_v, \mathscr{F}) \rightarrow {\rm{H}}^i(k_v, \mathscr{F}))\right),
\]
for $S$ a finite set of places of $k$, where $U \subset \prod_{v \in S} {\rm{H}}^i(k_v, \mathscr{F})$ is an arbitrary open subset. The second way to topologize ${\rm{H}}^i(\A, \mathscr{F})$ is to start with a finite set $S'$ of places of $k$ and an $\calO_{S'}$-sheaf $\mathscr{H}$ such that $\mathscr{F}$ is the restriction of $\mathscr{H}$ to $k$, and then to take as a basis those sets of the form
\[
U \times {\rm{im}}\left({\rm{H}}^i\left(\prod_{v \notin S} \calO_v, \mathscr{H}\right) \rightarrow {\rm{H}}^i(\A^S, \mathscr{F})\right).
\]
for finite sets $S \supset S'$ of places of $k$.

These two definitions agree for any reasonable sheaf $\mathscr{H}$; more precisely, they agree whenever $\mathscr{H}$ is locally of finite presentation, i.e., commutes with filtered direct limits of rings. (This is true of the sheaves $\mathscr{H} = \mathscr{G}$ or $\widehat{\mathscr{G}}$ for $\mathscr{G}$ a finite type group scheme over $\calO_{S'}$.) Indeed, this agreement follows from the fact that
\[
{\rm{H}}^i(\A, \mathscr{F}) = \varinjlim_S \left(\prod_{v \in S} {\rm{H}}^i(k_v, S) \times {\rm{H}}^i\left(\prod_{v \notin S} \calO_v, \mathscr{H}\right)\right),
\]
which holds by Proposition \ref{directlimitscohom}, because this equality shows that any adelic cohomology class  automatically arises from a class in ${\rm{H}}^i(\prod_{v \notin S} \calO_v, \mathscr{H})$ for suitably large $S$ (with, however, $S$ depending on the cohomology class).

Recall that if $G$ is a commutative group scheme of finite type over $k_v$, then we endow the group ${\rm{H}}^0(k_v, G)$ with the topology arising from the topology on $k_v$, and we endow the group ${\rm{H}}^0(k_v, \widehat{G})$ with the discrete topology. We endow ${\rm{H}}^1(k_v, G)$ with the minimal topology that makes every map $X \rightarrow BG$ for a locally finite type $k_v$-scheme $X$ continuous (as discussed at the beginning of \S \ref{sectiontopologyoncohomology}), and similarly for ${\rm{H}}^1(k_v, \widehat{G})$ when $G$ is in almost-torus (in which case, by Lemma \ref{hatrepresentable}, $\widehat{G}$ is represented by a locally of finite type $k_v$-group scheme), and we use this topology for almost-tori to define a topology on this group for any affine commutative $k_v$-group scheme of finite type, as discussed in \S \ref{h1dualgen}. Although we shall have no use of this beyond degree $2$, following \cite{ces1} let us also endow all of the groups ${\rm{H}}^n(k_v, G)$ and ${\rm{H}}^n(k_v, \widehat{G})$ for $n > 1$ with the discrete topology.

Now let $k$ be a global field (the subtler case being that of global function fields), and let $G$ be a commutative $k$-group scheme of finite type. Let $\mathscr{G}$ be a flat, finite type $\calO_S$-model of $G$ for some finite non-empty set $S$ of places of $k$. For $i = 0, 1$, we then topologize the groups ${\rm{H}}^i(\A, G)$ and ${\rm{H}}^i(\A, \widehat{G})$ in either of the two equivalent manners discussed above, and we remark (so as to be able to use his results) that, for $\mathscr{F} = G$, the first of the two definitions is the one given in \v{C}esnavi\v{c}ius \cite[\S 3]{ces2}. For $i > 1$, rather than the topologies described above, we endow the groups ${\rm{H}}^i(\A, \mathscr{F})$ with the discrete topology for any fppf sheaf $\mathscr{F}$ on $\A$. We first note that these topologies have various desirable properties.

\begin{proposition}
\label{adelictopcohombasics}
Let $G$ be a commutative group scheme of finite type over a global field $k$.
\begin{itemize}
\item[(i)] For $i > 2$, the map $${\rm{H}}^i(\A, G) \rightarrow \underset{v \mbox{ real}}{\prod} {\rm{H}}^i(k_v, G)$$ is a topological isomorphism.
\item[(ii)] For $i = 2$, the map $${\rm{H}}^i(\A, G) \rightarrow \bigoplus_v {\rm{H}}^i(k_v, G)$$ is a topological isomorphism, where the direct sum is endowed with the discrete topology. The same holds for $i = 1$ if $G$ is smooth and connected.
\item[(iii)] Let $\mathscr{G}$ be an $\calO_S$-model of $G$ for some finite non-empty set $S$ of places of $k$. Then for $i = 0, 1$, the map $${\rm{H}}^i(\A, G) \rightarrow {\prod}'_v {\rm{H}}^i(k_v, G)$$ to the restricted direct product with respect to the groups ${\rm{H}}^i(\calO_v, \mathscr{G})$ is a topological isomorphism (where the latter group has the restricted product topology).
\item[(iv)] The groups ${\rm{H}}^i(\A, G)$ are locally compact, Hausdorff, second-countable topological groups.
\item[(v)] These topologies are functorial, That is, given a $k$-homomorphism $G \rightarrow G'$, the induced maps ${\rm{H}}^i(\A, G) \rightarrow {\rm{H}}^i(\A, G')$ are continuous.
\end{itemize}
Now suppose that one has a short exact sequence of finite type commutative $k$-group schemes
\[
1 \longrightarrow G' \longrightarrow G \longrightarrow G'' \longrightarrow 1.
\]
\begin{itemize}
\item[(vi)] The maps between adelic cohomology groups in the associated long exact cohomology sequence are continuous.
\item[(vii)] If $G'$ is smooth and connected, then the map ${\rm{H}}^0(\A, G) \rightarrow {\rm{H}}^0(\A, G'')$ is open.
\item[(viii)] If $G$ is smooth and connected, then the map ${\rm{H}}^0(\A, G'') \rightarrow {\rm{H}}^1(\A, G')$ is open.
\item[(ix)] If $G''$ is smooth and connected, then the map ${\rm{H}}^1(\A, G') \rightarrow {\rm{H}}^1(\A, G)$ is open.
\item[(x)] The map ${\rm{H}}^1(\A, G) \rightarrow {\rm{H}}^1(\A, G'')$ is open.
\end{itemize}
\end{proposition}

\begin{proof}
(i) This follows from \cite[Th.\,2.13, Prop.\,3.4(b)]{ces2} and Proposition \ref{cohomologicalvanishing}. \\ \\
(ii) \cite[Prop.\,3.4(b), Thm.\,2.18]{ces2}. \\ \\
(iii) \cite[Prop.\,3.4(b)]{ces2}. \\ \\
(iv) That these are Hausdorff topological groups follows from the analogous results for the ${\rm{H}}^i(k_v, G)$ (the nontrivial case being Proposition \ref{topcohombasics}(i)) in conjunction with parts (i)-(iii). Second-countability follows from the analogous results for the local groups, which is clear for $i = 0$, follows from Proposition \ref{secondcountable} for $i = 1$ and $k_v \neq \mathbf{R}$, from Theorem \ref{H^2(G)G^(k)dualityprop} for $i = 2$ and $k_v \neq \mathbf{R}$, Proposition \ref{cohomologicalvanishing} for $i > 2$ and $k_v \neq \mathbf{R}$, from the finiteness of these cohomology groups for $i = 1$ and $k_v = \mathbf{R}$ \cite[\S, 4.2, \S4.3, Th.\,4]{serre}, and from the fact that the cohomology of $\mathbf{R}$ is periodic with period 2 for $i > 1$ and $k_v = \mathbf{R}$.

Finally, it remains to prove local compactness. This is trivial for $i > 1$, since in this case the groups are discrete.  For $i = 0$ or $1$, it follows from the local compactness of the groups ${\rm{H}}^i(k_v, G)$ (clear for $i =0$, and for $i = 1$ follows from Proposition \ref{topcohombasics}(i)), together with parts (i)-(iii), provided that one shows that the groups ${\rm{H}}^i(\calO_v, \mathscr{G}) \subset {\rm{H}}^i(k_v, G)$ are compact for all but finitely many places $v$ of $k$. For $i = 0$, this is well-known, and ultimately follows from the compactness of $\calO_v$. For $i = 1$, this was proven in Remark \ref{H^1(O_v,G)compact}. \\ \\
(v) \cite[Prop.\,3.5(a)]{ces2}. \\ \\
(vi) This is trivial for the maps from cohomology groups of degree $ > 1$. For the low degree part of the sequence, we spread out the short exact sequence to an exact sequence of $\calO_S$-group schemes
\begin{equation}
\label{adelictopcohombasicseqn1}
1 \longrightarrow \mathscr{G}' \longrightarrow \mathscr{G} \longrightarrow \mathscr{G}'' \longrightarrow 1.
\end{equation}
The continuity of all of the relevant maps then follows from the continuity of the maps on $k_v$-cohomology groups, which holds up to degree 2 by Proposition \ref{topcohombasics}(v), except that for the connecting map ${\rm{H}}^1(\A, \mathscr{G}'') \rightarrow {\rm{H}}^2(\A, \mathscr{G}')$, we must also use the injectivity of the map
\[
{\rm{H}}^2(\A, \mathscr{G}') \rightarrow \prod_v {\rm{H}}^2(k_v, G')
\]
to reduce to the assertion that ${\rm{H}}^2(\calO_v, \mathscr{G}') = 0$ for all but finitely many $v$. This follows from Proposition \ref{H^i(O,G)=0}. \\ \\
(vii) \cite[Prop.\,3.11(a)]{ces2}. \\ \\
(viii) \cite[Prop.\,3.11(b)]{ces2}. \\ \\
(ix) \cite[Prop.\,3.11(c)]{ces2}. \\ \\
(x) \cite[Prop.\,3.11(d)]{ces2}.
\end{proof}

\begin{proposition}
\label{adelictopcohombasicsG^}
Let $G$ be an affine commutative group scheme of finite type over a global field $k$.
\begin{itemize}
\item[(i)] The ${\rm{H}}^i(\A, \widehat{G})$ are topological groups.
\item[(ii)] The group ${\rm{H}}^0(\A, \widehat{G})$ is second-countable and Hausdorff. The canonical map
\[
\widehat{G}(\A) \longrightarrow \prod_v \widehat{G}(k_v)
\]
is an inclusion, and the topology on $\widehat{G}(\A)$ is the subspace topology arising from the product topology on $\prod_v \widehat{G}(k_v)$.
\item[(iii)] The map $${\rm{H}}^1(\A, \widehat{G}) \rightarrow {\prod_v}' {\rm{H}}^1(k_v, \widehat{G})$$ to the restricted product with respect to the subgroups ${\rm{H}}^1(\calO_v, \widehat{G})$ is a topological isomorphism $($where the latter group has the restricted product topology$)$. The group ${\rm{H}}^1(\A, \widehat{G})$ is Hausdorff, locally compact, and second-countable.
\item[(iv)] These topologies are functorial:  given a $k$-homomorphism $G \rightarrow G'$ between affine commutative $k$-group schemes of finite type, the induced maps ${\rm{H}}^i(\A, \widehat{G'}) \rightarrow {\rm{H}}^i(\A, \widehat{G})$ are continuous.
\end{itemize}
\end{proposition}

\begin{remark}
We remark that the map in part (ii) of the proposition is {\em not} surjective in general. For example, if $G = \Gm$, then we claim that the map
\[
\Z(\A) \longrightarrow \prod_v \Z(k_v) = \prod_v \Z
\]
is not surjective. Indeed, an element of $\Z(\A)$ is given by a disjoint union ${\rm{Spec}}(\A) = \coprod_{i \in I} U_i$ of open subschemes of ${\rm{Spec}}(\A)$, together with a function $I \rightarrow \Z$. Since affine schemes are quasi-compact, any element of $\Z(\A)$ therefore takes on only finitely many integer values, hence its image in $\prod_v \Z$ is an element whose entries lie in a finite set.
\end{remark}

\begin{proof}
(i) This is immediate from the corresponding property of the ${\rm{H}}^i(k_v, \widehat{G})$ (\S \ref{h1dualgen} and Proposition \ref{topcohombasics}(i)) and the fact that $G$ spreads out to an $\calO_S$-group scheme for some finite non-empty set $S$ of places of $k$. \\ \\
(ii) The second-countability follows from the corresponding property of the discrete finitely-generated abelian groups ${\rm{H}}^0(k_v, \widehat{G})$. The Hausdorffness follows from the second statement about the map into $\prod_v \widehat{G}(k_v)$, which we now prove. A homomorphism of $\A$-schemes $G_{\A} \rightarrow \mathbf{G}_{m,\, \A}$ is given by a global unit of $G_{\A}$. For an affine $k$-scheme $X$ of finite type, a global section of $X_{\A}$ is uniquely determined by the corresponding global sections of $X_{k_v}$ for all places $v$ of $k$. It follows that the map
\[
\widehat{G}(\A) \rightarrow \prod_v \widehat{G}(k_v)
\]
is an inclusion. Further, by Theorem \ref{H^2(G)G^(k)integraldualityprop}, if $\mathscr{G}$ is an $\calO_S$-model of $G$ (for some finite set $S$ of places of $k$), then $\widehat{\mathscr{G}}(\calO_v) = \widehat{G}(k_v)$ for all but finitely many $v$. It then follows from the definition of the topology on $\widehat{G}(\A)$ that the topology on it is indeed the subspace topology arising from the product topology on $\prod_v \widehat{G}(k_v)$. \\ \\
(iii) The map is an isomorphism by Proposition \ref{H^i(A,G^)--->prodH^i(k_v,G^)}, and it is a homeomorphism by definition of the topologies on the two groups. To show that ${\rm{H}}^1(\A, \widehat{G})$ is locally compact, we first note that the groups ${\rm{H}}^1(k_v, \widehat{G})$ are, by Lemma \ref{H^1(G^)closed}. It therefore suffices to prove that the groups ${\rm{H}}^1(\calO_v, \widehat{\mathscr{G}})$ are compact for all but finitely many $v \notin S$, and for this it suffices to show that their Pontryagin dual groups are discrete. By Theorem \ref{H^1(k,G^)/H^1(O,G^)}, for all but finitely many $v$ this dual group is ${\rm{H}}^1(k_v, G)/{\rm{H}}^1(\calO_v, \mathscr{G})$. It therefore suffices to show that the subgroup ${\rm{H}}^1(\calO_v, \mathscr{G}) \subset {\rm{H}}^1(k_v, G)$ is open, and this follows from \cite[Prop.\,2.9(e)]{ces1}. \\ \\
(iv) The assertion is trivial for $i > 1$, since then the groups are discrete. The analogous property for the groups ${\rm{H}}^i(k_v, \widehat{G})$ is trivial for $i = 0$ (since the groups are discrete in that case), and follows from Proposition \ref{H^1(G^)functorial} for $i = 1$. The property for ${\rm{H}}^i(\A, \widehat{G})$ then follows from the $k_v$-version together with the fact that the $k$-homomorphism spreads out and thereby yields a morphism of fppf sheaves $\widehat{\mathscr{G}'} \rightarrow \widehat{\mathscr{G}}$ over some $\calO_S$.
\end{proof}

We now turn to proving adelic analogues of the local duality theorems. If $G$ is an affine commutative group scheme of finite type over a global field $k$, then for $0 \leq i \leq 2$, we have the pairings
\begin{equation}
\label{adeliccohompairingeqn3}
{\rm{H}}^i(\A, G) \times {\rm{H}}^{2-i}(\A, \widehat{G}) \rightarrow \Q/\Z
\end{equation}
defined by cupping everywhere locally and then summing the invariants. To show that this makes sense, we need to verify that all but finitely many of the summands vanish. In fact, we have $\A = \varinjlim_{S'} \left( \prod_{v \in S'} k_v \times \prod_{v \notin S'} \calO_v \right)$, where the limit is over all finite sets $S'$ of places of $k$ containing the archimedean places. Therefore, by Proposition \ref{directlimitscohom}, we have ${\rm{H}}^i(\A, G) = \varinjlim_{S'} \left( \prod_{v \in S'} {\rm{H}}^i(k_v, G) \times {\rm{H}}^i(\prod_{v \notin S'} \calO_v, \mathscr{G}) \right)$, and similarly for $\widehat{G}$, where $\mathscr{G}$ is an $\calO_S$-model of $G$ for some finite non-empty set $S$ of places of $k$ containing all of the archimedean places, where the limit is over all finite sets of places $S' \supset S$. It follows that for any $\alpha \in {\rm{H}}^i(\A, G)$ and $\beta \in {\rm{H}}^{2-i}(\A, \widehat{G})$, the element $\langle \alpha_v, \beta_v \rangle \in {\rm{H}}^2(k_v, \Gm)$ lifts to ${\rm{H}}^2(\calO_v, \Gm) = 0$ for all but finitely many $v$, as desired.

\begin{proposition}
\label{G(A)_pro=prodG(k_v)_pro}
Let $k$ be a global field, $G$ an affine commutative $k$-group scheme of finite type. Then the natural map
\[
\widehat{G}(\A)_{\pro} \rightarrow \prod_v \widehat{G}(k_v)_{\pro}
\]
is a topological isomorphism.
\end{proposition}

\begin{proof}
By Proposition \ref{adelictopcohombasicsG^}(ii), the map
\[
\widehat{G}(\A) \longrightarrow \prod_v \widehat{G}(k_v)
\]
is an inclusion which identifies $\widehat{G}(\A)$ -- as a topological group -- with its image. For each place $v_0$ of $k$, we have $\A = k_{v_0} \times \A^{v_0}$, hence $\widehat{G}(k_{v_0}) \times \prod_{v \neq v_0} 0 \subset \widehat{G}(\A)$. The proposition therefore follows from Proposition \ref{completionofproduct}.
\end{proof}

The following result is the adelic analogue of our main local duality theorems.

\begin{proposition}
\label{H^1(A,G)=H^1(A,G^)^D}
Let $G$ be an affine commutative group scheme of finite type over a global function field $k$. Then the adelic pairings induce perfect pairings of locally compact Hausdorff abelian groups
\[
{\rm{H}}^2(\A, G) \times {\rm{H}}^0(\A, \widehat{G})_{\pro} \rightarrow \Q/\Z,
\]
\[
{\rm{H}}^1(\A, G) \times {\rm{H}}^1(\A, \widehat{G}) \rightarrow \Q/\Z,
\]
\[
{\rm{H}}^0(\A, G)_{\pro} \times {\rm{H}}^2(\A, \widehat{G}) \rightarrow \Q/\Z,
\]
\end{proposition}

\begin{proof}
Since the adelic pairings are given by summing the local duality pairings, we will deduce these results from the local duality theorems. First, to show that the pairing
\[
{\rm{H}}^1(\A, G) \times {\rm{H}}^1(\A, \widehat{G}) \rightarrow \Q/\Z
\]
is perfect, let $\mathscr{G}$ denote an $\calO_S$-model of $G$ for some finite non-empty set $S$ of places of $k$. Proposition \ref{adelictopcohombasics}(iii) identifies ${\rm{H}}^1(\A, G)$ with the restricted product $\prod'_v {\rm{H}}^1(k_v, G)$ with respect to the subgroups ${\rm{H}}^1(\calO_v, \mathscr{G})$, and Proposition \ref{adelictopcohombasicsG^}(iii) identifies ${\rm{H}}^1(\A, \widehat{G})$ with the restricted product $\prod'_v {\rm{H}}^1(k_v, \widehat{G})$ with respect to the subgroups ${\rm{H}}^1(\calO_v, \widehat{\mathscr{G}})$. Theorem \ref{H^1dualityprop} identifies ${\rm{H}}^1(k_v, G)$ and ${\rm{H}}^1(k_v, \widehat{G})$ as Pontryagin dual groups under the local duality pairing, and Theorem \ref{H^1(k,G^)/H^1(O,G^)} identifies ${\rm{H}}^1(\calO_v, \widehat{\mathscr{G}})$ as the annihilator of ${\rm{H}}^1(\calO_v, \mathscr{G}) \subset {\rm{H}}^1(k_v, G)$, which is a compact subgroup by Remark \ref{H^1(O_v,G)compact}, and is an open subgroup by \cite[Prop.\,2.9(e)]{ces1}. Thus, \cite[Ch.\,XV, Thm.\,3.2.1]{cf} completes the proof of the perfection of this pairing.

Next, we turn to the perfection of
\[
{\rm{H}}^2(\A, G) \times {\rm{H}}^0(\A, \widehat{G})_{\pro} \rightarrow \Q/\Z.
\]
Proposition \ref{G(A)_pro=prodG(k_v)_pro} identifies $\widehat{G}(\A)_{\pro}$ with $\prod_v \widehat{G}(k_v)_{\pro}$, while \cite[Thm.\,2.18]{ces2} identifies ${\rm{H}}^2(\A, G)$ with $\oplus {\rm{H}}^2(k_v, G)$. Therefore, Theorem \ref{H^2(G)G^(k)dualityprop} identifies the two groups as Pontryagin duals via the summation of the local duality pairings. Note that this is even a topological identification because the compact group ${\rm{H}}^0(\A, \widehat{G})_{\pro}$ has discrete dual.

Finally, we prove the perfection of
\[
{\rm{H}}^0(\A, G)_{\pro} \times {\rm{H}}^2(\A, \widehat{G}) \rightarrow \Q/\Z.
\]
The map
\[
{\rm{H}}^0(\A, G) \longrightarrow {\prod_v}' {\rm{H}}^0(k_v, G)
\]
is a topological isomorphism, where the restricted product is with respect to the subgroups ${\rm{H}}^0(\calO_v, \mathscr{G})$. Proposition \ref{restrictedprodprofinitecomp} therefore implies that we obtain via the natural map an isomorphism
\[
({\rm{H}}^0(\A, G)_{\pro})^D \xlongrightarrow{\sim} ({\prod_v}' ({\rm{H}}^0(k_v, G)_{\pro})^D)_{\tors},
\]
where the restricted product is with respect to the annihilators of the integral cohomology subgroups ${\rm{H}}^0(\calO_v, \mathscr{G})$. By Theorems \ref{H^2(G^)altori} and \ref{H^2(O,G^)dualityprop}, therefore, we obtain via the the sum of the local duality pairings an isomorphism
\[
({\rm{H}}^0(\A, G)_{\pro})^D \xlongrightarrow{\sim} ({\prod_v}' {\rm{H}}^2(k_v, \widehat{G}))_{\tors},
\]
where the restricted product is with respect to the groups ${\rm{H}}^2(\calO_v, \widehat{\mathscr{G}})$. Proposition \ref{H^i(A,G^)--->prodH^i(k_v,G^)}(ii) now completes the proof, since the Pontryagin dual of the compact group ${\rm{H}}^0(\A, G)_{\pro}$ is discrete.
\end{proof}

\section{Vanishing of ${{\rm{H}}}^3(\A, \widehat{G})$}

In this section, we will prove that ${\rm{H}}^3(\A_k, \widehat{G}) = 0$ for any affine commutative group scheme of finite type over a global function field $k$ (Proposition \ref{H^3(A,G^)=0prop}). We begin by introducing the notion of almost-isomorphisms. We call a homomorphism between abelian groups an {\em almost-isomorphism} if its kernel and cokernel are of finite exponent. We say that $A$ is almost isomorphic to $B$ if there is an almost-isomorphism $A \rightarrow B$. This is an equivalence relation. The only point that is not clear is symmetry. Suppose given an almost-isomorphism $\phi \colon A \rightarrow B$. Let $n, m > 0$ be such that $n$ kills $\coker(\phi)$ and $m$ kills $\ker(\phi)$. Then define a map $\psi \colon B \rightarrow A$ as follows. Let $b \in B$. Since $n$ kills $\coker(\phi)$, we have $nb = \phi(a)$ for some $a \in A$. Although $a$ is not necessarily unique, since $m$ kills $\ker(\phi)$, the element $ma \in A$ is unique. Then define $\psi(b) := ma$. One easily checks that $\psi$ is a homomorphism. To see that it is an almost-isomorphism, we first check that $mn$ kills $\coker(\psi)$. Let $a \in A$. Then set $b := \phi(a)$. Then $\phi(na) = nb$, hence $\psi(b) = (mn)a$. Next, we check that $mn$ kills $\ker(\psi)$. Indeed, let $b \in B$ be such that $\psi(b) = 0$. Then, for some $a \in A$ such that $ma = 0$, we have $\phi(a) = nb$. Therefore, $(mn)b = \phi(ma) = 0$.

We say that an abelian group $A$ is {\em quasi-divisible} if $A/A_{{\rm{div}}}$ is torsion. The following simple lemma says that quasi-divisibility is an almost-isomorphism invariant.

\begin{lemma}
\label{quasidivalmostisom}
If $A$ and $B$ are almost-isomorphic abelian groups such that $A$ is quasi-divisible, then $B$ is also quasi-divisible.
\end{lemma}

\begin{proof}
Let $\phi\colon A \rightarrow B$ be an almost isomorphism. Choose a positive integer $n$ which kills $\coker(\phi)$. Let $b \in B$. We must show that $mb \in B_{\rm{div}}$ for some $m > 0$. We have $nb = \phi(a)$ for some $a \in A$. For some $r > 0$, $ra \in A_{\rm{div}}$. Then $\phi(ra) = (rn)b \in B_{{\rm{div}}}$, so we may take $m = rn$.
\end{proof}

Quasi-divisible groups have the following desirable property.

\begin{lemma}
\label{quasidivrightexact}
Given an exact sequence of abelian groups
\[
A' \xlongrightarrow{\psi} A \xlongrightarrow{\phi} A'' \longrightarrow 1,
\]
if $A'$ is quasi-divisible, then the induced map
\[
A_{\tors} \longrightarrow A''_{\tors}
\]
is surjective.
\end{lemma}

\begin{proof}
Let $a'' \in A''_{\tors}$, say $ma'' = 0$. Choose $a \in A$ such that $\phi(a) = a''$. Then $\phi(ma) = 0$, so $ma \in \psi(A')$. Since $A'$ is quasi-divisible, we have $(nm)a \in \psi(A'_{{\rm{div}}})$ for some $n > 0$, so $(nm)a = \psi((nm)b')$ for some $b' \in A'$. Then $nm(a - \psi(b')) = 0$, so $a - \psi(b') \in A_{\tors}$, while $\phi(a - \psi(b')) = \phi(a) = a''$, so the proof is complete.
\end{proof}

The relevance of quasi-divisibility for us lies in the following lemma.

\begin{lemma}
\label{prodalmostdivisible}
Let $G$ be an affine commutative group scheme of finite type over a global function field $k$, and let $\mathscr{G}$ be an $\calO_S$-model of $G$ for some finite set $S$ of places of $k$. Then the group ${\prod_v}' {\rm{H}}^2(k_v, \widehat{G})$ is quasi-divisible, where the restricted product is with respect to the groups ${\rm{H}}^2(\calO_v, \widehat{\mathscr{G}})$.
\end{lemma}

\begin{proof}
Thanks to Proposition \ref{hatisexact} and Lemmas \ref{spreadingoutdualsheaves} and \ref{quasidivalmostisom}, if we have a short exact sequence
\[
1 \longrightarrow G' \longrightarrow G \longrightarrow G'' \longrightarrow 1
\]
of affine commutative group schemes of finite type such that one of the groups $G'$ or $G''$ is of finite exponent, and the lemma holds for another of the groups, then it also holds for the third. Proposition \ref{affinegroupstructurethm} therefore reduces us to the case in which $G$ is an almost-torus, and Lemma \ref{almosttorus}(iv) then reduces us to the case in which $G$ is the product of a finite group scheme and a group of the form $\R_{k'/k}(\Gm^n)$ for some finite separable extension $k'/k$. Finite group schemes are of finite exponent, so we are reduced to the case $G = \R_{k'/k}(\Gm^n)$, and we may assume that $n = 1$.

By Proposition \ref{charactersseparableweilrestriction}, we must show that ${\prod_v}' {\rm{H}}^2(k_v, \R_{k'/k}(\Z))$ is almost-divisible, where the restricted product is with respect to the subgroups ${\rm{H}}^2(\calO_v, \R_{\calO_{S'}/\calO_S}(\Z))$, where $S$ is the set of places of $k$ which ramify in $k'$. Since $\R_{k'/k}(\Z)$ and $\R_{\calO_{S'}/\calO_S}(\Z)$ are smooth group schemes, we may take our cohomology to be \'etale. Since finite pushforward is exact between categories of \'etale sheaves, we have
\[
{\prod_v}' {\rm{H}}^2(k_v, \R_{k'/k}(\Z)) = {\prod_{v'}}' {\rm{H}}^2(k'_{v'}, \Z),
\]
where the restricted product on the right is over the places $v'$ of $k'$ (and is with respect to the subgroups ${\rm{H}}^2(\calO_{v'}, \Z)$). Renaming $k'$ as $k$, therefore, we are reduced to the case in which $k' = k$. That is, we must show that the group ${\prod_v}' {\rm{H}}^2(k_v, \Z)$ is quasi-divisible.

In order to show this, in turn, it suffices to show that the groups ${\rm{H}}^2(\calO_v, \Z)$ are all divisible, and that each of the groups ${\rm{H}}^2(k_v, \Z)$ is quasi-divisible. We first check that ${\rm{H}}^2(\calO_v, \Z)$ is divisible. Using the exact sequence
\begin{equation}
\label{prodalmostdivisiblepfeqn1}
1 \longrightarrow \Z \xlongrightarrow{[n]} \Z \longrightarrow \Z/n\Z \longrightarrow 1,
\end{equation}
it suffices to show that ${\rm{H}}^2(\calO_v, \Z/n\Z) = 0$. Letting $\F_v$ denote the residue field of $\calO_v$, we have ${\rm{H}}^2(\calO_v, \Z/n\Z) \simeq {\rm{H}}^2(\F_v, \Z/n\Z)$ by \cite[Thm.\,11.7, 2]{briii}, and the latter group vanishes because finite fields have cohomological dimension 1.

Next, we check that ${\rm{H}}^2(k_v, \Z)$ is quasi-divisible. Let $N$ be a positive integer which kills the finite group $\mu(k_v)$ of roots of unity lying in $k_v$. Then we claim that $N\alpha \in {\rm{H}}^2(k_v, \Z)_{{\rm{div}}}$ for all $\alpha \in {\rm{H}}^2(k_v, \Z)$. Thanks to the exact sequence (\ref{prodalmostdivisiblepfeqn1}), it suffices to show that $N$ kills ${\rm{H}}^2(k_v, \Z/n\Z)$ for all $n > 0$. By local duality (Theorem \ref{H^2(G)G^(k)dualityprop}; in fact, we only use here the classical case of finite group schemes), it is equivalent to show that $N$ kills $\mu_n(k_v)$, which it does by our choice of $N$.
\end{proof}

We now prove the following lemma, which must hold if ${\rm{H}}^3(\A_k, \widehat{G})$ is to vanish for all affine commutative $G$ of finite type over $k$.

\begin{lemma}
\label{surjH^2(A,G^)}
Given an inclusion $G' \hookrightarrow G$ of affine commutative group schemes of finite type over a global field $k$, the induced map ${\rm{H}}^2(\A_k, \widehat{G}) \rightarrow {\rm{H}}^2(\A_k, \widehat{G'})$ is surjective.
\end{lemma}

\begin{proof}
Let $G'' := G/G'$, an affine commutative group scheme of finite type \cite[Ch.\,III, \S3, no.\,5, Th.\,5.6]{demazuregabriel}, and choose finite type $\calO_S$-models $G, G'$, and $G''$ of $\mathscr{G}, \mathscr{G}'$, and $\mathscr{G}''$, respectively. Then, thanks to Propositions \ref{hatisexact}, \ref{spreadingoutdualsheaves}, \ref{H^3(G^)=0}, and \ref{H^3(O, G^)=0}, the sequence
\[
{\prod_v}' {\rm{H}}^2(k_v, \widehat{G''}) \longrightarrow {\prod_v}' {\rm{H}}^2(k_v, \widehat{G}) \longrightarrow {\prod_v}' {\rm{H}}^2(k_v, \widehat{G'}) \longrightarrow 1
\]
is exact. It follows from Lemmas \ref{prodalmostdivisible} and \ref{quasidivrightexact} that the induced map
\[
({\prod_v}' {\rm{H}}^2(k_v, \widehat{G}))_{\tors} \longrightarrow ({\prod_v}' {\rm{H}}^2(k_v, \widehat{G'}))_{\tors}
\]
is surjective. The lemma now follows from Proposition \ref{H^i(A,G^)--->prodH^i(k_v,G^)}(ii).
\end{proof}

We now come to the main result of this section.

\begin{proposition}
\label{H^3(A,G^)=0prop}
For any affine commutative group scheme of finite type over a global function field $k$, we have ${\rm{H}}^3(\A_k, \widehat{G}) = 0$.
\end{proposition}

\begin{proof}
Proposition \ref{hatisexact} and Lemmas \ref{affinegroupstructurethm} and \ref{almosttorus}(ii) reduce us to the cases in which $G$ is $\Ga$, finite, or a torus. The case $G = \Ga$ is Proposition \ref{H^3(A,Ga^)=0}, while the case of finite $G$ follows from \cite[Thm.\,2.18]{ces2} together with Proposition \ref{H^3(G^)=0}. We are therefore reduced to the case in which $G = T$ is a torus.

By Lemma \ref{almosttorus}(iv), we may harmlessly modify $T$ in order to ensure that there is an exact sequence
\[
1 \longrightarrow A \longrightarrow X \longrightarrow T \longrightarrow 1,
\]
where $X = B \times \R_{k'/k}(\Gm^n)$, $A$ and $B$ are finite commutative $k$-group schemes, $k'/k$ is a finite separable extension, and $n \geq 0$ is a nonnegative integer. By Lemma \ref{surjH^2(A,G^)}, the induced map ${\rm{H}}^2(\A, \widehat{X}) \rightarrow {\rm{H}}^2(\A, \widehat{A})$ is surjective. Since we already know the desired vanishing for finite commutative group schemes, we are therefore reduced to showing it for $\R_{k'/k}(\Gm^n)$. We may assume that $n = 1$.

By Proposition \ref{charactersseparableweilrestriction}, we have $\widehat{\R_{k'/k}(\Gm)} \simeq \R_{k'/k}(\Z)$. By Lemma \ref{sepblepushforward}, ${\rm{H}}^3(\A_k, \R_{k'/k}(\Z)) \simeq {\rm{H}}^3(\A_{k'}, \Z)$. Renaming $k'$ as $k$, therefore, we are reduced to showing that ${\rm{H}}^3(\A_k, \Z) = 0$.

For integers $n, d > 0$, we have the commutative diagram
\[
\begin{tikzcd}
{\rm{H}}^2(\A, \Z/n\Z) \arrow{r}{[d]} \isoarrow{d} & {\rm{H}}^2(\A, \Z/nd\Z) \isoarrow{d} \\
({\rm{H}}^0(\A, \mu_n)_{\pro})^D \arrow{r}{[d]^D} & ({\rm{H}}^0(\A, \mu_{nd})_{\pro})^D
\end{tikzcd}
\]
in which the vertical arrows are the isomorphisms provided by the adelic duality pairings (Proposition 
\ref{H^1(A,G)=H^1(A,G^)^D}). Therefore,
\begin{equation}
\label{H^3(A,G^)=0proppfeqn2}
\varinjlim_n {\rm{H}}^2(\A, \Z/n/\Z) \simeq \varinjlim_n ({\rm{H}}^0(\A, \mu_n)_{\pro})^D.
\end{equation}
By Proposition \ref{G(A)_pro=prodG(k_v)_pro} applied to the Cartier dual of $\mu_n$, the natural map
\[
{\rm{H}}^0(\A, \mu_n)_{\pro} \xlongrightarrow{\sim} \prod_v \mu_n(k_v)
\]
is an isomorphism, since $\mu_n(k_v)$ is finite and discrete, hence profinite. Therefore,
\begin{equation}
\label{H^3(A,G^)=0proppfeqn1}
({\rm{H}}^0(\A, \mu_n)_{\pro})^D = \displaystyle \oplus_v \mu_n(k_v)^D.
\end{equation}
Since the group $\mu(k_v)$ of roots of unity lying in $k_v$ is finite for each $v$, it follows that, for each $v$, there is $d_v > 0$ such that the transition map $[d_v]: \mu_{nd}(k_v) \rightarrow \mu_n$ vanishes for all $n$. Hence, $\varinjlim_n \mu_n(k_v)^D = 0$. Equation (\ref{H^3(A,G^)=0proppfeqn1}) then implies that
\[
\varinjlim_n ({\rm{H}}^0(\A, \mu_n)_{\pro})^D = 0,
\]
hence, by (\ref{H^3(A,G^)=0proppfeqn2}),
\begin{equation}
\label{H^3(A,G^)=0proppfeqn3}
\varinjlim_n {\rm{H}}^2(\A, \Z/n/\Z) = 0,
\end{equation}
where recall that the transition maps in the above filtered system are $[d]: \Z/n\Z \rightarrow \Z/nd\Z$. Via the isomorphisms $\frac{1}{n}\Z \xrightarrow{\sim} \Z/n\Z$ given by multiplication by $n$, the above filtered system is isomorphic to the filtered system $\{\frac{1}{n}\Z/\Z\}$ with inclusions as transition maps. Proposition \ref{directlimitscohom} and (\ref{H^3(A,G^)=0proppfeqn3}) therefore imply that
\[
{\rm{H}}^2(\A, \Q/\Z) = 0.
\]
It now follows from the exact sequence
\[
0 \longrightarrow \Z \longrightarrow \Q \longrightarrow \Q/\Z \longrightarrow 0
\]
and Lemma \ref{H^i(O_S^,Q)=0} that ${\rm{H}}^3(\A, \Z) = 0$. This completes the proof of the proposition.
\end{proof}

\section{Exactness properties for profinite completions of rational points}

In this section we study the exactness properties of the sequence of profinite completions of groups of rational and adelic points induced by a short exact of affine commutative group schemes of finite type. The main results (Propositions \ref{localpointsprofinite} and \ref{globalpointsprofinite}) state that these sequences are left-exact, just as with the uncompleted sequences of points. These results will play an important role in the proofs that the global duality sequence in Theorem \ref{poitoutatesequence} is exact at ${\rm{H}}^0(k, G)_{\pro}$ and ${\rm{H}}^0(\A, G)_{\pro}$.

Here is our first main exactness result for profinite completions. (For the definition of the profinite completion of an abelian topological group, see Remark \ref{profinitecompletiondef}.)

\begin{proposition}
\label{localpointsprofinite}
Let $k$ be a local function field, and suppose that we have an exact sequence
\[
1 \longrightarrow G' \longrightarrow G \longrightarrow G'' \longrightarrow 1
\]
of affine commutative $k$-group schemes of finite type. Then the induced sequence
\[
1 \longrightarrow G'(k)_{\pro} \longrightarrow G(k)_{\pro} \longrightarrow G''(k)_{\pro}
\]
is exact.
\end{proposition}

\begin{proof}
Perhaps one can prove this directly, but it is easier at this stage to simply use local duality. It suffices to check that the
discrete Pontryagin dual sequence is algebraically exact. By local duality (in particular, Theorem \ref{H^2(G^)altori}), this is equivalent to the exactness of the sequence
\[
{\rm{H}}^2(k, \widehat{G''}) \longrightarrow {\rm{H}}^2(k, \widehat{G}) \longrightarrow {\rm{H}}^2(k, \widehat{G'}) \longrightarrow 0.
\]
This exactness follows from Propositions \ref{hatisexact} and \ref{H^3(G^)=0}.
\end{proof}

\begin{lemma}
\label{G(k)divtors}
Let $k$ be a global function field, $G$ a commutative $k$-group scheme of finite type. Then $G(k)_{\Div} = 0$ and $G(k)_{\tors}$ 
has finite exponent.
\end{lemma}

\begin{proof}
By Lemma \ref{nondivisibleextension}, if we have an exact sequence of commutative $k$-group schemes of finite type
\[
1 \longrightarrow G' \longrightarrow G \longrightarrow G'' \longrightarrow 1,
\]
and the lemma holds for $G', G''$, then it also holds for $G$. Note also that the lemma is clear for groups of finite exponent, so in particular for finite group schemes.

By \cite[VII$_{\rm{A}}$, Prop.\,8.3]{sga3}, there is a normal infinitesimal $k$-subgroup scheme $I \subset G$ such that $G/I$ is smooth, so we may assume that $G$ is smooth. Filtering $G$ by $G^0$ and $G/G^0$, we may also assume that $G$ is connected. There is therefore an anti-affine smooth connected $k$-group $G_{\rm{ant}} \subset G$ ({\em anti-affine} means that 
${\rm{H}}^0(G_{{\rm{ant}}}, \mathcal{O}_{G_{{\rm{ant}}}}) = k$) such that $G/G_{{\rm{ant}}}$ is affine \cite[Thm.\,A.3.9]{cgp}. We may therefore assume that $G$ is either affine or anti-affine. By \cite[Thm.\,A.3.9]{cgp} again, any anti-affine $k$-group is a semi-abelian variety. We may therefore assume that $G$ is either an abelian variety or affine. If $G$ is an abelian variety then $G(k)$ is finitely generated by the Mordell--Weil theorem \cite[Cor.\,7.2]{chowtrace}, so the lemma is clear. If $G$ is affine, then 
for the maximal $k$-torus $T \subset G$ the quotient $G/T$ is unipotent, so we may assume that $G$ is either a torus or unipotent. In the latter case, $G$ has finite exponent (this is where we use that ${\rm{char}}(k) > 0$), hence the lemma is clear. So we may assume that $G = T$ is a torus.

Let $k'/k$ be a finite separable extension that splits $T$. Then thanks to the inclusion $T(k) \hookrightarrow T(k')$, we may (renaming $k'$ as $k$) assume that $T$ is split. That is, we need to show that $(k^{\times})_{\Div} = 0$ and that $(k^{\times})_{\tors}$ has finite exponent. If $\lambda \in (k^{\times})_{\Div}$, then $\ord_v(\lambda) = 0$ for all non-archimedean places $v$ of $k$.
Thus, $\lambda$ is a global unit. Since the group of global units is finitely generated, and an 
element of $k^{\times}$ is a global unit whenever its $n$th power is for any fixed $n > 0$, 
we deduce that $\lambda = 1$. Finally, as is well-known, $(k^{\times})_{\tors} = \mu_{\infty}(k)$ is finite.
\end{proof}

\begin{remark}
Lemma \ref{G(k)divtors} is false for number fields (and in fact for any characteristic $0$ field). For example, $\Ga(k) = k$ is divisible. The result remains true, however, for groups whose reduced geometric identity component is a semi-abelian variety. The proof is basically the same as the one given above.
\end{remark}

Recall that if $X$ is a scheme over a global field $k$, then we endow $X(k)$ with the discrete topology.

\begin{proposition}
\label{globalpointsprofinite}
Let $k$ be a global function field, and suppose that we have an exact sequence
\[
1 \longrightarrow G' \longrightarrow G \longrightarrow G'' \longrightarrow 1
\]
of affine commutative $k$-group schemes of finite type. Then the induced sequences
\[
1 \longrightarrow G'(k)_{\pro} \longrightarrow G(k)_{\pro} \longrightarrow G''(k)_{\pro}
\]
and
\[
1 \longrightarrow G'(\A)_{\pro} \longrightarrow G(\A)_{\pro} \longrightarrow G''(\A)_{\pro}
\]
are exact.
\end{proposition}

\begin{proof}
We first show that the adelic sequence is exact. It suffices to check that the
discrete Pontryagin dual sequence is algebraically exact. By Proposition \ref{H^1(A,G)=H^1(A,G^)^D}, this is equivalent to exactness of
\[
{\rm{H}}^2(\A, \widehat{G''}) \longrightarrow {\rm{H}}^2(\A, \widehat{G}) \longrightarrow {\rm{H}}^2(\A, \widehat{G'}) \longrightarrow 1.
\]
which follows from Propositions \ref{hatisexact} and \ref{H^3(A,G^)=0prop}.

We next check exactness for the completed sequence of $k$-points. Let us first show that the map $G'(k)_{\pro} \rightarrow G(k)_{\pro}$ is injective. By Proposition \ref{profiniteexact} applied to the discrete groups of rational
points over the global function field $k$, it suffices to show that $(G(k)/G'(k))_{\Div} = 0$ and that $(G(k)/G'(k))_{\tors}$ has finite exponent. Since $G(k)/G'(k) \hookrightarrow G''(k)$, this follows from Lemma \ref{G(k)divtors}.

Next we check exactness in the middle. Let $f$ denote the map $G \rightarrow G''$. We have an exact sequence
\[
1 \longrightarrow G'(k) \longrightarrow G(k) \longrightarrow f(G(k)) \longrightarrow 1.
\]
Since profinite completion is right-exact (Proposition \ref{profiniterightexact}), this yields an exact sequence
\[
G'(k)_{\pro} \longrightarrow G(k)_{\pro} \longrightarrow f(G(k))_{\pro} \longrightarrow 1.
\]
If we show that the map $f(G(k))_{\pro} \rightarrow G''(k)_{\pro}$ is injective, it will follow that the sequence
\[
G'(k)_{\pro} \longrightarrow G(k)_{\pro} \longrightarrow G''(k)_{\pro}
\]
is exact, and the proof will be complete. 

To see this injectivity, we note that we have an inclusion $G''(k)/f(G(k)) \hookrightarrow {\rm{H}}^1(k, G')$, hence by Lemma \ref{H^1finiteexponent}, $G''(k)/f(G(k))$ has finite exponent. The desired injectivity therefore follows from Proposition \ref{profiniteexact} (applied to discrete topological groups). 
\end{proof}

\section{Exactness at ${\rm{H}}^0(k, G)_{\pro}$} 

The goal of this section is to prove that if $G$ is an affine commutative group scheme of finite type over a global function field, then the Tate duality sequence of Theorem \ref{poitoutatesequence} is exact at $G(k)_{\pro}$. That is, the map $G(k)_{\pro} \rightarrow G(\A)_{\pro}$ is injective (Proposition \ref{G(k)G(A)proinjective}). This remains true over number fields, but the proof must be slightly modified. We briefly indicate the required changes in Remark \ref{G(k)_pronofieldsrmk}.

\begin{lemma}
\label{k*A*proinjective}
For a global function field $k$, the map $(k^{\times})_{\pro} \rightarrow (\A^{\times})_{\pro}$ is injective.
\end{lemma}

\begin{proof}
Choose a finite non-empty set $S$ of places of $k$ such that $\Pic(\mathcal{O}_S) = 0$. Since $\Pic(\calO_S) = 0$, the divisor associated to each place $v \notin S$ is principal. That is, for each$v \notin S$, there is an element $\pi_v \in k_v^{\times}$ having valuation $1$ at $v$ and $0$ at all other places not in $S$. We therefore have a short exact sequence
\[
0 \longrightarrow \mathcal{O}_S^{\times} \longrightarrow k^{\times} \longrightarrow \bigoplus_{v \notin S} \Z \longrightarrow 0,
\]
in which the map from $k^{\times}$ is given by taking valuations at each place $v \notin S$. The elements $\pi_v$ yield a splitting of the sequence, hence an isomorphism $k^{\times} \simeq \mathcal{O}_S^{\times} \times \bigoplus_{v \notin S} \Z$. It suffices to show that for any subgroups $X \subset \mathcal{O}_S^{\times}$, $Y \subset \bigoplus_{v \notin S} \Z$ of finite index, there exist closed subgroups $X', Y' \subset \A^{\times}$ of finite index such that $X' \cap k^{\times} \subset X \times \bigoplus_{v \notin S} \Z$, and $Y' \cap k^{\times} \subset \mathcal{O}_S^{\times} \times Y$
(as then $X' \cap Y'$ is closed of finite index in $\A^{\times}$ and meets $k^{\times}$ inside
$X \times Y$). We first treat the case of subgroups of $\bigoplus_{v \notin S} \Z$. 

The map $k^{\times} \rightarrow \bigoplus_{v \notin S} \Z$ factors as a composition $k^{\times} \rightarrow \A^{\times} \rightarrow \bigoplus_{v \notin S} \Z$ whose second map has open kernel. Given a subgroup $Y \subset \bigoplus_{v \notin S} \Z$ of finite index, we simply let $Y'$ be its preimage in $\A^{\times}$.

Now we treat finite-index subgroups of $\mathcal{O}_S^{\times}$. This group is finitely generated, so cofinal among these subgroups are the groups $(\mathcal{O}_S^{\times})^n$ for $n$ a positive integer. We are free to replace $n$ with some positive multiple of itself, hence we may choose some $v_0 \notin S$ and assume that $(\mathcal{O}_S^{\times})^n \subset 1 + \pi_{v_0} \mathcal{O}_{v_0}$. Let $C$ be the smooth proper geometrically connected curve (over a finite field $\F$) such that $k$ is the function field of $C$, and let $N$ be the exponent of the finite group $\Pic^0(C)$. Let $X' \subset \A^{\times}$ be the closed finite index subgroup 
of ideles $a \in \A^{\times}$ such that $\ord_w(a_w/\prod_{v \notin S} \pi_v^{\ord_v(a_v)}) \equiv 0 \pmod{nN}$ for all $w \in S$
and $a_{v_0}/\prod_{v \notin S} \pi_{v}^{\ord_{v}(a_v)} \equiv 1 \pmod{\pi_{v_0}}$.

We claim that $X' \cap k^{\times} \subset (\mathcal{O}_S^{\times})^n \times \bigoplus_{v \notin S} \Z$ via the fixed isomorphism $k^{\times} \simeq \mathcal{O}_S^{\times} \times \bigoplus_{v \notin S} \Z$. Indeed, given an $f \in X' \cap k^{\times}$, let $g : = f/\prod_{v \notin S} \pi_v^{\ord_v(f)} \in \mathcal{O}_S^{\times}$. Then $g \equiv 1 \pmod{\pi_{v_0}}$ and ${\rm{div}}(g) \in nN\cdot{\rm{Div}}^0(C)$, where ${\rm{Div}}^0(C)$ denotes the group of degree-0 divisors on $C$. 
Note that $g \in \calO_S^{\times}$ is the image of $f \in k^{\times}$ under the projection $k^{\times} \rightarrow \calO_S^{\times}$ corresponding to the isomorphism $k^{\times} \simeq \calO_S^{\times} \times \oplus_{v \notin S} \Z$. By design, $N$ kills $\Pic^0(C)$, so ${\rm{div}}(g) = {\rm{div}}(h^n)$ for some $h \in  k^{\times}$, 
and $h$ must be an $S$-unit since $h^n$ is one (as $g$ is an $S$-unit). Since $h^n \equiv 1 \pmod{\pi_{v_0}}$ by our choice of $n$, we deduce that $g/h^n \equiv 1 \pmod{\pi_{v_0}}$. But $g/h^n$ is a global unit of $C$, hence lies in $\F^{\times}$, and the map $\F \rightarrow \mathcal{O}_v/\pi_v\mathcal{O}_v$ is an inclusion, so $g = h^n$, and we are done.
\end{proof}

\begin{remark}
If $k$ is a number field, then one proceeds similarly, but must use a theorem of Chevalley 
 \cite[Thm.\,1]{chevalley}
that every finite-index subgroup of $\mathcal{O}_k^{\times}$ is a congruence subgroup; that is, it contains a subgroup of the form $\{ x \in \mathcal{O}_k^{\times} \mid x \equiv 1 \pmod{\alpha} \}$ for some nonzero $\alpha \in \mathcal{O}_k$. 
\end{remark}

\begin{lemma}
\label{finiteprofinite}
Let $k$ be a global field, $G$ a finite $k$-group scheme. Then $G(k)$ and $G(\A)$ are profinite.
\end{lemma}

\begin{proof}
For $G(k)$, this is obvious because it is finite. For $G(\A)$, we claim that the natural map $G(\A) \rightarrow \prod_v G(k_v)$ is a topological isomorphism where the target is given the product topology, hence is profinite because each $G(k_v)$ is finite discrete. To see that this map is a topological isomorphism, we spread $G$ out to a finite $\mathcal{O}_S$-group scheme $\mathscr{G}$ for some non-empty finite set $S$ of places of $k$ containing the archimedean places.
For $v \notin S$, we have $\mathscr{G}(\mathcal{O}_v) = G(k_v)$ by the valuative criterion for properness. 
Thus, topologically, $$\textstyle \prod_v G(k_v) = \left(\prod_{v \in S} G(k_v) \times \prod_{v \not\in S} \mathscr{G}(\mathcal{O}_v)\right) =
G\left(\prod_{v \in S} k_v\right) \times \mathscr{G}\left(\prod_{v \not\in S} \mathcal{O}_v\right).$$
It follows that the map is indeed an isomorphism of topological groups.
\end{proof}

\begin{proposition}
\label{G(k)G(A)proinjective}
Let $k$ be a global function field, $G$ an affine commutative $k$-group scheme of finite type. Then the map $G(k)_{\pro} \rightarrow G(\A)_{\pro}$ is injective.
\end{proposition}

\begin{proof}
First suppose that we have an exact sequence
\[
1 \longrightarrow G' \longrightarrow G \longrightarrow G'' \longrightarrow 1
\]
of affine commutative $k$-group schemes of finite type. If the map $H(k)_{\pro} \rightarrow H(\A)_{\pro}$ is injective for $H = G', G''$, then we claim that the same holds for $H = G$.
Indeed, consider the commutative diagram
\[
\begin{tikzcd}
& G'(k)_{\pro} \arrow{r} \arrow[d, hookrightarrow] & G(k)_{\pro} \arrow{r} \arrow{d} & G''(k)_{\pro} \arrow[d, hookrightarrow] \\
1 \arrow{r} & G'(\A)_{\pro} \arrow{r} & G(\A)_{\pro} \arrow{r} & G''(\A)_{\pro}
\end{tikzcd}
\]
The rows are exact by Proposition \ref{globalpointsprofinite}, and 
the first and third vertical arrows are injective by hypothesis.  The five lemma shows that the middle vertical arrow is injective, which proves our claim.
It follows from Lemma \ref{affinegroupstructurethm} that we may assume that $G = \Ga$ or an almost-torus. By Lemmas \ref{almosttorus}(ii) and \ref{finiteprofinite}, we may therefore assume that $G$ is either $\Ga$ or a torus. First consider the case $G = \Ga$.

We need to show that the map $k_{\pro} \rightarrow \A_{\pro}$ is injective. It is equivalent to show that the 
associated map $(\A_{\pro})^D \rightarrow (k_{\pro})^D$ between discrete Pontryagin duals is surjective. But for any 
locally compact Hausdorff abelian topological group $X$ with finite exponent, the map $(X_{\pro})^D \rightarrow X^D$ is an (algebraic, not necessarily topological) isomorphism because the kernel of any continuous homomorphism $X \rightarrow 
\mathbf{R}/\Z$ is closed of finite index. We therefore want to show that the map $\A^D \rightarrow k^D$ is surjective, and this holds because $k$ is closed inside $\A$, due to \cite[Thm.\,2.1.2]{rudin}.

Finally, we come to the case when $G = T$ is a torus. First note that given an inclusion $T \hookrightarrow T'$
of $k$-tori, if Proposition \ref{G(k)G(A)proinjective} holds for $T'$, then it also holds for $T$ due to the commutative diagram
\[
\begin{tikzcd}
T(k)_{\pro} \arrow[r, hookrightarrow] \arrow{d} & T'(k)_{\pro} \arrow[d, hookrightarrow] \\
T(\A)_{\pro} \arrow{r} & T'(\A)_{\pro}
\end{tikzcd}
\]
in which the top horizontal arrow is an inclusion by Proposition \ref{globalpointsprofinite}, and the right vertical arrow is an inclusion by hypothesis. Let $k'/k$ be a finite separable extension splitting $T$. Using the inclusion $T \hookrightarrow \R_{k'/k}(T_{k'})$, and renaming $k'$ as $k$, we are reduced to showing that $(k^{\times})_{\pro} \rightarrow (\A^{\times})_{\pro}$ is injective, and this is the content of Lemma \ref{k*A*proinjective}.
\end{proof}

\begin{remark}
\label{G(k)_pronofieldsrmk}
We have throughout assumed that $k$ is a global function field. This is because we have made essential use of the fact that any unipotent $k$-group $U$ has finite exponent, a property that is completely false for number fields. However, one may modify the arguments in the case that $k$ is a number field by utilizing the fact that any connected commutative $k$-group scheme is the product of a torus and $\Ga^n$, by \cite[Cor.\,16.15]{milnealggroups}. Then one uses the fact that $\Ga(k)_{\pro} = k_{\pro} = 0$ because $k$ is divisible, and so one reduces everything to the case of tori, for which the proofs are essentially the same as in the function field case.
\end{remark}

\section{Exactness at ${\rm{H}}^0(\A, G)_{\pro}$}

In this section we prove that the global duality sequence of Theorem \ref{poitoutatesequence} is exact at ${\rm{H}}^0(\A, G)_{\pro}$ (Proposition \ref{exactnessatG(A)}). That is, if $H$ is an affine commutative group scheme of finite type over a global function field $k$, then we wish to show that the sequence
\begin{equation}
\label{G(A)_proexactseqeqn}
H(k)_{\pro} \longrightarrow H(\A)_{\pro} \longrightarrow {\rm{H}}^2(k, \widehat{H})^*,
\end{equation}
is exact, where the last map is induced by the global duality pairing $H(\A) \times {\rm{H}}^2(k, \widehat{H}) \rightarrow \Q/\Z$ which cups everywhere locally and then adds the invariants. As we have already mentioned previously, this sequence is a complex because the sum of the local invariants of a global Brauer class is $0$.

\begin{remark}
\label{G(k)-->G(A)embedding}
Let us note that if $G$ is an affine commutative group scheme of finite type over a global field $k$, then the map ${\rm{H}}^0(k, G) \rightarrow {\rm{H}}^0(\A, G)$, which is trivially continuous, is actually a topological embedding. Indeed, this is equivalent to saying that $G(k) \subset G(\A)$ is discrete, and this follows from the discreteness of $k \subset \A$ by embedding $G$ into some affine space.
\end{remark}

\begin{lemma}
\label{exactnessatG(A)devissage1}
Suppose that we have an exact sequence
\[
1 \longrightarrow G' \longrightarrow G \longrightarrow G'' \longrightarrow 1
\]
of affine commutative group schemes of finite type over a global field $k$ such that ${\rm{H}}^1(k, G') = 0$. If $(\ref{G(A)_proexactseqeqn})$ is exact for $H = G'$ and $H = G''$, then it is exact for $H = G$.
\end{lemma}

\begin{proof}
Consider the following commutative diagram: 
\[
\begin{tikzcd}
& G'(k)_{\pro} \arrow{r} \arrow{d} & G(k)_{\pro} \arrow{r} \arrow{d} & G''(k)_{\pro} \arrow{r} \arrow{d} & 0 \\
& G'(\A)_{\pro} \arrow{r} \arrow{d} & G(\A)_{\pro} \arrow{r} \arrow{d} & G''(\A)_{\pro} \arrow{d} & \\
0 \arrow{r} & {\rm{H}}^2(k, \widehat{G'})^* \arrow{r} & {\rm{H}}^2(k, \widehat{G})^* \arrow{r} & {\rm{H}}^2(k, \widehat{G''})^* &
\end{tikzcd}
\]
The first and the third columns are exact by hypothesis. Proposition \ref{globalpointsprofinite} implies the exactness of the first two rows except at $G''(k)_{\pro}$, where exactness holds because ${\rm{H}}^1(k, G') = 0$
(which implies that $G(k) \rightarrow G''(k)$ is surjective, hence so is the map on profinite completions). Finally, the bottom row is exact by Proposition \ref{H^3(G^)=0}. A diagram chase now shows that the middle column is exact.
\end{proof}

\begin{proposition}
\label{A/k=H^2(k,Ga^)*}
For a global field $k$ of positive characteristic, the map $\A_k/k \rightarrow {\rm{H}}^2(k, \widehat{\Ga})^*$ induced by the global duality pairing is a topological isomorphism.
\end{proposition}

\begin{proof}
The source is compact \cite[Ch.\,II, \S14, Thm.\,]{cf}, and the target is Hausdorff, so it suffices to show that the map is a continuous bijection. Continuity follows from Lemma \ref{G(A)toH^2(k,G^)*cts}. To show that the map is an (algebraic) isomorphism, we pass to the dual map and show that it is an isomorphism. The discrete group $(\A/k)^D$ dual to the compact
quotient $\A/k$ is the annihilator in $\A^D$ of $k$ and as a $k$-vector space is 1-dimensional by 
\cite[Ch.\,IV, \S2, Thm.\,3]{weilbasic}. Likewise, ${\rm{H}}^2(k, \widehat{\Ga})$ is 1-dimensional
for its natural $k$-vector space structure by Proposition \ref{omega=H^2(Ga^)}. Therefore, since the map is $k$-linear, in order to show that it is an isomorphism we only have to show that it is nonzero. That is, we want to show that the map $\A \rightarrow {\rm{H}}^2(k, \widehat{\Ga})^*$ is nonzero. 

In order to do this, it suffices to show that for a place $v$ of $k$, the image of the map ${\rm{H}}^2(k, \widehat{\Ga}) \rightarrow {\rm{H}}^2(k_v, \widehat{\Ga})$ is not killed by the local duality pairing with $k_v$ (since the global pairing is obtained by adding the local pairings). But by local duality (more precisely, by the special case given by Proposition \ref{k_pro=H^2(k,Ga^)*}), for this we only need to show that the map ${\rm{H}}^2(k, \widehat{\Ga}) \rightarrow {\rm{H}}^2(k_v, \widehat{\Ga})$ is nonzero. In fact, this map is injective by Lemma \ref{H^2(k,Ga^)injectiveseparable}, hence nonzero because ${\rm{H}}^2(k, \widehat{\Ga})$ is nonzero. (It is a 1-dimensional $k$-vector space.)
\end{proof}

\begin{lemma}
\label{exactnessatG(A)splittoriunip}
$(\ref{G(A)_proexactseqeqn})$ is exact if $H$ is a split torus or a split unipotent group over the global function field $k$.
\end{lemma}

\begin{proof}
First we treat split tori, so we may assume $G = \Gm$. Due to the right-exactness of profinite completion (Proposition \ref{profiniterightexact}), in conjunction with Remark \ref{G(k)-->G(A)embedding}, we have $(\A^{\times})_{\pro}/(k^{\times})_{\pro} = (\A^{\times}/k^{\times})_{\pro}$, so we want to show that the map $\A^{\times}/k^{\times} \rightarrow {\rm{H}}^2(k, \Z)^*$ induces an isomorphism from the profinite completion of the idele class group. Let $\mathfrak{g} : = \Gal(k_s/k)$. Via the exact sequence of Galois modules
\[
0 \longrightarrow \Z \longrightarrow \Q \longrightarrow \Q/\Z \longrightarrow 0,
\]
and the unique divisibility of $\Q$, we have ${\rm{H}}^2(k, \Z) = (\mathfrak{g}_k^{\ab})^D$
as torsion discrete topological groups, so ${\rm{H}}^2(k, \Z)^* = \mathfrak{g}_k^{\ab}$. Thus we have a map $\A^{\times}/k^{\times} \rightarrow \mathfrak{g}_k^{\ab}$, and we would like to say that it induces an isomorphism from the profinite completion of the idele class group. If this map were the reciprocity map of global class field theory, then the result would follow from class field theory. But it is indeed the reciprocity map, thanks to Lemma \ref{localreciprocitymap} and the compatibility of local and global class field theory.

Next we treat split unipotent groups. Lemma \ref{exactnessatG(A)devissage1} reduces us to the case of $\Ga$. Proposition \ref{A/k=H^2(k,Ga^)*} shows that the map $\A/k \rightarrow {\rm{H}}^2(k, \widehat{\Ga})^*$ is a topological isomorphism. Since $\A_{\pro}/k_{\pro} = (\A/k)_{\pro}$ by Proposition \ref{profiniterightexact} and Remark \ref{G(k)-->G(A)embedding}, to complete the proof we must show that $\A/k$ is profinite. It is compact, and it is Hausdorff because $k$ is closed in $\A$. It only remains to check that $\A/k$ is totally disconnected. By translation, and because $k \subset \A$ is discrete, it suffices to build a fundamental system of compact open neighborhoods of $0 \in \A$. For this, in turn, it suffices to construct a profinite open neighborhood of $0 \in \A$, for which we may take $\prod_v \mathcal{O}_v$.
\end{proof}

\begin{lemma}
\label{exactnessatG(A)devissage2}
For an inclusion $G \hookrightarrow G'$ of affine commutative group schemes of finite type  over a global field $k$, if $(\ref{G(A)_proexactseqeqn})$ is exact for $H = G'$, then it is exact for $H = G$.
\end{lemma}

\begin{proof}
Let $G'' : = G'/G$, an affine commutative group scheme of finite type \cite[Ch.\,III, \S3, no.\,5, Th.\,5.6]{demazuregabriel}. Consider the commutative diagram
\[
\begin{tikzcd}
&& 0 \arrow{d} & \\
& G(k)_{\pro} \arrow{r} \arrow{d} & G(\A)_{\pro} \arrow{r} \arrow{d} & {\rm{H}}^2(k, \widehat{G})^* \arrow{d} \\
& G'(k)_{\pro} \arrow{r} \arrow{d} & G'(\A)_{\pro} \arrow{r} \arrow{d} & {\rm{H}}^2(k, \widehat{G'})^* \\
0 \arrow{r} & G''(k)_{\pro} \arrow{r} & G''(\A)_{\pro} &
\end{tikzcd}
\]
The second row is exact by hypothesis, the bottom row is exact by Proposition \ref{G(k)G(A)proinjective}, and the columns are exact by Proposition \ref{globalpointsprofinite}. A diagram chase now shows that the first row is exact.
\end{proof}

\begin{lemma}
\label{exactnessatG(A)devissage4}
Let $k'/k$ be a finite separable extension of global function fields, and let $G$ be an affine commutative $k$-group scheme of finite type. If $(\ref{G(A)_proexactseqeqn})$ is exact for $H = G_{k'}$ (with $k$ replaced by $k'$), then it is exact for $G$ (with cohomology over $k$).
\end{lemma}

\begin{proof}
The canonical inclusion $G \hookrightarrow \R_{k'/k}(G_{k'})$, together with Lemma \ref{exactnessatG(A)devissage2}, shows that it suffices to prove that if $H'$ is an affine commutative $k'$-group scheme of finite type such that (\ref{G(A)_proexactseqeqn}) is exact (with $k$ replaced by $k'$) for $H = H'$, then it is also exact for $H = \R_{k'/k}(H')$ (with cohomology over $k$).

So now let $H : = \R_{k'/k}(H')$. By Proposition \ref{diagramcommutesglobal}, we have a commutative diagram
\[
\begin{tikzcd}
H(k) \arrow{r} \arrow[d, equals] & H(\A_k) \arrow{r} \arrow[d, equals] & {\rm{H}}^2(k, \widehat{H})^* \arrow{d} \\
H'(k') \arrow{r} & H'(\A_{k'}) \arrow{r} & {\rm{H}}^2(k', \widehat{H'})^*
\end{tikzcd}
\]
The lemma follows immediately.
\end{proof}

Recall the notion of the maximal geometrically reduced closed subscheme of a locally finite type scheme over a field \cite[\S C.4]{cgp}. For group schemes, this is a subgroup scheme \cite[Lemma C.4.1]{cgp}, and leads to the notion of the maximal {\em smooth} $k$-subgroup scheme of a locally finite type $k$-group scheme, since geometrically reduced group schemes locally of finite type over a field are smooth. We have the following lemma.

\begin{lemma}
\label{maxlredadelicpts}
For an affine scheme $X$ of finite type over a global field $k$, and the maximal geometrically reduced $k$-subscheme $X^{\sharp}$ of $X$, the map $X^{\sharp}(\A_k) \rightarrow X(\A_k)$ is a homeomorphism.
\end{lemma}

\begin{proof}
The map $X^{\sharp}(k_v) \rightarrow X(k_v)$ is a bijection \cite[Lemma C.4.1]{cgp}, and, by the definition of the topology on the $k_v$-points of an affine finite type $k_v$-scheme, it is even a homeomorphism. Since $X$ is affine, if we choose a finite type affine $\calO_S$-model $\mathscr{X}$ of $X$ for some finite set $S$ of places of $k$, then we have, topologically,
\[
X(\A) = {\prod_v}' X(k_v),
\]
where the restricted product is with respect to the groups $\mathscr{X}(\calO_v)$, and similarly for $X^{\sharp}$. Therefore, choosing a finite type affine $\calO_S$-model (perhaps after enlarging $S$) $\mathscr{X}^{\sharp}$ of $X^{\sharp}$ so that we have a closed embedding $\mathscr{X}^{\sharp} \rightarrow \mathscr{X}$ of $\calO_S$-schemes extending the embedding on the generic fiber, we need to show that the inclusion
\[
\mathscr{X}^{\sharp}(\calO_v) \hookrightarrow \mathscr{X}(\calO_v)
\]
is an isomorphism for all but finitely many $v$. We claim that it is an isomorphism for all $v \notin S$. Indeed, let $s$ denote the special point of ${\rm{Spec}}(\calO_v)$, and let $\eta$ denote the generic point. Given a morphism $f: {\rm{Spec}}(\calO_v) \rightarrow \mathscr{X}$, we have $f(\eta) \in \mathscr{X}^{\sharp}$, because $X(k_v) = X^{\sharp}(k_v)$. Since $f(s)$ lies in the closure of $f(\eta)$, it follows that also $f(s) \in \mathscr{X}^{\sharp}$. Since $\calO_v$ is reduced, it follows that $f$ factors through $\mathscr{X}^{\sharp}$.
\end{proof}

\begin{lemma}
\label{exactnessatG(A)devissage5}
Let $G$ be an affine commutative group scheme of finite type over a global function field $k$, and let $G^{{\rm{sm}}} \subset G$ be the maximal smooth $k$-subgroup scheme. If $(\ref{G(A)_proexactseqeqn})$ is exact for $H = G^{{\rm{sm}}}$, then it is exact for $H = G$.
\end{lemma}

\begin{proof}
We have a commutative diagram
\[
\begin{tikzcd}
G^{{\rm{sm}}}(k) \arrow{r} \arrow[d, equals] & G^{{\rm{sm}}}(\A) \arrow{r} \arrow[d, equals] & {\rm{H}}^2(k, \widehat{G^{{\rm{sm}}}})^* \arrow[d, hookrightarrow] \\
G(k) \arrow{r} & G(\A) \arrow{r} & {\rm{H}}^2(k, \widehat{G})^*
\end{tikzcd}
\]
in which the top row is exact by hypothesis, and the third vertical arrow is an inclusion by Proposition \ref{hatisexact} and Proposition \ref{H^3(G^)=0} applied to $G/G^{\rm{sm}}$. Further, the equality of $k$-points holds by \cite[Lemma C.4.1]{cgp}, and the equality on $\A$-points by Lemma \ref{maxlredadelicpts}. A simple diagram chase now shows that the bottom row is also exact.
\end{proof}

\begin{proposition}
\label{exactnessatG(A)}
Let $k$ be a global function field, $G$ an affine commutative $k$-group scheme of finite type. Then the global duality sequence
\[
G(k)_{\pro} \longrightarrow G(\A)_{\pro} \longrightarrow {\rm{H}}^2(k, \widehat{G})^*
\]
in Theorem $\ref{poitoutatesequence}$ is exact.
\end{proposition}

\begin{proof}
The case in which $G$ is finite follows from Poitou--Tate for finite commutative $k$-group schemes (see Remark \ref{prior}) together with Lemma \ref{finiteprofinite}. Now let $G$ be an arbitrary affine commutative group scheme of finite type over a global function field $k$. Lemma \ref{exactnessatG(A)devissage5} allows us to assume that $G$ is smooth. First suppose that $G$ contains no nontrivial $k$-torus (i.e., $G$ is almost-unipotent; see Definition \ref{almostunipdef} and Lemma \ref{almostunipotent}(ii)). Then by Lemma \ref{almostunipotent}(i), since $G$ is smooth its identity component $U$ is a smooth connected unipotent $k$-group scheme. Thus we have an exact sequence
\begin{equation}
\label{exactnessatG(A)pfeqn1}
1 \longrightarrow U \longrightarrow G \longrightarrow E \longrightarrow 1
\end{equation}
with $E$ a finite \'etale commutative $k$-group scheme. Choose an inclusion $U \hookrightarrow W$ for some split unipotent $k$-group $W$. One way to do this is to choose some finite extension $k'/k$ such that $U_{k'}$ is split, and then take the canonical inclusion $U \hookrightarrow \R_{k'/k}(U_{k'})$. The latter group is split unipotent because it admits a filtration by $\R_{k'/k}(\Ga) \simeq \Ga^{[k':  k]}$. Pushing out the exact sequence (\ref{exactnessatG(A)pfeqn1}) along the inclusion $U \hookrightarrow W$, we obtain a commutative diagram of exact sequences
\[
\begin{tikzcd}
1 \arrow{r} & U \arrow{r} \arrow[d, hookrightarrow] & G \arrow{r} \arrow[d, hookrightarrow] & E \arrow{r} \arrow[d, equals] & 1 \\
1 \arrow{r} & W \arrow{r} & H \arrow{r} & E \arrow{r} & 1
\end{tikzcd}
\]
in which the first square is a pushout diagram (and $H$ is defined to be the pushout of $U \hookrightarrow G$ along the map $U \hookrightarrow W$). The proposition holds for the finite $k$-group $E$, and it holds for the split unipotent group $W$ by Lemma \ref{exactnessatG(A)splittoriunip}. By Lemma \ref{exactnessatG(A)devissage1}, therefore, it holds for $H$. Lemma \ref{exactnessatG(A)devissage2} then implies the proposition for $G$.

Now consider the general case, in which $G$ may contain a nontrivial $k$-torus, and let $T \subset G$ denote the maximal $k$-torus. Let $H : = G/T$. Then $H$ contains no nontrivial $k$-torus, hence the proposition holds for $H$ by the previous paragraph. Let $k'/k$ be a finite separable extension over which $T$ splits. By Lemma \ref{exactnessatG(A)devissage4}, in order to prove the proposition for $G$ it suffices to prove it for $G_{k'}$, hence we may assume that $T$ is split. The proposition for $G$ then follows from Lemma \ref{exactnessatG(A)splittoriunip} and Lemma \ref{exactnessatG(A)devissage1} applied to the exact sequence $$1 \longrightarrow T \longrightarrow G \longrightarrow H \longrightarrow 1.$$ 
\end{proof}

\section{Exactness at ${\rm{H}}^0(\A, G)$}

In this section, we will prove that if $G$ is an affine commutative group scheme of finite type over a global function field $k$, then the complex
\begin{equation}
\label{exactatG(A)eqn}
G(k) \longrightarrow G(\A_k) \longrightarrow {\rm{H}}^2(k, \widehat{G})^*
\end{equation}
arising from the global duality pairing $G(\A_k) \times {\rm{H}}^2(k, \widehat{G}) \rightarrow \Q/\Z$ (which cups everywhere locally and then adds the invariants) is exact (Proposition \ref{exactnessatG(A)nocompletion}). That is, we want the uncompleted version of Proposition \ref{exactnessatG(A)}, though the proof will of course depend upon that proposition. In addition to being interesting in its own right, we shall require this exactness later (in \S \ref{Sha^1(G)-->Sha^2(G^)*injectssection}) in order to prove the exactness of the pairings in Theorem \ref{shapairing}.

Unlike most of the results that we prove for global fields, exactness of (\ref{exactatG(A)eqn}) really only holds for global function fields. Indeed, if $k$ is a number field, then (\ref{exactatG(A)eqn}) fails to be exact for $G = \Ga$, since ${\rm{H}}^2(k, \widehat{\Ga}) = 0$ because $k$ is perfect (Lemma \ref{cohomologyofG_adualwhenkisperfect}), but clearly $\A/k \neq 0$.

We begin by making a definition. Given a global field $k$, we have the norm map $|\!|\cdot |\!|:  \A^{\times} \rightarrow \mathbf{R}_{>0}$ given by $a \mapsto \prod_v |\!|a|\!|_v$. By the product formula, $|\!|\lambda|\!| = 1$ for any $\lambda \in k^{\times}$. Now given a finite type $k$-group scheme $G$, for any character $\chi \in \widehat{G}(k)$ we obtain a
continuous norm map $|\!|\chi|\!|:  G(\A) \rightarrow \mathbf{R}_{>0}$ via $a \mapsto |\!|\chi(a)|\!|$. We then make the following definition.

\begin{definition} 
\label{G(A)_1def}
$G(\A)_1 : = \underset{\chi \in \widehat{G}(k)}{\bigcap} \ker(|\!|\chi|\!|)$.
\end{definition}

Since $|\!|k^{\times}|\!| = \{1\}$, we have $G(k) \subset G(\A)_1$. Further, the group $\widehat{G}(k)$ is finitely generated. (Indeed, to check this assertion we may extend scalars to $\overline{k}$, and if true for a subgroup and the quotient, then this is true for $G$, so we are thereby reduced to the well-known cases of $\Ga$, $\Gm$, semisimple groups, abelian varieties, and finite group schemes.) Let $\chi_1, \dots, \chi_n$ be elements that freely generate $\widehat{G}(k)$ modulo torsion. Then $G(\A)_1$ is the kernel of the continuous map $(|\!|\chi_1|\!|, \dots, |\!|\chi_n|\!|)\colon  G(\A) \rightarrow (\mathbf{R}_{>0})^n$. By \cite[Ch.\,I, Prop.\,5.6]{oesterle}, if $k$ is a global function field then the image of this map is a lattice in $(\mathbf{R}_{>0})^n$ (i.e., a discrete subgroup that is cocompact, or equivalently, has full rank). If $k$ is a number field, then the image is $(\mathbf{R}_{>0})^n$. In particular, if $k$ is a global function field then $G(\A)_1$ is open and $G(\A)/G(\A)_1$ is a discrete free abelian group of finite rank, whereas if $k$ is a number field then this quotient is divisible.

\begin{lemma}
\label{filteredbetweencompactgps}
If we have an exact sequence
\[
0 \longrightarrow A' \longrightarrow A \longrightarrow A'' \longrightarrow 0
\]
of locally compact, second-countable, Hausdorff abelian groups such that $A'$ and $A''$ are compact, then so is $A$.
\end{lemma}

\begin{proof}
It suffices by Pontryagin duality to check that $A^D$ is discrete. First, we claim that the sequence 
\[
0 \longrightarrow A''^D \longrightarrow A^D \longrightarrow A'^D \longrightarrow 0
\]
is exact. Exactness at $A''^D$ is clear. Exactness at $A'^D$ holds because $A' \hookrightarrow A$ is an embedding onto a closed subgroup (since the source is compact and the target is Hausdorff). For exactness in the middle, we need to check that the map $A/A' \rightarrow A''$ is a homeomorphism, and this follows from Lemma \ref{isom=homeo}. Since $A'^D$ is discrete, it follows that the image of the map $A''^D \hookrightarrow A^D$ is open. If we check that this map is a homeomorphism onto its image, then we will be done, since $A''^D$ is discrete. The image is the kernel of the continuous map $A^D \rightarrow A'^D$, hence closed, so the result once again follows from Lemma \ref{isom=homeo}.
\end{proof}

\begin{proposition}
\label{G(A)_1/G(k)compact}
Let $k$ be a global field, $G$ an affine commutative $k$-group scheme of finite type. Then $G(\A)_1/G(k)$ is compact Hausdorff.
\end{proposition}

\begin{remark}
\label{G(A)_1/G(k)compactremark}
The compactness assertion in Proposition \ref{G(A)_1/G(k)compact} is made in \cite[Ch.\,IV, Thm.\,1.3]{oesterle} for any solvable smooth affine 
$k$-group $G$.  The proof requires that $G$ be connected (in order to apply the Lie--Kolchin theorem to $G$),
so the correct statement of \cite[Ch.\,IV,  Thm.\,1.3]{oesterle} is that if $G$ is a smooth {\em connected} affine solvable group scheme over the global field $k$, then $G(\A)_1/G(k)$ is compact. It is false without connectedness, as the following example due to Brian Conrad illustrates. 

Let $G = \Gm \rtimes \Z/2\Z$
where the semi-direct product structure is through inversion. 
As we saw in Remark \ref{characterfinitecokerremark}, any $\chi \in \widehat{G}(k)$ restricts to the trivial character on $\Gm$, hence has finite order. It follows that $G(\A)_1 = G(\A)$. But we claim that $G(\A)/G(k)$ is not compact.

The quotient group $\Gm(\A)/\Gm(k) = \A^{\times}/k^{\times}$
is non-compact (due to non-compactness of 
its image in $\mathbf{R}_{>0}$ under the idelic norm), so it suffices to show that the natural continuous injective map 
\begin{equation}
\label{Gm(A)/Gm(k)--->Z/2Z}
j: \Gm(\A)/\Gm(k) \hookrightarrow G(\A)/G(k)
\end{equation}
 is a homeomorphism onto a closed
subspace.  Since
$G(k) \rightarrow \Z/2\Z$ is surjective, this image coincides with the fiber over $1$ for 
the continuous map $G(\A)/G(k) \rightarrow (\Z/2\Z)(\A)/(\Z/2\Z)$, which is closed because 
$(\Z/2\Z)(\A)/(\Z/2\Z)$ is Hausdorff. To show that the map (\ref{Gm(A)/Gm(k)--->Z/2Z}) is a homeomorphism onto its image, we note that the continuous map $G(\A) \rightarrow \Gm(\A)$ given by projection onto the $\Gm$ factor descends to a well-defined map $\im(j) \rightarrow \Gm(\A)/\Gm(k)$ (though it does {\em not} descend to a well-defined map $G(\A)/G(k) \rightarrow \Gm(\A)/\Gm(k)$) that provides a continuous inverse (on $\im(j)$) to $j$, hence $j$ is a homeomorphism onto its image, as claimed.
\end{remark}

\begin{proof}[Proof of Proposition $\ref{G(A)_1/G(k)compact}$]
Since $k \hookrightarrow \A$ is a closed subgroup, by choosing an embedding of $G$ into some affine $n$-space over $k$, we can see that $G(k) \hookrightarrow G(\A)$ is closed, hence $G(\A)_1/G(k)$ is Hausdorff, so we concentrate on compactness. Replacing $G$ with its maximal smooth $k$-subgroup scheme $G^{\sm} \subset G$ (see \cite[Lemma C.4.1, Remark C.4.2]{cgp}) has no effect on its adelic or rational points. We claim that it also has no effect on the norm-$1$ points. That is, if $g \in G(\A)$ satisfies $|\!|\chi(g)|\!| = 1$ for every $\chi \in \widehat{G}(k)$, then $|\!|\chi(g)|\!| = 1$ for every $\chi \in \widehat{G^{\sm}}(k)$. For this, it suffices to show that for every $\chi \in \widehat{G^{\sm}}(k)$, $\chi^n$ extends to $\widehat{G}(k)$ for some $n > 0$. This follows from Corollary \ref{characterfinitecokernel}. So we may assume that $G$ is smooth. By Remark \ref{G(A)_1/G(k)compactremark}, based on results of Oesterl\'e \cite{oesterle}, we know the result when $G$ is smooth and connected. We will deduce the general case from this one.

Let $G^0$ be the identity component of $G$, and let $E : = G/G^0$ be the finite {\'e}tale quotient. Then $E(\A) = E(\A)_1$ is compact by Lemma \ref{finiteprofinite}, hence so is $E(\A)/E(k)$. Let $X$ be defined by the following exact sequence
\[
0 \longrightarrow X \longrightarrow G(\A)_1/G(k) \longrightarrow E(\A)/E(k)
\]
Then we claim that the map $G(\A)_1/G(k) \rightarrow E(\A)/E(k)$ has closed image and that $X$ is compact. Lemma \ref{filteredbetweencompactgps} will then complete the proof.

First we check that the map $G(\A)_1/G(k) \rightarrow E(\A)/E(k)$ has closed image. Since $E(k)$ is finite, it suffices to check that the map $G(\A)_1 \rightarrow E(\A)$ has closed image. We note that the map $G(\A) \rightarrow E(\A)$ has closed image, since it is the kernel of the continuous map $E(\A) \rightarrow {\rm{H}}^1(\A, G^0)$, in which the target group is Hausdorff by Proposition \ref{adelictopcohombasics}(iv). It follows from Lemma \ref{isom=homeo} that the map $G(\A)/G^0(\A) \hookrightarrow E(\A)$ induces a homeomorphism onto a closed subgroup. Therefore, if we show that the map $G(\A)_1 \rightarrow G(\A)/G^0(\A)$ has closed image, then the claim will follow. That is, we need to show that the subgroup $G(\A)_1G^0(\A) \subset G(\A)$ is closed.

First suppose that $k$ is a function field. Then the quotient $G(\A)/G(\A)_1$ is discrete, hence the quotient $G(\A)/G(\A)_1G^0(\A)$ is discrete, so $G(\A)_1G^0(\A) \subset G(\A)$ is closed. Next suppose that $k$ is a number field. Then the quotient $G(\A)/G(\A)_1$ is divisible, and the quotient $G(\A)/G^0(\A) \hookrightarrow E(\A)$ is of finite exponent, so the quotient $G(\A)/G(\A)_1G^0(\A)$ is divisible and of finite exponent, thus trivial.

Next we need to show that $X$ is compact. We will show that $[X:  G^0(\A)_1/G^0(k)]$ is finite, and the desired compactness will then follow from the known compactness of $G^0(\A)_1/G^0(k)$ (Remark \ref{G(A)_1/G(k)compactremark}). We prove finiteness in two steps: first, we prove that $G^0(\A) \cap G(\A)_1 = G^0(\A)_1$, and then that $[X:  (G^0(\A) \cap G(\A)_1)/G^0(k)]$ is finite.

Suppose that $g \in G^0(\A) \cap G(\A)_1$. We need to show that $|\!|\chi(g)|\!| = 1$ for all $\chi \in \widehat{G^0}(k)$. By Corollary \ref{characterfinitecokernel}, $\chi^n$ extends to a character of $\widehat{G}(k)$ for some positive integer $n$. Therefore, $|\!|\chi^n(g)|\!| = 1$, so $|\!|\chi(g)|\!| = 1$, as desired.

Next, we check that $[X:  (G^0(\A) \cap G(\A)_1)/G^0(k)]$ is finite. Define an inclusion
\[
\phi:  \frac{X}{(G^0(\A)\cap G(\A)_1)/G^0(k)} \hookrightarrow \Sha^1(G^0)
\]
as follows. We have a commutative diagram with exact rows: 
\[
\begin{tikzcd}
G^0(k) \arrow{r} \arrow{d} & G(k) \arrow{r} \arrow{d} & E(k) \arrow{r}{\delta} \arrow[d, hookrightarrow] & {\rm{H}}^1(k, G^0) \arrow{d} \\
G(\A)_1 \cap G^0(\A) \arrow{r} & G(\A)_1 \arrow{r}{f} & E(\A) \arrow{r} & {\rm{H}}^1(\A, G^0)
\end{tikzcd}
\]
Given $g \in G(\A)_1$ representing an element of $X$, $f(g)$ lifts to an element $e \in E(k)$ and we define $\phi(g) : = \delta(e)$. A straightforward diagram chase then shows that $\phi$ is indeed an inclusion into $\Sha^1(G^0)$. By \cite[Thm.\,7.1]{borelserre} over number fields and \cite[Thm.\,1.3.3(i)]{conrad} over function fields, $\Sha^1(k, H)$ is finite for all affine $k$-group schemes $H$ of finite type (even without commutativity hypotheses). This therefore proves the finiteness of $[X:  (G^0(\A) \cap G(\A)_1)/G^0(k)]$.
\end{proof}

\begin{proposition}
\label{G(A)_1/G(k)profinite}
Let $k$ be a global function field, $G$ an affine commutative $k$-group scheme of finite type. Then $G(\A)_1/G(k)$ is profinite.
\end{proposition}

\begin{proof}
By Proposition \ref{G(A)_1/G(k)compact}, $G(\A)_1/G(k)$ is compact Hausdorff, so we only need to check that $G(\A)/G(k)$ is totally disconnected. In order to do this, it suffices to show that if $g \in G(\A) - G(k)$, then there is a compact open set $U \subset G(\A)$ containing $g$ and disjoint from $G(k)$. Since $G(k) \subset G(\A)$ is closed, there exists an open such $U$. Since any neighborhood of $g$ contains a compact open neighborhood of $g$ (as $k$ is a function field, so it has no archimedean places), we are done.
\end{proof}

\begin{remark}
Proposition \ref{G(A)_1/G(k)profinite} is false if $k$ is a number field. Indeed, if $G = \Ga$ then $G(\A)_1 = \A$, but $\A/k$ is divisible and nonzero, so not profinite.
\end{remark}

\begin{lemma}
\label{G(A)/G(k)containedincompletion}
Let $k$ be a global function field, $G$ an affine commutative $k$-group scheme of finite type. Then the canonical map
\[
G(\A)/G(k) \longrightarrow (G(\A)/G(k))_{\pro}
\]
is an inclusion.
\end{lemma}

\begin{proof}
Consider the following commutative diagram: 
\[
\begin{tikzcd}
0 \arrow{r} & G(\A)_1/G(k) \isoarrow{d} \arrow{r} & G(\A)/G(k) \arrow{r} \arrow{d} & G(\A)/G(\A)_1 \arrow[d, hookrightarrow] \\ 
0 \arrow{r} & (G(\A)_1/G(k))_{\pro} \arrow{r} & (G(\A)/G(k))_{\pro} \arrow{r} & (G(\A)/G(\A)_1)_{\pro}
\end{tikzcd}
\]
The top row is clearly exact. The bottom row is exact by Proposition \ref{profiniteexactfgquotient}, since the quotient $G(\A)/G(\A)_1$ is discrete and finitely generated. The first vertical arrow is an isomorphism by Proposition \ref{G(A)_1/G(k)profinite}, and the third vertical arrow is an inclusion because $G(\A)/G(\A)_1$ is discrete and finitely generated. The five lemma now shows that the middle vertical arrow is an inclusion.
\end{proof}

\begin{proposition}
\label{exactnessatG(A)nocompletion}
Let $k$ be a global function field, $G$ an affine commutative $k$-group scheme of finite type. Then the complex
\[
G(k) \longrightarrow G(\A) \longrightarrow {\rm{H}}^2(k, \widehat{G})^*
\]
induced by the global duality pairing $G(\A) \times {\rm{H}}^2(k, \widehat{G}) \rightarrow \Q/\Z$ is exact.
\end{proposition}

\begin{proof}
Thanks to Proposition \ref{exactnessatG(A)} and the right-exactness of profinite completion (Proposition \ref{profiniterightexact}), in conjunction with Remark \ref{G(k)-->G(A)embedding}, this follows from Lemma \ref{G(A)/G(k)containedincompletion}.
\end{proof}

\section{Exactness at ${\rm{H}}^2(k, \widehat{G})^*$}

The goal of this section is essentially to prove that the global duality sequence of Theorem \ref{poitoutatesequence} is exact at ${\rm{H}}^2(k, \widehat{G})^*$. We say essentially because strictly speaking the map ${\rm{H}}^2(k, \widehat{G})^* \rightarrow {\rm{H}}^1(k, G)$ is not defined until we obtain the perfect pairing between the groups $\Sha^2(k, \widehat{G})$ and $\Sha^1(k, G)$. Indeed, recall that this map is defined to be the composition
\[
{\rm{H}}^2(k, \widehat{G})^* \twoheadrightarrow \Sha^2(\widehat{G})^* \xrightarrow{\sim} \Sha^1(G) \hookrightarrow {\rm{H}}^1(k, G),
\]
where the middle isomorphism is defined via still unproven Theorem \ref{shapairing}. Once that theorem is proven, however, the exactness of the complex at ${\rm{H}}^2(k, \widehat{G})^*$ is then equivalent to the exactness of the following sequence, which involves no input from any pairing between Tate-Shafarevich groups: 
\begin{equation}
\label{exactnessatH^2(G^)*seq12}
{\rm{H}}^0(\A, G)_{\pro} \longrightarrow {\rm{H}}^2(k, \widehat{G})^* \longrightarrow \Sha^2(\widehat{G})^*.
\end{equation}
What we shall show in this section is that the complex (\ref{exactnessatH^2(G^)*seq12}) is in fact an exact sequence: 

\begin{proposition}
\label{exactnessatH^2(k,G^)*}
Let $k$ be a global function field, $G$ a commutative affine $k$-group scheme of finite type. Then the global duality sequence
\[
{\rm{H}}^0(\A, G)_{\pro} \longrightarrow {\rm{H}}^2(k, \widehat{G})^* \longrightarrow \Sha^2(\widehat{G})^*
\]
is exact.
\end{proposition}

\begin{proof}
It is clear from the definitions that the diagram in Proposition \ref{exactnessatH^2(k,G^)*} is a complex (as was discussed in \S \ref{globalsectionpreliminaries}). The compact group $G(\A)_{\pro}$ has closed image inside the closed subgroup 
$$\ker({\rm{H}}^2(k, \widehat{G})^* \rightarrow \Sha^2(\widehat{G})^*) \subset {\rm{H}}^2(k, \widehat{G})^*.$$ 
Hence, it suffices to show that $G(\A)_{\pro}$ has dense image inside
this kernel, or equivalently, that the more tangible $G(\A)$ has dense image. That is, given a finite subset $T \subset {\rm{H}}^2(k, \widehat{G})$ and $\phi \in {\rm{H}}^2(k, \widehat{G})^*$ vanishing on $\Sha^2(\widehat{G})$, we seek $g \in G(\A)$ such that 
$\langle g, \alpha \rangle$ is close to $\phi(\alpha)$ inside $\mathbf{R}/\mathbf{Z}$ for every $\alpha \in T$. Since ${\rm{H}}^2(k, \widehat{G})$ is torsion (Lemma \ref{H^2(G^)istorsion}), we may replace $T$ with the {\em finite} group that it generates and thereby assume that $T$ is a subgroup.

Note that $\phi|_T$ factors through the quotient
$T/(T \cap \Sha^2(\widehat{G}))$ since $\phi$ vanishes on $\Sha^2(\widehat{G})$.
We claim that for some finite set $S$ of places of $k$, the map $T/(T \cap \Sha^2(\widehat{G})) \rightarrow \prod_{v \in S} {\rm{H}}^2(k_v, \widehat{G})$ is {\em injective}, so $\phi|_T$ is induced by an element of 
$(\prod_{v \in S} {\rm{H}}^2(k_v, \widehat{G}))^*$. 
Indeed, for each $\alpha \in T - (T \cap \Sha^2(\widehat{G}))$, there exists a place $v = v(\alpha)$ of $k$ such that $\alpha_v \neq 0$. We may choose $S$ to consist of the $v(\alpha)$. 

We will construct the desired $g \in G(\A)$ approximating $\phi$ on $T$ to satisfy
$g_v=0$ for $v \notin S$.  Our task is to appropriately choose $g_v \in G(k_v)$ for each $v \in S$.
(Note that $g = (g_v) \in \prod G(k_v)$ then lies in $G(\A)$.)  It has been shown that $\phi|_T$
lifts to $(\prod_{v \in S} {\rm{H}}^2(k_v, \widehat{G}))^*$, so to find the required $(g_v)_{v \in S} \in \prod_{v \in S} G(k_v)$, it suffices to show that the natural map $\prod_{v \in S} G(k_v) \rightarrow (\prod_{v \in S} {\rm{H}}^2(k_v, \widehat{G}))^*$ induced by the local duality pairings has dense image, or equivalently that the natural continuous
homomorphism $G(k_v) \rightarrow {\rm{H}}^2(k_v, \widehat{G})^*$
has dense image for each $v \in S$.  By local duality (Theorem \ref{H^2(G^)altori}), the target of this latter map is topologically identified via the local duality pairing with $G(k_v)_{\pro}$. The desired density follows, completing 
the proof of the proposition.
\end{proof}

\section{The fundamental exact sequence}\label{fes}

The goal of this section is to prove that the last three terms in Theorem \ref{poitoutatesequence} form an exact sequence. We
refer to this 3-term sequence as the {\em fundamental exact sequence} since when $G = \Gm$, it forms (most of) the fundamental exact sequence of class field theory (due to \cite[Th.\,2.13]{ces2} for degree-2 cohomology of $\Gm$).

\begin{proposition}
\label{funsequence}
Let $k$ be a global function field, $G$ an affine commutative $k$-group scheme of finite type. Then the global duality sequence
\[
{\rm{H}}^2(k, G) \longrightarrow {\rm{H}}^2(\A, G) \longrightarrow \widehat{G}(k)^* \longrightarrow 0
\]
of Theorem $\ref{poitoutatesequence}$ is exact.
\end{proposition}

\begin{proof}
Note that we may replace ${\rm{H}}^2(\A, G)$ with $\underset{v}{\bigoplus} {\rm{H}}^2(k_v, G)$, by \cite[Th.\,2.13]{ces2}. When $G$ is finite, this proposition is part of the Poitou--Tate sequence for finite commutative group schemes (see Remark \ref{prior}). Let us first treat separable Weil restrictions of split tori. So suppose that $G = \R_{k'/k}(T')$, where $k'/k$ is a finite separable extension and $T'$ is a split $k'$-torus. We may assume that $T' = \Gm$. By Proposition \ref{diagramcommutesglobal}, we have a commutative diagram
\[
\begin{tikzcd}
{\rm{H}}^2(k', T') \arrow{r} \isoarrow{d} & \underset{w}{\bigoplus} {\rm{H}}^2(k'_w, T') \arrow{r} \isoarrow{d} & \widehat{T'}(k)^* \arrow{r} \isoarrow{d} & 0 \\
{\rm{H}}^2(k,\R_{k'/k}(T')) \arrow{r} & \underset{v}{\bigoplus} {\rm{H}}^2(k_v,\R_{k'/k}(T')) \arrow{r} & \widehat{\R_{k'/k}(T')}(k)^* \arrow{r} & 0 
\end{tikzcd}
\]
where the top row is exact by class field theory (the fundamental exact sequence), the first two vertical arrows are isomorphisms by Lemma \ref{sepblepushforward}, and the third vertical arrow -- which is given by the norm map -- is an isomorphism by Proposition \ref{charactersseparableweilrestriction}. It follows that the bottom row is also exact.

Next we prove the proposition when $G$ is an almost-torus. By Lemma \ref{almosttorus}(iv), we may harmlessly modify $G$ and thereby assume that there is an exact sequence
\[
1 \longrightarrow B \longrightarrow X \longrightarrow G \longrightarrow 1
\]
with $X = C \times \R_{k'/k}(T')$, where $B$ and $C$ are finite commutative $k$-group schemes, $k'/k$ is a finite separable extension, and $T'$ is a split $k'$-torus. Consider the commutative diagram
\[
\begin{tikzcd}
& {\rm{H}}^2(k, X) \arrow{r} \arrow{d} & {\rm{H}}^2(k, G) \arrow{d} & \\
\underset{v}{\bigoplus} {\rm{H}}^2(k_v, B) \arrow{r} \arrow{d} & \underset{v}{\bigoplus} {\rm{H}}^2(k_v, X) \arrow{r} \arrow{d} & \underset{v}{\bigoplus} {\rm{H}}^2(k_v, G) \arrow{r} \arrow{d} & 0 \\
\widehat{B}(k)^* \arrow{r} \arrow{d} & \widehat{X}(k)^* \arrow{r} \arrow{d} & \widehat{G}(k)^* \arrow{r} \arrow{d} & 0 \\
0 & 0 & 0 &
\end{tikzcd}
\]
The third row is exact due to the injectivity of the group $\Q/\Z$, and the second row is exact by Proposition \ref{cohomologicalvanishing}. The first and second columns are exact due to the already-settled cases of finite group schemes and separable Weil restrictions of split tori. That the last column, which we know to be a complex, is exact now follows from a diagram chase.

Now we treat the general case. Let $G$ be an affine commutative $k$-group scheme of finite type. By Lemma \ref{affinegroupstructurethm}, there is an exact sequence 
\[
1 \longrightarrow H \longrightarrow G \longrightarrow U \longrightarrow 1
\]
with $H$ an almost-torus and $U$ split unipotent. We then have a commutative diagram
\[
\begin{tikzcd}
{\rm{H}}^2(k, H) \arrow{r} \arrow{d} & \underset{v}{\bigoplus} {\rm{H}}^2(k_v, H) \arrow{r} \arrow[d, twoheadrightarrow] & \widehat{H}(k)^* \arrow{r} \isoarrow{d} & 0 \\
{\rm{H}}^2(k, G) \arrow{r} & \underset{v}{\bigoplus} {\rm{H}}^2(k_v, G) \arrow{r} & \widehat{G}(k)^* \arrow{r} & 0
\end{tikzcd}
\]
in which the top row is exact by the already treated case of almost-tori; the second vertical arrow is surjective because ${\rm{H}}^2(k_v, U) =0$ by Proposition \ref{unipotentcohomology}(i); and the last vertical arrow is an isomorphism by Proposition \ref{hatisexact} and because $\widehat{U}(k) = 0$ (since the unipotent group $U$ has no nontrivial characters) and ${\rm{H}}^1(k, \widehat{U}) = 0$ (Proposition \ref{unipotentcohomology}(iii)). A diagram chase now shows that the second row (which, again, we already know to be a complex) is exact. This completes the proof of the proposition.
\end{proof}

\section{Defining the $\Sha$-pairings}
\label{sectiondefiningshapairings}

In this section we will define the pairings between Tate-Shafarevich groups arising in Theorem \ref{shapairing}. Our definition imitates Tate's original definition in terms of cocycles. First, we need a lemma.

\begin{lemma}\label{checkh3=0}
If $k$ is a global field, then $\check{\rm{H}}^3(k, \mathbf{G}_m) = 0$.
\end{lemma}

This is the \v{C}ech-fppf variant of the well-known derived-functor result ${\rm{H}}^3(k, \mathbf{G}_m) = 0$ \cite[Ch.\,VII, \S11.4]{cf}.

\begin{proof}
We have the \v{C}ech-to-derived functor spectral sequence
\[
E_2^{i,j} = \check{\rm{H}}^i(k, \mathscr{H}^j(\mathbf{G}_m)) \Longrightarrow {\rm{H}}^{i+j}(k, \mathbf{G}_m).
\]
Since ${\rm{H}}^3(k, \mathbf{G}_m) = 0$, it is enough to show that $E_2^{1,1}$ and $E_2^{0,2}$ vanish. First, $$E_2^{0,2} = \check{\rm{H}}^0(k, \mathscr{H}^2(\mathbf{G}_m)) = 0$$ because for any finite extension $L/k$ and any $\alpha \in H^2(L, \mathbf{G}_m) = \Br(L)$, there is a finite extension $L'/L$ such that $\alpha$ dies in $\Br(L') = H^2(L', \mathbf{G}_m)$. Next, to show that $E_2^{1,1} = \check{\rm{H}}^1(k, \mathscr{H}^1(\mathbf{G}_m))$ vanishes, it is enough to show that Pic$(L \otimes_k L) = 0$ for any finite extension $L/k$. But this is clear, because ${\rm{Spec}}(L \otimes_k L)$ is topologically a disjoint union of points.
\end{proof}

Now we will define the $\Sha$-pairings. We will obtain them as special cases of a more general construction.Let $k$ be a global function field. For any fppf sheaf $\mathscr{G}$ on ${\rm{Spec}}(k)$ and any $i \geq 0$, define
\[
\check{\Sha}^i(\mathscr{G}) := \ker\left(\check{{\rm{H}}}^i(k, \mathscr{G}) \longrightarrow \check{{\rm{H}}}^i(\A_k, \mathscr{G})\right).
\]
Let $\mathscr{F}_1$ and $\mathscr{F}_2$ be locally finitely presented sheaves on the big fppf site of ${\rm{Spec}}(k)$. The locally finitely presented condition means that -- for a filtered direct system $\{A_i\}$ of $k$-algebras, with $A := \varinjlim_i A_i$, the natural map $\varinjlim_i \mathscr{F}_j(A_i) \rightarrow \mathscr{F}_j(A)$ is an isomorphism ($j = 1, 2$). Assume also given a bi-additive pairing
\[
\chi\colon \mathscr{F}_2 \times \mathscr{F}_1 \longrightarrow \Gm.
\]
Then we define a pairing
\[
\langle \cdot, \cdot \rangle_{\Sha}\colon \check{\Sha}^2(\mathscr{F}_2) \times \check{\Sha}^1(\mathscr{F}_1) \longrightarrow \Q/\Z,
\]
functorial in the triple $(\mathscr{F}_2, \mathscr{F}_1, \chi)$, as follows.

Choose $$\xi \in \check{\Sha}^2(G) \subset \check{\rm{H}}^2(k, G),\,\,\,
\xi' \in \check{\Sha}^1(\widehat{G}) \subset \check{\rm{H}}^1(k, \widehat{G}).$$
Let $\alpha \in \check{Z}^2(k, G)$, $\alpha' \in \check{Z}^1(k, \widehat{G})$ be respective
representative \v{C}ech cocycles, so they are each everywhere locally coboundaries. That is, for every place $v$ of $k$, there exists a 1-cochain $\beta_v \in \check{C}^1(k_v, G)$ and a 0-cochain $\beta_v' \in \check{C}^0(k_v, \widehat{G})$ such that $\alpha_v = d \beta_v$, $\alpha'_v = d \beta_v'$. Since $\check{\rm{H}}^3(k, \mathbf{G}_m) = 0$, there exists a 2-cochain $h \in 
\check{C}^2(k, \mathbf{G}_m)$ such that $\alpha \cup \alpha' = d h$, where the cup product is via the pairing $\chi$. Then $d(\beta_v \cup \alpha'_v) = dh = d(\alpha_v \cup \beta'_v)$. It follows that $(\alpha_v \cup \beta'_v) - h_v$ and $(\beta_v \cup \alpha'_v) - h_v$ are 2-cocycles. Further, they yield the same class in $\check{\rm{H}}^2(k_v, \mathbf{G}_m) = {\rm{H}}^2(k_v, \mathbf{G}_m)$ (Proposition \ref{cech=derivedsmoothinf}), because $d(\beta_v \cup \beta'_v) = (\alpha_v \cup \beta'_v) - (\beta_v \cup \alpha'_v)$. We then define $\langle \xi, \xi' \rangle_{\Sha}$ to be the sum over all $v$ of the invariants of these elements of ${\rm{H}}^2(k_v, \mathbf{G}_m)$.

This pairing is clearly functorial in the triple $(\mathscr{F}_1, \mathscr{F}_2, \chi)$, and it is easy to check that it is bi-additive. To see that it is independent of choices, suppose that we replace $\beta_v$ by $\beta_v + \epsilon_v$ for some $1$-cocycle $\epsilon_v \in \check{Z}^1(k_v, G)$. Then the element of ${\rm{H}}^2(k_v, \Gm)$ whose invariant is taken in the above calculation is $((\beta_v + \epsilon_v)\cup \alpha'_v) - h_v$ rather than $(\beta_v \cup \alpha'_v) - h_v$. Thus we need to show that $\epsilon_v \cup \alpha'_v$ is trivial as a cohomology class. This follows from the equality $\epsilon_v \cup \alpha'_v = d(-\epsilon_v \cup\beta'_v)$, which holds because $d\epsilon_v = 0$. The independence of choice of $\beta'_v$ is proved similarly. The independence of the choices of $h, \alpha, \alpha'$ follows from the fact that the sum of the local invariants of a global Brauer class is $0$.

In order to check that the pairing is well-defined, it only remains to check that the sum appearing in the definition contains only finitely many nonzero terms. But by assumption, $\alpha'$ maps to a coboundary in $\check{{\rm{H}}}^1(\A, \widehat{G})$. Using the equality
\[
\A = \varinjlim_S\left( \widehat{\calO}_S \times \prod_{v \in S} k_v \right),
\]
and the fact that $\mathscr{F}_1$ is locally finitely presented, it follows that we may choose $\beta'_v$ to actually come from $\check{C}^0(\calO_v, \widehat{G})$ for almost every $v$. Since $\alpha$ and $h$ both extend to cocycles over some $\calO_S$, it follows that the cocycle $(\alpha \cup \beta'_v) - h \in \check{Z}^2(k_v, \Gm)$ extends to an element of $\check{Z}^2(\calO_v, \Gm)$ for almost every $v$, hence represents the trivial cohomology class, since ${\rm{H}}^2(\calO_v, \Gm) = 0$. The sum is therefore finite.

We now obtain the $\Sha$-pairings of Theorem \ref{shapairing} by applying the above construction with either $\mathscr{F}_1 := G$ and $\mathscr{F}_2 := \widehat{G}$, or vice versa, and the pairing $G \times \widehat{G} \rightarrow \Gm$ the natural one. This yields a pairing on the derived functor $\Sha$ groups (rather than just the \v{C}ech cohomology groups), thanks to Propositions \ref{cech=derivedsmoothinf} and \ref{cech=derivedG^}.

\begin{remark}
\label{shapairingcartier}
Suppose that $G$ is a finite commutative $k$-group scheme. We have now defined pairings between $\Sha^1(G)$ and $\Sha^2(\widehat{G})$, and between $\Sha^1(G^{\wedge\wedge})$ and $\Sha^2(\widehat{G})$. Via Cartier duality, of course, we have, canonically, $G^{\wedge\wedge} = G$, and we would like to know that these two pairings are compatible with Cartier duality. This follows immediately from the fact that the pairing above is defined in terms of the underlying sheaves and the pairing between them, with no reference to their definition in terms of $G$.
\end{remark}

\section{Injectivity of $\Sha^1(G) \rightarrow \Sha^2(\widehat{G})^*$}
\label{Sha^1(G)-->Sha^2(G^)*injectssection}

The goal of this section is to prove that if $G$ is an affine commutative group scheme of finite type over a global function field $k$, then the map $\Sha^1(k, G) \rightarrow \Sha^2(k, \widehat{G})^*$ induced by the pairing $\langle \cdot, \cdot \rangle_{\Sha}$ is injective (Proposition \ref{Sha^1--->Sha^2*injective}). As with many of the results in this work, the same result is true for number fields, but we concentrate on the function field setting. Unlike many of the other results of this manuscript, however, we do {\em not} take the finite case as a black box; see Remark \ref{prior}.

\begin{lemma}
\label{cartesianprofinitefg}
For any inclusion $A \hookrightarrow A'$ of discrete finitely generated abelian groups, the following diagram is Cartesian${\rm{:}}$
\[
\begin{tikzcd}
A \arrow[r, hookrightarrow] \arrow[d, hookrightarrow] & A' \arrow[d, hookrightarrow] \\
A_{\pro} \arrow{r} & A'_{\pro}
\end{tikzcd}
\]
\end{lemma}

\begin{proof}
Using the structure theorem for finitely generated abelian groups, one sees that the canonical map
$\Z_{\pro} \otimes_{\Z} A \rightarrow A_{\rm{pro}}$ is an isomorphism.   If we define $A'' = A'/A$, then by the $\Z$-flatness of
$\Z_{\pro}$ we have a commutative diagram of short exact sequences
\[
\begin{tikzcd}
0 \arrow{r} & A \arrow{r} \arrow{d} & A' \arrow{d} \arrow{r} & A'' \arrow{d} \arrow{r} & 0 \\
0 \arrow{r} & A_{\pro} \arrow{r} & A'_{\pro} \arrow{r} & A''_{\pro} \arrow{r} & 0
\end{tikzcd}
\]
Thus, the Cartesian assertion reduces to showing that $M \rightarrow M_{\rm{pro}}$ is injective
for any finitely generated abelian group $M$ (applied to $M = A''$).
Such injectivity is immediate either from the faithful flatness of $\Z_{\pro}$ over $\Z$, or
by using the structure theorem to reduce to the easy cases when $M = \Z$ and $M = \Z/n\Z$.
\end{proof}

\begin{lemma}
\label{G(A)/G(k)cartesian}
Let $k$ be a global function field. Suppose that we have an inclusion $G \hookrightarrow G'$ of affine commutative $k$-group schemes of finite type. Then the following diagram is Cartesian: 
\[
\begin{tikzcd}
G(\A)/G(k) \arrow[d, hookrightarrow] \arrow[r, hookrightarrow] & G'(\A)/G'(k) \arrow[d, hookrightarrow] \\
(G(\A)/G(k))_{\pro} \arrow{r} & (G'(\A)/G'(k))_{\pro}
\end{tikzcd}
\]
\end{lemma}

\begin{proof}
The vertical maps are inclusions by Lemma \ref{G(A)/G(k)containedincompletion}. For ease of notation, let $B : = G(\A)/G(k)$, $B' : = G'(\A)/G'(k)$. Also let $B_1 : = G(\A)_1/G(k)$ and $B'_1 : = G'(\A)_1/G'(k)$. (For the definition of $G(\A)_1$, see Definition \ref{G(A)_1def}.) Then $B/B_1$ and $B'/B'_1$ are discrete and finitely generated, and $B_1, B'_1$ are profinite (Proposition \ref{G(A)_1/G(k)profinite}).

We claim that $B/B_1 \rightarrow B'/B'_1$ is an inclusion. Indeed, let $g \in G(\A) \cap G'(\A)_1$
and $\chi \in \widehat{G}(k)$. We need to show that $|\!|\chi(g)|\!| = 1$. By Corollary \ref{characterfinitecokernel}, the character $\chi^n$ extends to a character in $\widehat{G'}(k)$ for some positive integer $n$. We then have that $|\!|\chi^n(g)|\!| = 1$, so $|\!|\chi(g)|\!| = 1$, hence $g \in G(\A)_1$, as desired.

We will repeatedly use the right-exactness of profinite completion (Proposition \ref{profiniterightexact}). Consider $b' \in B'$ and $b \in B_{\pro}$ such that $b$ maps to 
the image of $b'$ in $B'_{\pro}$. We want to show that $b$ comes from an element of $B$. The diagram 
\[
\begin{tikzcd}
B/B_1 \arrow[r, hookrightarrow] \arrow[d, hookrightarrow] & B'/B'_1 \arrow[d, hookrightarrow] \\
(B/B_1)_{\pro} \arrow{r} & (B'/B'_1)_{\pro}
\end{tikzcd}
\]
is Cartesian by Lemma \ref{cartesianprofinitefg}, so by modifying $b'$ by an element of $B$, we may assume that $b' \in B_1'$. Since the map $B_{\pro}/(B_1)_{\pro} = (B/B_1)_{\pro} \rightarrow (B'/B'_1)_{\pro} = B'_{\pro}/(B'_1)_{\pro}$ is an inclusion by Proposition \ref{profiniteexactfgquotient}, we see that $b$ lies in the image of the map $(B_1)_{\pro} = B_1 \rightarrow B_{\pro}$, as desired.
\end{proof}

\begin{lemma}
\label{sha^1sha^2*injectivedevissage}
Let $k$ be a global function field. Suppose that we have an inclusion $G \hookrightarrow G'$ of affine commutative $k$-group schemes of finite type such that the map $\Sha^1(G') \rightarrow \Sha^2(G')^*$ induced by $\langle \cdot, \cdot \rangle_{\Sha}$ is injective. Then the map $\Sha^1(G) \rightarrow \Sha^2(G)^*$ is also injective.
\end{lemma}

\begin{proof}
Let $H : = G'/G$, so we have an exact sequence
\[
1 \longrightarrow G \xlongrightarrow{j} G' \xlongrightarrow{\pi} H \longrightarrow 1.
\] 
Consider the following commutative diagram with exact rows: 
\[
\begin{tikzcd}
& G'(k) \arrow{r}{\pi} \arrow[d, hookrightarrow] & H(k) \arrow{r}{\delta} \arrow[d, hookrightarrow] & {\rm{H}}^1(k, G) \arrow{r}{j} \arrow{d} & {\rm{H}}^1(k, G') \arrow{d} \\
G(\A) \arrow{r}{j} & G'(\A) \arrow{r}{\pi} & H(\A) \arrow{r}{\delta} & {\rm{H}}^1(\A, G) \arrow{r} & {\rm{H}}^1(\A, G')
\end{tikzcd}
\]
Suppose that $\alpha \in \Sha^1(G)$ annihilates $\Sha^2(\widehat{G})$ under $\langle \cdot, \cdot \rangle_{\Sha}$. We want to show that $\alpha = 0$.
By functoriality of $\langle \cdot, \cdot \rangle_{\Sha}$ and our hypothesis that $\Sha^1(G') \rightarrow \Sha^2(\widehat{G})^*$ is injective, we deduce that $j(\alpha) = 0$, hence $\alpha = \delta(h)$ for some $h \in H(k)$. Then $\delta(h_{\A}) = 0$
in ${\rm{H}}^1(\A, G)$, so $h_{\A} = \pi(g')$ for some $g' \in G'(\A)$. 
We are going to show that by adjusting $g'$ by an element of $G'(k)$ (which corresponds to changing $h$ by an element
of $\pi(G'(k))$, as we are free to do for the purpose of studying $\alpha = \delta(h)$), we can arrange that $g'$ comes from $G(\A)$, so $h_{\A} \in \pi(j(G(\A))) = 0$,
forcing $h = 0$ and $\alpha = 0$, as desired. Note that in order to do this, we are free to modify $g'$ by an element of $G(\A)$.

The natural map $G'(\A) \rightarrow {\rm{H}}^2(k, \widehat{G'})^*$ induced by the global duality pairing assigns to $g'$ a homomorphism ${\rm{H}}^2(k, \widehat{G'}) \rightarrow \Q/\Z$. Consider the subgroup $\widehat{j}^{-1}(\Sha^2(\widehat{G})) \subset {\rm{H}}^2(k, \widehat{G'})$, where $\widehat{j}:  {\rm{H}}^2(k, \widehat{G'}) \rightarrow {\rm{H}}^2(k, \widehat{G})$ is the map induced by $j$. The key point for controlling the choice of $g'$ is the following lemma.

\begin{lemma}
\label{g'killsSha^2(G^)preimage}
The homomorphism assigned to $g'$ annihilates $\widehat{j}^{-1}(\Sha^2(\widehat{G})) \subset {\rm{H}}^2(k, \widehat{G'})$.
\end{lemma}

\begin{proof}
Since $\alpha$ annihilates $\Sha^2(\widehat{G})$ by assumption, it suffices to show (without any hypotheses
on $\alpha$!) that for any 
$z \in \widehat{j}^{-1}(\Sha^2(\widehat{G}))$ we have 
\begin{equation}
\label{g'killsSha^2(G^)preimageeqn1}
\langle g', z\rangle \stackrel{?}{=} -\langle \widehat{j}(z), \alpha \rangle_{\Sha}.
\end{equation}
Lift $h \in H(k)$ to some $x \in \check{C}^0(k, G') = G'(\overline{k})$. Then $\pi(dx) = dh = 0$, so $dx = j(\alpha')$ for some 
1-cochain $\alpha' \in \check{C}^1(k, G)$ that is a 1-cocycle (as $d(dx) = 0$ and $j$ is injective). 
The \v{C}ech cohomology class of $\alpha'$ coincides with 
the derived functor cohomology class of $\delta(h) = \alpha$ since the identification of
\v{C}ech and derived functor cohomology in degrees $\le 1$ is $\delta$-functorial by Proposition \ref{cechderivedconnectingmap},
so we will abuse notation and denote the 1-cocycle $\alpha'$ as $\alpha$. 

Since $\widehat{j}(z) \in \Sha^2(\widehat{G})$ by hypothesis, if we also write $z$ to denote
a representative class in $\check{Z}^2(k, \widehat{G})$ (which makes sense by Proposition \ref{cech=derivedG^}!), then 
for each place $v$ of $k$ we have $\widehat{j}(z) = dy_v$ for some $y_v \in \check{C}^1(k_v, \widehat{G})$. 
We have $\widehat{j}(z) \cup \alpha = dm$ for some $m \in \check{C}^2(k, \Gm)$, since
$\check{\rm{H}}^3(k, \Gm) = 0$ (Lemma \ref{checkh3=0}). 
Then by definition of the two pairings,  (\ref{g'killsSha^2(G^)preimageeqn1}) is equivalent to
\begin{equation}
\label{g'killsSha^2(G^)preimageeqn2}
\sum_v \inv_v(g_v' \cup z) \stackrel{?}{=} \sum_v -\inv_v((y_v \cup \alpha) - m).
\end{equation}
Since $\pi(g'_v) = h = \pi(x)$, we deduce that $g'_v = x + j(f_v)$ for some $f_v \in \check{C}^0(k_v, G)$. 
Applying $d$ to both sides yields (since $g'_v$ is a cocycle, representing an element of ${\rm{H}}^0(k_v, G')$)
 $$0 = dx + dj(f_v) = j(\alpha) + j(df_v),$$ so $df_v = -\alpha$. Now we compute at the cochain level
\[
g_v' \cup z = (x + j(f_v)) \cup z = (x \cup z) + (f_v \cup \widehat{j}(z)) = (x \cup z) + (f_v \cup dy_v) = (x \cup z) - (df_v \cup y_v) + d(f_v \cup y_v).
\]
The final expression is cohomologous to
\[
(x \cup z) - (df_v \cup y_v) = (x \cup z) + (\alpha \cup y_v). 
\]
The way we have arrived at this ensures that the final expression is a cocycle, and likewise
the differences $-((y_v \cup \alpha) - m)$ appearing on the right side of (\ref{g'killsSha^2(G^)preimageeqn2}) are cocycles, so
taking the difference gives that the global cochain $(x \cup z) - m$ is a cocycle (as this can be checked locally at any single place).
Thus, the desired identity (\ref{g'killsSha^2(G^)preimageeqn2}) is equivalent to the vanishing of the sum
\[
\sum_v \inv_v((x \cup z) - m)
\]
of the local invariants of a global Brauer class, and such sums always vanish. 
\end{proof}

Via the inclusion
\[
{\rm{H}}^2(k, \widehat{G'})/\widehat{j}^{-1}(\Sha^2(\widehat{G})) \hookrightarrow {\rm{H}}^2(k, \widehat{G})/\Sha^2(\widehat{G}),
\]
we can extend the homomorphism induced by $g'$ to a homomorphism on ${\rm{H}}^2(k, \widehat{G})$ that kills $\Sha^2(\widehat{G})$. We would like to say that this homomorphism is induced by an element of $G(\A)$. 
Thanks to Propositions \ref{exactnessatH^2(k,G^)*}, \ref{exactnessatG(A)}, and \ref{exactnessatG(A)nocompletion}, this follows from Lemma \ref{G(A)/G(k)cartesian}.

Modifying $g'$ by a suitable element of $G(\A)$, therefore (as we are free to do), we may assume that it annihilates all of ${\rm{H}}^2(k, \widehat{G'})$. By Proposition \ref{exactnessatG(A)nocompletion}, it follows that it lifts to an element of $G'(k)$. Modifying $h$ by this element of $G'(k)$ (as we may do), we obtain $h_{\A} = 0$, hence $h = 0$, so $\alpha = \delta(h) = 0$, and the proof of Lemma \ref{sha^1sha^2*injectivedevissage} is complete.
\end{proof}

\begin{lemma}
\label{separableH^1injectiveinfinitesimal}
Let $k$ be a field, $L/k$ a $($not necessarily algebraic$)$ separable extension field of $k$, and let $I$ be an infinitesimal $k$-group scheme. Then the map ${\rm{H}}^1(k, I) \rightarrow {\rm{H}}^1(L, I)$ is injective.
\end{lemma}

\begin{proof}
Since $L/k$ is separable, we may write $L$ as a filtered direct limit of smooth $k$-algebras. Therefore, $E(A) \neq \emptyset$ for some (nonzero) smooth $k$-algebra $A$. Since $A$ is $k$-smooth, $A(k') \neq \emptyset$ for some finite Galois extension $k'/k$. Pulling back along such a point, we are reduced to the case in which $L/k$ is finite Galois. Then $I$, hence also $E$, is still radicial over $L$, hence $E(L)$ consists of a singleton. This point is invariant under $\Gal(L/k)$ (since it is the only $L$-point of $E$), hence descends to a $k$-point of $E$.
\end{proof}

\begin{lemma}
\label{sha^1=0extendscalars}
Let $G$ be an affine commutative group scheme of finite type over a global field $k$. Then for some finite extension $k'/k$,  we have $\Sha^1(G_{k''}) = 0$ for all finite extensions $k''/k$.
\end{lemma}

\begin{proof}
After such a base change, we may assume that $G^0_{\red}$ is a subgroup scheme and is in fact the product of a split torus and a split unipotent group. Thus, ${\rm{H}}^1(k'', G^0_{\red}) = 0$ for all $k''/k$, so we only need to show that $\Sha^1(B_{k''}) = 0$ for all finite $k''/k$, where $B = G/G^0_{\red}$. By extending the base further, we may assume that $B$ is the product of an infinitesimal and a finite {\'e}tale group scheme. We know that $\Sha^1$ vanishes for infinitesimal groups (by Lemma \ref{separableH^1injectiveinfinitesimal} applied to the separable extension $k_v/k$), so we are left to 
treat the case when $B = E$ is {\'e}tale. Extending the base further, we may assume that $E$ is constant. Since $\Sha^1(k'', \Z/n\Z) = 0$ for all finite $k''/k$ by the Chebotarev Density Theorem, we are done.
\end{proof}

\begin{proposition}
\label{Sha^1--->Sha^2*injective}
Let $G$ be an affine commutative group scheme of finite type over a global function field $k$. Then the map $\Sha^1(G) \rightarrow \Sha^2(\widehat{G})^*$ induced by $\langle \cdot, \cdot \rangle_{\Sha}$ is injective.
\end{proposition}

\begin{proof}
Using Lemma \ref{sha^1=0extendscalars}, let $k'/k$ be a finite extension such that $\Sha^1(G_{k'}) = 0$. Using the canonical inclusion $G \hookrightarrow \R_{k'/k}(G_{k'})$ and Lemma \ref{sha^1sha^2*injectivedevissage}, it suffices to prove the proposition for $\R_{k'/k}(G_{k'})$. But we have a Leray spectral sequence associated to the morphism $f\colon \Spec(k') \rightarrow \Spec(k)$, namely
\[
E_2^{i,j} = {\rm{H}}^i(k, \R^jf_*G_{k'}) \Longrightarrow {\rm{H}}^{i+j}(k', G_{k'}), 
\]
and similarly for the morphism $\Spec(\A_{k'}) \rightarrow \Spec(\A_k)$. By compatibility of these spectral sequences,
we get an inclusion $\Sha^1(k, \R_{k'/k}(G_{k'}))$ $\hookrightarrow \Sha^1(k', G_{k'})$. Since the latter group vanishes, so does the former, hence the proposition holds trivially for $\R_{k'/k}(G_{k'})$.
\end{proof}

\section{Injectivity of $\Sha^2(G) \rightarrow \Sha^1(\widehat{G})^*$}

The purpose of this section is to prove that, for an affine commutative group scheme of finite type over a global function field $k$, the map $\Sha^2(k, G) \rightarrow \Sha^1(k, \widehat{G})^*$ induced by the pairing $\langle \cdot, \cdot \rangle_{\Sha}$ is injective, under an assumption which we will later see always holds (Proposition \ref{Sha^2(G)-->Sha^1(G^)*finite}). As a consequence of this injectivity and Proposition \ref{Sha^1--->Sha^2*injective}, we will deduce that Theorem \ref{shapairing} holds for finite commutative group schemes (Corollary \ref{shapairingfinite}).

\begin{lemma}
\label{Sha^1(G)=0}
Let $G$ be a finite commutative group scheme over a global field $k$. There exists a finite separable extension $k'/k$ such that $\Sha^1(k'', G_{k''}) = 0$ for all finite extensions $k''/k'$.
\end{lemma}

\begin{proof}
We have the connected-{\'e}tale exact sequence
\[
1 \longrightarrow I \longrightarrow G \longrightarrow E \longrightarrow 1
\]
with $I$ infinitesimal and $E$ {\'e}tale. Replacing $k$ with a finite separable extension, we may assume that $E$ is constant. Of course, this constancy is preserved under any further extension of scalars. It suffices to show, then, that if $E$ is constant, then $\Sha^1(k, G) = 0$. We have a commutative diagram with exact rows
\[
\begin{tikzcd}
E(k) \arrow{r} \arrow{d} & {\rm{H}}^1(k, I) \arrow{r} \arrow{d} & {\rm{H}}^1(k, G) \arrow{r} \arrow{d} & {\rm{H}}^1(k, E) \arrow[d, hookrightarrow] \\
\prod_v E(k_v) \arrow{r} & \prod_v {\rm{H}}^1(k_v, I) \arrow{r} & \prod_v {\rm{H}}^1(k_v, G) \arrow{r} & \prod_v {\rm{H}}^1(k_v, E),
\end{tikzcd}
\]
in which the last vertical arrow is an inclusion because $\Sha^1(E) = 0$ by the Chebotarev density Theorem. Now suppose that $\alpha \in \Sha^1(G)$. Then looking at the diagram, we see that $\alpha$ lifts to some $\beta \in {\rm{H}}^1(k, I)$. Fix a place $v$ of $k$. Then $\beta_v \mapsto 0 \in {\rm{H}}^1(k_v, G)$, hence lifts to some $e_v \in E(k_v)$. Since $E$ is constant, we may lift $e_v$ to an $e \in E(k)$. Then, modifying $\beta$ by $e$, we see that $\beta_v = 0$, hence $\beta = 0$ by Lemma \ref{separableH^1injectiveinfinitesimal}. Therefore $\alpha = 0$ as well, so $\Sha^1(G) = 0$.
\end{proof}

Next we would like to prove the analogue of Lemma \ref{Sha^1(G)=0} for $\Sha^2$ (see Lemma \ref{Sha^2(G)=0}). This lies deeper, and we require several preparatory lemmas.

\begin{lemma}
\label{H^2=0forinfdual}
If $G$ is a finite commutative group scheme over a field $k$, with $\widehat{G}$ infinitesimal, then ${\rm{H}}^2(k, G) = 0$.
\end{lemma}

\begin{proof}
Using the connected-\'etale sequence, we may assume that $G$ is either local-local or \'etale. The latter case follows from the fact that fields of characteristic $p$ have $p$-cohomological dimension at most $1$ \cite[Ch.\,II, \S2.2, Prop.\,3]{serre}. For the former, we may filter $G$ and thereby assume that $G$ is killed by both Frobenius and Verschiebung. Over $\overline{k}$, this implies by Dieudonn\'e theory that $G$ is isomorphic to $\alpha_p^n$ for some $n$. Since ${\rm{H}}^1(k, {\rm{Aut}}_{\alpha_p^n/k}) = {\rm{H}}^1(k, {\rm{GL}}_n) = 1$, this implies that $G \simeq \alpha_p^n$. We may therefore assume that $G = \alpha_p$. Then the lemma follows from the exact sequence
\[
1 \longrightarrow \alpha_p \longrightarrow \Ga \xlongrightarrow{F} \Ga \longrightarrow 1,
\]
in which $F$ is the relative Frobenius isogeny, and from the vanishing of the higher cohomology of $\Ga$.
\end{proof}

\begin{lemma}
\label{finiteexpcountable}
Let $k$ be a global field, $G$ an affine commutative $k$-group scheme of finite type. Then ${\rm{H}}^i(k, G)$ and ${\rm{H}}^i(k, \widehat{G})$ are countable for $i \leq 2$.
\end{lemma}

\begin{proof}
Lemma \ref{affinegroupstructurethm}, Proposition \ref{hatisexact}, and Proposition \ref{etrem} when ${\rm{char}}(k) = 0$ reduce us to the case when $G$ is either an almost-torus or $\Ga$. Further, Lemma \ref{almosttorus}(ii) reduces the almost-torus case to the case when $G$ is either finite or a torus. So we may assume that $G$ is either finite, a torus, or $\Ga$.

For $G = \Ga$, all of the relevant cohomology groups vanish except for ${\rm{H}}^0(k, \Ga) = k$, which is countable, and possibly ${\rm{H}}^2(k, \widehat{\Ga})$ (Propositions \ref{cohomologyofG_adualgeneralk} and \ref{unipotentcohomology}(ii)). So we are left to show that ${\rm{H}}^2(k, \widehat{\Ga})$ is countable. If ${\rm{char}}(k) = 0$, then this group vanishes (Proposition \ref{cohomologyofG_adualwhenkisperfect}), while if ${\rm{char}}(k) > 0$, then it is a one-dimensional $k$-vector space (Proposition \ref{omega=H^2(Ga^)}), hence countable.

We next treat the case when $G$ is finite, for which we only need to treat the cohomology of $G$ (and not also $\widehat{G}$) by Cartier duality. We proceed by induction on $i$. The assertion is trivial for $i = 0$. If it holds for $i -1$, then to deduce it for $i$, we first note that we may filter $B$ and therefore assume that it is either {\'e}tale, multiplicative, or local-local. If $B$ is local-local, then we may filter it further to assume that its relative Frobenius and Verschiebung morphisms vanish, hence $B = \alpha_p^n$, so then we may assume $B = \alpha_p$. Further note that we may replace $(k, B)$ with $(k', B_{k'})$ for any finite separable extension $k'/k$
(as follows by induction on $i$ via the canonical inclusion $B \hookrightarrow \R_{k'/k}(B_{k'})$ and the fact that ${\rm{H}}^i(k, \R_{k'/k}(B_{k'})) \simeq {\rm{H}}^i(k', B_{k'})$ (using Lemma \ref{sepblepushforward}). We may therefore assume that either $B = \alpha_p$, $B = \mu_n$, or $B = \Z/p\Z$, with $p = {\rm{char}}(k) > 0$ in the first and third cases.

If $p = {\rm{char}}(k)>0$, then ${\rm{H}}^1(k, \alpha_p) = k/k^p$ is countable and ${\rm{H}}^2(k, \alpha_p) = 0$ by Lemma \ref{H^2=0forinfdual}. For $\mu_n$
and arbitrary characteristic, we have ${\rm{H}}^1(k, \mu_n) = k^{\times}/(k^{\times})^n$ (which is countable) and ${\rm{H}}^2(k, \mu_n) = {\rm{H}}^2(k, \Gm)[n]$ (which is countable by global class field theory). Finally, if $p = {\rm{char}}(k)>0$ then ${\rm{H}}^1(k, \Z/p\Z) = k/\wp(k)$ (hence countable) with $\wp:  k \rightarrow k$ the Artin-Schreier map $x \mapsto x^p - x$, and ${\rm{H}}^2(k, \Z/p\Z) = 0$ by Lemma \ref{H^2=0forinfdual}.

Next we treat the remaining case when $G = T$ is a torus. The cases with $i = 0$ are clear, so we may assume that $i > 0$. Since $T$ and $\widehat{T}$ are represented by smooth $k$-group schemes, the cohomology may be taken to be \'etale. Since higher Galois cohomology is torsion, it suffices to show that the groups ${\rm{H}}^i(k, T)[n]$ are countable for all $n > 0$, and similarly for $\widehat{T}$. For this, we have an exact sequence
\[
1 \longrightarrow T[n] \longrightarrow T \xlongrightarrow{[n]} T \longrightarrow 1
\]
with $T[n]$ a finite $k$-group scheme. We therefore obtain for all $i$ a surjective map ${\rm{H}}^i(k, T[n]) \twoheadrightarrow {\rm{H}}^i(k, T)[n]$, so the already-treated finite case implies that the latter group is countable. The groups ${\rm{H}}^i(k, \widehat{T})[n]$ are treated similarly, by using the exact sequence
\[
1 \longrightarrow \widehat{T} \xlongrightarrow{[n]} \widehat{T} \longrightarrow \widehat{T}/[n]\widehat{T} \longrightarrow 1
\]
in which $\widehat{T}/[n]\widehat{T}$ is a finite $k$-group scheme.
\end{proof}

\begin{proposition}
\label{exactnessatH^1(G^)*finite}
Let $G$ be an affine commutative group scheme of finite type over a global function field $k$. Consider the following sequences{\rm{:}}
\[
{\rm{H}}^1(\A, G) \longrightarrow {\rm{H}}^1(k, \widehat{G})^* \longrightarrow \Sha^1(\widehat{G})^* \longrightarrow 0,
\]
\begin{equation}
\label{exactnessatH^1(G^)*finiteeqn10}
{\rm{H}}^1(\A, \widehat{G}) \longrightarrow {\rm{H}}^1(k, G)^* \longrightarrow \Sha^1(G)^* \longrightarrow 0,
\end{equation}
in which the maps ${\rm{H}}^1(\A, G) \rightarrow {\rm{H}}^1(k, \widehat{G})^*$ and ${\rm{H}}^1(\A, \widehat{G}) \rightarrow {\rm{H}}^1(k, G)^*$ are induced by the global duality pairings.
\begin{itemize}
\item[(i)] If $G$ is finite, then the above sequences are exact.
\item[(ii)] If the sequences
\[
{\rm{H}}^1(k, \widehat{G}) \longrightarrow {\rm{H}}^1(\A, \widehat{G}) \longrightarrow {\rm{H}}^1(k, G)^*
\]
\[
{\rm{H}}^1(k, G) \longrightarrow {\rm{H}}^1(\A, G) \longrightarrow {\rm{H}}^1(k, \widehat{G})^*
\]
are exact, then the sequences $(\ref{exactnessatH^1(G^)*finiteeqn10})$ are also exact.
\end{itemize}
\end{proposition}

\begin{proof}
Assertion (i) follows from (ii), since the exactness of those sequences for finite $G$ is part of the Poitou--Tate sequence for finite commutative group schemes (and even of the part that we are assuming to be exact! see Remark \ref{prior}). So we concentrate on (ii). We prove that the first sequence is exact. The proof for the second sequence is analogous. Exactness at $\Sha^1(k, \widehat{G})^*$ follows from the injectivity of $\Q/\Z$ as an abelian group, so we concentrate on proving exactness at ${\rm{H}}^1(k, \widehat{G})^*$.
Consider the exact sequence
\begin{equation}
\label{exactnessatH^1(G^)*finiteseq1}
0 \longrightarrow \Sha^1(\widehat{G}) \longrightarrow {\rm{H}}^1(k, \widehat{G}) \longrightarrow {\rm{H}}^1(\A, \widehat{G}).
\end{equation}
We claim that applying $(\cdot)^D$ to this sequence preserves exactness (where we are endowing ${\rm{H}}^1(k, \widehat{G})$ and $\Sha^1(\widehat{G})$ with the discrete topology).

For this, it suffices to show that the continuous inclusion ${\rm{H}}^1(k, \widehat{G})/\Sha^1(\widehat{G}) \hookrightarrow {\rm{H}}^1(\A, \widehat{G})$ is a homeomorphism onto a closed subgroup. (Recall that ${\rm{H}}^1(\A, \widehat{G})$ is a locally compact Hausdorff abelian group by Proposition \ref{adelictopcohombasicsG^}(iii), so the formalism and results of Pontryagin duality apply to it.)
 By Lemmas \ref{isom=homeo} and \ref{finiteexpcountable}, and Proposition \ref{adelictopcohombasicsG^}(iii), it suffices to simply show that the image is closed. This follows from exactness of the sequence of {\em continuous} maps
\[
{\rm{H}}^1(k, \widehat{G}) \longrightarrow {\rm{H}}^1(\A, \widehat{G}) \longrightarrow {\rm{H}}^1(k, G)^*,
\]
since this expresses this image as the kernel of a continuous map to a Hausdorff group. Note that continuity of the last map in the above sequence follows from the continuity of the adelic pairing ${\rm{H}}^1(\A, G) \times {\rm{H}}^1(\A, \widehat{G}) \rightarrow \Q/\Z$ (which is contained in Proposition \ref{H^1(A,G)=H^1(A,G^)^D}), together with continuity of the map ${\rm{H}}^1(\A, G)^D \rightarrow {\rm{H}}^1(k, G)^D$ dual to the trivially continuous map ${\rm{H}}^1(k, G) \rightarrow {\rm{H}}^1(\A, G)$.

Applying $(\cdot)^D$, therefore, to (\ref{exactnessatH^1(G^)*finiteseq1}), and using the fact that ${\rm{H}}^1(k, \widehat{G})^* = {\rm{H}}^1(k, \widehat{G})^D$
by Lemma \ref{H^1finiteexponent}, we get an exact sequence
\[
{\rm{H}}^1(\A, \widehat{G})^D \longrightarrow {\rm{H}}^1(k, \widehat{G})^* \longrightarrow \Sha^1(\widehat{G})^*.
\]
Proposition \ref{H^1(A,G)=H^1(A,G^)^D} (and the compatibility of the pairings ${\rm{H}}^1(\A, G) \times {\rm{H}}^1(\A, \widehat{G}) \rightarrow \Q/\Z$ and ${\rm{H}}^1(\A, G) \times {\rm{H}}^1(k, \widehat{G}) \rightarrow \Q/\Z$) then completes the proof of the proposition.
\end{proof}

\begin{lemma}
\label{Sha^2(G)=0devissage}
Suppose that we have an exact sequence
\[
1 \longrightarrow G' \xlongrightarrow{j} G \xlongrightarrow{\pi} G'' \longrightarrow 1
\]
of finite commutative group schemes over a global function field $k$, and suppose that $\alpha \in \Sha^2(G)$ annihilates $\Sha^1(\widehat{G})$ under $\langle \cdot, \cdot \rangle_{\Sha_G}$ and lifts to an element of ${\rm{H}}^2(k, G')$. Then $\alpha$ lifts to an element of $\Sha^2(G')$.
\end{lemma}

\begin{proof}
Consider the following commutative diagram with exact rows: 
\[
\begin{tikzcd}
{\rm{H}}^1(k, G) \arrow{d} \arrow{r} & {\rm{H}}^1(k, G'') \arrow{d} \arrow{r} & {\rm{H}}^2(k, G') \arrow{d} \arrow{r} & {\rm{H}}^2(k, G) \arrow{d} \\
{\rm{H}}^1(\A, G) \arrow{r} \arrow{d} & {\rm{H}}^1(\A, G'') \arrow{r}{\delta} \arrow{d} & {\rm{H}}^2(\A, G') \arrow{r} & {\rm{H}}^2(\A, G) \\
{\rm{H}}^1(k, \widehat{G})^* \arrow{d} \arrow{r} & {\rm{H}}^1(k, \widehat{G''})^*  \arrow{d} &&\\
\Sha^1(\widehat{G})^* \arrow{r} & \Sha^1(\widehat{G''})^* &&
\end{tikzcd}
\]
The two leftmost columns are exact by Proposition \ref{exactnessatH^1(G^)*finite} and by Poitou--Tate for finite commutative group schemes (see Remark \ref{prior}). Lift $\alpha$ to a $\beta \in {\rm{H}}^2(k, G')$, so $\beta_{\A} \in {\rm{H}}^2(\A, G')$ lifts to some $\gamma \in {\rm{H}}^1(\A, G'')$. Via the adelic pairing, $\gamma$ yields an element 
$\phi_{\gamma} \in {\rm{H}}^1(k, \widehat{G''})^*$.

Due to the injectivity of the map $${\rm{H}}^1(k, \widehat{G''})/\widehat{\pi}^{-1}(\Sha^1(\widehat{G})) \rightarrow
{\rm{H}}^1(k, \widehat{G})/\Sha^1(\widehat{G})$$ and the injectivity of the abelian group $\Q/\Z$, Lemma \ref{gammakillsShapreimage} below implies that $\phi_{\gamma}$ lifts to an element $\psi \in ({\rm{H}}^1(k, \widehat{G})/\Sha^1(\widehat{G}))^*$. 
By Proposition \ref{exactnessatH^1(G^)*finite}, there exists a $g \in {\rm{H}}^1(\A, G)$ inducing $\psi$. Modifying $\gamma$ by $g$ (which has no effect on $\delta(\gamma)$), we then obtain $\phi_{\gamma} = 0$, hence $\gamma$ lifts to an element $g'' \in {\rm{H}}^1(k, G'')$. Modifying $\beta$ by the image of $g''$ (which has no effect on $j(\beta) = \alpha$), we then obtain that $\beta \in \Sha^2(G')$, as desired.
\end{proof}

\begin{lemma}
\label{gammakillsShapreimage}
With notation as above, $\phi_{\gamma}$ annihilates $\widehat{\pi}^{-1}(\Sha^1(\widehat{G})) \subset {\rm{H}}^1(k, \widehat{G''})$.
\end{lemma}

\begin{proof}
The key point is to prove the following equality:  for $z \in \widehat{\pi}^{-1}(\Sha^1(\widehat{G}))$, we have
\begin{equation}
\label{gammakillsShapreimageeq1}
\langle \gamma, z \rangle = \langle \widehat{\pi}(z), \alpha \rangle_{\Sha_G}.
\end{equation}
Since $\alpha$ annihilates $\Sha^1(\widehat{G})$ by assumption, this would imply what we want.

Let us first note that (\ref{gammakillsShapreimageeq1}) is independent of the choice of the lift $\gamma \in
{\rm{H}}^1(\A,G'')$. Indeed, if we change $\gamma$ by $\pi(\zeta)$ for some $\zeta \in {\rm{H}}^1(\A, G)$, then for $c \in \widehat{\pi}^{-1}(\Sha^1(\widehat{G}))$, $\langle \gamma, c\rangle$ changes by $\langle \pi(\zeta), c\rangle = \langle \zeta, \widehat{\pi}(c) \rangle = 0$, since $\zeta$ pairs trivially with $\Sha^1(\widehat{G})$ (because an element of $\Sha^1(\widehat{G})$ is by definition locally trivial).

We have $\delta(\gamma) = \beta_{\A}$ as cohomology classes. For the rest of this proof we shall use the symbols $\gamma, \alpha$, etc. to denote fixed \v{C}ech cocycles representing the corresponding 
cohomology classes (which makes sense by Propositions \ref{cech=derivedsmoothinf} and \ref{cech=derivedG^}). 
Consider the \v{C}ech cocycle $\gamma'$ defined as follows.  
We know that the image $j(\beta_{\A}) \in {\rm{H}}^2(\A, G)$ vanishes, so its image in the \v{C}ech cohomology group 
$\check{\rm{H}}^2(\A,G)$  vanishes since the edge map $\check{\rm{H}}^2 \rightarrow {\rm{H}}^2$ is always injective
(due to $\check{\rm{H}}^1 \rightarrow {\rm{H}}^1$ being an isomorphism).  Hence, $j(\beta_{\A}) = dw$ 
in $\check{Z}^2(\A, G)$ for some $w \in \check{C}^1(\A, G)$. The \v{C}ech 1-cochain $\gamma' : = \pi(w) \in \check{C}^1(\A, G'')$
is a cocycle since $d\pi(w) = \pi(dw) = \pi(j(\beta_{\A})) = 0$ in $\check{C}^2(\A,G'')$ due to the vanishing of $\pi \circ j$. 
Hence, by Proposition \ref{cechderivedconnectingmap}, $\delta(\gamma') = \beta_{\A}$ as \v{C}ech cohomology classes. 
We may then replace $\gamma$ with the cocycle $\gamma'$, and thereby assume that $\gamma = \pi(w)$ and $dw = j(\beta_{\A})$
as \v{C}ech cochains. We may also replace the \v{C}ech cocycle $\alpha \in \check{Z}^2(k,G)$ 
with the cohomologous \v{C}ech cocycle $j(\beta)$ so that $\alpha=j(\beta)$ as \v{C}ech 2-cocycles and not just as cohomology classes.

The left side of (\ref{gammakillsShapreimageeq1}) is by definition
\[
\sum_v \inv_v(\gamma \cup z).
\]
Since $\widehat{\pi}(z)$ represents an element of $\Sha^1(\widehat{G}) = \ker({\rm{H}}^1(k,\widehat{G}) \rightarrow
{\rm{H}}^1(\A, \widehat{G}))$ with $\check{\rm{H}}^1 = {\rm{H}}^1$, we have $\widehat{\pi}(z) = dx$ for some $x \in \check{C}^0(\A, \widehat{G})$. Further, we have $\alpha \cup \widehat{\pi}(z) = dh$ for some $h \in \check{C}^2(k, \Gm)$ since
$\check{\rm{H}}^3(k,\Gm)=0$ by Lemma \ref{checkh3=0}.  The right side of (\ref{gammakillsShapreimageeq1}) is then by definition
\[
\sum_v \inv_v((\alpha \cup x_v) - h).
\]
Now we have, as adelic \v{C}ech 2-cocycles,
\[
\gamma \cup z = \pi(w) \cup z = w \cup \widehat{\pi}(z) = w \cup dx = dw \cup x - d(w \cup x).
\]
The last expression is cohomologous to
\[
dw \cup x = j(\beta) \cup x = \alpha \cup x.
\]
It follows that $h$ is a cocycle and that the difference between the two sides of (\ref{gammakillsShapreimageeq1}) is
\[
\sum_v \inv_v(h),
\]
which vanishes because the sum of the invariants of a global Brauer class is $0$.
\end{proof}

We are now ready for the promised analogue to Lemma \ref{Sha^1(G)=0} for $\Sha^2$.

\begin{lemma}
\label{Sha^2(G)=0}
Let $G$ be a finite commutative group scheme over a global function field $k$. There exists a finite separable extension $k'/k$ such that $\Sha^2(k'', G_{k''}) = 0$ for every finite extension $k''/k'$.
\end{lemma}

\begin{proof}
First suppose that we have an exact sequence
\[
1 \longrightarrow G' \longrightarrow G \longrightarrow G'' \longrightarrow 1
\]
of finite commutative $k$-group schemes, and that the proposition holds for $G'$ and $G''$. Then we claim that it also holds for $G$. Indeed, we may by hypothesis extend scalars and thereby assume that $\Sha^2(G') = \Sha^2(G'') = 0$ (and that this remains true after any further 
finite extension of scalars on $k$). By Lemma \ref{Sha^1(G)=0}, we may further assume that $\Sha^1(\widehat{G}) = 0$ and that this remains true after any further finite extension of scalars on $k$. 

We will show under these hypotheses that $\Sha^2(G) = 0$ (so we can reduce to treating $G'$ and $G''$ in place of $G$). 
Fix an $\alpha \in \Sha^2(G)$. Since $\Sha^2(G'') = 0$, the element $\alpha$ lifts to ${\rm{H}}^2(k, G')$. Since $\Sha^1(\widehat{G}) = 0$, by Lemma \ref{Sha^2(G)=0devissage}, the class $\alpha$ lifts to $\Sha^2(G') = 0$, so $\alpha = 0$. 

Now we run a d\'evissage based on the preceding reduction: using the connected-{\'e}tale sequence for $\widehat{G}$, it suffices to treat separately the cases when $\widehat{G}$ is either infinitesimal or {\'e}tale. In the former case, we even have ${\rm{H}}^2(k, G) = 0$ by Lemma \ref{H^2=0forinfdual}. In the latter, we may replace $k$ by a finite separable extension and thereby assume that $G \simeq \mu_n$ (since after a finite separable extension it becomes a product of such groups), and the lemma then follows because $\Sha^2(\mu_n) = \ker({\rm{H}}^2(k, \Gm)[n] \rightarrow \prod_v \Br(k_v)[n]) = 0$ by class field theory.
\end{proof}

In order to prove the injectivity of the map $\Sha^2(G) \rightarrow \Sha^1(\widehat{G}^*)$, we will show that, given a surjection $G' \twoheadrightarrow G$, and $\alpha \in \ker(\Sha^2(G) \rightarrow \Sha^1(\widehat{G}^*))$, the element $\alpha$ lifts to an element in $\Sha^2(G')$ (Lemma \ref{liftingtoSha^2(G')}). What will then allow us to finish is the following lemma.

\begin{lemma}
\label{surjectiontrivialSha^2}
Let $G$ be an almost-torus over a global function field $k$. There is a surjection $G' \twoheadrightarrow G$ with $G'$ an almost-torus over $k$ such that $\Sha^2(G') = 0$. If $G$ is finite, then we may choose $G'$ to also be finite.
\end{lemma}

\begin{proof}
We first treat the case when $G$ is finite, and in this case we may prove the lemma for $\widehat{G}$ instead of $G$ thanks to Cartier duality. By Lemma \ref{Sha^2(G)=0}, there is a finite separable extension $k'/k$ such that $\Sha^2(k', \widehat{G_{k'}}) = 0$. Consider the canonical inclusion $G \hookrightarrow \R_{k'/k}(G_{k'})$. Since $k'/k$ is separable, $\R_{k'/k}(G_{k'})$ is a finite $k$-group scheme. By dualizing this inclusion, therefore, we obtain the desired surjection (due to Proposition \ref{hatisexact}), since $\Sha^2(k, \widehat{\R_{k'/k}(G_{k'})}) \simeq \Sha^2(k, \R_{k'/k}(\widehat{G_{k'}})) = \Sha^2(k', \widehat{G_{k'}}) = 0$, where the first equality is due to Proposition \ref{charactersseparableweilrestriction}, and the second is due to Lemma \ref{sepblepushforward}.

Now we treat the general case. So let $G$ be an almost-torus over $k$. By Lemma \ref{almosttorus}(iv), we may assume that $G = \R_{k'/k}(T') \times B$ for some finite commutative $k$-group scheme $B$, finite separable extension $k'/k$, and split $k'$-torus $T'$. We may treat the two groups $B$ and $\R_{k'/k}(T')$ separately. The group scheme $B$ is finite, so the lemma for it has already been proven. On the other hand, the $k$-group $\R_{k'/k}(T')$ already has trivial $\Sha^2$. Indeed, we have $\Sha^2(k, \R_{k'/k}(T')) \simeq \Sha^2(k', T')$ by Lemma \ref{sepblepushforward}. Finally, we have $\Sha^2(k', \Gm) = 0$ by class field theory, hence $\Sha^2(k', T') = 0$.
\end{proof}

\begin{lemma}
\label{liftingtoSha^2(G')}
Suppose that we have a surjection $G' \twoheadrightarrow G$ between affine commutative group schemes of finite type over a global function field $k$, and 
that an element $\alpha \in \Sha^2(k, G)$ annihilates $\Sha^1(k, \widehat{G})$ under $\langle \cdot, \cdot \rangle_{\Sha}$. Assume also that the sequence
\[
{\rm{H}}^1(\A, G) \longrightarrow {\rm{H}}^1(k, \widehat{G})^* \longrightarrow \Sha^1(\widehat{G})^*
\]
is exact (where the first map is induced by adding the local duality pairings and the second is the dual of the natural inclusion). Then $\alpha$ lifts to an element of $\Sha^2(k, G')$.
\end{lemma}

\begin{proof}
The proof is very similar to that of Lemma \ref{Sha^2(G)=0devissage}. Let $H : = \ker(G' \rightarrow G)$. Then we have an exact sequence
\[
1 \longrightarrow H \xlongrightarrow{j} G' \xlongrightarrow{\pi} G \longrightarrow 1.
\]
We claim that the commutative diagram
\begin{equation}
\label{liftingtoSha^2(G')pfdiagram}
\begin{tikzcd}
& {\rm{H}}^2(k, H) \arrow{r} \arrow{d} & {\rm{H}}^2(k, G') \arrow{r} \arrow{d} & {\rm{H}}^2(k, G) \arrow{r} \arrow{d} & 0 \\
{\rm{H}}^1(\A, G) \arrow{r} \arrow{d} & {\rm{H}}^2(\A, H) \arrow{r} \arrow{d} & {\rm{H}}^2(\A, G') \arrow{r} & {\rm{H}}^2(\A, G) & \\
{\rm{H}}^1(k, \widehat{G})^* \arrow{r} \arrow{d} & \widehat{H}(k)^* &&& \\
\Sha^1(\widehat{G})^* &&&&
\end{tikzcd}
\end{equation}
has exact rows and columns. Indeed, the first row is exact because ${\rm{H}}^3(k, H) = 0$ by Proposition \ref{cohomologicalvanishing}, and the second row is clearly exact. The first column is exact by assumption, and the second column is exact by Proposition \ref{funsequence}. We deduce that $\alpha$ lifts to some $\beta \in {\rm{H}}^2(k, G')$ and that the adelic class $\beta_{\A} \in {\rm{H}}^2(\A, G')$ lifts to some $\gamma \in {\rm{H}}^2(\A, H)$. The dual sequence $1 \rightarrow \widehat{G} \rightarrow \widehat{G'} \rightarrow \widehat{H} \rightarrow 1$ is exact by Proposition \ref{hatisexact}. 
Let $\widehat{\delta}:  \widehat{H}(k) \rightarrow {\rm{H}}^1(k, \widehat{G})$ denote the connecting map.

By Lemma \ref{gammakillsdelta^-1(Sha^1(G^))} below, the homomorphism $\phi_{\gamma} \in \widehat{H}(k)^*$ induced by $\gamma$ under the adelic pairing kills $\widehat{\delta}^{-1}(\Sha^1(\widehat{G})) \subset \widehat{H}(k)$. It therefore extends to an element of $({\rm{H}}^1(k, \widehat{G})/\Sha^1(\widehat{G}))^*$. By
exactness of the first column of (\ref{liftingtoSha^2(G')pfdiagram}), this latter homomorphism is induced by an element $g \in {\rm{H}}^1(\A, G)$. Modifying $\gamma$ by $g$ (which has no effect on its image $\beta_{\A} \in {\rm{H}}^2(\A, G')$), we may assume that $\phi_{\gamma} = 0$. Therefore, due to the exactness of the second column of (\ref{liftingtoSha^2(G')pfdiagram}), $\gamma$ lifts to some element $h \in {\rm{H}}^2(k, H)$. Modifying $\beta$ by $h$ (which has no effect on its image $\alpha \in {\rm{H}}^2(k, G)$), we may assume that $\beta \in \Sha^2(G')$, as desired.
\end{proof}

\begin{lemma}
\label{gammakillsdelta^-1(Sha^1(G^))}
With notation as above, $\gamma$ annihilates $\widehat{\delta}^{-1}(\Sha^1(\widehat{G})) \subset \widehat{H}(k)$ under the adelic pairing.
\end{lemma}

\begin{proof}
Choose $z \in \widehat{\delta}^{-1}(\Sha^1(\widehat{G}))$. We will show that 
\begin{equation}
\label{gammakillsdelta^-1(Sha^1(G^))eq1}
\langle \gamma, z \rangle = -\langle \alpha, \widehat{\delta}(z) \rangle_{\Sha}.
\end{equation}
This will prove the lemma because $\alpha$ annihilates $\Sha^1(\widehat{G})$ by hypothesis.

Since $\check{\rm{H}}^2(k, G') = {\rm{H}}^2(k,G')$ by Proposition \ref{cech=derivedsmoothinf}, we can represent
$\beta$ by an element of $\check{Z}^2(k,G')$ (and we will abuse notation and also denote this cocycle by $\beta$) and then take $\pi(\beta) \in \check{Z}^2(k,G)$ as our representative for $\alpha$.  In what follows, 
we use the symbols $\alpha, \beta$, etc. to denote fixed \v{C}ech cocycles representing the corresponding cohomology classes (and
in particular we have arranged that $\pi(\beta) = \alpha$ as \v{C}ech cochains). 
The left side of (\ref{gammakillsdelta^-1(Sha^1(G^))eq1}) is by definition
\[
\sum_v \inv_v(\gamma_v \cup z).
\]
The cohomology class $\widehat{\delta}(z) \in {\rm{H}}^1(k,\widehat{G}) = \check{\rm{H}}^1(k,\widehat{G})$ is
represented by a \v{C}ech 1-cocycle computed as follows. We may write $z = \widehat{j}(x)$ for some $x \in \check{C}^0(k, \widehat{G'})$. Then we have $dx = \widehat{\pi}(y)$
for some $y \in \check{Z}^1(k, \widehat{G})$, and $y$ represents the cohomology class $\widehat{\delta}(z)$ by 
Proposition \ref{cechderivedconnectingmap}.

Since $\check{\rm{H}}^1 = {\rm{H}}^1$, we have $j(\gamma) = \beta + df$ for some $f \in \check{C}^1(\A, G')$. Therefore, $\alpha = \pi(\beta) = 
-d\pi(f)$ since $\pi \circ j = 0$. We also have $\alpha \cup y = dh$ for some $h \in \check{C}^2(k, \Gm)$
since $\check{\rm{H}}^3(k, \Gm) = 0$ by Lemma \ref{checkh3=0}. The right side of (\ref{gammakillsdelta^-1(Sha^1(G^))eq1}) is, by the definition in \S \ref{sectiondefiningshapairings},
\[
-\sum_v \inv_v\left(-(\pi(f_v) \cup y) - h\right).
\]
In $\check{C}^2(\A, \Gm)$, we have
\[
\pi(f) \cup y = f \cup \widehat{\pi}(y) = f \cup dx = df \cup x - d(f \cup x).
\]
This is cohomologous to $df \cup x$. Equation (\ref{gammakillsdelta^-1(Sha^1(G^))eq1}) is therefore equivalent to
\begin{equation}
\label{gammakillsdelta^-1(Sha^1(G^))eq2}
\sum_v \inv_v((\gamma_v \cup z) - (df_v \cup x) - h) = 0.
\end{equation}
But 
\[
\gamma_v \cup z = \gamma_v \cup \widehat{j}(x) = j(\gamma_v) \cup x = (\beta + df_v) \cup x = (\beta \cup x) + (df_v \cup x),
\]
so (\ref{gammakillsdelta^-1(Sha^1(G^))eq2}) is equivalent to
\[
\sum_v \inv_v((\beta \cup x) - h) = 0
\]
(where our calculation shows that $(\beta \cup x) - h$ is a cocycle), and this is true because the sum of the invariants of a global Brauer class is $0$.
\end{proof}

\begin{lemma}
\label{Sha=Sha(almosttorus)}
Suppose that we have an exact sequence
\[
1 \longrightarrow H \longrightarrow G \longrightarrow U \longrightarrow 1
\]
of affine commutative group schemes of finite type over a global function field $k$, with $U$ split unipotent. Then the maps $\Sha^2(H) \rightarrow \Sha^2(G)$ and $\Sha^1(\widehat{G}) \rightarrow \Sha^1(\widehat{H})$ are isomorphisms.
\end{lemma}

\begin{proof}
By filtering $U$ by $\Ga$'s, it suffices to treat the case $U = \Ga$. First we handle $\Sha^2$. In the commutative diagram
\[
\begin{tikzcd}
{\rm{H}}^2(k, H) \arrow{d} \arrow{r}{\sim} & {\rm{H}}^2(k, G) \arrow{d} \\
{\rm{H}}^2(\A, H) \arrow{r}{\sim} & {\rm{H}}^2(\A, G)
\end{tikzcd}
\]
the two horizontal arrows are isomorphisms because ${\rm{H}}^i(k, \Ga) = {\rm{H}}^i(\A, \Ga) = 0$ for $i > 0$ by Proposition \ref{unipotentcohomology}(ii), and \cite[Th.\,2.13]{ces2}. The $\Sha^2$ assertion follows immediately.

Now we treat the $\Sha^1$ assertion. In the commutative diagram
\[
\begin{tikzcd}
0 \arrow{r} & {\rm{H}}^1(k, \widehat{G}) \arrow{d} \arrow{r} & {\rm{H}}^1(k, \widehat{H}) \arrow{r} \arrow{d} & {\rm{H}}^2(k, \widehat{\Ga}) \arrow[d, hookrightarrow] \\
0 \arrow{r} & {\rm{H}}^1(\A, \widehat{G}) \arrow{r} & {\rm{H}}^1(\A, \widehat{H}) \arrow{r} & {\rm{H}}^2(\A, \widehat{\Ga})
\end{tikzcd}
\]
the rows are exact by Proposition \ref{unipotentcohomology}(iii) and \cite[Th.\,2.13]{ces2}, and the last vertical arrow is an inclusion by Lemma \ref{H^2(k,Ga^)injectiveseparable}. A diagram chase now shows that the map $\Sha^1(\widehat{G}) \rightarrow \Sha^1(\widehat{H})$ is an isomorphism.
\end{proof}

We now come to the main result of this section.

\begin{proposition}
\label{Sha^2(G)-->Sha^1(G^)*finite}
Let $G$ be an affine commutative group scheme of finite type over a global function field $k$. Assume that the sequence
\begin{equation}
\label{Sha^2(G)-->Sha^1(G^)*finiteeqn3}
{\rm{H}}^1(\A, G) \longrightarrow {\rm{H}}^1(k, \widehat{G})^* \longrightarrow \Sha^1(\widehat{G})^*
\end{equation}
is exact. Then the map $\Sha^2(k, G) \rightarrow \Sha^1(k, \widehat{G})^*$ induced by $\langle \cdot, \cdot \rangle_{\Sha}$ is injective.
\end{proposition}

\begin{proof}
Suppose given an exact sequence
\[
1 \longrightarrow H \longrightarrow G \longrightarrow \Ga \longrightarrow 1
\]
of affine commutative $k$-group schemes of finite type. We will show that, if (\ref{Sha^2(G)-->Sha^1(G^)*finiteeqn3}) is exact for $G$, then it is also exact for $H$. Lemmas \ref{affinegroupstructurethm} and \ref{Sha=Sha(almosttorus)}, the functoriality of $\langle \cdot, \cdot \rangle_{\Sha}$, and induction will then reduce the proof of the proposition to the case of almost-tori. In the commutative diagram below with exact columns, the map in the top row is surjective by Proposition \ref{A/k=H^2(k,Ga^)*}, and the vertical arrow at the bottom left is surjective because ${\rm{H}}^1(\A, \Ga) = 0$. A diagram chase now shows that the exactness of the bottom sequence implies the exactness of the top sequence.
\[
\begin{tikzcd}
\A \arrow{d} \arrow[r, twoheadrightarrow] & {\rm{H}}^2(k, \widehat{\Ga})^* \arrow{d} & \\
{\rm{H}}^1(\A, H) \arrow{r} \arrow[d, twoheadrightarrow] & {\rm{H}}^1(k, \widehat{H})^* \arrow{r} \arrow{d} & \Sha^1(\widehat{H})^* \arrow{d} \\
{\rm{H}}^1(\A, G) \arrow{r} & {\rm{H}}^1(k, \widehat{G})^* \arrow{r} & \Sha^1(\widehat{G})^*
\end{tikzcd}
\]

Now we treat the case of almost-tori. Suppose that $\alpha \in \Sha^2(G)$ annihilates $\Sha^1(\widehat{G})$. By Lemma \ref{surjectiontrivialSha^2}, there is a surjection $G' \twoheadrightarrow G$ for some $k$-almost-torus $G'$ such that $\Sha^2(G') = 0$. By Lemma \ref{liftingtoSha^2(G')}, $\alpha$ lifts to an element of $\Sha^2(G') = 0$, hence $\alpha = 0$.
\end{proof}

Let us finally note that we may already complete the proof of Theorem \ref{shapairing} for finite commutative group schemes.

\begin{corollary}
\label{shapairingfinite}
Theorem $\ref{shapairing}$ holds for finite commutative group schemes. More precisely, if $G$ is a finite commutative group scheme over a global function field $k$, then the groups $\Sha^i(k, G)$ and $\Sha^i(k, \widehat{G})$ are finite for $i = 1,2$. Further, for $i = 1,2$, the functorial bilinear pairings
\[
\Sha^i(k, G) \times \Sha^{3-i}(k, \widehat{G}) \rightarrow \Q/\Z.
\]
given by $\langle \cdot, \cdot \rangle_{\Sha^i_G}$ ($i = 1, 2$) are perfect pairings.
\end{corollary}

\begin{proof}
Combining Propositions \ref{Sha^1--->Sha^2*injective}, \ref{Sha^2(G)-->Sha^1(G^)*finite}, and \ref{exactnessatH^1(G^)*finite}(i) with Cartier duality, it only remains to prove the finiteness assertion. By Cartier duality, it suffices to show that the groups $\Sha^i(G)$ are finite for $i = 1, 2$. The finiteness of $\Sha^1(G)$ is contained in \cite[Thm.\,1.3.3(i)]{conrad}, but in the finite commutative setting the proof is actually much simpler, so we give it here.

To prove that $\Sha^1(G)$ is finite, we first observe that
if $G^{\rm{sm}}$ is the maximal smooth closed $k$-subgroup of $G$ (or equivalently the maximal \'etale closed $k$-subgroup
of the finite $k$-group scheme $G$) then $\Sha^1(G^{\rm{sm}}) \rightarrow \Sha^1(G)$ is bijective due to the separability
of $k_v/k$ for all places $v$ of $k$; see \cite[Ex.\,C.4.3]{cgp} (with $S = \emptyset$ there) for the details.  Thus, it suffices
to treat \'etale $G$, so finite discrete Galois modules. If $k'/k$ is a finite Galois extension that
splits $G$, then, by finiteness of ${\rm{H}}^1(k'/k, G(k'))$, we see via inflation-restriction that finiteness of $\Sha^1(G)$ reduces
to that of $\Sha^1(k', G_{k'})$.  In other words, we may assume that $G$ is the constant $k$-group associated
to a finite abelian group $A$. Hence, ${\rm{H}}^1(F, G) = {\rm{Hom}}_{\rm{cts}}({\rm{Gal}}(F_s/F), A)$ for any field $F/k$, so
$\Sha^1(G) = 0$ by the Chebotarev Density Theorem. Finally, the finiteness of $\Sha^2(G)$ follows from that of $\Sha^1(\widehat{G})$ and Propositions \ref{Sha^2(G)-->Sha^1(G^)*finite} and \ref{exactnessatH^1(G^)*finite}(i).
\end{proof}

\section{Exactness at ${\rm{H}}^1(\A, G)$ and ${\rm{H}}^1(\A, \widehat{G})$}
\label{sectionexactnessatvariousH^1's}

In this section we will show that the Poitou--Tate sequences of Theorems \ref{poitoutatesequence} and \ref{poitoutatesequencedual} are exact at ${\rm{H}}^1(\A, G)$ and ${\rm{H}}^1(\A, \widehat{G})$, respectively.

\begin{lemma}
\label{subgroupSha^2=0}
Let $k$ be a global function field, $G$ a finite commutative $k$-group scheme. Then there is an inclusion $G \hookrightarrow G'$ with $G'$ a finite commutative $k$-group scheme such that $\Sha^2(k, G') = 0$.
\end{lemma}

\begin{proof}
By Lemma \ref{Sha^2(G)=0}, there is a finite separable extension $k'/k$ such that $\Sha^2(k', G_{k'}) = 0$. Consider the canonical inclusion $G \hookrightarrow \R_{k'/k}(G_{k'})$. The latter group scheme is finite because $k'/k$ is separable, and we claim that it has trivial $\Sha^2$. To see this, note that, by Lemma \ref{sepblepushforward}, we have canonically ${\rm{H}}^2(k, \R_{k'/k}(G_{k'})) \simeq {\rm{H}}^2(k', G)$, and similarly ${\rm{H}}^2(k_v, \R_{k'/k}(G_{k'})) \simeq \prod_{v'\mid v} {\rm{H}}^2(k_{v'}, G)$. Therefore, $\Sha^2(k, \R_{k'/k}(G_{k'})) \simeq \Sha^2(k', G_{k'}) = 0$, as claimed.
\end{proof}

\begin{lemma}
\label{exactnessatH^1(A,G)X'}
Suppose that we have a short exact sequence
\[
1 \longrightarrow \R_{k'/k}(T') \longrightarrow G \longrightarrow C \longrightarrow 1
\]
of affine commutative group schemes of finite type over a global function field $k$, where $C$ is a finite commutative $k$-group scheme, $k'/k$ is a finite separable extension, and $T'$ is a split $k'$-torus. Then the global duality sequence
\[
{\rm{H}}^1(k, G) \longrightarrow {\rm{H}}^1(\A, G) \longrightarrow {\rm{H}}^1(k, \widehat{G})^*
\]
of Theorem {\rm{\ref{poitoutatesequence}}} is exact.
\end{lemma}

\begin{proof}
Consider the following commutative diagram: 
\[
\begin{tikzcd}
&&& 0 \arrow{d} \\
& {\rm{H}}^1(k, G) \arrow{r} \arrow{d} & {\rm{H}}^1(k, C) \arrow{r} \arrow{d} & {\rm{H}}^2(k, \R_{k'/k}(T')) \arrow{d} \\
0 \arrow{r} & {\rm{H}}^1(\A, G) \arrow{r} \arrow{d} & {\rm{H}}^1(\A, C) \arrow{r} \arrow{d} & {\rm{H}}^2(\A, \R_{k'/k}(T')) \\
& {\rm{H}}^1(k, \widehat{G})^* \arrow{r} & {\rm{H}}^1(k, \widehat{C})^* &
\end{tikzcd}
\]
The second column is exact by Poitou--Tate for finite commutative group schemes (see Remark \ref{prior}). The third column is exact because $\Sha^2(k, \R_{k'/k}(T')) \simeq \Sha^2(k', T') = 0$, the first equality by Lemma \ref{sepblepushforward}, and the second by the fundamental exact sequence (of Brauer groups) in class field theory. The exactness of the first row is clear, and the second row is exact because ${\rm{H}}^1(\A, \R_{k'/k}(T')) = 0$ (by \cite[Th.\,2.13]{ces2} and Proposition \ref{H^1=0Weilrestrictionsplittori}). A diagram chase now shows that the first column, which we already know to be a complex, is exact.
\end{proof}

\begin{proposition}
\label{exactnessatH^1(A,G)}
Let $k$ be a global function field, $G$ an affine commutative $k$-group scheme of finite type. Then the global duality sequence
\[
{\rm{H}}^1(k, G) \longrightarrow {\rm{H}}^1(\A, G) \longrightarrow {\rm{H}}^1(k, \widehat{G})^*
\]
of Theorem {\rm{\ref{poitoutatesequence}}} is exact.
\end{proposition}

\begin{proof}
First suppose that $G$ is an almost-torus. By Lemma \ref{almosttorus}(iv), we may harmlessly modify $G$ and thereby assume that there is an exact sequence
\[
1 \longrightarrow B \longrightarrow \R_{k'/k}(T') \times C \longrightarrow G \longrightarrow 1
\]
with $B, C$ finite commutative $k$-group schemes, $k'/k$ a finite separable extension, and $T'$ a split $k'$-torus. We have $\Sha^2(k, \R_{k'/k}(T')) \simeq \Sha^2(k', T') = 0$. By Lemma \ref{surjectiontrivialSha^2}, there is a surjection $C' \twoheadrightarrow C$ for some finite commutative $k$-group scheme $C'$ such that $\Sha^2(k, C') = 0$. The composition $\R_{k'/k}(T') \times C' \rightarrow \R_{k'/k}(T') \times C \rightarrow G$ is also an isogeny, so we may replace $C$ with $C'$ and thereby arrange that $\Sha^2(k, \R_{k'/k}(T') \times C) = 0$. 
For notational convenience let $X : = \R_{k'/k}(T') \times C$, so we have an exact sequence
\[
1 \longrightarrow B \longrightarrow X \longrightarrow G \longrightarrow 1
\]
with $\Sha^2(k, X) = 0$. 

Let $\alpha \in {\rm{H}}^1(\A, G)$ be such that $\alpha$ maps to $0$ in ${\rm{H}}^1(k, \widehat{G})^*$. We want $\alpha$
to come from ${\rm{H}}^1(k, G)$.  Consider the following commutative diagram: 
\[
\begin{tikzcd}
&&& 0 \arrow{d} \\
& {\rm{H}}^1(k, G) \arrow{r} \arrow{d} & {\rm{H}}^2(k, B) \arrow{r} \arrow{d} & {\rm{H}}^2(k, X) \arrow{d} \\
{\rm{H}}^1(\A, X) \arrow{r} & {\rm{H}}^1(\A, G) \arrow{r} \arrow{d} & {\rm{H}}^2(\A, B) \arrow{r} \arrow{d} & {\rm{H}}^2(\A, X) \\
& {\rm{H}}^1(k, \widehat{G})^* \arrow{r} & \widehat{B}(k)^* &
\end{tikzcd}
\]
The third column is exact by Proposition \ref{funsequence}. The fourth column is exact because $\Sha^2(X) = 0$. All of the rows are clearly exact. A diagram chase now shows that upon changing $\alpha$ by the image of an element of ${\rm{H}}^1(k, G)$, it lifts to an element in ${\rm{H}}^1(\A, X)$. We may therefore assume that $\alpha$ lifts to ${\rm{H}}^1(\A, X)$.

By Lemma \ref{subgroupSha^2=0}, there is an inclusion $B \hookrightarrow B'$ for some finite commutative $k$-group scheme $B'$ such that 
\begin{equation}
\label{Sha^2(B')=0eqn13}
\Sha^2(k, B') = 0.
\end{equation}
Pushing out the sequence
\[
1 \longrightarrow B \longrightarrow X \longrightarrow G \longrightarrow 1
\]
by the inclusion $B \hookrightarrow B'$, we obtain the following commutative diagram of exact sequences: 
\begin{equation}
\label{pushoutproofofH^1(A,G)}
\begin{tikzcd}
1 \arrow{r} & B \arrow[d, hookrightarrow] \arrow{r} & X \arrow{r} \arrow[d, hookrightarrow] & G \arrow[d, equals] \arrow{r} & 1 \\
1 \arrow{r} & B' \arrow{r} & X' \arrow{r} & G \arrow{r} & 1
\end{tikzcd}
\end{equation}
where $X'$ is defined by the above pushout diagram. Note that the map $\R_{k'/k}(T') \rightarrow X'$ is an inclusion with finite cokernel. That is, we have a short exact sequence
\begin{equation}
\label{exactseqX'eqn3}
1 \longrightarrow \R_{k'/k}(T') \longrightarrow X' \longrightarrow C' \longrightarrow 1,
\end{equation}
where $C'$ is $k$-finite.

Using the commutative diagram (\ref{pushoutproofofH^1(A,G)}) and the fact that the element $\alpha \in {\rm{H}}^1(\A, G)$ lifts to ${\rm{H}}^1(\A, X)$, we see that $\alpha$ lifts to ${\rm{H}}^1(\A, X')$. Consider the following commutative diagram: 
\[
\begin{tikzcd}
& {\rm{H}}^1(k, X') \arrow{r} \arrow{d} & {\rm{H}}^1(k, G) \arrow{d} \\
{\rm{H}}^1(\A, B') \arrow{r} \arrow{d} & {\rm{H}}^1(\A, X') \arrow{r} \arrow{d} & {\rm{H}}^1(\A, G) \arrow{d} \\
{\rm{H}}^1(k, \widehat{B'})^* \arrow{r} \arrow{d} & {\rm{H}}^1(k, \widehat{X'})^* \arrow{r} & {\rm{H}}^1(k, \widehat{G})^* \\
0  &&
\end{tikzcd}
\]
The first column is exact by Proposition \ref{exactnessatH^1(G^)*finite}, Corollary \ref{shapairingfinite}, and (\ref{Sha^2(B')=0eqn13}). The second column is exact due to the exact sequence (\ref{exactseqX'eqn3}) and Lemma \ref{exactnessatH^1(A,G)X'}. The rows are clearly exact. A diagram chase now shows that 
$\alpha$ (which, recall, lifts to ${\rm{H}}^1(\A,X')$) lifts to ${\rm{H}}^1(k, G)$. This completes the proof of the proposition for almost-tori.

Now suppose that $G$ is an arbitrary affine commutative $k$-group scheme of finite type. We will prove the proposition by induction on the dimension of the unipotent radical of $(G_{\overline{k}})^0_{\red}$, the $0$-dimensional case corresponding to the already-treated case of almost-tori. By Lemma \ref{affinegroupstructurethm}, therefore, we may assume that there is an exact sequence
\[
1 \longrightarrow H \longrightarrow G \longrightarrow \Ga \longrightarrow 1
\]
such that the proposition holds for $H$. Consider the following commutative diagram: 
\[
\begin{tikzcd}
& {\rm{H}}^1(k, H) \arrow{r} \arrow{d} & {\rm{H}}^1(k, G) \arrow{d} & \\
\A \arrow{r} \arrow{d} & {\rm{H}}^1(\A, H) \arrow{r} \arrow{d} & {\rm{H}}^1(\A, G) \arrow{r} \arrow{d} & 0 \\
{\rm{H}}^2(k, \widehat{\Ga})^* \arrow{r} \arrow{d} & {\rm{H}}^1(k, \widehat{H})^* \arrow{r} & {\rm{H}}^1(k, \widehat{G})^* & \\
0 &&&
\end{tikzcd}
\]
The first column is exact by Proposition \ref{A/k=H^2(k,Ga^)*}, and the second column is exact by hypothesis. The second row is exact because ${\rm{H}}^1(\A, \Ga) = 0$, and the third row is exact because $\Q/\Z$ is an injective abelian group. A diagram chase now shows that the last column, which we already know to be a complex, is exact. This completes the proof of Proposition \ref{exactnessatH^1(A,G)}.
\end{proof}

Next we prove the analogue of Proposition \ref{exactnessatH^1(A,G)} upon switching the roles of $G$ and $\widehat{G}$. We prove this result now (rather than together with the rest of Theorem \ref{poitoutatesequencedual} in \S \ref{sectiondual9termsequence}) because we will use it in the proof of Theorem \ref{shapairing}. First, however, we require some preparatory results.

\begin{lemma}
\label{H^1(A,G)finiteexponent}
For an affine commutative group scheme of finite type over a global function field $k$, ${\rm{H}}^1(\A, G)$ is of finite exponent.
\end{lemma}

\begin{proof}
Lemmas \ref{affinegroupstructurethm} and \ref{almosttorus}(iv) reduce us to the case in which $G = \Ga$ or $\R_{k'/k}(\Gm)$ for some finite separable extension $k'/k$, since finite commutative group schemes are of finite exponent. In both of these cases, the cohomology group vanishes, thanks to Lemmas \ref{H^1(prodO_v,G_a^)=0} and \ref{H^1(prodO_v,G^)=prodH^1(O_v,G^)weilrestrictiontori}.
\end{proof}

\begin{proposition}
\label{Che^1(G)compact}
For an affine commutative group scheme $G$ of finite type over a global field $k$, the group
\[
\Che^1(G) := \coker({\rm{H}}^1(k, G) \longrightarrow {\rm{H}}^1(\A, G))
\]
is compact.
\end{proposition}

\begin{proof}
We claim that there is an inclusion $G \hookrightarrow G'$ into an affine commutative $k$-group scheme $G'$ of finite type such that ${\rm{H}}^1(\A, G') = 0$. Indeed, there is an inclusion $G \hookrightarrow H$ into a smooth connected affine $k$-group scheme $H$. For some finite extension field $L/k$, $H_L$ is the product of a split torus and a split unipotent group. Then the inclusion $G \hookrightarrow H \hookrightarrow \R_{L/k}(H_L)$ does the job.

Now let $G'' := G'/G$, an affine commutative $k$-group scheme of finite type \cite[Ch.\,III, \S3, no.\,5, Th.\,5.6]{demazuregabriel}, and let $\overline{G''} := G''(\A)/G''(k)$. Since ${\rm{H}}^1(\A, G') = 0$, the connecting map $\overline{G''}(\A) \rightarrow \Che^1(G)$ is surjective. It is also continuous by Proposition \ref{adelictopcohombasics}(vi). Some integer $N > 0$ kills ${\rm{H}}^1(\A, G)$ by Lemma \ref{H^1(A,G)finiteexponent}, so the connecting map descends to a surjective continuous map $\overline{G''}/N\overline{G''} \rightarrow \Che^1(G)$. It therefore suffices to show that $\overline{G''}/N\overline{G''}$ is compact. By Proposition \ref{G(A)_1/G(k)compact} and \cite[Ch.\,I, Prop.\,5.6]{oesterle}, $\overline{G''}$ is an extension of $H$ by a compact group, where $H$ is either finitely generated or a power $\mathbf{R}^n$ of the reals, so the claim follows.
\end{proof}

\begin{proposition}
\label{exactnessatH^1(A,G^)}
Let $k$ be a global function field, $G$ an affine commutative $k$-group scheme of finite type. Then the global duality sequence
\[
{\rm{H}}^1(k, \widehat{G}) \longrightarrow {\rm{H}}^1(\A, \widehat{G}) \longrightarrow {\rm{H}}^1(k, G)^*
\]
of Theorem $\ref{poitoutatesequencedual}$ is exact.
\end{proposition}

\begin{proof}
The global duality sequence
\begin{equation}
\label{exactnessatH^1(A,G^)pfseq1}
{\rm{H}}^1(k, G) \longrightarrow {\rm{H}}^1(\A, G) \longrightarrow {\rm{H}}^1(k, \widehat{G})^*
\end{equation}
of Theorem \ref{poitoutatesequence} consists of continuous maps (the second map being the composition ${\rm{H}}^1(\A, G) \rightarrow {\rm{H}}^1(\A, \widehat{G})^D \rightarrow {\rm{H}}^1(k, \widehat{G})^D$, in which the first map is continuous by Proposition \ref{H^1(A,G)=H^1(A,G^)^D}, and the second is the dual of a trivially continuous map), and it is exact by Proposition \ref{exactnessatH^1(A,G)}. Further, all of the groups appearing are locally compact Hausdorff. For the global cohomology groups this is trivial, while for the adelic group it is contained in Proposition \ref{adelictopcohombasics}(iv). Note that ${\rm{H}}^1(k, \widehat{G})^*$ is the Pontryagin dual of ${\rm{H}}^1(k, \widehat{G})$ by Lemma \ref{H^1finiteexponent}.

We claim that the Pontryagin dual complex associated to (\ref{exactnessatH^1(A,G^)pfseq1}) is also exact. Assuming this, we would be done by Pontryagin double duality and Proposition \ref{H^1(A,G)=H^1(A,G^)^D}. 
To prove that the Pontryagin dual sequence is exact, it suffices to show that the continuous inclusion ${\rm{H}}^1(\A, G)/{\rm{H}}^1(k, G) \hookrightarrow {\rm{H}}^1(k, \widehat{G})^*$ is a homeomorphism onto a closed subgroup. This follows from the fact that the source is compact (Proposition \ref{Che^1(G)compact}) and the target Hausdorff.
\end{proof}

\begin{corollary}
\label{exactnessatH^1(G^)*finitecor}
For an affine commutative group scheme $G$ of finite type over a global function field $k$, the sequences
\[
{\rm{H}}^1(\A, G) \longrightarrow {\rm{H}}^1(k, \widehat{G})^* \longrightarrow \Sha^1(\widehat{G})^* \longrightarrow 0,
\]
\begin{equation}
\label{exactnessatH^1(G^)*finiteeqn10}
{\rm{H}}^1(\A, \widehat{G}) \longrightarrow {\rm{H}}^1(k, G)^* \longrightarrow \Sha^1(G)^* \longrightarrow 0
\end{equation}
are exact, where the maps ${\rm{H}}^1(\A, G) \rightarrow {\rm{H}}^1(k, \widehat{G})^*$ and ${\rm{H}}^1(\A, \widehat{G}) \rightarrow {\rm{H}}^1(k, G)^*$ are induced by the global duality pairings.
\end{corollary}

\begin{proof}
This follows from Propositions \ref{exactnessatH^1(G^)*finite}(ii), \ref{exactnessatH^1(A,G)}, and \ref{exactnessatH^1(A,G^)}.
\end{proof}

\begin{proposition}
\label{Sha^2(G)-->Sha^1(G^)*finitecor}
Let $G$ be an affine commutative group scheme of finite type over a global function field $k$. Then the map $\Sha^2(k, G) \rightarrow \Sha^1(k, \widehat{G})^*$ induced by $\langle \cdot, \cdot \rangle_{\Sha}$ is injective.
\end{proposition}

\begin{proof}
Combine Proposition \ref{Sha^2(G)-->Sha^1(G^)*finite} and Corollary \ref{exactnessatH^1(G^)*finitecor}.
\end{proof}

\section{Injectivity of $\Sha^2(\widehat{G}) \rightarrow \Sha^1(G)^*$}

In this section we prove that if $G$ is an affine commutative group scheme of finite type over a global field $k$, then the map $\Sha^2(\widehat{G}) \rightarrow \Sha^1(G)^*$ induced by the pairing $\langle \cdot, \cdot \rangle_{\Sha}$ is injective (Proposition \ref{Sha^2--->Sha^1injective}).

\begin{lemma}
\label{Sha^2(R^)=0}
If $k'/k$ is a finite separable extension of global fields, and $T'$ is a split $k'$-torus, then $\Sha^2(k, \widehat{\R_{k'/k}(T')}) = 0$.
\end{lemma}

\begin{proof}
Without loss of generality, $T' = \Gm$.
By Proposition \ref{charactersseparableweilrestriction}, it suffices to prove the vanishing of $\Sha^2(k, \R_{k'/k}(\Z))$. We may take 
the cohomology to be {\'e}tale, since $\R_{k'/k}(\Z)$ is represented by a smooth $k$-scheme. By Lemma \ref{sepblepushforward}, we are reduced (renaming $k'$ as $k$) to showing that $\Sha^2(k, \Z) = 0$. That is, we claim that the map
 $$f: {\rm{H}}^2(k, \Z) \rightarrow \prod_v {\rm{H}}^2(k_v, \Z)$$
 is injective. Using the exact sequence of Galois modules
\[
0 \longrightarrow \Z \longrightarrow \Q \longrightarrow \Q/\Z \longrightarrow 0
\]
over any field, and the vanishing (for any field) of
higher Galois cohomology with coefficients in the uniquely divisible group $\Q$, we have ${\rm{H}}^2(F, \Z) \simeq {\rm{H}}^1(F, \Q/\Z)$
functorially in any field $F$.  Thus, $f$ is identified with the natural map
${\rm{H}}^1(k, \Q/\Z) \rightarrow \prod_v {\rm{H}}^1(k_v, \Q/\Z)$.
But ${\rm{H}}^1(F, \Q/\Z) \simeq \Hom_{{\rm{cts}}}({\rm{Gal}}(F_s/F), \Q/\Z)$ naturally in any field $F$
(equipped with a specified separable closure), 
so the desired injectivity is reduced to the fact that the collection of decomposition groups
$D(\overline{v}|v) \subset {\rm{Gal}}(k_s/k)$ for non-archimedean places $v$ on $k$ and their lifts
$\overline{v}$ on $k_s$ generates a dense subgroup, which is the Chebotarev Density Theorem.
\end{proof}

Now we make a definition which will be useful below (and rests on the fact 
that a commutative affine group scheme $H$ of finite type over a field $F$ contains a unique $F$-torus $T$ which is
maximal in the sense that it contains all others).

\begin{definition}
\label{quasitrivialalmosttorus}
We say that an almost-torus $G$ over a field $k$ is {\em quasi-trivial} if its maximal $k$-torus is of the form $\R_{k'/k}(T')$ for some finite separable extension $k'/k$ and split $k'$-torus $T'$. This is equivalent to $G_{\red}^0$ being such a $k$-torus, as well as (due to Lemma \ref{almosttorus}(ii)) to the existence of an exact sequence of $k$-group schemes
\[
1 \longrightarrow \R_{k'/k}(T') \longrightarrow G \longrightarrow B \longrightarrow 1
\]
with $B$ a finite commutative $k$-group scheme, $k'/k$ finite separable, and $T'$ a split $k'$-torus.
\end{definition}

\begin{lemma}
\label{shapairingquasitrivial}
If $G$ is a quasi-trivial almost-torus over a global function field $k$, then the map $\Sha^2(k, \widehat{G}) \rightarrow \Sha^1(k, G)^*$ induced by the pairing $\langle \cdot, \cdot \rangle_{\Sha}$ is injective.
\end{lemma}

\begin{proof}
We have an exact sequence 
\[
1 \longrightarrow \R_{k'/k}(T') \longrightarrow G \longrightarrow B \longrightarrow 1
\]
with $k'/k$ a finite separable extension, $T'$ a split $k'$-torus, and $B$ a finite 
commutative $k$-group scheme. We will show that the maps $\Sha^2(\widehat{B}) \rightarrow \Sha^2(\widehat{G})$ and $\Sha^1(G) \rightarrow \Sha^1(B)$ are isomorphisms, so the lemma will follow from the already-known case of finite group 
schemes (Corollary \ref{shapairingfinite}) and the functoriality of $\langle \cdot, \cdot \rangle_{\Sha}$.

To see that $\Sha^2(\widehat{B}) \rightarrow \Sha^2(\widehat{G})$ is an isomorphism, consider the following commutative diagram: 
\[
\begin{tikzcd}
0 \arrow{r} & {\rm{H}}^2(k, \widehat{B}) \arrow{d} \arrow{r} & {\rm{H}}^2(k, \widehat{G}) \arrow{r} \arrow{d} & {\rm{H}}^2(k, \widehat{\R_{k'/k}(T')}) \arrow[d, hookrightarrow] \\
0 \arrow{r} & \prod_v {\rm{H}}^2(k_v, \widehat{B}) \arrow{r} & \prod_v {\rm{H}}^2(k_v, \widehat{G}) \arrow{r} & \prod_v {\rm{H}}^2(k_v, \widehat{\R_{k'/k}(T')})
\end{tikzcd}
\]
The rows are exact by Proposition \ref{hatisexact} and Lemma \ref{H^1=0Weilrestrictionsplittori}. The last vertical arrow is an inclusion by Lemma \ref{Sha^2(R^)=0}. A diagram chase shows that the map $\Sha^2(\widehat{B}) \rightarrow \Sha^2(\widehat{G})$ is an isomorphism.

To see that $\Sha^1(G) \rightarrow \Sha^1(B)$ is an isomorphism, consider the following commutative diagram: 
\[
\begin{tikzcd}
0 \arrow{r} & {\rm{H}}^1(k, G) \arrow{r} \arrow{d} & {\rm{H}}^1(k, B) \arrow{r} \arrow{d} & {\rm{H}}^2(k, \R_{k'/k}(T')) \arrow[d, hookrightarrow] \\
0 \arrow{r} & \prod_v {\rm{H}}^1(k_v, G) \arrow{r} & \prod_v {\rm{H}}^1(k_v, B) \arrow{r} & \prod_v {\rm{H}}^2(k_v, \R_{k'/k}(T'))
\end{tikzcd}
\]
The rows are exact by Lemma \ref{H^1=0Weilrestrictionsplittori}. The last vertical arrow is an inclusion due to
the isomorphism $\Sha^2(k, \R_{k'/k}(T')) \simeq \Sha^2(k', T')$, whose target vanishes by class field theory. A diagram chase now shows that the map $\Sha^1(G) \rightarrow \Sha^1(B)$ is an isomorphism.
\end{proof}

\begin{lemma}
\label{Sha^2(G^)-->Sha^1(G)*injectsalmosttori}
If $G$ is an almost-torus over a global function field $k$, then the map $\Sha^2(\widehat{G}) \rightarrow \Sha^1(G)^*$ induced by the pairing $\langle \cdot, \cdot \rangle_{\Sha}$ is injective.
\end{lemma}

\begin{proof}
By Lemma \ref{almosttorus}(iv), modifying $G$ in a harmless manner allows us to assume that there is an exact sequence
\begin{equation}
\label{Sha^2injectiveproofsequence1}
1 \longrightarrow B \xlongrightarrow{j} X \xlongrightarrow{\pi} G \longrightarrow 1
\end{equation}
with $B$ a finite commutative $k$-group scheme and $X$ a quasi-trivial almost-torus. By Lemma \ref{subgroupSha^2=0}, there is an inclusion $B \hookrightarrow B'$ for some finite 
commutative $k$-group scheme $B'$ such that $\Sha^2(B') = 0$. Pushing out the sequence (\ref{Sha^2injectiveproofsequence1}) by the inclusion $B \hookrightarrow B'$ and renaming, we may therefore assume that $\Sha^2(B) = 0$ (and $X$ is still a quasi-trivial almost-torus).

Suppose that $\alpha \in \Sha^2(\widehat{G})$ annihilates $\Sha^1(G)$. We want to show that $\alpha = 0$. By functoriality, $\widehat{\pi}(\alpha) \in \Sha^2(\widehat{X})$ annihilates $\Sha^1(X)$, so $\widehat{\pi}(\alpha) = 0$
by Lemma \ref{shapairingquasitrivial}. Hence (using Proposition \ref{hatisexact}) there exists a $\beta \in {\rm{H}}^1(k, \widehat{B})$ such that $\delta(\beta) = \alpha$, where $\delta:  {\rm{H}}^1(k, \widehat{B}) \rightarrow {\rm{H}}^2(k, \widehat{G})$ is the connecting map. Note that modifying $\beta$ by the image of an element of ${\rm{H}}^1(k, \widehat{X})$ has no effect on its image $\alpha \in {\rm{H}}^2(k, \widehat{G})$.

Let $\beta_{\A}$ denote the image of $\beta$ in ${\rm{H}}^1(\A, \widehat{B})$. Since $\delta(\beta) \in \Sha^2(\widehat{G})$, there exists an $x \in {\rm{H}}^1(\A, \widehat{X})$ such that $\widehat{j}(x) = \beta_{\A}$. 
Note that we are free to modify $x$ by an element of ${\rm{H}}^1(\A, \widehat{G})$. Recall that we have a natural pairing between ${\rm{H}}^1(\A, \widehat{X})$ and ${\rm{H}}^1(k, X)$ defined as usual by cupping everywhere locally and summing the invariants. Via this pairing, $x$ defines an element $\phi_x \in {\rm{H}}^1(k, X)^*$. Consider the subgroup 
$$\pi^{-1}(\Sha^1(G)) : = \{ \gamma \in {\rm{H}}^1(k, X) \mid \pi(\gamma) \in \Sha^1(G)\} \subset {\rm{H}}^1(k, X).$$

We claim that by modifying $x$ by an element of ${\rm{H}}^1(\A, \widehat{G})$ (as we are free to do), we can arrange that $x$ annihilates ${\rm{H}}^1(k, X)$. We have an inclusion 
\[
\frac{{\rm{H}}^1(k, X)}{\pi^{-1}(\Sha^1(G))} \hookrightarrow \frac{{\rm{H}}^1(k, G)}{\Sha^1(G)}, 
\]
so since $\phi_x$ annihilates $\pi^{-1}(\Sha^1(G))$ by Lemma \ref{xannihilatesShapreimage} below, we see that $\phi_x$ lifts to an element of $({\rm{H}}^1(k, G)/\Sha^1(G))^*$. By Corollary \ref{exactnessatH^1(G^)*finitecor}, it follows that $\phi_x$ agrees with the homomorphism induced by some element of ${\rm{H}}^1(\A, \widehat{G})$. Modifying $x$ by this element, which has no effect on its image $\beta_{\A} \in {\rm{H}}^1(\A, \widehat{B})$, we may therefore assume that 
$\phi_x = 0$, as claimed.

It then follows from Proposition \ref{exactnessatH^1(A,G^)} that $x$ lifts to an element $\xi \in {\rm{H}}^1(k, \widehat{X})$. Replacing $\beta$ 
with $\beta - \widehat{j}(\xi)$ (as we are free to do, since, again, this has no effect on its image $\alpha \in {\rm{H}}^2(k, \widehat{G})$), we may arrange that $\widehat{j}(x) = 0$.  But
$\beta_{\A} = \widehat{j}(x)$, so $\beta \in \Sha^1(\widehat{B})$. Since $\Sha^2(B) = 0$ by design, 
the group $\Sha^1(\widehat{B})$ vanishes by Corollary \ref{shapairingfinite}. Hence $\alpha = \delta(\beta) = 0$, and Lemma \ref{Sha^2(G^)-->Sha^1(G)*injectsalmosttori} is proved.
\end{proof}

\begin{lemma}
\label{xannihilatesShapreimage}
The element $x$ annihilates $\pi^{-1}(\Sha^1(G))$ under the adelic evaluation pairing discussed above.
\end{lemma}

\begin{proof}
Choose an arbitrary $y \in \pi^{-1}(\Sha^1(G))$. The key is to show that 
\begin{equation}
\label{xannihilatesShapreimageproof1}
\langle x, y \rangle = - \langle \alpha, \pi(y) \rangle_{\Sha},
\end{equation}
where the left side is the adelic evaluation pairing. Since $\alpha$ annihilates $\Sha^1(G)$ by assumption, the lemma would then follow immediately.  

The validity of 
(\ref{xannihilatesShapreimageproof1}) is independent of the choice of an $x$ with $\widehat{j}(x) = \beta_{\A}$ because ${\rm{H}}^1(\A, \widehat{G})$ annihilates $\Sha^1(G)$ under the adelic evaluation pairing. Thus, later in the calculation we may make whatever
choice of $x$ is convenient. We may also make whichever choice of $\beta$ satisfying $\delta(\beta) = \alpha$ that is convenient, since modifying $\beta$ by an element of ${\rm{H}}^1(k, \widehat{X})$ changes $x$ by the same element (more precisely, allows us to so modify $x$, since the choice of $x$ is not unique), and this has no effect on the adelic evaluation pairing with $y$, since the sum of the invariants of a global Brauer class is $0$.

Denote by $\check{\beta}$ a cocycle in $\check{Z}^1(k, \widehat{B})$ representing 
a choice for the cohomology class $\beta$ mapping to $\alpha$ under the connecting map. 
We want to describe the connecting map in
terms of \v{C}ech cohomology:  choose an $m \in \check{C}^1(k, \widehat{X})$ lifting $\check{\beta}$ through $\widehat{j}$, if such an $m$ exists, 
and then $dm$ lifts to an element of $\check{Z}^2(k, \widehat{G})$ whose cohomology class is $\delta(\beta)$. 
Unfortunately, there is no reason that a general choice of representative
$\check{\beta}$ should lift to $\check{C}^1(k, \widehat{X})$. This is the usual deficiency
of \v{C}ech theory for Grothendieck topologies (and even usual topologies with general abelian sheaves):  the sequence of \v{C}ech complexes associated to a short exact sequence of abelian sheaves need not be short exact. Nevertheless, we can
make the idea work as follows. Choose a 2-cocycle $\check{\alpha}$ representing $\alpha$. Then $\check{\alpha}$ satisfies
$\widehat{\pi}(\check{\alpha}) = dm$ for some $m \in \check{C}^1(k, \widehat{X})$ since the
cohomology class $\alpha$ has vanishing image in ${\rm{H}}^2(k, \widehat{X}) \subset {\rm{H}}^2(k, \widehat{X})$. Then we may choose $\check{\beta}$ to be the 
1-cocycle $\widehat{j}(m)$ (and modify $x$ accordingly to map to $\beta_{\A}$) since we have seen that for our purposes it does not matter which $\beta$ and $x$ we choose
(for the given $\alpha$). The connecting map sends the cocycle $\check{\beta}$ to $\check{\alpha}$ by Proposition \ref{cechderivedconnectingmap}, so 
$\check{\beta}$ represents a class $\beta$ satisfying $\delta(\beta) = \alpha$. Since $x \mapsto \beta_{\A}$ as cohomology classes, we have $\widehat{j}(x_v) = \check{\beta} + de_v$ for some $e_v \in \check{C}^0(k_v, \widehat{B})$.

Choose $x_v \in \check{Z}^1(k_v, \widehat{X})$ representing the localization of $x$
at a place $v$, and choose a cocycle $\check{y} \in \check{Z}^1(k, X)$ representing $y$. The left side of (\ref{xannihilatesShapreimageproof1}) is, by definition, $\sum_v \inv_v(x_v \cup \check{y})$. To evaluate the right side, note that since $\pi(y) \in \Sha^1(G)$, for each place $v$ of $k$ we have $\pi(\check{y}) = ds_v$ for some $s_v \in \check{C}^0(k_v, G)$. Choose a $c_v \in \check{C}^0(k_v, X)$ such that $\pi(c_v) = s_v$. Then $\pi(\check{y}) = d\pi(c_v) = \pi d(c_v)$, so $\check{y}_v = j(b_v) + dc_v$ in $\check{Z}^1(k_v, X)$ for some $b_v \in \check{Z}^1(k_v, B)$ and $c_v \in \check{C}^0(k_v, X)$. We may find $h \in \check{C}^2(k, \Gm)$ such that 
$\check{\alpha} \cup \pi(\check{y}) = dh$ since $\check{{\rm{H}}}^3(k, \Gm) = 0$ by Lemma \ref{checkh3=0}. The right side of (\ref{xannihilatesShapreimageproof1}) is, by the definition of $\langle \cdot, \cdot \rangle_{\Sha}$ given in \S \ref{sectiondefiningshapairings}, 
$-\sum_v \inv_v((\check{\alpha} \cup \pi(c_v)) - h)$.
Thus, what we need to check is: 
\begin{equation}
\label{xannihilatesShapreimageproof2}
\sum_v \inv_v(x_v \cup \check{y}) = \sum_v \inv_v(-(\check{\alpha} \cup \pi(c_v)) + h).
\end{equation}

Let us compute the left side of (\ref{xannihilatesShapreimageproof2}). We have $$x_v \cup \check{y} = x_v \cup j(b_v) + x_v \cup dc_v = \widehat{j}(x_v) \cup b_v + dx_v \cup c_v - d(x_v \cup c_v) = \widehat{j}(x_v) \cup b_v - d(x_v \cup c_v)$$ since $dx_v = 0$. As a cohomology class, therefore, this equals $$\widehat{j}(x_v) \cup b_v = (\check{\beta} + de_v) \cup b_v = (\check{\beta} \cup b_v) - (e_v \cup db_v) + d(e_v \cup b_v) = (\check{\beta} \cup b_v) + d(e_v \cup b_v)$$ since $db_v = 0$. As cohomology classes, therefore, this equals $\check{\beta} \cup b_v$. Thus, the left side of (\ref{xannihilatesShapreimageproof2}) is $$\sum_v \inv_v(\check{\beta} \cup b_v).$$ 

Now we compute the right side of (\ref{xannihilatesShapreimageproof2}). 
We have $$\check{\alpha} \cup \pi(c_v) = \widehat{\pi}(\check{\alpha}) \cup c_v = dm \cup c_v = m \cup dc_v + d(m \cup c_v),$$
which is cohomologous to $$m \cup dc_v = (m \cup \check{y}) - (m \cup j(b_v)) = (m \cup \check{y}) - (\widehat{j}(m) \cup b_v) = 
(m \cup \check{y}) - (\check{\beta} \cup b_v).$$ The right side of (\ref{xannihilatesShapreimageproof2}) is therefore $\sum_v \inv_v(-(m \cup \check{y}) + (\check{\beta} \cup b_v) + h)$. Thus, (\ref{xannihilatesShapreimageproof2}) is equivalent to the equality
\[
\sum_v \inv_v(h - (m \cup \check{y})) = 0
\]
(where the expression inside the sum is a cocycle due to the preceding calculation showing it is cohomologous to a cocycle), which holds because the sum of the local invariants of a global Brauer class is $0$.
\end{proof}

\begin{proposition}
\label{Sha^2--->Sha^1injective}
If $G$ is an affine commutative group scheme of finite type over a global function field $k$, then the map $\Sha^2(\widehat{G}) \rightarrow \Sha^1(G)^*$ induced by $\langle \cdot, \cdot \rangle_{\Sha}$ is injective.
\end{proposition}

\begin{proof}
We proceed by induction on the dimension of the unipotent radical of $(G_{\overline{k}})^0_{\red}$, the $0$-dimensional case corresponding to when $G$ is an almost-torus, and the proposition is Lemma \ref{Sha^2(G^)-->Sha^1(G)*injectsalmosttori}. So suppose that the unipotent radical of $(G_{\overline{k}})^0_{\red}$ is positive-dimensional. By Lemma \ref{affinegroupstructurethm} and induction, we may assume that there is an exact sequence
\begin{equation}
\label{Sha^2--->Sha^1*injectsproofseq1}
1 \longrightarrow H \xlongrightarrow{j} G \xlongrightarrow{\pi} \Ga \longrightarrow 1
\end{equation}
such that the proposition holds for $H$. This yields by Proposition \ref{hatisexact} the dual exact sequence
\[
1 \longrightarrow \widehat{\Ga} \xlongrightarrow{\widehat{\pi}} \widehat{G} \xlongrightarrow{\widehat{j}} \widehat{H} \longrightarrow 1.
\]
Suppose that $\alpha \in \Sha^2(\widehat{G})$ annihilates $\Sha^1(G)$. By functoriality, $\widehat{j}(\alpha) \in \Sha^2(\widehat{H})$ annihilates $\Sha^1(H)$, so by hypothesis, we have $\widehat{j}(\alpha) = 0$. We therefore have $\alpha = \widehat{\pi}(u)$ for some $u \in {\rm{H}}^2(k, \widehat{\Ga})$. Let $u_{\A}$ denote the image of $u$ in ${\rm{H}}^2(\A, \widehat{\Ga})$. Since $\widehat{\pi}(u) = \alpha \in \Sha^2(\widehat{G})$, 
we see that $u_{\A} = \delta'(w)$ for some $w \in {\rm{H}}^1(\A, \widehat{H})$, where $\delta':  {\rm{H}}^1(\A, \widehat{H}) \rightarrow {\rm{H}}^2(\A, \widehat{\Ga})$ is the connecting map. Via the adelic pairing for $H$,
$w$ defines an element $\phi_w \in {\rm{H}}^1(k, H)^*$.

We claim that by modifying $w$ by an element of ${\rm{H}}^1(\A, \widehat{G})$ (as we are free to do, since this does not affect its image $u_{\A} \in {\rm{H}}^2(\A, \widehat{\Ga})$), we can ensure that it annihilates ${\rm{H}}^1(k, H)$. We have an inclusion
\[
\frac{{\rm{H}}^1(k, H)}{j^{-1}(\Sha^1(G))} \hookrightarrow \frac{{\rm{H}}^1(k, G)}{\Sha^1(G)}.
\]
By Lemma \ref{wannihilatesShapreimage} below, $\phi_w$ descends to an element of 
$({\rm{H}}^1(k, H)/j^{-1}(\Sha^1(G)))^*$, hence extends to an element $\psi \in ({\rm{H}}^1(k, G)/\Sha^1(G))^*$. By Corollary \ref{exactnessatH^1(G^)*finitecor}, $\psi$ is induced by some element of ${\rm{H}}^1(\A, \widehat{G})$. Modifying $w$ by this element, we obtain that $w$ kills ${\rm{H}}^1(k, H)$, as desired.

It follows from Proposition \ref{exactnessatH^1(A,G^)} that $w$ lifts to some $h \in {\rm{H}}^1(k, \widehat{H})$, so, by replacing $u$ with
$u - \delta'(h)$ (which has no effect on its image $\alpha \in {\rm{H}}^2(k, \widehat{G})$), we may assume that $u \in \Sha^2(\widehat{\Ga})$. But 
$\Sha^2(\widehat{\Ga}) = 0$ by Lemma \ref{H^2(k,Ga^)injectiveseparable} (applied to the extension $k_v/k$ for any place $v$ of $k$), so $\alpha = \widehat{\pi}(u) = 0$, as desired. The proof of Proposition \ref{Sha^2--->Sha^1injective} is complete.
\end{proof}

\begin{lemma}
\label{wannihilatesShapreimage}
The element $w \in {\rm{H}}^1(\A, \widehat{H})$ annihilates $j^{-1}(\Sha^1(G)) \subset {\rm{H}}^1(k,H)$. 
\end{lemma}

\begin{proof}
Let $x \in j^{-1}(\Sha^1(G))$. We will show that
\begin{equation}
\label{wannihilatesShapreimageeqn1}
\langle w, x_{\A} \rangle = \langle \alpha, j(x) \rangle_{\Sha},
\end{equation}
where the left side is the adelic pairing for $H$. 
Since $\alpha$ annihilates $\Sha^1(G)$ by assumption, this will prove the lemma.

Let us first compute the left side of (\ref{wannihilatesShapreimageeqn1}). Let $\delta \colon  \Ga(\A) \rightarrow {\rm{H}}^1(\A, H)$ denote the connecting map. Since $j(x) \in \Sha^1(G)$, for each place $v$ we have $x_v = \delta(y_v)$ for some $y_v \in \Ga(k_v)$. We abuse notation and refer to $w, y_v$ as cocycles (fixed hereafter) representing the cohomology classes that we have been calling $w, y_v$. Let us recall how $\delta(y_v)$ is defined. We may choose a $z_v \in \check{C}^0(k_v, G)$ such that $\pi(z_v) = y_v$. Then $\delta(y_v)$ is represented by the cocycle in $\check{Z}^1(k_v, H)$ (which we still denote by $\delta(y_v)$) such that $j(\delta(y_v)) = dz_v$, so the left side of (\ref{wannihilatesShapreimageeqn1}) is 
\[
\sum_v \inv_v(w_v \cup \delta y_v).
\]

Now we compute the right side of (\ref{wannihilatesShapreimageeqn1}). 
Let $\alpha, u$, and so on denote representative cocycles for cohomology classes denoted above by the same
notation (here we are using Propositions \ref{cech=derivedsmoothinf} and \ref{cech=derivedG^}). By modifying the cocycle $\alpha$, we may assume that $\alpha = \widehat{\pi}(u)$ as cocycles rather than merely as cohomology classes. We have $x_v = \delta y_v + de_v$ for some $e_v \in \check{C}^0(k_v, H)$, 
so $j(x) = j(\delta y_v) + dj(e_v) = dz_v + dj(e_v)$. Also, $\alpha \cup j(x) = \widehat{j}(\alpha) \cup x = (\widehat{j} \circ \widehat{\pi}(u)) \cup x = 0$ as cocycles, so the right side of (\ref{wannihilatesShapreimageeqn1}) is, by definition,
\[
\sum_v \inv_v((\alpha \cup z_v) + (\alpha \cup j(e_v))).
\]
But $\alpha = \widehat{\pi}(u)$, so $$(\alpha \cup z_v) + (\alpha \cup j(e_v)) = (u \cup \pi(z_v)) + (u \cup \pi \circ j(e_v)) = u \cup \pi(z_v) = u \cup y_v.$$ 

We also have $u_v = \delta' w_v$ as cohomology classes, so, since $y_v$ is also a cocycle, the right side of (\ref{wannihilatesShapreimageeqn1}) equals 
\[
\sum_v \inv_v(\delta' w_v \cup y_v).
\]
By Proposition \ref{cechderivedconnectingmap} (perhaps after modifying the cocycle -- but not the cohomology class -- $w_v$ so that it lifts to an element of $\check{C}^1(k_v, \widehat{G})$; see the proof of Lemma \ref{xannihilatesShapreimage}), 
$\delta' w_v$ is (represented by) a cocycle such that $\widehat{\pi}(\delta' w_v) = dm_v$ for some $m_v \in \check{C}^1(k_v, \widehat{G})$ satisfying $\widehat{j}(m_v) = w_v$ as cohomology classes. Since we only care about the cohomology class of $w_v$, we may therefore modify the cocycle and assume that this holds as an equality of cocycles. So $$\delta' w_v \cup y_v = \delta' w_v \cup \pi(z_v) = \widehat{\pi}(\delta' w_v) \cup z_v = dm_v \cup z_v = m_v \cup dz_v + d(m_v \cup z_v),$$ which is cohomologous to $$m_v \cup dz_v = m_v \cup j(\delta y_v) = \widehat{j}(m_v) \cup \delta y_v = w_v \cup \delta y_v.$$ Thus, the right side of (\ref{wannihilatesShapreimageeqn1}) equals
\[
\sum_v \inv_v(w_v \cup \delta y_v), 
\]
which is the same as the left side.
\end{proof}

\section{Injectivity of $\Sha^1(\widehat{G}) \rightarrow \Sha^2(G)^*$ and finiteness of $\Sha$}

In this section we will complete the proof of Theorem \ref{shapairing} (Theorem \ref{Shapairingprop}). The main remaining issue is to show that if $G$ is an affine commutative group scheme of finite type over a global function field $k$, then the map $\Sha^1(\widehat{G}) \rightarrow \Sha^2(G)^*$ induced by the pairing $\langle \cdot, \cdot \rangle_{\Sha}$ is injective. This, together with the finiteness of these two $\Sha$ groups, is the content of the following proposition, whose proof will be the main work of this section.

\begin{proposition}
\label{Sha^1(G^)-->Sha^2(G)*}
Let $k$ be a global function field, $G$ an affine commutative $k$-group scheme of finite type. Then the map $\Sha^1(\widehat{G}) \rightarrow \Sha^2(G)^*$ induced by $\langle \cdot, \cdot \rangle_{\Sha}$ is injective, and both groups are finite.
\end{proposition}

\begin{proof}
Lemmas \ref{Sha=Sha(almosttorus)} and \ref{affinegroupstructurethm}, together with the functoriality of $\langle \cdot, \cdot \rangle_{\Sha}$, reduce us to the case in which $G$ is an almost-torus, so we will from now on assume that we are in this setting. By Lemma \ref{almosttorus}(iv), we may modify $G$ in order to assume that we have an exact sequence
\[
1 \longrightarrow B \xlongrightarrow{j}\R_{k'/k}(T') \times A \xlongrightarrow{\pi} G \longrightarrow 1
\]
for some finite commutative $k$-groups $A, B$, some finite separable extension $k'/k$, and some split $k'$-torus $T'$. For notational convenience, let $X : =\R_{k'/k}(T') \times A$, so that we have an exact sequence
\begin{equation}
\label{shapairingsequence3}
1 \longrightarrow B \xlongrightarrow{j} X \xlongrightarrow{\pi} G \longrightarrow 1.
\end{equation}

For an fppf abelian sheaf $\mathscr{F}$ on $\Spec(k)$, define $\Che^i(\mathscr{F}) : = \coker({\rm{H}}^i(k, \mathscr{F}) \rightarrow {\rm{H}}^i(\A, \mathscr{F}))$. We will first define an isomorphism
\[
f:  \frac{\ker(\Che^2(B) \rightarrow \Che^2(X))}{\delta(\Che^1(G))} \xrightarrow{\sim} \frac{\Sha^2(G)}{\pi(\Sha^2(X))},
\]
where $\delta$ denotes the map induced by the connecting map in the long exact cohomology sequence associated to (\ref{shapairingsequence3}). Consider the commutative diagram of exact sequences
\[
\begin{tikzcd}
{\rm{H}}^1(k, G) \arrow{r} \arrow{d} & {\rm{H}}^2(k, B) \arrow{r} \arrow{d} & {\rm{H}}^2(k, X) \arrow{r}{\gamma} \arrow{d}{\beta} & {\rm{H}}^2(k, G) \arrow{r} \arrow{d} & 0 \\
{\rm{H}}^1(\mathbf{A}, G) \arrow{r} & {\rm{H}}^2(\mathbf{A}, B) \arrow{r}{\alpha} & {\rm{H}}^2(\mathbf{A}, X) \arrow{r} & {\rm{H}}^2(\mathbf{A}, G) &
\end{tikzcd}
\]
(with the $0$ in the top row coming from Proposition \ref{cohomologicalvanishing}). For 
$x \in \ker(\Che^2(B) \rightarrow \Che^2(X))/\delta(\Che^1(G))$, choose $y \in {\rm{H}}^2(\mathbf{A}, B)$ representing $x$. Since $x \in \ker(\Che^2(B) \rightarrow \Che^2(X))$, there exists $z \in {\rm{H}}^2(k, X)$ such that $\beta(z) = \alpha(y)$. We define $f(x)$ to be the class of $\gamma(z)$ in $\Sha^2(G)/\pi(\Sha^2(X))$. It is straightforward to see from the above diagram that $f$ is a well-defined isomorphism.

Via Proposition \ref{funsequence}, for any affine commutative $k$-group scheme $H$ of finite type we have a functorial isomorphism \textcyr{Ch}$^2(H) \simeq \widehat{H}(k)^*$ by forming cup products over every $k_v$ and adding the local invariants. We therefore have an isomorphism
\begin{equation}
\label{shapairingisomorphism1}
\phi:  \frac{\ker(\widehat{B}(k)^* \rightarrow \widehat{X}(k)^*)}{\delta(\Che^1(G))} \xrightarrow{\sim} \frac{\Sha^2(G)}{\pi(\Sha^2(X))}
\end{equation}

The exact sequence
\[
1 \longrightarrow \widehat{G} \xlongrightarrow{\widehat{\pi}} \widehat{X} \xlongrightarrow{\widehat{j}} \widehat{B} \longrightarrow 1
\]
coming from Proposition \ref{hatisexact} gives rise to a commutative diagram with exact rows
\[
\begin{tikzcd}
\widehat{X}(k) \arrow{r}{\widehat{j}} \arrow{d} & \widehat{B}(k) \arrow{r}{\delta'} \arrow{d}{\epsilon} & {\rm{H}}^1(k, \widehat{G}) \arrow{r} \arrow{d} & {\rm{H}}^1(k, \widehat{X}) \arrow{d} \\
\widehat{X}(\mathbf{A}) \arrow{r}{\widehat{j}} & \widehat{B}(\mathbf{A}) \arrow{r} & {\rm{H}}^1(\mathbf{A}, \widehat{G}) \arrow{r} & {\rm{H}}^1(\mathbf{A}, \widehat{X})
\end{tikzcd}
\]
where we have abused notation by denoting the pullback map $\widehat{X} \rightarrow \widehat{B}$ on $k$-points and on $\mathbf{A}$-points (as well as the map on sheaves) as $\widehat{j}$.
Thanks to this diagram, the connecting map $\delta'$ defines an isomorphism
\begin{equation}
\label{shapairingisomorphism2}
\psi:  \frac{\{ \chi \in \widehat{B}(k) \mid \epsilon(\chi) \in \widehat{j}(\widehat{X}(\mathbf{A})) \}}{\widehat{j}(\widehat{X}(k))} \xrightarrow{\sim} \ker(\Sha^1(\widehat{G}) \rightarrow \Sha^1(\widehat{X})).
\end{equation}

The finiteness assertion in Proposition \ref{Sha^1(G^)-->Sha^2(G)*} follows immediately from this isomorphism. Indeed, we first note that $\Sha^1(\widehat{X})$ is finite by Corollary \ref{shapairingfinite} and Lemma \ref{H^1=0Weilrestrictionsplittori}. Since $\widehat{B}(k)$ is finite, therefore, the finiteness of $\Sha^1(\widehat{G})$ follows from the isomorphism (\ref{shapairingisomorphism2}). The finiteness of $\Sha^2(G)$ then follows from that of $\Sha^1(\widehat{G})$ and Corollary \ref{Sha^2(G)-->Sha^1(G^)*finitecor}. It therefore only remains to prove the injectivity assertion of Proposition \ref{Sha^1(G^)-->Sha^2(G)*}, which we concentrate on for the remainder of the proof.

There is an obvious pairing between the groups on the left sides of (\ref{shapairingisomorphism1}) and (\ref{shapairingisomorphism2}), namely the one coming from the natural pairing $\widehat{B}(k) \times \widehat{B}(k)^* \rightarrow \mathbf{Q}/\mathbf{Z}$. It is easy to see that this is well-defined, as follows. By definition, $\ker(\widehat{B}(k)^* \rightarrow \widehat{X}(k)^*)$ kills $\widehat{j}(\widehat{X}(k))$. To see that $\delta(\Che^1(G))$ annihilates $\{ \chi \in \widehat{B}(k) \mid \epsilon(\chi) \in \widehat{j}(\widehat{X}(\mathbf{A})) \}$, recall that the pairing \textcyr{Ch}$^2(B) \times \widehat{B}(k) \rightarrow \mathbf{Q}/\mathbf{Z}$ is defined by applying the cup product and adding the local invariants. Thus, it suffices to show that for $z \in {\rm{H}}^1(\mathbf{A}, G)$ and $\chi \in \widehat{B}(k)$ that everywhere locally comes from $\widehat{X}(k_v)$, $\delta(z)_v \cup \chi_v = 0$ in ${\rm{H}}^2(k_v, \mathbf{G}_m)$ for all places $v$. More specifically, it suffices to check that for any place $v$ of $k$, any $z \in {\rm{H}}^1(k_v, G)$, and any $\chi' \in \widehat{X}(k_v)$, the cup product $\delta(z) \cup \widehat{j}(\chi')$ vanishes. But by the functoriality of cup product, the left side equals ${\rm{H}}^1(j)(\delta(z)) \cup \chi'$, which is $0$ because 
${\rm{H}}^2(j) \circ \delta = 0$.

Now we may prove Proposition \ref{Sha^1(G^)-->Sha^2(G)*} for the almost-torus $G$, which, as we have seen, suffices to complete the proof in general. We know by Corollary \ref{shapairingfinite} and Lemma \ref{H^1=0Weilrestrictionsplittori} that Proposition \ref{Sha^1(G^)-->Sha^2(G)*} holds for $X$. Given $\alpha \in \Sha^1(\widehat{G})$ that annihilates $\Sha^2(G)$, therefore, the functoriality of $\Sha$ implies that $\alpha \in \ker(\Sha^1(\widehat{G}) \rightarrow \Sha^1(\widehat{X}))$. It therefore suffices to show that the map $\ker(\Sha^1(\widehat{G}) \rightarrow \Sha^1(\widehat{X})) \rightarrow (\Sha^2(G)/\pi(\Sha^2(X)))^*$ is injective. By Lemma \ref{shapairingcompatibility1}, it suffices to show that for any $\chi \in \widehat{B}(k)$ that annihilates $\ker(\widehat{B}(k)^* \rightarrow \widehat{X}(k)^*)$, we have $\chi \in \widehat{j}(\widehat{X}(k))$. Since the group $\widehat{B}(k)$ is finite, this is a piece of elementary group theory: it is a special case of the statement that, for a finite abelian group $A$, and a subgroup $A' \subset A$, if $a \in A$ is annihilated by every element of $A^*$ which annihilates $A'$, then $a \in A'$. The proof of Proposition \ref{Sha^1(G^)-->Sha^2(G)*} is therefore complete.
\end{proof}

\begin{lemma}
\label{shapairingcompatibility1}
The pairings between the left sides and between the right sides of 
$(\ref{shapairingisomorphism1})$ and $(\ref{shapairingisomorphism2})$ differ by a sign. That is, the following diagram commutes: 
\[
\begin{tikzcd}
\ker(\widehat{B}(k)^* \rightarrow \widehat{X}(k)^*)/\delta(\Che^1(G)) \arrow[r, phantom, "\times"] \arrow[d, "\mbox{\rotatebox{90}{$\sim$}}"] \arrow{d}[swap]{\phi} &
 \frac{\left \{ \chi \in \widehat{B}(k) \mid \epsilon(\chi) \in \widehat{j}(\widehat{X}(\mathbf{A})) \right\}}{\widehat{j}(\widehat{X}(k))} \arrow{r} \arrow[d, "\mbox{\rotatebox{90}{$\sim$}}"] \arrow{d}[swap]{\psi} & 
\mathbf{Q}/\mathbf{Z} \arrow[d, "x \mapsto -x"] \\
\Sha^2(G)/\pi(\Sha^2(X)) \arrow[r, phantom, "\times"] &
\ker(\Sha^1(\widehat{G}) \rightarrow \Sha^1(\widehat{X})) \arrow{r} &
\mathbf{Q}/\mathbf{Z}
\end{tikzcd}
\]
where the top pairing is the one induced by the natural pairing $\widehat{B}(k) \times \widehat{B}(k)^* \rightarrow \mathbf{Q}/\mathbf{Z}$, and the bottom pairing is the one induced by $\langle  \cdot, \cdot \rangle_{\Sha}:  \Sha^2(G) \times \Sha^1(\widehat{G}) \rightarrow \mathbf{Q}/\mathbf{Z}$.
\end{lemma}

\begin{proof}
To avoid a proliferation of Greek letters, we will simply abuse notation and use $j, \pi$, etc. to denote 
induced maps of group schemes, cohomology groups, \v{C}ech cohomology, etc. Choose a $\chi \in \widehat{B}(k)$ and a $z \in \widehat{B}(k)^*$. Assume that $z|_{\widehat{X}(k)} = 0$
and that for each place $v$ of $k$, the character $\chi$ extends to an element of $\widehat{X}(k_v)$. We want to show that 
\begin{equation}
\label{shapairingequation12}
- \langle z, \chi \rangle = \langle \phi(z), \psi(\chi) \rangle_{\Sha}.
\end{equation}
We first lift $z$ to an element of $\ker(\Che^2(B) \rightarrow \Che^2(X))$. That is, we choose a $y \in {\rm{H}}^2(\A, B)$ whose image in ${\rm{H}}^2(\A, X)$ lifts to a cohomology class in ${\rm{H}}^2(k, X)$, and such that $z(\chi') = \sum_v \mbox{inv}_v(\chi'(y_v))$ for all $\chi' \in \widehat{B}(k)$, where the sum is over all places $v$ of $k$. Choose cocycles $\check{y}_v$ representing the $y_v$ (using Proposition \ref{cech=derivedsmoothinf}). Thus, the left side of (\ref{shapairingequation12}) equals
\[
- \sum_v \mbox{inv}_v(\chi(\check{y}_v)).
\]

Now we compute the right side of (\ref{shapairingequation12}). First, $\psi(\chi) = \delta'(\chi)$. Since $y$ represents an element of $\ker(\Che^2(B) \rightarrow \Che^2(X))$, there exists a $w \in 
{\rm{H}}^2(k, X)$ such that $j(y_v) = w_v$ for all $v$. Then $\pi(w) = \phi(z)$, by definition of $\phi$. To compute $\langle \pi(w), \delta'(\chi) \rangle_{\Sha}$, we choose a 2-cocycle $\check{w} \in \check{Z}^2(k, X)$ representing $w$ (using Proposition \ref{cech=derivedsmoothinf}). Then 
\begin{equation}
\label{shapairingequation13}
\check{w}_v = j(\check{y}_v) + d(u_v)
\end{equation}
for some $u_v \in \check{C}^1(k_v, X)$. We have $\pi(\check{w})_v = \pi(j(\check{y}_v)) + \pi(du_v) = d\pi(u_v)$ since $\pi \circ j = 0$.

We compute $\delta'(\chi)$ as follows: choose a $\zeta \in \check{C}^0(k, \widehat{X})$ such that $\widehat{j}(\zeta) = \chi$. Then choose a lift $\gamma \in \check{Z}^1(k, \widehat{G})$ of $d\zeta$. That is, $\widehat{\pi}(\gamma) = d\zeta$. The cocycle $\gamma$ represents $\delta'(\chi)$ by Proposition \ref{cechderivedconnectingmap}. Using Lemma \ref{checkh3=0}, choose an $h \in \check{C}^2(k, \mathbf{G}_m)$ such that $dh = \pi(w) \cup \gamma$. The right side of (\ref{shapairingequation12}) is by definition
\[
\sum_v \mbox{inv}_v ((\pi(u_v) \cup \gamma_v) - h).
\]
Thus, we need to show that
\begin{equation}
\label{shapairingequation14}
- \sum_v \mbox{inv}_v(\chi(\check{y}_v)) = \sum_v \mbox{inv}_v ((\pi(u_v) \cup \gamma_v) - h).
\end{equation}

We have $$\pi(u_v) \cup \gamma_v = u_v \cup \widehat{\pi}(\gamma_v) = u_v \cup d\zeta_v = -d(u_v \cup \zeta_v) + du_v \cup \zeta_v.$$ Thus, as cohomology classes, $$\pi(u_v) \cup \gamma_v = du_v \cup \zeta_v = (\check{w}_v \cup \zeta_v) - (j(\check{y}_v) \cup \zeta_v) = (\check{w}_v \cup \zeta_v) - (\check{y}_v \cup \widehat{j}(\zeta_v)) = (\check{w}_v \cup \zeta_v) - \chi_v(\check{y}_v),$$ since $\widehat{j}(\zeta) = \chi$. Therefore, (\ref{shapairingequation14}) reduces to 
\[
\sum_v \mbox{inv}_v ((\check{w} \cup \zeta) - h) \stackrel{?}{=} 0
\]
(where the object in parentheses is a cocycle due to the above calculation), and this follows from the fact that the sum of the local invariants of a global Brauer class is $0$.
\end{proof}

We may now prove Theorem \ref{shapairing}.

\begin{theorem}
\label{Shapairingprop} $($Theorem $\ref{shapairing}$$)$ If $G$ is an affine commutative group scheme of finite type over a global function field $k$, then the groups $\Sha^i(G)$ and $\Sha^i(\widehat{G})$ are finite for $i = 1,2$. Further, for $i = 1,2$, the functorial pairings $\langle \cdot, \cdot \rangle_{\Sha}$ yield perfect pairings of finite groups: 
\[
\Sha^i(G) \times \Sha^{3-i}(\widehat{G}) \rightarrow \Q/\Z.
\]
\end{theorem}

\begin{proof}
The finiteness of $\Sha^1(\widehat{G})$ and $\Sha^2(G)$ are part of Proposition \ref{Sha^1(G^)-->Sha^2(G)*}. The finiteness of $\Sha^1(k, G)$ for all affine $k$-group schemes of finite type (even without commutativity hypotheses) is 
\cite[Thm.\,1.3.3(i)]{conrad}, whose proof in the commutative case is much easier than in the general case
(as one can reduce to the smooth connected case by essentially elementary arguments \cite[\S6.1--6.2]{conrad},
and the smooth connected {\em commutative} case is settled in \cite[IV, 2.6(a)]{oesterle}). The finiteness of $\Sha^2(\widehat{G})$ follows from that of $\Sha^1(G)$ and Proposition \ref{Sha^2--->Sha^1injective}. It therefore only remains to prove the perfectness of the pairings. This follows from Propositions \ref{Sha^1--->Sha^2*injective}, \ref{Sha^2--->Sha^1injective}, and \ref{Sha^1(G^)-->Sha^2(G)*}, and Corollary \ref{Sha^2(G)-->Sha^1(G^)*finitecor}.
\end{proof}

\section{The dual $9$-term exact sequence}
\label{sectiondual9termsequence}

In this section we prove Theorem \ref{poitoutatesequencedual}. We will do this by taking the Pontryagin dual of Theorem \ref{poitoutatesequence}. But first let us note that we have already essentially proved Theorem \ref{poitoutatesequence}: 

\begin{theorem}
\label{poitoutatesequenceprop} $($Theorem $\ref{poitoutatesequence}$$)$ For an affine commutative group scheme $G$ of finite type over a global function field $k$, the functorial $($in $G$$)$ global duality sequence
\[
\begin{tikzcd}
0 \arrow{r} & {\rm{H}}^0(k, G)_{\pro} \arrow{r} & {\rm{H}}^0(\A, G)_{\pro} \arrow{r} \arrow[d, phantom, ""{coordinate, name=Z_1}] & {\rm{H}}^2(k, \widehat{G})^* \arrow[dll, rounded corners,
to path={ -- ([xshift=2ex]\tikztostart.east)
|- (Z_1) [near end]\tikztonodes
-| ([xshift=-2ex]\tikztotarget.west) -- (\tikztotarget)}] & \\
& {\rm{H}}^1(k, G) \arrow{r} & {\rm{H}}^1(\A, G) \arrow{r} \arrow[d, phantom, ""{coordinate, name=Z_2}] & {\rm{H}}^1(k, \widehat{G})^* \arrow[dll, rounded corners,
to path={ -- ([xshift=2ex]\tikztostart.east)
|- (Z_2) [near end]\tikztonodes
-| ([xshift=-2ex]\tikztotarget.west) -- (\tikztotarget)}] & \\
& {\rm{H}}^2(k, G) \arrow{r} & {\rm{H}}^2(\A, G) \arrow{r} & {\rm{H}}^0(k, \widehat{G})^* \arrow{r} & 0
\end{tikzcd}
\]
is exact.
\end{theorem}

\begin{proof}
The discussion in \S \ref{globalsectionpreliminaries}, in conjunction with Theorem \ref{Shapairingprop}, shows that the sequence is a well-defined complex that is exact at ${\rm{H}}^i(k, G)$ for $i = 1, 2$. Combining Propositions \ref{G(k)G(A)proinjective}, \ref{exactnessatG(A)}, \ref{exactnessatH^1(A,G)}, and \ref{funsequence} proves exactness everywhere except for ${\rm{H}}^i(k, \widehat{G})^*$ for $i = 1, 2$. The map ${\rm{H}}^i(k, \widehat{G})^* \rightarrow {\rm{H}}^{3-i}(k, G)$ is defined to be the composition
\[
{\rm{H}}^i(k, \widehat{G})^* \twoheadrightarrow \Sha^i(\widehat{G})^* \xrightarrow{\sim} \Sha^{3-i}(G) \hookrightarrow {\rm{H}}^{3-i}(k, G),
\]
where the middle isomorphism is given by Theorem \ref{Shapairingprop}. The exactness of the sequence at ${\rm{H}}^2(k, \widehat{G})^*$ is therefore equivalent to the exactness of the sequence
\[
{\rm{H}}^0(\A, G)_{\pro} \longrightarrow {\rm{H}}^2(k, \widehat{G})^* \longrightarrow \Sha^2(\widehat{G})^*,
\]
and exactness at ${\rm{H}}^1(k, \widehat{G})^*$ is equivalent to the exactness of
\[
{\rm{H}}^1(\A, G) \longrightarrow {\rm{H}}^1(k, \widehat{G})^* \longrightarrow \Sha^1(\widehat{G})^*.
\]
The exactness of these two sequences follows from Proposition \ref{exactnessatH^2(k,G^)*} and Corollary \ref{exactnessatH^1(G^)*finitecor}.
\end{proof}

We now prove Theorem \ref{poitoutatesequencedual}.

\begin{theorem}
\label{poitoutatedualsequenceprop} $($Theorem $\ref{poitoutatesequencedual}$$)$ For an affine commutative group scheme $G$ of finite type over a global function field $k$, the functorial $($in $G$$)$ global duality sequence
\[
\begin{tikzcd}
0 \arrow{r} & {\rm{H}}^0(k, \widehat{G})_{\pro} \arrow{r} & {\rm{H}}^0(\A, \widehat{G})_{\pro} \arrow{r} \arrow[d, phantom, ""{coordinate, name=Z_1}] & {\rm{H}}^2(k, G)^* \arrow[dll, rounded corners,
to path={ -- ([xshift=2ex]\tikztostart.east)
|- (Z_1) [near end]\tikztonodes
-| ([xshift=-2ex]\tikztotarget.west) -- (\tikztotarget)}] & \\
& {\rm{H}}^1(k, \widehat{G}) \arrow{r} & {\rm{H}}^1(\A, \widehat{G}) \arrow{r} \arrow[d, phantom, ""{coordinate, name=Z_2}] & {\rm{H}}^1(k, G)^* \arrow[dll, rounded corners,
to path={ -- ([xshift=2ex]\tikztostart.east)
|- (Z_2) [near end]\tikztonodes
-| ([xshift=-2ex]\tikztotarget.west) -- (\tikztotarget)}] & \\
& {\rm{H}}^2(k, \widehat{G}) \arrow{r} & {\rm{H}}^2(\A, \widehat{G}) \arrow{r} & ({\rm{H}}^0(k, G)_{\rm{pro}})^D \arrow{r} & 0
\end{tikzcd}
\]
is exact.
\end{theorem}

\begin{proof}
We begin with the exact sequence of Theorem \ref{poitoutatesequenceprop}, but we replace ${\rm{H}}^0(k, \widehat{G})^*$ with $({\rm{H}}^0(k, \widehat{G})_{\pro})^D$, which makes no difference because ${\rm{H}}^0(k, \widehat{G})$ is finitely generated (this is part of Theorem \ref{H^2(G)G^(k)dualityprop}): 
\begin{equation}
\label{poitoutatesequencedualproofeqn1}
\begin{tikzcd}
0 \arrow{r} & {\rm{H}}^0(k, G)_{\pro} \arrow{r} & {\rm{H}}^0(\A, G)_{\pro} \arrow{r} \arrow[d, phantom, ""{coordinate, name=Z_1}] & {\rm{H}}^2(k, \widehat{G})^* \arrow[dll, rounded corners,
to path={ -- ([xshift=2ex]\tikztostart.east)
|- (Z_1) [near end]\tikztonodes
-| ([xshift=-2ex]\tikztotarget.west) -- (\tikztotarget)}] & \\
& {\rm{H}}^1(k, G) \arrow{r} & {\rm{H}}^1(\A, G) \arrow{r} \arrow[d, phantom, ""{coordinate, name=Z_2}] & {\rm{H}}^1(k, \widehat{G})^* \arrow[dll, rounded corners,
to path={ -- ([xshift=2ex]\tikztostart.east)
|- (Z_2) [near end]\tikztonodes
-| ([xshift=-2ex]\tikztotarget.west) -- (\tikztotarget)}] & \\
& {\rm{H}}^2(k, G) \arrow{r} & {\rm{H}}^2(\A, G) \arrow{r} & ({\rm{H}}^0(k, \widehat{G})_{\pro})^D \arrow{r} & 0
\end{tikzcd}
\end{equation}
where the global cohomology groups ${\rm{H}}^i(k, \cdot)$ are endowed with the discrete topology and the adelic cohomology groups with the topology described in \S \ref{sectionadeliccohomtopology}. Since the groups ${\rm{H}}^i(k, G)$ and ${\rm{H}}^i(k, \widehat{G})$ with $i=1,2$ are discrete torsion (Lemmas \ref{H^1finiteexponent}, \ref{H^2(G)istorsion}, and \ref{H^2(G^)istorsion}), their algebraic $\Q/\Z$-duals agree with their Pontryagin duals.

Next we observe that all of the maps in (\ref{poitoutatesequencedualproofeqn1}) are continuous. Indeed, this is trivial for the maps from global cohomology groups, and from the discrete group ${\rm{H}}^2(\A, G)$. For the maps from the adelic groups, when $i = 0$ this continuity follows from Lemma \ref{G(A)toH^2(k,G^)*cts}, and for $i =1$, it follows from the fact that each of the maps in the composition
\[
{\rm{H}}^1(\A, G) \longrightarrow {\rm{H}}^1(\A, \widehat{G})^D \longrightarrow {\rm{H}}^1(k, \widehat{G})^D
\]
is continuous (where the first map is given by the adelic duality pairing -- i.e., cup everywhere locally and add the invariants), the first map by Proposition \ref{H^1(A,G)=H^1(A,G^)^D}, and the second because it is the dual of the trivially continuous map ${\rm{H}}^1(k, \widehat{G}) \rightarrow {\rm{H}}^1(\A, \widehat{G})$. Finally, the maps from ${\rm{H}}^i(k, \widehat{G})^*$ ($i = 1, 2$) are continuous because they are given by the compositions
\[
{\rm{H}}^i(k, \widehat{G})^* \rightarrow \Sha^i(k, \widehat{G})^* \xrightarrow{\sim} \Sha^{3-i}(k, G) \rightarrow {\rm{H}}^{3-i}(k, G),
\]
in which the $\Sha$ groups have the discrete topology and all of the maps are trivially continuous (the first because it is the dual of a trivially continuous map). 

We may therefore apply $(\cdot)^D$ to (\ref{poitoutatesequencedualproofeqn1}) to obtain the dual complex: 
\begin{equation}
\label{poitoutatesequencedual1}
\begin{tikzcd}
0 \arrow{r} & {\rm{H}}^0(k, \widehat{G})_{\pro} \arrow{r} & {\rm{H}}^2(\A, G)^D \arrow{r} \arrow[d, phantom, ""{coordinate, name=Z_1}] & {\rm{H}}^2(k, G)^* \arrow[dll, rounded corners,
to path={ -- ([xshift=2ex]\tikztostart.east)
|- (Z_1) [near end]\tikztonodes
-| ([xshift=-2ex]\tikztotarget.west) -- (\tikztotarget)}] & \\
& {\rm{H}}^1(k, \widehat{G}) \arrow{r} & {\rm{H}}^1(\A, G)^D \arrow{r} \arrow[d, phantom, ""{coordinate, name=Z_2}] & {\rm{H}}^1(k, G)^* \arrow[dll, rounded corners,
to path={ -- ([xshift=2ex]\tikztostart.east)
|- (Z_2) [near end]\tikztonodes
-| ([xshift=-2ex]\tikztotarget.west) -- (\tikztotarget)}] & \\
& {\rm{H}}^2(k, \widehat{G}) \arrow{r} & ({\rm{H}}^0(\A, G)_{\pro})^D \arrow{r} & ({\rm{H}}^0(k, G)_{\pro})^D \arrow{r} & 0
\end{tikzcd}
\end{equation}
By Proposition \ref{H^1(A,G)=H^1(A,G^)^D}, we may replace the terms ${\rm{H}}^2(\A, G)^D$, ${\rm{H}}^1(\A, G)^D$, and $({\rm{H}}^0(\A, G)_{\pro})^D$ with ${\rm{H}}^0(\A, \widehat{G})_{\pro}$,  ${\rm{H}}^1(\A, \widehat{G})$, and ${\rm{H}}^2(\A, \widehat{G})$, respectively, and these  identifications are via the adelic duality pairings (which cup everywhere locally and sum the invariants).

Making these replacements in 
(\ref{poitoutatesequencedual1}) 
gives exactly the desired diagram in Theorem \ref{poitoutatedualsequenceprop} (with the correct maps), and so completes the proof of the proposition provided that the passage to the Pontryagin dual complex above preserves exactness. In order to show this, we first claim that all of the groups appearing in (\ref{poitoutatesequencedualproofeqn1}) are Hausdorff, and either (i) compact, or (ii) locally compact and second-countable. The profinite completions are compact Hausdorff by definition. The adelic cohomology groups are locally compact, second-countable, and Hausdorff by Proposition \ref{adelictopcohombasics}(iv). The global cohomology groups are discrete by definition, and they are (second-)countable by Lemma \ref{finiteexpcountable}. The $\Q/\Z$-duals of the global cohomology groups ${\rm{H}}^i(k, \widehat{G})$ ($i = 1, 2$) are the same as their Pontryagin duals by Lemmas \ref{H^1finiteexponent} and \ref{H^2(G^)istorsion}, hence (as duals of discrete groups) compact. Finally, $({\rm{H}}^0(k, \widehat{G})_{\pro})^D$ is discrete (as the dual of a compact group), and countable, since ${\rm{H}}^0(k, \widehat{G})$ is finitely generated. The exactness of the dual complex (\ref{poitoutatesequencedual1}) now follows from Lemma \ref{dualstillexact}. (The exactness at ${\rm{H}}^0(k, \widehat{G})_{\pro}$ follows by appending an extra zero at the end of the exact sequence (\ref{poitoutatesequencedualproofeqn1}).)
 \end{proof}

\appendix

\chapter{Products and Ultraproducts}
\label{ultraproductsappendix}

This appendix is allegedly concerned with certain properties of ultraproducts, but our real interest in them arises from the fact that ultraproducts appear when forming localizations of rings that are described as infinite products, such as $\prod_v \calO_v$, the product of the rings of integers of the places of a global field $k$. As one may imagine, such localizations play an important role when studying a ring such as the adeles of a global field $k$. More precisely, as we shall explain below, every local ring of a product ring $\prod_{i \in I} R_i$ is a local ring of an ultraproduct $\prod_U R_i$ for some ultrafilter $U$ on $I$.

\section{Ultraproducts as local rings}

Let $I$ be a non-empty set, $\{ R_i\}_{i \in I}$ a set of rings (commutative with identity and nonzero; throughout this discussion, a ring means a nonzero ring) indexed by $I$. The purpose of this section is to undertake a study of $\Spec(\prod_{i \in I} R_i)$, as this will be necessary to prove some important results about the cohomology of the adeles in \S\ref{sectionadeliclocalcohom}. We will be particularly interested in studying the local rings of $\Spec(\prod_{i \in I} R_i)$.

In order to do this, we need to recall the notion of an {\em ultraproduct} of commutative rings.
(As will be clear, one may consider ultraproducts of quite general mathematical structures; we focus on rings since that is the only case we will need.) Let $U$ be an {\em ultrafilter} on the set $I$. That is, $U$ is a {\em proper} collection of subsets of $I$
with three properties:  (i) if $A \in U$ and $A \subset B \subset I$, then $B \in U$; (ii) if $A, A' \in U$, then $A \cap A' \in U$; and (iii) for all $A \subset I$, one of $A$ or $I - A$ belongs to $U$. (Condition (iii)
implies that $U$ is not the empty collection.) In (iii), we cannot have both belonging
to $I$ or else by (ii) we would have $\emptyset \in U$ and hence by (i) {\em every}
subset of $I$ belongs to $U$, contradicting that $U$ is not the entire power set of $I$.
Conditions (i) and (ii) -- in conjunction with the assumption that $U \neq\emptyset$ -- say that $U$ is a {\em filter},
and (iii) is equivalent to the condition that $U$ is not strictly contained in any larger filter on $I$ ((iii) also implies that $U \neq \emptyset$). The ``obvious'' ultrafilters are those consisting of all subsets of $I$ containing a fixed element $i_0 \in I$; these are called {\em principal} ultrafilters.
It is elementary to check that if $I$ is finite, then the only ultrafilters on $I$ are the principal ones. 

\begin{remark}
One should regard an ultrafilter as defining a notion of ``bigness'' for subsets of $I$
(so one might say that $A \subset I$ is {\em $U$-big} when $A \in U$); the axioms of an ultrafilter express reasonable conditions on any notion of bigness for subsets of a given set. 
\end{remark}

We define the {\em ultraproduct} $\prod_U R_i$ to be the quotient ring 
$(\prod_{i \in I} R_i)/\sim$, where $\sim$ is the equivalence relation defined by $(a_i) \sim (b_i) \Longleftrightarrow \{i \in I \mid a_i = b_i\} \in U$. The ring structure is induced by that on $\prod R_i$ via coordinate-wise addition and multiplication. (In other words,
$\prod_U R_i$ is the quotient of $\prod R_i$ modulo the ideal $J(U)$ consisting of elements
$(a_i) \in \prod R_i$ for which $\{i \in I\,|\,a_i = 0\} \in U$; this is an ideal because $U$ is a filter.) Alternatively, we may define the ultraproduct $\prod_U R_i$ as a filtered direct limit via the formula
\[
\prod_U R_i := \varinjlim_{I \in U} \prod_{i \in I} R_i.
\]

Note that when $U$ is the principal ultrafilter corresponding to an element $i_0 \in I$, the associated
ultraproduct $\prod_U R_i$ is the factor ring $R_{i_0}$.  The ``interesting'' ultraproducts are
therefore those corresponding to non-principal ultrafilters (and by Zorn's Lemma such do exist whenever $I$ is infinite, though we will never use this fact; our entire discussion in this chapter is independent of the axiom of choice).

Informally, working in the ultraproduct $\prod_U R_i$ amounts to considering elements $(x_i)$ only for ``big'' sets of indices
$i$ where bigness is defined via $U$.  Note that if $x \in \prod_U R_i$ is the class of an $I$-tuple $(x_i)$,
then $x \ne 0$ if and only if $\{i \in I\,|\,x_i = 0\} \not\in U$, but the complement of this set of indices
is exactly $\{i \in I\,|\,x_i \ne 0\}$, so by the ultrafilter property of $U$ we conclude that
$x \ne 0$ if and only if $\{i \in I\,|\,x_i \ne 0\} \in U$.  In other words, $x=0$ 
precisely when $x_i$ vanishes for a ``big'' set of indices $i\in I$, and
$x \ne 0$ precisely when $x_i \ne 0$ for a ``big'' set of indices $i \in I$.  We will use this without comment in some arguments below.

Let us explain how ultraproducts make an appearance when studying the local rings of $\prod_{i \in I} R_i$. For a subset $S \subset I$, define the idempotent $z_S \in \prod_{i \in I} R_i$ by 
\[
(z_S)_i = 
\begin{cases}
0, & i \in S \\
1, & i \notin S.
\end{cases}
\]
To each prime ideal $\mathfrak{p}$ of $\prod R_i$ we can assign an ultrafilter $U(\mathfrak{p})$ on $I$ via
$$U(\mathfrak{p}) = \{S \subset I \mid z_S \in \mathfrak{p}\}.$$
It is easy to verify that this is an ultrafilter, using the primality of $\mathfrak{p}$. Indeed, first note that if $S \in U(\mathfrak{p})$ and $S \subset S' \subset I$, then $z_{S'} = z_{S'}z_S \in \mathfrak{p}$, so $S' \in U(\mathfrak{p})$. Next, we note that $z_S \cdot z_{I-S} = 0$
for any $S \subset I$, so either $z_S$ or $z_{I-S}$ lies in $\mathfrak{p}$, but not both, because their sum is $1$. Finally, we need to check that $U(\mathfrak{p})$ is closed under finite intersection. Given $S, S' \in U(\mathfrak{p})$ (so $z_S, z_{S'} \in \mathfrak{p}$), their complements $S^c, S'^c$ in $I$ are not in $U(\mathfrak{p})$. Therefore, $z_{S^c}, z_{S'^c} \notin \mathfrak{p}$, so because $\mathfrak{p}$ is prime, $z_{S^c}z_{S'^c} = z_{S^c\cup S'^c} \notin \mathfrak{p}$. That is, $S^c\cup S'^c \notin U(\mathfrak{p})$, so $(S^c \cup S'^c)^c = S \cap S' \in U(\mathfrak{p})$.

In the uninteresting case that $\mathfrak{p}$ arises from a factor ring $R_{i_0}$, $U(\mathfrak{p})$ is the principal ultrafilter associated to $i_0 \in I$. One may show (assuming the axiom of choice)
that every ultrafilter on $I$ arises as $U(\mathfrak{p})$ for some $\mathfrak{p}$, but we will not
require this fact.

The key point, and the whole reason for the introduction of the ultrafilter $U(\mathfrak{p})$, is the following proposition.

\begin{proposition}
\label{localringsprod=ultraprod}
Notation as above, the canonical quotient map $\prod_{i \in I} R_i \rightarrow \prod_{U(\mathfrak{p})} R_i$ induces an isomorphism
\[
(\prod_{i \in I} R_i)_{\mathfrak{p}} \xrightarrow{\sim} (\prod_{U(\mathfrak{p})} R_i)_{\overline{\mathfrak{p}}},
\]
where $\overline{\mathfrak{p}}$ is the image of $\mathfrak{p}$ in $\prod_{U(\mathfrak{p})} R_i$. In particular, every local ring of $\prod_{i \in I} R_i$ is a local ring of the ultraproduct of the $R_i$ associated to some ultrafilter on $I$.
\end{proposition}

\begin{proof}
Let $J(\mathfrak{p})$ denote the ideal $J(U(\mathfrak{p}))$ of $\prod R_i$, where we recall the ideal $J(U)$ appearing in the definition of the ultraproduct. Explicitly, $J(\mathfrak{p})$ consists of those elements $a = (a_i)$ for which the set $S$ of $i \in I$ such that $a_i = 0$ belongs to $U(\mathfrak{p})$,
which is to say $z_S \in \mathfrak{p}$.  But $a$ is a multiple of $z_S$, so in more concrete terms
$J(\mathfrak{p})$ is generated by the elements $z_S$ that belong to $\mathfrak{p}$.  Conversely,
if $S \subset I$ is such that $z_S \in J(\mathfrak{p})$ then necessarily $S \in U(\mathfrak{p})$ just by definition.

For the closed subscheme
$\Spec(\prod_{U(\mathfrak{p})} R_i) = \Spec((\prod R_i)/J(\mathfrak{p}))$ passing through the point $\{\mathfrak{p}\}$,
at all of its points the local ring {\em coincides}
with that of the ambient scheme $\Spec(\prod R_i)$.  We only need this equality at
the point $\mathfrak{p}$ of initial interest (though using the fact that any containment of ultrafilters must be an equality, one deduces that the points of $\Spec((\prod R_i)/J(\mathfrak{p}))$ are precisely the primes $\mathfrak{q}$ of $\prod_i R_i$ such that $U(\mathfrak{q}) = U(\mathfrak{p})$, hence the claim reduces to the case of the initial point of interest), and there it is easy to verify: 
we just need to check that each of the generators $z_S \in \mathfrak{p}$ of
$J(\mathfrak{p})$ vanish in $(\prod R_i)_{\mathfrak{p}}$,
and that is immediate because $z_{I - S} \not\in \mathfrak{p}$ and $z_{I-S} \cdot z_S = 0$.
\end{proof}

\section{Properties inherited by local rings of products}

Proposition \ref{localringsprod=ultraprod} says that every local ring of $\Spec(\prod_{i \in I} R_i)$ 
is a local ring of an ultraproduct $\prod_U R_i$ for some ultrafilter $U$ on $I$. This is in turn useful because many reasonable 
(more precisely:  ``first-order'') properties of 
commutative rings $R_i$ are inherited by ultraproducts. Any such property that is also inherited
by localization at primes is therefore inherited by
{\em every} local ring on $\prod R_i$ when it holds for each $R_i$!  

Let us consider some examples (sufficient for our needs) for the sake of illustration. If all of the rings $R_i$ are domains then 
we claim that so is any ultraproduct $\prod_U R_i$ (and thus so is the local ring of $\prod R_i$ at every prime ideal
$\mathfrak{p}$, which is not so obvious if one doesn't have the geometric idea to try proving
that the closed subscheme $\Spec(\prod_{U(\mathfrak{p})} R_i)$ passing through $\{\mathfrak{p}\}$
is actually a domain). Consider $x = (x_i), y = (y_i) \in \prod_U R_i$ (more precisely, $x, y$ are the classes of $(x_i), (y_i)$) 
such that $xy = 0$. Thus, $\{ i \in I \mid x_i y_i = 0\} \in U$. Since each $R_i$ is a domain, it follows that $\{ i \mid x_i = 0\} \cup \{i \mid y_i = 0\} \in U$, hence either $\{i \mid x_i = 0\} \in U$ or $\{i \mid y_i = 0\} \in U$ (or both); i.e., either $x = 0$ or $y = 0$
in $\prod_U R_i$.  (Here we have used the fact, immediate from the definition of an ultrafilter, that if $S, S' \subset I$ satisfy
$S \cup S' \in U$, then $S \in U$ or $S' \in U$; indeed, if $S \not\in U$ and $S' \not\in U$
then $I-S, I-S' \in U$ and hence $I - (S \cup S') = (I - S) \cap (I - S') \in U$, contradicting that $S \cup S' \in U$.) 
This proves that $\prod_U R_i$ is a domain when every $R_i$ is a domain.

Similarly, suppose that each $R_i$ is a domain such that any element of $K_i : = {\rm{Frac}}(R_i)$ satisfying a monic
polynomial of degree $3$ over $R_i$ lies in $R_i$. Then we claim that the same holds for any ultraproduct $\prod_U R_i$. Indeed, we already know that $\prod_U R_i$ is a domain, so suppose that $x = (x_i), y = (y_i) \in \prod_U R_i$
with $y \neq 0$ (so $y_i \ne 0$ precisely for $i$ belonging to a set $S \in U$) such that 
$$(x/y)^3 + a_2(x/y)^2 + a_1(x/y) + a_0 = 0$$
in $\prod_U R_i$ for some $a_0, a_1, a_2 \in \prod_U R_i$, or equivalently (since $\prod_U R_i$ is a domain), 
$$x^3 + a_2x^2y + a_1xy^2 + a_0y^3 = 0$$ in $\prod_U R_i$. The latter vanishing in 
the ultraproduct says that the set $S'$ of indices $i$ such that the corresponding
vanishing holds at index $i$ (without any condition on $y_i$ vanishing or not) is a member of $U$. 

But $S \cap S'$ must also belong to $U$,
and by our assumption about the $R_i$ it follows
that for any $i \in S \cap S'$ we have $x_i = z_iy_i$ for some $z_i \in R_i$. Defining $z$ be the class in $\prod_U R_i$ of the
$I$-tuple with $i$th component $z_i$ for $i \in S \cap S'$ and whatever component we wish (say 0, or 1)
for indices not in $S \cap S'$, we have $x = yz$ in the domain $\prod_U R_i$,
so $x/y  = z \in \prod_U R_i$, as desired. A similar argument with $3$ replaced by any positive integer
shows that if each $R_i$ is an integrally closed domain then so is $\prod_U R_i$. 
(Note that being integrally closed is not a ``first order'' condition, but being integrally closed with respect to polynomials of a fixed positive degree is a ``first order'' condition.) 

Since any localization of an integrally closed domain is an integrally closed domain,
and every local ring of $\prod R_i$ is a local ring of an ultraproduct of the $R_i$'s, we have proved the following useful result.

\begin{lemma}
\label{productofnormal=normal}
Let $R_i$ be a collection of rings indexed by the non-empty set $I$. If each $R_i$ is an integrally closed domain, then the local rings on the scheme $\Spec(\prod_{i \in I} R_i)$ are integrally closed domains.
\end{lemma}

In the same spirit, we also have: 

\begin{lemma}
\label{seminormalring}
Let $R_i$ be a collection of seminormal rings indexed by a non-empty set $I$. Then each local ring of $\Spec(\prod_{i \in I} R_i)$ is seminormal.
\end{lemma}

\begin{proof}
The idea of the proof is the same as for Lemma \ref{productofnormal=normal}. One first notes that reducedness is inherited by a product of rings, and then it is a straightforward exercise to see that seminormality is inherited by ultraproducts, using the formulation that a reduced ring is seminormal if (and only if) whenever $b^3 = c^2$, there exists $a$ such that $b = a^2$ and $c = a^3$.
\end{proof}

Here is another example that we shall find useful.

\begin{lemma}
\label{prodvalrings}
Let $R_i$ be a collection of valuation rings indexed by a non-empty set $I$. Then each local ring of $\Spec(\prod_{i \in I} R_i)$ is a valuation ring.
\end{lemma}

\begin{proof}
We need to show that the property of being a valuation ring is inherited by localizations and by ultraproducts. Both of these assertions are easily deduced from the following characterization of valuation rings: A ring $A$ is a valuation ring if and only if it is an integral domain with the following property: for every $a, b \in A$, either $a \mid b$ or $b\mid a$ (or both).
\end{proof}

Here is the last example of this phenomenon that we shall require.

\begin{lemma}
\label{proddegimpleq1}
Let $R_i$ be a collection of domains indexed by a non-empty set $I$. Assume that, for each $i \in I$, the fraction field $K_i$ of $R_i$ has degree of imperfection $\leq 1$. Then each local ring of ${\rm{Spec}}(\prod_{i \in I} R_i)$ is a domain whose fraction field has degree of imperfection $\leq 1$.
\end{lemma}

\begin{proof}
For any ultrafilter $U$ on $I$, one checks that $\prod_U R_i$ is a domain, and that its fraction field is $\prod_U K_i$. Since the property of being a domain with fraction field having degree of imperfection $\leq 1$ is inherited by local rings, it therefore suffices to show that the field $K := \prod_U K_i$ has degree of imperfection $\leq 1$. If for every prime $p$, a $U$-large set of $K_i$ have characteristic $\neq p$, then $p \neq 0$ in $K$ for every prime $p$, so $K$ has characteristic $0$, hence is perfect. We may therefore assume that a $U$-large set of $K_i$ have characteristic $p > 0$ for some fixed prime $p$. Then $K$ has characteristic $p$ as well. If a $U$-large set of $K_i$ are perfect, then every element of $K$ is a $p$th power (since a $U$-large set of its components are), hence $K$ is perfect. We may therefore assume that a $U$-large set of $K_i$ have degree of imperfection $1$. Replacing $I$ with this $U$-large set, we may assume that, for each $i \in I$, there is $t_i \in K_i$ such that $K_i = \oplus_{r=0}^{p-1} K_i^pt_i^r$. Let $t := \prod_i t_i \in K$. Then it follows that $K = \oplus_{r=0}^{p-1} K^pt^r$, so $K$ has degree of imperfection $1$.
\end{proof}

As the reader has likely noticed, we have used nothing particularly special about the class of integrally closed domains
or of valuation rings, etc.\,, beyond that their definitions are built up in terms of conditions expressible
in sufficiently finitistic terms (a notion made precise in terms of first-order logic).
{\L}o\'s's Theorem states that any first-order statement that holds for each of the factors $R_i$ holds for 
any ultraproduct $\prod_U R_i$ (and even more generally for ultraproducts of a wide class of
mathematical structures). As we have seen, in any special case of interest it is easy to prove by a direct argument
that such properties are inherited by ultraproducts, so there is no need for the general {\L}o\'s's Theorem anywhere in our work.

\chapter{Valuation Rings}
\label{valringschap}

Valuation rings play an important role in this text, especially in the study of properties of the ring $\A_k$ of adeles associated to a global field $k$. The relation between the adeles and valuation rings arises from the following observation. In studying $\A_k$, one is naturally led to consider rings of the form $\prod_{v \notin S} \calO_v$, where $S$ is a finite set of places of $k$. By Lemma \ref{prodvalrings}, each local ring of this ring is a valuation ring. We remark that -- assuming the existence of non-principal ultrafilters on infinite sets (a consequence, for example, of the axiom of choice) -- some (many) of these local rings will be non-Noetherian. We are therefore naturally led, even when studying products of DVR's, to consider arbitrary valuation rings. The purpose of the present appendix is to present some results on valuation rings that we shall require which are not in the literature (as far as the author is aware), particularly Proposition \ref{localringovervalringprop}.

\section{Iterates of matrices over rank-$1$ valuation rings}

The main result of this section is Lemma \ref{matrixpowers}, which says that, given a vector over a real-valued valuation ring, the absolute values of its iterates under a given matrix attain their infimum if this infimum is $> 0$. The proof of this seemingly innocuous statement is actually rather involved. We begin with a series of lemmas.

\begin{lemma}
\label{combinatorialid}
For integers $j, m \geq 0$, consider the polynomial
\[
F_{m,\,j}(X) := \sum_{s=0}^{m} (-1)^s\binom{m}{s}\binom{X-s}{j} \in \Q[X].
\]
Then
\begin{itemize}
\item[(i)] $F_{m,\,j}(X) = 0$ for $0 \leq j < m$.
\item[(ii)] $F_{m,\,m}(X) = 1$.
\end{itemize}
\end{lemma}

\begin{proof}
We prove the lemma for $0 \leq j \leq m$ by induction on $j + m$. For $j = m = 0$, one simply checks that the lemma holds. Now assume that $0 \leq j \leq m$ and that the lemma holds for all smaller values of $j + m$. Using Pascal's identity
\begin{equation}
\label{pascal}
\binom{Y+1}{k} = \binom{Y}{k} + \binom{Y}{k-1} \in \Q[Y],
\end{equation}
we first compute that, when $m > 0$,
\begin{align*}
& F_{m,\,j}(X) - F_{m-1,\,j}(X) \\
& = (-1)^m\binom{X-m}{j} + \sum_{s=1}^{m-1}(-1)^s\binom{m-1}{s-1}\binom{X-s}{j} \\
& = (-1)^m\binom{X-m}{j} - \sum_{s=0}^{m-2} (-1)^s\binom{m-1}{s}\binom{(X-1)-s}{j} \\
& = (-1)^m\binom{X-m}{j} + (-1)^{m-1}\binom{X-m}{j} - F_{m-1,\,j}(X-1) = -F_{m-1,\,j}(X-1).
\end{align*}

Therefore,
\[
F_{m,\,j}(X) = F_{m-1,\,j}(X) - F_{m-1,\,j}(X-1).
\]
By induction, $F_{m-1,\,j}$ is constant for $0 \leq j < m$, hence $F_{m,\,j}(X) = 0$ in this range. It remains to show that $F_{m,\,m}(X) = 1$.

Using Pascal's identity (\ref{pascal}) again, we compute that, when $m > 0$,
\[
F_{m,\,m}(X+1) - F_{m,\,m}(X) = \sum_{s=0}^m (-1)^s\binom{m}{s}\binom{X-s}{m-1} = F_{m,\,m-1}(X).
\]
We have just shown that $F_{m,\,m-1}(X) = 0$, so $F_{m,\,m}(X+1) - F_{m,\,m}(X) = 0$. Equivalently, $F_{m,\,m}(X)$ is a constant polynomial. In order to compute its value, it suffices to do so at a single point. We will show that $F_{m,\,m}(m) = 1$. Indeed, each term in the defining sum for $F_{m,\,m}(m)$ vanishes except the term with $s = 0$. Thus, 
\[
F_{m,\,m}(m) = (-1)^0\binom{m}{0}\binom{m-0}{m} = 1. \qedhere
\]
\end{proof}

\begin{lemma}
\label{vandermonde}
Let $R$ be a ring $($commutative with identity$)$, $\lambda_1, \dots, \lambda_n \in R$, and $f \geq m \geq 0$ integers. For $1 \leq i \leq n$ and $0 \leq j \leq m$, consider the vectors $\mathbf{v}_{i, j} \in R^n$ defined by
\[
\mathbf{v}_{i, j} := 
\left[\begin{array}{cccc}
\binom{f}{j}\lambda_i^{f-j} & \binom{f+1}{j}\lambda_i^{(f+1)-j} & \dots & \binom{f+(m+1)n-1}{j}\lambda_i^{(f+(m+1)n-1) - j}
\end{array}\right]^T
\]
\[
= \left[\binom{k}{j}\lambda_i^{k-j}\right]^T_{f \leq k < f + (m+1)n}.
\]
Let $\mathbf{M}$ be the $(m+1)n \times (m+1)n$ matrix with columns $\mathbf{v}_{i, j}$ for $1 \leq i \leq n$ and $0 \leq j \leq m$, ordered as follows: $\mathbf{v}_{1, 0}, \mathbf{v}_{1, 1}, \dots, \mathbf{v}_{1, m}, \mathbf{v}_{2,0}, \mathbf{v}_{2, 1}, \dots, \mathbf{v}_{2, m}, \dots, \mathbf{v}_{n, 0}, \mathbf{v}_{n, 1}, \dots, \mathbf{v}_{n, m}$. That is, if $\mathbf{M} = (M_{k, \ell})_{0\leq k, \ell < (m+1)n}$, then writing $\ell = (m+1)q_\ell + r_\ell$ for $0 \leq q_\ell < n$ and $0 \leq r_\ell < m+1$, we have
\begin{equation}
\label{M_kl}
M_{k, \ell} = \binom{k + f}{r_\ell}\lambda_{q_\ell+1}^{k + f - r_\ell}.
\end{equation}
Then we have
\begin{equation}
\label{vandermondedeteqn}
{\rm{det}}(\mathbf{M}) = \left(\prod_{i=1}^n \lambda_i\right)^{f(m+1)} \left(\prod_{1 \leq i < j \leq n} (\lambda_j - \lambda_i)\right)^{(m+1)^2}.
\end{equation}
In particular, if $R = K$ is a field, and $\lambda_1, \dots, \lambda_n \in K^{\times}$ are distinct, then the vectors $\mathbf{v}_{i, j}$ $($$1 \leq i \leq n$, $0 \leq j \leq m$$)$ are linearly independent.
\end{lemma}

\begin{remark}
When $f = m = 0$, the determinant appearing in Lemma \ref{vandermonde} is the classical Vandermonde determinant.
\end{remark}

\begin{proof}
Let $R_k$ denote the $k$th row of $\mathbf{M}$. Perform the following row operations on $\mathbf{M}$:
\begin{equation}
\label{rowops}
R_k \mapsto \sum_{\substack{0 \leq s_1, \dots, s_q \leq m+1 \\ 0 \leq s_{q+1} \leq r }} \binom{r}{s_{q+1}}\left(\prod_{i=1}^q \binom{m+1}{s_i}\right) \left(\prod_{i=1}^{q+1}(-\lambda_i)^{s_i}\right)R_{k-\sum_{i=1}^{q+1}s_i},
\end{equation}
where
\[
k = (m+1)q + r, \mbox{ } 0 \leq q < n,\mbox{ } 0 \leq r < m+1.
\]
Note that each term $R_e$ appearing in the above expression satisfies $0 \leq e \leq k$ and that the term $R_k$ appears on the right side with coefficient $1$. It follows that the above row operations correspond to multiplying the matrix $\mathbf{M}$ on the left by a lower-triangular matrix with $1$'s along the diagonal. These row operations therefore leave the determinant unchanged. Let $\mathbf{M}'$ be the resulting matrix, so that ${\rm{det}}(\mathbf{M}') = {\rm{det}}(\mathbf{M})$. We will show that
\begin{equation}
\label{M'_kluppertriangular}
M'_{k, \ell} = \begin{cases}
0, & \mbox{ if }k > \ell \\
\displaystyle\lambda_{q+1}^f\left(\prod_{i=1}^q(\lambda_{q+1} - \lambda_i)\right)^{m+1}, & \mbox{ if } k = \ell.
\end{cases}
\end{equation}

Equation (\ref{M'_kluppertriangular}) implies that $M'_{k, \ell}$ is upper-triangular, so its determinant is the product of the diagonal entries. Thus,
\[
{\rm{det}}(\mathbf{M}) = {\rm{det}}(\mathbf{M}') = \prod_{0 \leq k < (m+1)n} M'_{k, k} = \prod_{0 \leq k < (m+1)n} \lambda_{q+1}^f \prod_{1 \leq i < q+1} (\lambda_{q+1}-\lambda_i)^{m+1}.
\]
Recall that $k$ and $q$ are related by the formula $k = (m+1)q + r$, where $0 \leq r < m+1$. Thus, as $k$ ranges over the values $0 \leq k < (m+1)n$, $q+1$ ranges over each of the values $1, \dots, n$ exactly $m+1$ times. We therefore see that
\[
{\rm{det}}(\mathbf{M}) = \left(\prod_{q=0}^{n-1} \lambda_{q+1}^f \prod_{1 \leq i < q+1} (\lambda_{q+1}-\lambda_i)^{m+1}\right)^{m+1} = \left(\prod_{i=1}^n \lambda_i\right)^{f(m+1)} \left(\prod_{1 \leq i < j \leq n} (\lambda_j - \lambda_i)\right)^{(m+1)^2},
\]
which completes the proof of the lemma, modulo the proof of (\ref{M'_kluppertriangular}).

We now prove (\ref{M'_kluppertriangular}). Substituting the formula (\ref{M_kl}) into (\ref{rowops}), we see that
\begin{equation}
\label{M'_kl}
M'_{k, \ell} = \sum_{\substack{0 \leq s_1, \dots, s_q \leq m+1 \\ 0 \leq s_{q+1} \leq r }} \binom{r}{s_{q+1}}\left(\prod_{i=1}^q \binom{m+1}{s_i}\right) \left(\prod_{i=1}^{q+1}(-\lambda_i)^{s_i}\right)\binom{k- \sum_{i=1}^{q+1} s_i + f}{r_\ell}\lambda_{q_\ell+1}^{k-\sum_{i=1}^{q+1}s_i+f-r_{\ell}}.
\end{equation}
If $k \geq \ell$, then $q \geq q_\ell$. We first consider the case $q > q_\ell$. Then we may rewrite (\ref{M'_kl})  by first summing over $s_{q_\ell+1}$ and then summing over all of the other $s_i$:
\begin{align*}
\nonumber M'_{k, \ell} & = \sum_{\substack{0 \leq s_1, \dots, \widehat{s_{q_\ell +1}}, \dots, s_q \leq m+1, \dots, q\\ 0 \leq s_{q+1} \leq r}} \binom{r}{s_{q+1}}\left(\prod_{\substack{1 \leq i \leq q \\ i \neq q_\ell+1}}\binom{m+1}{s_i}\right)\left(\prod_{\substack{1 \leq i \leq q+1\\ i \neq q_\ell+1}}(-\lambda_i)^{s_i} \right) \\
& \times \lambda_{q_\ell+1}^{k -\sum_{\substack{1 \leq i \leq q+1 \\ i \neq q_\ell+1}} s_i + f - r_\ell}\sum_{s_{q_\ell+1} = 0}^{m+1} (-1)^{s_{q_\ell+1}}\binom{m+1}{s_{q_\ell+1}}\binom{(k - \sum_{\substack{1 \leq i \leq q+1 \\ i \neq q_\ell + 1}}s_i +f) - s_{q_\ell+1}}{r_\ell}.
\end{align*}
The inner sum over $s_{q_\ell+1}$ vanishes by Lemma \ref{combinatorialid} applied with $m$ replaced by $m+1$, $j = r_\ell < m+1$, and $X = k - \sum_{\substack{1 \leq i \leq q+1 \\ i \neq q_\ell + 1}} s_i +f$. Thus, 
\[
M'_{k, \ell} = 0 \mbox{ if } q > q_\ell.
\]
Next we consider the case $q = q_\ell$. Then we rewrite (\ref{M'_kl}) by first summing over $s_{q+1}$ and then summing over all of the other $s_i$:
\begin{align}
\label{M'_kleqn3}
\nonumber M'_{k, \ell} &= \lambda_{q+1}^{k - (m+1)q - r_\ell + f}\sum_{0 \leq s_1, \dots, s_q \leq m+1}\left(\prod_{i=1}^q\binom{m+1}{s_i}\right)\left(\prod_{i=1}^q(-\lambda_i)^{s_i}\right) \\
& \times \lambda_{q+1}^{(m+1)q-\sum_{i=1}^q s_i}\sum_{s_{q+1} = 0}^r(-1)^{s_{q+1}}\binom{r}{s_{q+1}}\binom{(k - \sum_{i=1}^q s_i + f) - s_{q+1}}{r_\ell}.
\end{align}
If $k > \ell$, then (since $q = q_\ell$) we have $r > r_\ell$, while if $k = \ell$, then $r = r_\ell$. Therefore, applying Lemma \ref{combinatorialid} with $m = r$, $j = r_\ell$, and $X = k - \sum_{i=1}^qs_i + f$ shows that the inner sum over $s_{q+1}$ vanishes if $k > \ell$ -- thereby completing the proof of (\ref{M'_kluppertriangular}) for $k > \ell$ -- and that if $k = \ell$, then the inner sum equals $1$. Thus, if $k = \ell$, then
\begin{align}
\nonumber M'_{k, k} &= \lambda_{q+1}^{k - (m+1)q - r + f}\sum_{0 \leq s_1, \dots, s_q \leq m+1}\left(\prod_{i=1}^q\binom{m+1}{s_i}\right)\left(\prod_{i=1}^q(-\lambda_i)^{s_i}\right) \lambda_{q+1}^{(m+1)q-\sum_{i=1}^q s_i} \\
&= \lambda_{q+1}^f\sum_{0 \leq s_1, \dots, s_q \leq m+1}\left(\prod_{i=1}^q\binom{m+1}{s_i}\right)\left(\prod_{i=1}^q(-\lambda_i)^{s_i}\right) \lambda_{q+1}^{(m+1)q-\sum_{i=1}^q s_i},
\end{align}
where in the second equality we have used the fact that $k - (m+1)q - r = 0$. We finally compute that
\[
M'_{k, k} = \lambda_{q+1}^f \prod_{i=1}^q \left(\sum_{s =0}^{m+1} \binom{m+1}{s}(-\lambda_i)^s\lambda_{q+1}^{m+1-s}\right) = \lambda_{q+1}^f\prod_{i=1}^q(\lambda_{q+1} - \lambda_i)^{m+1}.
\]
This completes the proof of (\ref{M'_kluppertriangular}) and of the lemma.
\end{proof}

The utility of Lemma \ref{vandermonde} for us lies in the following application.

\begin{lemma}
\label{absvalofsum}
Let $A$ be a rank-$1$ valuation ring with fraction field $K$, and absolute value $|\cdot|\colon K \rightarrow \mathbf{R}_{\geq 0}$. Let $\mu_1, \dots, \mu_t \in A^{\times}$ be such that $\mu_i - \mu_j \in A^{\times}$ for all pairs of distinct $i, j \in \{1, \dots, t\}$. Let $m \geq 0$ be a nonnegative integer. Finally, let $c_{i, j} \in K$ for $1 \leq i \leq t$ and $0 \leq j \leq m$. Then, for any integer $r \geq m$, there is an integer $s = s(r) \in [r, r + (m+1)n)$ such that
\[
\left|\sum_{\substack{1 \leq i \leq t \\ 0 \leq j \leq m}} c_{i, j}\binom{s}{j}\mu_i^{s-j}\right| = \max_{\substack{1 \leq i \leq t \\ 0 \leq j \leq m}} \{|c_{i, j}|\}.
\]
\end{lemma}

\begin{proof}
We may assume that $$M := \max_{\substack{1 \leq i \leq t \\ 0 \leq j \leq m}} \{|c_{i, j}|\}$$ is nonzero, and, dividing through by some $c_{i_0, j_0}$, we may assume that $M = 1$. Let $\mathfrak{m}$ be the maximal ideal of $A$, and let $k := A/\mathfrak{m}$ be the residue field. Then reducing modulo $\mathfrak{m}$, the lemma translates to the following statement: if $\overline{\mu}_1, \dots, \overline{\mu}_t \in k^{\times}$ are distinct, and $\overline{c}_{i, j} \in k$ are not all $0$, then there is an integer $s = s(r) \in [r, r + (m+1)n)$ such that
\[
\sum_{\substack{1 \leq i \leq t \\ 0 \leq j \leq m}} \overline{c}_{i, j}\binom{s}{j}\overline{\mu}_i^{s-j} \neq 0.
\]
If we consider the statement of Lemma \ref{vandermonde} with $n = t$, $f = r$, and $\lambda_i = \overline{\mu}_i$ (and $m = m$), then this is exactly the statement that the vectors $v_{i, j}$ are linearly independent over $k$, which is part of the statement of that lemma.
\end{proof}

\begin{lemma}
\label{sumsofpowersincsubsequence}
Let $A$ be a rank-$1$ valuation ring with fraction field $K$ and absolute value $|\cdot|\colon K \rightarrow \mathbf{R}_{\geq 0}$. Let $\mu_1, \dots, \mu_n \in A$ be distinct elements, and let $x_1, \dots, x_n \in K$ be such that, for some $i$, we have $\mu_i \in A^{\times}$ and $x_i \neq 0$. Then there is an integer $g \geq 0$ and a strictly increasing sequence $\{r_e\}_{e \geq 0}$ of nonnegative integers such that all the elements of the sequence
\[
\left|\sum_{i=1}^n \mu_i^{r_e}x_i\right|_{e \geq 0}
\]
are equal to the nonzero constant $\sup_{r \geq g} \{\left|\sum_{i=1}^n \mu_i^{r}x_i\right|\}$.
\end{lemma}

\begin{remark}
The proof of Lemma \ref{sumsofpowersincsubsequence} given below in fact shows that we may choose the sequence $\{r_e\}$ so that the differences $r_{e+1} - r_e$ between consecutive terms are bounded. We will never use this.
\end{remark}

\begin{proof}
Without loss of generality, we may assume that $|\mu_1| = 1$. We may also renumber the $\mu_i$ so that $\mu_1, \mu_2, \dots, \mu_t, \dots, \mu_{s}, \dots, \mu_n$ satisfy the following properties:
\begin{itemize}
\item[(i)] $|\mu_\ell| = 1$ for $1 \leq \ell \leq s$;
\item[(ii)] $|\mu_\ell| < 1$ for $s < \ell \leq n$;
\item[(iii)] For each $1 \leq \ell \leq s$, there is a unique $1 \leq i(\ell) \leq t$ such that $\mu_i \equiv \mu_{i(\ell)} \pmod{\mathfrak{m}}$, where $\mathfrak{m} \subset A$ is the maximal ideal.
\end{itemize}

For each $1 \leq \ell \leq s$, write $\mu_\ell = \mu_{i(\ell)} + \epsilon_\ell$, where $\epsilon_\ell \in \mathfrak{m}$. Then for any $r, m \geq 0$, we have
\begin{align}
\label{sumsofpowersincsubsequencepfeqn4}
\nonumber \sum_{\ell=1}^n \mu_\ell^rx_\ell & = \sum_{i=1}^t \sum_{i(\ell) = i} \mu_\ell^rx_\ell + \sum_{\ell > s} \mu_\ell^rx_\ell \\
&= \sum_{i=1}^t \sum_{j = 0}^m\left(\sum_{i(\ell)=i} \epsilon_\ell^j x_\ell\right) \binom{r}{j}\mu_i^{r-j} + \sum_{i=1}^t \sum_{j > m}\left(\sum_{i(\ell)=i} \epsilon_\ell^j x_\ell \right) \binom{r}{j}\mu_i^{r-j} + \sum_{\ell > s} \mu_\ell^rx_\ell.
\end{align}
Each term in the second sum above has absolute value $\leq \max_{1 \leq \ell \leq s} \{|\epsilon_\ell^{m+1}x_\ell| \}$, hence the same holds for the entire sum. For $1 \leq i \leq t$ and $j \geq 0$, let
\[
c_{i, j} := \sum_{i(\ell) = i} \epsilon_\ell^j x_\ell.
\]
We claim that $c_{i, j} \neq 0$ for some $i, j$. Indeed, otherwise the first two sums on the right side of (\ref{sumsofpowersincsubsequencepfeqn4}) would vanish, so that (\ref{sumsofpowersincsubsequencepfeqn4}) would simplify to
\[
\sum_{\ell = 1}^s \mu_\ell^rx_\ell = 0
\]
for all $r \geq 0$. Since $x_i \neq 0$ for some $1 \leq i \leq s$ by assumption, this would violate the nonvanishing of the classical Vandermonde determinant ${\rm{det}}(\mu_i^j)_{\substack{1 \leq i \leq s\\ 0 \leq j < s}}$, since the $\mu_i$ are distinct. Thus, some $c_{i, j}$ is nonzero.

Putting everything together, (\ref{sumsofpowersincsubsequencepfeqn4}) becomes, for every $r, m \geq 0$,
\begin{equation}
\label{combofpowersvalringpfeqn4}
\sum_{\ell=1}^n {\mu_\ell}^rx_\ell = \sum_{i=1}^t \sum_{j = 0}^mc_{i, j} \binom{r}{j}\mu_i^{r-j} + \sum_{\ell > s} \mu_\ell^rx_\ell + C_m,
\end{equation}
where 
\begin{equation}
\label{sumsofpowersincsubsequencepfeqn5}
|C_m| \leq \max_{1 \leq \ell \leq s} \{|\epsilon_\ell^{m+1}x_\ell| \}.
\end{equation}

Since some $c_{i, j} \neq 0$, the non-decreasing sequence
\[
M(m) := \max_{\substack{1 \leq i \leq t \\ 0 \leq j \leq m}}\{|c_{i, j}|\}
\]
is nonzero for all sufficiently large $m \geq 0$. We claim that it is even eventually constant, equal to $M$, say:
\[
M = \max_{\substack{1 \leq i \leq t \\ j \geq 0}}\{|c_{i, j}|\}
\] 
Indeed, this follows from the fact that $c_{i, j} \rightarrow 0$ as $j \rightarrow \infty$, since $\epsilon_\ell \in \mathfrak{m}$.

Since each $\epsilon_\ell \in \mathfrak{m}$, the maximum in (\ref{sumsofpowersincsubsequencepfeqn5}) tends to $0$ as $m \rightarrow \infty$. We may therefore choose a fixed integer $m$ so that
\begin{equation}
\label{sumsofpowersincsubsequencepfeqn6}
|C_m| < \max_{\substack{1 \leq i \leq t \\ 0 \leq j \leq m}}\{|c_{i, j}|\} = M.
\end{equation}
Next, since $|\mu_\ell| < 1$ for $s < \ell \leq n$, we may choose an $r' \geq m$ such that
\begin{equation}
\label{sumsofpowersincsubsequencepfeqn7}
\left|\sum_{\ell > s} \mu_\ell^rx_\ell \right| < M \mbox{ for } r \geq r'.
\end{equation}

By construction, for $1 \leq i \leq t$, we have $\mu_i \in A^{\times}$, and these $\mu_i$ are distinct modulo $\mathfrak{m}$. Lemma \ref{absvalofsum} therefore implies that for any $r \geq r' \geq m$, there is an integer $f = f(r) \in [r, r + (m+1)n)$ such that
\begin{equation}
\label{combofpowersvalringpfeqn8}
\left|\sum_{i=1}^t \sum_{j = 0}^mc_{i, j} \binom{f}{j}\mu_i^{f-j}\right| = \max_{\substack{1 \leq i \leq t \\ 0 \leq j \leq m}}\{|c_{i, j}|\} \neq 0.
\end{equation}
Combining (\ref{combofpowersvalringpfeqn4}), (\ref{combofpowersvalringpfeqn8}), (\ref{sumsofpowersincsubsequencepfeqn6}), and (\ref{sumsofpowersincsubsequencepfeqn7}), we find that, for all $r \geq r'$, there exists an integer $f = f(r) \in [r, r + (m+1)n)$ such that
\[
\left|\sum_{\ell = 1}^n \mu_\ell^fx_\ell \right| = \max_{\substack{1 \leq i \leq t \\ 0 \leq j \leq m}}\{|c_{i, j}|\} \neq 0.
\]
Since we may choose such $f$ for any $r \geq r'$, it follows that there is a strictly increasing sequence $\{r_e\}_{e \geq 0}$ of nonnegative integers such that, for all $e \geq 0$,
\[
\left|\sum_{\ell = 1}^n \mu_\ell^{r_e}x_\ell \right| = M = \max_{\substack{1 \leq i \leq t \\ j \geq 0}}\{|c_{i, j}|\} \neq 0.
\]
The right side is constant. We claim that it equals $\sup_{r \geq r_0} \{\left|\sum_{i=1}^n \mu_i^{r}x_i\right|\}$, which will prove the lemma. Indeed, this follows by combining (\ref{combofpowersvalringpfeqn4}), (\ref{sumsofpowersincsubsequencepfeqn7}), and (\ref{sumsofpowersincsubsequencepfeqn6}).
\end{proof}

\begin{lemma}
\label{combofpowersvalring}
Let $K$ be a real-valued field with absolute value $|\cdot|\colon K \rightarrow \mathbf{R}_{\geq 0}$. Suppose given $\lambda_1, \dots, \lambda_n, x_1, \dots, x_n \in K$, and $\alpha > 0$. Suppose also that there is a strictly increasing sequence $\{s_j\}_{j \geq 0}$ of nonnegative integers such that $|\sum_{i=1}^n\lambda_i^{s_j}x_i| > \alpha$ for all $j \geq 0$. Then there exist a strictly increasing sequence $\{r_j\}_{j\geq 0}$ of nonnegative integers and $\beta > \alpha$ such that $|\sum_{i=1}^n\lambda_i^{r_j}x_i| \geq \beta$ for all $j \geq 0$.
\end{lemma}

Lemma \ref{combofpowersvalring} is trivial if the valuation is discrete. The interesting case is when the value group is dense in $\mathbf{R}_{> 0}$, or equivalently, when $A$ is non-Noetherian.

\begin{proof}
By combining terms, we may assume that the $\lambda_i$ are distinct, and we may also assume that all of the $x_i$ are nonzero. In addition, the hypotheses imply that some $\lambda_i$ is nonzero. We may renumber and thereby assume that $|\lambda_1| = \max\{|\lambda_i|\}$. For $1 \leq i \leq n$, let $\mu_i := \lambda_i/\lambda_1$. Then the $\mu_i$ and $x_i$ satisfy the hypotheses of Lemma \ref{sumsofpowersincsubsequence}. There are therefore an integer $g \geq 0$ and a strictly increasing sequence $\{r_j\}_{j \geq 0}$ of nonnegative integers such that
\[
\left| \sum_{i=1}^n \mu_i^{r_j}x_i \right| = \sup_{r \geq g} \left\{\left| \sum_{i=1}^n \mu_i^{r}x_i \right|\right\} > 0 \mbox{ for all } j \geq 0.
\]
Now we break the proof of the lemma up into three cases: $|\lambda_1| < 1$, $|\lambda_1| = 1$, and $|\lambda_1| > 1$.
\\

\noindent \underline{$|\lambda_1| < 1$}: If $|\lambda_1| < 1$, then $|\sum_{i=1}^n \lambda_i^rx_i| \rightarrow 0$ as $r \rightarrow \infty$, in violation of the hypotheses of the lemma, so this is impossible.
\\

\noindent \underline{$|\lambda_1| = 1$}: In this case,
\[
\left|\sum_{i=1}^n \lambda_i^rx_i \right| = \left| \lambda_1^r\sum_{i=1}^n \mu_i^rx_i \right| = \left| \sum_{i=1}^n \mu_i^rx_i \right|.
\]
Thus, the sequence $|\sum_{i=1}^n \lambda_i^{r_j}x_i|$ is constant and equal to the supremum of $|\sum_i \lambda_i^rx_i|$ over $r \geq g$. In particular, each term is $\geq |\sum_i \lambda_i^{r_m}x_i| > \alpha$ for all sufficiently large $m$. Thus, we have a constant sequence of terms $ > \alpha$. It follows that the sequence is bounded away from $\alpha$.
\\

\noindent \underline{$|\lambda_1| > 1$}: We have
\[
\left|\sum_{i=1}^n \lambda_i^rx_i \right| = \left| \lambda_1^r\sum_{i=1}^n \mu_i^rx_i \right|.
\]
Since the sequence $\{|\sum_{i=1}^n \mu_i^{r_j}x_i|\}_{j \geq 0}$ is a positive constant, it follows that the expression on the right side of the above equation tends to $\infty$ for $r = r_j$ as $j \rightarrow \infty$. In particular, for any $\beta > \alpha$, it is eventually greater than $\beta$.
\end{proof}

For a field $K$ equipped with a real-valued absolute value $|\cdot|\colon K \rightarrow \mathbf{R}_{\geq 0}$, we may in the usual manner make $K$ into a topological field (and its valuation subring into a topological ring) by declaring a base of neighborhoods of $0$ to be the sets $\{x \in K \mid |x| < \epsilon\} \subset K$ for $\epsilon > 0$. Then we may topologize $K^n$ by endowing it with the product topology. These are the topologies used in the following lemma.

\begin{lemma}
\label{nonvanishingdense}
Let $K$ be a valued field with a nontrivial absolute value $|\cdot|\colon K \rightarrow \mathbf{R}_{\geq 0}$ $($nontrivial means that, for some $x \in K^{\times}$, one has $|x| \neq 1$$)$, and let $0 \neq F \in K[X_1, \dots, X_n]$. Then the set
\[
\{(\alpha_1, \dots, \alpha_n) \in K^n \mid F(\alpha_1, \dots, \alpha_n) \neq 0\}
\]
is dense in $K^n$.
\end{lemma}

\begin{proof}
Suppose that the set is not dense. This means that $F$ vanishes on some nonempty open subset $U \subset K^n$. By translating, we may assume that $0 \in U$. That is, there is $\epsilon > 0$ such that $F(\alpha_1, \dots, \alpha_n) = 0$ whenever $\alpha_i \in K^n$ satisfy $|\alpha_i| < \epsilon$ for all $i$. Since the valuation on $K$ is nontrivial, there exists $\beta \in K^{\times}$ such that $|\beta| < 1$. Replacing $\beta$ with $\beta^n$ for large $n$, we find that there exists $\beta \in K^{\times}$ with $|\beta| < \epsilon$. Then replacing $F$ by $F(\beta X_1, \dots, \beta X_n)$, we may assume that $F$ vanishes on $A^n$, where $A := \{a \in K \mid |a| \leq 1\}$ is the valuation ring of $K$. We must show that this implies that $F = 0$.

The proof is the same as the standard proof that a nonzero polynomial cannot vanish on all of $K^n$ if $K$ is infinite, but we give it again here for the reader's convenience. We prove by induction on $n$ that a polynomial $G \in K[X_1, \dots, X_n]$ vanishing on $A^n$ is the zero polynomial. We first note that the fact that $K$ is nontrivially valued implies that $K$, hence also $A$, is infinite. In the base case $n = 1$, the result follows from this infinitude. For $n > 1$, assume that the result is true for $n - 1$, and we prove it for $n$. Write
\[
G(X_1, \dots, X_n) = \sum_{i=0}^d G_i(X_1, \dots, X_{n-1}) X_n^i.
\]
For each $(\alpha_1, \dots, \alpha_{n-1}) \in K^{n-1}$, the polynomial $G(\alpha_1, \dots, \alpha_{n-1}, X) \in K[X]$ vanishes on $A$, hence is the zero polynomial by the case $n = 1$. It follows that each $G_i$ vanishes on $A^{n-1}$, hence is the zero polynomial by induction. Hence $G = 0$.
\end{proof}

Using the topology on $K^m$ discussed earlier, we may topologize the set ${\rm{M}}_n(K)$ of $n \times n$ matrices on $K$ by identifying it with $K^{n^2}$.

\begin{lemma}
\label{diagonalizablematricesdense}
Let $K$ be a real-valued field with nontrivial valuation $|\cdot|\colon K \rightarrow \mathbf{R}_{\geq 0}$. Then the geometrically diagonalizable matrices $($i.e., those diagonalizable over an algebraic closure $\overline{K}$ of $K$$)$ are dense in ${\rm{M}}_n(K)$.
\end{lemma}

\begin{proof}
A sufficient condition for geometric diagonalizability is that the characteristic polynomial have no repeated roots, or equivalently, that the discriminant of the characteristic polynomial be nonzero. The discriminant $D$ of the characteristic polynomial is a polynomial function on ${\rm{M}}_n(K) = K^{n^2}$. Further, it is not identically $0$. For example, since $K$ is infinite (because it is nontrivially valued), we may choose $\alpha_1, \dots, \alpha_n \in K$ distinct, and then consider the diagonal matrix with these diagonal entries. Thus, by Lemma \ref{nonvanishingdense}, the complement of the vanishing locus of $D$ is dense in ${\rm{M}}_n(K)$. Since this complement consists entirely of geometrically diagonalizable matrices, the proof is complete.
\end{proof}

For a matrix $M = (a_{ij})_{i, j}$ over a field $K$ equipped with a real absolute value $|\cdot|$, define $|M| := \max_{i, j}\{|a_{ij}|\}$. In particular, for a vector $\mathbf{v} \in K^n$, $|\mathbf{v}| := \max_i\{|v_i|\}$. We now prove a matrix analogue of Lemma \ref{combofpowersvalring}, which says that the set of absolute values of the iterates of a vector under a matrix attains its infimum, provided that this inf is $> 0$.

\begin{lemma}
\label{matrixpowers}
Let $A$ be a real-valued valuation ring with fraction field $K$ and absolute value $|\cdot|\colon K \rightarrow \mathbf{R}_{\geq 0}$, and let $M \in {\rm{M}}_n(A), \mathbf{v} \in K^n$. Suppose that
\[
\alpha := \inf_{r \geq 0} |M^r\mathbf{v}| > 0.
\]
Then there exists an $r_0 \geq 0$ such that $|M^{r_0}\mathbf{v}| = \alpha$.
\end{lemma}

\begin{proof}
If the valuation on $K$ is trivial, then the lemma is obvious. We may therefore assume that the valuation is nontrivial. The valuation on $K$ may be extended to a real-valued valuation on $\overline{K}$, by \cite[Chap.\,4, Th.\,1]{ribenboim}. (Note that, in the statement of that result, by ``valued field'' the author means ``real-valued field.'' See \cite[Chap.\,2, Definition 1]{ribenboim}.) We may therefore assume that $K$ is algebraically closed. Then, by Lemma \ref{diagonalizablematricesdense}, the set of diagonalizable (over $K = \overline{K}$) matrices is dense in ${\rm{M}}_n(K)$.

Choose an $\epsilon > 0$ such that
\begin{equation}
\label{matrixpowerspfeqn2}
\epsilon|\mathbf{v}| < \alpha,
\end{equation}
and let $M' \in {\rm{M}}_n(A)$ be a diagonalizable matrix such that $|M - M'| < \epsilon$. Then we claim that
\begin{equation}
\label{matrixpowerspfeqn1}
|M'^r - M^r| < \epsilon \mbox{ for all } r \geq 0.
\end{equation}
We proceed by induction on $r$, the case $r = 0$ being trivial. If (\ref{matrixpowerspfeqn1}) is true for $r$, then $M' = M+ E$ and $M'^r = M^r + E_r$, where $|E|, |E_r| < \epsilon$. Therefore,
\[
|M'^{r+1} - M^{r+1}| = |(M+E)(M^r+E_r) - M^{r+1}| = |ME_r + EM^r + EE_r| < \epsilon,
\]
since $M, M^r \in {\rm{M}}_n(A)$. This proves (\ref{matrixpowerspfeqn1}). Using (\ref{matrixpowerspfeqn2}), it follows that
\[
|M'^r\mathbf{v} - M^r\mathbf{v}| \leq \epsilon|\mathbf{v}| < \alpha \mbox{ for all } r \geq 0.
\]
Therefore, in order to prove the lemma for $M$, it suffices to prove it for $M'$. We may therefore assume that $M$ is diagonalizable.

We have
\[
M = Q\cdot{\rm{diag}}(\lambda_1, \dots, \lambda_n)\cdot Q^{-1}
\]
for some $Q \in {\rm{GL}}_n(K)$. Then
\begin{equation}
\label{matrixpowerspfeqn5}
M^rv = Q\cdot {\rm{diag}}(\lambda_1^r, \dots, \lambda_n^r)\cdot Q^{-1} v = \left[\sum_{i=1}^n \lambda_i^rc_{ij}\right]_{1 \leq j \leq n}^T
\end{equation}
for some constants $c_{ij} \in K$. Assume for the sake of contradiction that $|M^r\mathbf{v}| > \alpha$ for all $r \geq 0$. Then there is a $1 \leq j_0 \leq n$ such that there is a strictly increasing sequence $\{s_m\}_{m \geq 0}$ of nonnegative integers satisfying
\[
\left| \sum_{i=1}^n\lambda_i^{s_m}c_{ij_0} \right| > \alpha \mbox{ for all } m \geq 0.
\]
Lemma \ref{combofpowersvalring} then furnishes a $\beta > \alpha$ and a strictly increasing sequence $\{r_m\}_{m \geq 0}$ of nonnegative integers such that 
\[
\left| \sum_{i=1}^n\lambda_i^{r_m}c_{ij_0} \right| \geq \beta \mbox{ for all } m \geq 0.
\]
Together with (\ref{matrixpowerspfeqn5}), this implies that $|M^{r_m}\mathbf{v}| \geq \beta$. But since $M \in {\rm{M}}_n(A)$, we have $|M\mathbf{w}| \leq |\mathbf{w}|$ for every $\mathbf{w} \in A^n$. The sequence $|M^r\mathbf{v}|$ is therefore non-increasing, hence the existence of a sequence $\{r_m\}$ as above implies that in fact $|M^r\mathbf{v}| \geq \beta > \alpha$ for all $r \geq 0$, contradicting the fact that $\alpha = \inf_{r \geq 0} |M^r\mathbf{v}|$. This contradiction proves the lemma.
\end{proof}

\section{Smooth schemes over valuation rings}

In this section we prove a useful decomposition result for elements of local rings of smooth schemes over valuation rings (Proposition \ref{localringovervalringprop}). The proof depends crucially upon Lemma \ref{matrixpowers}. We begin with a simple divisibility lemma.

\begin{lemma}
\label{primeidealdivisibility}
Let $A$ be a valuation ring, $\mathfrak{p} \subset A$ a prime ideal. Then for every $a \in \mathfrak{p}$, and every $x \in A - \mathfrak{p}$, we have $a = xy$ for some $y \in \mathfrak{p}$.
\end{lemma}

\begin{proof}
If $a\mid x$, then $x \in \mathfrak{p}$, contrary to hypothesis. Since $A$ is a valuation ring, it follows that $x \mid a$, so that $a = xy$ for some $y \in A$. Since $xy \in \mathfrak{p}$ and $x \notin \mathfrak{p}$, it follows that $y \in \mathfrak{p}$.
\end{proof}

We now come to the main result of this section.

\begin{proposition}
\label{localringovervalringprop}
Let $A$ be a valuation ring with maximal ideal $\mathfrak{m}$, and let $B$ be a localization of a smooth $A$-algebra such that the special fiber of ${\rm{Spec}}(B) \rightarrow {\rm{Spec}}(A)$ is nonempty. Then every $b \in B$ may be written as a product $b = ab'$ for some $a \in A$ and $b' \in B - \mathfrak{m}B$.
\end{proposition}

The utility of Proposition \ref{localringovervalringprop} is that it often allows one to reduce questions about the local structure of smooth schemes over valuation rings to the generic fiber, which is a smooth scheme over a field. See especially Proposition \ref{genericunitalmost=unit}.

\begin{proof}
\underline{Step 1}: We first treat the easy case in which $B$ is a localization of a polynomial ring $A[X_1, \dots, X_n]$. So let $S \subset A[X_1, \dots, X_n]$ be a multiplicative subset disjoint from $\mathfrak{m}\cdot A[X_1, \dots, X_n]$, and let $B := S^{-1}\cdot A[X_1, \dots, X_n]$. Since $A$ is a valuation ring, any element of $A[X_1, \dots, X_n]$ may be written in the form $a\cdot F$, where $a \in A$ and $F \in A[X_1, \dots, X_n] - \mathfrak{m}\cdot A[X_1, \dots, X_n]$. Thus, any $b \in B$ may be written as a product $a\cdot F/s$, where $a \in A$, $F \in A[X_1, \dots, X_n] - \mathfrak{m}\cdot A[X_1, \dots, X_n]$, and $s \in S$. We claim that $F/s \notin \mathfrak{m}B$. Indeed, otherwise we would have $F/s = xH/s'$ for some $x \in \mathfrak{m}$, $H \in A[X_1, \dots, X_n]$, and $s' \in S$. We then obtain that $Fs' = xHs \in \mathfrak{m}\cdot A[X_1, \dots, X_n]$, but $F, s' \notin \mathfrak{m}\cdot A[X_1, \dots, X_n]$. This is impossible because $\mathfrak{m}\cdot A[X_1, \dots, X_n] \subset A[X_1, \dots, X_n]$ is a prime ideal (since the quotient by it is $k[X_1, \dots, X_n]$, where $k := A/\mathfrak{m}$). This completes the proof for localizations of polynomial rings.
\\

\noindent \underline{Step 2}: Now we prove the proposition in the key case when $A$ is a valuation ring of height $\leq 1$ -- that is, when $A$ is the valuation ring associated to a real absolute value $|\cdot|\colon K \rightarrow \mathbf{R}_{\geq 0}$, where $K := {\rm{Frac}}(A)$. We first treat the special case in which $B$ is standard \'etale domain over a polynomial ring $A[X_1, \dots, X_n]$ and $\mathfrak{m}B \subset B$ is a prime ideal. That is, $B = B'_g$ is a domain, where $B' := C[X]/(F)$, $C := A[X_1, \dots, X_n]$, $F \in C[X]$ is monic, $g \in B'$, and $C \rightarrow B$ is \'etale, and $\mathfrak{m}B \subset B$ is a prime ideal. Note in particular that $B'$ is a finite free $C$-algebra, and that, by Step 1, the proposition is true for $C$. Further, since every element of $B$ is the product of a unit and an element in the image of the map $i: B' \rightarrow B$, it suffices to show that every element $i(b) \neq 0$ with $b \in B'$ may be written as the product of an element of $A$ and an element of $B - \mathfrak{m}B$.

Note that the proposition for the ring $C$ implies that the local ring $C_{\mathfrak{m}C}$ is a valuation ring with the same value group as $A$. More precisely, if $c \in C$ and $c = ac'$ with $a \in A$, $c' \in C - \mathfrak{m}C$, then $|c| = |a|$. We therefore see that $C_{\mathfrak{m}C}$ is a height-$1$ valuation ring. Recall that $B'$ is a finite free $C$-module. Let $e_1, \dots, e_n \in B'$ be a $C$-basis. Then for each $r \geq 0$, write
\[
g^rb = \sum_{i=1}^n c_{ir}e_i, \hspace{.1 in} c_{ir} \in C.
\]
Let
\[
\alpha := \inf_{r \geq 0} \{\max_{1 \leq i \leq n} |c_{ir}|\},
\]
where we consider the $c_{ir}$ as elements of $C_{\mathfrak{m}C}$ so as to make sense of the above absolute values. If $i(b) \neq 0$, then we claim that $\alpha > 0$. Indeed, \cite[Tag 0ASJ]{stacks} implies that $B_{\mathfrak{m}B}$ is a valuation ring, and that the map $C_{\mathfrak{m}C} \rightarrow B_{\mathfrak{m}B}$ is an extension of valuation rings which induces an isomorphism of value groups. If $\alpha = 0$, then it follows that $|i(b)| = 0$, hence $i(b) = 0 \in B_{\mathfrak{m}B}$, hence $i(b) = 0$ because $B$ is an integral domain by hypothesis. We may therefore assume that $\alpha > 0$.

Let $M$ be the matrix of the $C$-linear map $B' \rightarrow B'$, $b' \mapsto gb'$, with respect to the basis $\{e_i\}$. Let $\mathbf{v} := [c_{i0}]_{1 \leq i \leq n}$. Then $\inf_{r \geq 0}|M^r\mathbf{v}| = \alpha > 0$. Lemma \ref{matrixpowers} (applied to the valuation ring $C_{\mathfrak{m}C}$) implies that the infimum $\alpha$ is attained. That is, there is $r_0 \geq 0$ such that
\[
\max_{1 \leq i \leq n} |c_{ir_0}| = \alpha = \inf_{r \geq 0} \{\max_{1 \leq i \leq n} |c_{ir}|\}.
\]
Renumbering, we may assume that $|c_{1r_0}| = \max_{1 \leq i \leq n} |c_{ir_0}|$. Writing $c_{ir_0} = a_ic_i$ for some $a_i \in A$ and $c_i \in C - \mathfrak{m}C$ via Step 1, we then have $a_1 \mid a_i$ for all $1 \leq i \leq n$. Therefore, $a_1 \mid g^{r_0}b$ in $B'$, hence $a_1 \mid b$ in $B = B'_g$. Write $b = a_1b'$ with $b' \in B$. Then we claim that $b' \notin \mathfrak{m}B$, which will complete the proof of Step 2.

Indeed, suppose for the sake of contradiction that $b' \in \mathfrak{m}B$. This means that there exist $r\geq 0$, $a' \in \mathfrak{m}$, and $b'' \in B'$ such that $g^rb' = a'b''$ in $B'$. Multiplying by $a_1$, we obtain $g^rb = a_1a'b''$. Writing $b''$ in terms of the $C$-basis $\{e_i\}$ for $B'$, this implies that
\[
c_{ir} \in a_1\cdot\mathfrak{m}C, \hspace{.1 in} 1 \leq i \leq n.
\]
Therefore, 
\[
|c_{ir}| < |a_1| = \alpha = \inf_{r \geq 0} \{\max_{1 \leq i \leq n} |c_{ir}|\},
\]
for all $1 \leq i \leq n$, a contradiction. This completes the proof of Step 2 when $B$ is standard \'etale over a polynomial ring.

Now we prove Step 2 in general. Choose $x \in {\rm{Spec}}(B)$ in the special fiber above ${\rm{Spec}}(A)$. Choose a neighborhood ${\rm{Spec}}(B_x)$ of $x$ such that $B_x$ and $B_x/\mathfrak{m}B_x$ are integral domains (possible because $B_x/\mathfrak{m}B_x$ is smooth over $k := A/\mathfrak{m}$ and $B_x$ is smooth over the normal domain $A$). Choose a standard \'etale algebra $B'_x$ over $A[X_1, \dots, X_n]$ of which $B_x$ is a localization. Again, since $B'_x/\mathfrak{m}B'_x$ is smooth over $k$, we may replace $B'_x$ by an open neighborhood of $x \in {\rm{Spec}}(B'_x)$ which is still standard \'etale, and thereby ensure that $B'_x/\mathfrak{m}B'_x$ is an integral domain. In particular, the map 
\begin{equation}
\label{localringovervalringproppfeqn31}
B'_x/\mathfrak{m}B'_x \rightarrow B_x/\mathfrak{m}B_x
\end{equation}
is injective. Since every element of $B_x$ is the product of a unit and an element lying in the image of $B'_x \rightarrow B_x$, and since the map (\ref{localringovervalringproppfeqn31}) is injective, in order to prove the proposition for the ring $B_x$ it suffices to prove it for the ring $B'_x$. Since $B'_x$ is standard \'etale over a polynomial ring over $A$ and $\mathfrak{m}B'_x \subset B'_x$ is a prime ideal, we have proven the proposition for $B'_x$, and therefore also for $B_x$.

We therefore see that, for every $x$ in the special fiber of ${\rm{Spec}}(B)$, there exist a neighborhood ${\rm{Spec}}(B_x) \subset {\rm{Spec}}(B)$ of $x$ and elements $a_x \in A$ and $b_x \in B_x - \mathfrak{m}B_x$ such that $b = a_xb_x$. Choose finitely many $x_1, \dots, x_n$ such that the union of the ${\rm{Spec}}(B_{x_i})$ contains the special fiber of ${\rm{Spec}}(B)$. Let $i_0$ be such that $|a_{x_{i_0}}| = \max_{1 \leq i \leq n} |a_{x_i}|$. Finally, write $b = a_{x_{i_0}}\beta$ in $K \otimes_A B$, and let $b'_i := (a_{x_i}/a_{x_{i_0}})b_{x_i} \in B_{x_i}$. We claim that the sections $\beta$ on ${\rm{Spec}}(K \otimes_A B)$ and $b'_i$ on ${\rm{Spec}}(B_{x_i})$ glue to give a global section $b \in B$. Indeed, we have $b = a_{x_{i_0}}b'_i$ and $b = a_{x_{i_0}}\beta$ on the relevant open subschemes. Since ${\rm{Spec}}(B)$ and its open subschemes are $A$-flat, and therefore $A$-torsion free, it follows that the $b'_i$ and $\beta$ agree on the intersections of their domains of definition, hence glue to give a global section $b'$, as claimed. This section satisfies $b = a_{x_{i_0}}b'$ (since this holds on each of the above open subsets), and it only remains to show that $b' \notin \mathfrak{m}B$. This follows from the fact that this holds on the nonempty open subsets ${\rm{Spec}}(B_{x_i})$.
\\

\noindent \underline{Step 3}: Now we prove the proposition for valuation rings of finite height. This is height (or dimension) in the algebraic sense. (In the case of height $\leq 1$, this notion agrees with the earlier notion of being the valuation ring associated to a real-valued valuation \cite[Def.\,I.1.3.1, Prop.\,I.1.3.6, Cor.\,I.1.4.5]{morel}. Let $b \in B$. If $b = 0$, then the proposition is trivial, so we may assume that $b \neq 0$. Let $0 = \mathfrak{p}_0 \subsetneq \mathfrak{p}_1 \subsetneq \dots \subsetneq \mathfrak{p}_n = \mathfrak{m}$ be the prime ideals of $A$. Suppose that $b \notin \mathfrak{p}_iB$. We prove the proposition by descending induction on $i$, the base case $i = n$ being trivial.

So suppose that $0 \leq i < n$, and that the proposition holds for $b \notin \mathfrak{p}_{i+1}B$. If $b \notin \mathfrak{p}_{i+1}B$, then the induction hypothesis implies the proposition for $b$, so we may assume that $b \in \mathfrak{p}_{i+1}B - \mathfrak{p}_iB$. Let $A' := A_{\mathfrak{p}_{i+1}}/\mathfrak{p}_{i}A_{\mathfrak{p}_{i+1}}$. Then $A$ is a height one valuation ring. Indeed, it is one-dimensional, and it is a valuation ring because the property of being a valuation ring is preserved by localization and by taking the quotient by a prime ideal. Let $B' := B \otimes_A A'$, and let $j: B \rightarrow B'$ be the natural map. 

Since $A'$ is a height one valuation ring, Step 2 implies that $j(b) = a_1b_1$ for some $a_1 \in A'$ and $b_1 \in B' - \mathfrak{p}_{i+1}B'$. It follows that there exist  $x \in A - \mathfrak{p}_{i+1}$, $a_2 \in \mathfrak{p}_{i+1}$, $b_2 \in B - \mathfrak{p}_{i+1}B$, $a_3 \in \mathfrak{p}_{i}$, and $b_3 \in B$ such that 
\begin{equation}
\label{xb=a_2b_2eqn3}
xb = a_2b_2 + a_3b_3.
\end{equation}
We claim that $a_2 \notin \mathfrak{p}_i$. For otherwise we would have $xb \in \mathfrak{p}_iB$. Since $B/\mathfrak{p}_iB$ is a local ring of a smooth scheme over the domain $A/\mathfrak{p}_i$, it too is a domain, hence $\mathfrak{p}_iB \subset B$ is a prime ideal. It would therefore follow, since $b \notin \mathfrak{p}_iB$, that $x \in \mathfrak{p}_iB$. But we claim that $A \cap \mathfrak{p}B = \mathfrak{p}$ for any prime ideal $\mathfrak{p} \subset A$. It would then follow that $x \in \mathfrak{p}_i \subset \mathfrak{p}_{i+1}$, which is false. To prove the claim, we note that it is equivalent to the injectivity of the map $A/\mathfrak{p} \rightarrow B/\mathfrak{p}B$. This in turn follows from the fact that this map is faithfully flat, since $B$ is faithfully flat over $A$. We therefore have $a_2 \in \mathfrak{p}_{i+1} - \mathfrak{p}_i$, as desired.

Since $x \notin \mathfrak{p}_{i+1}$ (and therefore also $x \notin \mathfrak{p}_i$), Lemma \ref{primeidealdivisibility} implies that $a_2 = xa_2'$ and $a_3 = xa_3'$ for some $a_2' \in \mathfrak{p}_{i+1} - \mathfrak{p}_i$ and some $a_3' \in \mathfrak{p}_i$. Since $B$ is flat over $A$, it is $A$-torsion free, hence (\ref{xb=a_2b_2eqn3}) implies that
\[
b = a_2'b_2 + a_3'b_3.
\]
Once again applying Lemma \ref{primeidealdivisibility} implies that $a_3' = a_2'y$ for some $y \in \mathfrak{p}_i$. Thus,
\[
b = a_2'(b_2 + yb_3).
\]
Since $b_2 \in B- \mathfrak{p}_{i+1}B$ and $y \in \mathfrak{p}_i$, it follows that $b_4 := b_2 + yb_3 \in B - \mathfrak{p}_{i+1}B$. Thus, $b = a_2'b_4$ with $a_2' \in A$ and $b_4 \in B - \mathfrak{p}_{i+1}B$. The induction hypothesis now implies that $b_4 = a_4b_5$ for some $a_4 \in A$ and $b_5 \in B - \mathfrak{m}B$. We therefore obtain the desired decomposition for $b$.
\\

\noindent \underline{Step 4}: Now we prove the proposition in general. Let $b \in B$. Let $D$ be a smooth $A$-algebra of which $B$ is a localization such that $b \in D$ (more precisely, $b$ comes from an element of $D$). Shrinking ${\rm{Spec}}(D)$, we may assume that the map $D \otimes_A k \rightarrow B \otimes_A k$ is injective, where $k := A/\mathfrak{m}$. The ring $A$ is the filtered direct limit of its valuation subrings of finite height \cite[Cor.\,2.1.3]{temkin}. The $A$-algebra $D$ and the section $b$ descend to such a subring. That is, there exist a valuation subring $A' \subset A$ of finite height, a smooth $A'$-algebra $D'$, and $b' \in D'$ such that $D = D' \otimes_{A'} A$ and $b' \mapsto b$. The desired result therefore follows from Step 3, provided that the maps
\[
D' \otimes_{A'} k' \longrightarrow D \otimes_A k \longrightarrow B \otimes_A k
\]
are all injective, where $k'$ is the residue field of $A'$. The second map is injective by choice of $D$. For the first map, use the inclusion $k' \hookrightarrow k$ and the $A'$-flatness of $D'$.
\end{proof}

\begin{remark}
\cite[Tag 0ASJ]{stacks} contains an assertion similar to Proposition \ref{localringovervalringprop}. It says that if $B$ is a local ring of an algebra \'etale over the valuation ring $A$, and $A \rightarrow B$ is a local homomorphism, then $B$ is a valuation ring extension of $A$ with the same value group. In our case, when we merely assume $B$ to be a local ring of a smooth algebra, we have that $B$ is a localization of an algebra \'etale over $C := A[X_1, \dots, X_n]$ for some $n \geq 0$. Since $C_{\mathfrak{m}C}$ is a valuation ring extension of $A$ with the same value group, it follows that $B_{\mathfrak{m}B}$ is a valuation ring with the same value group as $A$. Therefore, for any $b \in B$ we may write $b_1b = ab_2$ for some $a \in A$ and some $b_1, b_2 \in B - \mathfrak{m}B$. The seemingly innocuous improvement of Proposition \ref{localringovervalringprop} -- that one may take $b_1 = 1$ -- turns out to be much more difficult.
\end{remark}

We shall have occasion to use the following consequence of Proposition \ref{localringovervalringprop}.

\begin{proposition}
\label{genericunitalmost=unit}
Let $A$ be a valuation ring with fraction field $K$, $A \rightarrow B$ a local homomorphism of local rings such that $B$ is $($as an $A$-algebra$)$ a local ring of a smooth $A$-scheme. Then
\[
(K \otimes_A B)^{\times} = K^{\times} \cdot B^{\times} := \{ab \mid a \in K^{\times}, b \in B^{\times} \}.
\]
\end{proposition}

\begin{proof}
Let $\beta \in (K \otimes_A B)^{\times}$. Using Proposition \ref{localringovervalringprop}, any element of $K \otimes_A B$ may be written as a product $\gamma b$ for some $\gamma \in K$ and $b \in B - \mathfrak{m}B$. We may therefore assume that $\beta = b \in B - \mathfrak{m}B$. By assumption, there exists a $\beta' \in K \otimes_A B$ such that $b \beta' = 1$. Writing $\beta' = \gamma' b'$ for some $\gamma' \in K$, $b' \in B - \mathfrak{m}B$, we obtain $bb' \in K$. Since $\mathfrak{m}B$ is a prime ideal of $B$ (because $B/\mathfrak{m}B$ is a local ring of a smooth $k := A/\mathfrak{m}$-algebra), $bb' \in B - \mathfrak{m}B$. We claim that $K \cap (B - \mathfrak{m}B) = A^{\times}$. Assuming this, it follows that $bb' \in A^{\times}$, hence $b \in B^{\times}$, and the proof would be complete.

So let $\alpha \in K \cap (B- \mathfrak{m}B)$. We will first show that $\alpha \in A$. Indeed, if $\alpha \notin A$, then, because $A$ is a valuation ring, $\alpha^{-1} \in A$. Since $\alpha \in B$, it follows that $\alpha^{-1} \in B^{\times}$. But $A \rightarrow B$ is a local homomorphism, so an element of $A$ which maps to $B^{\times}$ lies in $A^{\times}$. It follows that $\alpha^{-1} \in A^{\times}$, hence $\alpha \in A$, contradicting our assumption that $\alpha \notin A$. Hence $\alpha \in A$. Since $\alpha \notin \mathfrak{m}B$, it follows that $\alpha \notin \mathfrak{m}$, hence $\alpha \in A^{\times}$.
\end{proof}

\chapter{Profinite Completions}
\label{chapterprofinitecompletions}

This appendix discusses some generalities that we require on profinite completions of topological abelian groups, and especially their behavior in exact sequences. We need to understand such completions because several of the main duality results of this manuscript involve profinite completions of various groups (generally groups of rational points), and several d\'evissage arguments will require an understanding of how the process of profinite completion behaves with respect to exact sequences.

\section{Right-exactness of profinite completion}
\label{profiniterightexactsection}

In this section we prove that profinite completion is right-exact. Before we do this, let us recall some basic definitions. An abelian topological group $A$ is said to be 
{\em profinite} if it is compact and totally disconnected. (One can make this definition for nonabelian groups also, but we will not need this and therefore will restrict ourselves to the simpler abelian case.) This is equivalent to saying that the canonical map $A \rightarrow \varprojlim_H A/H$ is an isomorphism, where the limit is over all open (equivalently, closed) subgroups $H \subset A$ of finite index \cite[Ch.\,I, \S1.1, Prop 0]{serre} (hence the term ``profinite''). 

\begin{remark}
\label{profinitecompletiondef}
Given an arbitrary abelian topological group $A$, the {\em profinite completion} $A_{\pro}$ of $A$ is the initial profinite group with a map from $A$. That is, we have a continuous homomorphism $\phi:  A \rightarrow A_{\pro}$ such that any (continuous) homomorphism $A \rightarrow B$ with $B$ profinite factors uniquely as a composition 
of continuous homomorphisms $A \rightarrow A_{\pro} \rightarrow B$. Then $A_{\pro}$ is clearly unique, and it also exists:  we have $A_{\pro} = \varprojlim_H A/H$, where the limit is over all open subgroups $H \subset A$ of finite index. If $A$ is an abelian group without any topology, then $A_{\pro}$ is taken to mean the profinite completion of $A$ with the discrete topology.
\end{remark}

\begin{lemma}
\label{profinitequotient}
Let $A$ be a profinite abelian group, $B \subset A$ a compact subgroup. Then $A/B$ endowed with the quotient topology is profinite.
\end{lemma}

\begin{proof}
We first claim that the map $\pi:  A \rightarrow A/B$ is both open and closed. Indeed, for closedness, the compactness of $A$ implies that a closed subset of $A$ is compact, hence has compact image in $A/B$. Since $B \subset A$ is closed, $A/B$ is Hausdorff, hence this image is closed. For openness, let $V \subset A$ be open. We need to show that $\pi(V) \subset A/B$ is open. By definition, this is equivalent to saying that $\pi^{-1}(\pi(V)) = V + B \subset A$ is open. But $V + B = \cup_{b \in B} (V + b)$, a union of open sets, hence is open.

Since $A$ is compact, so is $A/B$. It therefore only remains to show that $A/B$ is totally disconnected. For this, it suffices to show that if $a \in A - B$, then there is a clopen subset $U \subset A$ such that $a \in U$ and $U \cap B = \emptyset$. Indeed, since $\pi$ is a clopen map, it then follows that $\pi(U)$ is a clopen neighborhood of the image of $a$ not containing $0$. Since $A$ is totally disconnected, for each $b \in B$ there exists a clopen set $U_b \subset A$ such that $a \in U_b$ and $b \notin U_b$. We claim that there are finitely many $b_1, \dots, b_n \in B$ such that $U_{b_1} \cap \dots \cap U_{b_n} \cap B = \emptyset$. Assuming this, we may take $U : = U_{b_1} \cap \dots \cap U_{b_n}$ as our desired clopen subset. If the claim were false, then the sets $\{U_b \cap B\}_{b \in B}$ would be a collection of closed subsets of the compact set $B$ satisfying the finite intersection property, which is impossible since $B$ is compact and $B \cap (\cap_{b \in B} U_b) = \emptyset$. 
\end{proof}

Now we turn to the exactness properties of profinite completion, which are the main concern of this chapter. 

\begin{definition}
A sequence
\[
0 \longrightarrow A' \longrightarrow A \longrightarrow A'' \longrightarrow 0
\]
of topological abelian groups is said to be {\em short exact} if it is exact as a sequence of abelian groups, the map $A' \rightarrow A$ is a topological embedding, and the quotient map $A/A' \rightarrow A''$ is a topological isomorphism. We say that the sequence
\[
A' \xlongrightarrow{\phi} A \longrightarrow A'' \longrightarrow 0
\]
of continuous maps is {\em right-exact} if it is exact as a sequence of groups and if the map $A/\phi(A') \rightarrow A''$ is a topological isomorphism, where $\phi(A') \subset A$ is endowed with the subspace topology. Let TAb denote the category of abelian topological groups. We say that a functor $F:  {\rm{TAb}} \rightarrow {\rm{TAb}}$ is 
{\em right-exact} if it sends short-exact sequences to right-exact sequences.
\end{definition}

The main result of this section is the following.

\begin{proposition}
\label{profiniterightexact}
The functor ${\rm{TAb}} \rightarrow {\rm{TAb}}$ given by $A \rightsquigarrow A_{\pro}$ is right-exact.
\end{proposition}

\begin{proof}
Suppose that we have an inclusion of abelian topological groups $\phi:  B \hookrightarrow A$
(where $B$ has the subspace topology). We need to show that the natural continuous map $A_{\pro}/\phi(B_{\pro}) \rightarrow (A/B)_{\pro}$ is a topological isomorphism. We will do this by constructing an inverse. We have a natural
continuous map $A/B \rightarrow A_{\pro}/\phi(B_{\pro})$. The group $A_{\pro}/\phi(B_{\pro})$ is profinite by Lemma \ref{profinitequotient}, so we obtain a continuous map $(A/B)_{\pro} \rightarrow A_{\pro}/\phi(B_{\pro})$. To see that this is inverse to the other map, we note that it suffices to check that composition in either direction is the identity on the dense image of $A$ inside both groups, and this is clear.
\end{proof}

\section{Conditions under which profinite completion preserves exactness}

As we saw in the previous section, profinite completion is right-exact. It is not, however, exact. Indeed, consider the inclusion $\Z \rightarrow \Q$. The profinite completion $\Z_{\pro}$ of $\Z$ is nonzero, but $\Q_{\pro} = 0$ because $\Q$ is divisible (hence has no nontrivial finite quotients). The goal of this section is to give some useful conditions ensuring that exactness is preserved by profinite completion. Here is the first.

\begin{proposition}
\label{profiniteexactfgquotient}
Suppose that we have a short exact sequence of abelian topological groups
\[
0 \longrightarrow A' \longrightarrow A \xlongrightarrow{f} A'' \longrightarrow 0
\]
such that $A''$ is discrete and finitely generated. Then the induced sequence
\[
0 \longrightarrow A'_{\pro} \longrightarrow A_{\pro} \longrightarrow A''_{\pro} \longrightarrow 0
\]
is also short exact.
\end{proposition}

\begin{proof}
By Proposition \ref{profiniterightexact}, we only need to check that the map $A'_{\pro} \rightarrow A_{\pro}$ is a topological embedding. We only have to show that this continuous map is injective, as it is then necessarily a homeomorphism onto its image because the source is compact and the target Hausdorff. Let $H' \subset A'$ be an open subgroup of finite index. It will suffice to find a finite index open subgroup $H \subset A$ such that $H \cap A' = H'$. Let $A'' = F \times T$ with $F$ free of finite rank and $T$ finite. Let $a_1 \dots, a_n \in A$ be such that the elements $f(a_i) \in A''$ form a set of free generators for $F$. We take $H = \langle a_1, \dots, a_n, H' \rangle$. Then $H$ is open because $H'$ is open in $A$ (since $A' \subset A$ is open due to the discreteness of $A''$). Suppose that $b \in A' \cap H$. Then $b = h' + \sum_{i=1}^n r_ia_i$ for some $h' \in H', r_i \in \Z$. Since $b \in A'$, we have $0 = f(b) = \sum_{i=1}^n r_if(a_i)$, so $r_i = 0$ for every $i$. Therefore, $b = h' \in H'$, so $H \cap A' = H'$. Finally, $H$ is of finite index inside $A$ because we have an exact sequence
\[
0 \longrightarrow A'/H' \longrightarrow A/H \longrightarrow A''/f(H) \longrightarrow 0
\]
and $A'/H'$ is finite, as is $A''/f(H) \simeq T$.
\end{proof}

We will also require another result  (Proposition \ref{profiniteexact} below) that, at least when all of the groups are discrete, strengthens Proposition \ref{profiniteexactfgquotient}. Before turning to the proof, we require several lemmas.

\begin{lemma}
\label{finiteexpeltsinsubgp}
Let $A$ be an abelian group of finite exponent, and suppose that $0 \neq a \in A$ lies in every nonzero subgroup of $A$. Then $A \simeq \Z/p^n\Z$ for some prime $p$ and some $n > 0$. 
\end{lemma}

\begin{proof}
First assume that $A$ is finite. It suffices to show that we cannot have $A \simeq A_1 \times A_2$ with $A_1, A_2$ nontrivial. If such an isomorphism exists then $a \in (A_1 \times 0) \cap (0 \times A_2) = 0$, a contradiction.

Now consider the general case. By the previous paragraph, it suffices to show that $A$ is finite. In order to show this, in turn, it suffices to show that, if $N$ is the exponent of $A$, then for every finite collection $a, b_1, \dots, b_m \in A$ of elements of $A$ containing $a$, the subgroup $B$ that they generate has size at most $N$. Since $A$ is of finite exponent, $B$ must be finite. Since $a$ is contained in every nonzero subgroup of $B$, it follows from the first paragraph that the finite $B$ must be cyclic, of order $r$, say. But $N$ kills $B$, so $r|N$. Thus, $\# B = r \leq N$, as desired.
\end{proof}

\begin{lemma}
\label{maximalsubgp}
Let $A$ be an abelian group, $0 \neq a \in A$. Then there exists a subgroup $B \subset A$ maximal with respect to the property of not containing $a$. That is, $a \notin B$, and if $B' \supsetneq B$ is a strictly larger subgroup of $A$, then $a \in B'$.
\end{lemma}

\begin{proof}
Consider the set $S$ of all subgroups of $A$ not containing $a$. This set is non-empty, since $a \notin \{0\}$. Further, any chain in $S$ has an upper bound (the union of the elements of the chain). Hence, by Zorn's Lemma, $S$ has a maximal element.
\end{proof}

\begin{lemma}
\label{avoidingfindexsubgps}
Let $A$ be an abelian topological group such that $nA \subset A$ is open for every positive integer $n$, and let $a \in A$. In order for there to exist an open subgroup $B \subset A$ of finite index such that $a \notin B$, it is necessary and sufficient that $a \notin A_{\Div}$.
\end{lemma}

\begin{proof}
First, if $a$ is divisible and $B \subset A$ is a finite index subgroup, then $\overline{a} : = a \bmod B \in A/B$ is a divisible element of the finite group $A/B$, hence $\overline{a} = 0$; i.e., $a \in B$.

Conversely, suppose that $a \notin A_{\Div}$, so $a \notin nA$ for some $n>0$. Since $nA$ is open, $A/nA$ is discrete, so by replacing $A$ with $A/nA$ and $a$ with its nonzero image in $A/nA$ we may suppose that $A$ is discrete of finite exponent. By Lemma \ref{maximalsubgp}, there is a subgroup $B \subset A$ maximal with respect to the property of not containing $a$. Then we claim that $B \subset A$ is of finite index. Indeed, $A/B$ is of finite exponent, and $a$ represents a nonzero element of $A/B$ that is contained in every nonzero subgroup, hence $A/B$ is finite by Lemma \ref{finiteexpeltsinsubgp}.
\end{proof}

\begin{lemma}
\label{nondivisibleextension}
Suppose that we have a short exact sequence of abelian groups
\[
0 \longrightarrow A' \longrightarrow A \xlongrightarrow{f} A'' \longrightarrow 0
\]
with $A'_{\Div}, A''_{\Div} = 0$, and suppose that $A''_{\tors}$ has finite exponent. Then $A_{\Div} = 0$.
\end{lemma}

\begin{proof}
Choose an $a \in A_{\Div}$. Then $f(a) \in A''_{\Div} = 0$, so $a \in A'$. Let $N$ be the exponent of $A''_{\tors}$. For every positive integer $m$, we have $a = mNb$ for some $b \in A$. Since $f(a) = 0$, so $f(b) \in A''_{\tors}$,
we have $f(Nb) = Nf(b) = 0$ and thus $Nb \in A'$. We conclude that $a = m(Nb) \in mA'$. Since $m \in \Z_+$ was arbitrary, we deduce that $a \in A'_{\Div} = 0$.
\end{proof}

\begin{remark}
Lemma \ref{nondivisibleextension} is false if we omit the assumption that $A''_{\tors}$ is of finite exponent. Indeed, let $A$ be the free abelian group on generators $x_n$ ($n \in \Z_+$) modulo the relations $nx_n = x_1$. Clearly $x_1$ is divisible. It is also nonzero (even non-torsion), since we have the map $A \rightarrow \Q$ sending $x_n$ to $1/n$. If we let $A' = \langle x_1 \rangle \subset A$, then $A' \simeq \Z$, and $A'' : = A/A' \simeq \oplus_{n \geq 2} \Z/n\Z$. So $A'_{\Div}, A''_{\Div} = 0$, but $A_{\Div} \neq 0$.
\end{remark}

\begin{proposition}
\label{profiniteexact}
Suppose that we have a short exact sequence
\[
0 \longrightarrow A' \longrightarrow A \longrightarrow A'' \longrightarrow 0
\]
of abelian topological groups such that $A''_{\Div} = 0$ and $A''_{\tors}$ has finite exponent. $($This holds, for example, if $A''$ has finite exponent.$)$ Suppose also that for each positive integer $n$, the subgroup $nA \subset A$ is open. Then the induced sequence
\[
0 \longrightarrow A'_{\pro} \longrightarrow A_{\pro} \longrightarrow A''_{\pro} \longrightarrow 0
\]
is short exact.
\end{proposition}

\begin{proof}
Let $B' \subset A'$ be an open subgroup of finite index. We first claim that for any nonzero $\overline{a} : = a$ mod $B' \in A/B'$, there exists an open finite index subgroup $\overline{X} \subset A/B'$ not containing $\overline{a}$. The openness of $nA$ is preserved under quotients, hence the hypotheses of Lemma \ref{avoidingfindexsubgps} hold for $A/B'$. In order to prove the claim, therefore, it suffices to show that $(A/B')_{\Div} = 0$. We have an exact sequence of abelian groups
\[
0 \longrightarrow A'/B' \longrightarrow A/B' \longrightarrow A'' \longrightarrow 0
\]
Since $A'/B'$ is finite, so  $(A'/B')_{\Div} = 0$, the hypotheses 
that $A''_{\rm{div}} = 0$ and $A''_{\rm{tors}}$ has finite exponent imply that 
$(A/B')_{\Div} = 0$ by Lemma \ref{nondivisibleextension}.

It follows that for any $a \in A - B'$, there exists an open finite index subgroup $X \subset A$ such that $B' \subset X \subset A$ but $a \notin X$, hence
\begin{equation}\label{Bint}
B' = \bigcap_{B' \subset X \subset A} X,
\end{equation}
where the intersection is over all open finite index subgroups $X \subset A$ containing $B'$.
For any open subgroup $X \subset A$ of finite index containing $B'$, we have  $[A':  X \cap A'] \le [A':  B'] < \infty$. Thus, there exists such $X$ which maximizes $[A': X \cap A']$. We claim that such an $X$ satisfies $X \cap A' = B'$. If such equality fails, so $B'$ is a proper subgroup of $X \cap A'$,
then by (\ref{Bint}) there exists an open subgroup $Y \subset A$ of finite index containing $B'$ such that 
the inclusion $Y \cap X \cap A' \subset X \cap A'$ is strict. But then
$$[A':  (X \cap Y) \cap A'] > [A': X \cap A'],$$
violating the maximality property of $X$. Thus $B' = X \cap A'$.

We have shown that $B' = X \cap A'$ for some open subgroup $X \subset A$ of finite index. Since $B' \subset A'$ is an arbitrary open subgroup of finite index, it follows that the map $A'_{\pro} \rightarrow A_{\pro}$ is injective, hence also a topological embedding because the source is compact and the target Hausdorff. Combining this with Proposition \ref{profiniterightexact}, the proof of the proposition is complete.
\end{proof}

\section{Products}

In this section we examine how profinite completion interacts with the formation of various groups associated with a product of topological groups, most especially restricted products. This -- unsurprisingly -- plays an important role in understanding the profinite completions of groups obtained from the adeles of a global field. Rather than via a direct description, such profinite completions are most naturally described via their Pontryagin duals.

\begin{proposition}
\label{restrictedprodprofinitecomp}
Let $\{A_i\}_{i \in I}$ be a set of topological abelian groups, and, for all but finitely many $i \in I$, let $B_i \subset A_i$ be a subgroup. Then the map
\begin{equation}
\label{restrictedprodprofinitecompeqn1}
\left[\left({\prod}'_{ \in I} A_i\right)_{\pro}\right]^D \longrightarrow \prod_{i \in I} A_{i,\, \pro}^D
\end{equation}
induces an isomorphism
\begin{equation}
\label{restrictedprodprofinitecompeqn2}
\left[\left({\prod}'_{i \in I} A_i\right)_{\pro}\right]^D \longrightarrow \left[{\prod}'_{i \in I} A_{i,\, \pro}^D\right]_{\tors},
\end{equation}
where the restricted product on the left is with respect to the $B_i$, and the restricted product on the right is with respect to their annihilators in $A_{i,\, \pro}^D$, and the group on the right has the discrete topology.
\end{proposition}

\begin{proof}
The Pontryagin dual of a profinite group is discrete, so we only need to show that (\ref{restrictedprodprofinitecompeqn2}) is an algebraic isomorphism. We first check that the map (\ref{restrictedprodprofinitecompeqn1}) factors through the subgroup $[{\prod}'_{i \in I} A_{i,\, \pro}^D$. By the usual ``no small subgroups'' argument, any continuous homomorphism from a profinite group to $\mathbf{R}/\Z$ factors through $\frac{1}{n}\Z/\Z$ for some $n > 0$. It follows that (\ref{restrictedprodprofinitecompeqn1}) lands inside the torsion subgroup of the right side. We need to show that it also lands inside the restricted product. That is, we need to show that any continuous map
\begin{equation}
\label{restrictedprodprofinitecomppfeqn1}
f\colon \left({\prod}'_{ \in I} A_i\right)_{\pro} \longrightarrow \Z/n\Z
\end{equation}
kills $B_i$ for all but finitely many $i$. This follows from the fact that $f$ has open kernel.

It remains to show that the map (\ref{restrictedprodprofinitecompeqn2}) is an isomorphism. We first check injectivity. Given $f$ as in (\ref{restrictedprodprofinitecomppfeqn1}) such that $f|_{A_i} = 0$ for all $i$, we know that $f$ has open kernel, hence kills $\prod_{i \notin S} B_i$ for some finite $S \subset I$. It follows that, for any $a \in {\prod}' A_i$, we have $f(a) = \sum_{i \in S'} f(a_i)$ for some finite subset $S' \subset I$ containing $S$. Since $f$ kills each $A_i$, it follows that $f(a) = 0$. Since $a \in {\prod}' A_i$ was arbitrary, it follows that $f = 0$. This proves the injectivity.

Now we need to prove the surjectivity of (\ref{restrictedprodprofinitecompeqn2}). An element of the right side of (\ref{restrictedprodprofinitecompeqn2}) is a set $\{f_i\}_{i \in I}$ of continuous homomorphisms
\[
f_i\colon A_{i,\, \pro} \longrightarrow \Z/n\Z
\]
for some $n > 0$ such that $f_i|_{B_i} = 0$ for all but finitely many $i \in I$. Then we may define
\[
f \colon {\prod}'_{i \in I} A_i \longrightarrow \Z/n\Z
\]
by the formula $f(a) := \sum_{i \in I} f_i(a_i)$. Because $f_i$ kills $B_i$ for almost every $i$, and $a_i \in B_i$ for almost every $i$, the sum contains only finitely many nonzero terms. We claim that $f$ is continuous. This is equivalent to having open kernel, and this holds because each $f_i$ is continuous and because $f_i$ kills $B_i$ for almost every $i$. Finally, because this kernel is closed of finite index, it follows that $f$ extends to a continuous homomorphism
\[
\left({\prod}'_{i \in I} A_i\right)_{\pro} \longrightarrow \Z/n\Z.
\]
\end{proof}

Next we investigate how profinite completion behaves with respect to products. More generally, we determine how it behaves with respect to suitably large subgroups of products.

\begin{proposition}
\label{completionofproduct}
Let $\{A_i\}_{i \in I}$ be a set of topological abelian groups, and let $G \subset \prod_{i \in I} A_i$ be a subgroup endowed with the subspace topology of the product topology on $\prod_i A_i$. Assume that, for each $i \in I$, we have $A_i \times \prod_{j \in I - \{i\}} 0 \subset G$. Then the canonical map
\[
G_{\pro} \longrightarrow \prod_{i \in I} A_{i,\, \pro}
\]
is a topological isomorphism.
\end{proposition}

\begin{proof}
We will show that the map in the proposition induces an isomorphism of Pontryagin dual groups. First, when $G = \prod A_i$, then, by Proposition \ref{restrictedprodprofinitecomp} with $B_i = A_i$ for all $i$, we have
\[
(G_{\rm{pro}})^D = \oplus_{i \in I} (A_{i,\,\pro})^D,
\]
since the groups $(A_{i,\,\pro})^D$ are torsion, as duals of profinite groups. Since the Pontryagin dual of a product of groups is the direct sum of the duals, this proves the proposition when $G = \prod A_i$.

Now consider the general case. Since $A_i \times \prod_{i \in I - \{i\}} 0 \subset G$ for all $i \in I$, we have a natural map
\begin{equation}
\label{completionofproductpfeqn1}
(G_{\pro})^D \longrightarrow \prod_{i \in I} (A_{i,\,\pro})^D,
\end{equation}
which is an inclusion because the collection of subgroups $\prod_{i \in S} A_i \times \prod_{i \in I - S} A_i$ as $S$ varies over finite subsets of $I$ is dense in $\prod_{i \in I} A_i$. Now any continuous map $f: G \rightarrow F$ from $G$  into a discrete finite group must kill $G \cap (\prod_{i \in S} 0 \times \prod_{i \in I - S} A_i)$ for some finite subset $S \subset I$ (since $f$ has open kernel). It follows that (\ref{completionofproductpfeqn1}) actually yields an inclusion
\[
(G_{\pro})^D \hookrightarrow \oplus_{i \in I} (A_{i,\,\pro})^D.
\]
We claim that this map is an isomorphism, which will prove the proposition (since it identifies the Pontryagin duals of the two compact Hausdorff groups appearing in the statement of the proposition). In fact, the desired surjectivity statement is clear: given a finite subset $S \subset I$, and, for each $i \in S$, an $f_i \in (A_{i,\,\pro})^D$, we then get an element of $(G_{\pro})^D$ mapping to $\prod_{i \in S} f_i \times \prod_{i \in I - S} 0$ via $f(g) := \sum_{i \in S} f_i(g_i)$ for $g \in G$.
\end{proof}

\chapter{Duality Pairings and Weil Restriction}
\label{commutativediagramappendix}

The purpose of this appendix is to prove that the local and global duality pairings are compatible with Weil restriction. The meaning of this compatibility is made precise in Propositions \ref{diagramcommuteslocal} and \ref{diagramcommutesglobal}. 

We first prove that taking $\Gm$ dual sheaves commutes with pushforward through a finite \'etale map. Let $f:  X' \rightarrow X$ be a finite flat map between Noetherian schemes, $\mathscr{F}'$ an abelian fppf sheaf on $X'$, and let $\mathscr{F} : = f_*\mathscr{F}'$. Then we have a natural map of 
abelian fppf sheaves ${\rm{N}}_{X'/X}:  f_*(\widehat{\mathscr{F}'}) \rightarrow \widehat{\mathscr{F}}$ defined as follows. Given an $X$-scheme $Y$ and a character $\chi \in f_*(\widehat{\mathscr{F}'})(Y) = \widehat{\mathscr{F}'}(Y \times_X X')$, the character ${\rm{N}}_{X'/X}(\chi) \in \widehat{\mathscr{F}}(Y)$ is defined functorially on $Y$-schemes $Z$ as the composition $\mathscr{F}(Z) = \mathscr{F}'(Z \times_X X') \xrightarrow{\chi_{Z\times_X X'}} \Gamma(Z \times_X X', \mathcal{O}_{Z \times_X X'}^{\times}) \xrightarrow{\mbox{N}_{X'/X}} \Gamma(Z, \mathcal{O}_Z^{\times})$, where ${\rm{N}}_{X'/X}$ is the norm map, defined locally on $\Spec(A) \subset Z$ as follows. Let $X = \Spec(R)$ and $X' = \Spec(R')$. Given $\alpha \in A \otimes_R R'$, we define ${\rm{Nm}}_{X'/X}(\alpha) \in A$ to be the determinant of multiplication by $\alpha$ as an $A$-linear map on the finite free $A$-module $A \otimes_R R'$.

The following lemma is a (generalized) sheafified version of (part of) \cite[Ch.\,II, Thm.\,2.4]{oesterle}.

\begin{appendixproposition}
\label{charactersseparableweilrestriction}
Let $f:  X' \rightarrow X$ be a finite {\'e}tale map of Noetherian schemes, $\mathscr{F}'$ an fppf abelian sheaf on $X'$. Then ${\rm{N}}_{X'/X}:  f_*(\widehat{\mathscr{F}'}) \rightarrow \widehat{f_*(\mathscr{F}')}$ is an isomorphism of fppf sheaves.
\end{appendixproposition}

\begin{proof}
The assertion is fppf local, so since \'etale locally on $X$, $X'$ is a finite disjoint union of copies of $X$, we may assume that $X'$ is of this form. Since the norm map, pushforward, and $\widehat{(\cdot)}$ functor send such a union to a product, we may therefore assume that $X' = X$, in which case the assertion is trivial.
\end{proof}

\begin{appendixremark}
The {\'e}tale hypothesis in Proposition \ref{charactersseparableweilrestriction} is absolutely crucial:  unlike 
\cite[Ch.\,II, Thm.\,2.4]{oesterle}, the lemma fails for a finite inseparable extension of fields $k'/k$, even for $\mathscr{F}' = \mathbf{G}_m$. Indeed, if $k'/k$ is a finite nontrivial purely inseparable extension, then there is no isomorphism between $\widehat{\R_{k'/k}(\mathbf{G}_m)}$ and $\R_{k'/k}(\widehat{\mathbf{G}_m}) \simeq \R_{k'/k}(\Z)$, let alone via ${\rm{N}}_{k'/k}$,
since the latter sheaf is a smooth (even \'etale) group scheme whereas the former is not even representable.
To prove such non-representability, by \cite[App.\,3, A.3.6]{oesterle} there is a short exact sequence
\[
1 \longrightarrow U \longrightarrow \R_{k'/k}(\mathbf{G}_m)_{k'} \xlongrightarrow{\pi} \mathbf{G}_{m,k'} \longrightarrow 1
\]
with $U$ split unipotent and {\em nontrivial}.  The map $\pi$ is given functorially on $k'$-algebras $R'$ by the map $(R' \otimes_k k')^{\times} \rightarrow R'^{\times}$ induced by the map $R' \otimes_k k' \rightarrow R'$
defined by $r \otimes \lambda \mapsto r\lambda$. In particular, $\R_{k'/k}(\mathbf{G}_m)$ is not an almost-torus. That $\widehat{\R_{k'/k}(\mathbf{G}_m)}$ is not representable therefore follows from Proposition \ref{hatrepresentable}.
\end{appendixremark}

Suppose that we have a finite morphism $f:  X' \rightarrow X$ between Noetherian schemes. Let $\mathscr{F}'$ be an fppf abelian sheaf on $X'$, and let $\mathscr{F} : = f_*\mathscr{F}'$. We have a functorial map ${\rm{N}}_{X'/X}:  f_*\widehat{\mathscr{F}'} \rightarrow \widehat{\mathscr{F}}$ which is an isomorphism when $f$ is \'etale, by Proposition \ref{charactersseparableweilrestriction}. We also have a canonical map ${\rm{H}}^i(X, \mathscr{F}) \rightarrow {\rm{H}}^i(X', \mathscr{F}')$ which we will always, for notational simplicity, simply denote by $\phi$. If $f$ is \'etale, then $f_*$ is an exact functor between categories of fppf abelian sheaves (since this may be checked after pullback to a finite Galois covering $X'' \rightarrow X'$ splitting $X'/X$, and after such pullback $\mathscr{F}$ just becomes a product of pullbacks of $\mathscr{F}'$ along the $X$-morphisms $X'' \rightarrow X'$), hence $\phi$ is an isomorphism in this case.

\begin{appendixremark} 
\label{phiisanisom}
If $\mathscr{F}'$ is a smooth group scheme and $f$ is finite, then $\phi$ is an isomorphism.  Indeed, our assumption on $\mathscr{F}'$ implies that we may take our cohomology to be \'etale rather than fppf, and then we use the fact that pushforward through a finite map is an exact functor between categories of \'etale sheaves.
\end{appendixremark}

Note in particular that this discussion applies when $X' = \Spec(k')$, $X = \Spec(k)$, and $\mathscr{F}' = G'$ for some finite extension $k'/k$ of fields and some commutative $k'$-group scheme $G'$ of finite type. Then the sheaf $G : = f_*G'$ is better known as the Weil restriction of scalars $\R_{k'/k}(G')$, and it is a $k$-group scheme of finite type. One may also define another  norm map ${\rm{N}}_{k'/k}:  \R_{k'/k}(\Gm) \rightarrow \Gm$ between commutative $k$-group schemes that is defined functorially on $k$-algebras $R$ by the norm map $(k' \otimes_k R)^{\times} \rightarrow R^{\times}$ which sends an element $r'$ of the finite free $R$-algebra $k' \otimes_k R$ to the determinant of the $R$-linear map on $k' \otimes_k R$ defined by multiplication by $r'$. Before turning to our main compatiblity results, we need the following proposition.

\begin{appendixproposition}
\label{invnormcomp}
Let $k'/k$ be a finite extension of local fields. Then the following diagram commutes: 
\[
\begin{tikzcd}
{\rm{H}}^2(k, \Gm) \arrow{r}{\inv} & \Q/\Z \arrow[dd, equals] \\
{\rm{H}}^2(k, \R_{k'/k}(\Gm)) \arrow{u}{{\rm{H}}^2({\rm{N}}_{k'/k})} \arrow{d}{\phi} & \\
{\rm{H}}^2(k', \Gm) \arrow{r}{\inv} & \Q/\Z
\end{tikzcd}
\]
where $\phi$ is the canonical map discussed above and by abuse of notation and ${\rm{N}}_{k'/k}:  \R_{k'/k}(\Gm) \rightarrow \Gm$ is the norm map.
\end{appendixproposition}

\begin{proof}
Consider the map $i:  {\rm{H}}^2(k, \Gm) \rightarrow {\rm{H}}^2(k, \R_{k'/k}(\Gm))$ induced by the canonical inclusion of $k$-group schemes $\Gm \hookrightarrow \R_{k'/k}(\Gm)$. We claim that $i$ is surjective. Indeed, $\phi$ is an isomorphism by Remark \ref{phiisanisom}, so it suffices to check that $\phi \circ i:  {\rm{H}}^2(k, \Gm) \rightarrow {\rm{H}}^2(k', \Gm)$ is surjective. But this composition is none other than the pullback map (also known as restriction) res. If $k$ is archimedean, then either $k' = k$ or $k' = \mathbf{C}$. In either case the desired surjectivity is immediate. We may therefore assume that $k$ is non-archimedean. Then the desired surjectivity follows from the following commutative diagram: 
\begin{appendixequation}
\label{invnormcomppfeqn2}
\begin{tikzcd}
{\rm{H}}^2(k, \Gm) \arrow{d}{{\rm{res}}} \arrow{r}{\inv}[swap]{\sim} & \Q/\Z \arrow[twoheadrightarrow]{d}{[k':  k]} \\
{\rm{H}}^2(k', \Gm) \arrow{r}{\inv}[swap]{\sim} & \Q/\Z
\end{tikzcd}
\end{appendixequation}

In order to prove the proposition, therefore, it is equivalent to show that
\begin{appendixequation}
\label{invnormcomppfeqn1}
{\rm{inv}} \circ \phi \circ i \stackrel{?}{=} {\rm{inv}} \circ {\rm{H}}^2({\rm{N}}_{k'/k}) \circ i.
\end{appendixequation}
But the composition ${\rm{H}}^2({\rm{N}}_{k'/k}) \circ i:  {\rm{H}}^2(k, \Gm) \rightarrow {\rm{H}}^2(k, \Gm)$ is the map induced by the composition $\Gm \hookrightarrow \R_{k'/k}(\Gm) \xrightarrow{{\rm{N}}_{k'/k}} \Gm$, in which the first map is the canonical inclusion. This composition is just multiplication by $[k':  k]$ (as follows by considering the effect on $R$-valued points for $k$-algebras $R$). Further, as we have already noted, $\phi \circ i = {\rm{res}}$. Thus, (\ref{invnormcomppfeqn1}) is equivalent to
\[
{\rm{inv}} \circ {\rm{res}} \stackrel{?}{=} {\rm{inv}} \circ [k':  k],
\]
and since ${\rm{inv}} \circ [k':  k] = [k':  k] \circ {\rm{inv}}$, this is precisely the content of the commutativity of (\ref{invnormcomppfeqn2}), which also holds in the archimedean case (even though in this case the horizontal maps in that diagram are not isomorphisms).
\end{proof}

Now we prove that the local duality pairing is compatible with Weil restriction via $\phi$ and the norm maps: 

\begin{appendixproposition}
\label{diagramcommuteslocal}
Let $k'/k$ be a finite extension of local fields, $G'$ a commutative $k'$-group scheme of finite type, and $G : = \R_{k'/k}(G')$. Then the following diagram commutes: 
\[
\begin{tikzcd}
{\rm{H}}^i(k, G) \arrow[d, equals] \arrow[r, phantom, "\times"] & {\rm{H}}^{2-i}(k, \widehat{G}) \arrow{r}{\cup} & {\rm{H}}^2(k, \Gm) \arrow{r}{\inv} & \Q/\Z \arrow[dd, equals] \\
{\rm{H}}^i(k, G) \arrow{d}{\phi} \arrow[r, phantom, "\times"] & {\rm{H}}^{2-i}(k, \R_{k'/k}(\widehat{G'})) \arrow{u}{{\rm{H}}^{2-i}({\rm{N}}_{k'/k})} \arrow{d}{\phi} \arrow{r}{\cup} & {\rm{H}}^2(k, \R_{k'/k}(\Gm)) \arrow{u}{{\rm{H}}^2({\rm{N}}_{k'/k})} \arrow{d}{\phi} & \\
{\rm{H}}^i(k', G') \arrow[r, phantom, "\times"] & {\rm{H}}^{2-i}(k', \widehat{G'}) \arrow{r}{\cup} & {\rm{H}}^2(k', \Gm) \arrow{r}{\inv} & \Q/\Z
\end{tikzcd}
\]
Here the first ${\rm{N}}_{k'/k}$ is the norm map on character sheaves, and the second is the usual norm map. Further, if $k'/k$ is separable then all the vertical maps are isomorphisms.
\end{appendixproposition}

\begin{proof}
That all of the vertical maps are isomorphisms when $k'/k$ is separable follows from Proposition \ref{charactersseparableweilrestriction} and our remarks at the beginning of the section about $\phi$ being an isomorphism when $f$ is \'etale. The first three columns of the diagram in the proposition commute because of the functoriality of cup product, $\phi$, and the norm maps. It therefore only remains to prove commutativity of the last two columns, and this is the content of Proposition \ref{invnormcomp}.
\end{proof}

Finally, we turn to the global setting. That is, suppose that $k'/k$ is a finite extension of global fields, and let $G'$ be a commutative $k'$-group scheme of finite type, and $G : = \R_{k'/k}(G')$. Then we have the usual adelic pairing (that is, cupping everywhere locally and then adding the invariants)
\[
{\rm{H}}^i(\A_k, G) \times {\rm{H}}^{2-i}(\A_k, \widehat{G}) \rightarrow \Q/\Z,
\]
and similarly for $\A_{k'}$ and $G'$. The next result says that these two pairings are compatible via $\phi$ and ${\rm{N}}_{k'/k}$.

\begin{appendixproposition}
\label{diagramcommutesglobal}
If $k'/k$ is a finite extension of global fields, $G'$ is a commutative $k'$-group scheme of finite type, and $G : = \R_{k'/k}(G')$, then the following diagram commutes${\rm{:}}$
\begin{appendixequation}
\label{doesthiscommute}
\begin{tikzcd}
{\rm{H}}^i(\A_k, G) \arrow[r, phantom, "\times"] & {\rm{H}}^{2-i}(\A_k, \widehat{G}) \arrow{r} & \Q/\Z \\
{\rm{H}}^i(\A_k, G) \arrow{d}{\phi} \arrow[u, equals] & {\rm{H}}^{2-i}(\A_k, \R_{k'/k}(\widehat{G'})) \arrow{d}{\phi} \arrow{u}{{\rm{H}}^2({\rm{N}}_{k'/k})} & \\
{\rm{H}}^i(\A_{k'}, G') \arrow[r, phantom, "\times"] & {\rm{H}}^{2-i}(\A_{k'}, \widehat{G'}) \arrow{r} & \Q/\Z \arrow[uu, equals]
\end{tikzcd}
\end{appendixequation} 
Further, if $k'/k$ is separable, then all of the vertical maps in this diagram are isomorphisms.
\end{appendixproposition}

\begin{proof}
The latter assertion about all of the maps being isomorphisms in the separable case is a consequence of Proposition \ref{charactersseparableweilrestriction} and our general observations at the beginning of this section about $\phi$ being an isomorphism when $f$ is \'etale. So we concentrate on the commutativity of (\ref{doesthiscommute}).

We will reduce this commutativity to a purely local question. For each place $v$ of $k$, we have a diagram: 
\begin{appendixequation}
\label{placebyplacediagram}
\begin{tikzcd}
{\rm{H}}^i(k_v, G) \arrow[r, phantom, "\times"] \arrow[d, equals] & {\rm{H}}^{2-i}(k_v, \widehat{G}) 
\arrow{r}{\cup} & {\rm{H}}^2(k_v, \Gm) \arrow{r}{\inv} & \Q/\Z \\
{\rm{H}}^i(k_v, G) \arrow[r, phantom, "\times"] \arrow{d}{\phi} & {\rm{H}}^{2-i}(k_v, \R_{k'/k}(\widehat{G'})) 
\arrow{u}{{\rm{H}}^{2-i}({\rm{N}}_{k'/k})} \arrow{d}{\phi} \arrow{r}{\cup} & 
{\rm{H}}^2(k_v, \underset{v'\mid v}{\prod} \R_{k_{v'}/k_v}(\Gm)) \arrow{u}{\prod_{v'\mid v} {\rm{H}}^2({\rm{N}}
_{k_{v'}/k_v})} \arrow{d}{\prod_{v'\mid v} \phi_{v'}} & \\
\underset{v' \mid v}{\prod}{\rm{H}}^i(k_{v'}, G') \arrow[r, phantom, "\times"] & \underset{v' \mid v}{\prod}{\rm{H}}^{2-i}
(k_{v'}, \widehat{G'}) \arrow{r}{\cup} & \underset{v' \mid v}{\prod}{\rm{H}}^2(k_{v'}, \Gm) \arrow{r}{\prod_{v' 
\mid v} \inv_{v'}} & \underset{v' \mid v}{\prod} \Q/\Z \arrow{uu}{\sum_{v' \mid v}}
\end{tikzcd}
\end{appendixequation}
Here all products are over the places $v'$ of $k'$ lying above $v$. Since the adelic pairings are obtained by cupping at each place of $k$ (respectively $k'$) and then adding the invariants, diagram {(\ref{doesthiscommute})} breaks up into a diagram for each place $v$ as above, and the above diagram breaks up factor by factor into a diagram for each place $v'$ above $v$. The proposition therefore follows from Proposition \ref{diagramcommuteslocal}.
\end{proof}

\chapter{Cohomology and Direct Limits}
\label{chaptercohomdirectlimits}

The goal of this appendix is to show that fppf cohomology behaves well with respect to direct limits, a fact that we use repeatedly in various arguments. More precisely, we have the following result.

\begin{appendixproposition}
\label{directlimitscohom}
Let $\{X_{\lambda}\}$ be a filtered inverse system of quasi-compact quasi-separated schemes with affine transition maps $\pi_{\lambda', \lambda}:  X_{\lambda'} \rightarrow X_{\lambda}$ whenever $\lambda' \geq \lambda$. For each $\lambda$, let $\mathscr{F}_{\lambda}$ be an abelian sheaf on the big fppf site of $X_{\lambda}$, and suppose that we are given compatible $($in the obvious sense$)$ transition maps $\pi_{\lambda', \lambda}^*\mathscr{F}_{\lambda} \rightarrow \mathscr{F}_{\lambda'}$ for $\lambda' \geq \lambda$. 

For $X : = \varprojlim X_{\lambda}$, and the direct limit $\mathscr{F}$ of the pullbacks of all of the $\mathscr{F}_{\lambda}$ to $X$, the natural map
\[
\varinjlim {\rm{H}}^i(X_{\lambda}, \mathscr{F}_{\lambda}) \rightarrow {\rm{H}}^i(X, \mathscr{F})
\]
of fppf cohomology groups is an isomorphism.
\end{appendixproposition}

\begin{proof}
This is \cite[Exp.\,VI, Cor.\,5.2]{sga4}.
\end{proof}

Let us discuss how Proposition \ref{directlimitscohom} applies in the context of our work. Suppose that we have $\{X_{\lambda}\}$ 
and $X$ as in the result, and are given a finitely presented $X$-group scheme $\mathscr{G}$. Then $\mathscr{G}$ descends to
compatible finitely presented $X_{\lambda}$-group schemes $\mathscr{G}_{\lambda}$ for all sufficiently large $\lambda$, 
and any fppf $X$-scheme $U$ descends to a compatible system of fppf $X_{\lambda}$-schemes $U_{\lambda}$ for all sufficiently large $\lambda$. Further, any element $g \in \mathscr{G}(U)$ descends to compatible elements $g_{\lambda} \in \mathscr{G}_{\lambda}(U_{\lambda})$, and any two such $g_{\lambda}, g_{\lambda'}$ become the same under pullback to some $U_{\lambda''}$. 

We therefore have $\mathscr{G} = \varinjlim \pi_{\lambda}^*\mathscr{G}_{\lambda}$ as sheaves, where $\pi_{\lambda}:  X \rightarrow X_{\lambda}$ is the natural map. In contrast, if we consider the sheaf for the \'etale topology represented by a non-\'etale group scheme then the pullback sheaf under a non-\'etale map is generally {\em not} represented by
the corresponding base change of the given group scheme.
By Proposition \ref{directlimitscohom}, 
it follows that ${\rm{H}}^i(X, \mathscr{G}) = \varinjlim {\rm{H}}^i(X_{\lambda}, \mathscr{G}_{\lambda})$. 

The
same reasoning applies in many situations
where we will use implicitly that cohomology behaves well with respect to suitable direct limits.
For example, by very similar reasoning we have ${\rm{H}}^i(X, \widehat{\mathscr{G}}) = \varinjlim {\rm{H}}^i(X_{\lambda}, 
\widehat{\mathscr{G}}_{\lambda})$ where the functor $\widehat{(\cdot)} = \mathscr{H}om(\cdot, \Gm)$ denotes the $\Gm$-dual
of an fppf group scheme.  This dual is generally {\em not} represented by an $X$-scheme, but it is ``locally of finite
presentation'' as a functor, and that is what matters for the preceding limit considerations to be applicable.

\chapter{Compatibility Between \v{C}ech and Derived Functor Constructions}
\label{appendixcechvsderived}

This appendix proves the compatibility between certain constructions for \v{C}ech and derived functor cohomology, and especially between two constructions of connecting maps in cohomology. This compatibility is used in order to prove results on the agreement between \v{C}ech and derived functor cohomology (Propositions \ref{cech=derivedsmoothinf} and \ref{cech=derivedG^}), as well as to prove Theorem \ref{shapairing}, since the pairings between the Tate--Shafarevich groups are defined in terms of \v{C}ech cocycles.

\section{Edge maps}

Let $X$ be a scheme, $\mathscr{F}$ an abelian sheaf for a Grothendieck topology 
on $X$ (we need the fppf and \'etale topologies), and $\mathcal{U}$ a cover of $X$. For the \v{C}ech cohomology groups $\check{\rm{H}}^i({\mathcal{U}}, \mathscr{F})$, the \v{C}ech-to-derived functor spectral sequence yields canonical edge maps 
\begin{equation}\label{edgemap}
\check{\rm{H}}^i({\mathcal{U}}, \mathscr{F}) \rightarrow {\rm{H}}^i(X, \mathscr{F})
\end{equation}
for all $i$. Our goal in this appendix is twofold:  to show that a natural direct way of constructing such maps (without appeal
to spectral sequences) coincides with these edge maps, and to show that 
connecting maps on pieces of \v{C}ech cohomology are compatible (via the canonical \v{C}ech-to-derived maps
(\ref{edgemap})) with 
connecting maps for derived functor cohomology.  Both of these compatibilities are extremely useful, but
surprisingly there does not seem to be a 
discussion of either of them in any standard (or other) reference.

We begin by discussing a useful more direct construction of the edge maps; this will be used in our discussion of compatibility for connecting maps.
Let $\mathscr{I}^{\bullet}$ be an injective resolution of $\mathscr{F}$.  The \v{C}ech complex of sheaves
$\mathscr{C}^{\bullet}(\mathcal{U}, \mathscr{F})$ is also a resolution of $\mathscr{F}$, so there is a map of such resolutions
$\mathscr{C}^{\bullet}(\mathcal{U}, \mathscr{F}) \rightarrow \mathscr{I}^{\bullet}$ that is unique up to homotopy.
Passing to global sections, this induces canonical maps (independent of all choices)
$$f^i_{\mathcal{U},\mathscr{F}}: \check{\rm{H}}^i(\mathcal{U}, \mathscr{F}) \rightarrow {\rm{H}}^i(X, \mathscr{F})$$ for all $i$.

\begin{proposition}\label{cechtoderivedcomplex}
The map $f^i_{\mathcal{U}, \mathscr{F}}$ coincides with the edge map $(\rm{\ref{edgemap}})$ for all $i$.
\end{proposition}

\begin{proof} (O.\,Gabber)
Let us keep the notation above. We will suppress the $\mathcal{U}$ subscript for notational simplicity. The \v{C}ech-to-derived functor spectral sequence is the spectral sequence associated to the double complex
\[
K : = \Gamma(\mathscr{C}^{\bullet}(\mathscr{I}^{\bullet})),
\]
where $\Gamma(\cdot)$ denotes the functor $\Gamma(X, \cdot)$. For the reader's convenience, let us define this ``double complex'' and state our conventions for the total complex ${\rm{tot}}(K)$ of $K$. 

By definition, $K^{p,q} : = \Gamma(\mathscr{C}^p(\mathscr{I}^q))$ with horizontal differentials given along
the $q$th row by the \v{C}ech differentials and vertical differentials in the $p$th column induced by applying $\mathscr{C}^p$ to the complex $\mathscr{I}^q$. Note that this is a {\em commuting} diagram
(so not the standard anti-commuting condition in the definition of ``double complex''); i.e., the following diagram commutes for all $p, q$: 
\[
\begin{tikzcd}
\Gamma(\mathscr{C}^p(\mathscr{I}^{q+1})) \arrow{r}{\partial_{{\rm{hor}}}^{p,q+1}} & \Gamma(\mathscr{C}^{p+1}(\mathscr{I}^{q+1})) \\
\Gamma(\mathscr{C}^p(\mathscr{I}^q)) \arrow{u}{\partial_{{\rm{ver}}}^{p,q}} \arrow{r}{\partial_{{\rm{hor}}}^{p,q}} & \Gamma(\mathscr{C}^{p+1}(\mathscr{I}^q)) \arrow{u}{\partial_{{\rm{ver}}}^{p+1,q}}
\end{tikzcd}
\]
The total complex ${\rm{tot}}(K)$ of $K$ is the cochain complex defined by
\[
({\rm{tot}}(K))^n : = \underset{p+q=n}{\bigoplus} \Gamma(\mathscr{C}^p(\mathscr{I}^q)),
\]
with differential given on $K^{p,q}$ by $\partial_{{\rm{hor}}}^{p,q} + (-1)^p\partial_{{\rm{ver}}}^{p,q}$. 

The edge map from $\check{H}^n(X, \mathscr{F})$ to the $n$th cohomology of ${\rm{tot}}(K)$ (which is the abutment of the spectral sequence) is defined by the edge map
\[
g:  \Gamma(\mathscr{C}^{\bullet}(\mathscr{F})) \rightarrow {\rm{tot}}(K) 
\]
defined in degree $n$ by the map $$\mathscr{C}^n(\mathscr{F})(X) = \prod_{i_0\dots i_n} \mathscr{F}(U_{i_0\dots i_n}) \rightarrow K^{n,0} = \prod_{i_0\dots i_n} \mathscr{I}^0(U_{i_0\dots i_n})$$ induced by the inclusion $\mathscr{F} \hookrightarrow \mathscr{I}^0$. That this is a map of cochain complexes follows from the fact that the differential $\mathscr{I}^0 \rightarrow \mathscr{I}^1$ kills $\mathscr{F}$. 

The $n$th homology of the total complex of $K$ is identified with ${\rm{H}}^n(X, \mathscr{F})$ in the following manner.
We have an edge map 
\[
f_2:  \Gamma(\mathscr{I}^{\bullet}) \rightarrow {\rm{tot}}(K) 
\]
defined in degree $n$ by the restriction map $\mathscr{I}^n(X) \rightarrow K^{0,n} = \prod_i \mathscr{I}^n(U_i)$. That this is a map of cochain complexes follows from the fact that the \v{C}ech complex associated to $\mathscr{I}^n$ is a complex (and in particular, a complex at the term $\prod_i \mathscr{I}^n(U_i)$). (Note that the relevant differentials are the $\partial_{{\rm{ver}}}^{0,\,n}$, and, since $p = 0$ here, the $(-1)^p$ factor disappears.) When constructing the \v{C}ech-to-derived functor spectral sequence, one shows that this map is a quasi-isomorphism
(the main point being that each $\mathscr{C}^p(\mathscr{I})$ is flasque for an injective
abelian sheaf $\mathscr{I}$), thereby identifying the homology of ${\rm{tot}}(K)$ with the derived functor cohomology of $\mathscr{F}$. 

Now we bring in the (unique up to homotopy) map $h: \mathscr{C}^{\bullet}(\mathscr{F}) \rightarrow \mathscr{I}^{\bullet}$ of resolutions of $\mathscr{F}$. Consider the double complex
\[
K' : = \Gamma(\mathscr{C}^{\bullet}(\mathscr{C}^{\bullet}(\mathscr{F})))
\]
(same conventions as before). We have a commutative diagram
\[
\begin{tikzcd}
\Gamma(\mathscr{C}^{\bullet}(\mathscr{F})) \arrow[d, equals] \arrow{r}{e'} & {\rm{tot}}(\Gamma(\mathscr{C}^{\bullet}(\mathscr{C}^{\bullet}(\mathscr{F})))) \arrow{d} & \Gamma(\mathscr{C}^{\bullet}(\mathscr{F})) \arrow{l}{e''} \arrow{d}{f_1} \\
\Gamma(\mathscr{C}^{\bullet}(\mathscr{F})) \arrow{r}{g} & {\rm{tot}}(\Gamma(\mathscr{C}^{\bullet}(\mathscr{I}^{\bullet}))) & \Gamma(\mathscr{I}^{\bullet}) \arrow{l}{f_2}
\end{tikzcd}
\] 
of cochain complexes, where the lower horizontal maps are the ``edge maps'' defined above and the 
upper horizonal maps are defined analogously using the \v{C}ech resolution of $\mathscr{F}$
in place of the injective resolution $\mathscr{I}^{\bullet}$, and the middle and right vertical maps are induced by the map 
of resolutions $h: \mathscr{C}^{\bullet}(\mathscr{F}) \rightarrow \mathscr{I}^{\bullet}$. 

Our task is to show that $f_2\circ f_1$ and $g$ induce the same maps on cohomology in each degree, so by commutativity of the diagram it suffices to show that $e'$ and $e''$ are homotopic. For the reader's convenience, let us write out formulas for all of the relevant maps.

First, $\Gamma(\mathscr{C}^n(\mathscr{F})) = \prod_{i_0\dots i_n} \mathscr{F}(U_{i_0\dots i_n})$, and for $\sigma \in \Gamma(\mathscr{C}^n(\mathscr{F}))$, we have
\[
(\partial \sigma)(i_0, \dots, i_{n+1}) = \sum_{r=0}^{n+1} (-1)^r\sigma(i_0\dots \widehat{i_r}\dots i_{n+1})|_{U_{i_0\dots i_{n+1}}}
\]
where the notation $\widehat{i_r}$ means that the index $i_r$ is omitted. The group 
$$K'^{p,q} = \Gamma(\mathscr{C}^p(\mathscr{C}^q(\mathscr{F}))$$ is naturally identified with $\prod_{i_0\dots i_p j_0\dots j_q} \mathscr{F}(U_{i_0\dots i_p j_0\dots j_q})$. The vertical and horizontal differentials at the $(p,q)$-entry are
\[
(\partial^{p,q}_{{\rm{hor}}}\sigma)(i_0, \dots, i_{p+1}, j_0, \dots, j_q) = \sum_{r=0}^{p+1} (-1)^r\sigma(i_0, \dots, \widehat{i_r}, \dots, i_{p+1}, j_0, \dots, j_q)|_{U_{i_0\dots i_{p+1}j_0\dots j_q}},
\]
\[
(\partial^{p,q}_{{\rm{ver}}}\sigma)(i_0, \dots, i_p, j_0, \dots, j_{q+1}) = \sum_{r=0}^{q+1} (-1)^r\sigma(i_0, \dots, i_p, j_0, \dots, \widehat{j_r}, \dots, j_{q+1})|_{U_{i_0\dots i_pj_0\dots j_{q+1}}}.
\]
The total complex ${\rm{tot}}(K')$ of $K'$ is therefore given by 
\[
({\rm{tot}}(K'))^n = \bigoplus_{p+q=n} K'^{p,q} = \bigoplus_{p+q=n} \Gamma(\mathscr{C}^{n+1}(\mathscr{F}))
\]
with differential given on $K'^{p,q}$ by $\partial^{p,q}_{{\rm{hor}}} + (-1)^p\partial^{p,q}_{{\rm{ver}}}$. 

The edge maps are defined as follows for $\sigma \in \Gamma(\mathscr{C}^n(\mathscr{F}))$:  $e'(\sigma)^{p,q} = 0$ for $q > 0$, $e''(\sigma)^{p,q} = 0$ for $p > 0$, and for any $n \ge 0$ we define
$$e'(\sigma)^{n,0}(i_0, \dots, i_n, j_0) = \sigma(i_0, \dots, i_n)|_{U_{i_0\dots i_nj_0}},$$
$$e''(\sigma)^{0,n}(i_0, j_0, \dots, j_n) = \sigma(j_0, \dots, j_n)|_{U_{i_0j_0\dots j_n}}.$$

Let us now define our homotopy $H$. Let $H:  \Gamma(\mathscr{C}^{n+1}(\mathscr{F})) \rightarrow ({\rm{tot}}(K'))^{n}$ 
be defined to have its component map into the $(p,n-p)$-factor 
$({\rm{tot}}(K'))^{p, n-p} = \Gamma(\mathscr{C}^{n+1}(\mathscr{F}))$ given by multiplication by $(-1)^p$. Then one checks that
for $p, q \ge 0$ with $p+q=n$, as elements of the factor $\mathscr{F}(U_{i_0\dots i_p j_0\dots j_q})$
of $\Gamma(\mathscr{C}^{n+1}(\mathscr{F}))$ we have 
\begin{eqnarray*}
(H\partial \sigma)^{p,q}(i_0, \dots, i_p, j_0, \dots, j_q) &=& (-1)^p\partial \sigma(i_0, \dots, i_p, j_0, \dots, j_q)\\
&=& (-1)^p\sum_{r=0}^p (-1)^r\sigma(i_0, \dots, \widehat{i_r}, \dots, i_p, j_0, \dots, j_q)\\
&& + (-1)^p\sum_{r=0}^q (-1)^{r+p+1}\sigma(i_0, \dots, i_p, j_0, \dots, \widehat{j_r}, \dots, j_q)\\
&=& \sum_{r=0}^p(-1)^{p+r} \sigma(i_0, \dots, \widehat{i_r}, \dots, i_p, j_0, \dots, j_q)\\
&& + \sum_{r=0}^q (-1)^{r+1}\sigma(i_0, \dots, i_p, j_0, \dots, \widehat{j_r}, \dots, j_q)
\end{eqnarray*}
(suppressing the evident restriction to $U_{i_0\dots j_q}$ for all the terms in the sums), 
and
\begin{eqnarray*}
(\partial H\sigma)^{p,q}(i_0, \dots, i_p, j_0, \dots, j_q) &=&
\sum_{r=0}^p (-1)^rH(\sigma)^{p-1,q}(i_0, \dots, \widehat{i_r}, \dots, i_p, j_0, \dots, j_q) +\\
&& (-1)^p\sum_{r=0}^q (-1)^rH(\sigma)^{p,q-1}(i_0, \dots, i_p, j_0, \dots, \widehat{j_r}, \dots, j_q)
\end{eqnarray*}
(again suppressing notation for restriction to $U_{i_0\dots j_q}$) 
with the first summand only appearing if $p >0$ and the second appearing only if $q>0$.
This final expression is equal to 
\[
-\sum_{r=0}^p(-1)^{p+r}\sigma(i_0, \dots, \widehat{i_r}, \dots, i_p, j_0, \dots, j_q) - \sum_{r=0}^q (-1)^{r+1} \sigma(i_0, \dots, i_p, j_0, \dots, \widehat{j_r}, \dots, j_q)
\]
with the same suppression of notation for restriction maps
and the same caveat that the first (resp. second) summand appears only when $p > 0$  (resp. $q>0$). 

It is now a straightforward calculation with the definitions 
and inspection of the $(p,q)$-component of the output
in each of the four cases (i) $p, q > 0$, (ii) $p>0, q=0$, (iii) $p=0, q> 0$, (iv) $p=q=0$
that (after much cancellation) we have 
$$\partial H + H\partial = e'' - e'.$$
Hence, $H$ is  a homotopy between $e'$ and $e''$.
\end{proof}

\section{Connecting maps}

Now consider a short exact sequence 
\[
0 \longrightarrow \mathscr{F}' \xlongrightarrow{j} \mathscr{F} \xlongrightarrow{\pi} \mathscr{F}'' \longrightarrow 0
\]
of abelian sheaves for a Grothendieck topology on $X$ (such as the fppf or \'etale topologies, which are all we shall need).
For a cover $\mathcal{U}$ of $X$, consider a \v{C}ech cocycle $\check{\alpha} \in \check{Z}^i(\mathcal{U}, \mathscr{F}'')$
that happens to lift to a cochain $\check{\beta} \in \check{C}^i(\mathcal{U}, \mathscr{F})$.  Of course, in general
no such lift will exist; we only consider $\check{\alpha}$ admitting such a lift.  

Define a cocycle $\check{\gamma} \in \check{Z}^{i+1}(\mathcal{U}, \mathscr{F}')$
via the snake lemma procedure.  That is, since $\pi(d\check{\beta}) = d\check{\alpha} = 0$, there is a unique cochain
$\check{\gamma} \in \check{C}^{i+1}(\mathcal{U}, \mathscr{F}')$ such that $j(\check{\gamma}) = d\check{\beta}$.
The cochain $\check{\gamma}$ is a cocycle because $j(d\check{\gamma}) = d^2\check{\beta} = 0$ and $\ker j = 0$.
 
\begin{proposition}
\label{cechderivedconnectingmap}
Letting $[c]$ denote the derived-functor cohomology class arising from a \v{C}ech cocycle $\check{c}$ via
the edge map {\rm{(\ref{edgemap})}}, for 
$\check{\alpha}$ and $\check{\gamma}$ as considered above, we have 
$\delta([\alpha]) = [\gamma]$ for the connecting map
$\delta:  {\rm{H}}^i(X, \mathscr{F}'') \rightarrow {\rm{H}}^{i+1}(X, \mathscr{F}')$.
\end{proposition}

The main step in the proof is a general homological lemma.
We use the convention that all complexes have ascending degree and 
that for any complex $C^{\bullet}$ concentrated in degrees $\geq 0$, the group ${\rm{H}}^0(C^{\bullet}) : = \ker(d^0:  C^0 \rightarrow C^1)$ denotes the homology of $C^{\bullet}$ upon augmenting by $0$.

\begin{lemma}
\label{extendmapscomplexdiagram}
Consider left-exact sequences of complexes
\[
0 \longrightarrow C'^{\bullet} \xlongrightarrow{j_C} C^{\bullet} \xlongrightarrow{\pi_C} C''^{\bullet}
\]
\[
0 \longrightarrow I'^{\bullet} \xlongrightarrow{j_I} I^{\bullet} \xlongrightarrow{\pi_I} I''^{\bullet}
\]
in an abelian category, with each complex concentrated in non-negative degree and exact in positive degree.
Assume that $I'^{\bullet}, I^{\bullet}$, and $I''^{\bullet}$ are complexes of injectives. 

For any commutative diagram of degree-$0$ kernels
\begin{equation}
\label{extendmapscomplexdiagrameqn1}
\begin{tikzcd}
0 \arrow{r} &{\rm{H}}^0(C') \arrow{r}{j_C} \arrow{d}{g'} & {\rm{H}}^0(C) \arrow{r}{\pi_C} \arrow{d}{g} & {\rm{H}}^0(C'') \arrow{d}{g''} \\
0 \arrow{r} & {\rm{H}}^0(I') \arrow{r}{j_I} & {\rm{H}}^0(I) \arrow{r}{\pi_I} & {\rm{H}}^0(I'')
\end{tikzcd}
\end{equation}
in which the horizontal maps are those induced by the sequences of complexes above, there is a commutative diagram of complexes
\[
\begin{tikzcd}
0 \arrow{r} & C'^{\bullet} \arrow{r}{j_C} \arrow{d}{f'} & C^{\bullet} \arrow{r}{\pi_C} \arrow{d}{f} & C''^{\bullet} \arrow{d}{f''} \\
0 \arrow{r} & I'^{\bullet} \arrow{r}{j_I} & I^{\bullet} \arrow{r}{\pi_I} & I''^{\bullet}
\end{tikzcd}
\]
inducing diagram ${\rm{(\ref{extendmapscomplexdiagrameqn1})}}$.
\end{lemma}

In the special case when both given left-exact sequences of complexes are short exact, this lemma is in many references on homological algebra.  But to our surprise, the preceding version with just left-exactness 
does not seem to be in any reference.  

\begin{proof}
We use the method of ``chasing members'' that permits us to make arguments in general abelian categories
using notation as if working in the category of abelian groups; see \cite[Ch.\,VIII, \S4]{Cat} for a discussion of this technique
and its application to prove the snake lemma in a general abelian category.
This method is used here so that we may chase diagrams and present a proof that is actually comprehensible (to the author, as well as to the reader!). 

It is a standard fact in homological algebra that we may construct $f': C'^{\bullet} \rightarrow I^{\bullet}$ inducing $g'$ on ${\rm{H}}^0$'s, but for the convenience of the reader we now recall how this is proved, since the method involves arguments that will be adapted
to construct the map of complexes $f$ (compatibly with $f'$) and then $f''$ (compatibly with $f$). First, let $f'^0(c') = g'(c') \in {\rm{H}}^0(I') \subset I'^0$ for $c' \in {\rm{H}}^0(C') \subset C'^0$. By injectivity of $I'^0$, this extends to a map $f'^0:  C'^0 \rightarrow I'^0$. In general, for $n > 0$ we assume that $f'^{m}$ has already been constructed for $m < n$ so that the diagram of complexes commutes up to degree $n-1$, and we
shall now inductively construct $f'^n$ to satisfy $df'^{n-1} = f'^nd$. 

First, we define $f'^n$ on $dC'^{n-1}$ by setting $f'^n(dc') : = df'^{n-1}(c')$ for $c' \in C'^{n-1}$. Assuming that this is well-defined, we may use injectivity of $I'^n$ to extend this to a map $f'^n:  C'^n \rightarrow I'^n$ that satisfies the desired property by construction. To show well-definedness, we must check that if $c' \in C'^{n-1}$ satisfies $dc' = 0$, then $df'^{n-1}(c') = 0$. 
 If $n = 1$ then by design $f'^0$ carries $c'$ to a cocycle in $I'^0$, so we get what we want. 
 Next suppose $n > 1$, so $C'^{\bullet}$ is exact at $C'^{n-1}$ by hypothesis. Hence, $c' = db'$ for some $b' \in C'^{n-2}$. But then $df'^{n-1}(c') = df'^{n-1}(db') = d^2f'^{n-2}(b') = 0$, as desired, completing the construction of $f'$.

Next we construct $f$. We once again do this inductively. Let us first construct $f^0$. We need $f^0$ to satisfy $f^0j_C(c') = j_If'^0(c')$ for $c' \in C'^0$ and $f^0(c) = g(c)$ for $c \in {\rm{H}}^0(C)$. Assuming that these yield a well-defined map on the span of ${\rm{H}}^0(C)$ and $j_C(C'^0)$
inside $C^0$, by injectivity of $I^0$ it extends to a map $f^0:  C^0 \rightarrow I^0$ with the desired properties. For well-definedness, we need to check that if $c' \in C'^0$ satisfies $dj_C(c') = 0$ (i.e., $c: = j_C(c') \in {\rm{H}}^0(C)$), then necessarily $j_If'^0(c') = gj_C(c')$. Note first that $c' \in {\rm{H}}^0(C')$ because $j_C(dc') = dj_C(c') = 0$ (and $\ker j_C = 0$). Since $gj_C(c') = j_Ig'(c')$ (as (\ref{extendmapscomplexdiagrameqn1}) commutes)
and $g'(c') = f'^0(c')$ inside $I'^0 \supset {\rm{H}}^0(I)$ by design of $f'^0$, the well-definedness is proved.

We now construct $f^n$ for $n > 0$ inductively, assuming that the maps $\{f^m\}_{m<n}$ have been constructed 
satisfying all the required commutativity properties. The two properties that we need $f^n$ to satisfy are
\begin{itemize}
\item $f^n(dc) = df^{n-1}(c)$ for $c \in C^{n-1}$, 
\item $f^nj_C(c') = j_If'^n(c')$ for $c' \in C'^n$.
\end{itemize}
We want to simply declare $f^n$ to satisfy these two properties on the span of $j_C(C'^n)$ and $d(C^{n-1})$ inside $C^n$. Assuming that this is well-defined, injectivity of $I^n$ implies that $f^n$ extends to a map $C^n \rightarrow I^n$ that has the required properties. To check well-definedness, consider the following commutative diagram (in which the phantom map $f^n$ is yet to be defined): 
\begin{equation}
\label{commutativecube1}
 \begin{tikzcd}[row sep={40}, column sep={40}]
      & C'^{n-1} \ar{rr}{d} \ar[dd, pos = 0.8, "f'^{n-1}"] \ar{dl}[swap]{j_C} & & C'^n \ar{dd}{f'^n} \ar{dl}{j_C} \\
    C^{n-1} \ar[rr, crossing over, pos = 0.7, "d"] \ar{dd}{f^{n-1}} && C^n \ar[dd, dashed, crossing over, pos = 0.3, "f^n"] \\
      & I'^{n-1} \ar[rr, pos = 0.3, swap, "d"] \ar{dl}{j_I} & &  I'^n \ar{dl}{j_I} \\
      I^{n-1} \ar{rr}{d} && I^n 
    \end{tikzcd}
\end{equation}
We need to show for $c \in C^{n-1}, c' \in C'^n$ satisfying $dc = j_C(c')$ that $df^{n-1}(c) = j_If'^n(c')$. 

First, we claim $c' = db'$ for some $b' \in C'^{n-1}$. By exactness of $C'^{\bullet}$ at $C'^n$ (as $n>0$), it suffices to show that $dc' = 0$. But $j_Cd(c') = dj_C(c') = d^2c = 0$, so since $j_C^{n+1}$ is an inclusion, we see that $dc' = 0$ as desired. Our problem can now be recast
in terms of $b' \in C'^{n-1}$ such that $dc = dj_C(b')$:  we need to show that $df^{n-1}(c) = j_If'^ndb'$. Since $c$ differs from $j_C(b')$ by a cocycle, and both the hypothesis and the desired conclusion are additive in $c$ and $b'$, it suffices to treat the cases when either $c = j_C(b')$ or $dc = 0$ and $b' = 0$. 
That is, we want to show that $df^{n-1}j_C(b') = j_If'^ndb'$ for $b' \in C'^{n-1}$ and $df^{n-1}(c)=0$ when $dc=0$.

The first of these follows from the commutativity of (\ref{commutativecube1}). For the second, we have two subcases:  $n=1$ and $n>1$. If $n = 1$, then $f^0$ maps ${\rm{H}}^0(C)$ into ${\rm{H}}^0(I)$ 
(via $g$) by design, so we have what we need. If $n > 1$, then $C^{\bullet}$ is exact at $C^{n-1}$, so $c = db$ for some $b \in C^{n-2}$. Therefore, $df^{n-1}(c) = df^{n-1}d(b) = d^2f^{n-2}(b) = 0$. This completes the construction of $f$.

Finally, we construct $f''$, once again proceeding inductively. The construction of $f''^0$ will slightly differ from that of $f^0$ because $\ker \pi_C \ne 0$
(in contrast with $\ker j_C$), though it will still use the commutativity of (\ref{extendmapscomplexdiagrameqn1}). 
Carrying out the usual procedure to try to make a well-defined construction on the span of $\pi_C(C^0)$ and ${\rm{H}}^0(C'')$ inside
$C''^0$ (and then using that $I''^0$ is injective), we just have to show that if $c \in C^0$ satisfies
$c'' : = \pi_C(c) \in {\rm{H}}^0(C)$, then $\pi_I(f^0(c)) = g''(c'')$.  

Since $\pi_C(dc) = d\pi_C(c) = dc'' = 0$ (as $c'' \in {\rm{H}}^0(C'')$ by assumption), 
we have $dc = j_C(c')$ for some $c' \in C'^1$ by left-exactness. Then $j_C(dc') = dj_C(c') = d^2c = 0$, so the vanishing of
$\ker j_C$ in all degrees yields that $dc' = 0$.  By exactness of $C'^{\bullet}$ in all positive degrees, it follows
that $c' = db'$ for some $b' \in C'^0$, so 
$$dc = j_C(c') = j_C(db') = dj_C(b').$$
In other words, $b : = c - j_C(b') \in {\rm{H}}^0(C)$, so $f^0(b) = g(b)$ by design of $f^0$.  Thus, 
$$\pi_I(f^0(c)) = \pi_I(f^0(j_C(b')+b)) = \pi_I(f^0(j_C(b'))) + \pi_I(g(b)) = \pi_I (j_I(f'^0(b))) + \pi_I(g(b))$$
(as $f \circ j_C = j_I \circ f'$ by construction). 
But $\pi_I \circ j_I = 0$, and the commutativity of
(\ref{extendmapscomplexdiagrameqn1}) gives that $\pi_I(g(b)) = g''(\pi_C(b))$. Hence,
$\pi_I(f^0(c)) = g''(\pi_C(b))$, so we want to show that $g''(\pi_C(b)) = g''(c'')$.  
By definition, $c'' = \pi_C(c)$. Since $\pi_C(c) = \pi_C(j_C(b')+b)= \pi_C(b)$ (as $\pi_C \circ j_C = 0$), the construction of $f''^0$ is complete. 

Now suppose that $n > 0$ and that $\{f''^m\}_{m<n}$ have been constructed satisfying
all the desired properties. The requirements on $f''^n$ are: 
\begin{itemize}
\item $f''^n(dc'') = df''^{n-1}(c'')$ for $c'' \in C''^{n-1}$, 
\item $f''^n\pi_C(c) = \pi_If^n(c)$ for $c \in C^n$. 
\end{itemize}
Once again we want to define $f''^n$ to satisfy these on the span of  $d(C''^{n-1})$ and $\pi_C(C^n)$ inside $C''^n$. Assuming this is well-defined, we can use the injectivity of $I''^n$ to extend $f''^n$ to a map $C''^n \rightarrow I''^n$ (again denoted $f''^n$) that meets our needs. 
It therefore only remains to check this well-definedness. 

We need to show that if members $c'' \in C''^{n-1}$ and $c \in C^n$ satisfy $dc'' = \pi_C(c)$, then $df''^{n-1}(c'') = \pi_If^n(c)$. We first claim that $c = db + j_C(c')$ for some $b \in C^{n-1}, c' \in C'^n$. In fact, we claim that this follows from a diagram chase in the following commutative diagram with exact rows and columns, using the fact that $\pi_C(c) \in d(C''^{n-1})$: 
\[
\begin{tikzcd}
& C'^n \arrow{r} \arrow{d}{j_C} & C'^{n+1} \arrow{r} \arrow{d}{j_C} & C'^{n+2} \arrow[d, hookrightarrow] \\
C^{n-1} \arrow{d} \arrow{r}{d} & C^n \arrow{r} \arrow{d}{\pi_C} & C^{n+1} \arrow{d}{\pi_C} \arrow{r} & C^{n+2} \\
C''^{n-1} \arrow{r}{d} & C''^n \arrow{r} & C''^{n+1}
\end{tikzcd}
\]
Namely, $\pi_C(dc) = d\pi_C(c) \in d^2(C''^{n-1}) = 0$, so $dc = j_C(b')$ for some $b' \in C'^{n+1}$.  Necessarily, $db'=0$
since $j_C(db') = dj_C(b')=d^2c=0$ and $\ker j_C=0$, so $b' = dc'$ for some $c' \in C'^n$ since $C'^{\bullet}$ is exact in positive degrees.
It follows that $dj_C(c') = j_C(dc') = j_C(b') = dc$, so $c - j_C(c') \in \ker d^n = {\rm{im}} d^{n-1}$, which is to say
$c = db + j_C(c')$ for some $b \in C^{n-1}$ and $c' \in C'^n$, as desired. 

With $c$ now expressed in the asserted form $db + j_C(c')$, the hypothesis linking $c''$ and $c$ can be rewritten as
$dc'' = \pi_C(db) = d\pi_C(b)$ (since $\pi_C \circ j_C = 0$), and the desired conclusion can be rewritten as
$$df''^{n-1}(c'') \stackrel{?}{=} \pi_If^n(db+j_C(c')) = d\pi_I f^{n-1}(b) + \pi_Ij_If'^n(c') = d\pi_If^{n-1}(b)$$
since $\pi_I \circ j_I = 0$.   In other words, $c'' - \pi_C(b) \in \ker d^{n-1}_{C''}$ and we want to show that
$f''^{n-1}(c'') - \pi_I f^{n-1}(b) \in \ker d^{n-1}_{I''}$. But $\pi_I f^{n-1} = f''^{n-1}\pi_C$, so the desired
conclusion says $f''^{n-1}(c'' - \pi_C(b)) \in \ker d^{n-1}_{I''}$. Hence, it suffices to show that $f''^{n-1}$ carries
$\ker d^{n-1}_{C''}$ into $\ker d^{n-1}_{I''}$.  If $n=1$ then we use that $f''^0$ carries ${\rm{H}}^0(C'')$ into ${\rm{H}}^0(I'')$
(via $g''$) by design.  If $n>1$ then by exactness of $C''^{\bullet}$ and $I''^{\bullet}$ in positive degrees,
it is the same to show that $f''^{n-1}$ carries ${\rm{im}} d^{n-2}_{C''}$ into ${\rm{im}} d^{n-2}_{I''}$.
But this is obvious because $f''^{n-1}d = df''^{n-2}$. This completes the construction of $f''$ and the proof of the lemma.
\end{proof}

We may now complete the proof of Proposition \ref{cechderivedconnectingmap}. 
By standard homological algebra, we may construct a short exact sequence of complexes
\[
0 \longrightarrow \mathscr{I}'^{\bullet} \longrightarrow \mathscr{I}^{\bullet} \longrightarrow \mathscr{I}''^{\bullet} \longrightarrow 0
\]
such that $\mathscr{I}'^{\bullet}, \mathscr{I}^{\bullet}, \mathscr{I}''^{\bullet}$ are injective resolutions of $\mathscr{F}', \mathscr{F}, \mathscr{F}''$ respectively, and such that the maps on $0$th homologies are the given maps between these sheaves \cite[Tag 013T]{stacks}. For the 
left-exact sequence of \v{C}ech complexes
\[
0 \longrightarrow \mathscr{C}^{\bullet}(\mathcal{U}, \mathscr{F}) \longrightarrow \mathscr{C}^{\bullet}(\mathcal{U}, \mathscr{F}') \longrightarrow \mathscr{C}^{\bullet}(\mathcal{U}, \mathscr{F}''), 
\]
Lemma \ref{extendmapscomplexdiagram} inserts it into a commutative diagram of complexes
\[
\begin{tikzcd}
0 \arrow{r} & \mathscr{C}^{\bullet}(\mathcal{U}, \mathscr{F}) \arrow{r} \arrow{d} & \mathscr{C}^{\bullet}(\mathcal{U}, \mathscr{F}') \arrow{r} \arrow{d} & \mathscr{C}^{\bullet}(\mathcal{U}, \mathscr{F}'') \arrow{d} & \\
0 \arrow{r} & \mathscr{I}'^{\bullet} \arrow{r} & \mathscr{I}^{\bullet} \arrow{r} & \mathscr{I}''^{\bullet} \arrow{r} & 0
\end{tikzcd}
\]
This yields a commutative diagram
\begin{equation}
\begin{tikzcd}
0 \arrow{r} & \Gamma(X, \mathscr{C}^{\bullet}(\mathcal{U}, \mathscr{F})) \arrow{r} \arrow{d} & \Gamma(X, \mathscr{C}^{\bullet}(\mathcal{U}, \mathscr{F}')) \arrow{r} \arrow{d} & \Gamma(X, \mathscr{C}^{\bullet}(\mathcal{U}, \mathscr{F}'')) \arrow{d} & \\
0 \arrow{r} & \Gamma(X, \mathscr{I}'^{\bullet}) \arrow{r} & \Gamma(X, \mathscr{I}^{\bullet}) \arrow{r} & \Gamma(X, \mathscr{I}''^{\bullet}) \arrow{r} & 0
\end{tikzcd}
\end{equation}
(exactness at the right on the bottom due to the injectivity of the sheaves $\mathscr{I}'^n$). The connecting map in derived functor cohomology is obtained by applying the snake lemma construction to the bottom sequence in this diagram, and the snake lemma construction between the top and bottom sequences is compatible. Thus, Proposition \ref{cechderivedconnectingmap} follows immediately upon applying Proposition
\ref{cechtoderivedcomplex}. \hfill \qed

\chapter{Characteristic $0$}
\label{char0appendix}

Here we include a short discussion of how the proofs of our main results must be modified in characteristic $0$ (that is, for $p$-adic and number fields). Our goal is not to give full proofs, or even too many details, but rather to indicate how these arguments generally go, and why they are typically much simpler.

\begin{appendixdefinition}
\label{finitepresentationdef}
Let $S$ be a scheme, and let $\mathscr{F}:  {\rm{Sch}}/S \rightarrow {\rm{Ab}}$ be a contravariant functor from the category of $S$-schemes to the category of abelian groups. Then $\mathscr{F}$ is said to be {\em locally of finite presentation} if for every filtered inverse system $\{X_i = \Spec(R_i)\}$ of $S$-schemes that are affine, the canonical map $\varinjlim_i \mathscr{F}(R_i) \rightarrow \mathscr{F}(\varinjlim_i R_i)$ is an isomorphism.
\end{appendixdefinition}

All of the sheaves that we deal with in this work are locally of finite presentation. The reason for the terminology is that if $\mathscr{F}$ is represented by an $S$-scheme $X$, then $X$ is locally of finite presentation over $S$ precisely when $\mathscr{F}$ is locally of finite presentation \cite[IV$_3$, Prop.\,8.14.2]{ega}.

Let us remark that we may replace fppf with \'etale cohomology in all of our results in characteristic $0$. This actually does not change the content, thanks to the following proposition.

\begin{appendixproposition}
\label{etrem}
Let $k$ be a perfect field, and let $\mathscr{F}$ be an fppf abelian sheaf on 
the category of {\em all} $k$-schemes. Assume that $\mathscr{F}$ is locally of finite presentation. $($See Definition $\ref{finitepresentationdef}$.$)$ Then the natural map
\[
{\rm{H}}^i_{\et}(k, \mathscr{F}) \rightarrow {\rm{H}}^i_{{\rm{fppf}}}(k, \mathscr{F})
\]
is an isomorphism.
\end{appendixproposition}

\begin{proof}
For any finite Galois extension $L/k$, we have the Leray spectral sequence
\[
E_{2,\,L}^{i, j} = {\rm{H}}^i(\Gal(L/k), {\rm{H}}_{{\rm{fppf}}}^j(L, \mathscr{F})) \Longrightarrow {\rm{H}}^{i+j}_{\rm{fppf}}(k, \mathscr{F}),
\]
and similarly for \'etale cohomology. Since $k$ is perfect, $\overline{k} = \varinjlim L$, where the limit is over all finite Galois extensions $L/k$ contained in $\overline{k}$. Taking the directed limit over all such $L$, and using Proposition \ref{directlimitscohom} and the fact that $\mathscr{F}$ is locally of finite presentation, we obtain a spectral sequence
\[
E_{2,\,{\rm{fppf}}}^{i, j} = {\rm{H}}^i(\Gal(\overline{k}/k), {\rm{H}}_{{\rm{fppf}}}^j(\overline{k}, \mathscr{F})) \Longrightarrow {\rm{H}}^{i+j}_{\rm{fppf}}(k, \mathscr{F}),
\]
and similarly for \'etale cohomology. It therefore suffices to show that $E_{2,\,{\rm{fppf}}}^{i, j} = E_{2,\,\et}^{i, j}$ for all $i, j$. This is obvious for $j = 0$. For $j > 0$, we in fact have $E_{2,\,{\rm{fppf}}}^{i, j} = E_{2,\,\et}^{i, j} = 0$. Indeed, ${\rm{H}}^j_{{\rm{fppf}}}(\overline{k}, \mathscr{F}) = {\rm{H}}^j_{\et}(\overline{k}, \mathscr{F}) = 0$ for $j > 0$ because ${\rm{H}}^0(\overline{k}, \mathscr{F})$ is an exact functor on the categories of fppf or \'etale abelian sheaves on $\Spec(\overline{k})$, since any fppf or \'etale cover of $\Spec(\overline{k})$ may be refined by the trivial cover $\Spec(\overline{k}) \rightarrow \Spec(\overline{k})$ due to the Nullstellensatz.
\end{proof}

We also note that the topological issues that appear when stating Theorem \ref{H^1localduality} do not show up in characteristic $0$. In fact, we have the following simpler statement.

\begin{appendixtheorem}
\label{H^1localdualitychar0}
Let $k$ be a local field of characteristic $0$, $G$ an affine commutative $k$-group scheme of finite type. Then ${\rm{H}}^1(k, G)$ and ${\rm{H}}^1(k, \widehat{G})$ are finite, and the cup product pairing
\[
{\rm{H}}^1(k, G) \times {\rm{H}}^1(k, \widehat{G}) \rightarrow {\rm{H}}^2(k, \Gm) \underset{\inv}{\xrightarrow{\sim}} \Q/\Z
\]
is perfect.
\end{appendixtheorem}

For completeness, we outline the proof, skipping the steps which are identical to the function field case. See the proof of Lemma \ref{H^1localdualityalmosttori}.

\begin{proof}
Note that all cohomology may be taken to be \'etale by Proposition \ref{etrem}. First assume that $G$ is an almost-torus. By Lemma \ref{almosttorus}(iv), after harmlessly modifying $G$ we may assume that there is an exact sequence
\[
1 \longrightarrow A \longrightarrow B \times \R_{k'/k}(\Gm^n) \longrightarrow G \longrightarrow 1,
\]
with $A$ and $B$ finite $k$-group schemes, and $k'/k$ a finite separable extension. The finiteness of ${\rm{H}}^1(k, G)$ then follows from that of ${\rm{H}}^2(k, A)$ and ${\rm{H}}^1(k, B)$, the finiteness of which form part of Tate local duality. (It is this finiteness of ${\rm{H}}^1(k, B)$ that fails in general over local function fields. For the finiteness of ${\rm{H}}^1(k, B)$ in particular, see \cite[\S5, Prop. 14]{serre}.) Applying Propositions \ref{charactersseparableweilrestriction} and \ref{hatisexact} (the latter of which in characteristic $0$ follows from Proposition \ref{ext0}, a much easier result than the positive characteristic Proposition \ref{ext=0}), we obtain an exact sequence of \'etale sheaves
\[
1 \longrightarrow \widehat{G} \longrightarrow \widehat{B} \times \R_{k'/k}(\Z^n) \longrightarrow \widehat{A} \longrightarrow 1.
\]
Since $\widehat{A}(k)$ and ${\rm{H}}^1(k, \widehat{B})$ are finite (the latter finiteness again failing in general for finite commutative group schemes over local function fields), so is ${\rm{H}}^1(k, \widehat{G})$. Finally, the proof that the pairing between ${\rm{H}}^1(k, G)$ and ${\rm{H}}^1(k, \widehat{G})$ is perfect proceeds exactly as in characteristic $p$, but is actually slightly simpler because one may ignore the topological issues arising in positive characteristic.

Now consider an arbitrary affine commutative $k$-group $G$ of finite type. Applying Lemma \ref{affinegroupstructurethm}, there is an exact sequence
\[
1 \longrightarrow X \longrightarrow G \longrightarrow U
\]
with $X$ an almost-torus and $U$ split unipotent. Then we claim that we have a commutative diagram in which the indicated maps are isomorphisms or surjections: 
\begin{appendixequation}
\label{H^1localdualitychar0pfeqn}
\begin{tikzcd}
{\rm{H}}^1(k, X) \isoarrow{d} \arrow[r, twoheadrightarrow] & {\rm{H}}^1(k, G) \arrow{d} \\
{\rm{H}}^1(k, \widehat{X})^* \arrow{r}{\sim} & {\rm{H}}^1(k, \widehat{G})^*
\end{tikzcd}
\end{appendixequation}
Assuming this, the finiteness assertions as well as the perfection for $G$ follow from the corresponding assertions for $X$.

The diagram commutes simply due to the functoriality of cup product. The left vertical map is an isomorphism due to the perfection of the pairing for $X$. The top horizontal arrow is surjective because ${\rm{H}}^1(k, U) = 0$ by Proposition \ref{unipotentcohomology}(ii). Next, ${\rm{H}}^i(k, \widehat{U}) = 0$ for $i = 1, 2$ by Propositions \ref{hatisexact} and \ref{cohomologyofG_adualwhenkisperfect}. (Note that the vanishing for $i = 2$ fails completely over local function fields; see Proposition \ref{omega=H^2(Ga^)}.) It follows from Proposition \ref{hatisexact} that the map ${\rm{H}}^1(k, \widehat{G}) \rightarrow {\rm{H}}^1(k, \widehat{X})$ is an isomorphism, hence so is the bottom horizontal arrow in (\ref{H^1localdualitychar0pfeqn}).
\end{proof}

The local Theorems \ref{localdualityH^2(G)}-\ref{H^1localduality} all hold for $\mathbf{R}$ and $\C$ if one replaces ordinary cohomology with Tate cohomology. They also hold (with the usual cohomology groups) over non-archimedean local fields of characteristic $0$. Further, the main global duality theorems \ref{poitoutatesequence}--\ref{poitoutatesequencedual} also all hold for affine commutative group schemes of finite type over number fields.

Now we discuss how the arguments typically proceed in order to prove our main results for an affine commutative $G$ of finite type. One typically applies a d\'evissage from the finite (necessarily \'etale!) case to reduce to the case when $G$ is connected (and necessarily smooth). Such $G$ is of the form $T \times \Ga^n$ for some $k$-torus $T$ and some $n \geq 0$. We are therefore reduced to the cases when $G$ is a torus or $\Ga$. The proofs for tori are typically the same as in the function field case, reducing to the case $T = \Gm$ by applying Lemma \ref{almosttorus}(iv) or by using the injection $T \hookrightarrow \R_{k'/k}(T_{k'})$ for some finite separable extension $k'/k$ splitting $T$. 

For $\Ga$, on the other hand, the proofs are much simpler. Indeed, as in characteristic $p$, we have ${\rm{H}}^i(k, \Ga) = {\rm{H}}^i(\A, \Ga) = 0$ for all $i > 0$, since $\Ga$ is a quasi-coherent sheaf and the higher \'etale cohomology of quasi-coherent sheaves on any affine scheme vanishes. But what truly makes the situation much simpler is that, in characteristic $0$, we have ${\rm{H}}^i_{\et}(k, \widehat{\Ga}) = {\rm{H}}^i_{\et}(\A, \widehat{\Ga}) = 0$ for all $i$, since $\widehat{\Ga} = 0$ as an \'etale sheaf on these schemes because $\Ga$ has no nontrivial characters over a reduced ring (as follows from Lemma \ref{nonconstantunitsred}). So the cohomology of $\widehat{\Ga}$ -- which, over $k$, may by Proposition \ref{etrem} be taken to be \'etale or fppf without changing its value -- is both simpler and easier to compute. 

To prove Theorem \ref{localdualityH^2(G^)} for $G = \Ga$ over a characteristic-$0$ local field $k$, for example, one must show that $k_{\pro} = 0$. But this is clear, since $k$ is divisible, hence has no nontrivial finite quotients (since any such would be finite and divisible). Similarly, if $k$ is a number field, then $k_{\pro} = 0$, so in order to prove that the sequence in Theorem \ref{poitoutatesequence} is exact at $\Ga(\A)$, for example, one must show that $\A_{\pro} = 0$, and this once again follows from the fact that $\A$ is divisible.


\newpage

\begin{thebibliography}{ram}

\bibitem[SGA4$_{\rm{II}}$]{sga4} M.\,Artin, A.\,Grothendieck, J.-L.\,Verdier, {\em Th\'eorie des topos et cohomologie \'etale des
sch\'emas} (Tome 2), Lecture Notes in Math {\bf 270}, Springer-Verlag, 1972.

\bibitem[Ar]{artin} M.\,Artin, {\em Versal deformations and algebraic stacks}, Inventiones {\bf 27} (1974), pp.\,165--189.

\bibitem[BS]{borelserre} A.\,Borel, J.-P.\,Serre, {\em, Th\'eor\`emes de finitude en cohomologie galoisienne}, Comment.\, Math.\, Helv.\, {\bf 39} (1964), 111--164.

\bibitem[Bor]{borel} A.\,Borel, {Linear Algebraic Groups}, Second Edition, Springer-Verlag, New York, 1991.

\bibitem[Bou]{bourbakitopology} N.\,Bourbaki, {\em Topologie G{\'e}n{\'e}rale}, Springer-Verlag, Berlin, 2007.

\bibitem[Br1]{breen2} L. \, Breen, {\em On A Nontrivial Higher Extension of Representable Abelian Sheaves}, Bulletin of the American Mathematical Society, vol. 75, no. 6 (1969), pp. \, 1249-1253.

\bibitem[Br2]{breen} L.\,Breen, {\em Un th\'eor\`eme d'annulation pour certain $\Ext^i$ de faisceaux ab\'eliens},
Annales scientifiques de l'E.N.S.\,8(3) (1975), pp.\,339-352.

\bibitem[CF]{cf} J.W.S.\,Cassels, A.\,Fr\"{o}hlich, {\em Algebraic Number Theory}, Academic Press, 1993.

\bibitem[\v{C}es1]{ces1} K.\,\v{C}esnavi\v{c}ius, {\em Topology on Cohomology of Local Fields}, Forum of Mathematics, Sigma,
Vol.\,3 (2015), pp.\,1--55.

\bibitem[\v{C}es2]{ces2} K.\,\v{C}esnavi\v{c}ius, {\em Poitou--Tate Without Restrictions on the Order}, Mathematical Research Letters, \textbf{22} (2015), no.\,6, 1621-1666.

\bibitem[CCO]{cco} C.-L.\,Chai, B.\,Conrad, F.\,Oort, {\em Complex Multiplication and Lifting Problems}, American Mathematical Society, Providence, Rhode Island, 2014.

\bibitem[Che]{chevalley} C.\,Chevalley, {\em Deux Th{\'e}or{\`e}mes d'Arithm{\'e}tique}, J. Math. Soc. Japan Volume 3, Number 1 (1951), 36-44.

\bibitem[Con1]{conradchevalley} B.\,Conrad, {\em A modern proof of Chevalley's theorem on algebraic groups}, Journal of the Ramanujan Mathematical Society, {\bf 17} (1), 1--18.

\bibitem[Con2]{chowtrace} Brian Conrad, {\em Chow's $K/k$-image and $K/k$-trace, and the Lang--N\'eron  theorem}, Enseign.\,Math.\,(2) {\bf 52} (2006), 37-108.

\bibitem[Con3]{conradaffinequotients} B.\,Conrad {\em Affine Quotients}, available at {\tt{http: //virtualmath1.stanford.edu/~conrad/252Page/handouts/affineqt.pdf}}.

\bibitem[Con4]{conrad} B.\,Conrad, {\em Finiteness theorems for algebraic groups over function fields},
Compositio Math.\,{\bf 148} (2012), pp.\,555-639.

\bibitem[CGP]{cgp} B.\,Conrad, O.\,Gabber, G.\,Prasad, {\em Pseudo-reductive Groups}, Cambridge Univ.\,Press (2nd edition), 2015.

\bibitem[DG]{demazuregabriel} M.\,Demazure, P.\,Gabriel, {\em Groupes Alg\'ebriques}, Masson \& Cie, Paris, 1970.

\bibitem[SGA3]{sga3} M.\,Demazure, A.\,Grothendieck, {\em Sch\'emas en Groupes}
I, II, III, Lecture Notes in Math {\bf 151, 152, 153}, Springer-Verlag, New York (1970).

\bibitem[Ga]{gabber} O.\,Gabber, {\em Some Theorems On Azumaya Algebras}, in ``Groupe de Brauer, Séminaire, Les Plans-sur-Bex, Suisse 1980'', Lecture Notes in Mathematics, Vol 844 (1981), pp.\,129-209.

\bibitem[GS]{gille} P.\,Gille, T.\,Szamuely, {\em Central Simple Algebras and Galois Cohomology}, Cambridge Univ.\,Press, New York,
2006.

\bibitem[EGA]{ega} A.\,Grothendieck, {\em El\'ements de G\'eom\'etrie Alg\'ebrique},
Publ.\,Math.\,IHES {\bf 4, 8, 11, 17, 20, 24, 28, 32}, 1960--7.

\bibitem[SGA7]{sga7} A.\,Grothendieck, D.S.\,Rim, {\em Groupes de Monodromie en G\'eom\'etrie Alg\'ebriques}, Lecture
Notes in Mathematics 288, Springer--Verlag, Heidelberg, 1972.

\bibitem[BrII]{brii} A.\,Grothendieck, {\em Le groupe de Brauer II: Th\'eories Cohomologiques}, in ``Dix Expos\'es sur la cohomologie des sch\'emas'', North-Holland, Amsterdam, 
1968, pp.\,66--87.

\bibitem[BrIII]{briii} A.\,Grothendieck, {\em Le groupe de Brauer III:  exemples et compl\'ements},
in ``Dix Expos\'es sur la cohomologie des sch\'emas'', North-Holland, Amsterdam, 
1968, pp.\,88--188.

\bibitem[Ka]{kato} K.\,Kato, {\em Galois Cohomology of Complete Discrete Valuation Fields}, Lecture Notes in Mathematics, {\bf 967}, 215-238, Springer--Verlag, 1982.

\bibitem[Kat]{katz} N.\,Katz, {\em $p$-adic Properties of Modular Schemes and Modular Forms}, In Willem Kuijk and Jean-Pierre Seerre, editors, {\em Modular Functions in One Variable III}, volume 350, Springer, Berlin, Heidelberg, 1973.

\bibitem[Kn]{knutson} D.\,Knutson, {\em Algebraic Spaces}, LNM 203, Springer-Verlag, New York, 1971.

\bibitem[Mac]{Cat} S.\,MacLane, {\em Categories for the Working Mathematician} (2nd ed.), Graduate Texts in Mathematics {\bf 5},
Springer--Verlag, 1978.

\bibitem[Mat1]{matsumuracommalg} H.\,Matsumura, {\em Commutative Algebra}, Second Edition, Benjamin/Cummings Publishing Company, Inc.\,1980.

\bibitem[Mat2]{crt} H.\,Matsumura, {\em Commutative Ring Theory}, Cambridge Univ.\,Press, 1990.

\bibitem[Me]{messing} W.\,Messing, {\em The Crystals associated to Barsotti--Tate Groups
with applications to Abelian Schemes}, LNM {\bf 264}, Springer--Verlag, 1972.

\bibitem[Mi1]{milneetalecohomology} {\em \'Etale Cohomology}, Princeton University Press, 1980.

\bibitem[Mi2]{milne} J.\,Milne, {\em Arithmetic Duality Theorems}, Booksurge, LLC (2nd ed.), 2006.

\bibitem[Mi3]{milnealggroups} J.\,Milne, {\em Algebraic Groups:  The Theory of Group Schemes of Finite Type over a Field}, Cambridge Univ.\,Press, New York, 2017.

\bibitem[Mor]{morel} S.\,Morel, {\em Adic spaces}, available at {\tt{https://web.math.princeton.edu/~smorel/adic\_notes.pdf}}.

\bibitem[Oes]{oesterle}
J.\,Oesterl\'e, {\em Nombres de Tamagawa et Groupes Unipotents
en Caract\'eristique $p$},
Inv.\: Math.\,{\bf 78}(1984), 13--88.

\bibitem[Ono]{ono} T.\,Ono, {\em Arithmetic of Algebraic Tori}, Annals of Mathematics {\bf{74}}, 1961, pp.\,101--139.

\bibitem[Poo]{poonen} B.\,Poonen, {\em Rational Points on Varieties}, Graduate Studies in Mathematics, Vol.\,186, 2017.

\bibitem[Rib]{ribenboim} P.\,Ribenboim, {\em The Theory of Classical Valuations}, Springer, New York, 1999.

\bibitem[Ru]{rudin} W.\,Rudin, {\em Fourier Analysis On Groups}, John Wiley and Sons, New York, 1962.

\bibitem[Sch]{schwede} K.\,Schwede, {\em Gluing schemes and a scheme without closed points}, Recent Progress in Arithmetic and Algebraic Geometry, AMS Contemporary Mathematics Series (2005).

\bibitem[Ser1]{serremodularforms} J.P.\,Serre, {\em Modular forms of weight one and Galois representations}, in ``Algebraic number fields:  L-functions and Galois properties'' (Proc. Sympos., Univ. Durham, Durham, 1975), pp. 193-268. Academic Press, London, 1977.

\bibitem[Ser2]{serre} J-P.\,Serre, {\em Galois Cohomology}, Springer-Verlag, New York, 1997.

\bibitem[Stacks]{stacks} {\em The Stacks Project}, {\tt{https: //stacks.math.columbia.edu}}.

\bibitem[Sw]{swan} R.\,Swan, {\em On Seminormality}, Journal of Algebra {\bf{67}} (1980), pp.\,210--229.

\bibitem[Ta]{tateduality} J.\,Tate, {\em Duality Theorems in Galois cohomology over number fields}. Proc. Intern. Congress 1962, Stockholm 1963, pp. 288-295.

\bibitem[Tem]{temkin} M.\,Temkin, {\em Stable modification of relative curves}, J.\, Algebraic Geom.\, 19 (2010), 603--677.

\bibitem[Tre]{treger} R.\,Treger, {\em On $p$-torsion In Etale Cohomology And In The Brauer Group}, Proc. Amer. Math. Soc. 78, No.\, 2 (1980), 189-192.

\bibitem[Weil]{weilbasic} A.\,Weil, {\em Basic Number Theory} (3rd ed.), Springer-Verlag, New York, 1974.

\end{thebibliography}
\end{document}